\documentclass[graybox,envcountchap,envcountsame]{svmono}
\usepackage{mathptmx}
\usepackage{helvet}
\usepackage{courier}
\usepackage{type1cm}

\usepackage{makeidx}         % allows index generation
\usepackage{graphicx}        % standard LaTeX graphics tool
\usepackage{multicol}        % used for the two-column index
\usepackage[bottom]{footmisc}% places footnotes at page bottom

\usepackage{paralist}

\usepackage{amsmath,amsfonts,amssymb}
\usepackage{eucal}
\usepackage{hyperref}

\usepackage[nottoc,notlot,notlof]{tocbibind}

\def\Rp{{\mathbb R^+}}
\def\Zp{{\mathbb Z^+}}
\def\Z{{\mathbb Z}}
\def\N{{\mathbb N}}
\def\R{{\mathbb R}}
\def\P{{\mathbb P}}
\def\E{{\mathbb E}}
\def\I{{\mathbb I}}
\def\V{{\mathbb V\rm ar\,}}
\smartqed
\definecolor{c50}{rgb}{1,0,0}

\makeindex

\begin{document}
%\section*{\centering\Large\bf At the Edge of Criticality:\\ 
%Markov Chains with Asymptotically Zero Drift\\[3mm]
%\normalsize Cram\'er--Doob's approach to Lamperti's problem}

\author{Denis Denisov\\
Dmitry Korshunov\\
Vitali Wachtel}
\title{At the Edge of Criticality:\\ 
Markov Chains with Asymptotically Zero Drift\\[3mm]
\Large Cram\'er--Doob's Approach to Lamperti's Problem}
\maketitle

\frontmatter%%%%%%%%%%%%%%%%%%%%%%%%%%%%%%%%%%%%%%%%%%%%%%%%%%%%%%

\begingroup

\preface

%% Please write your preface here
The main goal of this text is comprehensive study of time homogeneous 
Markov chains on the real line whose drift tends to zero at infinity,
we call such processes Markov chains with asymptotically zero drift.
Traditionally this topic is referred to as Lamperti's problem.

Time homogeneous Markov chains with asymptotically zero drift 
may be viewed as a subclass of perturbed in space random walks.
The latter are of basic importance in the study of various 
applied stochastic models,
among them branching and risk processes, queueing systems etc.
Random walks generated by sums of independent identically distributed random 
variables are well studied, see e.g. classical textbooks
by W. Feller\index{Feller} \cite{Feller} 
V.V. Petrov\index{Petrov} \cite{Pet75}
or F. Spitzer\index{Spitzer} \cite{Spitzer};
for the recent development of the theory of random walks we refer to
A.A. Borovkov and K.A. Borovkov\index{Borovkov} \cite{BB}.
There are many monographs devoted to various applications 
where random walks play a crucial r\^ole, let us just mention
books on ruin and queueing processes by S. Asmussen~\cite{Aapq,A};\index{Asmussen} 
on insurance and finance by
P. Embrechts,\index{Embrechts} C. Kl\"uppelberg,\index{Kl\"uppelberg} 
and T. Mikosch\index{Mikosch} \cite{EKM}, and
T.~Rolski,\index{Rolski} H. Schmidli,\index{Schmidli} 
V. Schmidt,\index{Schmidt} and J. Teugels\index{Teugels} \cite{RSST};
and on stochastic difference equations by D. Buraczewski,\index{Buraczewski} 
E. Damek\index{Damek} and T. Mikosch\index{Mikosch} \cite{BDM}.

In the same applied stochastic models, if one allows the process
considered to be dependent on the current state of the process,
we often get a Markov chain which has asymptotically zero drift,
we demonstrate that in the last chapter, where we particularly discuss
branching and risk processes, stochastic difference
equations and ALOHA network.

The study of processes with asymptotically zero drift was initiated
by J. Lamperti\index{Lamperti} in 1960's in a series of papers.
In particular, he classified such Markov chains in
\cite{Lamp60,Lamp63} where conditions for positive recurrence,
recurrence and transience were derived via martingale technique. 
In \cite{Lamp62}, Lamperti discovered a new class of limit theorems 
for transient Markov chains, 
including weak convergence of properly normalised square 
of a Markov chain to a $\Gamma$-distribution; 
the proof is based on the method of moments.

Later the martingale approach for the study of Markov chains 
with asymptotically zero drift was further developed,
in each particular problem the main point is to construct 
an appropriate test (Lyapunov) function such that being applied 
to a Markov chain it produces a sub- or supermartingale. 
Modern state of the art of the research in this direction 
can be found in the recent monograph by M. Menshikov,\index{Menshikov}  
S. Popov\index{Popov} and A. Wade\index{Wade} \cite{MPW16},
preceded by monographs by G. Fayolle,\index{Fayolle} 
V. Malyshev\index{Malyshev} and M. Menshikov\index{Menshikov} \cite{FMM1995}
and A.A. Borovkov\index{Borovkov} \cite{B1998}.
We have been influenced by these books and by further contacts 
with their authors.

The main advantage of martingale approach is that
the test functions considered are mostly elementary
which on one hand simplifies calculations
while on the other hand allows us to derive deep results.

However it is clear that elementary test functions do not allow us
to track subtle asymptotic behaviour of Markov chains
when we are interested in precise asymptotics, 
say of the tail invariant measure.
For that reason, there is a necessity for a novel approach
to such kind of problems. Our approach developed in this book 
includes many novel elements and
much of the material presents original research.% for the first time in the literature.
The main two ingredients are as follows:
\begin{compactitem}
\item[(i)] %We follow Cram\'er's approach based on an appropriate change of measure, 
%also called Doob's $h$-transform. It allows us to work with a new Markov chain 
%whose jumps are stochastically bounded as the original jumps are, 
%in contrast to the approach based on consideration of a function of a Markov chain 
%where---in the case of functions growing faster than linear---the jumps 
%usually are not stochastically bounded, they blow up at infinity.
To study tails of recurrence times and tails of invariant measures of recurrent chains 
we follow Cram\'er's approach based on an appropriate change of measure. 
More precisely, we apply a kind of Doob's $h$-transform to the transition kernel 
of a chain killed at entering an appropriately chosen set. 
This approach differs from the method of Lyapunov test functions, 
where one considers functions of Markov chains.
The main advantage of Cram\'er's approach consists in the fact that it allows 
us to work with a new Markov chain whose jumps are stochastically bounded 
as the original jumps are, in contrast to the approach based on consideration 
of a function of a Markov chain where---in the case of functions growing faster 
than linear---the jumps usually are not stochastically bounded, they blow up at infinity.

To perform a Doob $h$-transform of a substochastic transition kernel 
one needs a positive harmonic function for that kernel. 
By the definition, every harmonic function is a solution to a certain equation.
Thus, analytical properties of the solutions are a-priori unclear and have to be studied. 
This problem is very hard in general. In order to overcome this difficulty 
we suggest the following modification of Doob's transform: 
instead of using harmonic functions with unclear properties we
perform change of measure with a superharmonic function 
which is chosen to be sufficiently close to a harmonic one while having 
needed for our analysis analytical properties. 
The resulting kernel is then substochastic, 
but the loss of mass can be controlled effectively.

\item[(ii)] %We develop an approach for construction of test functions needed for (i)---starting 
%from the ratio of the drift to the second moment of jumps---such that 
%after change of measure based on that test function we get a 
%transition kernel which is almost stochastic far away from the origin.
%Of course, the test functions constructed in this way are not
%that elementary as in martingale approach, however then
%we can derive precise asymptotics for various characteristics of Markov chains,
%and that is our main contribution.
We develop an approach for construction of superharmonic 
functions needed for (i)---starting from the ratio of the drift to the second moment 
of jumps---such that after change 
of measure based on that test function we get a transition kernel which 
is almost stochastic far away from the origin. 
It turns out that the same approach can be used to construct 
Lyapunov test functions for the classification of Markov chains. 
Of course, the test functions constructed in this way are not that elementary 
as in martingale approach, however then we can derive better criteria for transience, 
recurrence and positive recurrence and derive precise asymptotics 
for various characteristics of Markov chains, and that is our main contribution.
\end{compactitem}

In Chapter~\ref{ch:classification} we provide a basic classification
of Markov chains, with many improvements on the results known in the literature.
In Chapter~\ref{ch:return.transient} we are interested in return probabilities
for transient Markov chains. Chapters \ref{ch:transient}
and \ref{ch:transient.2} of the
present monograph deal comprehensively with limit theorems for
transient Markov chains, including convergence to $\Gamma$ and
normal distributions while Chapter \ref{ch:asy.renewal} deals
with the corresponding renewal measure.  
Chapter~\ref{ch:change} explains how we can apply Doob's $h$-transform
to Markov chains. Chapters \ref{ch:power.asymptotics} and
\ref{ch:Weibull.asymptotics} develop technique needed for deriving 
precise tail asymptotics of power and Weibullian type respectively.
In Chapter \ref{ch:asymp.hom} we demonstrate how powerful this
approach is by studying Markov chains with asymptotically constant negative drift. 
Finally, Chapter \ref{ch:applications} presents various applied
stochastic models where Markov chains with asymptotically zero drift 
naturally arise and hence the above results for Markov chains 
are applicable to that models that leads to novel results.

As discussed in Section \ref{sec:intr.rm} for random walks
and further in \cite{BK1} for Markov chains,
the invariant measure of a Markov chain with negative drift bounded away 
from zero far away from the origin is heavy-tailed---all positive exponential
moments are infinite---if and only if the jumps are so. 
As we discuss in this book, Markov chains with asymptotically zero drift
give rise to heavy-tailed invariant measure whatever the distribution of jumps,
even if they are bounded random variables. 
So, stationary Markov chains with asymptotically zero drift 
provide an important example of a stochastic model where light-tailed
input produces heavy-tailed output.

The most part of this research monograph is based on novel results
obtained following the approach described above.
This book may be of interest for PhD students and researchers
in the area of Markov chains and their applications.

We are thankful to many colleagues for helpful~discussions, 
contributions, and bibliographical comments notably to 
D. Buraczewski, S. Foss, M.V. Menshikov and S. Popov.

This book was mostly written while the authors worked, together
and individually, in Augsburg, Lancaster, Manchester, Munich, Bielefeld, and Novosibirsk; 
we thank our home institutions, 
Augsburg, Lancaster, Manchester, Bielefeld, and Ludwig-Maximilian Universities 
and the Sobolev Institute of Mathematics. 

\begin{flushright}\noindent
Manchester\hfill {Denis Denisov}\\
Lancaster\hfill {Dmitry Korshunov}\\
Bielefeld \hfill {Vitali Wachtel}\\
\noindent August 2023 \hfill {\hphantom{a}}\\
\end{flushright}

\tableofcontents

\mainmatter%%%%%%%%%%%%%%%%%%%%%%%%%%%%%%%%%%%%%%%%%%%%%%%%%%%%%%%
%\section*{\centering\normalsize\slshape\bfseries
%Denis Denisov\footnote{The University of Manchester, UK.
%E-mail: denis.denisov@manchester.ac.uk},
%Dmitry Korshunov\footnote{Lancaster University, UK.
%E-mail: d.korshunov@lancaster.ac.uk} and
%Vitali Wachtel\footnote{Augsburg University, Germany
%E-mail: vitali.wachtel@math.uni-augsburgf.de}}

%%%%%%%%%%%%%%%%%%%%%%acronym.tex%%%%%%%%%%%%%%%%%%%%%%%%%%%%%%%%%%%%%%%%%
% sample list of acronyms
%
% Use this file as a template for your own input.
%
%%%%%%%%%%%%%%%%%%%%%%%% Springer %%%%%%%%%%%%%%%%%%%%%%%%%%

\extrachap{Notation and conventions}

\begin{list}{}{\leftmargin=65pt\labelwidth=150pt}
\item[\it Intervals]
$(x,y)$ is an open, $[x,y]$ a closed interval; half-open intervals
are denoted by $(x,y]$ and $[x,y)$.

\item[\it Integrals] $\int_x^y$ is the integral over the interval $(x,y]$.

\item[$\R$, $\Rp$, $\R^d$]
stand for the real line, the positive real half-line $[0,\infty)$,
and  $d$-dimensional Cartesian space.

\item[$\Z$, $\Z^+$]
stand for the set of integers and for the set $\{0,1,2,\ldots\}$.

\item[${\mathcal B}(S)$]
stands for the Borel $\sigma$-algebra in the space $S$.

\item[$C^\gamma(\R)$]
stands for the class of $\gamma$ times continuously
differentiable functions.

\item[$\I(A)$]
stands for the indicator function of $A$, that is $\I(A)=1$ if $A$
holds and $\I(A)=0$ otherwise.

\item[\it $O$, $o$, and $\sim$]
Let $u$ and $v$ depend on a parameter~$x$ which tends, say, to
infinity. Assuming that $v$ is positive we write
\begin{eqnarray*}
u(x) = O(v(x)) &\mbox{ if }&\limsup_{x\to\infty}|u(x)|/v(x)<\infty;\\
u(x) = o(v(x)) &\mbox{ if }& \lim_{x\to\infty}u(x)/v(x)= 0;\\
u(x) \sim v(x) &\mbox{ if }& \lim_{x\to\infty}u(x)/v(x)= 1;\\
u_n(x)=o(v_n(x))\mbox{ uniformly for all }n &\mbox{ if }&
\lim_{x\to\infty} \sup_n\Bigl|\frac{u_n(x)}{v_n(x)}\Bigr|= 0.
\end{eqnarray*}

\item[$\P\{B\}$]
stands for the probability (on some appropriate space) of the event $B$.

\item[$\P\{B|A\}$]
stands for the conditional probability of the event
$B$ given $A$.%, that is, for the ratio $\P\{BA\}/\P\{A\}$.

\item[$\E\xi$]
stands for the mean of the random variable $\xi$.

\item[$\E\{\xi;B\}$]
stands for the mean of $\xi$ over the event $B$, that is, for
$\E\xi\I(B)$.

%\item[$F*G$]
%  stands for the convolution of the distributions $F$ and $G$.

%\item[$F^{*n}$]
%stands for the $n$-fold convolution of the distribution $F$ with itself.

\item[$\xi^+$, $F^+$]
for any random variable~$\xi$ on $\R$ with distribution $F$, the
random variable~$\xi^+=\max(\xi,0)$ and $F^+$ denotes its distribution.

\item[$:=$ ($=:$)]
  The quantity on the left (right) is defined to be equal
  to the quantity on the right (left).
  
\item[$\le_{\rm st}$ ($\ge_{\rm st}$)]  
  The random variable on the left is stochastically
  not greater (not less) than the random variable on the right.

\item[$=_{\rm st}$]  
  the sign of equality in distribution.
  
\item[$\Rightarrow$]
  the sign of weak convergence of random variables
  to a random variable or distribution.

\item[\qed]
indicates the end of a proof.

\item[$\{X_n\}$]
stands for a Markov chain.

\item[$P(x,B)$]
stands for the transition probabilities of $X_n$, 
that is, for $\P\{X_{n+1}\in B\mid X_n=x\}$.

\item[$P_x\{\cdot\}$]
stands for the distribution given $X_0=x$.

\item[$\xi(x)$]
stands for the jump of $X_n$ from $x$.

\item[$m_k(x)$]
stands for the $k$th moment of the jump $\xi(x)$, 
that is, for $\E\xi^k(x)$.

\item[$m_k^{[s]}(x)$]
stands for the $s$-truncated $k$th moment of the jump $\xi(x)$, 
that is, for $\E\{\xi^k(x);\ |\xi(x)|\le s\}$.

\item[$\tau_B$]
stands for the time of the first entry of $X_n$ to a Borel set $B$, 
that is, for $\min\{n\ge 1:X_n\in B\}$.

\item[$H_x(B)$]
stands for the renewal measure of a Borel set $B$ generated by $X_n$, 
that is, for $\sum_{n=0}^\infty\P_x\{X_n\in B\}$.

\item[$r(x)$]
stands for a reference function which describes the asymptotic
behaviour of the ratio $-2m_1^{[s(x)]}(x)/m_2^{[s(x)]}(x)$ 
in the case of a recurrent chain or $2m_1^{[s(x)]}(x)/m_2^{[s(x)]}(x)$ 
in the case of a transient chain.

\item[$R(x)$]
stands for the integral of a function $r(x)$, that is, for $\int_0^x r(y)dy$.

\item[$U(x)$]
stands for either $\int_0^x e^{R(y)}dy$ or $\int_x^\infty e^{-R(y)}dy$
depending on whether recurrent or transient chain is considered.

\item[$\Gamma_{k,\theta}$]
stands for $\Gamma$-distribution with shape parameter $k$ 
and scale parameter $\theta$, that is, with probability density
function $\frac{1}{\Gamma(k)\theta^k}x^{k-1}e^{-x/\theta}$, $x\ge 0$;
the expectation is $k\theta$ and variance $k\theta^2$.

\item[$N_{a,\sigma^2}$]
stands for normal distribution with expectation $a$ and variance $\sigma^2$.

\item[$\Phi(x)$]
stands for the standard normal cumulative distribution function.

\item[$\log_{(m)}x$]
stands for the $m$th iteration of the logarithm of $x$,
$\log_{(m)}x=\log\log_{(m-1)}x$.

\item[$e^{(m)}$]
stands for a solution to the equation $\log_{(m)}x=1$.

\end{list}

\chapter{Introduction}
\chaptermark{Introduction}
\label{ch:intro}

Let $X=\{X_n, n\ge0\}$ be a time homogeneous 
Markov chain\index{Markov chain} whose
state space is a Borel subset $S$ of $\R$, 
that is, for all $x\in S$ and Borel sets 
$B_0$, \ldots, $B_{n-1}$, $B_{n+1}\in{\mathcal B}(S)$,
\begin{eqnarray*}
\P\{X_{n+1}\in B_{n+1}\mid X_0\in B_0,\ldots,X_{n-1}\in B_{n-1},X_n=x\}
&=& \P\{X_{n+1}\in B_{n+1}\mid X_n=x\}. 
\end{eqnarray*}
We usually simply say that $X_n$ takes values in $\R$, 
keeping in mind that the corresponding transition 
probabilities may be only defined on some subset $S$ of the real line.

Denote by $P(\cdot,\cdot):S\times\mathcal B(S)\to[0,1]$ 
{\em the transition probabilities}
of $\{X_n\}$:\index{Markov chain!transition probabilities}
\index{Transition probabilities}
\begin{eqnarray*}
P(x,B) &=& \P\{X_{n+1}\in B\mid X_n=x\}; 
\end{eqnarray*}
this function is measurable in $x$ for each fixed $B$ and is a probability
measure for each fixed $x$, that is, this is a stochastic transition
kernel.\index{Transition kernel}
Then, for all $n$ and $B$,
\begin{eqnarray*}
\P\{X_{n+1}\in B\} &=& \int_S P(y,B)\P\{X_n\in dy\}. 
\end{eqnarray*}

Let $\P_x\{\cdot\}=\P\{\cdot\mid X_0=x\}$ and the corresponding
expectation be denoted by $\E_x$.

Denote by $\xi(x)$, $x\in S$, a random variable corresponding to {\em the
jump} of the chain at point $x\in S$,\index{Markov chain!jumps} 
that is, a random variable with distribution
\begin{eqnarray*}
{\mathbb P}\{\xi(x)\in B\}
&=& {\mathbb P}\{X_{n+1}-X_n\in B\mid X_n=x\}\\
&=& {\mathbb P}_x\{X_{1}\in x+B\},
\quad B\in{\mathcal B}({\mathbb R}).
\end{eqnarray*}

In the sequel we always assume that $S$ is a right unbounded set.
Furthermore, for ease of notation, we assume that $P(x,B)$
is defined for all $x\in\mathbb R$.

Denote the $k$th moment of the jump at point $x$ 
by\index{Markov chain!jumps!$k$th moment} 
$$
m_k(x)\ :=\ \E\xi^k(x).
$$

\begin{definition}
We say that a Markov chain $\{X_n\}$ has an
{\em asymptotically zero drift}\index{Markov chain!asymptotically zero drift} 
if $m_1(x)=\E\xi(x)\to 0$ as $x\to \infty$. 
\end{definition}

The study of processes with asymptotically zero drift was initiated by 
Lamperti\index{Lamperti}
in a series of papers \cite{Lamp60, Lamp62,Lamp63}.

The first topic of basic importance is a classification of Markov chains
which is discussed in detail in Chapter \ref{ch:classification}.
For any Borel set $B\subset\R$ denote by $\tau_B$ the time of the first 
entry of $\{X_n\}$ to $B$,
\begin{eqnarray*}
\tau_B := \inf\{n\ge 1: X_n\in B\},
\end{eqnarray*}
with standard convention 
$\inf\emptyset=\infty$.\index{Markov chain!time of entry}

\begin{definition}
A set $B$ is called {\em positive recurrent}\index{Markov chain!positive recurrent}
\index{Positive recurrence} if $\E_x\tau_B<\infty$ for all $x\in B$.
\end{definition}

\begin{definition}
A set $B$ is called {\em non-positive}\index{Markov chain!non-positive} 
\index{Non-positivity} if it is not positive recurrent; 
more precisely, if either $\P_x\{\tau_B=\infty\}>0$, 
or $\P_x\{\tau_B<\infty\}=1$ and $\E_x\tau_B=\infty$ for some $x\in B$.
\end{definition}

\begin{definition}
A set $B$ is called {\em recurrent}\index{Markov chain!recurrent}
\index{Recurrence} if $\tau_B$ is finite a.s. 
for all initial states $x\in B$.
\end{definition}

\begin{definition}
A set $B$ is called 
{\em null recurrent}\index{Markov chain!null recurrent}
\index{Null recurrence} if $\tau_B$ is finite a.s.\ and $\E_x\tau_B=\infty$ 
for all initial states $x\in B$.
\end{definition}

\begin{definition}
A set $B$ is called {\em transient}\index{Markov chain!transient}
\index{Transience} if $\P_x\{\tau_B<\infty\}<1$ 
for all initial states $x\in B$.
\end{definition}

Notice that, for an irreducible discrete Markov chain, 
the following solidarity property holds true:
positive recurrence, non-positivity, recurrence, null-recurrence, or transience
of a finite set $B$ implies the same property for any set $B$.

In \cite{Lamp60} Lamperti\index{Lamperti} has shown that if $S=\Rp$, 
$\limsup X_n=\infty$ 
and $\E|\xi(x)|^{2+\delta}$ is bounded for some $\delta>0$ then
\begin{itemize}
 \item $2x m_1(x)\le m_2(x)+O(x^{-\delta})$ yields that some neighborhood
 of zero is recurrent,
 \item $2x m_1(x)\ge (1+\varepsilon)m_2(x)$, for some $\varepsilon>0$
 and all sufficiently large $x$, yields that any compact set
 is transient.
\end{itemize}
In \cite{Lamp63} he has proved that $2xm_1(x)+m_2(x)\le-\varepsilon$ is 
sufficient for positive recurrence of any compact set 
and that $2xm_1(x)+m_2(x)\ge\varepsilon$
implies non-positivity of any compact set (either null-recurrence or transience). 
These criteria have been improved later by Menshikov,\index{Menshikov} 
Asymont\index{Asymont} and Yasnogorodskii\index{Yasnogorodskii} \cite{AMI}. 
Instead of the existence of moments of order $2+\delta$ 
they assume that $\E\xi^2(x)\log^{2+\delta}(1+|\xi(x)|)$ is bounded. 
Moreover, they have obtained more precise classification for positive 
recurrence, null-recurrence and transience which involves iterated logarithms.

In the next section we discuss classical random walks
to show difference between them and Lamperti's processes.
It is followed by a couple of sections devoted to two types of 
specific processes---nearest neighbour Markov chains and diffusion
processes---where many characteristics of interest may be computed in
closed form following quite elementary calculations; 
that provides basic intuition needed to approach general Markov chains
with asymptotically zero drift.

In Section \ref{sec:plan} we describe our approach to 
general Markov chains with asymptotically zero drift.

\section{Random walks}
\label{sec:intr.rm}

Let us consider a fundamental example of Markov chains, random walks.
We get started by recalling some important asymptotic results which will
be extended to Lamperti's Markov chains later.

\begin{definition}
{\em A random walk}\index{Random walk} with initial state $x$
is a sequence of partial sums, $S_0=x$ and
$$
S_n\ :=\ S_{n-1}+\xi_n\ =\ x+\xi_1+\ldots+\xi_n,\quad n\ge 1,
$$
where $\xi_n$'s are independent identically distributed random variables.
\end{definition}

Any random walk is a Markov chain with transition 
kernel\index{Random walk!transition kernel}
$$
P(x,B)\ =\ \P\{\xi_1\in B-x\},\quad x\in\R,\quad B\in{\mathcal B}(\R).
$$
It is a {\em space homogeneous}\index{Markov chain!homogeneous in space} 
Markov chain because all its jumps $\xi(x)$, $x\in\R$, are distributed as $\xi_1$.
Roughly speaking, it is a process with continuous statistics in the sense that 
there are no boundary effects in this model.
%both 
%its drift %and diffusion coefficients (if finite) are
%is a continuous function of the state.

If $\E|\xi_1|<\infty$ then the Strong Law of Large 
Numbers\index{Strong law of large numbers}\index{Random walk!strong law of large numbers} 
holds, that is $S_n/n\to\E\xi_1$ a.s.\ as $n\to\infty$.
This implies, in particular, that if $\E\xi_1>0$ then
the set $(-\infty,\widehat x]$ is transient,\index{Random walk!transience} 
for all $\widehat x\in\R$.
If $\E\xi_1<0$ then the set $(-\infty,\widehat x]$ is positive recurrent.
\index{Random walk!positive recurrence}
It is also well known that in the case $\E\xi_1=0$ the random walk $S_n$
is null recurrent, that is, any bounded set
is null recurrent.\index{Random walk!null recurrence}

In addition, if $\E\xi_1^2<\infty$ then the Central Limit 
Theorem\index{Central limit theorem}\index{Random walk!central limit theorem} 
holds, that is,
\begin{eqnarray*}
\frac{S_n-n\E\xi_1}{\sqrt{n\V\xi_1}} &\Rightarrow& N_{0,1}
\quad\mbox{ as }n\to\infty.
\end{eqnarray*}

The simplest process with discontinuous statistics---with boundary effects---is 
a random walk delayed at zero which is defined next.

%\section{Random walk delayed at zero}\label{sec:intr.rmd0}

\begin{definition}
{\em A random walk delayed at zero\index{Random walk!delayed at zero}
(the Lindley recursion)}\index{Lindley recursion}
is a stochastic process $W=\{W_n,n\ge 0\}$ such that, for all $n\ge 1$,
$$
W_n\ =\ (W_{n-1}+\xi_n)^+\ :=\ \max(0,W_{n-1}+\xi_n),
$$
where $\xi_n$'s are independent identically distributed random variables
independent of $W_0\ge 0$.
\end{definition}

It is a Markov chain with transition 
kernel\index{Random walk!delayed at zero!transition kernel}
$$
P(x,B)\ =\ \P\{(x+\xi_1)^+\in B\},\quad x\in\R^+,\quad B\in{\mathcal B}(\R),
$$
which is a particular example of asymptotically homogeneous in space Markov chain
\index{Markov chain!asymptotically homogeneous in space} defined below,
because its jumps satisfy the following weak 
(and in total variation distance) convergence
$$
\xi(x)\ =_{\rm st}\ (x+\xi_1)^+-x\ \Rightarrow\ \xi_1\quad\mbox{as }x\to\infty.
$$

\begin{definition}\label{def:asymp.hom}
We say that a Markov chain $\{X_n\}$ is {\it asymptotically homogeneous in space}
\index{Markov chain!asymptotically homogeneous in space} if
\begin{equation}\label{asymp.hom.conv}
\xi(x) \Rightarrow \xi\quad\mbox{as }x\to\infty,
\end{equation}
for some random variable $\xi$. Equivalently,
$P(x,x+\cdot) \Rightarrow \P\{\xi\in\cdot\}$.
\end{definition}

Let $W_0=0$. Then
$$
W_n\ =\ \max(0,\xi_n,\xi_n+\xi_{n-1},\xi_n+\xi_{n-1}+\xi_{n-2},
\ldots,\xi_n+\xi_{n-1}+\ldots+\xi_1),
$$
hence, for all $n$, $W_n$ is equal in distribution to the maximum
\begin{eqnarray*}
M_n &:=& \max(0,\xi_1,\xi_1+\xi_2,\xi_1+\xi_2+\xi_3,
\ldots,\xi_1+\xi_2+\ldots+\xi_n)\\
&=& \max_{0\le k\le n} S_k,\quad\mbox{where }S_0=0.
\end{eqnarray*}

If $\E\xi_1>0$ then $\{W_n\}$ is a transient Markov 
chain\index{Random walk!delayed at zero!transience} 
(any bounded set is transient),
which satisfies the Central Limit Theorem provided 
$\E\xi_1^2<\infty$,\index{Random walk!delayed at zero!central limit theorem}
$$
\frac{W_n-n\E\xi_1}{\sqrt{n\V \xi_1}}\ \Rightarrow\ N_{0,1}
\quad\mbox{as }n\to\infty.
$$

If $\E\xi_1=0$ then $\{W_n\}$ is null 
recurrent\index{Random walk!delayed at zero!null-recurrence}
(any bounded set is null recurrent),
and, by the functional central limit theorem (Donsker's theorem),
$$
\frac{W_n}{\sqrt{n\V \xi_1}}\ \Rightarrow\ \sup_{t\le 1}B(t)
\quad\mbox{as }n\to\infty,
$$
where $B(t)$ is a Brownian motion, see, e.g. Billingsley 
\cite[Section 10]{{Billingsley}}.

If $\E\xi_1<0$ then $\{W_n\}$ is positive 
recurrent\index{Random walk!delayed at zero!positive recurrence}
(any bounded set is positive recurrent), and 
possesses a unique invariant probability measure, say $\pi_W$. This measure is
the distribution of $M_\infty:=\max_{n\ge 0}S_n$
and the distribution of $W_n$ converges to $\pi_W$
in the total variation metric,\index{Convergence in total variation metric} that is,
$$
\sup_{B\in{\mathcal B}(\R)}|\P\{W_n\in B\}-\pi_W(B)|\ \to\ 0
\quad\mbox{as }n\to\infty.
$$
The distribution $\pi_W$ is explicitly known in few cases only.
The tail behaviour of $\pi_W$ has been understood very well
and it heavily depends on the existence of positive exponential moments of $\xi_1$.
For that reason the following classes of distributions are introduced:

\begin{definition}\label{def:light.tailed}
We say that a distribution $F$ is {\it light-tailed}
\index{Distribution!light-tailed} if
\begin{eqnarray*}
\int_\R e^{\lambda x}F(dx) &<& \infty\quad\mbox{for some }\lambda>0.
\end{eqnarray*}
A random variable $\xi$ is called {\it light-tailed}
\index{Random variable!light-tailed} if its distribution is so.
\end{definition}

\begin{definition}\label{def:heavy.tailed}
We say that a distribution $F$ is {\it heavy-tailed}
\index{Distribution!heavy-tailed} if
\begin{eqnarray*}
\int_\R e^{\lambda x}F(dx) &=& \infty\quad\mbox{for all }\lambda>0.
\end{eqnarray*}
A random variable $\xi$ is called {\it heavy-tailed}
\index{Random variable!heavy-tailed} if its distribution is so.
\end{definition}

\begin{definition}\label{def:long.tailed}
We say that a function $g(x)$ is {\it long-tailed}
\index{Function!long-tailed} if, for any fixed $y$, $g(x+y)\sim g(x)$
as $x\to\infty$.
A distribution $F$ with right-unbounded support is called {\it long-tailed}
\index{Distribution!long-tailed} if $F(x,\infty)$ is a long-tailed function.
\end{definition}

Any long-tailed distribution is necessarily heavy-tailed.

\begin{definition}\label{def:subexponential.distr}
A distribution $F$ on $\R^+$ is called {\it subexponential}
\index{Distribution!subexponential} if
\begin{eqnarray*}
(F*F)(x,\infty) &\sim& 2F(x,\infty)\quad\mbox{as }x\to\infty.
\end{eqnarray*}
A distribution $F$ of a random variable $\xi$ is called {\it subexponential}
if the distribution of $\xi^+$ is so.
\end{definition}

Any subexponential distribution is necessarily long-tailed
and hence heavy-tailed, see e.g. \cite[Lemma 3.2]{FKZ}.

In order to describe the tail behaviour of $\pi_W$, 
let us introduce $\varphi(\lambda)=\E e^{\lambda\xi_1}$ 
and $\beta=\sup\{\lambda\ge 0:\varphi(\lambda)\le 1\}$.
Given $\P\{\xi_1>0\}>0$, $\beta<\infty$. 
It turns out that the asymptotic behavior
of $\P\{M_\infty>x\}$ heavily depends on the values of $\beta$ and $\varphi(\beta)$;
the following three different cases are
considered:\index{Random walk!delayed at zero!classification of invariant measure}
\begin{enumerate}
\item[(i)] $\beta>0$ and $\varphi(\beta)=1$, the Cram\'er case;
\item[(ii)] $\beta=0$, the heavy-tailed case 
where all positive exponential moments of $\xi_1$ are infinite;
\item[(iii)] $\beta>0$ and $\varphi(\beta)<1$, the intermediate case.
\end{enumerate}

In the Cram\'er\index{Cram\'er} case, under the additional
assumption $\varphi'(\beta-0)<\infty$, for some $c\in(0,1)$,
\index{Random walk!delayed at zero!Cram\'er--Lundberg approximation}
$$
\P\{M_\infty>x\}\ \sim\ ce^{-\beta x}\quad\mbox{as }x\to\infty;
$$
this result goes back to H. Cram\'er, see e.g. \cite{Cramer}
or \cite[Chapter XII]{Feller}. In Chapter \ref{ch:asymp.hom}, 
a similar exponential asymptotics of invariant probabilities of this type 
is proven for a broad class of asymptotically homogeneous in space
Markov chains on $\R$ with asymptotically negative drift.

In the heavy-tailed case, the tail asymptotics for $M_\infty$
is only available under subexponential type conditions,
namely,\index{Random walk!delayed at zero!subexponential approximation}
$$
\P\{M_\infty>x\}\ \sim\ \frac{1}{|\E\xi_1|}
\int_x^\infty \P\{\xi_1>y\}dy\quad\mbox{as }x\to\infty
$$
if and only if the integrated tail distribution $F_I$ on $\R^+$\index{Distribution!integrated tail} 
defined by its tail
$$
\overline F_I(x)\ :=\ \min\Bigl(1,\int_x^\infty\P\{\xi_1>y\}dy\Bigr)
$$
is subexponential, see e.g. \cite[Theorem 5.12]{FKZ}.

In the intermediate case, $\E e^{\beta M_\infty}<\infty$.
In addition, if the function $e^{\beta x}\P\{\xi_1>x\}$ is long-tailed, 
then 
$$
\P\{M_\infty>x\}\ \sim\ c\P\{\xi_1>x\}\quad\mbox{as }x\to\infty,
$$
for some $c\in(0,\infty)$ (in the lattice case $x$ must be taken
as a multiple of the lattice step), if and only if
the distribution of the random variable
$\xi_1^+$ belongs to the so-called class 
${\mathcal S}(\beta)$, \index{Distribution!class $\mathcal S(\beta)$}
see \cite[Theorem 1]{BD96} and \cite[Theorem 2]{Korshunov1997}.
In that case $c=\E e^{\beta M_\infty}/(1-\varphi(\beta))$.

So the invariant measure of $\{W_n\}$ is light-tailed 
if and only if the distribution of $\xi_1$ is so. 
As we will see in the sequel, for Markov chains 
with asymptotically zero drift the situation is very different---the
invariant measure is always heavy-tailed apart from degenerate cases.

\section{Nearest neighbour Markov chains}
\label{sec:intr.nnrm}

In this section we discuss nearest neighbour Markov chains which
represent one of the two classes of Markov chains whose either invariant measure 
in the case of positive recurrence or Green function in the case of transience 
is available in closed form. Closed form makes possible direct analysis of 
such Markov chains: classification, tail asymptotics of the invariant probabilities 
or Green function. 
This discussion sheds some light on what we may expect for general Markov chains.
Another class is provided by diffusion processes
which are discussed in the next section.

\begin{definition}
A Markov chain $\{X_n\}$ on $\Z^+$ is called 
{\em a nearest neighbour}\index{Markov chain!nearest neighbour}
({\em skip-free} or {\em continuous}) {\it Markov chain},
if $\xi(x)$ only takes values $-1$, $1$ or $0$, with probabilities 
$p_-(x)$, $p_+(x)$ and $p_0(x)=1-p_-(x)-p_+(x)$ respectively, $p_-(0)=0$. 
\end{definition}

Let 
$$
p_+(x)=p+\varepsilon_+(x)\quad\mbox{and}\quad 
p_-(x)=p-\varepsilon_-(x),\quad p\le 1/2,
$$
where all probabilities are assumed to be neither $0$ nor $1$
in order to get an irreducible Markov chain.

Assume that $\varepsilon_\pm(x)\to 0$ as $x\to\infty$
which corresponds to the case of asymptotically zero drift, 
$m_1(x)=\varepsilon_+(x)+\varepsilon_-(x)\to 0$ as $x\to\infty$.
Then the second moment of jumps is convergent, $m_2(x)\to 2p$ as $x\to\infty$.

\subsection{Positive recurrence}

The drift of the test function $L(y)=y^2$ at state $x$ equals 
\begin{eqnarray*}
\E L(x+\xi(x))-L(x) &=& 2x\E\xi(x)+\E\xi^2(x)\\
&=& 2(\varepsilon_+(x)+\varepsilon_-(x))x+2p+\varepsilon_+(x)-\varepsilon_-(x)
\quad\mbox{for }x\ge 1,
\end{eqnarray*}
so the chain is positive recurrent if 
\begin{eqnarray}\label{nbmc.ps}
\limsup_{x\to\infty}(\varepsilon_+(x)+\varepsilon_-(x))x &<& -p,
\end{eqnarray}
see, e.g. Lamperti \cite{Lamp60} or Section \ref{sec:posrec}.
\index{Markov chain!nearest neighbour!positive recurrence}
\index{Positive recurrence!nearest neighbour Markov chain}

If $\{X_n\}$ is positive recurrent, then its stationary probabilities
$\pi(x)$, $x\in\Z^+$, satisfy the equations
\begin{eqnarray*}
\pi(0) &=& \pi(0)p_0(0)+\pi(1)p_-(1),\\
\pi(x) &=& \pi(x-1)p_+(x-1)+\pi(x)p_0(x)+\pi(x+1)p_-(x+1),\quad x\ge 1,
\end{eqnarray*}
which is equivalent to
\begin{eqnarray*}
\pi(0)p_+(0) &=& \pi(1)p_-(1),\\
\pi(x+1)p_-(x+1)-\pi(x)p_+(x) &=& \pi(x)p_-(x)-\pi(x-1)p_+(x-1)\\
&\vdots& \\
&=& \pi(1)p_-(1)-\pi(0)p_+(0)\ =\ 0,
\end{eqnarray*}
which yields $\pi(x)p_-(x)=\pi(x-1)p_+(x-1)$ for all $x\ge 1$.
Hence we obtain the following solution:
\begin{eqnarray}\label{cont.c}
\pi(x) &=& \pi(0)\prod_{k=1}^x\frac{p_+(k-1)}{p_-(k)},\quad x\ge 1,
\end{eqnarray}
where
\begin{eqnarray*}
\pi(0) &=& \biggl(1+\sum_{x=1}^\infty
\prod_{k=1}^x\frac{p_+(k-1)}{p_-(k)}\biggr)^{-1}.
\end{eqnarray*}
So $X$ is positive recurrent if and only 
if\index{Markov chain!nearest neighbour!positive recurrence}
\begin{eqnarray*}
\sum_{x=1}^\infty\prod_{k=1}^x\frac{p_+(k-1)}{p_-(k)} &<& \infty;
\end{eqnarray*}
see Harris\index{Harris} \cite{Harris1952} or 
Karlin\index{Karlin} and Taylor\index{Taylor} \cite[pp. 86--87]{KT1975}
where these calculations are carried out for the case 
where $p_0(k)=0$ for all $k\ge 1$.

Since $\varepsilon_\pm(k)\to 0$,
\begin{eqnarray*}
\prod_{k=1}^x\frac{p_+(k-1)}{p_-(k)} &=& \frac{p_+(0)}{p_+(x)} 
\prod_{k=1}^x\frac{1+\varepsilon_+(k)/p}{1-\varepsilon_-(k)/p}\\
&\sim& \frac{p_+(0)}{p}
\prod_{k=1}^x\frac{1+\varepsilon_+(k)/p}{1-\varepsilon_-(k)/p}
\quad\mbox{as }x\to\infty.
\end{eqnarray*}
The logarithm of the product on the right hand side equals
\begin{eqnarray}\label{nnmc.log}
\lefteqn{\sum_{k=1}^x\bigl(\log(1+\varepsilon_+(k)/p)-\log(1-\varepsilon_-(k)/p)\bigr)}\nonumber\\
&&\hspace{35mm}=\ 
\frac{1}{p}\sum_{k=1}^x\bigl(\varepsilon_+(k)+\varepsilon_-(k)\bigr)
+\sum_{k=1}^x \delta(k),
\end{eqnarray}
where $\delta(k)=O(\varepsilon^2(k))$ as $k\to\infty$, 
for $\varepsilon(k):=\max(|\varepsilon_-(k)|,\ |\varepsilon_+(k)|)$.
In the case where 
\begin{eqnarray}\label{nnmc.quad}
\sum_{k=0}^\infty \varepsilon^2(k) &<& \infty,
\end{eqnarray}
we get, for some $c_1\in\R$,
\begin{eqnarray}\label{nnmc.pi.asy.pre}
\pi(x)\ =\ \pi(0)\prod_{k=1}^x\frac{p_+(k-1)}{p_-(k)} &\sim& 
e^{\frac{1}{p}\sum_{k=1}^x(\varepsilon_+(k)+\varepsilon_-(k))+c_1}
\quad\mbox{as }x\to\infty.
\end{eqnarray}

Let us consider a couple of examples with specific $\varepsilon$'s.
Hereinafter we need the following result on the harmonic and generalised harmonic series.

\begin{proposition}\label{gen.harm.s}
For the truncated harmonic series,
\begin{eqnarray}\label{gen.harm.s.1}
\sum_{x=1}^n\frac{1}{x} &=& \log n+\gamma+O(1/n)\quad\mbox{as }n\to\infty,
\end{eqnarray}
where $\gamma$ is the Euler constant.

For the truncated generalised harmonic series, for any $\alpha\in(0,1)$,
\begin{eqnarray}\label{gen.harm.s.alpha}
\sum_{x=1}^n\frac{1}{x^\alpha} &=& 
\frac{n^{1-\alpha}}{1-\alpha}+\gamma_\alpha+O(1/n^\alpha)
\quad\mbox{as }n\to\infty.
\end{eqnarray}
\end{proposition}

The first example of $\varepsilon$'s concerns the drift of order $-\mu/x$.
\index{Markov chain!nearest neighbour!invariant probabilities}
\index{Invariant measure!nearest neighbour Markov chain}

\begin{example}\label{ex:nnmc.pi.asy.1x}
If $\varepsilon_+(x)\sim -\mu_+/x$ and 
$\varepsilon_-(x)\sim -\mu_-/x$ as $x\to\infty$ in such a way that
\begin{eqnarray*}
\sum_{x=0}^\infty\Bigl|\varepsilon_+(x)+\varepsilon_-(x)
+\frac{\mu_++\mu_-}{x}\Bigr| &<& \infty,
\end{eqnarray*}
then \eqref{nbmc.ps} yields positive recurrence of the chain  
provided $\mu:=\mu_++\mu_->p$ and 
\eqref{nnmc.pi.asy.pre} implies an asymptotic equivalence, for some $c_2\in\R$,
\begin{eqnarray}\label{nnmc.pi.asy}
\pi(x) &\sim& e^{-(\mu/p)\log x+c_2}\ =\ \frac{e^{c_2}}{x^{\mu/p}}
\quad\mbox{as }x\to\infty.
\end{eqnarray}
\end{example}

In Chapter \ref{ch:power.asymptotics} 
power asymptotics of invariant probabilities of this type are extended to 
a broad class of Markov chains on $\R$ with asymptotically zero drift 
of order $-\mu/x$.

The second example concerns the drift of order $-\mu/x^\alpha$, $\alpha\in(0,1)$.

\begin{example}\label{ex:nnmc.pi.asy.alpha}
If $\varepsilon_+(x)\sim -\mu_+/x^\alpha$ and
$\varepsilon_-(x)\sim -\mu_-/x^\alpha$ as $x\to\infty$ 
for some $\mu_+$, $\mu_->0$ and $\alpha\in(1/2,1)$, in such a way that
\begin{eqnarray*}
\sum_{x=0}^\infty\Bigl|\varepsilon_+(x)+\varepsilon_-(x)
+\frac{\mu_++\mu_-}{x^\alpha}\Bigr| &<& \infty,
\end{eqnarray*}
then the series $\sum\varepsilon^2(x)$ is convergent again
and we observe a Weibullian asymptotic behaviour of invariant probabilities,
\begin{eqnarray}\label{nnmc.pi.asy.Wei}
\pi(x) &\sim& c_3e^{-(\mu_++\mu_-) x^{1-\alpha}/p(1-\alpha)}
\quad\mbox{as }x\to\infty.
\end{eqnarray}
If now $\alpha\in(1/3,1/2]$, then the series \eqref{nnmc.quad} diverges
and quadratic terms in \eqref{nnmc.log} make a significant contribution
to the asymptotic behaviour of invariant probabilities,
\begin{eqnarray*}
\pi(x) &\sim& c_4\exp\Bigl(-\frac{\mu_++\mu_-}{p(1-\alpha)} x^{1-\alpha}
+\frac{\mu_-^2-\mu_+^2}{(2\alpha-1)2p^2}x^{1-2\alpha}\Bigr)
\quad\mbox{as }x\to\infty.
\end{eqnarray*}
If $\alpha\in(1/4,1/3]$ then we need to keep cubic terms in 
Taylor's expansion of the logarithm which adds a further
correction term of order 
$x^{1-3\alpha}$ to the exponential function, and so on.
\end{example}

General Markov chains on $\R$ with asymptotically zero drift of order 
$-\mu/x^\alpha$, $\alpha\in(0,1)$, are considered in Chapter \ref{ch:Weibull.asymptotics} 
where Weibullian type asymptotics of invariant probabilities are proven.

\subsection{Transience}
\label{subsec:nnmc.trans}

If a nearest neighbour Markov chain $\{X_n\}$ is irreducible and transient, 
then $\P_x\{\tau_x<\infty\}<1$ for all $x$ and 
hence the renewal measure (Green function)\index{Green function}
\begin{eqnarray*}
h_{x_0}(x) &:=& \sum_{n=0}^\infty \P_{x_0}\{X_n=x\}\\
&=& \E_{x_0}\sum_{n=0}^\infty \I\{X_n=x\}
\end{eqnarray*}
is finite for all $x_0$, $x\in\Z^+$, because 
\begin{eqnarray*}
h_{x_0}(x) &=& \P_{x_0}\{X_k=x\mbox{ for some }k\}
\sum_{n=0}^\infty \P_x\{X_n=x\}\\
&=& \P_{x_0}\{X_k=x\mbox{ for some }k\}
\frac{1}{1-\P_x\{\tau_x<\infty\}}\ <\ \infty.
\end{eqnarray*}

Since we consider a Markov chain that jumps up by $1$ only,
$h_{x_0}(x)=h_x(x)$ for all ${x_0}\le x$. 
Below we demonstrate how to find $h_{x_0}(x)$ in closed form.
 
We first look for a function $g(x,z)\ge 0$ such that,
for all $x$, the process 
\begin{eqnarray}\label{nnmc.rW}
Z_n &=& g(x,X_n)-\sum_{k=0}^{n-1}\I\{X_k=x\},\quad n\ge 0,
\end{eqnarray}
is a martingale which happens if $g$ satisfies 
the following system of equations
\begin{eqnarray*}
g(x,0) &=& p_0(0)g(x,0)+p_+(0)g(x,1)-\I\{x=0\},\\
g(x,y) &=& p_-(y)g(x,y-1)+p_0(y)g(x,y)+p_+(y)g(x,y+1) - \I\{y=x\},
\quad y\ge 1.
\end{eqnarray*}
Take $g(x,0)=g(x,1)=\ldots=g(x,x)=0$. Then for $y=x$ we get
\begin{eqnarray*}
g(x,x+1) &=& g(x,x+1)-g(x,x)\ =\ \frac{1}{p_+(x)},
\end{eqnarray*}
and, for $y\ge x+1$,
\begin{eqnarray*}
g(x,y+1)-g(x,y) &=& \frac{p_-(y)}{p_+(y)}(g(x,y)-g(x,y-1))\\
&=& \prod_{z=x+1}^y\frac{p_-(z)}{p_+(z)}(g(x,x+1)-g(x,x))\\
&=& \frac{1}{p_+(x)}\prod_{z=x+1}^y\frac{p_-(z)}{p_+(z)}.
\end{eqnarray*}
Therefore, for $y\ge x+1$,
\begin{eqnarray*}
g(x,y)\ =\ \sum_{u=x}^{y-1} (g(x,u+1)-g(x,u))
&=& \frac{1}{p_+(x)}\sum_{u=x}^{y-1}
\prod_{z=x+1}^u\frac{p_-(z)}{p_+(z)}\\
&=& \frac{1}{p_-(x)}\sum_{u=x}^{y-1}
\prod_{z=x}^u\frac{p_-(z)}{p_+(z)},
\end{eqnarray*}
which is increasing in $y$. This sequence is bounded provided
\begin{eqnarray}\label{nnmc.rcond}
\sum_{u=1}^\infty \prod_{z=1}^u\frac{p_-(z)}{p_+(z)} &<& \infty.
\end{eqnarray}
Then 
\begin{eqnarray*}
g(x,\infty)\ :=\ \lim_{y\to\infty} g(x,y) &=& \frac{1}{p_+(x)}
\sum_{u=x}^\infty \prod_{z=x+1}^u\frac{p_-(z)}{p_+(z)}\ <\ \infty.
\end{eqnarray*}
The sequence \eqref{nnmc.rW} is a martingale, so for all $n$, $x$, and $x_0$,
\begin{eqnarray*}
g(x,x_0)\ =\ \E_{x_0} Z_0\ =\ \E_{x_0} Z_n &=& 
\E_{x_0} g(x,X_n)-\E_{x_0}\sum_{k=0}^{n-1}\I\{X_k=x\}
\end{eqnarray*}
and hence
\begin{eqnarray*}
\sum_{k=0}^{n-1}\P_{x_0}\{X_k=x\} &=& 
\E_{x_0} g(x,X_n)-g(x,x_0)\ <\ g(x,\infty)\ <\ \infty.
\end{eqnarray*}
Finiteness of the Green function implies transience of $\{X_n\}$,
hence $X_n\to\infty$ a.s. as $n\to\infty$. 
Thus, we get the following explicit representation for the renewal 
measure\index{Markov chain!nearest neighbour!transience!renewal measure}
\begin{eqnarray*}\label{nnmc.rA}
h_{x_0}(x) \ =\ g(x,\infty)-g(x,x_0) &=& \frac{1}{p_+(x)}
\sum_{u=x\vee x_0}^\infty \prod_{z=x+1}^u\frac{p_-(z)}{p_+(z)}\nonumber\\
&=& \frac{1}{p_-(x)}
\sum_{u=x\vee x_0}^\infty \prod_{z=x}^u\frac{p_-(z)}{p_+(z)}.%\nonumber\\
%&=& \frac{1}{p_+(x)} \prod_{z=1}^x\frac{p_+(z)}{p_-(z)}
%\sum_{u=x\vee x_0}^\infty \prod_{z=1}^u\frac{p_-(z)}{p_+(z)}.
\end{eqnarray*}
We have
\begin{eqnarray*}
\prod_{z=x}^u\frac{p_-(z)}{p_+(z)} &=& 
\exp\biggl\{
\sum_{z=x}^u\log \frac{1-\varepsilon_-(z)/p}{1+\varepsilon_+(z)/p} 
\biggr\}. 
\end{eqnarray*}
Assume that
\begin{eqnarray}\label{assump.r}
\frac{2m_1(x)}{m_2(x)}\ =\ 
\frac{2(\varepsilon_++\varepsilon_-)}{2p+\varepsilon_+-\varepsilon_-} 
&\sim& r(x)\quad\mbox{as }x\to\infty,
\end{eqnarray}
where $r(x)$ is a differentiable decreasing function such that
the quotient $r'(x)/r^2(x)$ has a limit at infinity.
The last asymptotic equivalence is equivalent to 
\begin{eqnarray*}
\log \frac{1-\varepsilon_-(x)/p}{1+\varepsilon_+(x)/p} &\sim& -r(x)
\quad\mbox{as }x\to\infty.
\end{eqnarray*}
Fix an $\varepsilon>0$. Then for all sufficiently large $x$ we can write 
\[
-(1+\varepsilon) r(x)\ \le\ \log\frac{1-\varepsilon_-(x)/p}{1+\varepsilon_+(x)/p}
\ \le\ -(1-\varepsilon) r(x). 
\] 
Therefore, for such $x$, we have the following upper bound 
\begin{eqnarray*}
h_{x_0}(x) &\le& \frac{1}{p_-(x)}\sum_{u=x}^\infty
\exp\biggl\{-(1-\varepsilon) \sum_{z=x}^u r(z) \biggr\}\\
&\le& \frac{1}{p_-(x)}\sum_{u=x}^\infty
\exp\biggl\{-(1-\varepsilon) \int_x^{u+1} r(z)dz \biggr\}\\
&\le& \frac{1}{p_-(x)}\int_x^\infty
\exp\biggl\{-(1-\varepsilon) \int_x^u r(z) dz \biggr\}du,
\end{eqnarray*}    
due to the decrease of $r(z)$. Putting 
\[
U_\varepsilon(x)= \int_x^\infty
\exp\biggl\{-(1-\varepsilon) \int_0^u r(z) dz \biggr\}du 
\]
we observe that 
\[
\int_x^\infty \exp\biggl\{-(1-\varepsilon) \int_x^u r(z)dz\biggr\}du 
=\frac{U_\varepsilon(x)}{-U_\varepsilon'(x)}.
\]
By L'H\^opital's rule and the equality 
$U_\varepsilon''(x)=-(1-\varepsilon)r(x)U_\varepsilon'(x)$,  
\begin{eqnarray*}
\lim_{x\to\infty}\frac{U_\varepsilon(x)}{-U_\varepsilon'(x)/r(x)} &=&  
\lim_{x\to\infty}\frac{U_\varepsilon'(x)}
{-U_\varepsilon''(x)/r(x)+U_\varepsilon'(x)r'(x)/r^2(x)}\\
&=& \frac{1}{1-\varepsilon+\lim_{x\to\infty}r'(x)/r^2(x)}.
\end{eqnarray*}    
Therefore, 
\begin{eqnarray*}
\limsup_{x\to\infty} h_{x_0}(x)r(x) &\le& 
\frac{1}{p}
\frac{1}{1-\varepsilon+\lim_{x\to\infty}r'(x)/r^2(x)}.
\end{eqnarray*}    
Similarly, starting from inequalities
\begin{eqnarray*}
h_{x_0}(x) &\ge& \frac{1}{p_+(x)}\sum_{u=x}^\infty
\exp\biggl\{-(1+\varepsilon) \sum_{z=x+1}^u r(z) \biggr\}\\
&\ge& \frac{1}{p_+(x)}\sum_{u=x}^\infty
\exp\biggl\{-(1+\varepsilon) \int_x^u r(z)dz \biggr\}\\
&\ge& \frac{1}{p_+(x)}\int_x^\infty
\exp\biggl\{-(1+\varepsilon) \int_x^u r(z) dz \biggr\}du,
\end{eqnarray*}
we get a lower bound  
\begin{eqnarray*}
\liminf_{x\to\infty} h_{x_0}(x)r(x) &\ge& 
\frac{1}{p}
\frac{1}{1+\varepsilon+\lim_{x\to\infty}r'(x)/r^2(x)}.
\end{eqnarray*}
Since $\varepsilon>0$ is arbitrary we obtain that 
\begin{eqnarray*}
h_{x_0}(x) &\sim& \frac{1}{pr(x)}
\frac{1}{1+\lim_{y\to\infty}r'(y)/r^2(y)}
\quad\mbox{as }x\to\infty.
\end{eqnarray*}
\index{Markov chain!nearest neighbour!Green function asymptotics}
\index{Green function!nearest neighbour Markov chain!asymptotics}

\begin{example}\label{ex:nnmc.h.asy.1x}
If $\varepsilon_+(x)\sim \mu_+/x$ and
$\varepsilon_-(x)\sim \mu_-/x$ as $x\to\infty$ and $\mu:=\mu_++\mu_->p$,
then \eqref{assump.r} is valid with $r(x)=\mu/px$,
$r'(x)/r^2(x)\to -p/\mu$, and we deduce that  
\begin{eqnarray*}
h_{x_0}(x) &\sim& \frac{x}{\mu-p}\quad\mbox{as }x\to\infty.
\end{eqnarray*}
\end{example}

\begin{example}\label{ex:nnmc.h.asy.alpha}
If $\varepsilon_+(x)\sim \mu_+/x^\alpha$ and 
$\varepsilon_-(x)\sim \mu_-/x^\alpha$ as $x\to\infty$, 
$\mu:=\mu_++\mu_->0$, $\alpha\in(0,1)$, 
then \eqref{assump.r} is valid with $r(x)=\mu/px^\alpha$,
$r'(x)/r^2(x)\to 0$, and we deduce a Weibullian asymptotics for 
the renewal measure at infinity,
\begin{eqnarray*}
h_{x_0}(x) &\sim& \frac{x^\alpha}{\mu}
\ \sim\ \frac{1}{m_1(x)}\quad\mbox{as }x\to\infty.
\end{eqnarray*}
\end{example}

%Effects similar to those for invariant measure happen if $\alpha\le 1/2$.

The last two examples demonstrate what kind of asymptotic behaviour of 
the renewal measure we could expect for general Markov chains,
see Chapters \ref{ch:transient} and \ref{ch:asy.renewal}.

We conclude this section by showing that the condition \eqref{nnmc.rcond} 
is also necessary for transience of nearest neigbour Markov chains. 
The transience of $\{X_n\}$ implies that, for all $x$, 
the sequence $\sum_{k=0}^{n-1}\I\{X_k=x\}$ monotonically converges 
almost surely and in $L_1$ as $n\to\infty$. 
Therefore, the sequence \eqref{nnmc.rW} satisfies $\E\min_nZ_n>-\infty$. 
This allows us to apply the martingale convergence theorem: 
$Z_n$ converges almost surely to an integrable random variable $Z_\infty$. 
Combining this with convergence of $\sum_{k=0}^{n-1}\I\{X_k=x\}$, 
we infer that $g(x,X_n)$ converges almost surely too. 
If we assume now that \eqref{nnmc.rcond} is not valid, then 
$$
g(x,y)\ \uparrow\ g(x,\infty)=\infty\quad\mbox{as }y\to\infty,
$$ 
and irreducibility of $\{X_n\}$ implies that 
$$
\limsup_{n\to\infty} g(x,X_n)\ =\ \infty\quad\mbox{almost surely.}
$$ 
This contradicts the convergence of $g(x,X_n)$,
so hence \eqref{nnmc.rcond} is necessary for transience of $\{X_n\}$.

An alternative approach to classification of nearest neighbour Markov 
chains may be found in Karlin\index{Karlin} 
and Taylor\index{Taylor} \cite[Section 3.7]{KT1975}.

\subsection{Harmonic functions and $h$-transforms}

Let $V$ be a non-negative solution to the system of linear equations
\begin{equation}\label{n.harm.1}
V(x)=p_+(x)V(x+1)+p_0(x)V(x)+p_-(x)V(x-1),\quad x\ge1,
\end{equation}
with the initial condition $V(0)=0$. 

Let $\tau_y$ be the first hitting time of $y$, that is,
$$
\tau_y:=\inf\{n\ge1: X_n=y\}.
$$
Then the equations \eqref{n.harm.1} with initial condition $V(0)=0$
are equivalent to
\begin{eqnarray}\label{mart.pro}
V(x) &=& \E_x\{V(X_1);\ \tau_0>1\},\quad x\ge 1,
\end{eqnarray}
which defines a harmonic function for the chain $\{X_n\}$ killed at hitting zero.
\index{Markov chain!nearest neighbour!harmonic function}
\index{Harmonic function!nearest neighbour Markov chain}

It is clear that \eqref{n.harm.1} can be rewritten in the form
$$
p_+(x)[V(x+1)-V(x)]=p_-(x)[V(x)-V(x-1)].
$$
Consequently,
\begin{eqnarray}\label{n.harm.2.0}
V(x+1)-V(x) &=& [V(1)-V(0)]\prod_{k=1}^x\frac{p_-(k)}{p_+(k)},\quad x\ge1.
\end{eqnarray}
Recalling that $V(0)=0$, we then 
obtain\index{Markov chain!nearest neighbour!recurrence!harmonic function}
\begin{equation}\label{n.harm.2}
V(x)=\sum_{y=0}^{x-1}[V(y+1)-V(y)
]=V(1)\sum_{y=0}^{x-1}\prod_{k=1}^y\frac{p_-(k)}{p_+(k)}.
\end{equation}

Existence of a positive harmonic function allows us to transform
a strictly substochastic transition kernel into a stochastic one.
For every $x\ge1$, define
$$
\widehat p_+(x):=\frac{V(x+1)}{V(x)}p_+(x),\quad
\widehat p_0(x)=p_0(x)
\quad\text{and}\quad
\widehat p_-(x):=\frac{V(x-1)}{V(x)}p_-(x).
$$
The new transition kernel $\widehat P$ is stochastic because, 
as follows from \eqref{n.harm.1},
$$
\widehat{p}_-(x)+\widehat{p}_0(x)+\widehat{p}_+(x)=1
\quad\mbox{for all }x\ge1.
$$
This transformation is called Doob's $h$-transform,
for a killed at hitting zero Markov chain.\index{Doob's $h$-transform}
\index{Markov chain!nearest neighbour!Doob's $h$-transform}

Let $\{\widehat X_n\}$ be a Markov chain on $\{1,2,\ldots\}$ with
transition kernel $\widehat P$. %We can enlarge the state space
%by adding $0$ and putting $\widehat{p}_+(0)=1$. The boundary condition
%$V(0)=0$ implies that $\widehat{p}_-(1)=0$. Therefore, the chain 
%$\widehat X$ may start at zero, but $\widehat X_n\ge1$ for all $n\ge1$.
This chain is always transient. For that, as shown in the previous subsection, 
it suffices to show that \eqref{nnmc.rcond} holds for the 
transition probabilities $\widehat P$.
We first apply the definition of $\widehat P$:
$$
\sum_{u=1}^\infty\prod_{z=2}^u\frac{\widehat{p}_-(z)}{\widehat{p}_+(z)}
\ =\ 
\sum_{u=1}^\infty\prod_{z=2}^u\frac{V(z-1)}{V(z+1)}\frac{p_-(z)}{p_+(z)}
\ =\
\sum_{u=1}^\infty\frac{V(1)V(2)}{V(u)V(u+1)}\prod_{z=2}^u\frac{p_-(z)}{p_+(z)}.
$$
It follows from \eqref{n.harm.2.0} that
$$
\frac{1}{V(u)}-\frac{1}{V(u+1)}
\ =\ 
\frac{V(u+1)-V(u)}{V(u)V(u+1)}
\ =\
\frac{V(1)}{V(u)V(u+1)} \prod_{z=1}^u\frac{p_-(z)}{p_+(z)}.
$$
Therefore,
\begin{eqnarray*}
\sum_{u=1}^\infty\prod_{z=2}^u\frac{\widehat{p}_-(z)}{\widehat{p}_+(z)}
&=& \frac{p_+(1)}{p_-(1)}V(2)
\sum_{u=1}^\infty\left(\frac{1}{V(u)}-\frac{1}{V(u+1)}\right)\\
&\le& \frac{p_+(1)}{p_-(1)}\frac{V(2)}{V(1)}\ <\ \infty,
\end{eqnarray*}
which is equivalent to the transience of the transformed chain $\{\widehat X_n\}$.

One of the standard applications of Doob's $h$-transform is 
the random walk conditioned to stay positive. 
Let $\{X_n\}$ be a simple symmetric random walk on $\Z$, that is,
$p_-(x)=p_+(x)=1/2$ for all $x\in\Z$. 
Then it follows from \eqref{n.harm.2} that $V(x)=xV(1)$. 
As a result the transformed chain $\{\widehat X_n\}$ has transition probabilities
$$
\widehat{p}_-(x)=\frac{x-1}{2x}=\frac12-\frac{1}{2x},\quad 
\widehat{p}_+(x)=\frac{x+1}{2x}=\frac12+\frac{1}{2x},\quad x\ge1.
$$
It is immediate from this formula, that the transformed chain has 
an asymptotically zero drift and unit second moment of jumps.

If the original Markov chain $\{X_n\}$ is recurrent then one can use 
the $h$-transform to connect the stationary measure $\pi$
of $\{X_n\}$ with the Green function of $\{\widehat X_n\}$. 
The following representation for the invariant measure
$\pi$ via cycle structure (generated by the atom at $0$) 
of the Markov chain $\{X_n\}$ is well known---see, 
e.g. \cite[Theorem 10.4.9]{MT}, for $x\ge 1$,
$$
\pi(x)\ =\ \pi(0)\sum_{n=1}^\infty\P_0\{X_n=x,\ \tau_0>n\}
\ =\ \pi(0)p_+(0)\sum_{n=0}^\infty \P_1\{X_n=x,\ \tau_0>n\}.
$$
Noting that $\P_1\{X_n=x,\ \tau_0>n\}=\frac{V(1)}{V(x)}
\P_1\{\widehat{X}_n=x\}$ for all $x$, $n\ge1$, we obtain
\begin{eqnarray}\label{pi.V.h}
\pi(x) &=& \frac{\pi(0)p_+(0)V(1)}{V(x)}\widehat{h}_1(x),
\end{eqnarray}
where
$$
\widehat{h}_1(x)\ :=\ \sum_{n=0}^\infty\P_1\{\widehat{X}_n=x\},\quad x\ge1.
$$

Let us consider a couple of examples, we firstly discuss the drift of order $-\mu/x$.

\begin{example}\label{ex:nnmc.pi.asy.1x.via}
Let $\varepsilon_+(x)\sim -\mu_+/x$ and 
$\varepsilon_-(x)\sim -\mu_-/x$ as $x\to\infty$
in such a way that 
\begin{eqnarray*}
\sum_{x=0}^\infty\Bigl|\varepsilon_+(x)+\varepsilon_-(x)
+\frac{\mu_++\mu_-}{x}\Bigr| &<& \infty.
\end{eqnarray*}
Let $\mu:=\mu_++\mu_->p$, so the chain is positive recurrent. 
As follows from \eqref{n.harm.2.0}, for all $x\ge1$, 
\begin{eqnarray*}
V(x+1)-V(x) &=& [V(1)-V(0)]\prod_{k=1}^x\frac{p_-(k)}{p_+(k)}\\
&=& [V(1)-V(0)] e^{\sum_{k=1}^x(\log p_-(k)-\log p_+(k))}\\
&=& [V(1)-V(0)] e^{\sum_{k=1}^x(\log(1-\varepsilon_-(k)/p)-\log(1+\varepsilon_+(k)/p))}.
\end{eqnarray*}
As in \eqref{nnmc.pi.asy.pre}, we conclude an asymptotic relation, for some $c_1$,
\begin{eqnarray*}
V(x+1)-V(x) &\sim& [V(1)-V(0)] 
e^{-\frac{1}{p}\sum_{k=1}^x(\varepsilon_-(k)+\varepsilon_+(k))+c_1}\\
&\sim& [V(1)-V(0)] e^{\frac{\mu_-+\mu_+}{p}\log x+c_2}\\
&\sim& c_3 x^{\mu/p}\quad\mbox{as }x\to\infty.
\end{eqnarray*}
Therefore, as $x\to\infty$,
\begin{eqnarray*}
\frac{V(x+1)}{V(x)} &=& 1+\frac{V(x+1)-V(x)}{V(x)}
\ =\ 1+\frac{\mu/p+1}{x}+o(1/x),
\end{eqnarray*}
and
\begin{eqnarray*}
\frac{V(x-1)}{V(x)} &=& 1-\frac{V(x)-V(x-1)}{V(x)} 
\ =\ 1-\frac{\mu/p+1}{x}+o(1/x).
\end{eqnarray*}
Hence, the transition probabilities of the transformed Markov chain satisfy the relations
\begin{eqnarray*}
\widehat p_+(x) &:=& \frac{V(x+1)}{V(x)}p_+(x)
\ =\ p+\frac{\mu_-+p}{x}+o(1/x),\\
\widehat p_-(x) &:=& \frac{V(x-1)}{V(x)}p_-(x)
\ =\ p-\frac{\mu_++p}{x}+o(1/x).
\end{eqnarray*}
It follows from Example \ref{ex:nnmc.h.asy.1x} with 
$\widehat \mu_+=\mu_-+p$ and $\widehat\mu_-=\mu_++p$ that
\begin{eqnarray*}
\widehat h_1(x) &\sim& \frac{x}{\widehat\mu_++\widehat\mu_--p}
\ =\ \frac{x}{\mu+p},
\end{eqnarray*}
which being substitute into \eqref{pi.V.h} implies, as $x\to\infty$,
\begin{eqnarray*}
\pi(x) &=& c_3\frac{\widehat h_1(x)}{V(x)}
\ \sim\ \frac{c_4}{x^{\mu/p}},
\end{eqnarray*}
which coincides with the answer in \eqref{nnmc.pi.asy}.
\end{example}

This relation between the stationary measure of a nearest neighbour Markov chain
and the Green function of the transformed chain may be extended to general case. 
We follow this approach in Chapter \ref{ch:power.asymptotics} to derive
power asymptotics of invariant probabilities of this type for 
a broad class of Markov chains on $\R$ with asymptotically zero drift 
of order $-\mu/x$.

The second example concerns the drift of order $-\mu/x^\alpha$, $\alpha\in(0,1)$.

\begin{example}\label{ex:nnmc.pi.asy.alpha.via}
Let $\varepsilon_+(x)\sim -\mu_+/x^\alpha$ and
$\varepsilon_-(x)\sim -\mu_-/x^\alpha$ as $x\to\infty$ 
for some $\mu_+$, $\mu_->0$ and $\alpha\in(1/2,1)$, in such a way that
\begin{eqnarray*}
\sum_{x=0}^\infty\Bigl|\varepsilon_+(x)+\varepsilon_-(x)
+\frac{\mu_++\mu_-}{x^\alpha}\Bigr| &<& \infty.
\end{eqnarray*}
Similarly to the last example, for some $c_5$,
\begin{eqnarray*}
V(x+1)-V(x) &\sim& [V(1)-V(0)] 
e^{-\frac{1}{p}\sum_{k=1}^x(\varepsilon_-(k)+\varepsilon_+(k))+c_5}\\
&\sim& c_6 e^{\frac{\mu_-+\mu_+}{p(1-\alpha)} x^{1-\alpha}}
\quad\mbox{as }x\to\infty.
\end{eqnarray*}
Therefore, as $x\to\infty$,
\begin{eqnarray*}
\frac{V(x+1)}{V(x)} &=& 1+\frac{\mu_++\mu_-}{px^\alpha}+o(1/x),
\end{eqnarray*}
and
\begin{eqnarray*}
\frac{V(x-1)}{V(x)} &=& 1-\frac{\mu_++\mu_-}{px^\alpha}+o(1/x).
\end{eqnarray*}
Hence, the transition probabilities of the transformed Markov chain satisfy the relations
\begin{eqnarray*}
\widehat p_+(x) &:=& \frac{V(x+1)}{V(x)}p_+(x)
\ =\ p+\frac{\mu_-}{x^\alpha}+O(1/x^{2\alpha}),\\
\widehat p_-(x) &:=& \frac{V(x-1)}{V(x)}p_-(x)
\ =\ p-\frac{\mu_+}{x^\alpha}+O(1/x^{2\alpha}).
\end{eqnarray*}
It follows from Example \ref{ex:nnmc.h.asy.1x} with 
$\widehat \mu_+=\mu_-$ and $\widehat\mu_-=\mu_+$ that
\begin{eqnarray*}
\widehat h_1(x) &\sim& \frac{x^\alpha}{\widehat\mu_++\widehat\mu_-},
\end{eqnarray*}
which being substitute into \eqref{pi.V.h} implies
a Weibullian asymptotic behaviour of invariant probabilities, as $x\to\infty$,
\begin{eqnarray*}
\pi(x) &=& c_7\frac{\widehat h_1(x)}{V(x)}
\ \sim\ c_8 e^{-\frac{\mu_-+\mu_+}{p(1-\alpha)} x^{1-\alpha}},
\end{eqnarray*}
which coincides with the answer in \eqref{nnmc.pi.asy.Wei}.
\end{example}

General Markov chains on $\R$ with asymptotically zero drift of order 
$-\mu/x^\alpha$, $\alpha\in(0,1)$, are considered in Chapter \ref{ch:Weibull.asymptotics} 
where we again follow the approach above to derive 
Weibullian type asymptotics of invariant probabilities.

\subsection{Down-crossing probabilities for transient chain}
\label{subsec:dcp.nnmc}

Let $\{X_n\}$ be transient, that is, the probability of hitting the origin,
$\P_x\{\tau_0<\infty\}$, is less then 1 for all $x\ge 1$. 
The goal of the following calculations is to find this probability.

The function $V(x)$ computed in \eqref{n.harm.2} is increasing
and bounded provided the condition \eqref{nnmc.rcond} holds.
As it has already been noticed in \eqref{mart.pro},
the sequence $V(X_{n\wedge \tau_0})$ is a bounded non-negative martingale, 
so by the optional stopping theorem, 
\begin{eqnarray*}
V(x) \ =\ \E_x V(X_0) &=& \E_x V(X_{\tau_0})\\
&=& V(0)\P_x\{\tau_0<\infty\}+V(\infty) \P_x\{\tau_0=\infty\}
\end{eqnarray*}
and hence
\begin{eqnarray*}
\P_x\{\tau_0<\infty\}
&=& \frac{V(\infty)-V(x)}{V(\infty)-V(0)}
\ =\ \frac{\sum_{y=x}^\infty \prod_{k=1}^y\frac{p_-(k)}{p_+(k)}}
{\sum_{y=0}^\infty \prod_{k=1}^y\frac{p_-(k)}{p_+(k)}}.
\end{eqnarray*}
Owing to the left continuity of the Markov chain,
similarly we get, for all 
$0\le \widehat x<x$,\index{Markov chain!nearest neighbour!transience!returning probability}
\begin{eqnarray}\label{return.probab.pm.rp}
\P_x\{\tau_{\widehat x}<\infty\}
&=& \frac{V(\infty)-V(x)}{V(\infty)-V(\widehat x)}
\ =\ \frac{\sum_{y=x}^\infty \prod_{k=1}^y\frac{p_-(k)}{p_+(k)}}
{\sum_{y=\widehat x}^\infty \prod_{k=1}^y\frac{p_-(k)}{p_+(k)}}.
\end{eqnarray}

\index{Markov chain!nearest neighbour!down-crossing probabilities}
\index{Downcrossing probabilities!nearest neighbour Markov chain}

\begin{example}\label{ex:nnmc.h.asy.1x.rp}
In the case where $\varepsilon_+(x)\sim \mu_+/x$ and
$\varepsilon_-(x)\sim \mu_-/x$ as $x\to\infty$, $\mu:=\mu_++\mu_->p$, and
\begin{eqnarray*}
\sum_{x=0}^\infty\Bigl|\varepsilon_+(x)+\varepsilon_-(x)
-\frac{\mu}{x}\Bigr| &<& \infty,
\end{eqnarray*}
then similarly to \eqref{nnmc.pi.asy} we derive that
\begin{eqnarray*}
\prod_{k=1}^y\frac{p_-(k)}{p_+(k)} &\sim& c_5y^{-\mu/p}
\quad\mbox{as }y\to\infty,
\end{eqnarray*}
where $c_5>0$. Therefore, \eqref{return.probab.pm.rp} implies that
there exists a function $c(\widehat x)\to 1$ as $\widehat x\to\infty$
such that
\begin{eqnarray*}
\P_x\{\tau_{\widehat x}<\infty\} &\sim& 
c(\widehat x)(\widehat x/x)^{\mu/p-1}
\quad \mbox{as }x\to\infty,\mbox{ uniformly for all }\widehat x<x.
\end{eqnarray*}
In particular,
\begin{eqnarray*}
\P_x\{\tau_{\widehat x}<\infty\} &\sim& (\widehat x/x)^{\mu/p-1}
\quad \mbox{as }\widehat x,\ x\to\infty,\ x>\widehat x.
\end{eqnarray*}
Compare to Theorem \ref{thm:transient.return} and Corollary \ref{cor:transient.return}
where a general transient Markov chain with a drift of order $\mu/x$ is studied.
\end{example}

\begin{example}\label{ex:nnmc.h.asy.alpha.rp}
If $\varepsilon_+(x)\sim \mu_+/x^\alpha$ and 
$\varepsilon_-(x)\sim \mu_-/x^\alpha$ as $x\to\infty$, 
$\mu:=\mu_++\mu_->0$, $\alpha\in(1/2,1)$, and
\begin{eqnarray*}
\sum_{x=0}^\infty\Bigl|\varepsilon_+(x)+\varepsilon_-(x)
-\frac{\mu}{x^\alpha}\Bigr| &<& \infty,
\end{eqnarray*}
then the series $\sum\varepsilon^2(x)$ is convergent and we get that
\begin{eqnarray*}
\prod_{k=1}^y\frac{p_-(k)}{p_+(k)} &\sim& 
c_6e^{-\mu y^{1-\alpha}/p(1-\alpha)}
\quad\mbox{as }y\to\infty,
\end{eqnarray*}
where $c_6>0$. Therefore, \eqref{return.probab.pm.rp} implies
a Weibullian asymptotic behaviour of the down-crossing probability,
that is, there exists a function $c(\widehat x)\to 1$ as 
$\widehat x\to\infty$ such that
\begin{eqnarray*}
\P_x\{\tau_{\widehat x}<\infty\} &\sim& c(\widehat x)
\frac{\sum_{u=x}^\infty e^{-\mu u^{1-\alpha}/p(1-\alpha)}}
{\sum_{u=\hat x}^\infty e^{-\mu u^{1-\alpha}/p(1-\alpha)}}\\
&\sim&\ c(\widehat x)\left(\frac{x}{\hat x}\right)^\alpha
e^{\mu(\widehat x^{1-\alpha}-x ^{1-\alpha})/p(1-\alpha)}
\quad \mbox{as }x\to\infty\mbox{ uniformly for all }\widehat x<x.
\end{eqnarray*}
In particular,
\begin{eqnarray*}
\P_x\{\tau_{\widehat x}<\infty\} 
&\sim&\ \left(\frac{x}{\hat x}\right)^\alpha
e^{\mu(\widehat x^{1-\alpha}-x ^{1-\alpha})/p(1-\alpha)}
\quad \mbox{as }\widehat x,\ x\to\infty,\ x>\widehat x.
\end{eqnarray*}
Compare to Theorem \ref{thm:transient.return.hx} where a general
transient Markov chain with a drift of order $\mu/x^\alpha$,
$\alpha\in(1/2,1)$, is studied.
\end{example}

\section{Heuristics coming from diffusion processes}
\label{sec:introduction}

\subsection{Diffusions with bounded smooth infinitesimal parameters}
\label{sec:intr.rmdp}

Another example where various characteristics are 
available in closed form is provided by diffusion processes on $\R$
which are Markov processes with continuous paths. 
Being sampled at non-random equally spaced time epochs 
they give us examples of Markov chains
for which some characteristics are explicitly calculable.

Let us start with a result that demonstrates that the existence of 
an invariant probability measure for a diffusion process 
is equivalent to its positive recurrence.
\index{Diffusion process!positive recurrence}
\index{Positive recurrence!diffusion process}

\begin{lemma}\label{l:dp.im-pr}
For a diffusion process $\{X(t)\}$ with diffusion coefficient 
everywhere positive the following is equivalent:
\begin{itemize}
\item[\rm(i)] there is a stationary version of the process $\{X(t)\}$;
\item[\rm(ii)] the process $\{X(t)\}$ is positive recurrent, that is, 
$\E_x\tau_y<\infty$ for all states $x$ and $y$, where $\tau_y:=\inf\{t:X(t)=y\}$.
\end{itemize}
\end{lemma}

\begin{proof}
Let $\{X(t)\}$ possess an invariant probability measure $\pi$.
Then the same is true for the slotted Markov chain $X_n=X(n)$, $n\in\Zp$.
Since the diffusion coefficient is everywhere positive, 
the jumps of $\{X_n\}$ are absolutely continuous with positive density function,
so the chain $\{X_n\}$ is Harris recurrent. Therefore,
the existence of invariant probability measure for $\{X_n\}$
implies positive recurrence of any compact set $B$ 
of positive Lebesgue measure in the sense that $\E_x \tau_B<\infty$ for all $x$.
Hence, $B$ is positive recurrent for $\{X(t)\}$ too which implies positive recurrence
of the diffusion process due to the continuity of its paths.

Vice versa, let $\{X(t)\}$ be positive recurrent.
Then, for any two fixed distinct states $x$ and $y$, the stopping time
\begin{eqnarray*}
\tau &:=& \min\{t:X(t)=x\mbox{ and }X(s)=y\mbox{ for some }s<t\},
\end{eqnarray*}
is finite on average given $X(0)=x$, $\E_x\tau<\infty$.
In addition, $\tau>0$. For that reasons a measure
\begin{eqnarray*}
\mu(B) &:=& \E_x\int_0^\tau\I\{X(t)\in B\}dt\\
&=& \int_0^\infty \P_x\{X(t)\in B,\ \tau>t\}dt
\end{eqnarray*}
is non-zero and finite, $\mu(\R)=\E_x\tau\in(0,\infty)$. 
Let us show it is invariant for $\{X(t)\}$, 
that is, for any $s>0$ and any bounded continuous function $\varphi:\R\to\R$,
\begin{eqnarray*}
\int_\R \varphi(z)\mu(dz) &=& \int_\R \E\{\varphi(X(s))\mid X(0)=z\}\mu(dz).
\end{eqnarray*}
Indeed, the difference between the right and left hand side integrals equals to
\begin{eqnarray*}
\lefteqn{\int_\R \E\{\varphi(X(s))-\varphi(z)\mid X(0)=z\}\mu(dz)}\\ 
&=& \int_\R \E\{\varphi(X(t+s))-\varphi(X(t))\mid X(t)=z\}
\int_0^\infty \P_x\{X(t)\in dz,\ \tau>t\}dt\\
&=& \int_0^\infty \E_x\{\varphi(X(t+s))-\varphi(X(t)),\ \tau>t\}dt,
\end{eqnarray*}
because $\{\tau>t\}=\overline{\{\tau\le t\}}\in\sigma(X_u,\ u\le t)$.
Since
\begin{eqnarray*}
\int_0^\infty \E_x\{\varphi(X(t+s)),\ \tau>t\}dt &=&
\E_x \int_0^\tau \varphi(X(t+s))dt\\
&=& \E_x \int_s^{\tau+s} \varphi(X(t))dt,
\end{eqnarray*}
we get
\begin{eqnarray*}
\int_0^\infty \E_x\{\varphi(X(t+s))-\varphi(X(t)),\ \tau>t\}dt
&=& \E_x \int_s^{\tau+s} \varphi(X(t))dt-\E_x \int_0^\tau \varphi(X(t))dt\\
&=& \E_x \int_\tau^{\tau+s} \varphi(X(t))dt-\E_x \int_0^s \varphi(X(t))dt\\
&=& 0,
\end{eqnarray*}
by the Markov property, due to $X(\tau)=x$.
\qed
\end{proof}

Consider a diffusion process $X=\{X(t)\}$ on $\R$ with smooth drift $\mu(x)$ 
and diffusion coefficient $\sigma^2(x)>0$.
In the case of stationary diffusion process, the invariant density function $p(x)$ 
solves the stationary Kolmogorov forward equation\index{Kolmogorov forward equation}
\begin{eqnarray*}
0 &=& -\frac{d}{dx}(\mu(x)p(x))+\frac12\frac{d^2}{dx^2}(\sigma^2(x)p(x)),
\end{eqnarray*}
which has the following solution:
\begin{eqnarray}\label{ex:1}
p(x) &=& \frac{c}{\sigma^2(x)}e^{\int_0^x\frac{2\mu(y)}{\sigma^2(y)}dy},
\quad c>0.
\end{eqnarray}
It follows that a diffusion process possesses a probabilistic 
invariant distribution---is positive recurrent---if and only 
if\index{Diffusion process!condition for!positive recurrence}
\begin{eqnarray}\label{ex:1.cond}
\mbox{the function }\ \frac{1}{\sigma^2(x)}
e^{\int_0^x\frac{2\mu(y)}{\sigma^2(y)}dy}
\quad\mbox{is integrable at }\pm\infty.
\end{eqnarray}

It is also known that the half-line $(-\infty,0]$ is recurrent for 
a diffusion process \index{Diffusion process!condition for!recurrence}
in the sense that $\P_x\{X(t)\le 0\mbox{ for some }t\}=1$ for all $x>0$, if
\begin{eqnarray}\label{ex:1.cond.rec}
\mbox{the function }\ e^{-\int_0^x\frac{2\mu(y)}{\sigma^2(y)}dy}
\quad\mbox{is not integrable at }\infty;
\end{eqnarray}
see, e.g. \cite[Ch. 15, Theorem 7.3]{KT} or \cite[Section 4.1]{ChE}; 
and the other way around, 
it is transient\index{Diffusion process!condition for!transience} 
in the sense that $\P_x\{X(t)>0\mbox{ for all }t>0\}>0$ for all $x>0$, if
\begin{eqnarray}\label{ex:1.cond.tran}
\mbox{the function }\ e^{-\int_0^x\frac{2\mu(y)}{\sigma^2(y)}dy}
\quad\mbox{is integrable at }\infty,
\end{eqnarray}
see, e.g. \cite[Ch. 15, Lemma 6.1]{KT}.

As one can see, the classification of diffusion processes heavily relies 
on the asymptotic behaviour of the ratio $2\mu(x)/\sigma^2(x)$ at infinity.
In particular, if
\begin{eqnarray}\label{intro:mu.b.class}
\mu(x)\ \sim\ -\mu/x &\mbox{ and }& \sigma^2(x)\to \sigma^2>0\ \mbox{ as }x\to\infty
\end{eqnarray} 
for some $\mu\in\R$ and $\sigma^2>0$, then 
\begin{itemize}
\item integrability at infinity in \eqref{ex:1.cond} holds for $2\mu>\sigma^2$;
\item non-integrability at infinity in \eqref{ex:1.cond.rec} 
holds for $2\mu>-\sigma^2$;
\item integrability at infinity in \eqref{ex:1.cond.tran} 
holds for $2\mu<-\sigma^2$.
\end{itemize}

The knowledge of the invariant probability density function in closed form 
\eqref{ex:1} allows us to analyse its asymptotic behaviour under various
regularity conditions of the drift and diffusion coefficients at infinity.
\index{Diffusion process!invariant density function}
\index{Invariant density function!diffusion process}

\begin{example}\label{ex:dif.p.asy.1x}
Let $\{X(t)\}$ possess a probabilistic invariant measure
%be positive recurrent 
and let
\eqref{intro:mu.b.class} hold with $2\mu>\sigma^2$. If
\begin{eqnarray*}
\int_1^\infty \Bigl|\frac{\mu(x)}{\sigma^2(x)}+\frac{\mu}{\sigma^2x}\Bigr|dx 
&<& \infty,
\end{eqnarray*}
then \eqref{ex:1} yields the following asymptotic equivalence, 
for some $c_1>0$,
\begin{eqnarray*}
p(x) &\sim& \frac{c_1}{x^{2\mu/\sigma^2}} \quad\mbox{as }x\to\infty.
\end{eqnarray*} 
\end{example}

\begin{example}\label{ex:dif.p.asy.alpha}
If $\{X(t)\}$ possesses a probabilistic invariant measure,
%is positive recurrent, 
$\mu(x)\sim -\mu/x^\alpha$ and $\sigma^2(x)\to\sigma^2>0$ as $x\to\infty$ 
for some $\mu>0$ and $\alpha\in(0,1)$, in such a way that
\begin{eqnarray*}
\int_1^\infty \Bigl|\frac{\mu(x)}{\sigma^2(x)}+
\frac{\mu}{\sigma^2x^\alpha}\Bigr|dx &<& \infty,
\end{eqnarray*}
then 
\begin{eqnarray*}
p(x) &\sim& c_2e^{-2\mu x^{1-\alpha}/\sigma^2(1-\alpha)}
\quad\mbox{as }x\to\infty.
\end{eqnarray*}
\end{example}

Let $\{X(t)\}$ be a diffusion process satisfying the condition
\eqref{ex:1.cond.tran}, so the negative half-line $(-\infty,0]$ is transient.
A harmonic function $h(x)$ for such a diffusion process
with transition kernel $P(t,x,dy)$, that is, a solution to the equation
\begin{eqnarray}\label{eq:harmonic.diffusion}
\Bigl(\frac{\sigma^2(x)}{2}\frac{d^2}{dx^2}+\mu(x)\frac{d}{dx}\Bigr)h(x) 
&=& 0,
\end{eqnarray}
is computable in a closed form as 
follows \index{Diffusion process!transience!harmonic function}
\begin{eqnarray}\label{intr:dif.tran.harm}
h(x) &=& \int_x^\infty e^{-\int_0^z\frac{2\mu(y)}{\sigma^2(y)}dy}dz,\quad x\in\R.
\end{eqnarray}
It is a positive decreasing function.
By It\^{o}'s formula, the process $\{h(X(t))\}$ is a martingale, 
hence we can apply Doob's $h$-transform 
which returns a new stochastic transition kernel 
\begin{eqnarray*}
\widehat P(t,x,dy) &:=& \frac{h(y)}{h(x)}P(t,x,dy).
\end{eqnarray*}
Let us consider a diffusion process $\widehat X=\{\widehat X(t)\}$ 
with this transition kernel. 
The drift coefficient of $\widehat X$ equals
\begin{eqnarray}\label{intr:hat.m1}
\widehat\mu(x) &=& \lim_{t\to 0} 
\frac{1}{t}\int (y-x)\frac{h(y)}{h(x)}P(t,x,dy)\nonumber\\
&=& \lim_{t\to 0} \frac{1}{t}
\int (y-x)\Bigl(1+\frac{h'(x)}{h(x)}(y-x)
+O((y-x)^2)\Bigr)P(t,x,dy)\nonumber\\
&=& \mu(x)+\frac{h'(x)}{h(x)}\sigma^2(x),
\end{eqnarray}
and since $h'(x)<0$, $\widehat\mu(x)<\mu(x)$.
The diffusion coefficient does not change, $\widehat\sigma^2(x)=\sigma^2(x)$.

If, for some $\widetilde c>3$,
\begin{eqnarray*}
\frac{2\mu(x)}{\sigma^2(x)} &\ge& \frac{\widetilde c}{x}
\quad\mbox{ultimately in }x,
\end{eqnarray*}
then under some mild additional condition,
\begin{eqnarray*}
-h'(x) &\ge& c_2h(x)/x\quad\mbox{ for some }c_2>0,
\end{eqnarray*}
and the set $(-\infty,0]$ is positive recurrent for 
the transformed chain $\{\widehat X(t)\}$. Indeed, in this case 
\begin{eqnarray*}
h(x) &\le& \int_x^\infty e^{c_3-\int_1^z\frac{\widetilde c}{y}dy}dz
\ =\ c_4x^{1-\widetilde c},
\end{eqnarray*}
hence the function 
\begin{eqnarray*}
e^{\int_0^x\frac{2\widehat\mu(y)}{\widehat\sigma^2(y)}dy} &=& 
e^{\int_0^x\frac{2\mu(y)}{\sigma^2(y)}dy+\int_0^x 2\frac{h'(y)}{h(y)}dy}\\
&=& \frac{h^2(x)}{h^2(0)} e^{\int_0^x\frac{2\mu(y)}{\sigma^2(y)}dy}
\ =\ -\frac{h^2(x)}{h'(x)} \frac{1}{h^2(0)}\\ 
&\le& \frac{xh(x)}{c_2}
\ \le\ c_4x x^{1-\widetilde c}/c_2
\end{eqnarray*}
is integrable at infinity because $\widetilde c>3$ and the condition 
\eqref{ex:1.cond} for positive recurrence is met. 

If, for some $\widetilde c\in(1,3]$ 
and an absolutely integrable at infinity function $p(x)$,
\begin{eqnarray*}
\frac{2\mu(x)}{\sigma^2(x)} &=& \frac{\widetilde c}{x}+p(x),
\end{eqnarray*}
then the diffusion process $\{X(t)\}$ is transient by the criterion 
\eqref{ex:1.cond.tran} and the transformed process $\{\widehat X(t)\}$
is null recurrent because in this case
\begin{eqnarray*}
h'(x)\ \sim\ -e^{c_5-\int_1^x\frac{\widetilde c}{y}dy}=-e^{c_5}x^{-\widetilde c}
&\mbox{and}& h(x)\ =\ -\int_x^\infty h'(z)dz\ \sim\ c_6x^{1-\widetilde c},
\end{eqnarray*}
so, the function
\begin{eqnarray*}
e^{\int_0^x\frac{2\widehat\mu(y)}{\widehat\sigma^2(y)}dy} &=& 
-\frac{h^2(x)}{h'(x)}\ \sim\ c_7x^{2-\widetilde c}
\end{eqnarray*}
is not integrable at infinity because $\widetilde c\in(1,3]$
and hence $\{\widehat X(t)\}$ is not positive recurrent by \eqref{ex:1.cond}
but is still recurrent by \eqref{ex:1.cond.rec} because the function
\begin{eqnarray*}
e^{-\int_0^x\frac{2\widehat\mu(y)}{\widehat\sigma^2(y)}dy} &=& 
-\frac{h'(x)}{h^2(x)}\ \sim\ x^{\widetilde c-2}/c_7
\end{eqnarray*}
is not integrable at infinity too.

The other way around, let us consider a recurrent 
diffusion process $\{X(t)\}$, when 
$\tau=\tau_{(-\infty,0]}=\min\{t\ge 0:X(t)\le 0\}$
is finite with probability $1$.
Consider the process $Y(t):=X(t\wedge\tau)$ which is the original process
stopped at time of leaving the positive half line.
Its harmonic function solves \eqref{eq:harmonic.diffusion} 
with $h(0)=1$,\index{Diffusion process!recurrence!harmonic function}
\begin{eqnarray}\label{intr:dif.rec.harm}
h(x) &=& 1+\int_0^x e^{-\int_0^z\frac{2\mu(y)}{\sigma^2(y)}dy}dz,\quad x\ge 0.
\end{eqnarray}
It is an increasing function tending to infinity as $x\to\infty$,
due to the recurrence condition \eqref{ex:1.cond.rec}.
By It\^{o}'s formula, the process $\{h(Y(t))\}$ is a martingale, 
hence we can apply Doob's $h$-transform 
which returns a new stochastic transition kernel 
\begin{eqnarray*}
\widehat P_Y(t,x,dy) &:=& \frac{h(y)}{h(x)}P_Y(t,x,dy).
\end{eqnarray*}
Let us consider a diffusion process $\{\widehat Y(t)\}$ 
with this transition kernel. 
The drift coefficient of $\{\widehat Y(t)\}$ is calculated in 
\eqref{intr:hat.m1}. Since the function $h(x)$ increases, 
$\widehat\mu(x)>\mu(x)$.
The increase of the drift is so strong that 
the process $\{\widehat Y(t)\}$ is transient. Indeed, the function 
\begin{eqnarray*}
e^{-\int_0^x\frac{2\widehat\mu(y)}{\widehat\sigma^2(y)}dy} &=& 
e^{-\int_0^x\frac{2\mu(y)}{\sigma^2(y)}dy-\int_0^x 2\frac{h'(y)}{h(y)}dy}\\
&=& \frac{1}{h^2(x)} e^{-\int_0^x\frac{2\mu(y)}{\sigma^2(y)}dy}\\
&=& \frac{h'(x)}{h^2(x)}\ =\ \Bigl(\frac{-1}{h(x)}\Bigr)'
\end{eqnarray*}
is integrable at infinity because $h(x)\to\infty$ and, therefore, 
the condition \eqref{ex:1.cond.tran} for transience is met,
\begin{eqnarray*}
\int_z^\infty e^{-\int_0^x\frac{2\widehat\mu(y)}{\widehat\sigma^2(y)}dy}dx 
&=& \frac{1}{h(z)}\ <\ \infty.
\end{eqnarray*}

We follow the idea of these calculations related to harmonic functions
and change of measure for diffusion processes
in our tail analysis of invariant measures of Markov chains
in Chapters \ref{ch:power.asymptotics} and \ref{ch:Weibull.asymptotics}.

\subsection{Green function for transient diffusion}
\label{subsec:diffusion.tr}

Let $\{X(t)\}$ be a  transient diffusion on $\R$ (or $\R^+$) 
with the following generator
\begin{align*}
A = \mu(x)\frac{d}{dx}+\frac{\sigma^2(x)}{2}\frac{d^2}{dx^2}.
\end{align*}
We consider a regular diffusion, in the sense of properties (i)-(iii) 
of \cite[Chapter VII.3]{RY1999}. 
For the transience it is sufficient to assume that the following function 
\begin{eqnarray}\label{def:U.dif}
U(x) &:=& \int_x^\infty 
\exp\biggl\{-\int_0^v \frac{2\mu(y)}{\sigma^2(y)} dy\biggr\} dv
\end{eqnarray}
is finite for all $x$, see \eqref{ex:1.cond.tran};
this function solves the homogeneous equation
\begin{eqnarray}\label{eq:U}
AU &=& 0.
\end{eqnarray}
In this case $X(t)\to \infty $ a.s.\
and we are interested in the continuous time analogue of the renewal function,
\[
H_y(x,x+h]\ :=\ \int_0^\infty \P_y\{X(t)\in(x,x+h]\}dt,
\quad h>0.
\] 

By Proposition 1.6 in Revuz and Yor \cite[Ch. VII.1]{RY1999}, the process 
$$
f(X(t))-f(X(0))-\int_0^t Af(X(s))ds
$$ 
is a local martingale for a wide class of functions $f$.  
This suggests the following idea of computation of the renewal measure of $X(t)$.
Fix $x$ and $h$.
Suppose we can find a bounded function $f(z)=f_{h,x}(z)$ 
such that $f(z)\to 0$ as $z\to\infty$ and 
\begin{equation}\label{eq:ode.diffusion}
Af(z)= - \I\{z\in (x,x+h]\}.
\end{equation}
Then the optional stopping theorem and a.s.\ 
convergence $X(t)\to\infty$ as $t\to\infty$ give us an equality
\[
f(y)=\E_y f(X(0)) = \E_y\biggl[\int_0^{\infty} \I\{X(t)\in(x,x+h]\} dt\biggr]
=H_y(x,x+h],
\]
which allows us to analyse $H_y$. 

So, we need to solve the ordinary differential equation \eqref{eq:ode.diffusion}. 
To this end, consider 
\[
m(x)\ :=\ \int_0^x \frac{2dv}{-U'(v)\sigma^2(v)}\ =\
\int_0^x \frac{2}{\sigma^2(v)}
\exp\biggl\{\int_0^v \frac{2\mu(y)}{\sigma^2(y)} dy\biggr\}dv
\]
and then
\[
G_x(z)\ :=\ 
\begin{cases}
U(z)m(z)+\int_z^x U(v)m(dv),&  z\le x,\\
U(z)m(x),&  z>x.
\end{cases}
\]
We have 
\[
\frac{d}{dz}G_x(z)\ =\ 
\begin{cases}
U'(z)m(z),&  z\le x,\\
U'(z)m(x),&  z>x,
\end{cases}
\]
and
\[
\frac{d^2}{dz^2}G_x(z)\ =\ 
\begin{cases}
U''(z)m(z)-2/\sigma^2(z),&  z\le x,\\
U''(z)m(x),&  z>x,
\end{cases}
\]
which together with \eqref{eq:U} implies that
\[
AG_x(z) =
\begin{cases}
-1,& z\le x,\\
0,& z>x, 	
\end{cases}
\]
and hence the function 
\begin{eqnarray}\label{G.dif}
f(z)\ =\ G_{h,x}(z) &:=& G_{x+h}(z)-G_x(z)
\end{eqnarray} 
solves \eqref{eq:ode.diffusion}. 

Alternatively, one can notice that $U(x)$ is the scale function and 
$m(x)$ corresponds to the speed measure and that 
(see~\cite[Chapter VII, Theorem 3.12]{RY1999}) 
\[
AG_x(z) = \frac{d}{dm(z)}\left(\frac{dG_x(z)}{-dU(z)}\right).
\]

Thus, if follows from~\eqref{G.dif} that for $y<x$,
\begin{eqnarray*}
H_y(x,x+h] &=& f(y)\ =\ \int_x^{x+h} U(v)m(dv)\ =\
\int_x^{x+h}\frac{2 U(v)dv}{-U'(v)\sigma^2(v)}.
\end{eqnarray*}
More formally one can obtain the last equality 
from Corollary 3.8 and Exercise 3.20 in \cite[Ch. VII.3]{RY1999}.

If the function $W(v):=U(v)/U'(v)\sigma^2(v)$ is long-tailed 
at infinity, see Definition \ref{def:long.tailed}, then 
we get the following local renewal theorem for $X(t)$ starting at $y$,
\[
H_y(x,x+h]\ \sim \frac{2U(x)}{-U'(x)\sigma^2(x)}h\quad\mbox{as }x\to\infty.
\] 
Assume that
\begin{eqnarray}\label{assump.r.W}
2\mu(x)/\sigma^2(x) &\sim& r(x)\quad\mbox{as }x\to\infty,
\end{eqnarray} 
for some differentiable function $r(x)$ such that the quotient $r'(x)/r^2(x)$
has a limit at infinity. Hence, we can apply L'H\^opital's rule
and the equality $U''=-rU'$ to obtain 
\begin{eqnarray*}
\lim_{x\to\infty}\frac{U(x)}{-U'(x)/r(x)} &=& 
\lim_{x\to\infty}\frac{U'(x)}{-U''(x)/r(x)+U'(x)r'(x)/r^2(x)}\\
&=& \frac{1}{1+\lim_{x\to\infty}r'(x)/r^2(x)}.
\end{eqnarray*}
Therefore, for any fixed $h>0$,
\begin{eqnarray*}
H_y(x,x+h] &\sim& \frac{2}{\sigma^2(x)r(x)}
\frac{1}{1+\lim_{y\to\infty}r'(y)/r^2(y)}h\quad\mbox{as }x\to\infty.
\end{eqnarray*}

\index{Diffusion process!asymptotics for!Green function}
\index{Green function!diffusion process}

\begin{example}\label{ex:dif.h.asy.1x}
If $\mu(x)\sim\mu/x$ and $\sigma^2(x)\to\sigma^2>0$ as $x\to\infty$
with $2\mu>\sigma^2$, then \eqref{assump.r.W} is satisfied with 
$r(x)=2\mu/\sigma^2 x$, $r'(x)/r^2(x)\to -\sigma^2/2\mu$, and we get
\begin{eqnarray*}
H_y(x,x+h] &\sim& \frac{2h}{2\mu-\sigma^2}x\quad\mbox{as }x\to\infty.
\end{eqnarray*}
\end{example}

\begin{example}\label{ex:dif.h.asy.alpha}
If $\mu(x)\sim\mu/x^\alpha$, $\mu>0$, $\alpha\in(0,1)$, 
and $\sigma^2(x)\to\sigma^2>0$ as $x\to\infty$,
then \eqref{assump.r.W} is satisfied with 
$r(x)=2\mu/\sigma^2 x^\alpha$, $r'(x)/r^2(x)\to 0$, and we get
\begin{eqnarray*}
H_y(x,x+h] &\sim& \frac{h}{\mu}x^\alpha
\ \sim\ \frac{h}{\mu(x)}\quad\mbox{as }x\to\infty.
\end{eqnarray*}
Note that this asymptotic behaviour of the renewal function 
does not depend on the diffusion coefficient, 
as if it was a process with constant positive drift.
\end{example}

\subsection{Bessel processes}
\label{sec:intr.bessel}

A Bessel process\index{Bessel process} is an important example 
of diffusion processes with asymptotically zero drift 
whose many probabilistic characteristics can be calculated 
in closed form, which provides some intuition for what can
be expected for Markov chains. The simplest version of 
a Bessel process is defined as the Euclidean norm $\|B^{(d)}(t)\|$ 
of a $d$-dimensional Brownian motion $B^{(d)}(t)$ and solves 
a stochastic differential equation\index{Bessel process}
\begin{eqnarray}\label{eq:bessel}
dX(t) &=& dY(t)+\frac{d-1}{2}\frac{dt}{X(t)}
\ =\ dY(t)+\frac{2\nu+1}{2}\frac{dt}{X(t)},
\end{eqnarray}
where $Y(t)$ is a one-dimensional Brownian motion.
The parameter $\nu=(d-2)/2$ is called the 
{\em index}\index{Bessel process!index} of $X$.
By the same stochastic differential equation we define a Bessel process 
with an arbitrary index $\nu\in\R$. A Bessel process with a non-integer dimension
naturally appears as the norm of a multi-dimensional
Brownian motion in a cone and the dimension is determined by 
the cone geometry,
see Corollary 3 in \cite{DW15} and its proof.

In other words, $X$ is a diffusion with 
drift $(2\nu+1)/2x$ and diffusion coefficient $1$.
The intrinsic property of a Bessel process is that
its drift is singular at the origin which makes it impossible
to apply the results of the last subsection.

The drift of the squared Bessel process $X^2(t)$ at any state equals $2\nu+2$
which gives rise to the following classification, 
see e.g. \cite[Appendix 1.21]{BorSal}.\index{Bessel process!classification}

\begin{itemize}
\item If $\nu>0$ then the process $\{X(t)\}$ is transient 
and there is a unique strong solution to the equation \eqref{eq:bessel}.
The case of index $\nu=0$ corresponds to the process $\sqrt{B_1^2+B_2^2}$
which is null recurrent but the origin is never visited,
hence there is again a unique strong solution to the equation 
\eqref{eq:bessel}.

\item If $-1\le\nu<0$ then the hitting time of the origin from any state
$x>0$ is finite with probability $1$ and has infinite mean.
In the case $-1<\nu<0$, the origin is a repelling (instantaneously reflecting) state for $X$,
so there is a weak solution to the equation
\eqref{eq:bessel} which is not unique.
In the case of index $-1$ the origin is an absorbing state.

\item If $\nu<-1$ then the hitting time of the origin from any state
$x>0$ has finite mean $x^2/|2\nu+2|$ 
and the origin is an absorbing state for $\{X(t)\}$, 
so there is no weak solution to the equation \eqref{eq:bessel}.
\end{itemize}

In the first case where $\nu\ge 0$ the transition density of $\{X(t)\}$
is well known, see e.g. \cite[Appendix 1.21]{BorSal}, 
and given by the equality\index{Bessel process!transition density}
\begin{eqnarray}\label{eq:bessel.trp0}
p_t(x,y) &=& \frac{1}{t}\frac{y^{\nu+1}}{x^\nu} e^{-(x^2+y^2)/2t}I_\nu(xy/t),\\
p_t(0,y) &=& \frac{y^{2\nu+1}}{2^\nu t^{\nu+1}\Gamma(\nu+1)} e^{-y^2/2t},\nonumber
\end{eqnarray}
where $I_\nu(z)$ is a modified Bessel function.
The same formula is still valid for $\nu\in(-1,0)$
if we reflect the process $\{X(t)\}$ each time it reaches the origin.

In the positive recurrent case $\nu<-1$ or in the null recurrent case 
$\nu\in(-1,0)$, if we kill the process at $0$, 
the transition probability density function of $\{X(t)\}$ equals
\begin{eqnarray*}
p_t(x,y) &=& \frac{1}{t}\frac{y^{\nu+1}}{x^\nu} 
e^{-(x^2+y^2)/2t}I_{|\nu|}(xy/t).
\end{eqnarray*}

If $\nu\ge 0$ or $\nu\in(-1,0)$ and the process $\{X(t)\}$ 
is reflected each time it reaches the origin,
the probability density function of $X(t)$ given $X(0)=0$ equals
\begin{eqnarray}\label{eq:bessel.trp1}
p_t(x)\ =\ p_t(0,x) &=& \frac{1}{2^\nu\Gamma(\nu+1)}
\frac{x^{2\nu+1}}{t^{\nu+1}} e^{-x^2/2t}.
\end{eqnarray}
In both cases the probability density function of $X^2(t)/t$ equals
\begin{eqnarray*}
\frac{1}{2^{\nu+1}\Gamma(\nu+1)} x^{\nu} e^{-x/2},
\end{eqnarray*}
which is a gamma density function with mean $2(\nu+1)$ 
and variance $4(\nu+1)$. 

In the transient case $\nu>0$ we can write down the Green function $h_0$ of 
$\{X(t)\}$ in closed form by integration of \eqref{eq:bessel.trp1}:
\index{Bessel process!transience!Green function}
\begin{eqnarray*}
h_0(y) =\ \int_0^\infty p_t(0,y)dt &=& 
\frac{y^{2\nu+1}}{2^\nu\Gamma(\nu+1)}
\int_0^\infty\frac{1}{t^{\nu+1}} e^{-y^2/2t}dt\ =\ \frac{y}{\nu},
\end{eqnarray*}
which indicates what asymptotic behaviour of
the renewal measure we can expect for transient Markov chains with
drift of order $c/x$ at infinity, see Section \ref{sec:renewal}
for results in this direction.

It follows from the representation of the $\alpha$-potential density
$G_\alpha$ of $X$ in \cite[Appendix 1.21]{BorSal} that, for all $x\ge 0$,
\begin{eqnarray*}
h_x(y) =\ \int_0^\infty p_t(x,y)dt &=& \frac{1}{\nu}
\frac{y^{2\nu+1}}{\max(x,y)^{2\nu}},
\end{eqnarray*}
which implies that the first hitting time $\tau_{[0,y]}$
for the compact set $[0,y]$ is finite with probability
\begin{eqnarray}\label{bessel.tau1}
\P_x\{\tau_{[0,y]}<\infty\} &=& \P_x\{X(t)=y\mbox{ for some }t\} 
\nonumber\\
&=& \frac{H_x(y)}{H_y(y)} \ =\ \Bigl(\frac{y}{x}\Bigr)^{2\nu} 
\quad\mbox{for }x>y;
\end{eqnarray}
such kind of results for transient Markov chains are
discussed in Chapter \ref{ch:return.transient}.

For any $\nu$, the function $h(x)=x^{-2\nu}$ is harmonic  
for $\{X(t)\}$ as it solves the equation
\begin{eqnarray*}
\Bigl(\frac{1}{2}\frac{d^2}{dx^2}+\frac{2\nu+1}{2x}\frac{d}{dx}\Bigr)h(x) 
&=& 0.
\end{eqnarray*}
By It\^{o}'s formula, the process $\{h(X(t))\}$ is a local martingale. 
Let $y>0$. If $\nu>0$, then $h(x)$ is bounded on 
$[y,\infty)$ and if $\nu<0$ then it is bounded on $[0,y]$.
So in either case we can apply the optional stopping time theorem
for martingales and to conclude that, for $\nu>0$ and $x>y$,
\begin{eqnarray*}
h(x) &=& h(y)\P_x\{X(t)=y\mbox{ for some }t\}
+h(\infty)\P_x\{X(t)\not=y\mbox{ for all }t\}\\
&=& h(y)\P_x\{\tau_{[0,y]}<\infty\},
\end{eqnarray*}
which agrees with \eqref{bessel.tau1}. 

If $\nu<0$ and the origin is an absorbing state, then, for $x<y$,  
\begin{eqnarray*}
h(x) &=& h(y)\P_x\{X(t)=y\mbox{ for some }t\}
+h(0)\P_x\{X(t)\not=y\mbox{ for all }t\}\\
&=& h(y)\P_x\Bigl\{\sup_{t\ge 0} X(t)\ge y\Bigr\},
\end{eqnarray*}
which implies that
\begin{eqnarray*}
\P_x\Bigl\{\sup_{t\ge 0} X(t)\ge y\Bigr\} &=& 
\frac{h(x)}{h(y)}\ =\ 
\Bigl(\frac{x}{y}\Bigr)^{2|\nu|}.
\end{eqnarray*}
For recurrent Markov chains, the tail distribution of the trajectory 
supremum until the time of the first entry to a neighborhood of the origin 
is described in Theorem \ref{thm:maximum}.

In conclusion, let us establish a link to Markov chains by
sampling the process $\{X(t)\}$ at integer times and getting
a Markov chain $X_n:=X(n)$ in this way;
in null recurrent case we assume reflecting boundary condition. 
This Markov chain is of Lamperti's type with the mean drift $m_1(x)$
and the second moment of jumps $m_2(x)$ satisfying the relations
\begin{eqnarray}\label{eq:bessel.1}
m_1(x)\ \sim\ \frac{\nu+1/2}{x}\ =:\ \frac{c}{x} 
&\mbox{and}& m_2(x)\ \to\ 1
\quad\mbox{as }x\to\infty.
\end{eqnarray}
Indeed, it follows from \eqref{eq:bessel.trp0} that
\begin{eqnarray*}
\E_x X(1) &=& 
\int_0^\infty \frac{y^{\nu+2}}{x^\nu} e^{-(x^2+y^2)/2}I_\nu(xy)dy\\
&=& \frac{e^{-x^2/2}}{x^\nu} \int_0^\infty y^{\nu+2} e^{-y^2/2}I_\nu(xy)dy\\
&=& \frac{e^{-x^2/2}}{x^\nu} \frac{\Gamma(\nu+3/2)}{\frac{x}{2}\Gamma(\nu+1)}
e^{x^2/4}2^{\nu/2} M_{-\nu/2-1,\nu/2}(x^2/2),
\end{eqnarray*}
where $M_{\cdot}(\cdot)$ is the Whittaker function, 
see \cite[Formula 6.643(2)]{GR}. As $x\to\infty$,
\begin{eqnarray*}
M_{-\nu/2-1,\nu/2}(x^2/2) &=& \frac{\Gamma(\nu+1)}{\Gamma(\nu+3/2)}
e^{x^2/4}(x^2/2)^{\nu/2+1}\Bigl(1+\frac{2\nu+1}{2x^2}+O(1/x^4)\Bigr),
\end{eqnarray*}
which gives
\begin{eqnarray*}
\E_x X(1) &=& x\Bigl(1+\frac{2\nu+1}{2x^2}+O(1/x^4)\Bigr)
\quad\mbox{as }x\to\infty,
\end{eqnarray*}
which in its turn yields the first relation in \eqref{eq:bessel.1}.
In a similar way we conclude the asymptotic behaviour of
higher moments of jumps, for any fixed $j\ge1$,
\begin{equation}\label{bessel.higher}
\E_x X^{2j}(1)\ =\ x^{2j}+2j(\nu+j)x^{2j-2}+O(x^{2j-4})
\quad\mbox{as }x\to\infty.
\end{equation}
Choosing here $j=1$ and using the formula for the fist moment of $X(1)$
one gets the second convergence in \eqref{eq:bessel.1}.

If the Bessel process $\{X(t)\}$ is transient or null recurrent,
that is, if $\nu>-1$, then it follows from the distribution property 
of the Bessel process $\{X(t)\}$ discussed above that, for all $n$,
$X_n^2/n$ has a $\Gamma$-distribution with mean $2(\nu+1)$ 
and variance $4(\nu+1)$. In Sections \ref{sec:gamma}
and \ref{sec:pre-st.null} we discuss convergence of $X_n^2/n$ to 
a $\Gamma$-distribution for a general transient or null-recurrent 
Markov chain with asymptotic drift of order $c/x$.

\section{General approach to Markov chains with asymptotically zero drift
and plan of the book}
\label{sec:plan}

One of the most popular examples of Markov chains with asymptotically 
zero drift is a driftless random walk conditioned to stay positive. 
This process is an $h$-transform of a random walk killed at leaving $\R^+$. 
If the second moment of the original random walk is finite 
then the transformed process has drift of order $1/x$, that is,
$xm_1(x)\to c_1>0$. But the second moment of the transformed process 
is finite if and only if the third moment of the original walk is so, 
see calculations in Section~\ref{sec:h.x.cond.walks}. 
Therefore, Lamperti's\index{Lamperti} 
criterion for transience is not always applicable to this chain.

This observation motivated us to look for appropriate conditions 
for transience, null-recurrence and positive recurrence in terms 
of truncated moments and tail probabilities of jumps $\xi(x)$. 
For any $s>0$ we denote $s$-truncation of the $k$th moment 
of jump at state $x$ by\index{Markov chain!truncated moments of jumps}
\begin{eqnarray*}
m_k^{[s]}(x) &:=& \E\{\xi^k(x);\ |\xi(x)|\le s\}.
\end{eqnarray*}
Another reason for considering truncated moments comes from 
the case where the drift function decays slower than $1/x$,
say as $1/x^\beta$ with $\beta$ between $0$ and $1$.
In that case it is not practical to assume boundedness or
even existence of full second moment of jumps 
whereas an appropriate restriction on the growth of 
a truncated second moment is rational, see e.g. Section \ref{sec:h.x.lln}.

In Chapter~\ref{ch:classification} we introduce a classification 
of Markov chains with asymptotically zero drift, which relies on 
relations between $m_1^{[s(x)]}$ and $m_2^{[s(x)]}$. 
Additional assumptions are expressed in terms of truncated moments
of higher orders and tail probabilities of jumps. 
Another, more important, contrast to previous results 
on recurrence/transience is the fact that we do not use concrete
Lyapunov test functions (like $x^2$, $\log^a x$ or $x^2\log x\log\log x$). 
Instead, we construct an abstract Lyapunov function 
which is motivated by the harmonic function of diffusion process with drift 
$m_1(x)$ and diffusion coefficient $m_2(x)$, 
see Section \ref{sec:introduction} above.

Asymptotic behaviour of transient Markov chains and
tail analysis of recurrent ones is discussed in Chapters
\ref{ch:return.transient}--\ref{ch:asy.renewal} 
and \ref{ch:power.asymptotics}--\ref{ch:Weibull.asymptotics} respectively.
In Chapter \ref{ch:change},
motivated by exponential change of measure approach suggested 
by Cram\'er in 1920's for study of large deviations of sums 
of independent random variables in the context of risk processes,
we suggest the following general strategy for study
of positive recurrent Markov chains with asymptotically zero drift:
\begin{itemize}
\item Firstly, apply an appropriate Doob's $h$-transform to $\{X_n\}$ killed at
time of entry to the half-line $(-\infty,\widehat x]$ 
for some $\widehat x\in\R$ in order to change the sign
of the drift from negative to positive one so that we get a transition 
kernel that generates a transient embedded Markov chain;
with necessity an appropriate change of measure is generated by a 
subexponential function, either regularly varying or Weibullian-type
at infinity;
\item Secondly, apply limit results to a transient Markov chain obtained;
\item Thirdly, apply the inverse change of measure which makes
it possible to identify tail and local asymptotics of both stationary 
and pre-stationary distributions of the original
positive recurrent Markov chain.
\end{itemize}

In Chapter~\ref{ch:asymp.hom} we show that our approach also works
for Markov chains with asymptotically negative drift bounded away from zero.
We consider asymptotically homogeneous in space Markov chains, 
that is, Markov chains with jumps satisfying 
$\xi(x)\Rightarrow\xi$ as $x\to\infty$. 
This means that far away from the origin one can approximate 
$\{X_n\}$ by a random walk which makes it natural to apply
an exponential change of measure similarly to how it is done 
for sums of independent random variables. 
We study the tail asymptotic behaviour 
of the stationary and pre-stationary distributions of $\{X_n\}$ 
in the case where the limiting random variable $\xi$
has negative mean and satisfies the Cram\'er condition. 
It turns out that the tail behaviour of these distributions depends 
on the rate of convergence of $\xi(x)$ to $\xi$.  

In the last chapter we consider some important applications of our results.
Processes with asymptotically zero drift naturally appear in various 
stochastic models like random billiards, see Menshikov 
et al.~\cite{MVW08},\index{Menshikov}
and random polymers, see Alexander\index{Alexander} \cite{Alex11}, 
Alexander and Zygouras\index{Zygouras}
\cite{AZ09}, De Coninck et al.~\cite{DDH08}).\index{De Coninck}

Such chains appear when we study critical 
and near-critical branching processes. In critical branching processes 
one typically observes a linearly growing second moment of jumps, 
but considering the square root of the process one gets bounded
second moments and decreasing to zero drift. 
Then we can apply our theorems to this transformation. 
As a result we get limit theorems for population size-dependent processes
with migration of particles. To the best of our knowledge, 
there are no papers in the literature, where a combination of size 
dependence and migration has been considered. 

We have also found out
that processes with asymptotically zero drift can be used in the study 
of risk processes with reserve-dependent premium rate. 
More precisely, we have derived upper and lower bounds for ruin 
probabilities in the case when the premium rate approaches 
from above---as the risk reserve growths---the critical value 
for the model with constant rate. 

Besides these two main examples we consider also
random walk conditioned to stay positive and reflected random walk.
\chapter{Lyapunov functions and classification of Markov chains}
\chaptermark{Classification of Markov chains}
\label{ch:classification}

As one can see from results for diffusion processes
in Section \ref{sec:introduction}, their classification heavily relies 
on the asymptotic behaviour of the ratio $2m_1(x)/m_2(x)$ at infinity.
Roughly speaking,
\begin{itemize}
 \item If $2m_1(x)/m_2(x) \le -(1+\varepsilon)/x$ for all sufficiently large $x$,
then some neighborhood of zero is positive recurrent;
\item If $2m_1(x)/m_2(x) \le (1-\varepsilon)/x$ for all sufficiently large $x$,
then some neighborhood of zero is recurrent;
 \item If $2m_1(x)/m_2(x)\ge (1+\varepsilon)/x$ for all sufficiently large $x$,
then any compact set is transient.
\end{itemize}

For diffusion processes, the necessary and sufficient conditions for positive 
recurrence/recurrence/transience involving the ratio $2m_1(x)/m_2(x)$ 
are available, see \eqref{ex:1.cond}--\eqref{ex:1.cond.tran}.
For Markov chains, similar necessary and sufficient conditions
in terms of the ratio $2m_1(x)/m_2(x)$ are not available 
as it is for diffusion processes.

In this chapter we introduce criteria for transience,
recurrence and positive recurrence of discrete time Markov chains
by constructing Lyapunov functions which depend on the ratio of truncated moments
of the chain which are motivated by \eqref{ex:1.cond}--\eqref{ex:1.cond.tran}.
Let us recall standard sufficient conditions for positive recurrence, recurrence, and 
transience in terms of test functions.
\index{Positive recurrence!drift criterion}

\begin{theorem}[{\cite[Theorem 11.0.1]{MT}}]\label{cr.test.pos.rec}
Let $L(x)$ be a non-negative test function such that,
for some $x_*$ and $\varepsilon>0$,
\begin{eqnarray}\label{V.eps.1}
\E\{L(X_1)-L(x)\mid X_0=x\} &\le& -\varepsilon
\quad\mbox{for all }x>x_*,
\end{eqnarray}
and let
\begin{eqnarray}\label{V.eps.1.pre}
\E\{L(X_1)\mid X_0=x\} &<& \infty
\quad\mbox{for all }x\le x_*.
\end{eqnarray}
Then the set $(-\infty, x_*]$ is positive recurrent.
\end{theorem}

\index{Recurrence!drift criterion}
\begin{theorem}[{\cite[Theorem 8.0.2]{MT}}]\label{cr.test.rec}
Let $L(x)$ be a non-negative unbounded at infinity test function such that,
for some $x_*$,
\begin{eqnarray}\label{V.eps.1.rec}
\E\{L(X_1)-L(x)\mid X_0=x\} &\le& 0\quad\mbox{for all }x>x_*.
\end{eqnarray}
Then the set $(-\infty, x_*]$ is recurrent.
\end{theorem}

\index{Transience!drift criterion}
\begin{theorem}[{\cite[Theorem 8.0.2]{MT}}]\label{cr.test.tran}
Let $L(x)$ be a non-negative bounded test function such that, for some $x_*$,
\begin{eqnarray}\label{V.eps.1.tran}
\E\{L(X_1)-L(x)\mid X_0=x\} &\le& 0\quad\mbox{for all }x>x_*.
\end{eqnarray}
Then the set $(-\infty, x_*]$ is transient.
\end{theorem}

\section{Reference drift function}
\label{sec:step.size}

In this chapter, $r(x)>0$ is a reference drift function.\index{Reference drift function}
It is always assumed to be a decreasing continuous function which is
non-integrable at infinity, that is, for $x\ge 0$,
\begin{eqnarray}\label{def.of.R}
R(x)\ :=\ \int_0^x r(y)dy &\to& \infty\quad\mbox{as }x\to\infty;
\end{eqnarray}
hereinafter we define $R(x)=0$ for $x<0$.
The function $R(x)$ is concave on the positive half-line
because $r(x)$ is assumed decreasing. 
Therefore, for all $h>-xr(x)$,
\begin{eqnarray}\label{R.h.above}
R(x+h/r(x)) &\le& R(x)+R'(x)h/r(x)\nonumber\\
&=& R(x)+h.
\end{eqnarray}

If, in addition, $r(x)$ is differentiable and, for some $c>0$,
\begin{eqnarray}\label{cond.on.r}
0\ \ge\ r'(x) &\ge& -c r^2(x)
\quad\mbox{for all }x\ge 0,
\end{eqnarray}
then, for all $x\ge 0$ and $h>0$,
\begin{eqnarray*}
\frac{1}{r(x)}-\frac{1}{r(x+h/r(x))} &=& \int_x^{x+h/r(x)}\frac{r'(y)}{r^2(y)}dy\\
&\ge& -c\int_x^{x+h/r(x)}dy\ =\ -c\frac{h}{r(x)}.
\end{eqnarray*}
Therefore,
\begin{eqnarray}\label{r.h.below}
r(x+h/r(x)) &\ge& \frac{r(x)}{1+ch},\quad h>0.
\end{eqnarray}
Similarly, for $h\in(0,1/c)$ and $x$ such that $x-h/r(x)\ge 0$,
\begin{eqnarray}\label{r.h.above}
r(x-h/r(x)) &\le& \frac{r(x)}{1-ch}.
\end{eqnarray}
The lower bound \eqref{r.h.below} implies that, for all $h>0$,
\begin{eqnarray}\label{R.h.below}
R(x+h/r(x)) &=& R(x)+\int_x^{x+h/r(x)} r(y)dy\nonumber\\
&\ge& R(x)+\frac{h}{r(x)} r(x+h/r(x))\nonumber\\
&\ge& R(x)+\frac{h}{1+ch}.
\end{eqnarray}
Together with the upper bound \eqref{R.h.above} it gives a two-sided bound
\begin{eqnarray}\label{R.h.plus}
R(x)+\frac{h}{1+ch} &\le& R(x+h/r(x))\ \le\ R(x)+h
\quad\mbox{for all }h>0.
\end{eqnarray}
Similarly,
\begin{eqnarray}\label{R.h.minus}
R(x)-\frac{h}{1-ch} &\le& R(x-h/r(x))\ \le\ R(x)-h,
\end{eqnarray}
where the first inequality is valid for $h\in(0,1/c)$,
while the second one for $h\in(0,xr(x))$.

So, $1/r(x)$ is a natural $x$-step
responsible for the constant increase of the function $R(x)$.
Moreover, \eqref{R.h.plus} and \eqref{R.h.minus} imply that,
for any increasing function $s(x)$ of order $o(1/r(x))$,
\begin{eqnarray}\label{R.h.small}
R(x\pm s(x)) &=& R(x)+o(1)\quad\mbox{as }x\to\infty.
\end{eqnarray}
Notice also that \eqref{r.h.below} and \eqref{r.h.above} yield
a similar relation for $r(x)$,
\begin{eqnarray}\label{r.h.small}
r(x\pm s(x)) &\sim& r(x)\quad\mbox{as }x\to\infty.
\end{eqnarray}

\section{Positive recurrence}\label{sec:posrec}
\subsection{Positive recurrence motivated by diffusion processes}

In this section we are interested in sufficient conditions under which 
the set $(-\infty,x_*]$ is positive recurrent for some $x_*$, that is, 
$\E_x\tau_{(-\infty,x_*]}<\infty$ for all $x\le x_*$.

Conditions below are formulated in terms of truncated moments of jumps, 
\begin{eqnarray*}
m_k^{[s]}(x) &:=& \E\{\xi^k(x);\ |\xi(x)|\le s\}.
\end{eqnarray*}
Let $x_0$ be such that
\begin{eqnarray}\label{r-cond.5}
\frac{2m_1^{[x]}(x)}{m_2^{[x]}(x)} &\le& -r(x)\quad\mbox{for all }x>x_0.
\end{eqnarray}
For $r(x)$ decreasing not too fast---roughly speaking, if $r(x)>1/x$---this 
means that the drift towards the origin dominates the diffusion
and the corresponding Markov chain $X$ is positive recurrent.

In the theorem below it is shown that---similarly to diffusion
processes---the chain $\{X_n\}$ is positive recurrent provided 
\begin{eqnarray}\label{integr.cond}
\mbox{the function}\quad
\frac{1}{b(x)}e^{-R(x)} &=&
\frac{1}{b(x)}e^{-\int_0^x r(y)dy}\quad\mbox{is integrable},
\end{eqnarray}
where $b(x)>0$ is a differentiable function such that
\begin{eqnarray}\label{second.mu2}
\liminf_{x\to\infty} \frac{m_2^{[x]}(x)}{b(x)} &>& 0.
\end{eqnarray}
For Markov chains, we also need to impose some technical conditions
on $r(x)$ and on the function
\begin{eqnarray*}
W(x) &:=& e^{R(x)}\int_x^\infty \frac{1}{b(y)}e^{-R(y)}dy,
\end{eqnarray*}
which is a well defined function due to \eqref{integr.cond}.

In the next theorem sufficient conditions are given
that guarantee that the test function
\begin{eqnarray}\label{def.L.pos.rec}
L(x) &:=& \int_0^x W(y) dy,\quad x>0,
\end{eqnarray}
and $L(x)=0$ on $\R^-$, is appropriate for application of Theorem \ref{cr.test.pos.rec}.
In particular, it agrees with the case $r(x)\equiv\varepsilon>0$
where the most natural choice of the test function is a linear one;
and with the case $r(x)=c/x$ where the most effective test function is $x^2$.
\index{Markov chain!condition for!positive recurrence}

\begin{theorem}\label{thm:pos.recurrence}
Let the drift condition \eqref{r-cond.5} hold with some decreasing function 
$r(x)>0$ such that the conditions \eqref{integr.cond} 
and \eqref{second.mu2} are satisfied and
\begin{eqnarray}\label{r-cond.5.pre}
\E\{\xi(x)W(\xi(x));\ \xi(x)>0\} &<& \infty\quad\mbox{for all }x.
\end{eqnarray}
Let the following integrability conditions on positive jumps hold,
\begin{eqnarray}\label{uni.integr.L.3}
\E\{\xi^3(x);\ \xi(x)\in(0,x]\} &=& o(x^2/W(x)),\\
\label{uni.integr.L.W}
\E\{\xi(x)W(\xi(x));\ \xi(x)>x\} &\to& 0\quad\mbox{as }x\to\infty.
\end{eqnarray}
Assume that the function $W(x)$ is increasing and convex, and satisfies
the following conditions, for some constants $c_1$, $c_2$,
\begin{eqnarray}\label{W.le1}
W(2x) &\le& c_1W(x)\quad\mbox{for all }x>0,\\
\label{W.le2}
|W'(x+y)-W'(x)| &\le& c_2\frac{W(x)}{x^2}|y|\quad\mbox{for all }x>0,\ y\in[-x/2,x].
\end{eqnarray}
Then there exists an $x_*$ such that
the set $(-\infty,x_*]$ is positive recurrent.
\end{theorem}

The conditions \eqref{uni.integr.L.3} and \eqref{uni.integr.L.W}
are fulfilled if, for example, the function $x^2/W(x)$ increases and
\begin{eqnarray}\label{uni.integr.L}
\mbox{the family }\bigl\{\xi^+(x)W(\xi^+(x)),\ x\ge 0\bigr\}
\quad\mbox{is uniformly integrable};
\end{eqnarray}
justification follows from Lemmas \ref{l:g.fin} and \ref{l:p.V.ui.o}.
\index{Markov chain!condition for!positive recurrence}

\begin{corollary}\label{cor:posrec}
Let, for some $\varepsilon>0$ and $x_0>0$,
\begin{eqnarray*}
\frac{2m_1^{[x]}(x)}{m_2^{[x]}(x)} &\le&
-\frac{1+\varepsilon}{x}\quad\mbox{for all }x>x_0,
\end{eqnarray*}
and let $\E\{\xi^2(x);\xi(x)>0\} < \infty$ for all $x$.
Let the truncated second moments of jumps $\E\{\xi^2(x);|\xi(x)|\le x\}$
be bounded away from zero,
\begin{eqnarray}\label{uni.integr.L.3.2}
\E\{\xi^3(x),\ \xi(x)\in[0,x]\} &=& o(x),\\
\label{uni.integr.L.W.2}
\E\{\xi^2(x),\ \xi(x)>x\} &\to& 0\quad\mbox{as }x\to\infty.
\end{eqnarray}
Then there exists an $x_*$ such that
the set $(-\infty,x_*]$ is positive recurrent.
\end{corollary}

Notice that both \eqref{uni.integr.L.3.2} and \eqref{uni.integr.L.W.2} hold
provided the family of random variables $\{(\xi^+(x))^2,\ x>0\}$
is uniformly integrable.

\begin{theopargself}
\begin{proof}[of Corollary \ref{cor:posrec}]
It follows from Theorem \ref{thm:pos.recurrence} 
if we take $r(x) =\frac{1+\varepsilon}{1+x}$ for $x>0$ and $b(x)=1$, then 
\begin{eqnarray*}
R(x) &=& (1+\varepsilon)\log(1+x),\\
e^{-R(x)} &=& 1/(1+x)^{1+\varepsilon},\\
W(x) &=& (1+x)/\varepsilon.
\end{eqnarray*}
This leads to the test function $L(x)=((1+x)^2-1)/2\varepsilon$ for $x>0$.
Notice in passing that then $L(x)=x^2\I\{x>0\}$ is also an
appropriate test function.
\qed\end{proof}
\end{theopargself}

Notice that the last corollary relates to a quadratic Lyapunov function
and its assumptions on jumps are too restrictive compared to
the classical Lamperti's criterion that guarantees positive recurrence
of the set $(-\infty,x_0]$ under the condition
$2xm_1(x)+m_2(x)\le -\varepsilon$ for $x>x_0$ only.
On the other hand, Corollary \ref{cor:posrec} imposes no conditions 
on the left tail distribution of $\xi(x)$ below the level $-x$.

Let $\log_{(m)}x$ denote the $m$th iteration of the logarithm of $x$,
$\log_{(m)}x=\log\log_{(m-1)}x$.
\index{Markov chain!condition for!positive recurrence}

\begin{corollary}\label{cor:posrec.log}
Let, for some $m\in\mathbb N$ and $\varepsilon>0$,
\begin{eqnarray*}
\frac{2m_1^{[x]}(x)}{m_2^{[x]}(x)} &\le&
-\frac{1}{x}-\frac{1}{x\log x}-\ldots-\frac{1}{x\log x\cdot\ldots\cdot\log_{(m-1)}x}
-\frac{1+\varepsilon}{x\log x\cdot\ldots\cdot\log_{(m)}x}
\end{eqnarray*}
for all sufficiently large $x$, and let
\begin{eqnarray*}
\E\{\xi^2(x)\log\xi(x)\cdot\ldots\cdot\log_{(m)}\xi(x);\ 
\log_{(m)}\xi(x)>0\} &<& \infty
\quad\mbox{ for all }x.
\end{eqnarray*}
Let the truncated second moment
$\E\{\xi^2(x);|\xi(x)|\le x\}$ be bounded away from zero, let
\begin{eqnarray}\label{le.s.m3.posrec}
\E\{\xi(x)^3;\ \xi(x)\in[0,x]\} &=&
o\Bigl(\frac{x}{\log x\cdot\ldots\cdot\log_{(m)}x}\Bigr),
\end{eqnarray}
and let
\begin{eqnarray}\label{le.s.m1.posrec}
\E\{\xi^2(x)\log\xi(x)\cdot\ldots\cdot\log_{(m)}\xi(x);\ \xi(x)>x\} &\to& 0.
\end{eqnarray}
Then there exists an $x_*$ such that
the set $(-\infty,x_*]$ is positive recurrent.
\end{corollary}

%This result was first established by Menshikov et al. \cite{AMI}
%under the condition that moments of order $x^2\log^{2+\delta}x$ are bounded.
Notice that both \eqref{le.s.m3.posrec} and \eqref{le.s.m1.posrec} hold
provided the family
\begin{eqnarray*}
\{\xi^2(x)\log\xi(x)\cdot\ldots\cdot\log_{(m)}\xi(x)\I\{\log_{(m)}\xi(x)>0\},\ x>0\}
\mbox{ is uniformly integrable.}
\end{eqnarray*}

\begin{theopargself}
\begin{proof}[of Corollary \ref{cor:posrec.log}]
Let $x=e^{(m)}$ be a solution to the equation $\log_{(m)}x=1$.
Consider
\begin{eqnarray*}
r(x) &:=& \Bigl(\frac{1}{y}+\frac{1}{y\log y}+\ldots
+\frac{1}{y\log y\cdot\ldots\cdot\log_{(m-1)}y}
+\frac{1+\varepsilon}{y\log y\cdot\ldots\cdot\log_{(m)}y}\Bigr)\Big|_{y=e^{(m)}+x}
\end{eqnarray*}
and $b(x)=1$; then
\begin{eqnarray*}
R(x) &=& \Bigl(\log y+\log\log y+\ldots+\log_{(m)}y
+(1+\varepsilon)\log_{(m+1)}y\Bigr)\Big|_{y=e^{(m)}+x}\\
&& -\bigl(e^{(m-1)}+e^{(m-2)}+\ldots+1\bigr)\\
e^{-R(x)} &=&
\frac{e^{(m)}e^{(m-1)}\cdot\ldots\cdot 1}{y\cdot\log y\cdot\ldots\cdot\log_{(m-1)}y\cdot
\log_{(m)}^{1+\varepsilon}y}\Big|_{y=e^{(m)}+x},\\
W(x) &=& \frac1\varepsilon y\log y\cdot\ldots\cdot
\log_{(m-1)}y\cdot\log_{(m)}y\Big|_{y=e^{(m)}+x},\\
L(x) &\sim& \frac1{2\varepsilon} x^2\log x\cdot\ldots\cdot\log_{(m-1)}x\cdot\log_{(m)}x.
\end{eqnarray*}
\qed\end{proof}
\end{theopargself}

The next corollary deals with the case when the second moment of jumps
is vanishing at infinity.\index{Markov chain!condition for!positive recurrence}

\begin{corollary}\label{cor:posrec.small}
Let, for some $\alpha>0$, $c_1$, $c_2>0$, and $x_0>0$,
\begin{eqnarray*}
m_1^{[x]}(x) &\le& -c_1/x^{1+\alpha}\quad\mbox{for all }x>x_0,\\
m_2^{[x]}(x) &\sim& c_2/x^\alpha\quad\mbox{as }x\to\infty,\\
\E\{\xi^{2+\alpha}(x);\ \xi(x)>0\} &<& \infty
\quad\mbox{ for all }x.
\end{eqnarray*}
Let
\begin{eqnarray}\label{uni.integr.L.3.2.small}
\E\{\xi^3(x);\ \xi(x)\in[0,x]\} &=& o(x^{1-\alpha}),\\
\label{uni.integr.L.W.2.small}
\E\{\xi^{2+\alpha}(x);\ \xi(x)>x\} &\to& 0\quad\mbox{as }x\to\infty,
\end{eqnarray}
If $2c_1/c_2>1+\alpha$, then there exists an $x_*$ such that
the set $(-\infty,x_*]$ is positive recurrent.
\end{corollary}

In the case $\alpha\in(0,1)$, both \eqref{uni.integr.L.3.2.small}
and \eqref{uni.integr.L.W.2.small} hold provided
the family of random variables $\{(\xi^+(x))^{2+\alpha},\ x>0\}$
is uniformly integrable.

\begin{theopargself}
\begin{proof}[of Corollary \ref{cor:posrec.small}]
It follows if we take $c\in(1+\alpha,2c_1/c_2)$, $r(x)=\frac{c}{1+x}$
for $x>0$ and $b(x)=1/(1+x)^\alpha$,
then $R(x)=c\log(1+x)$, $e^{-R(x)}=1/(1+x)^c$,
\begin{eqnarray*}
W(x) &=& (1+x)^c\int_x^\infty \frac{(1+y)^\alpha}{(1+y)^c}dy
\ =\ \frac{(1+x)^{\alpha+1}}{\alpha-c+1},
\end{eqnarray*}
and
$$
L(x)=\frac{(1+x)^{2+\alpha}-1}{(\alpha-c+1)(2+\alpha)}.
$$
\qed\end{proof}
\end{theopargself}

The advantage of Theorem \ref{thm:pos.recurrence} is that it covers all
functions considered in the corollaries above in a unified way; 
the main condition \eqref{integr.cond}
is motivated by the existence condition \eqref{ex:1.cond} for stationary density
of a diffusion process. But at the same time this link to diffusion processes
results in necessity of finite second moments
which is natural in Corollaries \ref{cor:posrec} and \ref{cor:posrec.log}
while there are other examples where the existence of second moments
of jumps is clearly excessive.
In the next subsection we discuss amended moment conditions for drifts
like $-1/x^\alpha$, $0<\alpha<1$, that may be characterised
by the convergence $xm_1(x)\to\infty$ as $x\to\infty$.

\begin{theopargself}
\begin{proof}[of Theorem \ref{thm:pos.recurrence}]
We consider the test function \eqref{def.L.pos.rec}
for which we need to show \eqref{V.eps.1} and \eqref{V.eps.1.pre}.
Since $W(x)$ is increasing,
\begin{eqnarray}\label{L.lexW.pos}
L(x) &\le& xW(x)\quad\mbox{for all }x>0,
\end{eqnarray}
hence \eqref{V.eps.1.pre} follows from the condition \eqref{r-cond.5.pre},
and it remains to show \eqref{V.eps.1}.
By the construction, $L'(x)=W(x)$ and
\begin{eqnarray}\label{L.2prime}
L''(x) &=& W'(x)\ =\ r(x)W(x)-1/b(x).
\end{eqnarray}

Let us prove that the mean drift of $L(x)$ is negative
and bounded away from zero for all sufficiently large $x$.
First we analyse Taylor's expansion for the function $L$,
with the Lagrange form of the remainder, here $x$, $x+y>0$:
\begin{eqnarray}\label{Taylor.posrec}
L(x+y)-L(x) &=& L'(x)y+L''(x+\theta y)y^2/2\nonumber\\
&=& W(x)y+W'(x+\theta y)y^2/2,
\end{eqnarray}
where $0\le\theta=\theta(x,y)\le 1$. Since $W(x)$ is assumed convex,
$W'$ is increasing, hence, for all $y\in[-x,0]$,
\begin{eqnarray}\label{Taylor.posrec.neg}
L(x+y)-L(x) &\le& W(x)y+W'(x)y^2/2\nonumber\\
&=& W(x)y+r(x)W(x)\frac{y^2}{2}-\frac{y^2}{2b(x)},
\end{eqnarray}
as follows from \eqref{L.2prime}.
Next, by the condition \eqref{W.le2}, for $y\in[0,x]$,
\begin{eqnarray}\label{W.prpr.W}
W'(x+\theta y) &\le& W'(x)+c_2\frac{W(x)}{x^2}y.
\end{eqnarray}
Substituting this into \eqref{Taylor.posrec} we get,
for all $y\in[0,x]$,
\begin{eqnarray}\label{Taylor.posrec.pos}
L(x+y)-L(x) &\le& W(x)y+r(x)W(x)\frac{y^2}{2}-\frac{y^2}{2b(x)}
+c_3\frac{W(x)}{x^2}y^3.
\end{eqnarray}
Using the fact that $L$ is increasing and
the inequalities \eqref{L.lexW.pos} and \eqref{W.le1}, we deduce that
\begin{eqnarray}\label{Taylor.posrec.far}
L(x+y) &\le& L(2y)\ \le\ 2yW(2y)\ \le\ 2c_1yW(y)\quad\mbox{for all }y>x.
\end{eqnarray}

Now we are ready to bound the mean drift of $\{L(X_n)\}$.
We start with the following upper bound
\begin{eqnarray}\label{estimate.1}
\E L(x+\xi(x))-L(x) &\le& \E\{L(x+\xi(x))-L(x);\ \xi(x)\ge -x\}\nonumber\\
&\le& \E\{L(x+\xi(x))-L(x);\ \xi(x)\in[-x,0]\}\nonumber\\
&&\hspace{8mm}+\E\{L(x+\xi(x))-L(x);\ \xi(x)\in[0,x]\}\nonumber\\
&&\hspace{20mm}+\E\{L(x+\xi(x));\ \xi(x)>x\}.
\end{eqnarray}
It follows from \eqref{Taylor.posrec.neg} that
\begin{eqnarray}\label{estimate.1.1}
\lefteqn{\E\{L(x+\xi(x))-L(x);\ \xi(x)\in[-x,0]\}}\nonumber\\
&\le& W(x) \E\{\xi(x);\ \xi(x)\in[-x,0]\}
+\frac12 r(x)W(x)\E\{\xi^2(x);\ \xi(x)\in[-x,0]\}\nonumber\\
&&\hspace{45mm}-\frac{1}{2b(x)} \E\{\xi^2(x);\ \xi(x)\in[-x,0]\}.
\end{eqnarray}
It follows from \eqref{Taylor.posrec.pos} that
\begin{eqnarray}\label{estimate.1.2}
\lefteqn{\E\{L(x+\xi(x))-L(x);\ \xi(x)\in[0,x]\}}\nonumber\\
&\le& W(x) \E\{\xi(x);\ \xi(x)\in[0,x]\}
+\frac12 r(x)W(x)\E\{\xi^2(x);\ \xi(x)\in[0,x]\}\nonumber\\
&& -\frac{1}{2b(x)} \E\{\xi^2(x);\ \xi(x)\in[0,x]\}
+c_3\frac{W(x)}{x^2} \E\{\xi^3(x);\ \xi(x)\in[0,x]\}\nonumber\\
&\le& W(x) \E\{\xi(x);\ \xi(x)\in[0,x]\}
+\frac12 r(x)W(x)\E\{\xi^2(x);\ \xi(x)\in[0,x]\}\nonumber\\
&&\hspace{20mm} -\frac{1}{2b(x)} \E\{\xi^2(x);\ \xi(x)\in[0,x]\}+o(1)
\quad\mbox{as }x\to\infty,
\end{eqnarray}
due to the condition \eqref{uni.integr.L.3}. Finally, it follows from
\eqref{Taylor.posrec.far} by the condition \eqref{uni.integr.L.W} that
\begin{eqnarray}\label{estimate.1.3}
\E\{L(x+\xi(x));\ \xi(x)>x\} &\le&
2c_1\E\{\xi(x)W(\xi(x));\ \xi(x)>x\}\nonumber\\
&\to& 0\quad\mbox{as }x\to\infty.
\end{eqnarray}
Substituting the upper bounds \eqref{estimate.1.1}--\eqref{estimate.1.3}
into \eqref{estimate.1} we deduce that
\begin{eqnarray*}
\lefteqn{\E\{L(x+\xi(x))-L(x)\}}\\
&\le& W(x)\E\{\xi(x);\ |\xi(x)|\le x\}
+\frac12 r(x)W(x)\E\{\xi^2(x);\ |\xi(x)|\le x\}\\
&&\hspace{55mm}-\frac{1}{2b(x)}\E\{\xi^2(x);\ |\xi(x)|\le x\}+ o(1)\\
&=& W(x)\frac{m_2^{[x]}(x)}{2}
\Bigl(\frac{2m_1^{[x]}(x)}{m_2^{[x]}(x)}+r(x)\Bigr)
-\frac{1}{2b(x)}m_2^{[x]}(x)+ o(1)\\
&\le& -\frac{1}{2b(x)}m_2^{[x]}(x)+ o(1)
\quad\mbox{as }x\to\infty,
\end{eqnarray*}
owing to \eqref{r-cond.5}. Then \eqref{second.mu2} implies
\eqref{V.eps.1} for all sufficiently large $x$ and the proof is complete.
\qed\end{proof}
\end{theopargself}

\subsection{Non-diffusive positive recurrence in the case
$xm_1(x)\to-\infty$}

If the drift approaches zero value at rate slower than $1/x$,
say $1/x^\alpha$ with $\alpha\in(0,1)$, then it is possible
to relax positive recurrence conditions inspired by diffusion processes.

Let, for some decreasing function $r(x)$ and $x_0>0$,
\begin{eqnarray}\label{r-cond.5.r}
m_1^{[x/2]}(x) &\le& -r(x)\quad\mbox{for all }x\ge x_0.
\end{eqnarray}
Define
\begin{eqnarray*}
W(x) &:=& \int_0^x \min\Bigl(1,\ \frac{1}{yr(y)}\Bigr)dy.
\end{eqnarray*}
Let $xr(x)$ be increasing to infinity, then
\begin{eqnarray}\label{def:W.inf}
W(x) &\ge& \frac{1}{r(x)}\quad\mbox{ultimately in }x.
\end{eqnarray}
Consider a test function $L$ defined as $L(x)=0$ for all $x\le 0$ and
\begin{eqnarray*}
L(x) &:=& \int_0^x W(y)dy\quad\mbox{for }x>0.
\end{eqnarray*}
Since the second moment of jumps is not assumed finite,
there is no diffusion motivated intuition behind the last test function.
\index{Markov chain!condition for!positive recurrence}

\begin{theorem}\label{thm:recurrence.r}
Let the drift condition \eqref{r-cond.5.r} hold with some decreasing 
function $r(x)>0$ such that $xr(x)$ is increasing to infinity. 
Assume that the jumps satisfy the following integrability conditions: 
\begin{eqnarray}\label{uni.int.x.over.r}
\E\{\xi(x)W(\xi(x));\ \xi(x)>x/2\} &\to& 0,\\
\label{second.mu2.r}
\E\{\xi^2(x);\ |\xi(x)|\le x/2\} &=& o(xr(x))
\quad\mbox{as }x\to\infty.
\end{eqnarray}
Let
\begin{eqnarray}\label{uni.int.x.over.r.gen}
\E\{\xi(x)W(\xi(x));\ \xi(x)>0\} &<& \infty
\quad\mbox{for all }x.
\end{eqnarray}
Then there exists an $x_*$ such that
the set $(-\infty,x_*]$ is positive recurrent.
\end{theorem}

Due to \eqref{def:W.inf}, the conditions \eqref{uni.int.x.over.r}
and \eqref{second.mu2.r} are fulfilled if, for example,
\begin{eqnarray}\label{uni.integr.L.inf}
\mbox{the family }\bigl\{|\xi(x)|W(|\xi(x)|),\ x\ge 0\bigr\}
\quad\mbox{is uniformly integrable};
\end{eqnarray}
justification follows from Lemmas \ref{l:g.fin} and \ref{l:p.V.ui.o}.
\index{Markov chain!condition for!positive recurrence}

\begin{corollary}\label{cor:posrec.inf}
Let, for some $\alpha\in(0,1)$, $\varepsilon>0$ and $x_0>0$,
\begin{eqnarray*}
\E\{\xi(x);\ |\xi(x)|\le x/2\} &\le&
-\varepsilon/x^\alpha\quad\mbox{for all }x>x_0.
\end{eqnarray*}
Let also, as $x\to\infty$,
\begin{eqnarray}\label{uni.integr.L.W.2.inf}
\E\{\xi^{1+\alpha}(x);\ \xi(x)>x/2\} &\to& 0,\\
\label{uni.integr.L.3.2.inf}
\E\{\xi^2(x);\ |\xi(x)|\le x/2\} &=& o(x^{1-\alpha}),
\end{eqnarray}
and 
\begin{eqnarray}\label{uni.integr.L.W.2.inf.gen}
\E\{\xi^{1+\alpha}(x);\ \xi(x)>0\} &<& \infty\quad\mbox{for all }x.
\end{eqnarray}
Then there exists an $x_*$ such that
the set $(-\infty,x_*]$ is positive recurrent.
\end{corollary}

Notice that both \eqref{uni.integr.L.3.2.inf} and \eqref{uni.integr.L.W.2.inf} hold
provided the family of random variables $\{|\xi(x)|^{1+\alpha},\ x>0\}$
is uniformly integrable.

\begin{theopargself}
\begin{proof}[of Corollary \ref{cor:posrec.inf}]
It follows if we take $r(x) =\varepsilon/(1+x)^\alpha$ for $x>0$,
then $W(x)\sim c_1 x^\alpha$ and $L(x)\sim c_2 x^{1+\alpha}$.
\qed\end{proof}
\end{theopargself}

\begin{theopargself}
\begin{proof}[of Theorem \ref{thm:recurrence.r}]
By the construction, $L'(x)=W(x)$ and
\begin{eqnarray}\label{L.2prime.2.r}
L''(x)\ =\ W'(x) &=& \min\Bigl(1,\frac{1}{xr(x)}\Bigr)\ >\ 0\quad\mbox{is decreasing;}
\end{eqnarray}
in particular, $W(x)$ is a concave function.

Since $W$ is increasing, $L(x)\le xW(x)$ for $x>0$,
hence \eqref{V.eps.1.pre} follows from the condition \eqref{uni.int.x.over.r.gen},
and it remains to show that the mean drift of $L(x)$ is negative
and bounded away from zero for all sufficiently large $x$.
We start with the following upper bound
\begin{eqnarray}\label{estimate.1.r}
\E L(x+\xi(x))-L(x) &\le& \E\{L(x+\xi(x))-L(x);\ \xi(x)\ge -x/2\}\nonumber\\
&\le& \E\{L(x+\xi(x))-L(x);\ |\xi(x)|\le x/2\}\nonumber\\
&&+\E\{L(x+\xi(x));\ \xi(x)>x/2\}\nonumber\\
&=:& E_1(x)+E_2(x).
\end{eqnarray}

Let us estimate the first term on the right hand side via Taylor's expansion:
\begin{eqnarray*}
E_1(x) &=& L'(x)\E\{\xi(x);\ |\xi(x)|\le x/2\}
+\frac12 \E\{L''(x+\theta\xi(x))\xi^2(x);\ |\xi(x)|\le x/2\}\\
&=& W(x)\E\{\xi(x);\ |\xi(x)|\le x/2\}
+\frac12 \E\{W'(x+\theta\xi(x))\xi^2(x);\ |\xi(x)|\le x/2\},
\end{eqnarray*}
where $0\le \theta=\theta(x,\xi(x))\le 1$. Since $W'$ decreases and
\begin{eqnarray*}
W'(x/2) &=& \frac{2}{xr(x/2)}\ \le\ \frac{2}{xr(x)},
\end{eqnarray*}
we deduce
\begin{eqnarray*}
E_1(x) &\le& W(x)\E\{\xi(x);\ |\xi(x)|\le x/2\}
+\frac12 W'(x/2)\E\{\xi^2(x);\ |\xi(x)|\le x/2\}\\
&\le& W(x)\E\{\xi(x);\ |\xi(x)|\le x/2\}
+\frac{1}{xr(x)}\E\{\xi^2(x);\ |\xi(x)|\le x/2\}.
\end{eqnarray*}
The condition \eqref{second.mu2.r} allows us to conclude that
\begin{eqnarray}\label{estimate.E1}
E_1(x) &\le& W(x)\E\{\xi(x);\ |\xi(x)|\le x/2\}+o(1)
\quad\mbox{as }x\to\infty.
\end{eqnarray}

In order to estimate the second expectation on the right hand side 
of \eqref{estimate.1.r} first notice that,
since the function $\min(1,1/yr(y))$ is decreasing, we get
\begin{eqnarray*}
W(3x) &\le& 3W(x),
\end{eqnarray*}
and therefore
\begin{eqnarray*}
L(3x) &\le& 9xW(x),
\end{eqnarray*}
because $L(x)\le xW(x)$. Hence,
\begin{eqnarray}\label{estimate.E2}
E_2(x) &\le& \E\{L(3\xi(x));\ \xi(x)>x/2\}\nonumber\\
&\le& 9\E\{\xi(x)W(\xi(x));\ \xi(x)>x/2\}\
\to\ 0\quad\mbox{as }x\to\infty,
\end{eqnarray}
owing to the condition \eqref{uni.int.x.over.r}.
Substituting \eqref{estimate.E1} and \eqref{estimate.E2}
into \eqref{estimate.1.r} we get
\begin{eqnarray*}
\E L(x+\xi(x))-L(x) &\le& W(x)\E\{\xi(x);\ |\xi(x)|\le x/2\}+o(1)\\
&\le& -W(x)r(x)+o(1)\quad\mbox{as }x\to\infty,
\end{eqnarray*}
by \eqref{r-cond.5.r}. The inequality \eqref{def:W.inf}
implies that the drift of $\{L(X_n)\}$ is negative and bounded away
from zero for all sufficiently large $x$.
\qed\end{proof}
\end{theopargself}

\section{Non-positivity}\label{sec:nonpos}

In this section we are interested in conditions that provide a kind of
{\it non-positivity}\index{Markov chain!non-positive}
\index{Non-positivity}
of a Markov chain $\{X_n\}$, that is, conditions for existence of $x_*$
such that $\E_x\tau_{(-\infty,x_*]}=\infty$ for some $x\le x_*$.
Below we show even stronger result that
$\E_y\tau_{(-\infty,x_*]}=\infty$ for all $y>x_*$.

As follows from the condition \eqref{ex:1.cond}
for positive recurrence of a diffusion process,
the condition for non-positivity of a diffusion process
just negates \eqref{ex:1.cond}, so it happens when
\begin{eqnarray}\label{ex:1.cond.nonpos}
\mbox{the function }\ \frac{1}{m_2(x)}e^{\int_0^x\frac{2m_1(y)}{m_2(y)}dy}
\quad\mbox{is not integrable at infinity.}
\end{eqnarray}

One could expect that, in terms of test functions,
the existence of a non-negative function $L$ such that,
for some $x_*$ and $\varepsilon>0$, 
$\E\{L(X_1)-L(X_0)\mid X_0=x\}\ge \varepsilon$ for all $x>x_*$
would imply non-negativity of $\{X_n\}$; however just negation of \eqref{V.eps.1}
does not imply that as follows from the following counterexample.
Let $\{X_n\}$ be a Markov chain on $\Z^+$ with transition probabilities
\begin{eqnarray*}
p(x,y) &:=& \left\{
\begin{array}{ll}
1/2 &\mbox{ if }y=2(x+1),\\
1/2 & \mbox{ if }y=0.
\end{array}
\right.
\end{eqnarray*}
Then $m_1(x)=1$ whatever $x$, while this chain is geometrically ergodic,
since the returning time to zero is geometrically distributed with success probability $1/2$.
This counterexample shows that to conclude non-positivity we need 
to ensure some compactness conditions on the jumps, see below.

Fix an increasing function $s(x)\le x/2$. Let
\begin{eqnarray}\label{r-cond.5.nonpos}
\frac{2m_1^{[s(x)]}(x)}{m_2^{[s(x)]}(x)}
&\ge& -r(x)\quad\mbox{for all }x>x_0,
\end{eqnarray}
for a decreasing function $r(x)>0$.
In the next theorem we show that the chain
$\{X_n\}$ is not positive recurrent provided 
\begin{eqnarray}\label{integr.cond.nonpos}
\mbox{the function}\quad
e^{-R(x)} &=& e^{-\int_0^x r(y)dy}
\quad\mbox{is not integrable at infinity},
\end{eqnarray}
which is motivated by the condition \eqref{ex:1.cond.nonpos}
for non-positivity of diffusion processes.
It turns out to be very close to guarantee non-positivity of $\{X_n\}$
but we still need some additional technical conditions
on $r(x)$ and on the function
\begin{eqnarray*}
W(x) &:=& e^{R(x)}\int_0^x e^{-R(y)}dy,
\end{eqnarray*}
which grows as $x$ at least.
Proving non-positivity seems to be the hardest problem
we consider in this chapter.
\index{Markov chain!condition for!non-positivity}
%which is not so surprising
%if we bear in mind the following classical example of
%a Markov chain on $\Z^+$ with positive drift
%which is geometrically ergodic. Let, for $x\in\Z^+$,
%$$
%\P_x\{X_n=0\}\ =\ 1-\P_x\{X_n=2x+1\}\ =\ 1/2,
%$$
%so $m_1(x)=1$ for all $x$. Clearly, then $\tau_0$ is geometrically
%distributed regardless the starting state. It becomes possible
%because the family of all jumps is not tight.
%This example may be developed further to a Markov chain
%with a positive drift everywhere whose jumps compose a tight family
%but the chain is still positive recurrent.
%For example, the following example works:
%$$
%\P_x\{X_n=0\}\ =\ 1-\P_x\{X_n=x+2\}\ =\ 3/(x+4),
%$$
%so again $m_1(x)=1$ for all $x$ but
%$$
%\P_x\{\tau_0=n\}\ =\
%\frac{1}{x+2n}\prod_{k=1}^{n-1}\Bigl(1-\frac{1}{x+2k}\Bigr)
%$$

\begin{theorem}\label{thm:nonpos}
Let the drift condition \eqref{r-cond.5.nonpos} hold with some  
differentiable decreasing function $r(x)=O(1/x)$ such that the condition 
\eqref{integr.cond.nonpos} is satisfied. Assume that the twice differentiable function 
$W(x)$ is convex and satisfies the conditions \eqref{W.le1} and \eqref{W.le2}. 
Let negative jumps satisfy the following integrability conditions:
\begin{eqnarray}\label{uni.integr.L.3.nonpos}
\E\{|\xi(x)|^3,\ \xi(x)\in[-s(x),0]\} &=& o(x^2/W(x)),\\
\label{uni.integr.L.W.nonpos}
\P\{\xi(x)\le -s(x)\} &=& o(1/xW(x))\quad\mbox{as }x\to\infty,
\end{eqnarray}
and, additionaly,
\begin{eqnarray}\label{m1.ge.cx}
m_1(x) &\ge& -c_3/x,\quad c_3\in(0,\infty),\quad\mbox{for all }x>x_0,
\end{eqnarray}
\begin{eqnarray}
\label{m2.nonpos}
c_4\ :=\ \sup_{x>0} m_2(x) &<& \infty,\\
\label{second.mu2.nonpos}
\liminf_{x\to\infty} m_2^{[s(x)]}(x) &>& 0.
\end{eqnarray}
Then there is an $x_*$ such that
$\E_x\tau_{(-\infty,y]} =\infty$ for all $x>y>x_*$.
\end{theorem}

We require the bounds \eqref{m1.ge.cx} and \eqref{m2.nonpos}
on the full moments of jumps to derive a square integrable martingale
from $\{X_n\}$.

The conditions \eqref{uni.integr.L.3.nonpos} and
\eqref{uni.integr.L.W.nonpos} are fulfilled for some $s(x)=o(x)$ if,
for example, the function $W(x)$ is regularly varying at infinity
and the family of random variables
$\{\xi^-(x)W(\xi^-(x)),\ x\ge 0\}$ is uniformly integrable,
sufficiency follows from Lemmas \ref{l:g.fin} and \ref{l:p.V.ui.o}.
\index{Markov chain!condition for!non-positivity}

\begin{corollary}\label{cor:nonpos}
Let, for some $\varepsilon>0$ and $x_0>0$,
\begin{eqnarray*}
\frac{2m_1^{[s(x)]}(x)}{m_2^{[s(x)]}(x)} &\ge&
-\frac{1-\varepsilon}{x}\quad\mbox{for all }x>x_0.
\end{eqnarray*}
Let the conditions \eqref{m1.ge.cx}--\eqref{second.mu2.nonpos} hold,
\begin{eqnarray}\label{uni.integr.L.3.2.nonpos}
\E\{\xi^3(x),\ \xi(x)\in[-s(x),0]\} &=& o(x),\\
\label{uni.integr.L.W.2.nonpos}
\P\{\xi(x)\le -s(x)\} &=& o(1/x^2)\quad\mbox{as }x\to\infty,
\end{eqnarray}
Then there is an $x_*$ such that
$\E_x\tau_{(-\infty,y]} =\infty$ for all $x>y>x_*$.
\end{corollary}

Notice that both \eqref{uni.integr.L.3.2.nonpos} and
\eqref{uni.integr.L.W.2.nonpos} hold for some $s(x)=o(x)$
provided the family of random variables $\{(\xi^-(x))^2,\ x>0\}$
is uniformly integrable.

\begin{theopargself}
\begin{proof}[of Corollary \ref{cor:nonpos}]
It follows from Theorem \ref{thm:nonpos} 
if we take $r(x) =\frac{1-\varepsilon}{1+x}$ for $x>0$, then 
\begin{eqnarray*}
R(x) &=&(1-\varepsilon)\log(1+x),\\
e^{-R(x)} &=& 1/(1+x)^{1-\varepsilon},\\
W(x) &=& (1+x)/\varepsilon,
\end{eqnarray*} 
which implies the test function $L(x)=((1+x)^2-1)/2\varepsilon$.
\qed\end{proof}
\end{theopargself}

\index{Markov chain!condition for!non-positivity}

\begin{corollary}\label{cor:nonpos.log}
Let, for some $m\in\mathbb N$ and $\varepsilon>0$,
\begin{eqnarray*}
\frac{2m_1^{[s(x)]}(x)}{m_2^{[s(x)]}(x)} &\ge&
-\frac{1}{x}-\frac{1}{x\log x}-\ldots-\frac{1}{x\log x\cdot\ldots\cdot\log_{(m-1)}x}
-\frac{1-\varepsilon}{x\log x\cdot\ldots\cdot\log_{(m)}x}
\end{eqnarray*}
for all sufficiently large $x$.
Let the conditions \eqref{m1.ge.cx}--\eqref{second.mu2.nonpos} hold, let
\begin{eqnarray}\label{le.s.m3.nonpos}
\E\{\xi(x)^3;\ \xi(x)\in[-s(x),0]\} &=&
o(x/\log x\cdot\ldots\cdot\log_{(m)}x),
\end{eqnarray}
and let
\begin{eqnarray}\label{le.s.m1.nonpos}
\P\{\xi(x)\le -s(x)\} &=& o(1/x^2\log x\cdot\ldots\cdot\log_{(m)}x).
\end{eqnarray}
Then there is an $x_*$ such that
$\E_x\tau_{(-\infty,y]} =\infty$ for all $x>y>x_*$.
\end{corollary}

%This result was first established by
%Asymont et al. \cite {AMI}\index{Asymont}
%under the condition that moments of order $x^2\log^{2+\delta}x$ are bounded.
Notice that both \eqref{le.s.m3.nonpos} and \eqref{le.s.m1.nonpos} hold
provided the family of random variables
\begin{eqnarray*}
\{(\xi^-(x))^2\log\xi^-(x)\cdot\ldots\cdot\log_{(m)}\xi^-(x)
\I\{\log_{(m)}\xi^-(x)>0\},\ x>0\}
\end{eqnarray*}
is uniformly integrable.

\begin{theopargself}
\begin{proof}[of Corollary \ref{cor:nonpos.log}]
Consider
\begin{eqnarray*}
r(x) &:=& \Bigl(\frac{1}{y}+\frac{1}{y\log y}+\ldots+
\frac{1}{y\log y\cdot\ldots\cdot\log_{(m-1)}y}
+\frac{1-\varepsilon}{y\log y\cdot\ldots\cdot\log_{(m)}y}\Bigr)\Big|_{y=e^{(m)}+x};
\end{eqnarray*}
where $\log_{(m)} e^{(m)}=1$. Then
\begin{eqnarray*}
R(x) &=& \Bigl(\log y+\log\log y+\ldots+\log_{(m)}y
+(1-\varepsilon)\log_{(m+1)}y\Bigr)\Big|_{y=e^{(m)}+x}\\
&& -e^{(m-1)}-e^{(m-2)}-\ldots-1,\\
e^{-R(x)} &=& \frac{e^{(m)}\cdot e^{(m-1)}\cdot\ldots\cdot 1}{y\cdot\log y\cdot\ldots\cdot
\log_{(m-1)}y\cdot\log_{(m)}^{1-\varepsilon}y}\Big|_{y=e^{(m)}+x},\\
W(x) &=& \frac1\varepsilon y\log y\cdot\ldots\cdot\log_{(m-1)}y\cdot
\log_{(m)}y\Big|_{y=e^{(m)}+x},\\
L(x) &\sim& \frac1{2\varepsilon} x^2\log x\cdot\ldots\cdot\log_{(m-1)}x\cdot\log_{(m)}x.
\end{eqnarray*}
\qed\end{proof}
\end{theopargself}

\begin{theopargself}
\begin{proof}[of Theorem \ref{thm:nonpos}]
Consider a non-negative test function $L(x)$ defined zero
on the negative half-line and
\begin{eqnarray*}
L(x) &:=& \int_0^x W(y)dy\quad\mbox{for all }x\ge0.
\end{eqnarray*}
First let us prove that the mean drift of $L(x)$
is positive and bounded away from zero for all sufficiently large $x$,
more precisely, let us prove that, for some $x_*$ and $\varepsilon>0$,
\begin{eqnarray}\label{V.eps.1.nonpos}
\E\{L(x+\xi(x))-L(x);\ \xi(x)\le s(x)\} &\ge& \varepsilon
\quad\mbox{for all }x>x_*.
\end{eqnarray}
Having this in mind, we analyse Taylor's expansion for the function $L$
with the Lagrange form of the remainder, here $x$, $x+y>0$:
\begin{eqnarray}\label{Taylor.nonpos}
L(x+y)-L(x) &=& L'(x)y+L''(x+\theta y)y^2/2\nonumber\\
&=& W(x)y+W'(x+\theta y)y^2/2,
\end{eqnarray}
where $0\le\theta=\theta(x,y)\le 1$.
Since $W(x)$ is assumed to be convex, $W'$ is increasing, hence
\begin{eqnarray}\label{Taylor.nonpos.neg}
L(x+y)-L(x) &\ge& W(x)y+W'(x)y^2/2\nonumber\\
&=& W(x)y+r(x)W(x)\frac{y^2}{2}+\frac{y^2}{2}
\quad\mbox{for all }y>0.
\end{eqnarray}
We deduce from \eqref{W.le2} that
$$
W'(x+\theta y)\ \ge\ W'(x)-c_2 W(x)|y|/x^2\quad\mbox{ for }y\in[-x/2,0],
$$
hence it follows from \eqref{Taylor.nonpos} that
\begin{eqnarray}\label{Taylor.nonpos.pos}
\lefteqn{L(x+y)-L(x)}\nonumber\\
&\ge& W(x)y+r(x)W(x)\frac{y^2}{2}+\frac{y^2}{2}-c_2\frac{W(x)}{x^2}|y|^3
\ \mbox{ for all }y\in[-x/2,0].
\end{eqnarray}

Now we are ready to estimate the mean drift of $L(X)$.
Since $L$ is non-negative and non-decreasing, the following lower bound holds
\begin{eqnarray}\label{estimate.1.nonpos}
\E L(x+\xi(x))-L(x) &\ge& -L(x)\P\{\xi(x)\le -s(x)\}\nonumber\\
&&+\E\{L(x+\xi(x))-L(x);\ \xi(x)\in[-s(x),0]\}\nonumber\\
&&+\E\{L(x+\xi(x))-L(x);\ \xi(x)\in[0,s(x)]\}.
\end{eqnarray}
It follows from \eqref{Taylor.nonpos.neg} that
\begin{eqnarray}\label{estimate.1.1.nonpos}
\lefteqn{\E\{L(x+\xi(x))-L(x);\ \xi(x)\in[0,s(x)]\}}\nonumber\\
&\ge& W(x) \E\{\xi(x);\ \xi(x)\in[0,s(x)]\}
+\frac12 r(x)W(x)\E\{\xi^2(x);\ \xi(x)\in[0,s(x)]\}\nonumber\\
&&\hspace{50mm}+\frac12 \E\{\xi^2(x);\ \xi(x)\in[0,s(x)]\}.
\end{eqnarray}
It follows from \eqref{Taylor.nonpos.pos} that
\begin{eqnarray}\label{estimate.1.2.nonpos}
\lefteqn{\E\{L(x+\xi(x))-L(x);\ \xi(x)\in[-s(x),0]\}}\nonumber\\
&\ge& W(x) \E\{\xi(x);\ \xi(x)\in[-s(x),0]\}
+\frac12 r(x)W(x)\E\{\xi^2(x);\ \xi(x)\in[-s(x),0]\}\nonumber\\
&& +\frac12 \E\{\xi^2(x);\ \xi(x)\in[-s(x),0]\}
-c_2\frac{W(x)}{x^2} \E\{|\xi(x)|^3;\ \xi(x)\in[-s(x),0]\}\nonumber\\
&\ge& W(x) \E\{\xi(x);\ \xi(x)\in[-s(x),0]\}
+\frac12 r(x)W(x)\E\{\xi^2(x);\ \xi(x)\in[-s(x),0]\}\nonumber\\
&&\hspace{25mm} +\frac12 \E\{\xi^2(x);\ \xi(x)\in[-s(x),0]\}+o(1)
\quad\mbox{as }x\to\infty,
\end{eqnarray}
due to the condition \eqref{uni.integr.L.3.nonpos}.
Finally, it follows from \eqref{uni.integr.L.W.nonpos}
and inequality $L(x)\le xW(x)$ that the first term on the right
of \eqref{estimate.1.nonpos} tends to zero as $x\to\infty$.
Together with the lower bounds \eqref{estimate.1.1.nonpos}
and \eqref{estimate.1.2.nonpos} it implies that
\begin{eqnarray*}
\lefteqn{\E\{L(x+\xi(x))-L(x);\ \xi(x)\le s(x)\}}\\
&\ge& W(x) m_1^{[s(x)]}(x)+\frac12 r(x)W(x) m_2^{[s(x)]}(x)
+\frac{1}{2}m_2^{[s(x)]}(x)+ o(1)\\
&\ge& m_2^{[s(x)]}(x)/2+ o(1)
\quad\mbox{as }x\to\infty,
\end{eqnarray*}
owing to \eqref{r-cond.5.nonpos}. Then \eqref{second.mu2.nonpos} implies
\eqref{V.eps.1.nonpos} for all sufficiently large $x$, say for $x>x_*$.

Let $x_0>x_*$ and let $x_1>x_0+s(x_0)$.
Consider an auxiliary Markov chain $\{Y_n\}$ living on $(-\infty,x_1+s(x_1)]$
whose jumps $\eta(x)$ satisfy
\begin{eqnarray*}
x+\eta(x) &=& \min\{x+\xi(x),x_1+s(x_1)\},
\end{eqnarray*}
so the trajectories of $\{X_n\}$ and $\{Y_n\}$ coincide until the first time
when $\{X_n\}$ leaves the set $(-\infty,x_1]$.
By the construction of $\{Y_n\}$ and because $s(x)$ increases, we also have
\begin{eqnarray}\label{drift.Y.eps}
\E\{L(x+\eta(x))-L(x);\ \eta(x)\le s(x)\}
&\ge& \varepsilon\quad\mbox{for all }x\in(x_*,x_1].
\end{eqnarray}
Consider the following stopping time:
\begin{eqnarray*}
\theta &:=& \min\{n\ge 1:\ Y_n\le x_*\mbox{ or }Y_n>x_1\}\\
&=& \min\{n\ge 1:\ X_n\le x_*\mbox{ or }X_n>x_1\},
\end{eqnarray*}
and define one more auxiliary Markov chain $Z_n$ which equals
$Y_n$ for all $n\le\theta$ and $Z_n=Y_\theta$ for all $n>\theta$;
as follows from \eqref{drift.Y.eps},
the process $L(Z_n)-\varepsilon(\theta\wedge n)$ is a submartingale.
It follows from the optional stopping time theorem that
\begin{eqnarray*}
\E\{\theta\mid Y_0=x_0\} &\le& \frac{L(x_1+s(x_1))-L(x_0)}{\varepsilon}
\ <\ \infty.
\end{eqnarray*}
Then, since the submartingale $\{L(Z_n)\}$ is bounded,
\begin{eqnarray*}
\E\{L(Z_\theta)\mid Y_0=x_0\} &\ge& \E\{L(Z_0)\mid Y_0=x_0\}\ =\ L(x_0).
\end{eqnarray*}
On the other hand,
\begin{eqnarray*}
\lefteqn{\E\{L(Z_\theta)\mid Y_0=x_0\}}\\
&\le& L(x_*)\P\{Z_\theta\le x_*\mid Y_0=x_0\}
+L(x_1+s(x_1))\P\{Z_\theta>x_1\mid Y_0=x_0\}\\
&\le& L(x_*)+L(x_1+s(x_1))\P\{Z_\theta>x_1\mid Y_0=x_0\}.
\end{eqnarray*}
Therefore,
\begin{eqnarray*}
\P\{Z_\theta>x_1\mid Y_0=x_0\}
&\ge& \frac{L(x_0)-L(x_*)}{L(x_1+s(x_1))}.
\end{eqnarray*}
The condition \eqref{W.le1} implies that
\begin{eqnarray}\label{L.2x.x}
L(2x) &=& \int_0^{2x} W(y)dy\ =\ 2\int_0^x W(2y)dy\nonumber\\
&\le& 2c_1\int_0^x W(y)dy\ =\ 2c_1L(x)\quad\mbox{for all }x>0,
\end{eqnarray}
hence
\begin{eqnarray*}
\P\{Z_\theta>x_1\mid Y_0=x_0\}
&\ge& \frac{L(x_0)-L(x_*)}{2c_1L(x_1)}.
\end{eqnarray*}
So, for all $x_1>x_0+s(x_0)$,
\begin{eqnarray}\label{theta.ge}
\P\{X_\theta>x_1\mid X_0=x_0\}
&\ge& \frac{L(x_0)-L(x_*)}{2c_1L(x_1)};
\end{eqnarray}
in words, starting at point $x_0$, the chain $\{X_n\}$ exceeds the level $x_1$
before touching the set $(-\infty,x_*]$ with probability not less
than the ratio on the right hand side of \eqref{theta.ge}.

Consider now a starting state $x_1>2x_*$, a stopping time
\begin{eqnarray*}
\tau=\tau_{(-\infty,x_1/2]} &=& \min\{n:\ X_n\le x_1/2\},
\end{eqnarray*}
and a stopped Markov chain $\widehat X_n=X_{n\wedge\tau}$
with initial state $\widehat X_0=x_1$ and
with jumps $\widehat\xi(x)$ defined as $\widehat\xi(x)=\xi(x)$
for all $x>x_1/2$ and $\widehat\xi(x)=0$ for all $x\le x_1/2$. 
Denote $\widehat m_1(x):=\E\widehat\xi(x)$;
by the condition \eqref{m1.ge.cx} we have
\begin{eqnarray}\label{widetilde.m1}
\widehat m_1(x) &\ge& -2c_3/x_1\quad\mbox{for all }x\in\R.
\end{eqnarray}
Given $\widehat X_0=x_1$, the process
\begin{eqnarray*}
M_n &:=& \widehat X_n-x_1-\sum_{k=0}^{n-1} \widehat m_1(\widehat X_k)\
=\ \sum_{k=0}^{n-1} (\xi(\widehat X_k)-\widehat m_1(\widehat X_k))
\end{eqnarray*}
is a square integrable---by \eqref{m2.nonpos}---martingale, $M_0=0$.
Then, by \eqref{widetilde.m1},
\begin{eqnarray*}
\widehat X_n &=& x_1+M_n+\sum_{k=0}^{n-1} \widehat m_1(\widehat X_k)
\ \ge\ x_1+M_n-2c_3n/x_1,
\end{eqnarray*}
which implies, for $n\le x_1^2/8c_3$,
\begin{eqnarray*}
\P\{\widehat X_n\le x_1/2\mid \widehat X_0=x_1\}
&=& \P\{M_n\le -x_1/2+2c_3n/x_1\}\\
&\le& \P\{M_n\le -x_1/4\}\\
&\le& 16\frac{\E M_n^2}{x_1^2}\ \le\ 16c_4 \frac{n}{x_1^2},
\end{eqnarray*}
owing to Chebyshev's inequality and the upper bound for the second moment
of square integrable martingale, $\E M_n^2\le c_4n$,
which follows from \eqref{m2.nonpos}. Hence, for $n\le x_1^2/32c_4$,
\begin{eqnarray*}
\P\{\widehat X_n>x_1/2\mid \widehat X_0=x_1\} &\ge& 1/2.
\end{eqnarray*}
Since $\{\widehat X_n\}$ is $\{X_n\}$ stopped when it enters $(-\infty,x_1/2]$,
the event $\widehat X_n>x_1/2$ yields $\tau\ge n$, so
\begin{eqnarray*}
\P\{\tau_{(-\infty,x_1/2]}\ge x_1^2/32c_4\mid X_0=x_1\} &\ge& 1/2.
\end{eqnarray*}
So, starting at point $x_0$,
with probability estimated from below in \eqref{theta.ge},
$\{X_n\}$ reaches level $x_1$ before it enters $(-\infty,x_*]$,
and then does not drop below level $x_1/2$ within
time interval of length $[x_1^2/32c_4]$ with probability at least $1/2$.
Therefore,
\begin{eqnarray*}
\P\{\tau_{(-\infty,x_*]}\ge x_1^2/32c_4\mid X_0=x_0\}
&\ge& \frac{L(x_0)-L(x_*)}{4c_5L(x_1)}.
\end{eqnarray*}
Thus, due to \eqref{L.2x.x},
\begin{eqnarray*}
\P\{\tau_{(-\infty,x_*]}\ge j\mid X_0=x_0\}
&\ge& \frac{L(x_0)-L(x_*)}{4c_5L(\sqrt{32c_4 j})}
\ \ge\ c_6\frac{L(x_0)-L(x_*)}{L(\sqrt j)},\quad c_6<\infty.
\end{eqnarray*}
It remains to prove that the function $1/L(\sqrt x)$ is not integrable.
Indeed, since $L(y)\le yW(y)$,
\begin{eqnarray*}
\int_1^\infty\frac{1}{L(\sqrt x)}dx
&=& 2\int_1^\infty\frac{y}{L(y)}dy\
\ge\ 2\int_1^\infty\frac{1}{W(y)}dy.
\end{eqnarray*}
Taking into account that
\begin{eqnarray*}
\frac{1}{W(y)} &=& \frac{e^{-R(y)}}{\int_0^y e^{-R(z)}dz}
\ =\ \frac{d}{dy}\log \int_0^y e^{-R(z)}dz,
\end{eqnarray*}
we conclude non-integrability of $1/L(\sqrt x)$
from \eqref{integr.cond.nonpos}. Therefore
\begin{eqnarray*}
\sum_{j=1}^\infty\P\{\tau_{(-\infty,x_*]}\ge j\mid X_0=x_0\}
&=& \infty,
\end{eqnarray*}
hence $\E_{x_0}\tau_{(-\infty,x_*]}$ cannot be finite.
\qed\end{proof}
\end{theopargself}

\section{Recurrence and null recurrence}\label{sec:nullrec}
\subsection{Recurrence}

Assume that, for some decreasing function $r(x)\downarrow 0$,
\begin{eqnarray}\label{r-cond.4.rec}
\frac{2m_1^{[x]}(x)}{m_2^{[x]}(x)} &\le& r(x)\quad\mbox{for all }x>x_0.
\end{eqnarray}

The main condition for recurrence is that the function
\begin{eqnarray}\label{integr.cond.rec}
e^{-R(x)} &=& e^{-\int_0^x r(y)dy}
\quad\mbox{is non-integrable at infinity},
\end{eqnarray}
it is motivated by the recurrence condition \eqref{ex:1.cond.rec}
for diffusion processes and turns out to be very close
to guarantee recurrence of $X$. Similarly to positive recurrence,
proving recurrence of a Markov chain is more difficult 
than for a diffusion process and it requires some additional 
regularity conditions on $r(x)$ and moment-like conditions on jumps.

In the next theorem we formulate conditions for recurrence
in terms of a decreasing function
$\widetilde r(x)$ dominating $r(x)$, $\widetilde r(x)>r(x)$,
such that the function $e^{-\widetilde R(x)}$ is also non-integrable where
\begin{eqnarray}\label{widetilde.R.x}
\widetilde R(x) &:=& \int_0^x \widetilde r(y)dy.
\end{eqnarray}
Consider the function $\widetilde L(x)$ which is zero for negative $x$ and
\begin{eqnarray*}
\widetilde L(x) &:=& \int_0^x e^{-\widetilde R(y)}dy\quad\mbox{for all }x\ge 0,
\end{eqnarray*}
which is an unboundedly increasing function
because $e^{-\widetilde R(x)}$ is assumed non-integrable at infinity.
When we apply the next general theorem to particular regular function
$r$ in Corollaries \ref{cor:rec} and \ref{cor:rec.log} below,
we need to choose $\widetilde r$ sufficiently greater than $r$
in order to increase the difference $\widetilde r-r$ and
to satisfy the conditions \eqref{xi.3.x.rec} and \eqref{rec.1a.rec};
on the other hand a larger function $\widetilde r(x)$ produces
smaller values of $e^{-\widetilde R(x)}$, so the choice of
a suitable $\widetilde r$ is a rather delicate task in each particular case.
\index{Markov chain!condition for!recurrence}

\begin{theorem}\label{thm:recurrence}
Let the drift condition \eqref{r-cond.4.rec} hold. Let
\begin{eqnarray}\label{1.r.le1.rec}
\widetilde r\, '(x) &=& O(1/x^2)\quad\mbox{as }x\to\infty.
\end{eqnarray}
Let positive jumps satisfy the following integrability conditions: as $x\to\infty$,
\begin{eqnarray}\label{xi.3.x.rec}
\E\{\xi^3(x);\ \xi(x)\in(0,x]\} &=&
o\bigl(x^2(\widetilde r(x)-r(x))m_2^{[x]}(x)\bigr)\\
\label{rec.1a.rec}
\E\{\widetilde L(\xi(x));\ \xi(x)\ge x\} &=&
o\Bigl((\widetilde r(x)-r(x))e^{-\widetilde R(x)}m_2^{[x]}(x)\Bigr).
\end{eqnarray}
If the function $\widetilde L(x)$ as $x\to\infty$, 
then there exists an $x_*$ such that the set $(-\infty,x_*]$ is recurrent.
\end{theorem}

\index{Markov chain!condition for!recurrence}
\begin{corollary}\label{cor:rec}
Let, for some $\varepsilon>0$ and $x_0>0$,
\begin{eqnarray*}
\frac{2m_1^{[x]}(x)}{m_2^{[x]}(x)} &\le&
\frac{1-\varepsilon}{x}\quad\mbox{for all }x>x_0.
\end{eqnarray*}
Let, as $x\to\infty$,
\begin{eqnarray}\label{eps.rec.1}
\E\{\xi^3(x);\ \xi(x)\in[0,x]\} &=& o(xm_2^{[x]}(x)),\\
\label{eps.rec.2}
\E\{\xi^{\varepsilon/2}(x);\ \xi(x)\ge x\} &=&
o(m_2^{[x]}(x)/x^{2-\varepsilon/2}).
\end{eqnarray}
Then there exists an $x_*$ such that the set $(-\infty,x_*]$ is recurrent.
\end{corollary}

As follows from Lemmas \ref{l:g.fin} and \ref{l:p.V.ui.o},
both \eqref{eps.rec.1} and \eqref{eps.rec.2} hold
provided the family of random variables $\{(\xi^+(x))^2,\ x>0\}$
is uniformly integrable.

\begin{theopargself}
\begin{proof}[of Corollary \ref{cor:rec}]
It follows if we take $\widetilde r(x)=\frac{1-\varepsilon/2}{1+x}$
for $x>0$ which dominates $r(x)=(1-\varepsilon)/x$, then 
\begin{eqnarray*}
\widetilde R(x) &=& (1-\varepsilon/2)\log(1+x),\\
e^{-\widetilde R(x)} &=& 1/(1+x)^{1-\varepsilon/2}, 
\end{eqnarray*}
which implies the test function 
$\widetilde L(x)=2((1+x)^{\varepsilon/2}-1)/\varepsilon$.
\qed\end{proof}
\end{theopargself}

\index{Markov chain!condition for!recurrence}
\begin{corollary}\label{cor:rec.log}
Let, for some $m\in\mathbb N$ and $\varepsilon>0$,
\begin{eqnarray*}
\frac{2m_1^{[x]}(x)}{m_2^{[x]}(x)} &\le&
\frac{1}{x}+\frac{1}{x\log x}+\ldots+\frac{1}{x\log x\cdot\ldots\cdot\log_{(m-1)}x}
+\frac{1-\varepsilon}{x\log x\cdot\ldots\cdot\log_{(m)}x}
\end{eqnarray*}
for all sufficiently large $x$. Let, as $x\to\infty$,
\begin{eqnarray}\label{le.s.m3.rec}
\E\{\xi(x)^3;\ \xi(x)\in[0,x]\} &=&
o\Bigl(\frac{xm_2^{[x]}(x)}{\log x\cdot\ldots\cdot\log_{(m)}x}\Bigr),
\end{eqnarray}
and
\begin{eqnarray}\label{le.s.m1.rec}
\E\{\log_{(m)}^{\varepsilon/2}\xi(x);\ \xi(x)>x\Bigr\} &=&
o\Bigl(\frac{m_2^{[x]}(x)}{x^2\cdot\log x\cdot\ldots\cdot\log_{(m-1)}x
\cdot\log_{(m)}^{1-\varepsilon/2}x}\Bigr).
\end{eqnarray}
Then there exists an $x_*$ such that the set $(-\infty,x_*]$ is recurrent.
\end{corollary}

%This result was first established by Menshikov et al. \cite{AMI} \index{Menshikov}
%under the condition that moments of order $x^2\log^{2+\delta}x$ are bounded.
Notice that both \eqref{le.s.m3.rec} and \eqref{le.s.m1.rec} hold
provided the family of random variables
\begin{eqnarray*}
\{\xi^2(x)\log\xi(x)\cdot\ldots\cdot\log_{(m)}\xi(x) \I\{\log_{(m)}\xi(x)>0\},\ x>0\}
\end{eqnarray*}
is uniformly integrable, see Lemmas \ref{l:g.fin} and \ref{l:p.V.ui.o} for justification.

\begin{theopargself}
\begin{proof}[of Corollary \ref{cor:rec.log}]
Consider
\begin{eqnarray*}
\widetilde r(x) &:=& \Bigl(
\frac{1}{y}+\frac{1}{y\log y}+\ldots+\frac{1}{y\log y\cdot\ldots\cdot\log_{(m-1)}y}
+\frac{1-\varepsilon/2}{y\log y\cdot\ldots\cdot\log_{(m)}y}\Bigr)\Big|_{y=e^{(m)}+x};
\end{eqnarray*}
where $\log_{(m)} e^{(m)}=1$. Then
\begin{eqnarray*}
\widetilde R(x) &=& \Bigl(
\log y+\log\log y+\ldots+\log_{(m)}y+(1-\varepsilon/2)\log_{(m+1)}y
\Bigr)\Big|_{y=e^{(m)}+x}\\
&& -e^{(m-1)}-e^{(m-2)}-\ldots-1,\\
r(x)-\widetilde r(x) &=& O\Bigl(\frac{1}{x\log x\cdot\ldots\cdot\log_{(m)}x}\Bigr),\\
e^{-\widetilde R(x)} &=&
\frac{e^{(m)}\cdot e^{(m-1)}\cdot\ldots\cdot 1}
{y\cdot\log y\cdot\ldots\cdot\log_{(m-1)}y\cdot
\log_{(m)}^{1-\varepsilon/2}y}\Big|_{y=e^{(m)}+x},\\
\widetilde L(x) &=& \frac2\varepsilon \Bigl(\log_{(m)}^{\varepsilon/2}(e^{(m)}+x)-1\Bigr).
\end{eqnarray*}
\qed\end{proof}
\end{theopargself}

\begin{theopargself}
\begin{proof}[of Theorem \ref{thm:recurrence}]
Following Theorem \ref{cr.test.rec},
we construct a non-negative increasing unbounded test function whose
mean drift is non-positive outside the set $(-\infty,x_*]$, for some $x_*$.

Let us prove that the increasing Lyapunov function
$\widetilde L(x)$ constructed above is appropriate.
Since $\widetilde L(x)$ is increasing, for $x>0$,
\begin{eqnarray}\label{E.L.rec}
\lefteqn{\E \widetilde L(x+\xi(x))-\widetilde L(x)}\nonumber\\
&\le& \E\{\widetilde L(x+\xi(x))-\widetilde L(x);\ \xi(x)\ge -x\}\nonumber\\
&\le& \E\{\widetilde L(x+\xi(x))-\widetilde L(x);\ |\xi(x)|\le x\}
+\E\{\widetilde L(x+\xi(x));\ \xi(x)>x\}\nonumber\\
&\le& \widetilde L'(x) m_1^{[x]}(x)
+\frac12\widetilde L''(x) m_2^{[x]}(x)
+\frac16\E\{\xi^3(x)\widetilde L'''(x+\theta\xi(x));\ |\xi(x)|\le x\}\nonumber\\
&&\hspace{55mm}+\E\{\widetilde L(2\xi(x));\ \xi(x)>x\},
\end{eqnarray}
where $0\le\theta=\theta(x,\xi(x))\le 1$,
by Taylor's expansion with the remainder in the Lagrange form.

The derivative $\widetilde L'(x)=e^{-\widetilde R(x)}$ is decreasing,
so $\widetilde L(x)$ is concave on $\R^+$.
Thus $\widetilde L(2x)\le 2\widetilde L(x)$ and hence the fourth
term on the right hand side of \eqref{E.L.rec} may be bounded above as follows:
\begin{eqnarray}\label{E.L.rec.4}
\E\{\widetilde L(2\xi(x));\ \xi(x)>x\} &=&
o\Bigl((r(x)-\widetilde r(x))e^{-\widetilde R(x)}m_2^{[x]}(x)\Bigr),
\end{eqnarray}
owing to the condition \eqref{rec.1a.rec}.

By the construction, $\widetilde L'(x)=e^{-\widetilde R(x)}$ and
$\widetilde L''(x)=-\widetilde r(x)e^{-\widetilde R(x)}$,
so the sum of the first and second terms on the right hand side
of \eqref{E.L.rec} equals
\begin{eqnarray}\label{E.L.rec.12}
\frac12 e^{-\widetilde R(x)}m_2^{[x]}(x)
\Bigl(\frac{2m_1^{[x]}(x)}{m_2^{[x]}(x)}-\widetilde r(x)\Bigr)
&\le& -\frac12 e^{-\widetilde R(x)}
\bigl(\widetilde r(x)-r(x)\bigr)m_2^{[x]}(x),
\end{eqnarray}
owing to \eqref{r-cond.4.rec}.
Again by the construction of $\widetilde L$,
\begin{eqnarray*}
\widetilde L'''(x) &=& (-\widetilde r\,'(x)+\widetilde r^2(x))
e^{-\widetilde R(x)},
\end{eqnarray*}
hence $\widetilde L'''(x)\ge 0$ for all $x$ due to $\widetilde r'\le 0$ 
and, for all $x$ and $y>0$,
\begin{eqnarray*}
\widetilde L'''(x+y) &\le&
(-\widetilde r\,'(x+y)+\widetilde r^2(x))e^{-\widetilde R(x)}\\
&\le& (c_1/x^2+\widetilde r^2(x)) e^{-\widetilde R(x)}\\
&\le& c_2 e^{-\widetilde R(x)}/x^2,
\end{eqnarray*}
due to \eqref{1.r.le1.rec}, which particularly implies $\widetilde r(x)=O(1/x)$.
Hence,
\begin{eqnarray}\label{E.L.rec.3}
\E\{\widetilde L'''(x+\theta\xi(x))\xi^3(x);\ |\xi(x)|\le x\}
&\le& \E\{\widetilde L'''(x+\theta\xi(x))\xi^3(x);\ \xi(x)\in[0,x]\}\nonumber\\
&\le& c_2\frac{e^{-\widetilde R(x)}}{x^2}\E\{\xi(x)^3; \ \xi(x)\in[0,x]\}\nonumber\\
&=& o\bigl(e^{-\widetilde R(x)}(\widetilde r(x)-r(x))
m_2^{[x]}(x)\bigr),
\end{eqnarray}
by the condition \eqref{xi.3.x.rec}.
Substituting \eqref{E.L.rec.4}--\eqref{E.L.rec.3}
into \eqref{E.L.rec} we finally get
\begin{eqnarray*}
\E \widetilde L(x+\xi(x))-\widetilde L(x)
&\le& -\frac{1+o(1)}{2}e^{-\widetilde R(x)}
\bigl(\widetilde r(x)-r(x)\bigr) m_2^{[x]}(x)\quad\mbox{as }x\to\infty,
\end{eqnarray*}
where the right hand side is negative for all sufficiently large $x$, say for $x>x_*$. 
Hence, Theorem \ref{cr.test.tran} applies, as required.
\qed\end{proof}
\end{theopargself}

\subsection{Null recurrence}

Combining Corollaries \ref{cor:rec} and \ref{cor:nonpos}
we get the following conditions for null recurrence.
\index{Markov chain!condition for!null recurrence}

\begin{corollary}\label{cor:null}
Let, for some $\varepsilon>0$ and $x_0>0$,
\begin{eqnarray*}
\biggl|\frac{2m_1^{[s(x)]}(x)}{m_2^{[s(x)]}(x)} \biggr|
&\le& \frac{1-\varepsilon}{x}\quad\mbox{for all }x>x_0.
\end{eqnarray*}
Let the conditions \eqref{m1.ge.cx} and \eqref{second.mu2.nonpos} hold,
and let the family of random variables $\{(\xi^2(x)),\ x>0\}$
be uniformly integrable. Then there is an $x_*$ such that
$\P_x\{\tau_{(-\infty,x_*]}<\infty\}$ 
but $\E_x\tau_{(-\infty,x_*]} =\infty$ for all initial states $x>x_*$.
\end{corollary}

Combining Corollaries \ref{cor:rec.log} and \ref{cor:nonpos.log}
we get another set of conditions for null recurrence.
\index{Markov chain!condition for!null recurrence}

\begin{corollary}\label{cor:null.log}
Let, for some $m\in\mathbb N$ and $\varepsilon>0$,
\begin{eqnarray*}
\biggl|\frac{2m_1^{[x]}(x)}{m_2^{[x]}(x)}\biggr| &\le&
\frac{1}{x}+\frac{1}{x\log x}+\ldots+\frac{1}{x\log x\cdot\ldots\cdot\log_{(m-1)}x}
+\frac{1-\varepsilon}{x\log x\cdot\ldots\cdot\log_{(m)}x}
\end{eqnarray*}
for all sufficiently large $x$.
Let the conditions \eqref{m1.ge.cx} and \eqref{second.mu2.nonpos} hold,
and let the family
\begin{eqnarray*}
\{\xi^2(x)\log\xi(x)\cdot\ldots\cdot\log_{(m)}\xi(x),\ x>0\}
\quad\mbox{be uniformly integrable.}
\end{eqnarray*}
Then there is an $x_*$ such that
$\P_x\{\tau_{(-\infty,x_*]}<\infty\}$ is finite a.s.
but $\E_x\tau_{(-\infty,x_*]} =\infty$ for all initial states $x>x_*$.
\end{corollary}

\section{Transience}\label{sec:trans}
\subsection{Condition motivated by diffusions}

Fix an increasing function $s(x)\to\infty$ as $x\to\infty$ such that $s(x)=o(x)$.
Assume that, for some decreasing function $r(x)>0$,
\begin{eqnarray}\label{r-cond.4.tr}
\frac{2m_1^{[s(x)]}(x)}{m_2^{[s(x)]}(x)} &\ge& r(x)\quad\mbox{for }x>x_0;
\end{eqnarray}
in general, this means that the drift to the right dominates the diffusion
and then the Markov chain $\{X_n\}$ is transient
provided $r(x)$ decreases sufficiently slow---roughly speaking, if $r(x)>1/x$.

The main condition in the next theorem is that the function
\begin{eqnarray}\label{integr.cond.tr}
e^{-R(x)} &=& e^{-\int_0^x r(y)dy}
\quad\mbox{is integrable},
\end{eqnarray}
it is motivated by the transience condition \eqref{ex:1.cond.tran}
for a diffusion process and turns out to be very close
to guarantee the transience of $\{X_n\}$. Similarly to positive recurrence,
proving transience of a Markov chain is more complicated 
than for a diffusion process and it requires some additional
regularity conditions on $r(x)$ together with moment-like conditions on jumps.
\index{Markov chain!condition for!transience}

\begin{theorem}\label{thm:transience}
Let the drift condition \eqref{r-cond.4.tr} hold with a decreasing function 
$r(x)>0$, $r(x)=O(1/x)$, such that the condition \eqref{integr.cond.tr} 
is satisfied. 
Let a decreasing differentiable function $\widetilde r(x)\le r(x)$ be such that
\begin{eqnarray}\label{1.r.le1.tr}
\widetilde r\,'(x) &=& O(1/x^2),\\
\label{def.of.R.aux}
\widetilde R(x) := \int_0^x \widetilde r(y)dy &\to& \infty\quad\mbox{as }x\to\infty,\\
\label{aux.function}
e^{-\widetilde R(x-s(x))} &=& O\bigl(e^{-\widetilde R(x)}\bigr)
\quad\mbox{as }\ x\to\infty,
\end{eqnarray}
and let the function $e^{-\widetilde R(x)}$ is integrable.
Let negative jumps satisfy the following conditions: as $x\to\infty$,
\begin{eqnarray}\label{xi.3.x}
\E\{|\xi(x)|^3;\ \xi(x)\in[-s(x),0]\} &=&
o\bigl(x^2(r(x)-\widetilde r(x))m_2^{[s(x)]}(x)\bigr),\\
\label{rec.1a}
\P\{\xi(x)\le -s(x)\} &=&
o\Bigl((r(x)-\widetilde r(x))e^{-\widetilde R(x)}m_2^{[s(x)]}(x)\Bigr).
\end{eqnarray}
Then, for all $x\in\R$,
\begin{eqnarray}\label{tran.sslln}
\P_y\{X_n>x\mbox{ for all }n\ge0\} &\to& 1\quad\mbox{as }y\to\infty.
\end{eqnarray}
If, in addition, for some $x_0\in\R$,
\begin{eqnarray}\label{rec.2.tr}
\P_{x_0}\Bigl\{\limsup_{n\to\infty}X_n=\infty\Bigr\} &=& 1,
\end{eqnarray}
then 
\begin{eqnarray}\label{lim.Xn=infty}
\P_{x_0}\Bigl\{\lim_{n\to\infty}X_n=\infty\Bigr\} &=& 1.
\end{eqnarray}
\end{theorem}

The condition \eqref{rec.2.tr} (which was first proposed
in this framework by Lamperti\index{Lamperti} \cite{Lamp60}) can be
equivalently restated as follows:
for any $N$ the exit time from the set $(-\infty,N]$
is finite with probability 1. In this way it is clear that,
for a countable Markov chain, the irreducibility implies
\eqref{rec.2.tr}. For a Markov chain on general state space,
the related topic is $\psi$-irreducibility, see
\cite[Sections 4 and 8]{MT}.

If, for instance, $r(x)=1/x^\alpha$ for some $\alpha\in(0,1)$,
then $e^{-R(x)}=e^{-x^{1-\alpha}/(1-\alpha)}$ and the condition
\eqref{aux.function} fails for $s(x)$ growing faster than $x^\alpha$.
Hence \eqref{aux.function} allows us to consider
an arbitrary $s(x)$ of order $o(x)$ in the only case where
the drift is of order $O(1/x)$, see corollaries below.
In the next subsection we present conditions that are
more appropriate for a drift characterised by
the convergence $x m_1(x)\to\infty$ as $x\to\infty$.
\index{Markov chain!condition for!transience}

\begin{corollary}\label{cor:tr.log}
Let, for some $\varepsilon>0$,
\begin{eqnarray*}
\frac{2m_1^{[s(x)]}(x)}{m_2^{[s(x)]}(x)} &\ge& \frac{1+\varepsilon}{x}
\end{eqnarray*}
for all sufficiently large $x$. Let the truncated second moments
$m_2^{[s(x)]}(x)$ be bounded away from zero and infinity, let
\begin{eqnarray}\label{le.s.m2.rec.cor}
\E\{|\xi(x)|^3;\ \xi(x)\in[-s(x),0]\} &=& o(x)\quad\mbox{as } x\to\infty,
\end{eqnarray}
and let
\begin{eqnarray}\label{le.s.m1.rec.cor}
\P\{\xi(x)\le -s(x)\} &=& o(1/x^2\log^{1+\varepsilon} x)
\ \mbox{as } x\to\infty.
\end{eqnarray}
Then \eqref{tran.sslln} holds and 
the condition \eqref{rec.2.tr} implies \eqref{lim.Xn=infty}.
\end{corollary}

As follows from Lemma \ref{l:g.fin},
both \eqref{le.s.m2.rec.cor} and \eqref{le.s.m1.rec.cor}
hold for some $s(x)=o(x)$ provided
$$
\sup_{x>0}\E\{\xi^2(x)\log^{1+2\varepsilon}|\xi(x)|;\ \xi(x)<-1\} <\infty.
$$

\begin{theopargself}
\begin{proof}[of Corollary \ref{cor:tr.log}]
It follows if we take
\begin{eqnarray*}
r(x)\ :=\ \frac{1+\varepsilon}{1+x}\ \mbox{ and }\
\widetilde r(x)\ :=\ \frac{1}{1+x}+\frac{1+\varepsilon}{(1+x)\log(1+x)};
\end{eqnarray*}
then $r(x)-\widetilde r(x) = O(1/x)$,
$\widetilde R(x) = \log(1+x)+(1+\varepsilon)\log\log(1+x)$,
and $e^{-\widetilde R(x)} = 1/(1+x)\log^{1+\varepsilon}(1+x)$.
\qed\end{proof}
\end{theopargself}

\index{Markov chain!condition for!transience}
\begin{corollary}\label{cor:tr.log.drift}
Let, for some $m\in\mathbb N$ and $\varepsilon>0$,
\begin{eqnarray*}
\frac{2m_1^{[s(x)]}(x)}{m_2^{[s(x)]}(x)} &\ge&
\frac{1}{x}+\frac{1}{x\log x}+\ldots+\frac{1}{x\log x\cdot\ldots\cdot\log_{(m-1)}x}
+\frac{1+\varepsilon}{x\log x\cdot\ldots\cdot\log_{(m)}x}
\end{eqnarray*}
for all sufficiently large $x$. Let the truncated second moments $m^{[s(x)]}_2(x)$
be bounded away from zero and infinity, let, as $x\to\infty$,
\begin{eqnarray}\label{le.s.m1.qq1}
\E\{|\xi(x)|^3;\ \xi(x)\in[-s(x),0]\} &=&
o\Bigl(\frac{x}{\log x\cdot\ldots\cdot\log_{(m)}x}\Bigr),
\end{eqnarray}
and let
\begin{eqnarray}\label{le.s.m1.qq2}
\P\{\xi(x)\le -s(x)\} &=& o\Bigl(
\frac{1}{x^2\cdot\log^2 x\cdot\ldots\cdot\log_{(m)}^2x\cdot
\log_{(m+1)}^{1+\varepsilon}x}\Bigr).
\end{eqnarray}
Then \eqref{tran.sslln} holds and 
the condition \eqref{rec.2.tr} implies \eqref{lim.Xn=infty}.
\end{corollary}

%This result was first established by Menshikov\index{Menshikov} et al. \cite{AMI}
%under the condition that moments of order $x^2\log^{2+\delta}x$ are bounded.
As follows from Lemma \ref{l:g.fin},
both conditions \eqref{le.s.m1.qq1} and \eqref{le.s.m1.qq2} hold
for some $s(x)=o(x)$ if
\begin{eqnarray*}
\sup_{x>0}\E \xi^2(x)\log^2|\xi(x)|\ldots\log_{(m)}^2|\xi(x)|
\log_{(m+1)}^{1+2\varepsilon}|\xi(x)|\I\{\log_{(m+1)}(-\xi(x))>0\} &<& \infty.
\end{eqnarray*}

\begin{theopargself}
\begin{proof}[of Corollary \ref{cor:tr.log.drift}]
Consider
\begin{eqnarray*}
r(x) &:=& \Bigl(
\frac{1}{y}+\frac{1}{y\log y}+\ldots+
\frac{1+\varepsilon}{y\log y\cdot\ldots\cdot\log_{(m)}y}\Bigr)\Big|_{y=e^{(m)}+x}
\end{eqnarray*}
and
\begin{eqnarray*}
\widetilde r(x) &:=& \Bigl(
\frac{1}{y}+\frac{1}{y\log y}+\ldots+\frac{1}{y\log y\cdot\ldots\cdot\log_{(m)}y}
+\frac{1+\varepsilon}{y\log y\cdot\ldots\cdot\log_{(m+1)}y}\Bigr)\Big|_{y=e^{(m)}+x};
\end{eqnarray*}
where $\log_{(m)} e^{(m)}=1$. Then
\begin{eqnarray*}
r(x)-\widetilde r(x) &=& O\Bigl(\frac{1}{x\log x\cdot\ldots\cdot\log_{(m)}x}\Bigr),\\
\widetilde R(x) &=& \Bigl(
\log y+\log\log y+\ldots+\log_{(m+1)}y+(1+\varepsilon)\log_{(m+2)}y
\Bigr)\Big|_{y=e^{(m)}+x}\\
&& -e^{(m-1)}-e^{(m-2)}-\ldots-1,
\end{eqnarray*}
and
\begin{eqnarray*}
e^{-\widetilde R(x)} &=&
\frac{e^{(m)}\cdot e^{(m-1)}\cdot\ldots\cdot 1}
{y\cdot\log y\cdot\ldots\cdot\log_{(m)}y\cdot
\log_{(m+1)}^{1+\varepsilon}y}\Big|_{y=e^{(m)}+x}.
\end{eqnarray*}
\qed\end{proof}
\end{theopargself}

\begin{theopargself} 
\begin{proof}[of Theorem \ref{thm:transience}]
We follow Theorem \ref{cr.test.tran} to prove transience,
so we construct a nonnegative bounded test function $L_*(x)\downarrow 0$
such that $\{L_*(X_n)\}$ is a supermartingale.

Consider a decreasing function
\begin{eqnarray*}
\widetilde L(x) &:=& \int_x^\infty e^{-\widetilde R(y)}dy\quad\mbox{for all }x\ge 0,\\
\widetilde L(x) &:=& \widetilde L(0)\quad\mbox{for all }x<0,
\end{eqnarray*}
which is well-defined due to the assumption that
$e^{-\widetilde R(x)}$ is integrable;
this function is bounded, $\widetilde L(x)\le \widetilde L(0)<\infty$.

Let us prove that the mean drift of $\widetilde L(x)$
is negative for all sufficiently large $x$.
Since $\widetilde L(x)$ is decreasing, we have
\begin{eqnarray*}
\lefteqn{\E \widetilde L(x+\xi(x))-\widetilde L(x)}\\
&\le& \E\{\widetilde L(x+\xi(x))-\widetilde L(x);\ \xi(x)\le s(x)\}\\
&\le& \widetilde L(0)\P\{\xi(x)<-s(x)\}
+\E\{\widetilde L(x+\xi(x))-\widetilde L(x);\ |\xi(x)|\le s(x)\}\\
&=& \widetilde L(0)\P\{\xi(x)<-s(x)\}
+\widetilde L'(x) m_1^{[s(x)]}(x)+\frac12\widetilde L''(x)m_2^{[s(x)]}(x)\\
&&\hspace{50mm}+\frac16 \E\{\widetilde L'''(x+\theta\xi(x))\xi^3(x);\
|\xi(x)|\le s(x)\},
\end{eqnarray*}
where $0\le\theta=\theta(x,\xi(x))\le 1$,
by Taylor's expansion with the remainder in the Lagrange form.
By the construction, $\widetilde L'(x)=-e^{-\widetilde R(x)}<0$,
$\widetilde L''(x)=\widetilde r(x)e^{-\widetilde R(x)}>0$, and
\begin{eqnarray}\label{L.3.neg}
\widetilde L'''(x+y) &=& (\widetilde r\,'(x+y)-\widetilde r^2(x+y))
e^{-\widetilde R(x+y)}\ <\ 0
\end{eqnarray}
due to $r'\le 0$, and
\begin{eqnarray}\label{L.3.O}
\widetilde L'''(x+y) &=& O\bigl(e^{-\widetilde R(x)}/x^2\bigr)
\end{eqnarray}
as $x\to\infty$ uniformly for all $|y|\le s(x)=o(x)$,
due to \eqref{1.r.le1.tr}, $\widetilde r(x)\le r(x)=O(1/x)$,
and \eqref{aux.function}. Hence,
\begin{eqnarray*}
\E\{\widetilde L'''(x+\theta\xi(x))\xi^3(x);|\xi(x)|\le s(x)\}
&\le& \E\{\widetilde L'''(x+\theta\xi(x))\xi^3(x); \xi(x)\in[-s(x),0]\}\\
&\le& c_1\frac{e^{-\widetilde R(x)}}{x^2}\E\{|\xi(x)|^3; \xi(x)\in[-s(x),0]\}\\
&=& o\bigl(e^{-\widetilde R(x)}(r(x)-\widetilde r(x))m_2^{[s(x)]}(x)\bigr),
\end{eqnarray*}
by the condition \eqref{xi.3.x}, and therefore,
\begin{eqnarray*}
\lefteqn{\E \widetilde L(x+\xi(x))-\widetilde L(x)}\\
&\le& \widetilde L(0)\P\{\xi(x)\le -s(x)\}
-e^{-\widetilde R(x)}\Bigl(m_1^{[s(x)]}(x)
-\frac12\widetilde r(x)m_2^{[s(x)]}(x)\Bigr)\\
&&\hspace{70mm}+o\bigl(e^{-\widetilde R(x)}(r(x)-\widetilde r(x))\bigr)m_2^{[s(x)]}(x)\\
&\le& \widetilde L(0)\P\{\xi(x)\le -s(x)\}
-e^{-\widetilde R(x)}\frac{m_2^{[s(x)]}(x)}{2}
(1+o(1))\bigl(r(x)-\widetilde r(x)\bigr),
\end{eqnarray*}
by \eqref{r-cond.4.tr} and $r(x)-\widetilde r(x)\ge 0$.
Applying now the condition \eqref{rec.1a} we conclude that
the right hand side is negative for all sufficiently large $x$,
so there exists a sufficiently large $x_*$ such that
\begin{eqnarray*}
\E \widetilde L(x+\xi(x))-\widetilde L(x) &\le& 0
\quad\mbox{ for all }x\ge x_*.
\end{eqnarray*}
Now take $L_*(x):=\min(\widetilde L(x),\widetilde L(x_*))$. Then
\begin{eqnarray*}
\E L_*(x+\xi(x))-L_*(x) &\le& \E \widetilde L(x+\xi(x))-\widetilde L(x)\le 0
\end{eqnarray*}
for all $x\ge x_*$ and
\begin{eqnarray*}
\E L_*(x+\xi(x))-L_*(x) &=&
\E\{\widetilde L(x+\xi(x))-\widetilde L(x_*);x+\xi(x)\ge x_*\}
\le 0
\end{eqnarray*}
for all $x<x_*$. Therefore, $\{L_*(X_n)\}$ constitutes a
positive bounded supermartingale.
Thus Doob's inequality for nonnegative supermartingales 
(see, e.g. \cite[Chap. VII.9]{Feller}) implies \eqref{tran.sslln}.

For \eqref{lim.Xn=infty}, we apply Doob's convergence theorem, 
by which $L_*(X_n)$ has an a.s. limit as $n\to\infty$.
Due to the condition \eqref{rec.2.tr}, this limit equals $L_*(\infty)=0$,
and the proof is complete.
\qed\end{proof}
\end{theopargself}

\subsection{An alternative approach to transience}

Again let us fix some increasing function $s(x)=o(x)$.
\index{Markov chain!condition for!transience}

\begin{theorem}\label{thm:transience.inf}
Let, for some $\varepsilon>0$ and $x_0>0$, the drift satisfy
\begin{eqnarray}\label{r-cond.5.tr.inf}
\frac{2m_1^{[s(x)]}(x)}{m_2^{[s(x)]}(x)} &\ge& \frac{1+\varepsilon}{x}
\quad\mbox{for all }x>x_0,
\end{eqnarray}
and negative jumps be such that
\begin{eqnarray}\label{rec.1a.inf}
\P\{\xi(x)<-s(x)\} &\le& p(x)m_1^{[s(x)]}(x),
\end{eqnarray}
where a decreasing function $p(x)>0$ is integrable.
Then \eqref{tran.sslln} follows.
If, in addition, the irreducibily condition \eqref{rec.2.tr} holds,
then \eqref{lim.Xn=infty} is valid.
\end{theorem}

Clearly the condition \eqref{rec.1a.inf} is weaker than \eqref{le.s.m1.rec.cor}.
\index{Markov chain!condition for!transience}

\begin{corollary}\label{cor:tr.inf}
Let, for some $\alpha\in(0,1)$, $\varepsilon>0$ and $x_0>0$,
\begin{eqnarray*}
\E\{\xi(x);|\xi(x)|\le s(x)\} &\ge&
\frac{\varepsilon}{x^\alpha}\quad\mbox{for all }x>x_0.
\end{eqnarray*}
Let also, as $x\to\infty$,
\begin{eqnarray}\label{uni.integr.L.W.2.tr}
\P\{\xi(x)\le-s(x)\} &=& o(p(x)/x^\alpha),\\
\label{uni.integr.L.3.2.tr}
\E\{\xi^2(x),\ |\xi(x)|\le s(x)\} &=& o(x^{1-\alpha}),
\end{eqnarray}
where a decreasing function $p(x)>0$ is integrable.
Then \eqref{tran.sslln} follows.
If, in addition, the irreducibily condition \eqref{rec.2.tr} holds,
then \eqref{lim.Xn=infty} is valid.
\end{corollary}

Notice that both \eqref{uni.integr.L.W.2.tr} and \eqref{uni.integr.L.3.2.tr}
hold for some $s(x)=o(x)$ provided the family of random variables
$\{|\xi(x)|^{1+\alpha},\ x>0\}$ possesses an integrable majorant,
see Lemmas \ref{l:maj.p.e} and \ref{l:p.V.ui.o}.

\begin{theopargself}
\begin{proof}[of Theorem \ref{thm:transience.inf}]
By Lemma \ref{l:g.fin.p}, there exists a slower decreasing function $p_1(x)$ 
which is still integrable and $p_1(x)/p(x)\to\infty$, so we can strengthen 
the condition \eqref{rec.1a.inf} to the following one
\begin{eqnarray}\label{rec.1a.inf.}
\P\{\xi(x)<-s(x)\} &=& o\bigl(p(x)m_1^{[s(x)]}(x)\bigr)
\quad\mbox{as } x\to\infty.
\end{eqnarray}
Since $p(x)$ is decreasing and integrable at infinity,
by Lemma \ref{l:denis},
there exists a continuous decreasing integrable regularly
varying at infinity with index $-1$ function $V_1(x)$
such that $p(x)\le V_1(x)$. Take
$$
V(x):=\int_x^\infty V_2(y)dy,
\quad\mbox{where}\quad
V_2(x):=\int_x^\infty \frac{V_1(y)}{y} dy .
$$
By Theorem 1(a) from \cite[Ch  VIII, Sec 9]{Feller}
we know that $V_2$ is regularly varying at infinity with index $-1$
and $V_2(x)\sim V_1(x)$ as $x\to\infty$.
Since $V_1$ is integrable, the nonnegative decreasing
function $V(x)$ is bounded, $V(0)<\infty$,
and $V(x)$ is slowly varying by the same reference.

Let us prove that the mean drift of $V(x)$ is negative
for all sufficiently large $x$. Since $V(x)$ is decreasing, we have
\begin{eqnarray*}
\lefteqn{\E V(x+\xi(x))-V(x)\ \le\
\E\{V(x+\xi(x))-V(x);\xi(x)\le s(x)\}}\\
&\le& V(0)\P\{\xi(x)<-s(x)\}
+\E\{V(x+\xi(x))-V(x);|\xi(x)|\le s(x)\}\\
&=& V(0)\P\{\xi(x)<-s(x)\}
+V'(x)\E\{\xi(x);|\xi(x)|\le s(x)\}\\
&&\hspace{50mm} +\frac12 \E\{V''(x+\theta\xi(x))\xi^2(x);|\xi(x)|\le s(x)\},
\end{eqnarray*}
where $0\le\theta=\theta(x,\xi(x))\le 1$,
by Taylor's expansion with the remainder in the Lagrange form.
By the construction, $V'(x)=-V_2(x)$ and
$$
V''(x+y)=\frac{V_1(x+y)}{x+y}=(1+o(1))\frac{V_1(x)}{x}
\quad\mbox{as }x\to\infty\mbox{ uniformly for }|y|\le s(x).
$$
Hence,
\begin{eqnarray*}
\lefteqn{\E V(x+\xi(x))-V(x)}\\
&\le& V(0)\P\{\xi(x)\le -s(x)\}
-V_2(x)m_1^{[s(x)]}(x)+(1+o(1))\frac{V_1(x)}{2x}m_2^{[s(x)]}(x).
\end{eqnarray*}
The first term on the right hand side is of order
$o(V_1(x)m_1^{[s(x)]}(x))$ by \eqref{rec.1a.inf.}
and the inequality $p(x)\le V_1(x)$.
The third term is not greater than
$$
(1+o(1))V_1(x)\frac{m_1^{[s(x)]}(x)}{1+\varepsilon}
$$
because of the condition \eqref{r-cond.5.tr.inf}. Then
\begin{eqnarray*}
\E V(x+\xi(x))-V(x) &\le&
-V_1(x)m_1^{[s(x)]}(x)+V_1(x)\frac{m_1^{[s(x)]}(x)}{1+\varepsilon}
+o(V_1(x)m_1^{[s(x)]}(x)).
\end{eqnarray*}
This yields that there exists a sufficiently large $x_*$ such that
\begin{eqnarray*}
\E V(x+\xi(x))-V(x)
&\le& -\frac{\varepsilon}{1+2\varepsilon}m_1^{[s(x)]}(x)V_1(x)
\quad\mbox{for all }x\ge x_*.
\end{eqnarray*}
Then the rest of the proof is the same
as of the proof of Theorem \ref{thm:transience}.
\qed\end{proof}
\end{theopargself}

\section{Auxiliary lemmas on dominating functions and random variables}
\sectionmark{Auxiliary lemmas}
\label{sec:auxiliary}

We repeatedly need to construct some majorants for functions
or random variables that satisfy certain properties.
In this section we have collected all results in this direction
required in our calculations.

\begin{definition}\label{def:uni.int}
A family $\{\xi_\theta,\ \theta\in\Theta\}$ of positive random variables
is called {\it uniformly integrable} if\index{Random variables!uniformly integrable}
$$
\sup_{\theta\in\Theta}\E\{\xi_\theta;\ \xi_\theta>A\}\ \to\ 0\quad\mbox{as }A\to\infty.
$$
Equivalently, $\{\xi_\theta,\ \theta\in\Theta\}$ is called uniformly integrable if
$$
\sup_{\theta\in\Theta}\E\xi_\theta\ <\infty
$$
and, for any $\varepsilon>0$ there exists a $\delta>0$ such that
$$
\sup_{\theta\in\Theta}\E\{\xi_\theta;\ B\}\ \le\ \varepsilon
\quad\mbox{whenever }\P\{B\}\le\delta.
$$
\end{definition}

\index{Random variables!condition for uniformly integrability}
\begin{lemma}\label{l:g.fin}
Let $\xi_\theta\ge 0$, be a family of positive random variables
indexed by $\theta\in\Theta$.
Then the following statements are equivalent:

(i) the family $\{\xi_\theta,\ \theta\in\Theta\}$ is uniformly integrable;

(ii) there exists an increasing non-negative function $g(x)\to\infty$ such that
$$
\sup_{\theta\in\Theta}\E\xi_\theta g(\xi_\theta)\ <\ \infty.
$$
\end{lemma}

\begin{proof}
(i)$\Rightarrow$(ii).
Uniform integrability implies existence of an increasing sequence
$n_k\to\infty$, $k\ge 0$, such that $n_0=0$ and
$$
\E\{\xi_\theta;\ \xi_\theta>n_k\}
\ \le\ 1/k^2\quad\mbox{for all }\theta\in\Theta\mbox{ and }k\ge 1.
$$
Define an increasing unbounded function $g(x)$ as $g(0)=0$ and
\begin{eqnarray}\label{def.of.g.n}
g(x) &:=& \sum_{k=0}^\infty(k+1)\I\{x\in(n_k,n_{k+1}]\},\quad x>0.
\end{eqnarray}
The expectation of $\xi_\theta g(\xi_\theta)$ may be bounded as follows:
\begin{eqnarray*}
\E\xi_\theta g(\xi_\theta) &=&
\sum_{k=0}^\infty(k+1)\E\{\xi_\theta;\ \xi_\theta\in(n_k,n_{k+1}]\}\\
&=& \sum_{k=0}^\infty\E\{\xi_\theta;\ \xi_\theta>n_k\}
\ \le\ \E\xi_\theta+\sum_{k=1}^\infty 1/k^2,
\end{eqnarray*}
where the right hand side is uniformly bounded for all $\theta\in\Theta$
which completes the proof of the direct implication.

The implication (ii)$\Rightarrow$(i) is immediate.
\qed\end{proof}

\begin{lemma}\label{l:uni.stop}
Let $\mathcal F_{\theta,n}$ be a $\sigma$-field indexed by $\theta\in\Theta$. 
Let $Y_{\theta,n}$, $n\ge 0$, be a family of increasing processes, 
$Y_{\theta,n+1}\ge Y_{\theta,n}$ for all $n$ and $\theta$,
while $Y_{\theta,0}=0$.
Let the family of conditional distributions of $Y_{\theta,n+1}-Y_{\theta,n}$
given $\mathcal F_{\theta,n}$ be uniformly integrable a.s.\ 
for all $n\ge 0$, $\theta\in\Theta$.
Let $\tau_\theta$ be a family of stopping times with respect to $\mathcal F_{\theta,n}$.
%and such that, for all $n$, $\sigma\{Y_{\theta,k},k\le n,\tau_\theta\le n\}\subseteq\mathcal F_{\theta,n}$.
%and the $\sigma$-algebras $\mathcal F_{\theta,n}$ 
%and $\sigma\{Y_{\theta,k+1}-Y_{\theta,k},k\ge n\}$ are independent 
%given $\mathcal F_{\theta,n}$.
Then the following holds true:
\begin{itemize}
\item[(i)] If%, for some counting random variable $\tau$ with finite mean, $\E\tau<\infty$,
\begin{eqnarray}\label{Y.tau.maj}
\mbox{the family }\{\tau_\theta,\ \theta\in\Theta\}\mbox{ is uniformly integrable,}
\end{eqnarray}
then the family of random variables $Y_{\theta,\tau_\theta}$, $\theta\in\Theta$,
is uniformly integrable too.
\item[(ii)] If, for some $E_\theta$, 
\begin{eqnarray}\label{Y.tau.maj.N}
\mbox{the family }\{\tau_\theta/E_\theta,\ \theta\in\Theta\}\mbox{ is uniformly integrable,}
\end{eqnarray}
then the family of random variables $Y_{\theta,\tau_\theta}/E_\theta$, $\theta\in\Theta$,
is uniformly integrable too.
\end{itemize}
\end{lemma}

\begin{proof}
Firstly let us show that
\begin{eqnarray}\label{E.Y.tau}
\E Y_{\theta,\tau_\theta} &\le& C\E\tau_\theta,
\end{eqnarray}
where
\begin{eqnarray*}
C &:=& \sup_{n,\theta,\omega}\E\{Y_{\theta,n+1}-Y_{\theta,n}\mid \mathcal F_{\theta,n}\} 
\ <\ \infty.
\end{eqnarray*}
Indeed, 
\begin{eqnarray*}
\E Y_{\theta,\tau_\theta} &=& 
\E\sum_{k=0}^\infty (Y_{\theta,k+1}-Y_{\theta,k})\I\{k<\tau_\theta\}\\
&=& \E\sum_{k=0}^\infty \E\{(Y_{\theta,k+1}-Y_{\theta,k})\I\{k<\tau_\theta\}\mid 
\mathcal F_{\theta,k}\}\\
&=& \E\sum_{k=0}^\infty \I\{k<\tau_\theta\} \E\{Y_{\theta,k+1}-Y_{\theta,k}\mid 
\mathcal F_{\theta,k}\},
\end{eqnarray*}
because
$\{k<\tau_\theta\}=\overline{\{k\ge \tau_\theta\}}\in \mathcal F_{\theta,k}$.
Hence,
\begin{eqnarray*}
\E Y_{\theta,\tau_\theta} &\le& 
C\E\sum_{k=0}^\infty \I\{k<\tau_\theta\}
\ =\ C \E\tau_\theta,
\end{eqnarray*}
and \eqref{E.Y.tau} follows. Similarly, for any natural $N$,
\begin{eqnarray}\label{E.Y.tau.N}
\E\{Y_{\theta,\tau_\theta}-Y_{\theta,N};\ \tau_\theta>N\} &\le& 
C\E\{\tau_\theta-N;\ \tau_\theta>N\},
\end{eqnarray}
because
\begin{eqnarray*}
\E\{Y_{\theta,\tau_\theta}-Y_{\theta,N};\ \tau_\theta>N\} &=&
\E\sum_{k=N}^\infty (Y_{\theta,k+1}-Y_{\theta,k})\I\{k<\tau_\theta\}\\
&=& \E\sum_{k=N}^\infty \E\{(Y_{\theta,k+1}-Y_{\theta,k})\I\{k<\tau_\theta\}\mid 
\mathcal F_{\theta,k}\}\\
&=& \E\sum_{k=N}^\infty \I\{k<\tau_\theta\} \E\{Y_{\theta,k+1}-Y_{\theta,k}\mid 
\mathcal F_{\theta,k}\}\\ 
&\le& C\sum_{k=N}^\infty \P\{k<\tau_\theta\}\ =\ C\E\{\tau_\theta-N;\ \tau_\theta>N\}.
\end{eqnarray*}

Under the uniform integrability condition \eqref{Y.tau.maj}, 
it follows from \eqref{E.Y.tau} that $\E Y_{\theta,\tau_\theta}$ is bounded.
Further, for any natural $N$ and event $B$, 
\begin{eqnarray}\label{uni.int.1}
\E\{Y_{\theta,\tau_\theta};\ B\} &=&
\E\{Y_{\theta,\tau_\theta};\ \tau_\theta\le N,\ B\}
+\E\{Y_{\theta,N};\ \tau_\theta>N,\ B\}\nonumber\\
&&\hspace{30mm}
+\E\{Y_{\theta,\tau_\theta}-Y_{\theta,N};\ \tau_\theta>N,\ B\}\nonumber\\
&\le& 2\E\{Y_{\theta,N};\ B\}
+\E\{Y_{\theta,\tau_\theta}-Y_{\theta,N};\ \tau_\theta>N\},
\end{eqnarray}
by the increase of the process $Y_\theta$.
For any fixed $N$, the first expected value on the right hand side tends to zero 
as $\P\{B\}\to 0$ due to the uniform integrability of the jumps of $Y_\theta$, because
\begin{eqnarray*}
\sup_\theta\E Y_{\theta,N} &\le& CN\ <\ \infty,
\end{eqnarray*}
due to \eqref{E.Y.tau} with $\tau=N$, and
\begin{eqnarray*}
\E\{Y_{\theta,N};\ B\} &=&
\sum_{k=0}^{N-1}\E\{Y_{\theta,k+1}-Y_{\theta,k};\ B\}.
\end{eqnarray*}
The second expected value on the right hand side of \eqref{uni.int.1}
tends to zero as $N\to\infty$ uniformly for all $\theta$ due to \eqref{E.Y.tau.N}
and the uniform integrability of $\{\tau_\theta\}$.

Under the condition \eqref{Y.tau.maj.N}, 
it follows from \eqref{E.Y.tau} that $\E Y_{\theta,\tau_\theta}/E_\theta$ is bounded.
Further, for any natural $N$ and event $B$, 
\begin{eqnarray}\label{uni.int.1.E}
\E\biggl\{\frac{Y_{\theta,\tau_\theta}}{E_\theta};\ B\biggr\} 
&\le& 2\E\biggl\{\frac{Y_{\theta,NE_\theta}}{E_\theta};\ B\biggr\}
+\E\biggl\{\frac{Y_{\theta,\tau_\theta}-Y_{\theta,NE_\theta}}{E_\theta};\ 
\tau_\theta>NE_\theta\biggr\},
\end{eqnarray}
by the increase of the process $Y_\theta$.
For any fixed $N$, the first expected value on the right hand side tends to zero 
as $\P\{B\}\to 0$ due to the uniform integrability of the jumps of $Y_\theta$, because
\begin{eqnarray*}
\sup_\theta\E Y_{\theta,NE_\theta}/E_\theta &\le& CN\ <\ \infty,
\end{eqnarray*}
due to \eqref{E.Y.tau} with $\tau=NE_\theta$, and
\begin{eqnarray*}
\E\biggl\{\frac{Y_{\theta,NE_\theta}}{E_\theta};\ B\biggr\} &=&
\sum_{k=0}^{NE_\theta-1}\E\biggl\{\frac{Y_{\theta,k+1}-Y_{\theta,k}}{E_\theta};\ B\biggr\}.
\end{eqnarray*}
The second expected value on the right hand side of \eqref{uni.int.1.E}
tends to zero as $N\to\infty$ uniformly for all $\theta$ due to \eqref{E.Y.tau.N}
and the uniform integrability of $\{\tau_\theta/E_\theta\}$.
\qed\end{proof}

\begin{lemma}\label{l:p.V.ui.o}
Let $p>0$ and $V(x)\le x^p$ be a function such that both functions $V(x)$ and $x^p/V(x)$  
are increasing and unbounded. If the family of random variables
$\{V(|\xi_\theta|),\ \theta\in\Theta\}$ is uniformly integrable then
$$
\sup_{\theta\in\Theta}\E\{|\xi_\theta|^p;\ |\xi_\theta|\le x\}
=o\left(\frac{x^p}{V(x)}\right)\quad\text{as }x\to\infty.
$$
\end{lemma}

\begin{proof}
Fix an $A<x$. Then, for all $\theta\in\Theta$,
\begin{eqnarray*}
\E\{|\xi_\theta|^p;\ |\xi_\theta|\le x\}
&\le& A^p+\E\{|\xi_\theta|^p;\ A<|\xi_\theta|\le x\}\\
&=& A^p+\E\biggl\{\frac{|\xi_\theta|^p}{V(|\xi_\theta|)}V(|\xi_\theta|);\ 
A<|\xi_\theta|\le x\biggr\}\\
&\le& A^p+\frac{x^p}{V(x)} \E\{V(|\xi_\theta|);\ |\xi_\theta|>A\},
\end{eqnarray*}
due to the increase of the function $y^p/V(y)$. 
Since $x^p/V(x)\to\infty$, for any fixed $A$,
$$
\limsup_{x\to\infty}\frac{V(x)}{x^p}
\sup_{\theta\in\Theta}\E\{|\xi_\theta|^p;\ |\xi_\theta|\le x\}
\le \sup_{\theta\in\Theta} \E\{V(|\xi_\theta|);\ |\xi_\theta|>A\},
$$
and the conclusion follows by letting $A\to\infty$, 
owing to the uniform integrability of the family $\{V(|\xi_\theta|),\ \theta\in\Theta\}$
and the convergence $V(y)\uparrow\infty$.
\qed\end{proof}

\begin{lemma}\label{l:p.V.maj.o}
Let $\alpha\in(0,1]$ and $\gamma\ge\alpha$. 
Let a family of positive random variables $\{\xi_\theta,\ \theta\in\Theta\}$
possess a majorant $\Xi$ with $\gamma+1-\alpha$ moment finite, 
that is, $\E\Xi^{\gamma+1-\alpha}<\infty$ and
$$
\xi_\theta\ \le_{st}\ \Xi\quad\mbox{for all }\theta\in\Theta.
$$
Then there exists a decreasing integrable at infinity function $p(x)$ such that
$$ 
\sup_{\theta\in\Theta}\E\{\xi_\theta^{\gamma+1};\ \xi_\theta\le x\}
=o(x^{1+\alpha} p(x))\quad\text{as }x\to\infty.
$$
\end{lemma}

\begin{proof}
Integration by parts yields that
\begin{eqnarray*}
\E\{\xi_\theta^{\gamma+1};\ \xi_\theta\le x\}
&=& -\int_0^x y^{\gamma+1}d\P\{\xi_\theta>y\}\\
&=& -x^{\gamma+1}\P\{\xi_\theta>x\}
+(\gamma+1)\int_0^x y^\gamma\P\{\xi_\theta>y\}dy\\
&\le& (\gamma+1)\int_0^x y^\gamma\P\{\Xi>y\}dy,
\end{eqnarray*}
by the majorisation condition. Therefore, by the Markov inequality,
\begin{eqnarray*}
\E\{\xi_\theta^{\gamma+1};\ \xi_\theta\le x\}
&\le& (\gamma+1)\int_0^x y^\alpha\E\{\Xi^{\gamma-\alpha};\ \Xi>y\}dy\\
&=& (\gamma+1)x^{1+\alpha} p(x),
\end{eqnarray*}
where
$$
p(x)\ :=\ \frac{1}{x^{1+\alpha}}\int_0^x y^\alpha\E\{\Xi^{\gamma-\alpha};\ \Xi>y\}dy.
$$
The finiteness of $\E\Xi^{\gamma+1-\alpha}$ implies integrability at infinity of $p(x)$.
Indeed,
\begin{eqnarray*}
\int_0^\infty p(x)dx &=& \int_0^\infty \frac{dx}{x^{1+\alpha}}
\int_0^x y^\alpha \E\{\Xi^{\gamma-\alpha};\ \Xi>y\}dy\\
&=& \int_0^\infty y^\alpha\E\{\Xi^{\gamma-\alpha};\ \Xi>y\}dy 
\int_y^\infty\frac{dx}{x^{1+\alpha}}\\
&=& \frac{1}{\alpha}\int_0^\infty \E\{\Xi^{\gamma-\alpha};\ \Xi>y\}dy\\
&=& \frac{\E\Xi^{\gamma+1-\alpha}}{\alpha}\ <\ \infty,
\end{eqnarray*}
by the moment condition on $\Xi$.
In addition, the function $p(x)$ is decreasing because
\begin{eqnarray*}
\lefteqn{\frac{d}{dx}\frac{1}{x^{1+\alpha}}
\int_0^x y^\alpha\E\{\Xi^{\gamma-\alpha};\ \Xi>y\}dy} \\
&&\hspace{10mm}=\ -\frac{1+\alpha}{x^{2+\alpha}}
\int_0^x y^\alpha\E\{\Xi^{\gamma-\alpha};\ \Xi>y\}dy
+\frac{1}{x}\E\{\Xi^{\gamma-\alpha};\ \Xi>x\}\\
&&\hspace{10mm}\le\ -\frac{1+\alpha}{x^{2+\alpha}}
\E\{\Xi^{\gamma-\alpha};\ \Xi>x\}\int_0^x y^\alpha dy
+\frac{1}{x}\E\{\Xi^{\gamma-\alpha};\ \Xi>x\}\\
&&\hspace{10mm}=\ 0.
\end{eqnarray*}
The proof is complete due to the next Lemma \ref{l:g.fin.p}.
\qed\end{proof}

\begin{lemma}\label{l:g.fin.p}
Let $p(x)>0$ be a decreasing function which is integrable at infinity.
Then there exists a decreasing integrable at infinity
function $p_1(x)>0$ such that $p_1(x)/p(x)\to\infty$ as $x\to\infty$.
\end{lemma}

\begin{proof}
Since $p(x)$ is integrable at infinity,
there exists an increasing sequence $n_k\to\infty$,
$k\ge 0$, such that $n_0=0$ and
$$
\int_{n_k}^\infty p(y)dy\ \le\ 1/k^2\quad\mbox{for all }k\ge 1.
$$
Define an increasing unbounded function $g(x)$ as in \eqref{def.of.g.n},
then the function $p_1(x):=p(x)g(x)$ satisfies the condition
$p_1(x)/p(x)\to\infty$ as $x\to\infty$.
Since $p(x)$ decreases, the sequence $n_k$ may be chosen in such a way that
$$
(k+2)p(n_{k+1})\ <\ (k+1)p(n_k)\quad\mbox{for all }k\ge 1,
$$
which guarantees that the function $p_1(x)$ is decreasing.
In addition, its integral may be bounded as follows:
\begin{eqnarray*}
\int_0^\infty p(x)g(x)dx &=& \sum_{k=0}^\infty(k+1)
\int_{n_k}^{n_{k+1}}p(x)dx\\
&=& \sum_{k=0}^\infty\int_{n_k}^\infty p(x)dx
\ \le\ \int_0^\infty p(x)dx+\sum_{k=1}^\infty 1/k^2\ <\ \infty,
\end{eqnarray*}
which completes the proof.
\qed\end{proof}

\index{Denisov}
\begin{lemma}[Denisov \cite{D2006}]\label{l:denis}
Let $p(x)>0$ be a decreasing function which is integrable at infinity.
Then there exists a decreasing integrable at infinity function $p_1(x)>0$
which dominates $p(x)$ and is regularly varying at infinity with index $-1$.
\end{lemma}

\begin{lemma}\label{l:g.fin.p.2}
Let $p(x)>0$ be a decreasing function which is integrable at infinity.
Then, for any $k\ge 1$, there exists a decreasing integrable at infinity
function $p_k(x)\ge p(x)$ such that it is $k$ times
differentiable and, for all $j\le k$,
$$
\frac{d^j}{dx^j}p_k(x)\ =\ O(1/x^{1+j})\quad\mbox{as }x\to\infty.
$$
\end{lemma}

\begin{proof}
Consider a decreasing function $p_k(x)$ defined by the equality
\begin{eqnarray*}
p_k(x) &:=& 2^k\int_{x/2}^\infty dy_k\int_{y_k/2}^\infty dy_{k-1}
\ldots\int_{y_3/2}^\infty dy_2
\int_{y_2/2}^\infty\frac{p(y_1)}{y_1^k}dy_1.
\end{eqnarray*}
Firstly, since the function $p(x)/x^k$ decreases,
\begin{eqnarray*}
\int_{y_2/2}^\infty\frac{p(y_1)}{y_1^k}dy_1 &\ge&
\int_{y_2/2}^{y_2}\frac{p(y_1)}{y_1^k}dy_1
\ \ge\ \frac{y_2}{2}\frac{p(y_2)}{y_2^k}
\ =\ \frac{1}{2}\frac{p(y_2)}{y_2^{k-1}},
\end{eqnarray*}
so repetition of this lower bound eventually
leads to the inequalities
\begin{eqnarray*}
p_k(x) &\ge& 2^k\int_{x/2}^x \frac{1}{2^{k-1}}\frac{p(y_k)}{y_k}dy_k
\ \ge\ 2^k\frac{x}{2}\frac{1}{2^{k-1}}\frac{p(x)}{x}\ =\ p(x).
\end{eqnarray*}

Secondly, $p_k(x)$ is integrable at infinity because
\begin{eqnarray*}
\int_{y_2/2}^\infty\frac{p(y_1)}{y_1^k}dy_1 &\le&
p(y_2/2)\int_{y_2/2}^\infty\frac{1}{y_1^k}dy_1
\ =\ O\Bigl(\frac{p(y_2/2)}{y_2^{k-1}}\Bigr),
\end{eqnarray*}
and hence after $k-1$ steps we arrive at upper bound
\begin{eqnarray*}
p_k(x) &\le& c\int_{x/2}^\infty\frac{p(y_k/2^{k-1})}{y_k}dy_k,
\quad c<\infty,
\end{eqnarray*}
where the integral on the right hand side is integrable with respect
to $x$, since
\begin{eqnarray*}
\int_0^\infty dx\int_{x/2}^\infty\frac{p(y/2^{k-1})}{y}dy
&=& \int_0^\infty \frac{p(y/2^{k-1})}{y}dy\int_0^{2y} dx\\
&=& 2\int_0^\infty p(y/2^{k-1})dy\ <\ \infty.
\end{eqnarray*}

Thirdly,
\begin{eqnarray*}
\frac{d^k}{dx^k}p_k(x) &=&
-\frac{2^k}{2} \frac{d^{k-1}}{dx^{k-1}} \int_{x/4}^\infty dy_{k-1}
\ldots\int_{y_3/2}^\infty dy_2
\int_{y_2/2}^\infty\frac{p(y_1)}{y_1^k}dy_1\\
&\ldots& \\
&=& (-1)^k\frac{2^k}{2\cdot 4\cdot\ldots\cdot 2^k}
\frac{p(x/2^k)}{(x/2^k)^k}\ =\ O(p(x/2^k)/x^k)\quad\mbox{as }x\to\infty.
\end{eqnarray*}
Since $p(x)$ is decreasing and integrable at infinity,
$p(x)=O(1/x)$ as $x\to\infty$, so $p_k^{(k)}(x)=O(1/x^{1+k})$.
Integrating the $k$th derivative $k-j$ times we get that the $j$th
derivative of $p_k(x)$ is not greater than $(k-j)$th
integral of $c/x^{1+k}$ which is of order $O(1/x^{1+j})$.
This completes the proof.
\qed\end{proof}

\begin{lemma}\label{l:maj.p.e.V}
Let $\xi\ge 0$ be a random variable and let $V(x)\ge 0$
be an increasing function such that $\E V(\xi)<\infty$.
Let $U(x)\ge 0$ be a function such that
the function $f(x):=V(x)/xU(x)$ increases and satisfies the condition
\begin{eqnarray}\label{con.on.f}
\sup_{x>1}\frac{f(2x)}{f(x)} &<& \infty.
\end{eqnarray}
Then there exists an increasing
function $s(x)\to\infty$ of order $o(x)$ such that
\begin{eqnarray*}
\E\{U(\xi);\ \xi>s(x)\} &=& o(p(x)xU(x)/V(x))
\quad\mbox{as }x\to\infty,
\end{eqnarray*}
where $p(x)$ is a decreasing integrable at infinity function
which is only determined by $\xi$ and $V(x)$.
\end{lemma}

\begin{proof}
Since $\E V(\xi)<\infty$, the decreasing function
$$
p_1(x)\ :=\ \E\{V(\xi)/\xi;\ \xi>x\}
$$
is integrable at infinity. Then by Lemmas \ref{l:g.fin.p} and \ref{l:denis},
$$
\E\{V(\xi)/\xi;\ \xi>x\}\ =\ o(p(x))\quad\mbox{as }x\to\infty,
$$
where a decreasing function $p(x)$ is integrable and regularly varying
at infinity with index $-1$. Hence, due to the increase of $V(x)/xU(x)$,
\begin{eqnarray*}
\E\{U(\xi);\ \xi>x\} &=&
\E\Bigl\{\frac{U(\xi)\xi}{V(\xi)} V(\xi)/\xi;\ \xi>x\Bigr\}\\
&\le& \frac{\E\{V(\xi)/\xi;\ \xi>x\}}{V(x)/xU(x)}\\
&=& o(p(x)xU(x)/V(x))\quad\mbox{as }x\to\infty.
\end{eqnarray*}
Therefore, for any $n\in\N$,
\begin{eqnarray*}
\E\{U(\xi);\ \xi>x/n\}\ =\ o(p(x)xU(x)/V(x))
\quad\mbox{as }x\to\infty
\end{eqnarray*}
because the function $p(x)$ is regularly varying at infinity
and owing to \eqref{con.on.f}.
This implies existence of level $s(x)=o(x)$
which delivers the stated result.
\qed\end{proof}

\begin{lemma}\label{l:maj.p.e}
Let $\xi\ge 0$ be a random variable with finite $\gamma$th moment
for some $\gamma\in[1,\infty)$. Let $\alpha\in[1/\gamma,1]$.
Then there exists an increasing
function $s(x)\to\infty$ of order $o(x^\alpha)$ such that,
for all $\beta\in[0,\gamma-1/\alpha]$,
\begin{eqnarray*}
\E\{\xi^\beta;\ \xi>s(x)\} &=& o(p(x)/x^{\alpha(\gamma-\beta)-1})
\quad\mbox{as }x\to\infty,
\end{eqnarray*}
where $p(x)$ is a decreasing integrable at infinity function
which is only determined by $\xi$, $\gamma$, and $\alpha$.
\end{lemma}

\begin{proof}
Put $\eta=\xi^{1/\alpha}$ and $V(x)=x^{\alpha\gamma}$.
As follows from Lemma \ref{l:maj.p.e.V} with $U(x)=x^{\alpha\beta}$,
since $\E\xi^\gamma=\E V(\eta)<\infty$,
there exists a regularly varying at infinity
with index $-1$ function $p(x)$ which is integrable at infinity
and a function $s(x)=o(x)$ such that
\begin{eqnarray*}
\E\{\eta^{\alpha\beta};\ \eta>s(x)\} &=& o(p(x)xU(x)/V(x))\\
&=& o(p(x)/x^{\alpha(\gamma-\beta)-1})\quad\mbox{as }x\to\infty,
\end{eqnarray*}
which can be rewritten as
\begin{eqnarray*}
\E\{\xi^\beta;\ \xi>s^\alpha(x)\} &=& o(p(x)/x^{\alpha(\gamma-\beta)-1})
\quad\mbox{as }x\to\infty,
\end{eqnarray*}
and the proof is complete.
\qed\end{proof}

We also need a generalisation of the last result onto levels $s(x)$
of more general form. To this end we prove the following result.

\begin{lemma}\label{l:maj.p.e.g}
Let $\xi\ge 0$ be a random variable and let $V(x)\ge 0$, $V(x)\to\infty$,
be a strictly increasing function such that $\E V(\xi)<\infty$ and
\begin{eqnarray}\label{con.on.V}
c_V\ :=\ \sup_{x>1}V(2x)/V(x) &<& \infty.
\end{eqnarray}
Let $g(x)\ge 0$, $g(x)\to\infty$, be an increasing function such that
\begin{eqnarray}\label{con.on.g}
\sup_{x>1}g(2x)/g(x) &<& \infty.
\end{eqnarray}
Then there exists an increasing function $s(x)\to\infty$ of order
$o(V^{-1}(xg(x)))$ such that
\begin{eqnarray*}
\P\{\xi>s(x)\} &=& o(p(x)/g(x))\quad\mbox{as }x\to\infty,
\end{eqnarray*}
where $p(x)$ is a decreasing integrable at infinity function.
\end{lemma}

\begin{proof}
Since $V$ is strictly increasing and $g$ increasing,
the function $f(x):=V^{-1}(xg(x))$ is strictly increasing too and,
owing to the condition \eqref{con.on.V},
\begin{eqnarray}\label{f.f.2cV}
\frac{f(x/c_V)}{f(x)}\ =\ \frac{V^{-1}(xg(x/c_V)/c_V)}{V^{-1}(xg(x))}
&\le& \frac{V^{-1}(xg(x)/c_V)}{V^{-1}(xg(x))}
\ \le\ \frac{1}{2}.
\end{eqnarray}
In particular, we can define a random variable $\eta$
such that $f(\eta)=\xi$.
Then the probability under question may be represented as
\begin{eqnarray*}
\P\{\xi>f(x)\} &=& \P\{f(\eta)>f(x)\}\ =\ \P\{\eta>x\}.
\end{eqnarray*}
Since $V(\xi)=V(f(\eta))=\eta g(\eta)$ and $\E V(\xi)<\infty$,
$\E \eta g(\eta)<\infty$ too. Hence,
$$
p_1(x)\ :=\ \E\{g(\eta);\ \eta>x\}
$$
is integrable at infinity. Then by Lemmas \ref{l:g.fin.p} and \ref{l:denis},
$$
\E\{g(\eta);\ \eta>x\}\ =\ o(p(x))\quad\mbox{as }x\to\infty,
$$
where a decreasing function $p(x)$ is integrable and
regularly varying at infinity with index $-1$. Therefore,
$$
\P\{\eta>x\}\ \le\
\frac{\E\{g(\eta);\ \eta>x\}}{g(x)}
\ =\ o(p(x)/g(x))\quad\mbox{as }x\to\infty.
$$
This implies that, for any $n\in\N$,
\begin{eqnarray*}
\P\{\eta>x/n\} &=& o(p(x/n)/g(x/n))\ =\ o(p(x)/g(x)),
\quad\mbox{as }x\to\infty
\end{eqnarray*}
because the function $p(x)$ is regularly varying at infinity
and due to the condition \eqref{con.on.g}.
Equivalently, for any $n\in\N$,
\begin{eqnarray*}
\P\{\xi>f(x/n)\} &=& o(p(x)/g(x))
\quad\mbox{as }x\to\infty.
\end{eqnarray*}
Together with \eqref{f.f.2cV} this implies existence of a level
$s(x)=o(f(x))$ which completes the proof.
\qed\end{proof}

Taking $V(x)=x^2$ we get the following corollary.

\begin{corollary}\label{cor:maj.p.e.g}
Let $\xi\ge 0$ be a random variable with finite second moment.
Let $g(x)\ge 0$, $g(x)\to\infty$, be an increasing function
satisfying the condition \eqref{con.on.g}.
Then there exists an increasing function $s(x)\to\infty$ of order
$o(\sqrt{xg(x)})$ such that
\begin{eqnarray*}
\P\{\xi>s(x)\} &=& o(p(x)/g(x))\quad\mbox{as }x\to\infty,
\end{eqnarray*}
where $p(x)$ is a decreasing integrable at infinity function.
\end{corollary}

\begin{lemma}\label{l:maj.p.e.log}
Let $\xi\ge 0$ be a random variable and let $V(x)$ be a non-negative function
such that $\E V(\xi)\log(1+\xi)<\infty$. Then there exists an increasing
function $s(x)\to\infty$ of order $o(x)$ such that,
\begin{eqnarray*}
\E\{V(\xi);\ \xi>s(x)\} &=& o(p(x)x)\quad\mbox{as }x\to\infty,
\end{eqnarray*}
where $p(x)$ is a decreasing integrable at infinity function.
\end{lemma}

\begin{proof}
It follows almost immediately because
\begin{eqnarray*}
\int_1^\infty \frac{\E\{V(\xi);\ \xi>x\}}{x}dx &=&
\int_1^\infty \frac{dx}{x}\int_x^\infty V(y)\P\{\xi\in dy\}\\
&=& \int_1^\infty V(y)\P\{\xi\in dy\}\int_1^y \frac{dx}{x}\\
&=& \int_1^\infty V(y)(\log y)\P\{\xi\in dy\}\ <\ \infty.
\end{eqnarray*}
Hence, by Lemmas \ref{l:g.fin.p} and \ref{l:denis},
\begin{eqnarray*}
\E\{V(\xi);\ \xi>x\} &=& o(p(x)x)\quad\mbox{as }x\to\infty,
\end{eqnarray*}
where a decreasing function $p(x)$ is integrable and regularly varying
at infinity with index $-1$. Then concluding arguments as
in Lemma \ref{l:maj.p.e.V} complete the proof.
\qed\end{proof}

\begin{lemma}\label{Fuk_Nagaev}
Let $\xi_1$, \ldots, $\xi_n$ be independent random variables with zero mean
and finite variance. Denote $S_n:=\xi_1+\ldots+\xi_n$. Then, for all $x$, $y>0$,
\begin{eqnarray}\label{bp.NF.class.l}
\P\{S_n>x\} &\le&
e^{x/y}\Bigl(\frac{\V S_n}{xy}\Bigr)^{x/y}+\sum_{i=1}^n\P\{\xi_i>y\},
\end{eqnarray}
and, for all $x>\max(y,2\sqrt{\V S_n})$,
\begin{eqnarray}\label{bp.NF.2.l}
\E\{S_n^2;\ S_n>x\}
&\le& e^{x/y}\Bigl(\frac{\V S_n}{xy}\Bigr)^{x/y}x^2
+\sum_{i=1}^n\E\{\xi_i^2;\ \xi_i>y\}+\V S_n\sum_{i=1}^n\P\{\xi_i>y\}.\nonumber\\[-1mm]
\end{eqnarray}
\end{lemma}

\begin{proof}
The inequality \eqref{bp.NF.class.l} is due to Fuk\index{Fuk} 
and Nagaev,\index{Nagaev} see e.g. Corollary 1.11 in
\cite{Nag79}, Theorem 4 in \cite{FukNagaev}.

This inequality \eqref{bp.NF.class.l} allows us to get a bound similar to
\eqref{bp.NF.2.l} as follows. For any $x>y$,
the function $z^{1-2x/y}$ is integrable at infinity with respect to $z$, so
\begin{eqnarray*}
\E\{S_n^2;\ S_n>x\} &=& x^2\P\{S_n>x\} +2\int_x^\infty z\P\{S_n>z\}dz\\
&\le& e^{x/y}(\V S_n)^{x/y}\left[\Bigl(\frac{1}{xy}\Bigr)^{x/y}
+2 \int_x^\infty z\Bigl(\frac{1}{z^2y/x}\Bigr)^{x/y}dz\right]\\
&& +\ \sum_{i=1}^n \left[x^2\P\{\xi_i>y\}
+2\int_x^\infty z\P\Bigl\{\xi_i>z\frac{y}{x}\Bigr\}dz\right]\\
&=& e^{x/y}\Bigl(\frac{\V S_n}{xy}\Bigr)^{x/y}\Bigl(\frac{x^2}{x/y-1}+1\Bigr)
+(x/y)^2 \sum_{i=1}^n\E\{\xi_i^2;\ \xi_i>y\}.
\end{eqnarray*}

Let us now prove \eqref{bp.NF.2.l} following the idea of the proof of \eqref{bp.NF.class.l}
from \cite[Theorem 4]{FukNagaev}. We start with the following upper bounds
\begin{eqnarray}\label{NF.upper}
\E\{S_n^2;\ S_n>x\} &\le& \E\{S_n^2;\ S_n>x,\ \xi_i\le y\mbox{ for all }i\le n\}
+\sum_{i=1}^n\E\{S_n^2;\ S_n>x,\ \xi_i>y\}\nonumber\\
&\le& \E\{T_n^2;\ T_n>x\}+\sum_{i=1}^n\E\{S_n^2;\ S_n>x,\ \xi_i>y\},
\end{eqnarray}
where $T_n=\eta_1+\ldots+\eta_n$, and $\eta_i=\xi_i\I\{\xi_i\le y\}$, so $\E\eta_i\le 0$.
Since $T_n$ is bounded by $ny$, all its positive exponential moments are finite, 
hence for all $\lambda>0$,
\begin{eqnarray*}
\E\{T_n^2;\ T_n>x\} &=& 
\E\Bigl\{\frac{e^{\lambda T_n}}{e^{\lambda T_n}/T_n^2};\ T_n>x\Bigr\} \\
&\le& \frac{\E e^{\lambda T_n}}{e^{\lambda x}/x^2}\quad\mbox{for all }x\ge 2/\lambda,
\end{eqnarray*}
because the function $e^{\lambda x}/x^2$ 
is increasing in the range $x\ge 2/\lambda$. Further,
\begin{eqnarray*}
\E e^{\lambda\eta_i} &=& 1+\lambda \E\eta_i+\E(e^{\lambda\eta_i}-1-\lambda\eta_i)\\
&\le& 1+\lambda \E\eta_i
+\frac{e^{\lambda y}-1-\lambda y}{y^2}\E\eta_i^2,
\end{eqnarray*}
since $\eta_i\le y$ and the function $(e^z-1-z)/z^2$ is increasing in $z\in\R$. Thus,
\begin{eqnarray*}
\E e^{\lambda\eta_i} &\le& 1+\frac{e^{\lambda y}-1-\lambda y}{y^2}\V\xi_i\\
&\le& e^{\frac{e^{\lambda y}-1-\lambda y}{y^2}\V\xi_i}
\ \le\ e^{\frac{e^{\lambda y}-1}{y^2}\V\xi_i},
\end{eqnarray*}
and then
\begin{eqnarray*}
\E e^{\lambda T_n} &\le& e^{\frac{e^{\lambda y}-1}{y^2}\V S_n},
\end{eqnarray*}
Take 
$$
\lambda\ =\ \frac{1}{y}\log\Bigl(\frac{xy}{\V S_n}+1\Bigr),
$$
so that $x>2/\lambda$ because it is equivalent to
$$
\frac{xy}{\V S_n}+1\ >\ e^{2y/x},
$$
which is satisfied due to $x>\max(y,2\sqrt{\V S_n})$.
Then $\E e^{\lambda T_n} \le e^{x/y}$, so
\begin{eqnarray}\label{NF.upper.1}
\frac{\E e^{\lambda T_n}}{e^{\lambda x}} &\le& 
e^{x/y}e^{-\frac{x}{y}\log(xy/(\V S_n)+1)}\nonumber\\
&\le& e^{x/y}\Bigl(\frac{\V S_n)}{xy}\Bigr)^{x/y}.
\end{eqnarray}

By the independence of $\xi_i$'s,
\begin{eqnarray*}
\E\{S_n^2;\ S_n>x,\ \xi_n>y\}
&\le& \E\{(S_{n-1}+X_n)^2;\ \xi_n>y\}\\
&=& \E\{\E\{(S_{n-1}+\xi_n)^2\mid \xi_n\};\ \xi_n>y\}\\
&=& \E\{\V S_{n-1}+\xi_n^2;\ \xi_n>y\}.
\end{eqnarray*}
Therefore,
\begin{eqnarray*}
\E\{S_n^2;\ S_n>x,\ \xi_n>y\}
&\le& \V S_{n-1}\P\{\xi_n>y\}+\E\{\xi_n^2;\ \xi_n>y\},
\end{eqnarray*}
which implies that
\begin{eqnarray}\label{NF.upper.2}
\sum_{i=1}^n\E\{S_n^2;\ S_n>x,\ \xi_i>y\}
&\le& \V S_n\sum_{i=1}^n\P\{\xi_i>y\}+\sum_{i=1}^n\E\{\xi_i^2;\ \xi_i>y\}.\hspace{10mm}
\end{eqnarray}
Substituting \eqref{NF.upper.1} and \eqref{NF.upper.2} into \eqref{NF.upper}
we conclude the proof of the upper bound for the tail second moment of $S_n$.
\qed\end{proof}

\begin{lemma}\label{M-Z}
Let $\xi_1$, \ldots, $\xi_n$ be independent random variables with zero mean
and finite absolute moments of order $p\ge2$. Denote $S_n:=\xi_1+\ldots+\xi_n$. 
Then, for some $C_p$ which only depends on $p$,
\begin{eqnarray}\label{M-Z.ge2}
\E|S_n|^p &\le& C_pn^{p/2-1}\sum_{i=1}^n\E|\xi_i|^p.
\end{eqnarray}
If $\V\xi_i<\infty$ for all $i$, then for all $p\le 2$,
\begin{eqnarray}\label{M-Z.le2}
\E|S_n|^p &\le& \Bigl(\sum_{i=1}^n\V\xi_i\Bigr)^{p/2}.
\end{eqnarray}
In particular, if $\xi_i$'s are independent identically distributed 
random variables with finite moment of order $p\vee 2$, then
\begin{eqnarray}\label{M-Z.gen}
\E|S_n|^p &\le& Cn^{p/2}\quad\mbox{for all }n\ge 1\mbox{ and }p>0,
\end{eqnarray}
where 
\begin{eqnarray*}
C\ =\ C(p,\xi_1) &=& 
\left\{
\begin{array}{ll}
C_p\E\xi_1^p &\mbox{if }p>2,\\
(\V\xi_1)^{p/2} &\mbox{if }p\le2.
\end{array}
\right.
\end{eqnarray*}
\end{lemma}

\begin{proof}
For $p\ge 2$, it goes back to Dharmadhikari and Jogdeo \cite[Theorem 2]{DhJog}.

For $p\le 2$, the function $x^{p/2}$ is concave, so
\begin{eqnarray*}
\E|S_n|^p &\le& (\E S_n^2)^{p/2}\ =\
\Bigl(\sum_{i=1}^n\V\xi_i\Bigr)^{p/2},
\end{eqnarray*}
by the independence of $\xi_i$'s.
\qed\end{proof}

\section{Comments to Chapter \ref{ch:classification}}

First classification of nearest-neighbour Markov chains with drift of
order $c/x$ goes back to Harris\index{Harris} \cite{Harris1952} and
to Hodges\index{Hodges} and Rosenblatt\index{Rosenblatt} \cite{HR1953}.

A regular study of processes with asymptotically zero drift on $\R^+$
was initiated by Lamperti\index{Lamperti}
in a series of papers \cite{Lamp60, Lamp62,Lamp63}.\index{Lamperti}
In \cite[Theorem 2.2]{Lamp60} he showed that if $\limsup X_n=\infty$
and $\E|\xi(x)|^{2+\delta}$ are bounded for some positive $\delta$ then
\begin{itemize}
 \item $2x m_1(x)\le m_2(x)+O(x^{-\delta})$ yields recurrence of $X_n$,
 \item $2x m_1(x)\ge (1+\varepsilon)m_2(x)$ yields transience of $X_n$.
\end{itemize}
In \cite[Theorem 2.1]{Lamp63} he proved that $2xm_1(x)+m_2(x)\le-\varepsilon$ 
is sufficient for the positive recurrence of $\{X_n\}$.
It was shown in \cite[Theorem 3.1]{Lamp60} that
$2xm_1(x)+m_2(x)\ge\varepsilon$ implies that $\{X_n\}$ is non-positive
(either null-recurrent or transient) provided $xm_1(x)$ and
$m_2(x)$ are bounded and $m_4(x)=o(x^2)$.

These criteria were improved later by Menshikov,\index{Menshikov}
Asymont\index{Asymont} and Yasnogorodskii\index{Yasnogorodskii} \cite{AMI}.
Instead of the existence of $2+\delta$ bounded moment they assume that
$\E\xi^2(x)\log^{2+\delta}(1+|\xi(x)|)$ is bounded.
Moreover, they established more precise classification for positive
recurrence, null-recurrence and transience based on iterated logarithms
which are improved further in Corollaries \ref{cor:posrec.log},
\ref{cor:nonpos.log}, \ref{cor:rec.log} and \ref{cor:tr.log.drift}.

Corollary \ref{cor:posrec.small} on positive recurrence in the absence
of second moments goes back to
Korshunov\index{Korshunov} \cite[Theorem 5]{Korshunov96}.
Corollary 2.19 on transience in the absence of the second moments is due
to Menshikov\index{Menshikov} and Wade\index{Wade} \cite[Theorem 2.1]{MW2010};
we prove it under minimal moment conditions.
Sandri\'c \cite[Theorem 1.3]{Sandric2013}\index{Sandri\'c} has
managed to suggest
some sufficient condition for recurrence of a chain with drift of order
$c/x^\beta$ where jumps have moment of order $1+\beta$ infinite,
so results like Corollary \ref{cor:posrec.inf} do not work;
it is only done under the assumption that the tails of jumps
are regularly varying.

We are aware of two different approaches to proving non-positivity,
one is due to Lamperti\index{Lamperti} \cite{Lamp63} and another one
goes back to Asymont et al. \cite{AMI}.\index{Asymont}
In Theorem \ref{thm:nonpos} we follow the first approach significantly
improving the non-positivity results from both \cite{Lamp63} and \cite{AMI}.
\chapter{Down-crossing probabilities for transient Markov chain}
\chaptermark{Down-crossing probabilities}
\label{ch:return.transient}

In this chapter we consider a (right) transient Markov chain $\{X_n\}$ 
taking values in $\R$, that is, for any fixed $\widehat x\in\R$,
\begin{eqnarray*}
\P_x\{\tau_B<\infty\} &\to& 0\quad\mbox{as }x\to\infty,
\end{eqnarray*}
where $\tau_B:=\min\{n\ge 1:X_n\in B\}$, $B:=(-\infty,\widehat x]$.
We are interested in the rate of convergence to zero of this probability as $x\to\infty$.
It clearly depends on the asymptotic properties of the drift of $\{X_n\}$ at infinity.

\section{Markov chains with asymptotically zero drift:
slow decay of down-crossing probability}
\sectionmark{Markov chains with asymptotically zero drift}
\label{sec:heavy-taildness.return}

We start with the following result which states that,
for almost any Markov chain with asymptotically zero drift, 
the down-crossing probability decays slower than any exponential function.
\index{Down-crossing probability!heavy-tailedness}

\begin{theorem}\label{non.exp.return}
Let a Markov chain $\{X_n\}$ on $\R$ be such that
\begin{eqnarray}\label{non.gen.return.1}
\limsup_{x\to\infty}\ \E\{\xi(x);\ \xi(x)>-x\} &\le& 0
\end{eqnarray}
and, in addition, 
\begin{eqnarray}\label{non.gen.return}
\liminf_{x\to\infty}\ \E\{\xi^2(x);\ \xi(x)\in(-x,0)\} &>& 0.
\end{eqnarray}
Then there exists an $\widehat x$ such that, for $B:=(-\infty,\widehat x]$
and all $\lambda>0$,
\begin{eqnarray*}
e^{\lambda x}\P_x\{\tau_B<\infty\}\ \to\ \infty\quad\mbox{as }x\to\infty.
\end{eqnarray*}
\end{theorem}

\begin{proof}
Let $\lambda>0$. Consider a bounded decreasing function 
$U_\lambda(x):=\min(e^{-\lambda x},1)$. For all $x>0$,
\begin{eqnarray*}
\E(U_\lambda(x+\xi(x))-U_\lambda(x)) &\ge& 
\E\{e^{-\lambda(x+\xi(x))}-e^{-\lambda x};\ x+\xi(x)>0\}\\
&=& e^{-\lambda x}\E\{e^{-\lambda \xi(x)}-1;\ x+\xi(x)>0\}.
\end{eqnarray*}
Since $e^{-y}\ge 1-y$ for all $y$ and
$e^{-y}\ge 1-y+y^2/2$ for all $y<0$,
\begin{eqnarray*}
\lefteqn{\E\{e^{-\lambda \xi(x)}-1;\ x+\xi(x)>0\}}\\
&\ge& -\lambda\E\{\xi(x);\ x+\xi(x)>0\}
+\frac{\lambda^2}{2}\E\{\xi^2(x);\ \xi(x)\in(-x,0)\}.
\end{eqnarray*}
Then, due to the conditions \eqref{non.gen.return.1} and \eqref{non.gen.return}, 
there exists a sufficiently large $\widehat x_\lambda>0$ such that 
\begin{eqnarray*}
\E(U_\lambda(x+\xi(x))-U_\lambda(x)) &\ge& 0\quad\mbox{for all }x>\widehat x_\lambda.
\end{eqnarray*}
Therefore, the process $\{U_\lambda(X_{n\wedge\tau_{B_\lambda}})\}$ 
is a bounded submartingale, where $B_\lambda:=(-\infty,\widehat x_\lambda]$.
Hence by the optional stopping theorem, for $z>\widehat x_\lambda$ 
and $x\in(\widehat x_\lambda,z)$,
$$
\E_x U_\lambda(X_{\tau_{B_\lambda}\wedge\tau_{(z,\infty)}})
\ \ge\ \E_x U_\lambda(X_0)\ =\ U_\lambda(x).
$$
Letting $z\to\infty$ we conclude that
\begin{eqnarray*}
\E_x\{U_\lambda(X_{\tau_{B_\lambda}});\ \tau_{B_\lambda}<\infty\} &=& 
\lim_{z\to\infty}\E_x\{U_\lambda(X_{\tau_{B_\lambda}});\ \tau_{B_\lambda}<\tau_{(z,\infty)}\}\\
&=& \lim_{z\to\infty}\E_x U_\lambda(X_{\tau_{B_\lambda}\wedge\tau_{(z,\infty)}})
-\lim_{z\to\infty}\E_x\{U_\lambda(X_{\tau_{(z,\infty)}});\ \tau_{B_\lambda}>\tau_{(z,\infty)}\}\\
&\ge& U_\lambda(x)-0\ =\ U_\lambda(x).
\end{eqnarray*}
On the other hand, since $U_\lambda$ is bounded by $1$,
\begin{eqnarray*}
\E_x\{U_\lambda(X_{\tau_{B_\lambda}});\ \tau_{B_\lambda}<\infty\} 
&\le& \P_x\{\tau_{B_\lambda}<\infty\}.
\end{eqnarray*}
This allows us to deduce the lower bound
\begin{eqnarray*}
\P_x\{\tau_{B_\lambda}<\infty\} &\ge& U_\lambda(x)\ =\ e^{-\lambda x}
\quad\mbox{for all }x>\widehat x_\lambda,
\end{eqnarray*}
and hence the theorem conclusion follows with $\widehat x=\widehat x_1$
and $B:=(-\infty,\widehat x_1]$, because by the Markov property, 
for all $\lambda<1$ and $x>\widehat x_1$,
\begin{eqnarray*}
\P_x\{\tau_{B_1}<\infty\} &\ge& 
\P_x\{\tau_{B_\lambda}<\infty\}
\inf_{y\in(\widehat x_1,\widehat x_\lambda]} \P_y\{\tau_{B_1}<\infty\},
\end{eqnarray*}
and
\begin{eqnarray*}
\P_y\{\tau_{B_1}<\infty\} &\ge& U_1(y)\ =\ e^{-y}\\
&\ge& e^{-\widehat x_\lambda}\ >\ 0
\quad\mbox{ for all }y\in(\widehat x_1,\widehat x_\lambda].
\end{eqnarray*}
\qed\end{proof}

Let us show by example that the condition \eqref{non.gen.return}
which is a kind of non-degeneracy of jumps
is essential for the conclusion to hold.
Consider a skip-free Markov chain $\{X_n\}$ on $\Z^+$ described in
Section \ref{sec:intr.nnrm},
that is, $\xi(x)$ takes values $-1$, $1$ or $0$ only,
with probabilities $p_-(x)$, $p_+(x)$ and $p_0(x)$
respectively, $p_-(0)=0$. The hitting zero probability
is computed in \eqref{return.probab.pm.rp}, 
\begin{eqnarray*}
\P_x\{\tau_0<\infty\}
&=& \frac{\sum_{y=x}^\infty \prod_{k=1}^y\frac{p_-(k)}{p_+(k)}}
{\sum_{y=0}^\infty \prod_{k=1}^y\frac{p_-(k)}{p_+(k)}}\quad\mbox{for all }x>0.
\end{eqnarray*}
Consider the case where $p_+(x):=1/(x+1)$ and
$p_-(x):=1/2(x+1)$. In this case the drift is asymptotically
zero while the probability of hitting zero is exponentially decreasing, $1/2^x$.
Clearly, the condition \eqref{non.gen.return} fails here.

\section{Drift of order $1/x$}
\label{sec:return.1x}

In this section $r(x)>0$ is a bounded decreasing differentiable function
satisfying \eqref{cond.on.r} with $c=1$, that is,
\begin{eqnarray}\label{rx.ge.1x.pr}
0\ \ge\ r'(x) &\ge& -r^2(x)\quad\mbox{for all }x\ge 0,
\end{eqnarray}
which yields
\begin{eqnarray*}
r(x) &\ge& \frac{1}{c_1+x}\quad\mbox{for all }x\ge 0,
\end{eqnarray*}
where $c_1=1/r(0)$. Then, in particular,
\begin{eqnarray}\label{def.of.R.2}
R(x) := \int_0^x r(y)dy &\to& \infty\quad\mbox{as }x\to\infty;
\end{eqnarray}
hereinafter we define $R(x)=0$ for $x<0$.
The increasing function $R(x)$ is concave on the positive half line
because $r(x)$ is decreasing.
As shown in \eqref{R.h.plus} and \eqref{R.h.minus},
\begin{eqnarray}\label{R.r.c.12}
R(x)+\frac{h}{1+h} &\le& R(x+h/r(x))\ \le\ R(x)+h,\\
\label{R.r.c.3}
R(x)-\frac{h}{1-h} &\le& R(x-h/r(x))\ \le\ R(x)-h.
\end{eqnarray}
Then, as already discussed, $1/r(x)$ is a natural $x$-step 
responsible for constant increase of the function $R(x)$ and,
for any increasing function $s(x)$ of order $o(1/r(x))$,
\begin{eqnarray}\label{R.r.c.4}
R(x\pm s(x)) &=& R(x)+o(1),\\
\label{r.r.c.4}
r(x\pm s(x)) &\sim& r(x)\quad\mbox{as }x\to\infty.
\end{eqnarray}

Fix an increasing function $s(x)\to\infty$ as $x\to\infty$
such that $x-s(x)$ increases and $s(x)=o(x)$.

Specifically, in this section we consider a transient Markov chain $\{X_n\}$
whose jumps are such that
\begin{eqnarray}\label{m1.and.m2.return}
m_2^{[s(x)]}(x)\ \to\ b>0\quad\mbox{ and }
\quad m_1^{[s(x)]}(x)\ \sim\ \mu/x\quad\mbox{ as }x\to\infty,
\end{eqnarray}
where $\mu\ge b/2$. If $\mu>b/2$ then $\{X_n\}$ is transient, under some 
minor additional conditions, see Theorem \ref{thm:transience.inf}. 
If $\mu=b/2$ then $\{X_n\}$ can still be transient, provided there exists an appropriate 
logarithmic expansion of the first two truncated moments
of jumps, see Corollary \ref{cor:tr.log.drift} for details.
In addition, we assume that
\begin{eqnarray}\label{r-cond.2.return}
\frac{2m_1^{[s(x)]}(x)}{m_2^{[s(x)]}(x)}
&=& r(x)+o(p(x))\quad\mbox{as }x\to\infty
\end{eqnarray}
for some decreasing positive function $r(x)\to0$ satisfying $r(x)x\to 2\mu/b\ge 1$
as $x\to\infty$ and some decreasing integrable function $p(x)\ge 0$.
Since $p(x)$ is decreasing and integrable, $p(x)x\to0$ as $x\to\infty$.
We also assume that
\begin{eqnarray}\label{r.p.prime.return}
r'(x) &\sim& -r(x)/x\ \sim\ -(b/2\mu)r^2(x)
\quad\mbox{ and }\quad p'(x)\ =\ O(r^2(x)).
\end{eqnarray}
It follows from Lemma \ref{l:g.fin.p.2} that the condition on $p'(x)$
is always satisfied for a properly chosen function $p$.
Since $xr(x)\sim 2\mu/b\ge 1$,
\begin{eqnarray*}
R(x)\ =\ \int_0^x r(y)dy &\sim& \frac{2\mu}{b}\log x\quad\mbox{as }x\to\infty.
\end{eqnarray*}

Assume that the function $e^{-R(x)}$ is integrable at infinity,
which automatically holds if $2\mu/b>1$.
It allows us to define the following bounded decreasing function which plays
the most important r\^ole in our analysis of the down-crossing probability
for a transient Markov chain:
\begin{eqnarray}\label{def.u.return}
U(x) &:=& \int_x^\infty e^{-R(y)}dy\quad\mbox{for }x\ge 0;
\end{eqnarray}
and $U(x)=U(0)$ for $x\le 0$. Analogously to diffusion processes,
see Section \ref{subsec:diffusion.tr}, it is almost the scale function
for the chain $\{X_n\}$, see Corollary \ref {cor:lyapunov.return} below.

We have $U(x)\to 0$ as $x\to\infty$.
%This function solves the equation $U''+rU'=0$ on $\Rp$.
According to our assumptions,
$$
r(x)=\frac{2\mu}{b}\frac{1}{x}+\frac{\varepsilon(x)}{x},
$$
where $\varepsilon(x)\to0$ as $x\to\infty$.
In view of the representation theorem for slowly varying
functions, there exists a slowly varying at infinity function
$\ell(x)$ such that $e^{-R(x)}=x^{-\rho-1}\ell(x)$
and $U(x)\sim x^{-\rho} \ell(x)/\rho$ where $\rho:=2\mu/b-1\ge 0$.

The main result in this subsection is the following theorem
that provides lower and upper bounds for the down-crossing probability
of transient Markov chains with asymptotically zero drift described above.
\index{Down-crossing probability!drift of order $1/x$}

\begin{theorem}\label{thm:transient.return}
Let the drift conditions \eqref{m1.and.m2.return} and \eqref{r-cond.2.return}
be valid with $\mu\ge b/2$ and $r(x)$ satisfying 
the regularity condition \eqref{r.p.prime.return}.
Let the function $e^{-R(x)}$ be integrable at infinity and $\{X_n\}$ be a transient Markov chain.
Let, for some increasing $s(x)=o(x)$, the following integrability condition hold
\begin{eqnarray}\label{cond.3.moment.return}
\E\bigl\{|\xi(x)|^3;\ |\xi(x)|\le s(x)\bigr\} &=& o(p(x)/r^2(x))
\quad\mbox{as }x\to\infty.
\end{eqnarray}
If the right jump tails satisfy an upper bound
\begin{eqnarray}\label{cond.xi.le.return}
\P\{\xi(x)>s(x)\} &=& o(p(x)e^{-R(x)}/U(x))\quad\mbox{as }x\to\infty,
\end{eqnarray}
then there exist a constant $c_1>0$ and a level $\widehat x$ such that
$$
\P_x\{X_n\le x_0\mbox{ for some }n\}\ \ge\ c_1\frac{U(x)}{U(x_0)}
\quad\mbox{for all }x>x_0\ge\widehat x
$$
and, uniformly for all $x>x_0$,
$$
\P_x\{X_n\le x_0\mbox{ for some }n\}\ \ge\ (1+o(1))\frac{U(x)}{U(x_0)}
\quad\mbox{as }x_0\to\infty.
$$
If the negative jumps satisfy the following condition
\begin{eqnarray}\label{cond.xi.ge.return}
\E\bigl\{U(x+\xi(x));\ \xi(x)<-s(x)\bigr\} &=& o(p(x)e^{-R(x)})
\quad\mbox{as }x\to\infty,
\end{eqnarray}
then there exist a constant $c_2<\infty$ and a level $\widehat x$ such that
$$
\P_x\{X_n\le x_0\mbox{ for some }n\}\ \le\ c_2\frac{U(x)}{U(x_0)}
\quad\mbox{for all }x>x_0\ge\widehat x
$$
and, uniformly for all $x>x_0$,
$$
\P_x\{X_n\le x_0\mbox{ for some }n\}\ \le\ (1+o(1))\frac{U(x)}{U(x_0)}
\quad\mbox{as }x_0\to\infty.
$$
\end{theorem}

Compare to down-crossing results for Bessel processes, see \ref{bessel.tau1};
or nearest-neighbour Markov chains, see Section \ref{subsec:dcp.nnmc}.

In the case $\rho=2\mu/b-1>0$ the last asymptotic results may be specified as follows.
\index{Down-crossing probability!drift of order $1/x$}

\begin{corollary}\label{cor:transient.return}
Let $\{X_n\}$ be a transient Markov chain.
Let the drift conditions \eqref{m1.and.m2.return} and \eqref{r-cond.2.return}
be valid with $\mu>b/2$ and $r(x)$ satisfying 
the regularity condition \eqref{r.p.prime.return}.
Let, for some increasing $s(x)=o(x)$, the following integrability condition hold
\begin{eqnarray*}
\E\bigl\{|\xi(x)|^3;\ |\xi(x)|\le s(x)\bigr\} &=& o(p(x)x^2)
\quad\mbox{as }x\to\infty.
\end{eqnarray*}
If the right jump tails satisfy an upper bound
\begin{eqnarray*}
\P\{\xi(x)>s(x)\} &=& o(p(x)/x)\quad\mbox{as }x\to\infty,
\end{eqnarray*}
and the negative jumps satisfy the condition
\begin{eqnarray*}
\E\bigl\{U(x+\xi(x));\ \xi(x)<-s(x)\bigr\} &=& o(p(x)e^{-R(x)})
\quad\mbox{as }x\to\infty,
\end{eqnarray*}
then, for any $\varepsilon>0$,
$$
\P_x\{X_n\le\gamma x\mbox{ for some }n\}\ \to\ \gamma^\rho\quad\mbox{as }x\to\infty
$$
uniformly for all $\gamma\in(\varepsilon,1)$.
\end{corollary}

To specify the asymptotics in the case $\rho=2\mu/b-1=0$, 
we need to consider the logarithmic expansions 
of the first two truncated moments of jumps.
We assume that, for some $m\in\mathbb N$ and $\varepsilon>0$,
\begin{eqnarray}\label{rx.deco.cor3.4}
r(x) &=& \biggl(
\frac{1}{y}+\frac{1}{y\log y}+\ldots+\frac{1}{y\log y\cdot\ldots\cdot\log_{(m-1)}y}
+ \frac{1+\varepsilon}{y\log y\cdot\ldots\cdot\log_{(m)}y}\biggr)\bigg|_{y=x+e^{(m)}}.
\nonumber\\
\end{eqnarray}
Then
\begin{eqnarray*}
R(x) &=& \bigl(
\log y+\log\log y+\ldots+\log_{(m)}y
+ (1+\varepsilon)\log_{(m+1)}y\bigr)\big|_{y=x+e^{(m)}}\\
&& -\bigl(e^{(m-1)}+e^{(m-2)}+\ldots+1\bigr),
\end{eqnarray*}
and
\begin{eqnarray*}
U(x) &=& \frac{e^{(m)}e^{(m-1)}\ldots 1}{\varepsilon \log^\varepsilon_{(m)}(x+e^{(m)})}.
\end{eqnarray*}

\index{Down-crossing probability!drift of order $1/x$}
\begin{corollary}\label{cor:transient.log.return}
Let the drift conditions \eqref{m1.and.m2.return} and \eqref{r-cond.2.return}
be valid with $\mu=b/2$ and $r(x)$ satisfying \eqref{rx.deco.cor3.4} and
the regularity condition \eqref{r.p.prime.return}.
Let $\{X_n\}$ be a transient Markov chain.
Let, for some increasing $s(x)=o(x)$, the following integrability condition hold
\begin{eqnarray*}
\E\bigl\{|\xi(x)|^3;\ |\xi(x)|\le s(x)\bigr\} &=& o(p(x)x^2)
\quad\mbox{as }x\to\infty.
\end{eqnarray*}
If the right jump tails satisfy an upper bound
\begin{eqnarray*}
\P\{\xi(x)>s(x)\} &=& o(p(x)/x\log x\cdot\ldots\cdot\log_{(m)}x)\quad\mbox{as }x\to\infty,
\end{eqnarray*}
and the negative jumps satisfy the condition
\begin{eqnarray*}
\E\bigl\{1/\log_{(m)}^\varepsilon(x+\xi(x));\ \xi(x)<-s(x)\bigr\} &=& 
o(p(x)/x\log x\cdot\ldots\cdot\log^{1+\varepsilon}_{(m)}x)
\quad\mbox{as }x\to\infty,
\end{eqnarray*}
then, uniformly for all $x>x_0$,
$$
\P_x\{X_n\le x_0\mbox{ for some }n\}\ \sim\ 
\biggl(\frac{\log_{(m)}x_0}{\log_{(m)}x}\biggr)^\varepsilon\quad\mbox{as }x_0\to\infty.
$$
\end{corollary}
 
To prove Theorem \ref{thm:transient.return}, first let us prove some auxiliary results.
We start by defining decreasing Lyapunov functions needed.
Without loss of generality we assume that $p(x)\le r(x)$ for all $x$.
Consider the functions $r_+(x):=r(x)+p(x)$ and $r_-(x):=r(x)-p(x)$
and let
\begin{eqnarray}\label{Upm.def}
\nonumber
R_\pm(x) &:=& \int_0^x r_\pm(y)dy,\\
U_\pm(x) &:=& \int_x^\infty e^{-R_\pm(y)}dy,\quad x\ge 0,
\end{eqnarray}
and $U_\pm(x)=U_\pm(0)$ for $x\le 0$.
We have $0\le r_-(x)\le r(x)\le r_+(x)$, $0\le R_-(x)\le R(x)\le R_+(x)$ and
$U_-(x)\ge U(x)\ge U_+(x)>0$. Since
\begin{eqnarray*}
C_p &:=& \int_0^\infty p(y)dy\quad\mbox{is finite},
\end{eqnarray*}
we have
\begin{eqnarray}\label{equiv.for.R.return}
R_\pm(x) &=& R(x)\pm C_p+o(1)\quad\mbox{as }x\to\infty.
\end{eqnarray}
Therefore,
\begin{eqnarray}\label{equiv.for.U.return}
U_\pm(x) &\sim& e^{\mp C_p}U(x)\to 0 \quad\mbox{as }x\to\infty.
\end{eqnarray}

\begin{lemma}\label{L.Lyapunov.return}
If the integrability conditions \eqref{cond.3.moment.return} 
and \eqref{cond.xi.le.return} hold, then, as $x\to\infty$,
\begin{eqnarray}\label{b.for.u.return+}
\E\{U_+(x+\xi(x))-U_+(x);\ \xi(x)\ge -s(x)\} &\ge&
p(x)(1+o(1))e^{-R_+(x)}.
\end{eqnarray}
If the integrability conditions \eqref{cond.3.moment.return} 
and \eqref{cond.xi.ge.return} hold, then, as $x\to\infty$,
\begin{eqnarray}\label{b.for.u.return-}
\E U_-(x+\xi(x))-U_-(x) &\le& -p(x)(1+o(1))e^{-R_-(x)}.
\end{eqnarray}
\end{lemma}

Since the function $U_+$ is decreasing, the lower bound \eqref{b.for.u.return+}
yields that
\begin{eqnarray*}
\E U_+(x+\xi(x))-U_+(x) &\ge& p(x)(1+o(1))e^{-R_+(x)},
\end{eqnarray*}
which is symmetric to \eqref{b.for.u.return-}.
However it is stated as in \eqref{b.for.u.return+}
because we apply it to truncated Markov chains, 
see the proof of Theorem \ref{thm:transient.return}
in its part concerning the lower bound.

\begin{theopargself}
\begin{proof}[of Lemma \ref{L.Lyapunov.return}.]
We start with the following decomposition:
\begin{eqnarray}\label{L.harm2.1.return}
\lefteqn{\E U_\pm(x+\xi(x))-U_\pm(x)}\nonumber\\
&=& \E\{U_\pm(x+\xi(x))-U_\pm(x);\ \xi(x)<-s(x)\}\nonumber\\
&&\hspace{5mm} +\E\{U_\pm(x+\xi(x))-U_\pm(x);\ |\xi(x)|\le s(x)\}\nonumber\\
&&\hspace{10mm} +\E\{U_\pm(x+\xi(x))-U_\pm(x);\ \xi(x)>s(x)\}.
\end{eqnarray}
Here the third term on the right hand side is negative because
$U_\pm$ decreases and it may be bounded below as follows:
\begin{eqnarray}\label{L.harm2.2a.return}
\E\{U_\pm(x+\xi(x))-U_\pm(x);\ \xi(x)>s(x)\}
&\ge& -U_\pm(x)\P\{\xi(x)>s(x)\}\nonumber\\
&=& o\bigl(p(x)e^{-R_\pm(x)}\bigr),
\end{eqnarray}
provided the condition \eqref{cond.xi.le.return} holds
and due to the relations \eqref{equiv.for.R.return} and \eqref{equiv.for.U.return}.
Further, the first term on the right hand side of \eqref{L.harm2.1.return}
is positive and possesses the following upper bound:
\begin{eqnarray}\label{L.harm2.2b.return}
\E\{U_\pm(x+\xi(x))-U_\pm(x);\ \xi(x)<-s(x)\}
&\le& \E\{U_\pm(x+\xi(x));\ \xi(x)<-s(x)\}\nonumber\\
&=& o\bigl(p(x)e^{-R_\pm(x)}\bigr),
\end{eqnarray}
provided the condition \eqref{cond.xi.ge.return} holds
and due to the relations \eqref{equiv.for.R.return} and \eqref{equiv.for.U.return}.
To estimate the second term on the right hand side of \eqref{L.harm2.1.return},
we make use of Taylor's expansion:
\begin{eqnarray}\label{L.harm2.2.return}
\lefteqn{\E\{U_\pm(x+\xi(x))-U_\pm(x);\ |\xi(x)|\le s(x)\}}\nonumber\\
&&\hspace{15mm} =\ U_\pm'(x)\E\{\xi(x);|\xi(x)|\le s(x)\}
+\frac{1}{2} U_\pm''(x)\E\{\xi^2(x);|\xi(x)|\le s(x)\}\nonumber\\
&&\hspace{35mm} +\frac{1}{6}\E\bigl\{U_\pm'''(x+\theta\xi(x))\xi^3(x);
|\xi(x)|\le s(x)\bigr\},
\end{eqnarray}
where $0\le\theta=\theta(x,\xi(x))\le 1$. By the construction of $U_\pm$,
\begin{eqnarray}\label{U.12.prime.return}
U_\pm'(x)=-e^{-R_\pm(x)},\qquad
U_\pm''(x)=r_\pm(x)e^{-R_\pm(x)}=(r(x)\pm p(x))e^{-R_\pm(x)}.
\end{eqnarray}
Then it follows that
\begin{eqnarray}\label{L.harm2.2.1.return}
\lefteqn{U_\pm'(x)m_1^{[s(x)]}(x)+\frac{1}{2}U_\pm''(x)m_2^{[s(x)]}(x)}\nonumber\\
&&\hspace{10mm}=\ 
e^{-R_\pm(x)} \Bigl(-m_1^{[s(x)]}(x)+(r(x)\pm p(x))\frac{m_2^{[s(x)]}(x)}{2}\Bigr)\nonumber\\
&&\hspace{25mm}=\ \frac{m_2^{[s(x)]}(x)}{2}e^{-R_\pm(x)}
\biggl(-\frac{2m_1^{[s(x)]}(x)}{m_2^{[s(x)]}(x)}+r(x)\pm p(x)\biggr)\nonumber\\
&&\hspace{40mm}=\ \pm\frac{m_2^{[s(x)]}(x)}{2}e^{-R_\pm(x)} p(x) (1+o(1)),
\end{eqnarray}
by \eqref{r-cond.2.return}.
Finally, let us estimate the last term in \eqref{L.harm2.2.return}.
Notice that by the condition \eqref{r.p.prime.return} on the derivative
of $r(x)$ and $p(x)$,
\begin{eqnarray*}
U_\pm'''(x) &=& \bigl(r'(x)\pm p'(x)-(r(x)\pm p(x))^2\bigr)e^{-R_\pm(x)}\\
&=& O(r^2(x))e^{-R_\pm(x)},
\end{eqnarray*}
hence, due to \eqref{R.r.c.4} and \eqref{r.r.c.4},
\begin{eqnarray*}
U_\pm'''(x+y) &=& O(r^2(x))e^{-R_\pm(x)}
\end{eqnarray*}
as $x\to\infty$ uniformly for $|y|\le s(x)$ which implies
\begin{eqnarray*}
\bigl|\E\bigl\{U_\pm'''(x+\theta\xi(x))\xi^3(x);
|\xi(x)|\le s(x)\bigr\}\bigr|
&\le& c_1r^2(x)\E\bigl\{|\xi^3(x)|;\ |\xi(x)|\le s(x)\bigr\}e^{-R_\pm(x)}.
\end{eqnarray*}
Then, in view of \eqref{cond.3.moment.return},
\begin{eqnarray}\label{full.vs.cond.return}
\bigl|\E\bigl\{U_\pm'''(x+\theta\xi(x))\xi^3(x);\ |\xi(x)|\le s(x)\bigr\}\bigr|
&=& o\bigl(p(x)e^{-R_\pm(x)}\bigr).
\end{eqnarray}
Substituting \eqref{L.harm2.2.1.return} and \eqref{full.vs.cond.return}
into \eqref{L.harm2.2.return},  we obtain that
\begin{eqnarray}\label{L.harm2.2.2.return}
\E\{U_\pm(x+\xi(x))-U_\pm(x);\ |\xi(x)|\le s(x)\}
&=& \pm m_2^{[s(x)]}(x) p(x)(1+o(1))e^{-R_\pm(x)}.\nonumber\\[-1mm]
\end{eqnarray}
Substituting \eqref{L.harm2.2a.return}---or \eqref{L.harm2.2b.return}---and
\eqref{L.harm2.2.2.return} into \eqref{L.harm2.1.return},
we finally come to the desired conclusions.
\qed\end{proof}
\end{theopargself}

Lemma \ref{L.Lyapunov.return} implies the following result.

\begin{corollary}\label{cor:lyapunov.return}
Under the conditions of Lemma \ref{L.Lyapunov.return},
there exists an $\widehat x$ such that, for all $x>\widehat x$,
\begin{eqnarray*}
\E U_-(x+\xi(x))-U_-(x) &\le& 0,\\
\E\{U_+(x+\xi(x))-U_+(x);\ \xi(x)\ge -s(x)\} &\ge& 0.
\end{eqnarray*}
\end{corollary}

\begin{theopargself}
\begin{proof}[of Theorem \ref{thm:transient.return}]
The process $U_-(X_n)$ is bounded above by $U_-(0)$.
Let $\widehat x$ be any level guaranteed by the last corollary,
$x_0\ge\widehat x$,
$B=(-\infty,x_0]$ and $\tau_B=\min\{n\ge 1:X_n\in B\}$.

By Corollary \ref{cor:lyapunov.return}, $U_-(X_{n\wedge\tau_B})$
is a bounded supermartingale. 
Hence by the optional stopping theorem, for $z>\widehat x$ and $x\in(\widehat x,z)$,
$$
\E_x U_-(X_{\tau_B\wedge\tau_{(z,\infty)}})\ \le\ \E_x U_-(X_0)\ =\ U_-(x).
$$
Letting $z\to\infty$ we conclude that
\begin{eqnarray*}
\E_x\{U_-(X_{\tau_B});\ \tau_B<\infty\} &=& 
\lim_{z\to\infty}\E_x\{U_-(X_{\tau_B});\ \tau_B<\tau_{(z,\infty)}\}\\
&=& \lim_{z\to\infty}\E_x U_-(X_{\tau_B\wedge\tau_{(z,\infty)}})
-\lim_{z\to\infty}\E_x\{U_-(X_{\tau_{(z,\infty)}});\ \tau_B>\tau_{(z,\infty)}\}\\
&\le& U_-(x)-0\ =\ U_-(x).
\end{eqnarray*}
On the other hand, since $U_-$ is decreasing,
\begin{eqnarray*}
\E_x\{U_-(X_{\tau_B});\ \tau_B<\infty\} &\ge& U_-(x_0)\P_x\{\tau_B<\infty\}.
\end{eqnarray*}
Therefore,
\begin{eqnarray}\label{tau.B.infty.U-}
\P_x\{\tau_B<\infty\} &\le& \frac{U_-(x)}{U_-(x_0)},
\end{eqnarray}
which implies both upper bounds of the theorem,
by \eqref{equiv.for.U.return}.

On the other hand, let
$$
U_{+0}(x)\ :=\ \left\{
\begin{array}{ll}
U_+(x_0-s(x_0)) &\mbox{if }x\le x_0-s(x_0);\\
U_+(x) &\mbox{if }x>x_0-s(x_0).
\end{array}
\right.
$$
Due to the increase of $x-s(x)$,
$$
\E\{U_{+0}(x+\xi(x));\ \xi(x)\ge -s(x)\}\ =\ \E\{U_+(x+\xi(x));\ \xi(x)\ge -s(x)\}
$$
for all $x>x_0$.
Therefore the process $\{U_{+0}(X_{n\wedge\tau_B})\}$ is a bounded submartingale
due to the lower bound provided by Corollary \ref{cor:lyapunov.return}.
Hence again by the optional stopping theorem, for $x>x_0$,
$$
\E_x\{U_{+0}(X_{\tau_B});\ \tau_B<\infty\}\ \ge\ \E_x U_{+0}(X_0)\ =\ U_+(x).
$$
On the other hand, since $U_{+0}$ is bounded by $U_+(x_0-s(x_0))$,
$$
\E_x\{U_{+0}(X_{\tau_B});\ \tau_B<\infty\}\
\le\ U_+(x_0-s(x_0))\P_x\{\tau_B<\infty\}.
$$
This allows us to deduce a lower bound
\begin{eqnarray*}
\P_x\{\tau_B<\infty\} &\ge& \frac{U_+(x)}{U_+(x_0-s(x_0))},
\end{eqnarray*}
which completes the proof of both lower bounds,
due to \eqref{R.r.c.4} and \eqref{equiv.for.U.return}.
\qed\end{proof}
\end{theopargself}

\section{The case where $xm_1(x)\to\infty$ but $m_1(x)=o(1/\sqrt x)$}
\label{subsec:return.hx}

In this section we consider a transient Markov chain $\{X_n\}$
whose jumps are such that
\begin{eqnarray}\label{m1.and.m2.return.hx}
m_2^{[s(x)]}(x)\ \to\ b>0\quad\mbox{ and }
\quad xm_1^{[s(x)]}(x)\ \to\ \infty\quad\mbox{ as }x\to\infty,
\end{eqnarray}
for some increasing function $s(x)=o(x)$, which implies transience subject
to some minor additional conditions, see Theorem \ref{thm:transience.inf}.
In addition, we assume that
\begin{eqnarray}\label{r-cond.2.return.hx}
\frac{2m_1^{[s(x)]}(x)}{m_2^{[s(x)]}(x)}
&=& r(x)+o(p(x))\quad\mbox{as }x\to\infty
\end{eqnarray}
for some decreasing positive differentiable function $r(x)\to0$
satisfying $r(x)x\to\infty$ as $x\to\infty$
and some decreasing differentiable function $p(x)\ge 0$
which is assumed to be integrable,
\begin{eqnarray}\label{4.p.int.hx}
C_p\ :=\ \int_0^\infty p(x)dx &<& \infty.
\end{eqnarray}
Since $p(x)$ is decreasing and integrable, $p(x)x\to0$ as $x\to\infty$.

In this subsection we consider the case where
$r(x)=o(1/\sqrt x)$, more precisely,
\begin{eqnarray}\label{4.r.2.hx}
r^2(x) &=& o(p(x))\quad\mbox{as }x\to\infty.
\end{eqnarray}
We also assume that
\begin{eqnarray}\label{r.p.prime.return.hx}
p'(x)\ =\ o(r^2(x))\ \mbox{ and }\
r'(x) &=& o(r^2(x)) \quad\mbox{as }x\to\infty.
\end{eqnarray}
In view of \eqref{m1.and.m2.return.hx}, the condition \eqref{r-cond.2.return.hx}
is equivalent to
\begin{eqnarray}\label{4.r-cond.2.equiv.hx}
-m^{[s(x)]}_1(x)+\frac{m^{[s(x)]}_2(x)}{2}r(x)
&=& o(p(x))\quad\mbox{as }x\to\infty.
\end{eqnarray}

Define the increasing function $R(x)$ as in \eqref{def.of.R.2}.
Since $xr(x)\to \infty$, the function $e^{-R(x)}$ is integrable at infinity.
It allows us to define the decreasing function $U(x)$ as in \eqref{def.u.return}
which plays a key r\^ole in the next result.
\index{Down-crossing probability!drift slower than $1/x$}

\begin{theorem}\label{thm:transient.return.hx}
Let $\{X_n\}$ be a transient Markov chain whose first two moments
of jumps truncated at some level $s(x)=o(1/r(x))$ satisfy
\eqref{m1.and.m2.return.hx} and \eqref{r-cond.2.return.hx}
while $r(x)$ satisfies \eqref{4.r.2.hx}. 
Assume the regularity condition \eqref{r.p.prime.return.hx}.
Let the following integrability condition on jumps hold,
\begin{eqnarray}\label{cond.3.moment.return.hx}
\E\bigl\{|\xi(x)|^3;\ |\xi(x)|\le s(x)\bigr\} &=& o(p(x)/r^2(x))
\quad\mbox{as }x\to\infty.
\end{eqnarray}
If the right jump tails satisfy an upper bound
\begin{eqnarray}\label{cond.xi.le.return.hx}
\P\{\xi(x)>s(x)\} &=& o(p(x)r(x))\quad\mbox{as }x\to\infty,
\end{eqnarray}
then there exist a constant $c_1>0$ and a level $\widehat x$ such that
$$
\P_x\{X_n\le x_0\mbox{ for some }n\}\ \ge\ c_1\frac{U(x)}{U(x_0)}
\quad\mbox{for all }x>x_0\ge\widehat x
$$
and, uniformly for all $x>x_0$,
$$
\P_x\{X_n\le x_0\mbox{ for some }n\}\ \ge\ (1+o(1))\frac{U(x)}{U(x_0)}
\quad\mbox{as }x_0\to\infty.
$$
If the negative jumps satisfy the following condition
\begin{eqnarray}\label{cond.xi.ge.return.hx}
\E\bigl\{U(x+\xi(x));\ \xi(x)<-s(x)\bigr\} &=& o\bigl(p(x)e^{-R(x)}\bigr)
\quad\mbox{as }x\to\infty,
\end{eqnarray}
then there exist a constant $c_2<\infty$ and a level $\widehat x$ such that
$$
\P_x\{X_n\le x_0\mbox{ for some }n\}\ \le\ c_2\frac{U(x)}{U(x_0)}
\quad\mbox{for all }x>x_0\ge\widehat x
$$
and, uniformly for all $x>x_0$,
$$
\P_x\{X_n\le x_0\mbox{ for some }n\}\ \le\ (1+o(1))\frac{U(x)}{U(x_0)}
\quad\mbox{as }x_0\to\infty.
$$
\end{theorem}

Notice that the right hand side of \eqref{cond.3.moment.return.hx}
may be bounded away from $0$ in the only case where $p(x)/r^2(x)\to\infty$,
which is equivalent to the condition \eqref{4.r.2.hx}.

To prove the last theorem, we consider the same functions
$r_\pm(x)$, $R_\pm(x)$ and $U_\pm(x)$ as in the previous subsection.
The only difference is that, due to \eqref{r.p.prime.return.hx},
\begin{eqnarray*}
\frac{U'(x)}{(\frac{1}{r(x)} e^{-R(x)})'} &=&
\frac{-e^{-R(x)}}{(-r'(x)/r^2(x)-1) e^{-R(x)}}\ \to\ 1
\quad\mbox{as }x\to\infty,
\end{eqnarray*}
so L'H\^opital's rule yields
\begin{eqnarray}\label{equiv.for.U.1.return.hx}
U(x) &\sim& \frac{1}{r(x)} e^{-R(x)}
\quad\mbox{as }x\to\infty.
\end{eqnarray}
Then similarly to Lemma \ref{L.Lyapunov.return} the following result holds.

\begin{lemma}\label{L.Lyapunov.return.hx}
If the integrability conditions \eqref{cond.3.moment.return.hx}
and \eqref{cond.xi.le.return.hx} hold, then, as $x\to\infty$,
\begin{eqnarray}\label{b.for.u.return.1}
\E\{U_+(x+\xi(x))-U_+(x);\ \xi(x)\ge -s(x)\} &\ge&
\frac{b+o(1)}{2}p(x)e^{-R_+(x)}.
\end{eqnarray}
If the integrability conditions \eqref{cond.3.moment.return.hx}
and \eqref{cond.xi.ge.return.hx} hold, then
\begin{eqnarray}\label{b.for.u.return.2}
\E U_-(x+\xi(x))-U_-(x) &\le& -\frac{b+o(1)}{2}p(x)e^{-R_-(x)}
\quad\mbox{as }x\to\infty.
\end{eqnarray}
\end{lemma}

\begin{proof}
The calculations are the same as in Lemma \ref{L.Lyapunov.return}
apart from the estimation of the third derivative of $U_\pm$.
By the condition \eqref{r.p.prime.return.hx} on the derivatives
of $r(x)$ and $p(x)$,
\begin{eqnarray*}
U_\pm'''(x) &=& \bigl(r'(x)\pm p'(x)+(r(x)\pm p(x))^2\bigr)e^{-R_\pm(x)}\\
&=& O\bigl(r^2(x)e^{-R_\pm(x)}\bigr).
\end{eqnarray*}
As is shown in \eqref{R.r.c.4}, $R(x+s(x))=R(x)+o(1)$ for any $s(x)=o(1/r(x))$.
Therefore,
\begin{eqnarray}\label{full.vs.cond.return.hx}
\bigl|\E\bigl\{U_\pm'''(x+\theta\xi(x))\xi^3(x);
|\xi(x)|\le s(x)\bigr\}\bigr|
&\le& c_1r^2(x)\E\bigl\{|\xi^3(x)|;\ |\xi(x)|\le s(x)\bigr\} e^{-R_\pm(x)}\nonumber\\
&=& o\bigl(p(x) e^{-R_\pm(x)}\bigr),
\end{eqnarray}
owing to the condition \eqref{cond.3.moment.return.hx}
on the third absolute moment.

This upper bound makes it possible to conclude the desired results
in the same way as it is done in Lemma \ref{L.Lyapunov.return}.
\qed\end{proof}

Lemma \ref{L.Lyapunov.return.hx} implies the following result.

\begin{corollary}\label{cor:lyapunov.return.hx}
There exists an $\widehat x$ such that, for all $x>\widehat x$,
\begin{eqnarray*}
\E U_-(x+\xi(x))-U_-(x) &\le& 0,\\
\E\{U_+(x+\xi(x))-U_+(x);\ \xi(x)\ge -s(x)\} &\ge& 0.
\end{eqnarray*}
\end{corollary}

The last corollary allows us to conclude the proof of
Theorem \ref{thm:transient.return.hx} in the same way as
that of Theorem \ref{thm:transient.return}.

\section{General case where $xm_1(x)\to\infty$}
\label{subsec:return.hy}

If $r(x)$ decreases slower than $1/\sqrt x$, then the function $r^2(x)$
is not integrable and, since $U_\pm'''(x)$ is of order $r^2(x)e^{-R_\pm(x)}$,
it does not possess a bound like $o\bigl(p(x)e^{-R_\pm(x)}\bigr)$.
So, the last term in Taylor's expansion \eqref{L.harm2.2.return}
is not negligible and instead it makes a significant contribution
to the drift of $U_\pm$. If $r(x)$ is sandwiched between $1/\sqrt x$
and $1/\sqrt[3]x$, then we need to consider Taylor's expansion
that includes the forth derivative of $U_\pm$ and, consequently,
the forth moment of jumps. More slower decreasing $r(x)$ is,
the higher moments of jumps are required.

So, in this subsection we consider the same setting as in the last one
but now we consider a general case and do not assume that $r(x)=o(1/\sqrt x)$.
Instead, we assume that, for some $\gamma\in\{2,3,4,\ldots\}$,
\begin{eqnarray}\label{4.r.gamma.hy}
r^\gamma(x) &=& o(p(x))\quad\mbox{as }x\to\infty
\end{eqnarray}
and
\begin{eqnarray}\label{4.r-cond.function.gen.hy}
-m^{[s(x)]}_1(x)+\sum_{j=2}^{\gamma} (-1)^j\frac{m^{[s(x)]}_j(x)}{j!}r^{j-1}(x)
&=& o(p(x))\quad\mbox{as }x\to\infty.
\end{eqnarray}
We further assume that the function $r(x)$ is $\gamma$ times
differentiable and, for all $1\le k\le\gamma-1$,
\begin{eqnarray}\label{4.r.prime.gen.hy}
r^{(k)}(x) = o(r^\gamma(x)), &&
p^{(k)}(x) = o(r^\gamma(x))\quad\mbox{as }x\to\infty.
\end{eqnarray}
If $r(x)\sim c/x^\alpha$ where $\gamma\alpha<2$,
then it follows from Lemma \ref{l:g.fin.p.2} that the condition on
the derivatives of $p(x)$
is always satisfied for a properly chosen function $p$,
so the condition \eqref{4.r.prime.gen.hy} on the derivatives
of $p$ does not restrict generality under this specific choice of $r(x)$.

In the next result, we consider the same functions $R(x)$ and $U(x)$
as in the previous subsection.
\index{Down-crossing probability!drift slower than $1/x$}

\begin{theorem}\label{thm:transient.return.hy}
Let $\{X_n\}$ be a transient Markov chain whose first $\gamma$ moments
of jumps truncated at some level $s(x)=o(1/r(x))$ satisfy the conditions
\eqref{m1.and.m2.return.hx} and \eqref{4.r-cond.function.gen.hy}
where $\gamma$ is defined in \eqref{4.r.gamma.hy}.
Assume the regularity condition \eqref{4.r.prime.gen.hy}
and the integrability condition
\begin{eqnarray}\label{cond.3.moment.return.hy}
\E\bigl\{|\xi(x)|^{\gamma+1};\ |\xi(x)|\le s(x)\bigr\}
&=& o(p(x)/r^\gamma(x))\quad\mbox{as }x\to\infty.
\end{eqnarray}
If the right jump tails satisfy an upper bound
\begin{eqnarray}\label{cond.xi.le.return.hy}
\P\{\xi(x)>s(x)\} &=& o(p(x)r(x))\quad\mbox{as }x\to\infty,
\end{eqnarray}
then there exist a constant $c_1>0$ and a level $\widehat x$ such that
$$
\P_x\{X_n\le x_0\mbox{ for some }n\}\ \ge\ c_1\frac{U(x)}{U(x_0)}
\quad\mbox{for all }x>x_0\ge\widehat x
$$
and, uniformly for all $x>x_0$,
$$
\P_x\{X_n\le x_0\mbox{ for some }n\}\ \ge\ (1+o(1))\frac{U(x)}{U(x_0)}
\quad\mbox{as }x_0\to\infty.
$$
If the negative jumps satisfy the following condition
\begin{eqnarray}\label{cond.xi.ge.return.hy}
\E\bigl\{U(x+\xi(x));\ \xi(x)<-s(x)\bigr\} &=& o\bigl(p(x)e^{-R(x)}\bigr)
\quad\mbox{as }x\to\infty,
\end{eqnarray}
then there exist a constant $c_2<\infty$ and a level $\widehat x$ such that
$$
\P_x\{X_n\le x_0\mbox{ for some }n\}\ \le\ c_2\frac{U(x)}{U(x_0)}
\quad\mbox{for all }x>x_0\ge\widehat x
$$
and, uniformly for all $x>x_0$,
$$
\P_x\{X_n\le x_0\mbox{ for some }n\}\ \le\ (1+o(1))\frac{U(x)}{U(x_0)}
\quad\mbox{as }x_0\to\infty.
$$
\end{theorem}

Notice that the right hand side of \eqref{cond.3.moment.return.hy}
may be bounded away from $0$ in the only case where
$p(x)/r^\gamma(x)\to\infty$
which is equivalent to the condition \eqref{4.r.gamma.hy}.

We consider the same functions $r_\pm(x)$, $R_\pm(x)$ and $U_\pm(x)$
as in the previous subsection and similarly
to Lemma \ref{L.Lyapunov.return.hx} we get the following result.

\begin{lemma}\label{L.Lyapunov.return.hy}
If the integrability conditions \eqref{cond.3.moment.return.hy}
and \eqref{cond.xi.le.return.hy} hold, then, as $x\to\infty$,
\begin{eqnarray}\label{b.for.u.return.1.hy}
\E\{U_+(x+\xi(x))-U_+(x);\ \xi(x)\ge-s(x)\} &\ge&
\frac{b+o(1)}{2}p(x)e^{-R_+(x)}.
\end{eqnarray}
If the integrability conditions \eqref{cond.3.moment.return.hy}
and \eqref{cond.xi.ge.return.hy} hold, then
\begin{eqnarray}\label{b.for.u.return.2.hy}
\E U_-(x+\xi(x))-U_-(x) &\le& -\frac{b+o(1)}{2}p(x)e^{-R_-(x)}
\quad\mbox{as }x\to\infty.
\end{eqnarray}
\end{lemma}

\begin{proof}
We start with the decomposition \eqref{L.harm2.1.return},
where the first and third terms on the right hand side possess
the same bounds as in the proof of Lemma \ref{L.Lyapunov.return.hx}.

To estimate the second term on the right hand side of \eqref{L.harm2.1.return},
we make use of Taylor's expansion with $\gamma+1$ terms:
\begin{eqnarray}\label{4.L.harm2.2.hy}
\lefteqn{\E\{U_\pm(x+\xi(x))-U_\pm(x);\ |\xi(x)|\le s(x)\}}\nonumber\\
&=& \sum_{k=1}^\gamma\frac{U_\pm^{(k)}(x)}{k!} m^{[s(x)]}_k(x)
+\E\Bigl\{\frac{U_\pm^{(\gamma+1)}(x+\theta\xi(x))}{(\gamma+1)!}\xi^{\gamma+1}(x);\
|\xi(x)|\le s(x)\Bigr\},\nonumber\\[-2mm]
\end{eqnarray}
where $0\le\theta=\theta(x,\xi(x))\le 1$. By the construction of $U_\pm$,
\begin{eqnarray}\label{4.U.12.prime.hy}
U_\pm'(x)=-e^{-R_\pm(x)},\qquad
U_\pm''(x)=r_\pm(x)e^{-R_\pm(x)}=(r(x)\pm p(x))e^{-R_\pm(x)},
\end{eqnarray}
and, for $k=3$, \ldots, $\gamma+1$,
\begin{eqnarray*}
U_\pm^{(k)}(x) = -(e^{-R_\pm(x)})^{(k-1)}
&=& (-1)^k\bigl(r_\pm^{k-1}(x)+o(p(x))\bigr)e^{-R_\pm(x)}
\quad\mbox{as }x\to\infty,
\end{eqnarray*}
where the remainder terms in the parentheses on the right
are of order $o(p(x))$ by the conditions \eqref{4.r.prime.gen.hy}
and \eqref{4.r.gamma.hy}.
By the definition of $r_\pm(x)$,
\begin{eqnarray*}
r_\pm^{k-1}(x) &=& (r(x)\pm p(x))^{k-1}= r^{k-1}(x)+o(p(x))
\quad\mbox{for all }k\ge 3,
\end{eqnarray*}
which implies the relation
\begin{eqnarray}\label{4.U.k.prime.hy}
U_\pm^{(k)}(x) &=& (-1)^k\bigl(r^{k-1}(x)+o(p(x))\bigr)e^{-R_\pm(x)}
\quad\mbox{as }x\to\infty.
\end{eqnarray}
It follows from the equalities \eqref{4.U.12.prime.hy}
and \eqref{4.U.k.prime.hy} that
\begin{eqnarray}\label{4.L.harm2.2.1.hy}
\lefteqn{\sum_{k=1}^\gamma \frac{U_\pm^{(k)}(x)}{k!}m^{[s(x)]}_k(x)}\nonumber\\
&=& e^{-R_\pm(x)} \biggl(\sum_{k=1}^\gamma
(-1)^k\frac{r^{k-1}(x)}{k!}m^{[s(x)]}_k(x)
+o(p(x))\pm p(x)\frac{m^{[s(x)]}_2(x)}{2}\biggr)\nonumber\\
&=& e^{-R_\pm(x)} \biggl(o(p(x))\pm p(x)\frac{m^{[s(x)]}_2(x)}{2}\biggr),
\end{eqnarray}
by the condition \eqref{4.r-cond.function.gen.hy}.
Owing to the condition \eqref{4.r.prime.gen.hy}
on the derivatives of $r(x)$ and \eqref{4.r.gamma.hy},
\begin{eqnarray*}
U_\pm^{(\gamma+1)}(x) &=&
(-1)^{\gamma+1}(r^\gamma(x)+o(r^\gamma(x))) e^{-R_\pm(x)}.
\end{eqnarray*}
Then, similarly to \eqref{full.vs.cond.return.hx}, the last term in
\eqref{4.L.harm2.2.hy} possesses the following bound:
\begin{eqnarray*}
\lefteqn{\Bigl|\E\Bigl\{\frac{U_\pm^{(\gamma+1)}(x+\theta\xi(x))}{(\gamma+1)!}\xi^{\gamma+1}(x);\
|\xi(x)|\le s(x)\Bigr\}\Bigr|}\\
&&\hspace{40mm} \le O\bigl(r^\gamma(x)e^{-R_\pm(x)}\bigr)
\E\bigl\{|\xi(x)|^{\gamma+1};\ |\xi(x)|\le s(x)\bigr\}\\
&&\hspace{40mm} = o\bigl(p(x)e^{-R_\pm(x)}\bigr),
\end{eqnarray*}
by the condition \eqref{cond.3.moment.return.hy}. Therefore, it follows from
\eqref{4.L.harm2.2.hy} and \eqref{4.L.harm2.2.1.hy} that
\begin{eqnarray*}
\lefteqn{\E\{U_\pm(x+\xi(x))-U_\pm(x);\ |\xi(x)|\le s(x)\}}\nonumber\\
&&\hspace{20mm} =\pm p(x)\frac{m^{[s(x)]}_2(x)}{2} e^{-R_\pm(x)}
+o\bigl(p(x)e^{-R_\pm(x)}\bigr).
\end{eqnarray*}
Together with \eqref{L.harm2.2a.return}, \eqref{L.harm2.2b.return},
and \eqref{L.harm2.1.return} this completes the proof.
\qed\end{proof}

Lemma \ref{L.Lyapunov.return.hy} implies an analogue of Corollary
\ref{cor:lyapunov.return.hx} which allows us to conclude the proof of
Theorem \ref{thm:transient.return.hy} in the same way as of
Theorem \ref{thm:transient.return}.

\section{Upper bound for down-crossing probability}
\label{sec:escape}

Now we produce some upper bounds for the down-crossing probability
for a transient Markov chain which are rough versions of
more precise bounds derived in the previous sections.
The main goal is to have upper bounds under weaker moment conditions than above.

Assume that there exists an $\widehat x$ such that
\begin{eqnarray}\label{r-cond.2}
\frac{2m^{[s(x)]}_1(x)}{m^{[s(x)]}_2(x)}
&\ge& r(x)\ >\ \frac{1}{x}\quad\mbox{for all }x>\widehat x,
\end{eqnarray}
where a decreasing differentiable function $r(x)$ satisfies the condition 
\begin{eqnarray}\label{rx.ge.1x.pr.eps}
r'(x) &\ge& -(1-\varepsilon)r^2(x),\quad \varepsilon>0,\quad\mbox{for all }x>\widehat x.
\end{eqnarray}
Then the drift to the right dominates the diffusion and the corresponding
Markov chain $X$ is typically transient, see Theorem \ref{thm:transience.inf}.
\index{Down-crossing probability!upper bound}

\begin{theorem}\label{l:est.for.return}
Assume that the drift of $\{X_n\}$ possesses the lower bound \eqref{r-cond.2},
with some $r(x)$ satisfying the condition \eqref{rx.ge.1x.pr.eps}, 
and $s(x)=o(1/r(x))$. Let, for some $\delta<\varepsilon$,
\begin{eqnarray}\label{rec.4}
\E\{e^{-\delta R(x+\xi(x))};\ \xi(x)<-s(x)\} &=&
o\bigl(r^2(x)e^{-\delta R(x)} m_2^{[s(x)]}(x)\bigr)\ \mbox{as }x\to\infty. 
\end{eqnarray}
Then there exists an $x_*$ such that, for all $y>x\ge x_*$,
\begin{eqnarray*}
\P_y\{X_n\le x\mbox{ for some }n\ge1\} &\le& e^{\delta(R(x)-R(y))}.
\end{eqnarray*}
In particular, for any fixed $h>0$,
\begin{eqnarray*}
\P_x\{X_n\le x-h/r(x)\mbox{ for some }n\ge1\} &\le& e^{-\delta h/2}\quad\mbox{ultimately in }x.
\end{eqnarray*}
\end{theorem}

The condition \eqref{rx.ge.1x.pr.eps} is satisfied for $r(x)=(1+2\varepsilon)/(1+x)$, 
hence the following corollary. \index{Down-crossing probability!upper bound}

\begin{corollary}\label{c:est.for.return.power}
Assume that the drift of $\{X_n\}$ possesses the lower bound \eqref{r-cond.2}
with $r(x)=(1+2\varepsilon)/x$ for some $\varepsilon\in(0,1/2]$, and $s(x)=o(x)$. 
Let, for some $\delta\in(0,\varepsilon)$,
\begin{eqnarray}\label{rec.4.power}
\E\{(x+\xi(x))^{-\delta};\ \xi(x)<-s(x)\} &=&
o\bigl(m_2^{[s(x)]}(x)/x^{2+\delta}\bigr)\ \mbox{as }x\to\infty. 
\end{eqnarray}
Then there exists an $x_*$ such that, for all $y>x\ge x_*$,
\begin{eqnarray*}
\P_y\{X_n\le x\mbox{ for some }n\ge1\} &\le& \biggl(\frac{1+x}{1+y}\biggr)^\delta.
\end{eqnarray*}
\end{corollary}

The condition \eqref{rx.ge.1x.pr.eps} is also satisfied 
for $r(x)=c/(1+x)^\beta$, $c>0$, $\beta\in(0,1)$, with any $\varepsilon\in(0,1)$.
Thus the following corollary holds true. 
\index{Down-crossing probability!upper bound}

\begin{corollary}\label{c:est.for.return.W}
Assume that the drift of $\{X_n\}$ possesses the lower bound \eqref{r-cond.2}
with $r(x)=c/(1+x)^\beta$ for some $c>0$, $\beta\in(0,1)$, and $s(x)=o(x^\beta)$. 
Let, for some $\delta>0$,
\begin{eqnarray}\label{rec.4.W}
\E\{e^{-\delta(x+\xi(x))^{1-\beta}};\ \xi(x)<-s(x)\} &=&
o\bigl(m_2^{[s(x)]}(x)e^{-\delta x^{1-\beta}}\bigr)\ \mbox{as }x\to\infty.
\end{eqnarray}
Then there exists an $x_*$ such that, for all $y>x\ge x_*$,
\begin{eqnarray*}
\P_y\{X_n\le x\mbox{ for some }n\ge1\} &\le& e^{\delta(x^{1-\beta}-y^{1-\beta})}.
\end{eqnarray*}
\end{corollary}

\begin{theopargself}
\begin{proof}[of Theorem \ref{l:est.for.return}]
Consider a decreasing test function $W(x):=e^{-\delta R(x)}$,
which is bounded by $1$. Let us prove that the mean drift
of $W(x)$ is negative for all sufficiently large $x$.
Indeed, since the function $W(x)$ decreases,
\begin{eqnarray}\label{return.deco}
\E W(x+\xi(x))-W(x) &\le& 
\E\{W(x+\xi(x))-W(x);\ \xi(x)\le s(x)\}\nonumber\\
&\le& \E\{W(x+\xi(x));\ \xi(x)<-s(x)\}\nonumber\\
&& +W'(x)\E\{\xi(x);\ |\xi(x)|\le s(x)\}\nonumber\\
&& +\frac{1}{2}W''(x+\theta\xi(x))\E\{\xi^2(x);\ |\xi(x)|\le s(x)\}\nonumber\\
&=:& E_1+E_2+E_3,
\end{eqnarray}
where $0\le\theta=\theta(x,\xi(x))\le 1$, by Taylor's expansion.
By the condition \eqref{rec.4},
the first term on the right hand side is of order
\begin{eqnarray}\label{return.E1}
E_1 &=& o\bigl(r^2(x)W(x)m_2^{[s(x)]}(x)\bigr)\quad\mbox{as }x\to\infty.
\end{eqnarray}
The second term on the right hand side of \eqref{return.deco} equals
\begin{eqnarray}\label{return.E2}
E_2 &=& -\delta r(x) W(x) m_1^{[s(x)]}(x)\nonumber\\
&\le& -\frac{\delta}{2} r^2(x)W(x)m_2^{[s(x)]}(x)
\quad\mbox{for }x\ge\widehat x,
\end{eqnarray}
due to \eqref{r-cond.2}.
In order to bound the third term on the right hand side of \eqref{return.deco},
we first notice that, due to \eqref{rx.ge.1x.pr.eps},
\begin{eqnarray*}
W''(x) &=& \delta\bigl(\delta r^2(x)- r'(x)\bigr) W(x)\\
&\le& \delta(\delta+1-\varepsilon) r^2(x) W(x)\quad\mbox{for }x\ge 0.
\end{eqnarray*}
By \eqref{r.r.c.4} and \eqref{R.r.c.4},
\begin{eqnarray*}
W''(x+y) &\le& \delta(\delta+1-\varepsilon)(1+o(1)) r^2(x) W(x)
\end{eqnarray*}
as $x\to\infty$ uniformly for all $|y|\le s(x)=o(1/r(x))$. Thus
\begin{eqnarray}\label{return.E3}
E_3 &\le& \frac{\delta}{2}(\delta+1-\varepsilon)(1+o(1)) r^2(x) W(x) m_2^{[s(x)]}(x)
\quad\mbox{as }x\to\infty.
\end{eqnarray}
Substituting \eqref{return.E1}--\eqref{return.E3} into
\eqref{return.deco} we deduce that
\begin{eqnarray*}
\E W(x+\xi(x))-W(x) &\le& 
\frac{\delta}{2}\bigl(\delta-\varepsilon+o(1)\bigl) r^2(x) W(x)m_2^{[s(x)]}(x)
\quad\mbox{as }x\to\infty.
\end{eqnarray*}
Then there exists a sufficiently large $x_*$ such that
\begin{eqnarray*}
\E W(x+\xi(x))- W(x) &<& 0\quad\mbox{for all }x\ge x_*.
\end{eqnarray*}
Now take $W_*(x):=\min(W(x),W(x_*))$ so that $\{W_*(X_n)\}$
constitutes a positive bounded supermartingale with respect to
the filtration $\{\mathcal F_n\}=\{\sigma(X_k,k\le n)\}$.
Hence we may apply Doob's inequality for nonnegative
supermartingales and deduce that, for all $y\ge x\ge 0$
(so that $W_*(y)\le W_*(x)$),
\begin{eqnarray*}
\P\Bigl\{\sup_{n\ge 1} W_*(X_n)\ge W_*(x)\Big|X_0=y\Bigr\}
&\le& \frac{\E_y W_*(X_0)}{W_*(x)}
=e^{\delta(R_*(x)-R_*(y))},
\end{eqnarray*}
which is equivalent to the first conclusion of the theorem.
\qed\end{proof}
\end{theopargself}

Notice that the condition \eqref{rx.ge.1x.pr.eps} fails for functions $r(x)$
asymptotically equivalent to $1/x$ which arise when we consider
the case of iterated logarithms. To cope with such functions,
we introduce a decreasing twice differentiable function $\widetilde r(x)>0$ 
such that $\widetilde r\le r$, and, for some $\varepsilon>0$,
\begin{eqnarray}\label{cond.tilde.r.2}
\widetilde r'(x) &\ge& -\widetilde r^2(x)\Bigl(\frac{r(x)}{\widetilde r(x)}-\varepsilon\Bigr)
\quad\mbox{for }x\ge\widehat x,
\end{eqnarray}
which, in particular, implies $\widetilde r'(x)\ge -\widetilde r(x)r(x)$. 
Notice that, for $\widetilde r(x)=r(x)$, 
the condition \eqref{cond.tilde.r.2} reduces to \eqref{rx.ge.1x.pr.eps}.
We also assume that
\begin{eqnarray}\label{cond.tilde.r.1}
\widetilde r''(x)\ =\ O(\widetilde r(x)r^2(x))\quad\mbox{as }x\to\infty.
\end{eqnarray}
Denote
\begin{eqnarray*}
\widetilde R(x) &:=& \int_0^x \widetilde r(y)dy\quad\mbox{for all }x>0, 
\end{eqnarray*}
and $\widetilde R(x):=0$ for all $x\le 0$.

\index{Down-crossing probability!upper bound}

\begin{theorem}\label{l:est.for.return.log}
Assume that the drift of $\{X_n\}$ possesses the lower bound \eqref{r-cond.2}
with function $r(x)$ satisfying \eqref{rx.ge.1x.pr},
$\widetilde r(x)$ satisfies \eqref{cond.tilde.r.2}--\eqref{cond.tilde.r.1},
and $s(x)=o(\widetilde r(x)/r^2(x))$. Let, for some $\delta<\varepsilon$,
\begin{eqnarray}\label{rec.4.log.}
\E\{e^{-\delta\widetilde R(x+\xi(x))};\ \xi(x)<-s(x)\} &=&
o\bigl(\widetilde r^2(x)e^{-\delta\widetilde R(x)} m_2^{[s(x)]}(x)\bigr)\ \mbox{as }x\to\infty. 
\end{eqnarray}
Then there exists an $x_*$ such that, for all $y>x\ge x_*$,
\begin{eqnarray*}
\P_y\{X_n\le x\mbox{ for some }n\ge1\} &\le& e^{\delta(\widetilde R(x)-\widetilde R(y))}.
\end{eqnarray*}
\end{theorem}

The condition \eqref{cond.tilde.r.2} is satisfied for 
$$
r(x)=\Bigl(\frac{1}{y}+\ldots+\frac{1}{y\log y\cdot\ldots\cdot\log_{(m-1)}y}
+\frac{1+\varepsilon}{y\log y\cdot\ldots\cdot\log_{(m)}y}\Bigr)\Big|_{y=e^{(m)}+x},
\ \varepsilon>0,\ m\ge 1,
$$
and 
$$
\widetilde r(x)=\frac{1}{y\log y\cdot\ldots\cdot\log_{(m)}y}\Big|_{y=e^{(m)}+x}.
$$ 
In this case
$$
\widetilde R(x)=\log_{(m+1)}(e^{(m)}+x),
$$ 
and hence the following corollary holds true. 

\index{Down-crossing probability!upper bound}
\begin{corollary}\label{c:est.for.return.log}
Assume that the drift of $\{X_n\}$ possesses the lower bound \eqref{r-cond.2}
with $r(x)$ defined above, and $s(x)=o(x/\log x\cdot\ldots\cdot\log_{(m)}x)$. 
Let, for some $\delta\in(0,\varepsilon)$, as $x\to\infty$,
\begin{eqnarray}\label{rec.4.log}
\E\{\log_{(m)}^{-\delta}(x+\xi(x));\ \xi(x)<-s(x)\} &=&
o\bigl(m_2^{[s(x)]}(x)/x^2\log^2 x\cdot\ldots\cdot\log_{(m)}^{2+\delta}x\bigr).
\nonumber\\
\end{eqnarray}
Then there exists an $x_*$ such that, for all $y>x\ge x_*$,
\begin{eqnarray*}
\P_y\{X_n\le x\mbox{ for some }n\ge1\} &\le& 
\biggl(\frac{\log_{(m)}(e^{(m)}+x)}{\log_{(m)}(e^{(m)}+y)}\biggr)^\delta.
\end{eqnarray*}
\end{corollary}

\begin{theopargself}
\begin{proof}[of Theorem \ref{l:est.for.return.log}]
We consider a decreasing test function $\widetilde W(x):=e^{-\delta\widetilde R(x)}$,
which is bounded by $1$ and prove that the mean drift
of $\widetilde W(x)$ is negative for all sufficiently large $x$.
Indeed, since the function $\widetilde W(x)$ decreases,
\begin{eqnarray}\label{return.deco.log}
\E \widetilde W(x+\xi(x))-\widetilde W(x) &\le& 
\E\{\widetilde W(x+\xi(x))-\widetilde W(x);\ \xi(x)\le s(x)\}\nonumber\\
&\le& \E\{\widetilde W(x+\xi(x));\ \xi(x)<-s(x)\}\nonumber\\
&& +\widetilde W'(x)\E\{\xi(x);\ |\xi(x)|\le s(x)\}\nonumber\\
&& +\frac{1}{2}\widetilde W''(x)\E\{\xi^2(x);\ |\xi(x)|\le s(x)\}\nonumber\\
&& +\frac{1}{6}\E\{\widetilde W'''(x+\theta\xi(x))\xi^3(x);\ |\xi(x)|\le s(x)\}\nonumber\\
&=:& E_1+E_2+E_3+E_4,
\end{eqnarray}
where $0\le\theta=\theta(x,\xi(x))\le 1$, by Taylor's expansion.
By the same arguments as in the last proof, as $x\to\infty$,
\begin{eqnarray}\label{return.E1.log}
E_1 &=& o\bigl(\widetilde r^2(x)\widetilde W(x)m_2^{[s(x)]}(x)\bigr),\\
\label{return.E2.log}
E_2 &\le& -\frac{\delta}{2} \widetilde r(x)r(x)W(x)m_2^{[s(x)]}(x),\\
\label{return.E3.log}
E_3 &=& \frac{\delta}{2} (\delta\widetilde r^2(x)-\widetilde r'(x)) \widetilde W(x) m_2^{[s(x)]}(x).
\end{eqnarray}
Next, owing to \eqref{cond.tilde.r.2}, \eqref{cond.tilde.r.1}, and the inequality $\widetilde r\le r$,
\begin{eqnarray*}
|\widetilde W'''(x)| &=& \delta
\bigl|-\delta^2\widetilde r^3(x)+3\delta\widetilde r(x)\widetilde r'(x)-\widetilde r''(x)\bigr| 
\widetilde W(x)\\
&\le& c\widetilde r(x)r^2(x) \widetilde W(x)
\quad\mbox{for some }c<\infty.
\end{eqnarray*}
By \eqref{rx.ge.1x.pr}, $r(x+y)\sim r(x)$, and by \eqref{cond.tilde.r.1}, 
$\widetilde r(x+y)\sim \widetilde r(x)$,
$\widetilde R(x+y)\sim\widetilde R(x)$, and $\widetilde W(x+y)\sim \widetilde W(x)$ 
as $x\to\infty$ uniformly for all $|y|\le s(x)=o(\widetilde r(x)/r^2(x))$, which implies
\begin{eqnarray*}
|\widetilde W'''(x+y)| &\le& c_1 \widetilde r(x)r^2(x) \widetilde W(x)
\end{eqnarray*}
as $x\to\infty$ uniformly for all $|y|\le s(x)$. Then
\begin{eqnarray}\label{return.E4.log}
|E_4| &\le& c_1\widetilde r(x)r^2(x) \widetilde W(x) \E\{|\xi^3(x)|;\ |\xi(x)|\le s(x)\}\nonumber\\
&\le& c_1s(x)\widetilde r(x)r^2(x) \widetilde W(x) m_2^{[s(x)]}(x)\nonumber\\
&=& o\bigl(\widetilde r^2(x)\bigr) \widetilde W(x) m_2^{[s(x)]}(x)\quad\mbox{as }x\to\infty,
\end{eqnarray}
since $s(x)=o(\widetilde r(x)/r^2(x))$.
Substituting \eqref{return.E1.log}--\eqref{return.E4.log} into
\eqref{return.deco.log} we deduce that
\begin{eqnarray*}
\E \widetilde W(x+\xi(x))-\widetilde W(x) &\le& 
\frac{\delta}{2}\bigl(-r(x)\widetilde r(x)+\delta\widetilde r^2(x)-\widetilde r'(x)+o(\widetilde r^2(x))\bigr)
\widetilde W(x)m_2^{[s(x)]}(x)\\
&\le& \frac{\delta}{2}\bigl((\delta-\varepsilon)\widetilde r^2(x)+o(\widetilde r^2(x))\bigr)
\widetilde W(x)m_2^{[s(x)]}(x)\quad\mbox{as }x\to\infty.
\end{eqnarray*}
due to the condition \eqref{cond.tilde.r.2}.
Then there exists a sufficiently large $x_*$ such that
\begin{eqnarray*}
\E \widetilde W(x+\xi(x))-\widetilde W(x) &<& 0\quad\mbox{for all }x\ge x_*,
\end{eqnarray*}
which concludes the proof in the same way as in Theorem \ref{l:est.for.return}.
\qed\end{proof}
\end{theopargself}

\section{Comments to Chapter \ref{ch:return.transient}}

The only result on down-crossing probabilities for
transient Markov chains with asymptotically zero drift
we are aware of was obtained by Vatutin\index{Vatutin} \cite{Vatutin74}
for critical branching processes with immigration.
He derived asymptotics for the probability
of hitting zero for such processes, which agrees with our lower and
upper bounds presented in Theorem \ref{thm:transient.return}
for general Markov chains.
A reduction of a critical branching process with immigration to a
Markov chain with drift of order $c/x$ and bounded second moment of
jumps via $\sqrt x$-transform is discussed in Section \ref{sec:branching}.
\chapter{Limit theorems for transient and null-recurrent Markov chains
with drift proportional to $1/x$}
\chaptermark{Limit theorems}% for transient and null-recurrent chains}
\label{ch:transient}

Assume that the first two moments of jumps of a Markov chain $\{X_n\}$
demonstrate regular behaviour at infinity, namely
$$
m_1(x)\sim \mu/x,\quad m_2(x)\to b>0\quad\mbox{as }x\to\infty.
$$
Then, as follows from Corollaries \ref{cor:null} and \ref{cor:tr.log},
under additional technical conditions,
\begin{itemize}
\item if $\mu\in(-b/2,b/2)$ then $\{X_n\}$ is null recurrent and 
$X_n\to\infty$ in probability as $n\to\infty$, say if $X$ is countable;
\item if $\mu>b/2$ then $\{X_n\}$ is transient and
$X_n\to\infty$ with probability $1$ as $n\to\infty$.
\end{itemize}
It turns out that in both cases $X_n$ increases at rate $\sqrt n$,
more precisely, the following weak convergence is observed:
$$
\frac{X_n^2}{bn}\ \Rightarrow\ \Gamma_{1/2+\mu/b,2}
\quad\mbox{as }n\to\infty.
$$
This is the main topic we discuss in this chapter, including results
concerning the renewal function, which is well defined in the transient case.

\section{Truncation of jumps}
\label{sec:thresholds}

In the sequel, we repeatedly make use of the truncation technique
for proving various limit theorems.
The idea behind this is that if we truncate jumps at sufficiently high
level, then we get a new Markov chain whose trajectory diverges from
that of the original chain with a small probability.

Let $\{B(x)\subseteq\R,\ x\in\R\}$ be a collection of Borel sets.
Given a Markov chain $\{X_n\}$ with jumps $\xi(x)$, consider a modified Markov
chain $\{\widetilde X_n\}$ %on the same probability space 
whose jumps $\widetilde\xi(x)$ are defined as
\begin{eqnarray*}\label{def:Z.jumps.B}
\widetilde\xi(x) &=& \left\{
\begin{array}{ll}
\xi(x) &\mbox{if }\xi(x)\in B(x);\\
\mbox{any value} &\mbox{if }\xi(x)\not\in B(x).
\end{array}
\right.
\end{eqnarray*}
In the sequel our standard choice is either $B(x)=[-s(x),s(x)]$
or $[-s(x),\infty)$ and `any value' is $0$ which corresponds to the
truncation of the original jumps $\xi(x)$ at levels $-s(x)$ or $s(x)$.

In this section, we prove a coupling that allows us to compare
two Markov chains which have asymptotically equal jumps.
The following result is repeatedly used each time we want
to simplify our calculations related to the characteristics of $\{X_n\}$.
We formulate this result in a more general setting as follows.

Let $Y=\{Y_n\}$ and $Z=\{Z_n\}$ be two Markov chains with jumps
$\eta(x)$ and $\zeta(x)$ respectively.
Denote by $H_y^Z$ the renewal measure
generated by the chain $Z$ with initial state $Z_0=y$, that is,
\begin{eqnarray*}
H_y^Z(A) &:=& \sum_{n=0}^\infty \P_y\{Z_n\in A\},\quad A\in\mathcal B(\R).
\end{eqnarray*}

\begin{lemma}\label{l:XY.equiv}
Assume that the random variables $\eta(x)$ and $\zeta(x)$ can
be constructed on the same probability space in such a way that
\begin{eqnarray}\label{rec.3.1.hy}
\P\{\eta(x)\not=\zeta(x)\} &\le& p(x)v(x)\quad\mbox{for all }x,
\end{eqnarray}
where $v(x)>0$ and $p(x)>0$ are decreasing functions
and $p(x)$ is integrable at infinity.
Let also, for all $z\in\R$,
\begin{eqnarray}\label{rec.3.1.sslln}
\P\{Z_n>z\ \mbox{ for all }n\ge0\mid Z_0=y\} &\to& 1\quad\mbox{as }y\to\infty,
\end{eqnarray}
and, for some $c<\infty$ and an increasing function $l(x)>0$
satisfying $l(x+l(x))\le c_1 l(x)$ for all $x$,
\begin{eqnarray}\label{rec.3.1.hyz}
H_y^Z(x,x+l(x)] &\le& c\frac{l(x)}{v(x)}
\quad\mbox{for all }y\mbox{ and }x.
\end{eqnarray}
Then, for any $\varepsilon>0$ there exists an $x_\varepsilon$ such that
the chains $\{Y_n\}$ and $\{Z_n\}$ can be constructed on the same probability
space in such a way that
\begin{eqnarray}\label{YZ.disp.eps}
\P\{Y_n=Z_n\ \mbox{ for all }n\ge 0\} &\ge& 1-\varepsilon
\quad\mbox{provided }Y_0=Z_0\ge x_\varepsilon.
\end{eqnarray}
\end{lemma}

\begin{proof}
Let us construct a probability space and sequences of independent
random fields $\{\eta_n(x),x\in\R\}_{n\ge 0}$ and
$\{\zeta_n(x),x\in\R\}_{n\ge 0}$ on this space such that
\begin{eqnarray}\label{rec.3.1.hy.n}
\P\{\eta_n(x)\not=\zeta_n(x)\} &\le& p(x)v(x)
\quad\mbox{for all }x\in\R\mbox{ and }n\ge 0,
\end{eqnarray}
which is possible due to \eqref{rec.3.1.hy}.
Then let us define Markov chains $\{Y_n\}$ and $\{Z_n\}$ as follows: $Y_0=Z_0$,
\begin{eqnarray*}
Y_{n+1}\ =\ Y_n+\eta_{n+1}(Y_n), && Z_{n+1}\ =\ Z_n+\zeta_{n+1}(Z_n),
\quad n\ge 0.
\end{eqnarray*}
Fix an $\varepsilon>0$. For any $z$,
\begin{eqnarray*}
\lefteqn{\P\{Y_n\neq Z_n\mbox{ for some }n\mid Z_0=y\}}\\
&&\hspace{7mm}\le\ \P\{Z_n\le z+l(z)\mbox{ for some }n\mid Z_0=y\}\\
&&\hspace{20mm}+\P\{Y_n\neq Z_n\mbox{ for some }n,
Z_n>z+l(z)\mbox{ for all }n\mid Z_0=y\}.
\end{eqnarray*}
Owing to \eqref{rec.3.1.sslln}, there exists an $y_1(z)$ such that
\begin{eqnarray*}
\P\{Z_n\le z+l(z)\mbox{ for some }n\mid Z_0=y\}
&\le& \varepsilon/2\quad\mbox{for all }y>y_1(z).
\end{eqnarray*}
Given $Y_0=Z_0>z+l(z)$,
\begin{eqnarray*}
\lefteqn{\P\{Y_n\neq Z_n\mbox{ for some }n,\ Z_n>z+l(z)\mbox{ for all }n\mid Z_0=y\}}\\
&\le& \P\{\eta_{n+1}(Y_n)\not=\zeta_{n+1}(Z_n),\
Y_n=Z_n \mbox{ for some }n,\ Z_n>z+l(z)\mbox{ for all }n\mid Z_0=y\}.
\end{eqnarray*}
The probability on the right hand side does not exceed the following sum
\begin{eqnarray*}
\lefteqn{\sum_{n=0}^\infty\P\{\eta_{n+1}(Y_n)\not=\zeta_{n+1}(Z_n),\
Y_n=Z_n>z+l(z)\mid Z_0=y\}}\\
&&\hspace{40mm}=\ \int_{z+l(z)}^\infty \P\{\eta(x)\not=\zeta(x)\}H^Z_y(dx)\\
&&\hspace{70mm}\le\ \int_{z+l(z)}^\infty p(x)v(x) H^Z_y(dx),
\end{eqnarray*}
by the condition \eqref{rec.3.1.hy}.
The last integral tends to $0$ as $z\to\infty$.
Indeed, both functions $p(z)$ and $v(x)$ are decreasing, hence
\begin{eqnarray*}
\int_{z+l(z)}^\infty p(x)v(x) H^Z_y(dx) &\le&
\sum_{i=1}^\infty p(x_i)v(x_i) H^Z_y(x_i,x_{i+1}],
\end{eqnarray*}
where $x_0:=z$ and $x_{i+1}:=x_i+l(x_i)$ for $i\ge 0$.
Then, by the condition \eqref{rec.3.1.hyz} on $H_y^Z$
and the property $l(x+l(x))\le c_1l(x)$,
\begin{eqnarray*}
\int_{z+l(z)}^\infty p(x)v(x) H^Z_y(dx)
&\le& c\sum_{i=1}^\infty p(x_i)l(x_i)\\
&=& c\sum_{i=1}^\infty p(x_i)l(x_{i-1}+l(x_{i-1}))\\
&\le& cc_1\sum_{i=1}^\infty p(x_i)l(x_{i-1})\\
&=& cc_1\sum_{i=1}^\infty p(x_i)(x_i-x_{i-1}).
\end{eqnarray*}
The function $p(x)$ is decreasing, therefore
\begin{eqnarray*}
\sum_{i=1}^\infty p(x_i)(x_i-x_{i-1})
&\le& \int_z^\infty p(u)du\ \to\ 0\quad\mbox{as }z\to\infty,
\end{eqnarray*}
because $p(x)$ is integrable. Hence,
\begin{eqnarray}\label{int.prH.fin}
\int_{z+l(z)}^\infty p(x)v(x) H^Z_y(dx) &\to& 0
\quad\mbox{as }z\to\infty\mbox{ uniformly for all }y,
\end{eqnarray}
which implies convergence to $0$ of the integral from $z$ to $\infty$.
Then the integral from $z$ to $\infty$ is less than $\varepsilon/2$
for a sufficiently large $z=z(\varepsilon)$ which
concludes the proof with $x_\varepsilon=y_1(z(\varepsilon))$.
\qed\end{proof}

Assume that
\begin{eqnarray}\label{Z.limsup.to.infty}
\P\Bigl\{\limsup_{n\to\infty}Y_n=\infty\Bigr\} &=& 1
\end{eqnarray}
and, for any distribution of $Z_0$,
\begin{eqnarray}\label{Y.disp.to.infty}
Z_n &\stackrel{a.s.}\to& \infty\quad\mbox{as }n\to\infty.
\end{eqnarray}
Then, under the conditions of Lemma \ref{l:XY.equiv},
\begin{eqnarray}\label{Z.disp.to.infty}
Y_n &\stackrel{a.s.}\to& \infty\quad\mbox{as }n\to\infty.
\end{eqnarray}
Indeed, given any $\varepsilon\in(0,1)$, by Lemma \ref{l:XY.equiv}
there exists a level $x_\varepsilon$ such that \eqref{YZ.disp.eps} holds.
By the condition \eqref{Z.limsup.to.infty}, the stopping time
\begin{eqnarray}\label{tau.exists}
\tau_\varepsilon &:=&
\min\{n\ge 0:\ Y_n\ge x_\varepsilon\}
\end{eqnarray}
is finite with probability $1$.
Set $Z_0=Y_{\tau_\varepsilon}$. Since then $Z_0\ge x_\varepsilon$,
it follows from \eqref{YZ.disp.eps} that, for all $A$,
\begin{eqnarray*}
\P\Bigl\{\liminf_{n\to\infty}Y_{\tau_\varepsilon+n}>A\Bigr\} 
&\ge& \P\Bigl\{\liminf_{n\to\infty}Z_n>A\Bigr\}-\varepsilon,
\end{eqnarray*}
which due to \eqref{Y.disp.to.infty} implies that, for all $A$,
\begin{eqnarray*}
\P\Bigl\{\liminf_{n\to\infty}Y_{\tau_\varepsilon+n}>A\Bigr\} 
&\ge& 1-\varepsilon.
\end{eqnarray*}
Therefore, due to the finiteness of $\tau_\varepsilon$,
\begin{eqnarray*}
\P\Bigl\{\liminf_{n\to\infty}Y_n>A\Bigr\} &\ge& 1-\varepsilon,
\end{eqnarray*}
for all $\varepsilon>0$ and $A<\infty$. Due to the arbitrary choice of $\varepsilon>0$,
\begin{eqnarray*}
\P\Bigl\{\liminf_{n\to\infty}Y_n>A\Bigr\} &=& 1\quad\mbox{for all }A<\infty,
\end{eqnarray*}
hence \eqref{Z.disp.to.infty} follows, due to the arbitrary choice of $A$.

If, instead of \eqref{Y.disp.to.infty}, for any distribution of $Z_0$,
\begin{eqnarray}\label{Y.disp.to.infty.p}
Z_n &\stackrel{p}\to& \infty\quad\mbox{as }n\to\infty,
\end{eqnarray}
then 
\begin{eqnarray}\label{Z.disp.to.infty.p}
Y_n &\stackrel{p}\to& \infty\quad\mbox{as }n\to\infty.
\end{eqnarray}
To show this convergence, we again consider the stopping time \eqref{tau.exists}
and define the same $Z_0=Y_{\tau_\varepsilon}$. 
Since $\tau_\varepsilon$ is finite, there exists an $N$ such that
\begin{eqnarray*}
\P\{\tau_\varepsilon>N\} &\le& \varepsilon.
\end{eqnarray*}
Then, for $n>N$,
\begin{eqnarray*}
\P\{Y_n>A\} &\ge& 1-\P\{\tau_\varepsilon>N\}-\P\{\tau_\varepsilon\le N,\ Y_n\le A\}\\
&\ge& 1-\varepsilon
-\sum_{k=0}^N\P\{\tau_\varepsilon=k,\ Z_{n-k}\le A\}-\varepsilon,
\end{eqnarray*}
owing to \eqref{YZ.disp.eps}. Therefore,
\begin{eqnarray*}
\P\{Y_n>A\} &\ge& 1-2\varepsilon-\sum_{k=0}^N\P\{Z_{n-k}\le A\},
\end{eqnarray*}
where each of the probabilities $\P\{Z_{n-k}\le A\}$ tends to zero as $n\to\infty$
uniformly for all $k\le N$. Thus,
\begin{eqnarray*}
\liminf_{n\to\infty}\P\{Y_n>A\} &\ge& 1-2\varepsilon
\end{eqnarray*}
and \eqref{Z.disp.to.infty.p} follows because of the arbitrary choice of 
$\varepsilon>0$.

In particular, if for some increasing function $V(x)$ 
and normalising sequence $c_n$,
\begin{eqnarray*}
\frac{V(Z_n)}{c_n} &\to& 1\quad\mbox{as }n\to\infty\ \mbox{ a.s. or in probability},
\end{eqnarray*}
then
\begin{eqnarray*}
\frac{V(Y_n)}{c_n} &\to& 1\quad\mbox{as }n\to\infty\ \mbox{ a.s. or in probability}.
\end{eqnarray*}

\section{Upper bound for average up-crossing time for transient chain}
\label{sec:h.pre}

Let us define
\begin{eqnarray}\label{def.L}
L(x,n) &:=& \sum_{k=0}^{n-1}\I\{X_k\ge x\}.
\end{eqnarray}
The next theorem is devoted to the properties of $L(x,T(t))$, where
$T(t)$ is the first up-crossing time
$$
T(t):=\min\{n\ge 1: X_n>t\}.
$$
Let $v(z)\downarrow 0$ be a decreasing function. Denote
\begin{eqnarray}\label{def.V.v}
V(u) &:=& \int_0^u \frac{1}{v(z)} dz\quad\mbox{for }u\ge 0,
\end{eqnarray}
and $V(u)=0$ for $u<0$.
Since the function $1/v(z)$ increases, $V$ is convex.

\begin{theorem}\label{l:uniform}
Let, for some increasing function $s(x)>0$ and for some $\widehat x\ge 0$,
\begin{eqnarray}\label{T.above.cond.1}
\E\{\xi(x);\ \xi(x)\le s(x)\} &\ge& v(x)\quad\mbox{for all }x\ge\widehat x.
\end{eqnarray}
Then, for all $t\ge y\ge\widehat x$,
\begin{eqnarray}\label{T.mean.bound.ori}
\E_y L(\widehat x,T(t)) &\le& V(t+s(t))-V(y)\
=\ \int_y^{t+s(t)} \frac{1}{v(z)}dz.
\end{eqnarray}
Further, the family of random variables
\begin{eqnarray}\label{l:uniform.r.x}
\frac{1}{V(t+s(t))-V(x)} L(x,T(t)),\quad t\ge y\ge x\ge\widehat x,\ X_0=y,
\end{eqnarray}
is uniformly integrable.
\end{theorem}

\begin{proof}
Let us consider the following continuous test function
$$
\widehat V(u)\ :=\ V(\widehat x\vee u)\ =\
\left\{
\begin{array}{cc}
V(\widehat x)&\mbox{if }u<\widehat x,\\
V(u)&\mbox{if }u\ge \widehat x.
\end{array}
\right.
$$
This function is convex as $V$ is, so 
\begin{eqnarray*}
\E_u\{\widehat V(X_1)-\widehat V(u);\ X_1-u\le s(u)\} 
&\ge& \widehat V'(u)\E\{\xi(u);\ \xi(u)\le s(u)\},
\end{eqnarray*}
where the right derivative of $\widehat V$ equals
\begin{eqnarray*}
\widehat V'(u) &=& \left\{
\begin{array}{cll}
0 &\mbox{if}&u<\widehat x,\\
1/v(u) &\mbox{if}&u\ge \widehat x.
\end{array}
\right.
\end{eqnarray*}
Therefore,
\begin{eqnarray*}
\E_u\{\widehat V(X_1)-\widehat V(u);\ X_1-u\le s(u)\} &\ge& \left\{
\begin{array}{cll}
1 &\mbox{if}&u\ge\widehat x,\\
0 &\mbox{if}&u<\widehat x,
\end{array}
\right.
\end{eqnarray*}
by the condition \eqref{T.above.cond.1}. 
Since the function $u+s(u)$ is increasing, 
\begin{eqnarray}\label{Y.ge.eps}
\E_u\{\widehat V(X_1)-\widehat V(u);\ X_1\le t+s(t)\} &\ge& \left\{
\begin{array}{cll}
1 &\mbox{if}&u\in[\widehat x,t],\\
0 &\mbox{if}&u<\widehat x,
\end{array}
\right.
\end{eqnarray}
Therefore, the process
$Y_n:=\widehat V(X_n\wedge(t+s(t)))$ satisfies the following inequality
\begin{eqnarray}\label{EYT.below}
\E_y Y_{T(t)} &\ge& V(y)+\E_y \sum_{k=0}^{T(t)-1}\I\{X_k\ge\widehat x\},
\quad y\in[\widehat x,t],
\end{eqnarray}
due to the following adapted version of the proof of Dynkin's
formula\index{Dynkin's formula}
(see, e.g. \cite[Theorem 11.3.1]{MT}):
\begin{eqnarray*}
\E_y Y_{T(t)} &=& \E_y Y_0+\E_y\sum_{n=1}^\infty
\I\{n\le T(t)\}(Y_n-Y_{n-1})\\
&=& V(y)+\E_y\sum_{n=1}^\infty
\E\{\I\{n\le T(t)\}(Y_n-Y_{n-1})\mid{\mathcal F}_{n-1}\}\\
&=& V(y)+\E_y\sum_{n=1}^\infty \I\{T(t)\ge n\}
\E\{Y_n-Y_{n-1}\mid{\mathcal F}_{n-1}\},
\end{eqnarray*}
because $\{n\le T(t)\}=\overline{\{T(t)\le n-1\}}\in{\mathcal F}_{n-1}$.
Hence, \eqref{Y.ge.eps} implies that
\begin{eqnarray*}
\E_y Y_{T(t)} &\ge& V(y)+\E_y\sum_{n=1}^\infty \I\{T(t)\ge n,X_{n-1}\ge\widehat x\}\\
&=& V(y)+\E_y \sum_{n=1}^{T(t)} \I\{X_{n-1}\ge\widehat x\},
\end{eqnarray*}
and the inequality \eqref{EYT.below} follows.

On the other hand, $Y_{T(t)}\le V(t+s(t))$, by the construction of $\{Y_n\}$. Hence,
\begin{eqnarray}\label{EYT.upper.1}
\E_y Y_{T(t)} &\le& \widehat V(t+s(t))\ =\ V(t+s(t)),
\end{eqnarray}
because $x<t$, which together with \eqref{EYT.below} yields
\begin{eqnarray*}
\E_y L(\widehat x,T(t)) &\le& V(t+s(t))-V(y),
\end{eqnarray*}
and the upper bound \eqref{T.mean.bound.ori} follows.

Now let us proceed with the proof of the uniform integrability in \eqref{l:uniform.r.x}
which is equivalent to the following convergence 
\begin{eqnarray}\label{uniform.integrability.x}
\sup_{\widehat x\le x\le y\le t}\E_y\Bigl\{\frac{L(x,T(t))}{V(t+s(t))-V(x)};\
\frac{L(x,T(t))}{V(t+s(t))-V(x)}>A\Bigr\} &\to& 0
\quad\mbox{as }A\to\infty.\nonumber\\[-2mm]
\end{eqnarray}
For $N\in\N$, define $\theta_N$ to be the following stopping time
$$
\theta_N\ =\ \theta_N(x)\ :=\ \inf\Bigl\{n:L(x,n)=\sum_{k=0}^{n-1}\I\{X_k\ge x\} = N\Bigr\}-1.
$$
Similarly to \eqref{EYT.below},
\begin{eqnarray}\label{EYT.below.N.x}
\E_yY_{T(t)} &\ge& \E_y Y_{\theta_N\wedge T(t)}
+\E_y \sum_{n=\theta_N+1}^{T_N-1}\I\{X_n\ge x\}\nonumber\\
&=& \E_y Y_{\theta_N\wedge T(t)}+\E_y \{L(x,T(t)-N;\ L(x,T(t)>N\}.
\end{eqnarray}
Therefore,
\begin{eqnarray*}
\E_y \{L(x,T(t)-N;\ L(x,T(t)>N\}
&\le& \E_y(Y_{T(t)}-Y_{\theta_N\wedge T(t)})\\
&=& \E_y\{Y_{T(t)}-Y_{\theta_N};\ T(t)>\theta_N\}\\
&\le& \E_y\{V(X_{T(t)})-V(X_{\theta_N}));\ T(t)>\theta_N\},
\end{eqnarray*}
by the definition of $\{Y_n\}$.
Taking into account that $X_{T(t)}\le t+s(t)$ and $X_{\theta_N}\ge x$, we deduce that
\begin{eqnarray}\label{L.upper.x}
\E_y \{L(x,T(t));\ L(x,T(t))>N\} &\le& (V(t+s(t))-V(x))\P_y\{L(x,T(t))>N\}.\nonumber\\
\end{eqnarray}
Taking
\begin{eqnarray*}
N:=[A(V(t+s(t))-V(x))],
\end{eqnarray*}
we get from \eqref{L.upper.x} that the mean in
\eqref{uniform.integrability.x} is not greater than
\begin{eqnarray*}
\P_y\{L(x,T(t))>N\},
\end{eqnarray*}
which in its turn is not greater than
\begin{eqnarray*}
\frac{\E_y L(x,T(t))}{N+1},
\end{eqnarray*}
by the Markov inequality.
Due to the upper bound \eqref{T.mean.bound.ori} already proven, for $y\ge x$,
\begin{eqnarray*}
\frac{\E_y L(x,T(t))}{N+1} &\le& \frac{V(t+s(t))-V(y)}{A(V(t+s(t))-V(x))}
\ \le\ \frac{1}{A},
\end{eqnarray*}
and the proof of the uniform integrability \eqref{l:uniform.r.x} is complete.
\qed\end{proof}

\section{Transient chain: integro-local upper bound for renewal function}
\sectionmark{Integro-local upper bound for renewal function}
\label{sec:renewal.upper}

A transient Markov chain $\{X_n\}$ visits any bounded set finitely many
times only. As noticed in Section \ref{subsec:nnmc.trans},
then for countable Markov chains the renewal functions\index{Renewal measure}
\begin{eqnarray*}
H_y(x,x+h] &:=& \E_y\sum_{n=0}^\infty \I\{x<X_n\le x+h\}
\ =\ \sum_{n=0}^\infty \P_y\{x<X_n\le x+h\},\\
H(x,x+h] &:=& \sum_{n=0}^\infty \P\{x<X_n\le x+h\}
=\int_0^\infty H_y(x,x+h]\P\{X_0\in dy\},
\end{eqnarray*}
are well-defined for all $x\in\R$ and $h>0$. For general Markov chains,
they are also well-defined under some minor technical conditions.
In the next result we derive upper bounds for these renewal functions.
As shown in the sequel, under some regularity conditions,
the upper bounds derived are asymptotically correct up to a constant multiplier.
\index{Markov chain!renewal function!upper bound}
\index{Renewal function!Markov chain!upper bound}

\begin{theorem}\label{thm:Hy.above}
Let the drift of $\{X_n\}$ possess the lower bound \eqref{r-cond.2} with some
$r(x)$ satisfying \eqref{rx.ge.1x.pr.eps} and increasing
function $s(x)=o(1/r(x))$.
Assume \eqref{T.above.cond.1} for some decreasing $v(x)$ satisfying
\begin{eqnarray}\label{def.cv}
c_v:=\sup_{x>0}\frac{v(x)}{v(x+1/r(x))} &<& \infty.
\end{eqnarray}
Assume also an upper bound for the left tail
\begin{eqnarray}\label{rec.3.7}
\P\{\xi(x)\le -s(x)\} &\le& p(x)v(x)\ \mbox{ for all }x\ge \widehat x,
\end{eqnarray}
where a decreasing function $p(x)>0$ is integrable at infinity.
Then the family of random variables
\begin{eqnarray*}
v(x)r(x)\sum_{n=0}^\infty \I\{x<X_n\le x+1/r(x)\},
\quad x\ge\widehat x,\ X_0=y,
\end{eqnarray*}
is uniformly integrable.

In particular, there exists a $c_1<\infty$ such that
\begin{eqnarray*}
H_y(x,x+1/r(x)] &\le& \frac{c_1}{v(x)r(x)},
\end{eqnarray*}
for all $x\ge\widehat x$ and $y$, and further, 
\begin{eqnarray*}
H_y(\widehat x,x] &\le& c_1\int_{\widehat x}^{x+1/r(x)}\frac{dz}{v(z)}.
\end{eqnarray*}
\end{theorem}

These upper bounds are rather accurate for $y\le x$. 
In the opposite case $y>x$ sharper bounds can be obtained 
by combining the upper bounds for the renewal function in Theorem \ref{thm:Hy.above} 
with estimates for down-crossing probabilities, 
that is either with Theorem \ref{l:est.for.return} 
or exact asymptotic results in Chapter \ref{ch:return.transient}.

\begin{proof}
Considering the first entry of $\{X_n\}$ into the segment $(x,x+1/r(x)]$
we see that the first conclusion is equivalent to
the uniform integrability of the family
\begin{eqnarray}\label{uni.equiv}
v(x)r(x)\sum_{n=0}^\infty \I\{x<X_n\le x+1/r(x)\},
\ x\ge\widehat x,\ X_0=y,\ y\in(x,x+1/r(x)].\hspace{5mm}
\end{eqnarray}

First let us consider a Markov chain $\{Y_n\}$ with jumps
\begin{eqnarray*}
\eta(x) &:=& \max(\xi(x),\ -s(x)).
\end{eqnarray*}
This Markov chain satisfies the conditions \eqref{rec.4},
because $\eta(x)\ge -s(x)$, and \eqref{r-cond.2}.
So Theorem \ref{l:est.for.return} applies to the chain $\{Y_n\}$
with $\delta<\varepsilon$ where $\varepsilon$ is defined in\eqref{rx.ge.1x.pr.eps}, hence
\begin{eqnarray}\label{Y.x.y}
\P\{Y_n\le x\mbox{ for some }n\ge 1\mid Y_0=y\} &\le&
e^{\delta(R(x)-R(y))}
\ \mbox{for all }y>x\ge x_*,\hspace{10mm}
\end{eqnarray}
where $x_*$ is delivered by Theorem \ref{l:est.for.return}.
Without loss of generality we assume that $x_*>\widehat x$.
Consider a stopping time
\begin{eqnarray*}
T^Y(t) &=& \min\{n\ge 1:Y_n>t\},
\end{eqnarray*}
where
$$
t\ :=\ \left\{
\begin{array}{ll}
x+2/r(x)&\mbox{for }x\ge x_*,\\
x_*+2/r(x_*)&\mbox{for }x\in[\widehat x,x_*].
\end{array}
\right.
$$
For any $Y_0=y\in(x,x+1/r(x)]$,
\begin{eqnarray}\label{vrVsum}
v(x)r(x)\sum_{n=0}^{T^Y(t)-1} \I\{x<Y_n\le x+1/r(x)\}
&\le& v(x)r(x)\sum_{n=0}^{T^Y(t)-1} \I\{Y_n>x\}.\hspace{10mm}
\end{eqnarray}
It follows from the convexity of the function $V(x)$ defined in \eqref{def.V.v} that
\begin{eqnarray*}
V(t+s(t))-V(x) &\le& V'(t+s(t))(t+s(t)-x)\\
&=& \frac{t+s(t)-x}{v(t+s(t))}.
\end{eqnarray*}
Thus,
\begin{eqnarray*}
\frac{1}{V(t+s(t))-V(x)} &\ge& \frac{v(t+s(t))}{t+s(t)-x}.
\end{eqnarray*}
For all sufficiently large $x$, $s(x)\le 1/r(x)$ and hence $s(t)\le 1/r(t)$.
In addition, $t\le x+2/r(x)$ for $x\ge x_*$. 
Therefore, for all sufficiently large $x$,
\begin{eqnarray}\label{vr.1}
\frac{v(x)}{v(t+s(t))} &=& \frac{v(x)}{v(x+1/r(x))}
\frac{v(x+1/r(x))}{v(t)}\frac{v(t)}{v(t+s(t))}\nonumber\\
&\le& \frac{v(x)}{v(x+1/r(x))}
\frac{v(x+1/r(x))}{v(x+2/r(x))}\frac{v(t)}{v(t+1/r(t))}
\ \le\ c_v^3,
\end{eqnarray}
by \eqref{def.cv}. Further,
\begin{eqnarray}\label{vr.2}
r(x)(t+s(t)-x) =r(x)(2/r(x)+s(t)) &\to& 2\quad\mbox{as }x\to\infty.
\end{eqnarray}
Therefore, there exists a $\gamma>0$ such that
\begin{eqnarray*}
\frac{1}{V(t+s(t))-V(x)} &\ge& \gamma v(x)r(x)
\quad\mbox{for all }x\ge\widehat x,
\end{eqnarray*}
which being applied to \eqref{vrVsum} yields that
\begin{eqnarray*}
\lefteqn{v(x)r(x)\sum_{n=0}^{T^Y(t)-1} \I\{x<Y_n\le x+1/r(x)\}}\\
&&\hspace{40mm}\le\ \frac{1}{\gamma}\frac{1}{V(t+s(t))-V(x)}
\sum_{n=0}^{T^Y(t)-1} \I\{Y_n>x\}.
\end{eqnarray*}
Finally, the family with respect to $x\ge\widehat x$, $Y_0=y$,
$y\in(x,x+1/r(x)]$ of random variables on the right hand side
is uniformly integrable, due to Theorem \ref{l:uniform}
applied to the chain $\{Y_n\}$. So, the family of random variables
\begin{eqnarray*}
v(x)r(x)\sum_{n=0}^{T^Y(t)-1} \I\{x<Y_n\le x+1/r(x)\},\quad x\ge\widehat x,
\end{eqnarray*}
is uniformly integrable too.

Further, after the stopping time $T^Y(t)$
the chain $\{Y_n\}$ falls down below the level
$$
t_1\ :=\ \left\{
\begin{array}{ll}
x+1/r(x)&\mbox{for }x\ge x_*,\\
x_*+1/r(x_*)&\mbox{for }x\in[\widehat x,x_*]
\end{array}
\right.
$$
with probability $e^{\delta(R(t_1)-R(t))}$ at the most,
see \eqref{Y.x.y} which is applicable because $t_1>x_*$.
Since the function $R(x)$ is concave,
\begin{eqnarray*}
e^{\delta(R(x+1/r(x))-R(x+2/r(x)))} &\le&
e^{-\delta R'(x+2/r(x))/r(x)}
\ =\ e^{-\delta r(x+2/r(x))/r(x)}.
\end{eqnarray*}
As is shown in \eqref{r.h.below}, $r(x+2/r(x))/r(x)\ge 1/(1+2c)$
for all $x\ge 0$, hence we conclude that
\begin{eqnarray*}
\sup_{x\ge 0}e^{\delta(R(x+1/r(x))-R(x+2/r(x)))} &\le&
e^{-\delta/(1+2c)}\ <\ 1.
\end{eqnarray*}
Therefore, for all $y\ge t$,
\begin{eqnarray*}
\P\{Y_n\le t_1\mbox{ for some }n\ge 1\mid Y_0=y\}
&\le& e^{-\delta/(1+2c)}\ <\ 1.
\end{eqnarray*}
Hence, we obtain by the Markov property that the family
\begin{eqnarray*}
v(x)r(x)\sum_{n=0}^\infty \I\{x<Y_n\le x+1/r(x)\}
\end{eqnarray*}
is dominated by a geometric number at the most of summands taken from
a uniformly integrable family of random variables,
which yields the first conclusion of theorem for the chain $\{Y_n\}$,
by Lemma \ref{l:uni.stop}(i) with $\sigma$-algebra $\mathcal F_n$
generated by the history of the chain up to $n$th falling down below the level $t_1$.
In particular, for some $c_3<\infty$,
\begin{eqnarray}\label{renewal.for.Y}
H_y^Y(x,x+1/r(x)] &\le& \frac{c_3}{v(x)r(x)}
\quad\mbox{for all } x\ge \widehat x\mbox{ and }y.
\end{eqnarray}

Further, in order to pass from $\{Y_n\}$ to $\{X_n\}$
we first notice that these two chains may be constructed
on the same probability space as described in the beginning
of the proof of Lemma \ref{l:XY.equiv}. This makes possible
the following calculations: for all $x<y$,
\begin{eqnarray*}
\lefteqn{\P\{X_n\le x\mbox{ for some }n\ge 1\mid X_0=Y_0=y\}}\\
&\le& \P\{Y_n\le x\mbox{ for some }n\ge 1\mid Y_0=y\}\\
&&\hspace{15mm}+\P\{X_n\neq Y_n\mbox{ for some }n\ge 1,\
Y_k>x\mbox{ for all }k\ge 1\mid X_0=Y_0=y\}\\
&\le& \P\{Y_n\le x\mbox{ for some }n\ge 1\mid Y_0=y\}\\
&&\hspace{15mm}+\P\{X_n\neq Y_n\mbox{ for some }n\ge 1\mid X_0=Y_0=y\}.
\end{eqnarray*}
The second probability on the right hand side tends to $0$
as $y\to\infty$, by Lemma \ref{l:XY.equiv} which is applicable
due to \eqref{rec.3.7} and
because the upper bound \eqref{Y.x.y} implies \eqref{rec.3.1.sslln}
and \eqref{renewal.for.Y} implies \eqref{rec.3.1.hyz} with
$l(x)=1/r(x)$, due to \eqref{r.h.below}.
Together with \eqref{Y.x.y} it yields that
\begin{eqnarray}\label{y.back.x}
\P\{X_n\le x\mbox{ for some }n\ge 1\mid X_0=y\}
&\le& e^{\delta(R(x)-R(y))}+o(1)\quad\mbox{as }x\to\infty
\end{eqnarray}
uniformly for all $y>x$. In particular,
there exists a sufficiently large $x_0\ge x_*$ such that, for some $q<1$,
\begin{eqnarray}\label{y.back.x.x0}
\P\{X_n\le x+1/r(x)\mbox{ for some }n\ge 1\mid X_0=y\} &\le& q
\end{eqnarray}
for all $x\ge x_0$ and $y\ge x+2/r(x)$.

In the same way as it was done for $\{Y_n\}$, we now consider a stopping time
$$
T(t)=\min\{n\ge 1:X_n>t\},
$$
where 
$$
t\ :=\ \left\{
\begin{array}{ll}
x+2/r(x)&\mbox{for }x\ge x_0,\\
x_0+2/r(x_0)&\mbox{for }x\in[\widehat x,x_0].
\end{array}
\right.
$$
Similarly to the chain $\{Y_n\}$,
the family with respect to $x\ge\widehat x$, $X_0=y$,
$y\in(x,x+1/r(x)]$ of random variables
\begin{eqnarray*}
v(x)r(x)\sum_{n=0}^{T(t)-1} \I\{x<X_n\le x+1/r(x)\}
\end{eqnarray*}
is uniformly integrable too, due to Theorem \ref{l:uniform} applied to $\{X_n\}$.

Further, after the stopping time $T(t)$ the chain $\{X_n\}$ falls down
below the level
$$
t_1\ :=\ \left\{
\begin{array}{ll}
x+1/r(x)&\mbox{for }x\ge x_0,\\
x_0+1/r(x_0)&\mbox{for }x\in[\widehat x,x_0]
\end{array}
\right.
$$
with probability $q<1$ at the most, see \eqref{y.back.x.x0} which
is applicable because $t_1\ge x_0$.
By the same reasons as for the Markov chain $\{Y_n\}$,
\begin{eqnarray*}
v(x)r(x)\sum_{n=0}^\infty \I\{x<X_n\le x+1/r(x)\}
\end{eqnarray*}
is majorised by a geometric number at the most of summands taken from
a uniformly integrable family of random variables,
which yields the first theorem conclusion for the chain $\{X_n\}$,
by Lemma \ref{l:uni.stop}(i) with $\sigma$-algebra $\mathcal F_n$
generated by the history of the chain up to $n$th falling down below
the level $t_1$. %and $\tau$ a geometric random variable.

The second conclusion of the theorem follows if we consider the points
$x_0:=\widehat x$, $x_{n+1}:=x_n+1/r(x_n)$ and then, by the first result,
\begin{eqnarray*}
H_y(\widehat x,x] &\le& \sum_{n=0}^{N-1} H_y(x_n,x_{n+1}]
\ \le\ c_1 \sum_{n=0}^{N-1} \frac{1}{v(x_n)r(x_n)},
\end{eqnarray*}
where $N:=\min\{n\ge 1:x_n>x\}$, so $x_N\le x+1/r(x)$.
Since $1/v(z)$ increases, we finally get
\begin{eqnarray*}
\sum_{n=0}^{N-1} \frac{1}{v(x_n)r(x_n)} &\le&
\sum_{n=0}^{N-1} \int_{x_n}^{x_n+1/r(x_n)}\frac{dz}{v(z)}\\ 
&=& \int_{\widehat x}^{x_N} \frac{dz}{v(z)}
\ \le\ \int_{\widehat x}^{x+1/r(x)} \frac{dz}{v(z)}.
\end{eqnarray*}
\qed\end{proof}

Now consider the case where the iterated logarithms play a r\^ole.
Assume that there exist $\varepsilon>0$, $m\ge 1$, 
and $\widehat x$ such that, for all $x>\widehat x$,
\begin{eqnarray}\label{r-cond.2.log}
\frac{2m^{[s(x)]}_1(x)}{m^{[s(x)]}_2(x)} &\ge& r(x)\nonumber\\
&=&\Bigl(\frac{1}{y}+\ldots+\frac{1}{y\log y\cdot\ldots\cdot\log_{(m-1)}y}
+\frac{1+\varepsilon}{y\log y\cdot\ldots\cdot\log_{(m)}y}\Bigr)\Big|_{y=e^{(m)}+x}.
\nonumber\\
\end{eqnarray}
\index{Markov chain!renewal function!upper bound}

\begin{theorem}\label{thm:Hy.above.log}
Let the drift of $\{X_n\}$ possess the lower bound \eqref{r-cond.2.log} 
with some increasing function $s(x)=o(x/\log x\cdot\ldots\cdot\log_{(m)}x)$.
Assume \eqref{T.above.cond.1} for $v(x)=\gamma/x$, $\gamma>0$.
Assume also an upper bound for the left tail, for some $\delta<\varepsilon$,
\begin{eqnarray}\label{rec.3.7.log}
\P\{\xi(x)\le -s(x)\} &=& o(m_2^{[s(x)]}/x^2\log^2 x\cdot\ldots\cdot\log^{2+\delta}_{(m)}x)
\ \mbox{ for all }x\ge \widehat x.
\end{eqnarray}
Then the family of random variables
\begin{eqnarray*}
\frac{1}{x^2\log x\cdot\ldots\cdot\log_{(m)}x}\sum_{n=0}^\infty \I\{\widehat x<X_n\le x\},
\quad x>\widehat x,\ X_0=y,
\end{eqnarray*}
is uniformly integrable.
In particular, there exists a $c<\infty$ such that
\begin{eqnarray*}
H_y(\widehat x,x] &\le& c_1x^2\log x\cdot\ldots\cdot\log_{(m)}x
\quad\mbox{for all }x>\widehat x\mbox{ and }y.
\end{eqnarray*}
\end{theorem}

\begin{proof}
By the same arguments as in the last proof, 
we see that the first conclusion is equivalent to the uniform integrability of the family
\begin{eqnarray}\label{uni.equiv.log}
\frac{1}{x^2\log x\cdot\ldots\cdot\log_{(m)}x}\sum_{n=0}^\infty \I\{x<X_n\le 2x\},
\ x>\widehat x,\ X_0=y,\ y\in(x,2x].
\end{eqnarray}

The Markov chain $\{X_n\}$ satisfies the conditions \eqref{rec.4.log} due to \eqref{rec.3.7.log}.
So Corollary \ref{c:est.for.return.log} applies, hence
\begin{eqnarray}\label{Y.x.y.log}
\P\{X_n\le x\mbox{ for some }n\ge 1\mid X_0=y\} &\le&
\biggl(\frac{\log_{(m)}(e^{(m)}+x)}{\log_{(m)}(e^{(m)}+y)}\biggr)^\delta
\ \mbox{for all }y>x\ge x_*,\nonumber\\[-1mm]
\end{eqnarray}
where $x_*$ is delivered by Corollary \ref{c:est.for.return.log}.
Without loss of generality we assume that $x_*>\widehat x$.
Similarly to how it was introduced for the Markov chain $\{Y_n\}$ in the last proof,
let us consider the stopping time  
\begin{eqnarray*}
T^X(t) &=& \min\{n\ge 1:X_n>t\},
\end{eqnarray*}
where
$$
t\ :=\ \left\{
\begin{array}{ll}
3x &\mbox{for }x\ge x_*,\\
3x_* &\mbox{for }x\in[\widehat x,x_*].
\end{array}
\right.
$$
As concluded in the last proof for $\{Y_n\}$, the family of random variables
\begin{eqnarray*}
\frac{1}{x^2}\sum_{n=0}^{T^X(t)-1} \I\{x<X_n\le 2x\},\quad x>\widehat x,
\end{eqnarray*}
is uniformly integrable.

Further, after the stopping time $T^X(t)$
the chain $\{X_n\}$ falls down below the level
$$
t_1\ :=\ \left\{
\begin{array}{ll}
2x&\mbox{for }x\ge x_*,\\
2x_*&\mbox{for }x\in[\widehat x,x_*]
\end{array}
\right.
$$
with probability \eqref{Y.x.y} at the most, which is applicable because $t_1>x_*$.
Observe that, for $x>x_*$,
\begin{eqnarray*}
\biggl(\frac{\log_{(m)}(e^{(m)}+2x)}{\log_{(m)}(e^{(m)}+3x)}\biggr)^\delta
&\le& 1-\frac{c_2}{\log x\cdot\ldots\cdot\log_{(m)}x}\quad\mbox{for some }c_2>0.
\end{eqnarray*}
Therefore, for all $y\ge t$,
\begin{eqnarray*}
\P\{X_n\le t_1\mbox{ for some }n\ge 1\mid X_0=y\}
&\le& 1-\frac{c_2}{\log x\cdot\ldots\cdot\log_{(m)}x}.
\end{eqnarray*}
Hence, we obtain by the Markov property that the family
\begin{eqnarray*}
\frac{1}{x^2}\sum_{n=0}^\infty \I\{x<X_n\le 2x\}
\end{eqnarray*}
is dominated by a geometric number---with success probability
$c_2/\log x\cdot\ldots\cdot\log_{(m)}x$---at the most of summands taken from
a uniformly integrable family of random variables,
which yields the first conclusion of theorem,
by Lemma \ref{l:uni.stop}(ii) with $E_x=\log x\cdot\ldots\cdot\log_{(m)}x$.
In particular, for some $c_3<\infty$,
\begin{eqnarray*}%\label{renewal.for.Y}
H_y^X(x,2x] &\le& c_3 x^2\log x\cdot\ldots\cdot\log_{(m)}x
\quad\mbox{for all } x> \widehat x\mbox{ and }y.
\end{eqnarray*}
\qed\end{proof}

\section{Factorisation result for renewal function with weights}
\sectionmark{Factorisation result for renewal function}
\label{sec:renewal.w}

In this section, either $n(x)\equiv\infty$ or $n(x)\to\infty$ as $x\to\infty$.
Let $A(x)\subset\R$ be a family of Borel sets.

For a function $q(z)\ge 0$ on $\R$, we look at the impact of $q(z)$
on the asymptotic behaviour of the partial renewal measure with weights
\begin{eqnarray}\label{def.Hq}
\sum_{n=0}^{n(x)} \E\bigl\{e^{-\sum_{k=0}^{n-1}q(X_k)};\ X_n\in A(x)\bigr\},
\end{eqnarray}
compared to that of
\begin{eqnarray*}
\sum_{n=0}^{n(x)} \P\{X_n\in A(x)\}.
\end{eqnarray*}

\begin{lemma}\label{thm:renewal.2}
Let $a(x)>0$ be a function on $\R^+$. Let the family of random variables
\begin{eqnarray}\label{conv.to.r2.square}
a(x)\sum_{n=0}^{n(x)}\I\{X_n\in A(x)\},\quad x>0,\ X_0=z,
\end{eqnarray}
be uniformly integrable and let there exist a $c>0$ such that,
for all $N\in\Z^+$ and $z\in\R$,
\begin{eqnarray}\label{conv.to.r2.uni}
a(x) \sum_{n=N}^{n(x)}\P_z\{X_n\in A(x)\} &\to& c \quad\mbox{as }x\to\infty.
\end{eqnarray}
If $q(z)\ge 0$, then
\begin{eqnarray*}
a(x) \sum_{n=0}^{n(x)} \E\bigl\{e^{-\sum_{k=0}^{n-1}q(X_k)};\ X_n\in A(x)\bigr\}
&\to& c\E e^{-\sum_{k=0}^\infty q(X_k)}\quad\mbox{as }x\to\infty.
\end{eqnarray*}
\end{lemma}

\begin{proof}
The conditions \eqref{conv.to.r2.square} and
\eqref{conv.to.r2.uni} imply that
\begin{eqnarray*}
a(x) \sum_{n=N}^{n(x)}\P\{X_n\in A(x)\} &\to& c \quad\mbox{as }x\to\infty
\end{eqnarray*}
for any distribution of $X_0$ and for all $N$. Therefore, for any fixed $N\in\N$,
\begin{eqnarray*}
a(x) \sum_{n=0}^{N-1}\P\{X_n\in A(x)\} &\to& 0\quad\mbox{as }x\to\infty.
\end{eqnarray*}
Then
\begin{eqnarray*}
\lefteqn{a(x) \sum_{n=0}^{n(x)} \E\bigl\{e^{-\sum_{k=0}^{n-1}q(X_k)};\ X_n\in A(x)\bigr\}
-c\E e^{-\sum_{k=0}^\infty q(X_k)}}\\
&=& a(x)\biggl(\sum_{n=0}^{n(x)} \E\bigl\{e^{-\sum_{k=0}^{n-1}q(X_k)};\ X_n\in A(x)\bigr\}
-\E e^{-\sum_{k=0}^\infty q(X_k)} \sum_{n=0}^{n(x)}\P\{X_n\in A(x)\}\biggr)+o(1)\\
&=& a(x)\biggl(\E\sum_{n=N}^{n(x)} \Bigl(e^{-\sum_{k=0}^{n-1}q(X_k)}
-\E e^{-\sum_{k=0}^\infty q(X_k)}\Bigr)
\I\{X_n\in A(x)\}\biggr)+o(1).
\end{eqnarray*}
In its turn, the mean on the right hand side equals 
the sum of the mean values of the following random variables:
\begin{eqnarray*}
\sum_{n=N}^{n(x)} &=& \zeta_1(x,N) +\zeta_2(x,N) +\zeta_3(x,N),
\end{eqnarray*}
where
\begin{eqnarray*}
\zeta_1(x,N) &:=& \sum_{n=N}^{n(x)} \Bigl(e^{-\sum_{k=0}^{N-1}q(X_k)}
- \E e^{-\sum_{k=0}^{N-1} q(X_k)}\Bigr)\I\{X_n\in A(x)\},\\
\zeta_2(x,N) &:=& \sum_{n=N}^{n(x)} \Bigl(e^{-\sum_{k=0}^{n-1}q(X_k)}
-e^{-\sum_{k=0}^{N-1}q(X_k)}\Bigr) \I\{X_n\in A(x)\},\\
\zeta_3(x,N) &:=& \sum_{n=N}^{n(x)} \Bigl(\E e^{-\sum_{k=0}^{N-1} q(X_k)}
- \E e^{-\sum_{k=0}^\infty q(X_k)}\Bigr) \I\{X_n\in A(x)\}.
\end{eqnarray*}
By the condition \eqref{conv.to.r2.square}, both families of random variables
$\{a(x)\zeta_2(x,N),\ x>0,\ N\ge 1\}$ and $\{a(x)\zeta_3(x,N),\ x>0,\ N\ge 1\}$
are uniformly integrable.
Then, taking into account that $q(z)\ge 0$ implies the convergence
\begin{eqnarray}\label{conve.of.e}
e^{-\sum_{k=0}^{N-1}q(X_k)} &\stackrel{a.s.}\to& e^{-\sum_{k=0}^\infty q(X_k)}
\quad\mbox{as }N\to\infty,
\end{eqnarray}
we conclude that both $\sup_x a(x) |\E\zeta_2(x,N)|$
and $\sup_x a(x) |\E\zeta_3(x,N)|$ go to $0$ as $N\to\infty$.
This proves the required result when we show in addition that, for any fixed $N$,
\begin{eqnarray}\label{E_1xN}
a(x)\E\zeta_1(x,N) &\to& 0\quad\mbox{as }x\to\infty.
\end{eqnarray}
Indeed, conditioning on $X_0$, \ldots, $X_{N-1}$ leads to the equality
\begin{eqnarray*}
\lefteqn{\E\zeta_1(x,N)}\\
&=& \E \Bigl\{\Bigl(e^{-\sum_{k=0}^{N-1}q(X_k)}
- \E e^{-\sum_{k=0}^{N-1} q(X_k)}\Bigr)
\E\Bigl\{\sum_{n=N}^{n(x)} \I\{X_n\in A(x)\}\Big| X_0,\ldots,X_{N-1}\Bigr\}\Bigr\}\\
&&\hspace{10mm} =\E \Bigl\{\Bigl(e^{-\sum_{k=0}^{N-1}q(X_k)}
- \E e^{-\sum_{k=0}^{N-1} q(X_k)}\Bigr)
\E_{X_{N-1}}\sum_{n=N}^{n(x)} \I\{X_n\in A(x)\}\Bigr\},
\end{eqnarray*}
by the Markov property. By the uniform integrability \eqref{conv.to.r2.square},
the family of random variables
\begin{eqnarray*}
a(x)\Bigl(e^{-\sum_{k=0}^{N-1}q(X_k)}-\E e^{-\sum_{k=0}^{N-1} q(X_k)}\Bigr)
\E_{X_{N-1}}\sum_{n=N}^{n(x)} \I\{X_n\in A(x)\},\quad x>0,
\end{eqnarray*}
is uniformly integrable too. By the condition \eqref{conv.to.r2.uni},
\begin{eqnarray*}
a(x)\E_{X_{N-1}}\sum_{n=N}^{n(x)} \I\{X_n\in A(x)\}
&\stackrel{a.s.}\to& c\quad\mbox{as } x\to\infty.
\end{eqnarray*}
This allows us to conclude that
\begin{eqnarray*}
a(x)\E\zeta_1(x,N) &\to&
c\E\Bigl(e^{-\sum_{k=0}^{N-1}q(X_k)}-\E e^{-\sum_{k=0}^{N-1} q(X_k)}\Bigr)
\ =\ 0\quad\mbox{as }x\to\infty,
\end{eqnarray*}
and \eqref{E_1xN} follows which completes the proof.
\qed\end{proof}

\begin{lemma}\label{thm:renewal.2.E}
Let $g_n:\R^{n+1}\to\R$ be a sequence of uniformly bounded functions and let
$E\in\R$ be a number such that, for all $N\in\N$
and $z_0$, \ldots, $z_N$,
\begin{eqnarray}\label{conv.to.r2.uni.E}
\E\{g_n(X_0,\ldots,X_n)\mid X_0=z_0,\ldots,X_N=z_N\}
&\to& E \quad\mbox{as }n\to\infty.
\end{eqnarray}
If $q(z)\ge 0$, then
\begin{eqnarray*}
\E e^{-\sum_{k=0}^{n-1}q(X_k)}g_n(X_0,\ldots,X_n) &\to&
E\cdot\E e^{-\sum_{k=0}^\infty q(X_k)}\quad\mbox{as }n\to\infty.
\end{eqnarray*}
\end{lemma}

\begin{proof}
Fix any $N\in\N$. Then
\begin{eqnarray*}
\lefteqn{\Bigl|\E e^{-\sum_{k=0}^{n-1}q(X_k)}g_n(X_0,\ldots,X_n)
-\E g_n(X_0,\ldots,X_n)\E e^{-\sum_{k=0}^\infty q(X_k)}\Bigr|}\\
&&\hspace{20mm}\le\ \Bigl|\E \Bigl(e^{-\sum_{k=0}^{N-1}q(X_k)}
-\E e^{-\sum_{k=0}^{N-1} q(X_k)}\Bigr)g_n(X_0,\ldots,X_n)\Bigr|\\
&&\hspace{40mm}+ \|g_n\|_\infty\E \Bigl|e^{-\sum_{k=0}^{n-1}q(X_k)}
-e^{-\sum_{k=0}^{N-1} q(X_k)}\Bigr|\\
&&\hspace{60mm}+ \|g_n\|_\infty\Bigl|\E e^{-\sum_{k=0}^{N-1}q(X_k)}
-\E e^{-\sum_{k=0}^\infty q(X_k)}\Bigr|\\
&&\hspace{20mm}=:\ |E_1(N,n)|+E_2(N,n)+E_3(N).
\end{eqnarray*}
We have $E_2(N,n)\to 0$ and $E_3(N)\to 0$ as $n$, $N\to\infty$
by the dominated convergence in \eqref{conve.of.e} because $q(z)\ge 0$.
Further, conditioning on $X_0$, \ldots, $X_{N-1}$
leads to the equality and the convergence
\begin{eqnarray*}
E_1(N,n) &=& \E \Bigl\{\Bigl(e^{-\sum_{k=0}^{N-1}q(X_k)}
- \E e^{-\sum_{k=0}^{N-1} q(X_k)}\Bigr)
\E\{g_n(X_0,\ldots,X_n)\mid X_0,\ldots,X_{N-1}\}\Bigr\}\\
&\to& E\cdot\E \Bigl(e^{-\sum_{k=0}^{N-1}q(X_k)}
- \E e^{-\sum_{k=0}^{N-1} q(X_k)}\Bigr)\quad\mbox{as }n\to\infty,
\end{eqnarray*}
by the condition \eqref{conv.to.r2.uni.E},
which allows us to conclude that, for any fixed $N$,
\begin{eqnarray*}
E_1(N,n) &\to& 0\quad\mbox{as }n\to\infty,
\end{eqnarray*}
and the proof is complete.
\qed\end{proof}

\begin{lemma}\label{thm:renewal.2.p}
Let $p$ be a number between $0$ and $1$ and $A_n\subset\R$
be a sequence of Borel sets such that, for all $z$,
\begin{eqnarray}\label{conv.to.r2.uni.p}
\P_z\{X_n\in A_n\} &\to& p \quad\mbox{as }n\to\infty.
\end{eqnarray}
If $q(z)\ge 0$, then
\begin{eqnarray*}
\E e^{-\sum_{k=0}^{n-1}q(X_k)}\I\{X_n\in A_n\} &\to&
p\E e^{-\sum_{k=0}^\infty q(X_k)}\quad\mbox{as }n\to\infty.
\end{eqnarray*}
\end{lemma}

\begin{proof}
Take $g_n(X_0,\ldots,X_n)=\I\{X_n\in A_n\}$ which is a bounded function
satisfying the condition \eqref{conv.to.r2.uni.E} with $E=p$ because
\begin{eqnarray*}
\P\{X_n\in A\mid X_0,\ldots,X_N\} &=& \P\{X_n\in A\mid X_N\},
\end{eqnarray*}
by the Markov property and because
\begin{eqnarray*}
\P\{X_n\in A\mid X_N\} &\stackrel{a.s.}\to& p\quad\mbox{as } n\to\infty,
\end{eqnarray*}
by the condition \eqref{conv.to.r2.uni.p}.
\qed\end{proof}

\section{Convergence to $\Gamma$-distribution for transient chain}
\sectionmark{Convergence to $\Gamma$-distribution}
\label{sec:gamma}

In this section we are interested in the
growth rate of a Markov chain $\{X_n\}$ on $\R$
that tends to infinity with probability 1 as $n\to\infty$
which happens when the chain is transient.
\index{Transience!convergence to!$\Gamma$-distribution}

\begin{theorem}\label{thm:gamma}
Suppose there exist $b>0$ and $\mu>b/2$ such that,
for some increasing function $s(x)=o(x)$,
\begin{eqnarray}\label{1.2.G}
m_1^{[s(x)]}(x)\sim \mu/x\ &\mbox{and}&\ m_2^{[s(x)]}(x)\to b
\quad\mbox{ as }x\to\infty,
\end{eqnarray}
and, for all $x>\widehat x$,
\begin{eqnarray}\label{rec.3.1.gamma}
\P\{|\xi(x)|>s(x)\} &\le& p(x)/x,\\
\label{rec.3.1.e}
\E\{|\xi(x)|;\ \xi(x)\le -s(x)\} &\le& p(x),
\end{eqnarray}
where a decreasing function $p(x)>0$ is integrable at infinity. If
\begin{eqnarray}\label{eq:irreducibility.gamma}
\limsup_{n\to\infty}X_n=\infty\quad\mbox{with probability }1,
\end{eqnarray}
then $X_n^2/nb$ converges weakly to a $\Gamma_{1/2+\mu/b,1/2}$-distribution
with mean $1+2\mu/b$ and variance $2(1+2\mu/b)$ whose probability density function is
\begin{eqnarray*}
\frac{1}{\Gamma(1/2+\mu/b)2^{1/2+\mu/b}}x^{\mu/b-1/2}e^{-x/2},\quad x>0.
\end{eqnarray*}
\end{theorem}

Let us give a sufficient condition for \eqref{rec.3.1.gamma}
and \eqref{rec.3.1.e} to hold.
If the family $\{|\xi(x)|,\ x\ge 0\}$, possesses a majorant $\Xi$,
that is, $|\xi(x)|\le_{st}\Xi$ for all $x$,
which is square integrable, $\E\Xi^2<\infty$, then there exists an
increasing function $s(x)=o(x)$ such that
\eqref{rec.3.1.gamma} and \eqref{rec.3.1.e} hold,
see Lemma \ref{l:maj.p.e} with $\gamma=2$, $\alpha=1$, and $\beta=0$, $1$.
Hence the following result.

\begin{corollary}\label{cor:gamma.maj}
Assume that, for some $b>0$ and $\mu>-b/2$,
$m_1(x)\sim \mu/x$ and $m_2(x)\to b$ as $x\to\infty$.
Assume that the family $\{|\xi(x)|,\ x\in\R\}$
possesses a square integrable majorant $\Xi$, that is,
$\E\Xi^2<\infty$ and $\xi^2(x) \le_{st} \Xi$ for all $x$.
If the condition \eqref{eq:irreducibility.gamma} holds, 
then $X_n^2/nb$ converges weakly to a $\Gamma$-distribution
with mean $1+2\mu/b$ and variance $2(1+2\mu/b)$.
\end{corollary}

\begin{theopargself}
\begin{proof}[of Theorem \ref{thm:gamma}]
The proof is based on the method of moments,
see e.g. Durrett \cite[Theorem 3.3.26]{Durrett}.

Consider a modified Markov chain $\{\widetilde X_n\}$ on the same probability 
space as $X$ with jumps $\widetilde\xi(x)=\xi(x)\I\{|\xi(x)|\le s(x)\}$.
If $\{\widetilde X_n\}$ does not satisfy the weak irreducibility condition 
\eqref{eq:irreducibility.gamma}, then we can increase the value of $s(x)$
on some set bounded above in such a way that then $\{\widetilde X_n\}$ does satisfy 
\eqref{eq:irreducibility.gamma}. Indeed, it follows from the condition
\eqref{1.2.G} that there exist a sufficiently high level $x_0$ 
and an $\varepsilon>0$
such that $\P\{\xi(x)\ge\varepsilon\}>0$ for all $x\ge x_0$.
Then it suffices to increase $s(x)$ on the set $(-\infty,x_0]$
to ensure the condition \eqref{eq:irreducibility.gamma} for $\{\widetilde X_n\}$.

Since \eqref{1.2.G} holds with $\mu>b/2$, $\{\widetilde X_n\}$ satisfies
the condition \eqref{r-cond.5.tr.inf} for any $\varepsilon\in(0,2\mu/b-1)$.
Moreover, \eqref{rec.3.1.gamma} implies \eqref{rec.1a.inf}
with a possibly slower decreasing $p(x)$ which is still integrable.
Therefore, Theorem \ref{thm:transience.inf} is applicable to $\{\widetilde X_n\}$, 
so we conclude the transience and the convergence, for all $z$,
\begin{eqnarray*}
\P\{\widetilde X_n>z\mbox{ for all }n\ge 0\mid X_0=y\} &\to& 1 \quad\mbox{as } y\to\infty.
\end{eqnarray*}

By Theorem \ref{thm:Hy.above}, there exist $c$ and $x_*$ such that
\begin{eqnarray*}
H_y^{\widetilde X}(x,2x) &\le& cx^2\quad\mbox{for all } x>x_*.
\end{eqnarray*}
So, all the conditions of Lemma \ref{l:XY.equiv} are satisfied
for the chains $Y=X$ and $Z=\widetilde X$.
By Theorem \ref{thm:transience.inf}, the chain $Z=\widetilde X$ tends 
to infinity as $n\to\infty$,
so it suffices to prove weak convergence to the same
$\Gamma$-distribution for the process $\{Z_n\}$
with jumps $\zeta(x)=\xi(x)\I\{|\xi(x)|\le s(x)\}$, 
see the discussion at the end of Section \ref{sec:thresholds}. 
That is, it is sufficient to show that
\begin{eqnarray}\label{YAtoGamma.cond}
\frac{Z_n^2}{nb} &\Rightarrow& \Gamma_{(2\mu+b)/2b,2}
\quad\mbox{as }n\to\infty.
\end{eqnarray}
%Recall that the process $Z_n$ is transient, see \eqref{Z.disp.to.infty.as}.

For all $x$,
\begin{eqnarray}\label{1.Y}
\E\zeta(x) &=& m_1^{[s(x)]}(x)
\quad \mbox{ and }\quad \E\zeta^2(x)\ = \ m_2^{[s(x)]}(x).
\end{eqnarray}
In addition, the inequality $|\zeta(x)|\le s(x)=o(x)$
implies that, for all $j\ge 3$,
\begin{eqnarray}\label{k.Y}
|\E\zeta^j(x)| &\le& m_2^{[s(x)]}(x)s^{j-2}(x)=o(x^{j-2})
\quad\mbox{as }x\to\infty.
\end{eqnarray}

Let us compute the mean of the increment of $Z_n^{2i}$.
For $i=1$ we have
\begin{eqnarray*}
\E\{Z^2_{n+1}-Z^2_n\mid Z_n=x\}
&=& \E(2x\zeta(x)+\zeta^2(x))\\
&=& 2\mu+b+o(1)\quad\mbox{as }x\to\infty,
\end{eqnarray*}
by \eqref{1.Y} and \eqref{1.2.G}.
Applying now the convergence of $Z_n$ to infinity we get
\begin{eqnarray*}
\E(Z^2_{n+1}-Z^2_n) &\to& 2\mu+b \quad\mbox{as }n\to\infty.
\end{eqnarray*}
Hence,
\begin{eqnarray}\label{asy.2}
\E Z^2_n &\sim& (2\mu+b)n \quad\mbox{as }n\to\infty.
\end{eqnarray}
For $i\ge2$, we have
\begin{eqnarray}\label{incr.n.2i.x}
\lefteqn{\E\{Z^{2i}_{n+1}-Z^{2i}_n\mid Z_n=x\}}\nonumber\\
&=& \E\Biggl(2ix^{2i-1}\zeta(x)+i(2i-1)x^{2i-2}\zeta^2(x)
+\sum_{l=3}^{2i}x^{2i-l}\zeta^l(x)\binom{2i}{l}\Biggr)\nonumber\\
&=& i[2\mu+(2i-1)b+o(1)]x^{2i-2}
+\sum_{l=3}^{2i}x^{2i-l}\E\zeta^l(x)\binom{2i}{l}
\end{eqnarray}
as $x\to\infty$, by \eqref{1.Y}. Owing to \eqref{k.Y},
\begin{eqnarray*}
\sum_{l=3}^{2i}x^{2i-l}\E\zeta^l(x)\binom{2i}{l}
&=& \sum_{l=3}^{2i}x^{2i-l}o(x^{l-2}) = o(x^{2i-2})
\quad\mbox{as }x\to\infty.
\end{eqnarray*}
%Thus,
%\begin{eqnarray*}
%\sum_{l=3}^{2i}x^{2i-l}{\mathbb E}\eta^l(n,Y_k(n))\binom{2i}{l}
%&=& o({\mathbb E}Y_k^{2i-2}(n))
%+\sum_{l=3}^{2i}{\mathbb E}Y_k^{2i-l}(n)o(n^{(l-2)/2})
%\end{eqnarray*}
%as $k$, $n\to\infty$, $k\le n$.
Substituting this into \eqref{incr.n.2i.x} with $x=Z_n$
and taking into account convergence $Z_n\to\infty$, we deduce that
\begin{eqnarray}\label{incr.n.2i}
\E\{Z^{2i}_{n+1}-Z^{2i}_n\}
&=& i[2\mu+(2i-1)b+o(1)]\E Z_n^{2i-2}
\quad\mbox{as }n\to\infty.
\end{eqnarray}
In particular, for $i=2$ we get
\begin{eqnarray*}
\E\{Z^4_{n+1}-Z^4_n\} &=& 2(2\mu+3b+o(1))\E Z_n^2\\
&\sim& 2(2\mu+3b)(2\mu+b)n\quad\mbox{as }n\to\infty,
\end{eqnarray*}
due to \eqref{asy.2}. This implies that
\begin{eqnarray*}
\E Z^4_n &\sim& (2\mu+3b)(2\mu+b)n^2\quad\mbox{as }n\to\infty.
\end{eqnarray*}
By induction, we deduce from \eqref{incr.n.2i} that, for all $i\ge 1$,
\begin{eqnarray*}
\E Z^{2i}_n &\sim& (nb)^i\prod_{k=1}^{i} (2\mu/b+2k-1)
\quad\mbox{as }n\to\infty,
\end{eqnarray*}
which yields convergence of all moments of $Z_n^2/nb$
to that of Gamma distribution with mean $1+2\mu/b$
and variance $2(1+2\mu/b)$. Hence \eqref{YAtoGamma.cond}
is proven and the proof is complete.
\qed\end{proof}
\end{theopargself}

\section{Convergence to Gamma distribution for non-positive chain}
\sectionmark{Convergence to $\Gamma$-distribution}
\label{sec:pre-st.null}

The next result is on the convergence to a $\Gamma$-distribution
covers both transient and null-recurrent chains.
\index{Null recurrence!convergence to!$\Gamma$-distribution}

\begin{theorem}\label{thm:pre-st.null}
Assume that, for some $b>0$ and $\mu>-b/2$,
\begin{eqnarray}\label{1.2.pre}
m_1(x)\sim \mu/x\ \mbox{ and }\
m_2(x)\to b
\quad\mbox{ as }x\to\infty
\end{eqnarray}
and that the family $\{\xi^2(x),\ x\in\R\}$
possesses an integrable majorant $\Xi$, that is,
${\mathbb E}\Xi<\infty$ and
\begin{eqnarray}\label{mom.cond.ui}
\xi^2(x) &\le_{st}& \Xi
\quad\mbox{ for all }x.
\end{eqnarray}
If $X_n\to\infty$ in probability as $n\to\infty$, then
$X_n^2/nb$ converges weakly to a $\Gamma$-distribution
with mean $1+2\mu/b$ and variance $2(1+2\mu/b)$.
\end{theorem}

The main difference between this result and Theorem~\ref{thm:gamma}
is that here we impose conditions on the asymptotic
behaviour of the first two {\it full} moments of jumps, $m_1(x)$ and $m_2(x)$.
Further, as we have commented after Theorem~\ref{thm:gamma},
\eqref{mom.cond.ui} implies \eqref{rec.3.1.gamma}.
The rationale behind these more restrictive assumptions is that
the renewal function of any null-recurrent chain is infinite,
hence we cannot
use time homogeneous truncations as it has been done in the proof
of Theorem~\ref{thm:gamma}. In order to prove
Theorem~\ref{thm:pre-st.null} we introduce truncation of jumps
which depends not only on the spatial coordinate $x$ but also on time $n$.

\begin{proof}%[Proof of Theorem \ref{thm:pre-st.null}]
For any $n\in\N$,
consider a new Markov chain $Y_k(n)$, $k=0$, $1$, $2$, \ldots,
with transition probabilities depending on the parameter $n$,
whose jump $\eta(n,x)$ is just the original jump $\xi(x)$ 
truncated at levels $\pm(x\vee \sqrt n)$
depending on both point $x$ and time $n$, that is,
$$
\eta(n,x)=\left\{
\begin{array}{ll}
\xi(x) &\mbox{if }\ |\xi(x)|\le x\vee \sqrt n\\
0&\mbox{else.}
\end{array}
\right.
$$
Given $Y_0(n)=X_0$, the probability of discrepancy
between the trajectories of $\{Y_k(n)\}$ and $\{X_k\}$
by time $n$ is at the most
\begin{eqnarray}\label{BK}
\P\{Y_k(n)\neq X_k\mbox{ for some }k\le n\}
&\le& \sum_{k=0}^{n-1}\P\{|X_{k+1}-X_k|\ge\sqrt n\}\nonumber\\
&\le& n\P\{\Xi\ge n\}\nonumber\\
&\le& \E\{\Xi;\Xi\ge n\} \to 0
\ \mbox{ as }n\to\infty.
\end{eqnarray}
Since $X_n\to\infty$ in probability, \eqref{BK} implies that,
for every $c$,
\begin{eqnarray}\label{Y.to.infty.new}
\inf_{n>n_0,k\in[n_0,n]} \P\{Y_k(n)>c\} &\to& 1
\quad\mbox{ as }n_0\to\infty.
\end{eqnarray}

By the choice of the truncation level,
$$
|\xi(x)-\eta(n,x)|\ \le\ |\xi(x)|\I\{|\xi(x)|>x\}.
$$
Therefore, by the condition \eqref{mom.cond.ui},
\begin{eqnarray}\label{1.Y.new}
\E\eta(n,x) &=& \E\xi(x)+o(1/x)
\quad\mbox{ as }x\to\infty \mbox{ uniformly for all }n
\end{eqnarray}
and
\begin{eqnarray}\label{2.Y.new}
\E\eta^2(n,x) &=& \E\xi^2(x)+o(1)
\quad\mbox{ as }x\to\infty \mbox{ uniformly for all }n.
\end{eqnarray}
In addition, the inequality $|\eta(n,x)|\le x\vee\sqrt n$
and the condition \eqref{mom.cond.ui} imply that, for all $j\ge 3$,
\begin{eqnarray}\label{k.Y.new}
\E\eta^j(n,x) &=& o(x^{j-2}+n^{(j-2)/2})
\quad\mbox{ as }x\to\infty \mbox{ uniformly for all }n.
\end{eqnarray}

Let us evaluate the mean of the increment of $Y_k^j(n)$.
For $j=2$ we have
\begin{eqnarray*}
\E\{Y^2_{k+1}(n)-Y^2_k(n)|Y_k(n)=x\}
&=& \E(2x\eta(n,x)+\eta^2(n,x))\\
&=& 2\mu+b+o(1)
\end{eqnarray*}
as $x\to\infty$ uniformly for all $n$,
by \eqref{1.Y.new} and \eqref{2.Y.new}.
Applying now \eqref{Y.to.infty.new} we get
\begin{eqnarray*}
\E(Y^2_{k+1}(n)-Y^2_k(n)) &\to& 2\mu+b
\quad\mbox{ as }k,n\to\infty,\ k\le n.
\end{eqnarray*}
Hence,
\begin{eqnarray}\label{asy.2.new}
\E Y^2_n(n) &\sim& (2\mu+b)n \quad\mbox{as }n\to\infty.
\end{eqnarray}
Let now $j=2i$, $i\ge2$. We have
\begin{eqnarray}\label{incr.n.2i.x.new}
\lefteqn{\E\{Y^{2i}_{k+1}(n)-Y^{2i}_k(n)|Y_k(n)=x\}}\nonumber\\
&=& \E\Biggl(2ix^{2i-1}\eta(n,x)+i(2i-1)x^{2i-2}\eta^2(n,x)
+\sum_{l=3}^{2i}x^{2i-l}\eta^l(n,x)\binom{2i}{l}\Biggr)\nonumber\\
&=& i[2\mu+(2i-1)b+o(1)]x^{2i-2}
+\sum_{l=3}^{2i}x^{2i-l}{\mathbb E}\eta^l(n,x)\binom{2i}{l}
\end{eqnarray}
as $x\to\infty$ uniformly for all $n$,
by \eqref{1.Y.new} and \eqref{2.Y.new}. Owing to \eqref{k.Y.new},
\begin{eqnarray*}
\sum_{l=3}^{2i}x^{2i-l}\E\eta^l(n,x)\binom{2i}{l}
&=& \sum_{l=3}^{2i}x^{2i-l}o(x^{l-2}+n^{(l-2)/2})\\
&=& o(x^{2i-2})+\sum_{l=3}^{2i}x^{2i-l}o(n^{(l-2)/2})
\end{eqnarray*}
as $x\to\infty$ uniformly for all $n$.
%Thus,
%\begin{eqnarray*}
%\sum_{l=3}^{2i}x^{2i-l}{\mathbb E}\eta^l(n,Y_k(n))\binom{2i}{l}
%&=& o({\mathbb E}Y_k^{2i-2}(n))
%+\sum_{l=3}^{2i}{\mathbb E}Y_k^{2i-l}(n)o(n^{(l-2)/2})
%\end{eqnarray*}
%as $k$, $n\to\infty$, $k\le n$.
Substituting this into \eqref{incr.n.2i.x.new} with $x=Y_k(n)$
and taking into account \eqref{Y.to.infty.new}, we deduce that
\begin{eqnarray}\label{incr.n.2i.new}
\E\{Y^{2i}_{k+1}(n)-Y^{2i}_k(n)\}
&=& i[2\mu+(2i-1)b+o(1)]\E Y_k^{2i-2}(n)\nonumber\\
&&+\sum_{l=3}^{2i} \E Y_k^{2i-l}(n)o(n^{(l-2)/2}).
\end{eqnarray}
In particular, for $j=2i=4$ we get
\begin{eqnarray*}
\E\{Y^4_{k+1}(n)-Y^4_k(n)\}
&=& 2(2\mu+3b)\E Y_k^2(n)+\E Y_k(n) o(\sqrt n)+o(n)\\
&\sim& 2(2\mu+3b)(2\mu+b)n,
\end{eqnarray*}
due to \eqref{asy.2.new}. It implies that
\begin{eqnarray*}
\E Y^4_n(n) &\sim& (2\mu+3b)(2\mu+b)n^2
\quad\mbox{ as }n\to\infty.
\end{eqnarray*}
By induction, we deduce from \eqref{incr.n.2i.new} that
\begin{eqnarray*}
\E Y^{2i}_n(n)
&\sim& (nb)^i\prod_{k=1}^{i} (2\mu/b+2k-1)\quad\mbox{ as }n\to\infty,
\end{eqnarray*}
which yields---by the method of moments---that $Y^2_n(n)/nb$ converges
weakly to a $\Gamma$-distribution with mean $1+2\mu/b$ and variance $2(1+2\mu/b)$.
Together with \eqref{BK} this completes the proof.
\qed\end{proof}

\section{Functional convergence to Bessel process for non-positive chain}
\sectionmark{Functional convergence to Bessel process}
\label{sec:gamma.func}

Once the weak convergence of $X_n^2/n$ to a $\Gamma$-distribution
is proven, it is natural to guess diffusion approximation to $X_n^2/n$
by a Bessel process. This question was originally positively
answered by Lamperti\index{Lamperti} in \cite{Lamp62}.
In the next theorem the result of Lamperti is given under minimal moment conditions;
our proof is based on the method of moments as the proof
of the weak convergence to a $\Gamma$-distribution.

Introduce a family of piece-wise constant processes
$$
X^{(n)}(t)\ =\ \frac{X_{[tn]}}{\sqrt{bn}},\quad t\in[0,1],
$$
so $X^{(n)}(t)\in D[0,1]$ where $D[0,1]$ is the space of real-valued
functions on $[0,1]$ which are right continuous with left limits.
\index{Transience!convergence to!Bessel process}

\begin{theorem}\label{thm:gamma.func}
Suppose that either $\mu>b/2$ and the conditions of Theorem \ref{thm:gamma}
hold or $\mu>-b/2$ and the conditions of Theorem \ref{thm:pre-st.null} hold.
Then the process $\{X^{(n)}(t)\}$ converges weakly in $D[0,1]$ to a Bessel
process $Bes(t)$ starting at zero,
with reflecting boundary condition in null-recurrent case,
with drift $\mu/bx$ and diffusion coefficient $1$, that is,
$f(X^{(n)}(\cdot))\Rightarrow f(Bes(\cdot))$ as $n\to\infty$ for all
bounded functionals $f:D[0,1]\to\R$ continuous in the Skorokhod topology.
\end{theorem}

Notice that since the limiting process is continuous,
the last result is equivalent to the weak convergence in the space $C[0,1]$
if we define $\{X^{(n)}(t)\}$ as a continuous piece-wise linear process
whose trajectory connects points $(k/n,X_k/\sqrt{bn})$ by segments,
for justification see, e.g. Ethier\index{Ethier}
and Kurtz\index{Kurtz} \cite[Proposition 10.4]{EK}.

All the arguments in the proof below are still valid if we consider
a triangular array setting where the initial distribution of the chain
depends on $n$ in such a way that, for some $x_0\in\Rp$,
\begin{eqnarray*}
X^{(n)}_0/\sqrt{bn} &\stackrel{p}\to& x_0\quad\mbox{as }n\to\infty.
\end{eqnarray*}
Then the process $\{X^{(n)}(t)\}$ converges weakly in $D[0,1]$
to a Bessel process $Bes(t)$ with starting point $x_0$,
drift $\mu/bx$ and diffusion coefficient $1$.
In its turn, this implies that, if 
\begin{eqnarray*}
X^{(n)}_0/\sqrt{bn} &\Rightarrow& \nu\quad\mbox{as }n\to\infty
\end{eqnarray*}
for some probability distribution $\nu$ on $\Rp$,
then the process $\{X^{(n)}(t)\}$ converges weakly in $D[0,1]$
to a Bessel process $Bes(t)$ with initial distribution $\nu$.

\begin{proof}
Let the conditions of Theorem \ref{thm:gamma} hold,
then as in the proof of that theorem it is sufficient to prove weak
convergence to a Bessel process of the sequence of $D[0,1]$-processes 
$\{Z^{(n)}(t)\}$ which are defined as
$$
Z^{(n)}(t)\ =\ \frac{Z_{[tn]}}{\sqrt{bn}},\quad t\in[0,1],
$$
where the process $\{Z_k\}$ is defined in Section \ref{sec:thresholds}.

By Prokhorov's Theorem, we need to prove weak convergence of
finite dimensional distributions and tightness in $D[0,1]$.
We start with finite-dimensional distributions. By the method of moments,
it suffices to prove that, for any sequence of time epochs
$t_1<t_2<\ldots<t_k$ and natural numbers $i_1$, $i_2$, \ldots, $i_k$,
the mixed moment
\begin{eqnarray}\label{mixed.Z}
\E Z^{(n)}(t_1)^{2i_1}\ldots Z^{(n)}(t_k)^{2i_k}
\end{eqnarray}
converges to that of the Bessel process $Bes$, that is, to
\begin{equation}\label{mixed.B}
\E Bes^{2i_1}(t_1)\ldots Bes^{2i_k}(t_k).
\end{equation}
Indeed, conditioning on $Z^{(n)}(t_1)$, \ldots,
$Z^{(n)}(t_{k-1})$ yields an equality
\begin{eqnarray*}
\lefteqn{\E\{Z^{(n)}(t_1)^{2i_1}\ldots
Z^{(n)}(t_{k-1})^{2i_{k-1}}}\\
&& \times[Z^{(n)}(t_k)^{2i_k}
-Z^{(n)}(t_{k-1})^{2i_k}
+Z^{(n)}(t_{k-1})^{2i_k}]
\mid Z^{(n)}(t_1),\ldots,Z^{(n)}(t_{k-1})\}\\
&=& Z^{(n)}(t_1)^{2i_1}\ldots
Z^{(n)}(t_{k-1})^{2i_{k-1}+2i_k}\\
&&\hspace{8mm}+ Z^{(n)}(t_1)^{2i_1}\ldots
Z^{(n)}(t_{k-1})^{2i_{k-1}}
\E\biggl\{\frac{Z_{[nt_k]}^{2i_k}
-Z_{[nt_{k-1}]}^{2i_k}}{(nb)^{i_k}}
\ \bigg|\ Z_{[nt_{k-1}]}\biggr\}.
\end{eqnarray*}
The conditional expectation in the second term on the right hand side equals
\begin{eqnarray*}
\sum_{j=[nt_{k-1}]}^{[nt_k]-1}
\E\biggl\{\frac{Z_{j+1}^{2i_k}
-Z_j^{2i_k}}{(nb)^{i_k}}
\ \bigg|\ Z_{[nt_{k-1}]}(n)\biggr\},
\end{eqnarray*}
where the $j$th term in the sum, by \eqref{incr.n.2i},
may be evaluated as follows
\begin{eqnarray*}
\E\biggl\{\frac{Z_{j+1}^{2i_k}-Z_j^{2i_k}}{(nb)^{i_k}}
\ \bigg|\ Z_{[nt_{k-1}]}(n)\biggr\} &=&
(c_{i_k}+o(1)) \E\biggl\{\frac{Z_j^{2i_k-2}}{(nb)^{i_k}}
\ \bigg|\ Z_{[nt_{k-1}]}\biggr\},
\end{eqnarray*}
where $c_i=i(2\mu+(2i-1)b)$. In the case $i_k=1$ we get
\begin{eqnarray*}
\E\biggl\{\frac{Z_{j+1}^2-Z_j^2}{nb}\ \bigg|\
Z_{[nt_{k-1}]}\biggr\}
&=& \frac{c_1+o(1)}{nb}\quad\mbox{as }n\to\infty\mbox{ uniformly for all }j,
\end{eqnarray*}
so
\begin{eqnarray*}
\E\biggl\{\frac{Z_{[nt_k]}^2-Z_{[nt_{k-1}]}^2}{nb}
\ \bigg|\ Z_{[nt_{k-1}]}\biggr\} &\to& \frac{c_1(t_k-t_{k-1})}{b}
\quad\mbox{as }n\to\infty,
\end{eqnarray*}
and hence
\begin{eqnarray*}
\lefteqn{\E Z^{(n)}(t_1)^{2i_1}\ldots
Z^{(n)}(t_{k-1})^{2i_{k-1}}
Z^{(n)}(t_k)^2}\\
&&\hspace{10mm} =\ \E Z^{(n)}(t_1)^{2i_1}\ldots
Z^{(n)}(t_{k-1})^{2i_{k-1}+2}\\
&&\hspace{30mm}+ \frac{c_1(t_k-t_{k-1})}{b}\E Z^{(n)}(t_1)^{2i_1}\ldots
Z^{(n)}(t_{k-1})^{2i_{k-1}}+o(1).
\end{eqnarray*}
In the case $i_k=2$ we get, as in the proof of Theorem \ref{thm:gamma},
\begin{eqnarray*}
\lefteqn{\E\biggl\{\frac{Z_{j+1}^4-Z_j^4}{(nb)^2}
\ \bigg|\ Z_{[nt_{k-1}]}\biggr\}}\\
&=& (c_2+o(1))\E\biggl\{\frac{Z_j^2}{(nb)^2}
\ \bigg|\ Z_{[nt_{k-1}]}\biggr\}\\
&=& (c_2+o(1))\E\biggl\{\frac{Z_j^2-Z_{[nt_{k-1}]}^2}{(nb)^2}
\ \bigg|\ Z_{[nt_{k-1}]}\biggr\}
+(c_2+o(1))\frac{Z_{[nt_{k-1}]}^2}{(nb)^2}\\
&=& \frac{c_2c_1+o(1)}{(nb)^2}(j-[nt_{k-1}])
+(c_2+o(1))\frac{Z_{[nt_{k-1}]}^2}{(nb)^2},
\end{eqnarray*}
so, as $n\to\infty$,
\begin{eqnarray}\label{conv.Z.4}
\lefteqn{\E\biggl\{\frac{Z_{[nt_k]}^4
-Z_{[nt_{k-1}]}^4}{(nb)^2}
\ \bigg|\ Z_{[nt_{k-1}]}\biggr\}}\nonumber\\
&&=\ (c_2c_1+o(1))\frac{(t_k-t_{k-1})^2}{2b^2}
+(c_2+o(1))\frac{t_k-t_{k-1}}{b}\frac{Z_{[nt_{k-1}]}^2}{nb},
\end{eqnarray}
and hence
\begin{eqnarray*}
\lefteqn{\lim_{n\to\infty}
\E Z^{(n)}(t_1)^{2i_1}\ldots Z^{(n)}(t_{k-1})^{2i_{k-1}} Z^{(n)}(t_k)^4}\\
&=& \lim_{n\to\infty}\E Z^{(n)}(t_1)^{2i_1}\ldots Z^{(n)}(t_{k-1})^{2i_{k-1}+4}\\
&&\hspace{10mm}+ c_2\frac{t_k-t_{k-1}}{b}\lim_{n\to\infty}
\E Z^{(n)}(t_1)^{2i_1}\ldots Z^{(n)}(t_{k-1})^{2i_{k-1}+2}\\
&&\hspace{20mm}+c_2c_1\frac{(t_k-t_{k-1})^2}{2b^2}
\lim_{n\to\infty}\E Z^{(n)}(t_1)^{2i_1}\ldots
Z^{(n)}(t_{k-1})^{2i_{k-1}}.
\end{eqnarray*}
Similar relations hold for all $i_k\in\N$, with clear pattern;
for instance, for $i_k=3$,
\begin{eqnarray*}
\lefteqn{\lim_{n\to\infty}
\E Z^{(n)}(t_1)^{2i_1}\ldots Z^{(n)}(t_{k-1})^{2i_{k-1}}
Z^{(n)}(t_k)^6}\\
&=& \lim_{n\to\infty}\E Z^{(n)}(t_1)^{2i_1}\ldots
Z^{(n)}(t_{k-1})^{2i_{k-1}+6}\\
&&\hspace{10mm}+ c_3\frac{t_k-t_{k-1}}{b}\lim_{n\to\infty}
\E Z^{(n)}(t_1)^{2i_1}\ldots Z^{(n)}(t_{k-1})^{2i_{k-1}+4}\\
&&\hspace{20mm}+c_3c_2\frac{(t_k-t_{k-1})^2}{2b^2}
\lim_{n\to\infty}\E Z^{(n)}(t_1)^{2i_1}\ldots
Z^{(n)}(t_{k-1})^{2i_{k-1}+2}\\
&&\hspace{30mm}+c_3c_2c_1\frac{(t_k-t_{k-1})^3}{3!b^3}
\lim_{n\to\infty}\E Z^{(n)}(t_1)^{2i_1}\ldots
Z^{(n)}(t_{k-1})^{2i_{k-1}}.
\end{eqnarray*}

Now let us show how to approximate the mixed even moments \eqref{mixed.B}
via slotting the Bessel process $Bes(t)$, for any $\nu>-b/2$.
Consider a Markov chain $Bes_j$ defined as a skeleton of $Bes(t)$,
$Bes_j:=Bes(j)$. On the one hand, by the self-similarity
and continuity of a Bessel process,
\begin{eqnarray*}
\frac{(Bes([nt_1]),\ldots,Bes([nt_k])}{\sqrt n} &=_{st}&
(Bes([nt_1]/n),\ldots,Bes([nt_k]/n))\\
&\Rightarrow& (Bes(t_1),\ldots,Bes(t_k))\quad\mbox{as }n\to\infty,
\end{eqnarray*}
which implies convergence of mixed even moments
$$
\E \biggl(\frac{Bes_{[nt_1]}}{\sqrt n}\biggr)^{2i_1}\ldots
\biggl(\frac{Bes_{[nt_k]}}{\sqrt n}\biggr)^{2i_k}\ \to\
\E Bes^{2i_1}(t_1)\ldots Bes^{2i_k}(t_k)\quad\mbox{as }n\to\infty.
$$
On the other hand, the mean drift of the chain $Bes_n$ is of order $\mu/bx$
and the second moment of jumps converges to $1$ as $x\to\infty$,
see \eqref{eq:bessel.1}; in the null recurrent case
\eqref{eq:bessel.1} is applicable because we assume reflecting
boundary condition for $X(t)$. In addition, \eqref{bessel.higher} holds.
Therefore, a relation similar to \eqref{incr.n.2i} follows,
for all $i\ge 1$,
\begin{eqnarray*}
\E\{Bes^{2i}_{n+1}-Bes^{2i}_n\}
&=& i[2\mu+(2i-1)b+o(1)]\E Bes_n^{2i-2}.
\end{eqnarray*}
So, all the calculations carried out for evaluation
of mixed even moments of $Z^{(n)}(t)$ are applicable to that of $Bes_n$.
Therefore, the mixed even moments \eqref{mixed.Z} of $\{Z^{(n)}(t)\}$
converge to the corresponding mixed even moments \eqref{mixed.B}
of the Bessel process $Bes(t)$,
hence the weak convergence of finite dimensional distributions
of $\{Z^{(n)}(t)\}$ follows by the method of moments.

Now it only remains to prove tightness. For that it is enough to show that
there exists a $c<\infty$ such that, for all $0\le t_1<t_2<t_3\le 1$
\begin{eqnarray}\label{for.tight}
\E(Z^{(n)}(t_2)^2-Z^{(n)}(t_1)^2)^2
(Z^{(n)}(t_3)^2-Z^{(n)}(t_2)^2)^2
&\le& c(t_3-t_1)^2,
\end{eqnarray}
see, e.g. Billingsley\index{Billingsley} \cite[Theorem 15.6]{Billingsley}.
Let us bound this expectation.
Since we can always modify the chain $\{Z_n\}$ below any specific level,
there is no loss of generality if we assume that, for all $x$,
\begin{eqnarray}\label{posit.2.4}
\E\{Z_1^2-Z_0^2\mid Z_0=x\} &>& 0,\\
\label{posit.2.4.4}
\E\{Z_1^4-Z_0^4\mid Z_0=x\} &>& 0 .
\end{eqnarray}
Conditioning on $Z^{(n)}(t_1)$ and $Z^{(n)}(t_2)$
yields the following expression for the left hand side of \eqref{for.tight}
\begin{eqnarray*}
\lefteqn{\E(Z^{(n)}(t_2)^2-Z^{(n)}(t_1)^2)^2
\E\{(Z^{(n)}(t_3)^2-Z^{(n)}(t_2)^2)^2\mid Z^{(n)}(t_1),\ Z^{(n)}(t_2)\}}\\
&& =\ \E(Z^{(n)}(t_2)^2-Z^{(n)}(t_1)^2)^2
\E\{(Z^{(n)}(t_3)^2-Z^{(n)}(t_2)^2)^2\mid Z^{(n)}(t_2)\}.
\end{eqnarray*}
In its turn, the conditional expectation may be bounded as follows:
\begin{eqnarray*}
\lefteqn{\E\{(Z^{(n)}(t_3)^2-Z^{(n)}(t_2)^2)^2\mid Z^{(n)}(t_2)\}}\\
&& =\ \E\{Z^{(n)}(t_3)^4-Z^{(n)}(t_2)^4\mid Z^{(n)}(t_2)\}\\
&&\hspace{45mm}-2Z^{(n)}(t_2)^2
\E\{Z^{(n)}(t_3)^2-Z^{(n)}(t_2)^2\mid Z^{(n)}(t_2)\}\\
&&\le\ \E\{Z^{(n)}(t_3)^4-Z^{(n)}(t_2)^4\mid Z^{(n)}(t_2)\}
\end{eqnarray*}
owing to \eqref{posit.2.4}. Calculations leading to \eqref{conv.Z.4}
also imply that, for some $c_1<\infty$,
\begin{eqnarray*}
\E\{Z^{(n)}(t_3)^4-Z^{(n)}(t_2)^4\mid Z^{(n)}(t_2)\}
&\le& c_1(t_3-t_2)Z^{(n)}(t_2)^2.
\end{eqnarray*}
Therefore,
\begin{eqnarray}\label{for.tight.1}
\E\{(Z^{(n)}(t_3)^2-Z^{(n)}(t_2)^2)^2\mid Z^{(n)}(t_2)\}
&\le& c_1(t_3-t_2)Z^{(n)}(t_2)^2.
\end{eqnarray}
Further,
\begin{eqnarray*}
\lefteqn{\E(Z^{(n)}(t_2)^2-Z^{(n)}(t_1)^2)^2 Z^{(n)}(t_2)^2}\\
&=& \E(Z^{(n)}(t_2)^6-Z^{(n)}(t_1)^6)
-\E(Z^{(n)}(t_2)^4-Z^{(n)}(t_1)^4) Z^{(n)}(t_1)^2\\
&& -\E((Z^{(n)}(t_2)^2-Z^{(n)}(t_1)^2)^2 Z^{(n)}(t_1)^2
-\E(Z^{(n)}(t_2)^2-Z^{(n)}(t_1)^2) Z^{(n)}(t_1)^4\\
&\le& \E(Z^{(n)}(t_2)^6-Z^{(n)}(t_1)^6)
-\E(Z^{(n)}(t_2)^4-Z^{(n)}(t_1)^4) Z^{(n)}(t_1)^2\\
&& -\E(Z^{(n)}(t_2)^2-Z^{(n)}(t_1)^2) Z^{(n)}(t_1)^4\\
&\le& \E(Z^{(n)}(t_2)^6-Z^{(n)}(t_1)^6),
\end{eqnarray*}
because the second and third terms on the right hand side of the first inequality
are negative due to the assumptions \eqref{posit.2.4.4} and \eqref{posit.2.4}. 
Hence,
\begin{eqnarray*}
\E(Z^{(n)}(t_2)^2-Z^{(n)}(t_1)^2)^2 Z^{(n)}(t_2)^2
&\le& c_2(t_2-t_1),
\end{eqnarray*}
which together with \eqref{for.tight.1} implies \eqref{for.tight}.
Hence diffusion approximation follows under the conditions of
Theorem \ref{thm:gamma}.

Under the conditions of Theorem \ref{thm:pre-st.null}
the proof is the same but starts with time-dependent truncation of jumps.
\qed\end{proof}

\section{Integral renewal theorem for transient chain with Gamma limit}
\sectionmark{Integral renewal theorem}
\label{sec:renewal}

%If a Markov chain $X$ is transient then it visits any bounded set finitely many times only.
The next result determines the asymptotic behaviour
of the renewal functions $H_y(x)$ and $H(x)$ in the case of
convergence to a $\Gamma$-distribution in the transient case.
\index{Renewal theorem!for $\Gamma$ limit}
The proof is based on preliminary upper bound delivered in
Theorem \ref{thm:Hy.above}.
\index{Markov chain!renewal function!integral asymptotics}
\index{Renewal function!Markov chain!integral asymptotics}

\begin{theorem}\label{thm:renewal.gamma}
Under the conditions of Theorem \ref{thm:gamma},
for any initial distribution of the chain $\{X_n\}$,
\begin{eqnarray*}
\sum_{n=0}^{[Bx^2]} \P\{X_n\in(\widehat x,x]\}
&=& (I(B)+o(1)) x^2\ \mbox{ as }x\to\infty\mbox{ uniformly for all }B\ge 0,
\end{eqnarray*}
where
\begin{eqnarray*}
I(B) &:=& \int_0^B\Gamma(1/z)dz
\ =\ B\Gamma(1/B)+\int_{1/B}^\infty \frac{1}{z}\gamma(z)dz,
\quad I(\infty)=\frac{1}{2\mu-b},
\end{eqnarray*}
$\widehat x$ is defined in Theorem \ref{thm:gamma},
and $\Gamma(t)$ and $\gamma(t)$ denote the cumulative distribution
function and the probability density function respectively of the
$\Gamma$-distribution with mean $2\mu+b$ and variance $(2\mu+b)2b$.
In particular,
\begin{eqnarray}\label{asy.for.H.infty}
H(\widehat x,x] &\sim& \frac{1}{2\mu-b}x^2\ \mbox{ as }x\to\infty.
\end{eqnarray}
\end{theorem}

\begin{proof}
By Theorem \ref{thm:gamma}, for every fixed $B>0$,
\begin{eqnarray*}
\sum_{n=0}^{[Bx^2]}\P\{X_n\in(\widehat x,x]\}
&=& \sum_{n=0}^{[Bx^2]}(\Gamma(x^2/n)+o(1))\\
&=& \sum_{n=0}^{[Bx^2]}\Gamma(x^2/n)+o(x^2)
\quad\mbox{as }x\to\infty.
\end{eqnarray*}
Due to
\begin{eqnarray*}
\sum_{n=0}^{[Bx^2]}\Gamma(x^2/n) &\sim&
x^2\int_0^B \Gamma(1/z)dz\ \mbox{ as }x\to\infty,
\end{eqnarray*}
we conclude that, for any fixed $B>0$,
\begin{eqnarray}\label{asy.B.fixed}
\sum_{n=0}^{[Bx^2]} \P\{X_n\in(\widehat x,x]\}
&\sim& I(B) x^2\ \mbox{ as }x\to\infty.
\end{eqnarray}
Since the sum is increasing in $B$, 
it remains to prove that \eqref{asy.for.H.infty} holds.
Firstly, since
\begin{eqnarray*}
\int_0^B\Gamma(1/z)dz &\to& \frac{1}{2\mu-b}\ \mbox{ as }B\to\infty,
\end{eqnarray*}
we conclude a lower bound
\begin{eqnarray}\label{bound.Hy.lower}
\liminf_{x\to\infty}\frac{H(\widehat x,x]}{x^2} &\ge& \frac{1}{2\mu-b}.
\end{eqnarray}
Secondly, for an arbitrary $y$, let us now prove the matching upper bound,
\begin{eqnarray}\label{bound.Hy.upper.1}
\limsup_{x\to\infty}\frac{H_y(\widehat x,x]}{x^2} &\le& \frac{1}{2\mu-b}.
\end{eqnarray}
For any $A>1$, $T(Ax)$ is the first up-crossing time of the level $Ax$.
By the Markov property,
\begin{eqnarray}\label{estimate.for.Hy.2.1x.pre}
\lefteqn{H_y(\widehat x,x]}\nonumber\\
&\le&
\E_y\sum_{n=0}^{T(Ax)-1} \I\{X_n\in(\widehat x,x]\}
+\P\{X_n\le x\mbox{ for some }n\mid X_0>Ax\}
\sup_{z\le x}H_z(\widehat x,x]\nonumber\\
&\le& \E_y\sum_{n=0}^{T(Ax)-1} \I\{X_n\in(\widehat x,x]\}
+\bigl(e^{\delta(R(x)-R(Ax))}+o(1)\bigr)
\sup_{z\le x}H_z(\widehat x,x]
\end{eqnarray}
as $x\to\infty$ uniformly for all $A>1$,
due to \eqref{y.back.x} where $R(x)$ is determined by
$r(x)=\gamma/x$ with $\gamma\in(0,2\mu/b)$, hence
\begin{eqnarray*}
e^{\delta(R(x)-R(Ax))} &=& 1/A^{\delta\gamma}.
\end{eqnarray*}
Thus, applying the upper bound proven in Theorem \ref{thm:Hy.above}
on the right hand side of \eqref{estimate.for.Hy.2.1x.pre} we deduce that,
for some $c<\infty$,
\begin{eqnarray}\label{estimate.for.Hy.2.1x}
H_y(\widehat x,x] &\le& \E_y\sum_{n=0}^{T(Ax)-1}
\I\{X_n\in(\widehat x,x]\}
+ \bigl(c/A^{\delta\gamma}+o(1)\bigr)x^2
\end{eqnarray}
as $x\to\infty$ uniformly for all $A>1$.
The expectation of the sum on the right hand side of
\eqref{estimate.for.Hy.2.1x} may be estimated as follows: for $C>1$,
\begin{eqnarray*}
\E_y\sum_{n=0}^{T(Ax)-1} \I\{X_n\in(\widehat x,x]\}
&\le& \E_y\sum_{n=0}^{[CA^2x^2]} \I\{X_n\in(\widehat x,x]\}\\
&&+\E_y\Bigl\{\sum_{n=0}^{T(Ax)-1} \I\{X_n>\widehat x\}; T(Ax)>CA^2x^2\Bigr\}.
\end{eqnarray*}
The second term on the right hand side is not greater than
\begin{eqnarray*}
\lefteqn{\E_y\Bigl\{L(\widehat x,T(Ax));\ X_n\le \widehat x
\mbox{ for some }n\ge A^2x^2\Bigr\}}\\
&&\hspace{2mm}+\E_y\Bigl\{L(\widehat x,T(Ax));\
X_n>\widehat x\mbox{ for all }n\in[A^2x^2,T(Ax)-1],\
T(Ax)>CA^2x^2\Bigr\}\\
&& \le\ \E_y\Bigl\{L(\widehat x,T(Ax));\ X_n\le \widehat x
\mbox{ for some }n\ge A^2x^2\Bigr\}\\
&&\hspace{20mm}+\E_y\Bigl\{L(\widehat x,T(Ax));\
L(\widehat x,T(Ax))>(C-1)A^2x^2\Bigr\}.
\end{eqnarray*}
Since conditions of Theorem \ref{l:uniform} are met with
$v(x)=\mu/2x$, the family of random variables
$$
\frac{L(\widehat x,T(Ax))}{(Ax)^2}
$$
is uniformly integrable, so, for any fixed $A>1$,
\begin{eqnarray*}
\sup_{x>\widehat x,\ y}\frac{1}{x^2}\E_y\Bigl\{L(\widehat x,T(Ax));\
L(\widehat x,T(Ax))>(C-1)A^2x^2\Bigr\}
&\le& \psi(C),
\end{eqnarray*}
where $\psi(C)\to 0$ as $C\to\infty$.
Since $X_n\to\infty$ with probability $1$,
$$
\P\{X_n\le \widehat x\mbox{ for some }n\ge A^2x^2\}\to 0
\quad\mbox{as }x\to\infty.
$$
Therefore, again by the uniform integrability,
\begin{eqnarray*}
\frac{1}{x^2}\E_y\Bigl\{L(\widehat x,T(Ax));\
X_n\le \widehat x\mbox{ for some }n\ge A^2x^2\Bigr\}
&\to& 0\quad\mbox{as }x\to\infty.
\end{eqnarray*}
Altogether yields
\begin{eqnarray*}
\limsup_{x\to\infty}\sup_{y}
\frac{1}{x^2}\E_y\Bigl\{\sum_{n=0}^{T(Ax)-1} \I\{X_n>\widehat x\};
T(Ax)>CA^2x^2\Bigr\}
&\le& \psi(C),
\end{eqnarray*}
hence, uniformly for all $y$,
\begin{eqnarray*}
\limsup_{x\to\infty}\frac{1}{x^2}\E_y\sum_{n=0}^{T(Ax)-1}
\I\{X_n\in(\widehat x,x]\}
&\le& \E_y\sum_{n=0}^{[CA^2x^2]} \I\{X_n\in(\widehat x,x]\}+\psi(C),
\end{eqnarray*}
which being substituted into \eqref{estimate.for.Hy.2.1x} gives
\begin{eqnarray*}
\limsup_{x\to\infty}\frac{H_y(\widehat x,x]}{x^2} &\le&
\limsup_{x\to\infty}\frac{1}{x^2}\E_y\sum_{n=0}^{[CA^2x^2]}
\I\{X_n\in(\widehat x,x]\}
+\psi(C)+c/A^{\delta\gamma}.
\end{eqnarray*}
As has already been shown,
\begin{eqnarray*}
\frac{1}{x^2}\sum_{n=0}^{[CA^2x^2]}\P_y\{X_n\in(\widehat x,x]\}
&\to& I(CA^2)\quad\mbox{as }x\to\infty,
\end{eqnarray*}
which implies the following upper bound, for each fixed $A$, $C>1$,
\begin{eqnarray*}
\limsup_{x\to\infty}\frac{H_y(\widehat x,x]}{x^2} &\le&
I(CA^2)+\psi(C)+c/A^{\delta\gamma}.
\end{eqnarray*}
Letting now first $C\to\infty$ and then $A\to\infty$, we get
the required upper bound \eqref{bound.Hy.upper.1}.
The lower \eqref{bound.Hy.lower} and upper \eqref{bound.Hy.upper.1}
bounds yield the equivalence, for every fixed $y$,
$$
H_y(\widehat x,x]\sim \frac{1}{2\mu-b}x^2\ \mbox{ as }x\to\infty.
$$
Together with the uniform in $y$ bound of Theorem
\ref{thm:Hy.above} this completes the proof of 
\eqref{asy.for.H.infty} and hence the result follows.
\qed\end{proof}

The next result will be used later to find tail asymptotics
for the stationary measure when $\{X_n\}$ is recurrent.

\begin{theorem}\label{thm:renewal.q}
Let the conditions of Theorem \ref{thm:gamma} hold.
Then, for $q(z)\ge 0$ and any distribution of $X_0$,
\begin{eqnarray*}
\sum_{n=0}^{[Bx^2]} \E\bigl\{e^{-\sum_{k=0}^{n-1}q(X_k)};\ X_n\in(\widehat x,x]\bigr\}
&=& (I(B)+o(1))x^2\ \E e^{-\sum_{k=0}^\infty q(X_k)}
\end{eqnarray*}
as $x\to\infty$ uniformly for all $B\in[0,\infty]$, 
where $I(B)$ is defined in Theorem \ref{thm:renewal.gamma}.
\end{theorem}

\begin{proof}
We may apply Lemma \ref{thm:renewal.2} because its condition
\eqref{conv.to.r2.square} is guaranteed by Theorem \ref{thm:Hy.above},
while the condition \eqref{conv.to.r2.uni} by Theorem \ref{thm:renewal.gamma}.
\qed\end{proof}

\section{Local renewal theorem for transient chain on $\Z$ with Gamma limit}
\sectionmark{Local renewal theorem}
\label{sec:loc.renewal}

In this section we discuss a local version of the renewal
theorem in the case of convergence to a $\Gamma$-distribution. 
In this section we do this for a lattice Markov chain.
Without loss of generality, let the minimal lattice where $\{X_n\}$
is living on be $\Z$. It is unclear whether the local renewal theorem
would be valid if we only assumed a regular asymptotic behaviour
of moments of jumps. It is very likely that it can be only proven
for an asymptotically homogeneous in space Markov chain
as it is defined in Definition \ref{def:asymp.hom}, that is,
if we assume weak convergence of jumps $\xi(x)$
to some random variable $\xi$ on $\Z$, that is,
\begin{equation}\label{weak.lim}
\xi(x)\Rightarrow\xi\quad\mbox{as } x\to\infty.
\end{equation}

\index{Markov chain!renewal function!local asymptotics}
\index{Renewal function!Markov chain!local asymptotics}
\index{Local renewal theorem for!$\Gamma$ limit}
\begin{theorem}\label{thm:srt}
Let there exist $b>0$ and $\mu>b/2$ such that
\begin{eqnarray}\label{1.2.G.key}
m_1(x)\sim \mu/x\ &\mbox{and}&\ m_2(x)\to b
\quad\mbox{ as }x\to\infty,
\end{eqnarray}
and
$$
\limsup_{n\to\infty}X_n=\infty\quad\mbox{with probability }1.
$$
Furthermore we assume the convergence \eqref{weak.lim}.
Let $\Z$ be the minimal lattice for $\xi$, and let the limit $\xi$ satisfy
\begin{equation}\label{lim.mom}
\E\xi=0,\quad \E\xi^2=b.
\end{equation}
In addition, let the jumps $\xi(x)$ be bounded below and above
by $J$ uniformly for all $x\in\Z^+$, that is,
\begin{equation}\label{cond.J}
|\xi(x)|\ \le\ J\ \mbox{ for all }x\in\Z^+.
\end{equation}
Then
\begin{eqnarray}\label{ren.loc.h}
h(x):=H\{x\} &\sim& \frac{2}{2\mu-b}x\quad\mbox{as }x\to\infty.
\end{eqnarray}
Moreover,
\begin{eqnarray}\label{ren.loc.h.1}
\P\Bigl\{\sum_{n=0}^\infty\I\{X_n=x\}>N\Bigr\}
&=& c_1(x)\Bigl(1-\frac{c_2(x)}{x}\Bigr)^N,
\end{eqnarray}
where $c_1(x)\to 1$ and $c_2(x)\to \mu-b/2>0$ as $x\to\infty$,
so the family of random variables
\begin{equation}\label{ren.loc.h.2}
\frac{1}{x}\sum_{n=0}^\infty \I\{X_n=x\},\quad x\in\{1,2,3,\ldots\},
\end{equation}
is uniformly integrable.
\end{theorem}

More general results are derived in Chapter \ref{ch:asy.renewal},
via different technique based on the martingale approach.

\begin{proof}
Consider a stopping time
$$
\tau(x):=\inf\{n\ge1:X_n\le x\}.
$$
Since $\{X_n\}$ is transient, $\P_{x+k}\{\tau(x)=\infty\}>0$ for all $k\ge 1$.
First let us understand the asymptotic behaviour of this probability as $x$ grows.
To this end, let us fix a $\delta\in(0,2\mu/b-1)$
and define two decreasing functions
$$
U_\pm(x)\ :=\ \frac{1}{(2\mu/b-1\pm\delta)
x^{2\mu/b-1\pm\delta}},\quad x\ge 1.
$$
By the mean value theorem, for all $x$ and $j\in\Z$
there is a $\theta\in(0,1)$ such that
\begin{eqnarray*}
U_\pm(x+j)-U_\pm(x) &=& -\frac{j}{(x+\theta j)^{2\mu/b\pm\delta}}
\ \sim\ -\frac{j}{x^{2\mu/b\pm\delta}}\quad\mbox{as }x\to\infty,
\end{eqnarray*}
which implies
$$
U_\pm(x+j)-U_\pm(x)\sim -j(2\mu/b-1\pm\delta)\frac{U_\pm(x)}{x}.
$$
Then, since $\xi(x)$ are bounded below, we get for all fixed $k\ge 1$ that
\begin{eqnarray}\label{srt.1}
\nonumber
&&\E_{x+k}\left\{U_\pm(X_{\tau(x)})-U_\pm(x+k);\ \tau(x)<\infty\right\}\\
&&\hspace{5mm}\sim\ (2\mu/b-1\pm\delta)\frac{U_\pm(x+k)}{x}
\E_{x+k}\{x+k-X_{\tau(x)};\ \tau(x)<\infty\}.
\end{eqnarray}

Let us compute the drift of $U_\pm(X_n)$.
Since the jumps are bounded, by Taylor's expansion,
\begin{eqnarray*}
\lefteqn{\E(U_\pm(x+\xi(x))-U_\pm(x))}\\
&&\hspace{10mm}=\ U_\pm'(x)m_1(x)+\frac12 m_2(x)U_\pm''(x)m_2(x)
+O(U_\pm'''(x))\\
&&\hspace{10mm}=\ -x^{-2\mu/b\mp\delta}m_1(x)
+(\mu/b\pm\delta/2)x^{-2\mu/b-1\mp\delta}m_2(x)
+O(x^{-2\mu/b-2\mp\delta})\\
&&\hspace{10mm}=\ \pm(\delta b/2+o(1))x^{-2\mu/b-1\mp\delta}
\quad\mbox{as }x\to\infty.
\end{eqnarray*}
Therefore, the sequence $U_-(X_{n\wedge\tau(x)})$ is a supermartingale
for all sufficiently large $x$. Then, by the optional stopping theorem,
\begin{eqnarray*}
\E_{x+k}\{U_-(X_{\tau(x)});\ \tau(x)<\infty\} &\le& U_-(x+k).
\end{eqnarray*}
This is equivalent to
$$
\E_{x+k}\left\{U_-(X_{\tau(x)})-U_-(x+k);\ \tau(x)<\infty\right\}
\ \le\ U_-(x+k)\P_{x+k}\{\tau(x)=\infty\}.
$$
Using now \eqref{srt.1}, we get, for all sufficiently large $x$,
\begin{equation}\label{srt.2}
\P_{x+k}\{\tau(x)=\infty\}\ \ge\ \frac{2\mu/b-1-2\delta}{x}
\E_{x+k}\{x+k-X_{\tau(x)};\ \tau(x)<\infty\}.
\end{equation}
Since $\{U_+(X_{n\wedge\tau(x)})\}$ is a submartingale
for all sufficiently large $x$,
\begin{eqnarray*}
\E_{x+k}\{U_+(X_{\tau(x)});\ \tau(x)<\infty\} &\ge& U_+(x+k).
\end{eqnarray*}
This implies that, for all sufficiently large $x$,
\begin{equation*}
\P_{x+k}\{\tau(x)=\infty\}\ \le\
\frac{2\mu/b-1+2\delta}{x}\E_{x+k}\{x+k-X_{\tau(x)};\ \tau(x)<\infty\}.
\end{equation*}
Combining this lower bound with \eqref{srt.2} and due to
the arbitrary choice of $\delta>0$, we conclude that, as $x\to\infty$,
\begin{equation}\label{srt.3}
\P_{x+k}\{\tau(x)=\infty\}\ =\
\frac{2\mu/b-1+o(1)}{x}\E_{x+k}\{x+k-X_{\tau(x)};\ \tau(x)<\infty\}.
\end{equation}

Now let us determine the limit of
$\E_{x+k}\{x+k-X_{\tau(x)};\ \tau(x)<\infty\}$.
Let $\xi_k$, $k\ge1$, be independent copies of the random variable $\xi$.
Define $S_0=0$, $S_k:=S_{k-1}+\xi_k$ for $k\ge 1$, and
$$
\theta_j:=\min\{k\ge 1: S_k< -j\},\quad \psi_j=-S_{\theta_j}.
$$
Assumption \eqref{weak.lim} implies that, for every $n\ge1$,
$(X_1-X_0,X_2-X_0,\ldots X_n-X_0)$ converges weakly, as $X_0\to\infty$, to
$(S_1,S_2,\ldots,S_n)$. In particular,
$$
\E_{x+k}\{x+k-X_{\tau(x)};\ \tau(x)\le n\}\ \to\
\E\{\psi_{k-1};\ \theta_k\le n\},\quad n\ge 1.
$$
Noting that both $x+k-X_{\tau(x)}$ and $\psi_{k-1}$ are bounded, we conclude that
$$
\lim_{x\to\infty}\E_{x+k}\{x+k-X_{\tau(x)};\ \tau(x)<\infty\}
\ =\ \E\psi_{k-1}.
$$
Plugging this into \eqref{srt.3}, we get for all $k\ge 1$
\begin{equation}\label{srt.4}
\P_{x+k}\{\tau(x)=\infty\}\ \sim\ \frac{2\mu-b}{bx}\E\psi_{k-1}
\quad\mbox{as }x\to\infty.
\end{equation}

We now use these asymptotics to study asymptotic behaviour
of the renewal mass function $h(x)$.
Choose any $j_0\in[1,J]$ such that
\begin{eqnarray}\label{choice.j0}
\P\{\xi=j_0\} &>& 0
\end{eqnarray}
and consider the following upcrossing stopping times:
\begin{eqnarray*}
\sigma(x) &:=& \min\{n\ge 1: X_{n-1}\le x,\ X_n>x\},\\
\gamma(x) &:=& \min\{n\ge 1:\ X_{n-1}\le x,\ X_n=x+j_0\},
\end{eqnarray*}
and let us evaluate the probabilities
$$
p_i(x)\ :=\ \P_{x+i}\{\gamma(x)=\infty\}
$$
for $i=1$, \ldots, $J$, and large values of $x$. For all $i\le J$,
the Markov property leads to the equation
\begin{eqnarray*}
\lefteqn{\P_{x+i}\{\gamma(x)=\infty\}}\\
&&\ =\ \P_{x+i}\{\sigma(x)=\infty\}
+\sum_{j=1,\ j\not=j_0}^J \P_{x+i}\{\sigma(x)<\infty,X_{\sigma(x)}=x+j\}
\P_{x+j}\{\gamma(x)=\infty\}\\
&&\ =\ \P_{x+i}\{\tau(x)=\infty\}
+\sum_{j=1,\ j\not=j_0}^J \P_{x+i}\{\tau(x)<\infty,X_{\sigma(x)}=x+j\}
\P_{x+j}\{\gamma(x)=\infty\},
\end{eqnarray*}
because the transience of $\{X_n\}$ implies
$$
\P_{x+i}\{\tau(x)<\infty,\sigma(x)=\infty\}\ =\ 0.
$$
Hence, the $(J-1)$-dimensional vector
$$
p(x):=(p_1(x),\ldots,p_{j_0-1}(x),p_{j_0+1}(x),\ldots,p_J(x))^\top
$$
satisfies the equation
\begin{eqnarray*}
p(x) &=& q(x)+A(x)p(x),
\end{eqnarray*}
where
$$
q_i(x)\ =\ \P_{x+i}\{\sigma(x)=\infty\}\ =\ \P_{x+i}\{\tau(x)=\infty\}
$$
and $A(x)$ is a matrix with entries $A_{ij}(x)$,
$i$, $j\in\{1,\ldots,j_0-1,j_0+1,\ldots,J\}$, where
$$
A_{ij}(x)\ =\ \P_{x+i}\{\sigma(x)<\infty,X_{\sigma(x)}=x+j\}.
$$
Therefore, provided the matrix $I-A(x)$ is invertible,
\begin{eqnarray}\label{p.solution.A}
p(x) &=& (I-A(x))^{-1}q(x).
\end{eqnarray}

In view of $\P\{\xi=j_0\}>0$---see \eqref{choice.j0}---and because $\Z$ is
the minimal lattice for $\xi$, it follows that there exists an $\varepsilon>0$
such that
\begin{equation*}
A_{ij_0}\ :=\ \P\{S_\sigma=j_0\mid S_0=i\}\ >\ 2\varepsilon\quad\mbox{for all }i\le J,
\end{equation*}
where $\sigma:=\inf\{n\ge1:\ S_{n-1}\le 0,\ S_n>0\}$ is finite a.s.
By the condition \eqref{weak.lim},
\begin{eqnarray*}
\P\{X_{\sigma(x)}=x+j_0,\sigma(x)<\infty\mid X_0=x+i\} 
&\to& \P\{S_\sigma=j_0\mid S_0=i\}
\quad\mbox{as }x\to\infty,
\end{eqnarray*}
hence there is an $x_0$ such that, for all $x\ge x_0$,
\begin{eqnarray*}
A_{ij_0}(x)=\P\{X_{\sigma(x)}=x+j_0,\sigma(x)<\infty\mid X_0=x+i\} 
&>& \varepsilon\quad\mbox{for all }i\le J.
\end{eqnarray*}
Then each row of the matrix $A(x)$ sums to a number less than 
$1-\varepsilon$, hence the matrix $I-A(x)$ is invertible and
$$
(I-A(x))^{-1}\ \to\ (I-A)^{-1}\quad\mbox{as }x\to\infty,
$$
where
$$
A_{ij}\ =\ \P\{S_\sigma=j\mid S_0=i\},
$$
and it follows from \eqref{p.solution.A} and \eqref{srt.4} that,
as $x\to\infty$,
\begin{eqnarray*}
p(x) &\sim& \frac{2\mu-b}{bx}(I-A)^{-1}
(\E\psi_1,\ldots,\E\psi_{j_0-1},\E\psi_{j_0+1},\ldots,\E\psi_J)^\top\\
&=:& \frac{1}{x}(c_1,\ldots,c_{j_0-1},c_{j_0+1},\ldots,c_J)^\top.
\end{eqnarray*}
Thus,
\begin{eqnarray}\label{asymp.pj0}
p_{j_0}(x)\ =\ \P_{x+j_0}\{\gamma(x)=\infty\}
&=& \P_{x+j_0}\{\tau(x)=\infty\}+\sum_{j=1,\ j\not= j_0}^J A_{j_0j}(x)p_j(x)
\nonumber\\
&\sim& \frac{1}{x}\Bigl(\frac{2\mu-b}{b}\E\psi_{j_0-1}
+\sum_{j=1,\ j\not=j_0}^J A_{j_0j}c_j\Bigr)
\quad\mbox{as }x\to\infty.\nonumber\\[-2mm]
\end{eqnarray}

Denote by $N(x)$ the number of visits of $\{X_n\}$ to the state $x$. We have
\begin{eqnarray}\label{h.decomp}
h(x) &=& \E\sum_{n=1}^{\gamma(x)-1}\I\{X_n=x\}+
\E_{x+j_0}N(x)\P\{\gamma(x)<\infty\}.
\end{eqnarray}
Since the random variable
\begin{eqnarray*}
\sum_{n=1}^{\gamma(x)-1}\I\{X_n=x\}
\end{eqnarray*}
is stochastically dominated by a geometric random variable
with parameter $1-\P\{\xi(x)=j_0\}$ and
$\P\{\xi(x)=j_0\}\to\P\{\xi=j_0\}>0$ as $x\to\infty$,
there exists a sufficiently large $x_1\in\Z^+$ such that
the first term on the right hand side of \eqref{h.decomp}
is bounded above for all $x\ge x_1$,
\begin{eqnarray}\label{h.decomp.1}
\sup_{x\ge x_1} \E\sum_{n=1}^{\gamma(x)-1}\I\{X_n=x\} &<& \infty.
\end{eqnarray}
In addition, since all $p_i(x)\to 0$,
\begin{eqnarray}\label{h.decomp.2}
\P\{\gamma(x)<\infty\} &\to& 1\quad\mbox{as }x\to\infty.
\end{eqnarray}
Further, by the Markov property,
\begin{eqnarray*}
\E_{x+j_0}N(x) &=& \E_{x+j_0}\sum_{n=1}^{\gamma(x)-1}\I\{X_n=x\}+
\P_{x+j_0}\{\gamma(x)<\infty\}\E_{x+j_0}N(x),
\end{eqnarray*}
which yields, by \eqref{asymp.pj0},
\begin{eqnarray*}
\E_{x+j_0}N(x) &=& \frac{1}{p_{j_0}(x)}
\E_{x+j_0}\sum_{n=1}^{\gamma(x)-1}\I\{X_n=x\}\\
&\sim& cx \E_{x+j_0}\sum_{n=1}^{\gamma(x)-1}\I\{X_n=x\}.
\end{eqnarray*}
Taking into account that
\begin{eqnarray*}
\E_{x+j_0}\sum_{n=1}^{\gamma(x)-1}\I\{X_n=x\}
&\to& \E\sum_{n=1}^{\gamma-1}\I\{S_n=0\mid S_0=j_0\}
\quad\mbox{as }x\to\infty,
\end{eqnarray*}
where
\begin{eqnarray*}
\gamma &:=& \inf\{k:S_{k-1}\le 0,\ S_k=j_0\},
\end{eqnarray*}
we conclude
\begin{eqnarray*}
\E_{x+j_0}N(x) &\sim& \widehat cx\quad\mbox{as }x\to\infty.
\end{eqnarray*}
Substituting this together with \eqref{h.decomp.1} and \eqref{h.decomp.2}
into \eqref{h.decomp} we deduce that $h(x)\sim \widehat c x$ as $x\to\infty$.
Then it follows from the integral renewal Theorem \ref{thm:renewal.gamma}
that necessarily $\widehat c=2/(2\mu-b)$ and \eqref{ren.loc.h} is proven.

To prove \eqref{ren.loc.h.1}, let us first notice
that the Markov property implies
\begin{eqnarray*}
\lefteqn{\P\Bigl\{\sum_{n=1}^\infty\I\{X_n=x\}>N\Bigr\}}\\
&=& \P\{X_n=x\mbox{ for some }n\ge 0\}
\P^N\{X_n=x\mbox{ for some }n\ge 1\mid X_0=x\}.
\end{eqnarray*}
We take
$$
c_1(x)\ :=\ \P\{X_n=x\mbox{ for some }n\ge 0\};
$$
it tends to $1$ as $x\to\infty$ because,
by the boundedness of jumps from above---see \eqref{cond.J}, for $X_0<x$,
\begin{eqnarray*}
1 &=& \P\{X_n\in[x,x+J]\mbox{ for some }n\}\\
&=& \P\{X_n=x\mbox{ for some }n\}\\
&&\hspace{20mm} +\ \P\{X_n\in[x+1,x+J]\mbox{ for some }n,\
X_n\not= x\mbox{ for all }n\},
\end{eqnarray*}
and because the second probability on the right hand side
tends to zero as $x\to\infty$. Indeed, it is not greater than
\begin{eqnarray*}
\lefteqn{\sum_{i=1}^J \P\{X_n=x+i\mbox{ for some }n,\
X_n\not= x\mbox{ for all }n\}}\\
&&\hspace{10mm}\le\ \sum_{i=1}^J
\P\{X_n\not= x\mbox{ for all }n\mid X_0=x+i\}
\end{eqnarray*}
and the $i$th probability on the right hand side
converges as $x\to\infty$ to
\begin{eqnarray*}
\P\{S_n\not= 0\mbox{ for all }n\mid S_0=i\} &=& 0,
\end{eqnarray*}
due to $\E\xi=0$. Then \eqref{ren.loc.h.1} holds with
$$
c_2(x)\ =\ x\P\{X_n\not=x\mbox{ for all }n\ge 1\mid X_0=x\}
$$
because
\begin{eqnarray*}
\E\sum_{n=1}^\infty\I\{X_n=x\}
&=& \frac{c_1(x)}
{1-\P\{X_n=x\mbox{ for some }n\ge 1\mid X_0=x\}}
\ \sim\ \frac{2x}{2\mu-b}
\end{eqnarray*}
as $x\to\infty$, by \eqref{ren.loc.h}.
\qed\end{proof}

\begin{theorem}\label{thm:srt.q}
Let the conditions of Theorem \ref{thm:srt} hold.
Then, for $q(z)\ge 0$ and any distribution of $X_0$,
\begin{eqnarray*}
\sum_{n=0}^\infty \E\bigl\{e^{-\sum_{k=0}^{n-1}q(X_k)};\ X_n=x\bigr\}
&\sim& \frac{2x}{2\mu-b}\ \E e^{-\sum_{k=0}^\infty q(X_k)}
\quad\mbox{as }x\to\infty.
\end{eqnarray*}
\end{theorem}

\begin{proof}
We may apply Lemma \ref{thm:renewal.2} whose all conditions
are satisfied by Theorem \ref{thm:srt}.
\qed\end{proof}

We now turn to the case when \eqref{1.2.G.key} holds with
$\mu=b/2$. In this case we prove the following result.
\index{Markov chain!renewal function!local asymptotics}
\index{Renewal function!Markov chain!local asymptotics}
\index{Local renewal theorem for!$\Gamma$ limit}

\begin{theorem}\label{thm:srt.crit}
Let there exists an $\gamma>0$ such that
\begin{eqnarray}\label{1.2.G.key.crit}
\frac{2m_1(x)}{m_2(x)}=\frac{1}{x}+\frac{1}{x\log x}+\ldots
+\frac{1}{x\log x\cdot\ldots\cdot\log_{(m-1)}x}
+\frac{\gamma+1+o(1)}{x\log x\cdot\ldots\cdot\log_{(m)}x}
\end{eqnarray}
as $x\to\infty$ and
$$
\limsup_{n\to\infty}X_n=\infty\quad\mbox{with probability }1.
$$
Furthermore we assume the convergence \eqref{weak.lim}
and that the limit $\xi$ satisfies \eqref{lim.mom}
In addition, let the jumps $\xi(x)$ are bounded below and above
by $J$ uniformly for all $x\in\Z^+$, that is,
\begin{equation*}
|\xi(x)|\ \le\ J\ \mbox{ for all }x\in\Z^+.
\end{equation*}
Then there exists a positive constant $c$ such that
\begin{eqnarray}\label{ren.loc.h.crit}
h(x):=H\{x\} &\sim& c x\log x\cdot\ldots\cdot\log_{(m)}x\quad\mbox{as }x\to\infty.
\end{eqnarray}
Moreover,
\begin{eqnarray}\label{ren.loc.h.1.crit}
\P\Bigl\{\sum_{n=0}^\infty\I\{X_n=x\}>N\Bigr\}
&=& c_1(x)\Bigl(1-\frac{c_2(x)}{x\log x\cdot\ldots\cdot\log_{(m)}x}\Bigr)^N,
\end{eqnarray}
where $c_1(x)\to 1$ and $c_2(x)\to 1/c$ as $x\to\infty$,
hence the family of random variables
\begin{equation}\label{ren.loc.h.2.crit}
\frac{1}{x\log x\cdot\ldots\cdot\log_{(m)}x}
\sum_{n=0}^\infty \I\{X_n=x\},\quad x\in\{1,2,3,\ldots\},
\end{equation}
is uniformly integrable.
\end{theorem}

\begin{proof}
We first derive an asymptotic formula for the probability
$\P_{x+k}\{\tau(x)=\infty\}$.
We define functions $U_{\pm}$ by the relations
\begin{eqnarray*}
U_{\pm}(x) &=&
\frac{1}{(\gamma\pm\delta)(\log_{(m)}x)^{\gamma\pm\delta}}.
\end{eqnarray*}
It is easy to see that
\begin{eqnarray}\label{srt.crit.1}
U'_{\pm}(x) &=&
-\frac{1}{x\log x\cdot\ldots\cdot\log_{(m-1)}x\cdot(\log_{(m)}x)^{\gamma+1\pm\delta}}
\end{eqnarray}
and
\begin{eqnarray}\label{srt.crit.2}
\nonumber
\frac{U''_\pm(x)}{U'_\pm(x)}&=&-\frac{1}{x}-\frac{1}{x\log x}-\ldots
-\frac{1}{x\log x\cdot\ldots\cdot\log_{(m-1)}x}\\
&&\hspace{30mm}-\frac{(\gamma+1\pm\delta)}{x\log x\cdot\ldots\cdot
\log_{(m-1)}x\cdot\log_{(m)}x}.
\end{eqnarray}
Let us compute the drift of $U_\pm(X_n)$.
Since the jumps are bounded, by Taylor's expansion,
\begin{eqnarray*}
\lefteqn{\E(U_\pm(x+\xi(x))-U_\pm(x))}\\
&&\hspace{10mm}=\ U_\pm'(x)m_1(x)+\frac12 m_2(x)U_\pm''(x)m_2(x)
+O(U_\pm'''(x))\\
&&\hspace{10mm}=\frac{U'_\pm(x)m_2(x)}{2}
\left[\frac{2m_1(x)}{m_2(x)}+\frac{U''_\pm(x)}{U'_\pm(x)}\right]
+O\left(\frac{1}{x^3}\right).
\end{eqnarray*}
Taking into account \eqref{1.2.G.key.crit}, \eqref{srt.crit.1} and
\eqref{srt.crit.2}, we infer that
\begin{eqnarray*}
\E(U_\pm(x+\xi(x))-U_\pm(x)) &\sim&
\pm\frac{b\delta}{2}\frac{1}{(x\log x\cdot\ldots\cdot
\log_{(m)}x)^2(\log_{(m)}x)^{\gamma+2\pm\delta}}.
\end{eqnarray*}
In particular, the sequences $\{U_-(X_{n\wedge\tau(x)})\}$ and
$\{U_+(X_{n\wedge\tau(x)})\}$ are super- and submartingale
respectively for all sufficiently large $x$.

Furthermore, it follows from the definition of $U_\pm$ and from \eqref{srt.crit.1} that
\begin{eqnarray*}
U_\pm(x+j)-U_\pm(x) &\sim& -jU'_\pm(x)
\ \sim\ -j(\gamma\pm\delta)\frac{U_\pm(x)}{x\log x\cdot\ldots\cdot\log_{(m)}x}.
\end{eqnarray*}
Using this relation and repeating the arguments from the derivation of \eqref{srt.4},
we obtain, for every $k\geq1$,
\begin{equation}\label{srt.crit.4}
\P_{x+k}\{\tau(x)=\infty\}\ \sim\ 
\frac{\gamma}{x\log x\cdot\ldots\cdot\log_{(m)}x}\E\psi_{k-1}
\quad\mbox{as }x\to\infty.
\end{equation}
The remaining part of the proof coincides with that of Theorem~\ref{thm:srt}.
\qed\end{proof}

\section{Comments to Chapter \ref{ch:transient}}

First time a limit theorem for Markov chains with asymptotically
zero drift was produced by Lamperti\index{Lamperti} in \cite{Lamp62},
where the convergence to a $\Gamma$-distribution was proven for
null-recurrent and transient Markov chains with jumps whose all moments
are finite. His proof was based on the method of moments.
He also claimed that his proof combined with truncation argument
for jumps continues to work for chains if we only assume that
$\sup_{x}\E \xi^4(x)\ <\ \infty$, but no proof was provided.

This result was also proven later by Klebaner\index{Klebaner}
\cite{Kleb89} for a more general random sequences
of martingale type with jumps satisfying
the following condition: for all $k\ge 3$,
$\E|\xi(x)|^k\ =\ o(x^{k-2})$ as $x\to\infty$.
The corresponding result is restricted to transient sequences.

Later the convergence to a $\Gamma$-distribution was extended by
Kersting\index{Kersting} \cite{Kersting} to martingale-type transient
random sequences with jumps having moments of order $2+\delta$ bounded
for some $\delta>0$ and under some additional smoothness conditions
on the drift.

Diffusion approximation by a Bessel process was originally proven
by Lamperti\index{Lamperti} in \cite{Lamp62} again under the
condition that absolute moments of jumps of any order
are bounded. His proof is based on the method of moments
as the proof of weak convergence to a $\Gamma$-distribution.

Convergence to the three-dimensional Bessel process for a simple 
symmetric random walk conditioned to stay non-negative has been known 
for a long time from the classic paper by Pitman\index{Pitman} \cite{Pitman}.
Bryn-Jones\index{Bryn-Jones} and Doney\index{Doney} \cite{BJD}
proved this convergence to the
three-dimensional Bessel process for a general random walk on the lattice
conditioned to stay non-negative under minimal moment conditions;
see also Caravenna\index{Caravenna} and Chaumont\index{Chaumont}
\cite{CC08} for general results
in this area. In our text application of functional results to
random walk conditioned to stay non-negative
is discussed in Section \ref{sec:h.x.cond.walks}.

For a general lattice Markov chain with drift proportional to $c/x$
and the $2+\delta$ moment of jumps bounded,
the weak convergence to a Bessel process was only proven by
Bertoin\index{Bertoin} and Kortchemski\index{Kortchemski} \cite{BK16}
for a high level initial state, $X_0=\sqrt n$.

Cs\'aki\index{Cs\'aki} et al. \cite{CFR2009} proved a strong approximation
of certain nearest neighbour random walk by a transient Bessel process
that was constructed from the latter by using stopping times.

Rosenkrantz\index{Rosenkrantz} \cite{Rosenkrantz1966} 
considered the following nearest neighbour Markov chain
with special transition probabilities: $p(0,1)=1$,
\begin{equation}\label{Rozenkrantz.model}
p(x,x-1)=\frac12\Bigl(1-\frac{\lambda}{x+\lambda}\Bigr),\quad
p(x,x+1)=1-p(x,x-1),\quad x\ge1,
\end{equation}
where $\lambda>-1/2$. This Markov chain was introduced
by Karlin\index{Karlin} and McGregor\index{McGregor}
\cite{KarlinMcgegor1959} who managed to compute the $n$-step transition
probabilities using the theory of orthogonal polynomials.
Using orthogonal polynomials Rosenkrantz proved a local version
of Lamperti's result on convergence to a $\Gamma$-distribution
and estimated large deviation probabilities
$\P\{X_n^2/nb>x\}$ in the range $x=o(\sqrt n)$ where
$\Gamma$-tail approximation still works.

In \cite{BRosenkrantzS}, Br\'ezis,\index{Br\'ezis}
Rosenkrantz\index{Rosenkrantz} and Singer\index{Singer}
again considered a nearest neighbour Markov chain
with transition probabilities similar to \eqref{Rozenkrantz.model},
for which they have got a large deviation result in the range where
$x^2-(\log n)/2\to-\infty$ by a different techniques
based on estimation of how close are expected values of a
smooth function of the original scaled process and that of
the limiting diffusion. This result allowed them to prove
the law of the iterated logarithm, that is,
\begin{eqnarray*}
\P_0\Bigl\{\limsup_{n\to\infty}\frac{X_n^2}{2n\log\log n}\le 1\Bigr\}
&=& 1.
\end{eqnarray*}

Guivarc'h\index{Guivarc'h} et al. \cite[Theorems 42 and 43]{GKR}
obtained the weak convergence to a $\Gamma$-distribution
in the transient case and the local renewal theorem in that case,
for the nearest neighbour chain with transition probabilities
\eqref{Rozenkrantz.model}. They used the orthogonal polynomials technique,
as Rosenkrantz\index{Rosenkrantz} \cite{Rosenkrantz1966}.

The orthogonal Laguerre polynomials technique was used by Voit\index{Voit}
in \cite{Voit} for proving convergence to a $\Gamma$-distribution
for critical branching processes with immigration.

A local version of Lamperti's $\Gamma$-convergence \cite{Lamp62}
was proven by Alexander\index{Alexander} in \cite[Theorem 2.4]{Alex11}
for a nearest neighbour null-recurrent Markov chain
with transition probabilities
\begin{equation*}
p(x,x-1)=1/2-\lambda/x+o(1/x),\quad
p(x,x+1)=1-p(x,x-1),\quad x\ge1.
\end{equation*}

An integral (elementary) renewal theorem for a transient Markov chain with 
drift $m_1(x)$ asymptotically proportional to $1/x$ at infinity
was proved in~\cite{DKW2013}; it was shown there that then 
the renewal function behaves as $cx^2$ for large values of $x$.
\chapter{Limit theorems for transient Markov chains
with drift decreasing slower than $1/x$}
\chaptermark{Limit theorems}% for transient and null-recurrent chains}
\label{ch:transient.2}

As in the last chapter we again assume that the first two moments
of jumps of a Markov chain $\{X_n\}$ demonstrate regular behaviour at infinity
but now we consider the case where the drift decreases at a rate slower
than $1/x$, that is,
$$
x\E\xi(x)\ \to\ \infty\quad\mbox{as }x\to\infty.
$$
A particular example is if, for some $c$, $b>0$ and $\beta\in(0,1)$,
$$
m_1(x)\sim c/x^\beta,\quad m_2(x)\to b>0\quad\mbox{as }x\to\infty.
$$
Then clearly $\{X_n\}$ escapes to infinity at a faster rate
than it happens in the case of drift of order $1/x$,
and, in contrast to the case of convergence to
a $\Gamma$-distribution, the law of large numbers holds,
$$
\frac{X_n^{1+\beta}}{n}\ \stackrel{p}\to\ c(1+\beta)\quad\mbox{as }n\to\infty.
$$
The asymptotic behaviour of the renewal measure is as follows
$$
\sum_{n=0}^\infty \P\{X_n\le x\}\ \sim\ 
\frac{x^{1+\beta}}{c(1+\beta)}\quad\mbox{as }x\to\infty.
$$
In addition, the following weak convergence to a normal distribution holds
$$
\frac{X_n-(c(1+\beta)n)^{1/(1+\beta)}}{\sqrt{b\frac{1+\beta}{1+3\beta}n}}
\ \Rightarrow\ N_{1/2+\mu/b,2}
\quad\mbox{as }n\to\infty.
$$
In this chapter, we study such kind of results.

\section{Law of Large Numbers}
\label{sec:h.x.lln}

As seen from the results discussed in the last chapter,
in the case of a drift of order $1/x$ there is no law of large
numbers for $X_n$ with a positive limit. 
%, the limit after normalisation is not a constant.
In this section we show that a drift approaching zero slower than $1/x$
gives rise to a law of large numbers for $X_n$.

Let $v(x)>0$ be a decreasing differentiable function such that
\begin{eqnarray}\label{vx.to.infty}
xv(x) &\to& \infty\quad\mbox{as }x\to\infty,
\end{eqnarray}
which is equivalent to $1/v(x)=o(x)$ and thus
\begin{eqnarray}\label{Vx.le.x.vx}
V(x) &\le& x/v(x)\ =\ o(x^2)\quad\mbox{as }x\to\infty,
\end{eqnarray}
where a convex function $V$ is defined as
\begin{eqnarray*}
V(x) &:=& \int_0^x \frac{1}{v(y)}dy\quad\mbox{for }x>0;
\end{eqnarray*}
$V(x)=0$ for $x\le 0$. In this chapter, the function $v(x)$ is
responsible for the drift of the chain, it describes 
the asymptotic behaviour of the truncated drift function, that is,
\begin{eqnarray}\label{1.2.h}
m_1^{[s(x)]}(x) &\sim& v(x)\quad\mbox{as }x\to\infty.
\end{eqnarray}

In the previous chapter we have considered the case where
the second truncated moment $m_2^{[s(x)]}(x)$ is convergent
to a positive constant,
so the drift function $m_1^{[s(x)]}(x)$ and the quotient 
$2m_1^{[s(x)]}(x)/m_2^{[s(x)]}(x)$ are asymptotically proportional
to each other
which means that $v(x)$ is typically asymptotically proportional
to the reference function $r(x)$ describing the latter quotient.

In this chapter we do not assume convergence of the second moment,
it is allowed to grow unboundedly as $x$ tends to infinity in which case
the quotient $2m_1^{[s(x)]}(x)/m_2^{[s(x)]}(x)$ decays faster than the drift.
For that reason we introduce here two functions, 
one responsible for the drift, see \eqref{1.2.h}, and another one,
a decreasing differentiable function $r(x)$ which gives a lower bound 
for the quotient, that is, we assume that, for some $\widehat x>0$,
\begin{eqnarray}\label{1.2.h.r}
\frac{2m_1^{[s(x)]}(x)}{m_2^{[s(x)]}(x)} &\ge& 
r(x)\quad\mbox{for all }x>\widehat x,
\end{eqnarray}
and that the derivative of $r(x)$ satisfies the condition 
\begin{eqnarray}\label{rx.ge.1x.pr.eps.n}
r'(x) &\ge& -(1-\varepsilon)r^2(x),\quad \varepsilon>0,\quad\mbox{for all }x>\widehat x,
\end{eqnarray}
and that
\begin{eqnarray}\label{r.times.x.infty}
r(x)\ =\ O(v(x))\ \mbox{ and }\ xr(x) &\to& \infty\quad\mbox{as }x\to\infty.
\end{eqnarray}
As usual, we define
\begin{eqnarray*}
R(x) &:=& \int_0^x r(y)dy\quad\mbox{for }x>0.
\end{eqnarray*}

Throughout this chapter we assume some regular behaviour of
both $v(x)$ and $r(x)$. The first assumption is that the function
$r(x)$ does not change much on the interval of length $1/r(x)$,
it is $1/r(x)$-{\it insensitive}, that is, for any fixed $c>0$,
\begin{eqnarray}\label{rx.2.cr}
r(x\pm c/r(x)) &\sim& r(x)\quad\mbox{as }x\to\infty.
\end{eqnarray}
We also assume that the function $v(x)$ is differentiable and
\begin{eqnarray}\label{vx.prime.v2}
v'(x) &=& o(v(x)r(x))\quad\mbox{as }x\to\infty.
\end{eqnarray}
Then the function $v(x)$ is $1/r(x)$-insensitive too, 
\begin{eqnarray}\label{vx.2.cv}
v(x\pm c/r(x)) &\sim& v(x)\quad\mbox{as }x\to\infty,
\end{eqnarray}
and hence, for any fixed $c>0$,
\begin{eqnarray*}
V(x) &\ge& \int_{x-c/r(x)}^x \frac{1}{v(y)}dy
\ \ge\ \frac{c}{r(x)}\frac{1}{v(x-c/r(x))}
\ \sim\ \frac{c}{v(x)r(x)}
\end{eqnarray*}
and
\begin{eqnarray*}
V(x\pm c/r(x)) &=& V(x)+\int_x^{x\pm c/r(x)} \frac{1}{v(y)}dy
\ =\ V(x)\pm\frac{c+o(1)}{r(x)v(x)},
\end{eqnarray*}
which yield, respectively,
\begin{eqnarray}\label{Vx2.to.infty}
V(x)v(x)r(x) &\to& \infty\quad\mbox{as }x\to\infty
\end{eqnarray}
and
\begin{eqnarray}\label{Vxr.equiv.Vx}
V(x\pm c/r(x)) &\sim& V(x)\quad\mbox{as }x\to\infty.
\end{eqnarray}
\index{Transience!law of large numbers}

\begin{theorem}\label{th:lln}
Let, for some increasing function $s(x)=o(1/r(x))$ as $x\to\infty$,
the drift conditions \eqref{1.2.h} and \eqref{1.2.h.r} hold. 
Let the following conditions hold
\begin{eqnarray}\label{lln.tail.-}
\E\{|\xi(x)|;\ \xi(x)<-s(x)\} &=& o(v(x)),\\
\label{lln.tail}
\P\{|\xi(x)| > s(x)\} &\le& p(x)v(x),
\end{eqnarray}
where $p(x)$ is a decreasing integrable at infinity function.
Assume also that 
\begin{eqnarray}\label{eq:irreducibility.lln}
\limsup_{n\to\infty}X_n &=& \infty \quad\mbox{with probability }1.
\end{eqnarray}
Then
\begin{eqnarray*}
\frac{V(X_n)}{n} &\stackrel{p}\to& 1\quad\mbox{as }n\to\infty.
\end{eqnarray*}
\end{theorem}

Since the function $V$ is convex, its inverse $V^{-1}$ is concave and hence
\begin{eqnarray}\label{lln.X.}
\frac{X_n}{V^{-1}(n)} &\stackrel{p}\to& 1\quad\mbox{as }n\to\infty.
\end{eqnarray}

Let us give a sufficient condition for \eqref{lln.tail.-} and \eqref{lln.tail}.
If the family $\{|\xi(x)|$, $x\ge 0\}$ possesses a majorant $\Xi$
satisfying $\E V(\Xi)<\infty$, that is,
$|\xi(x)|\le_{st}\Xi$ for all $x$, then there exists a function
$s(x)=o(1/r(x))$ such that \eqref{lln.tail.-} and \eqref{lln.tail} hold,
the second one follows from Lemma \ref{l:maj.p.e.V} with $U(x)\equiv 1$.
Here Lemma \ref{l:maj.p.e.V} applies because $V(x)/x$ is increasing
due to the inequality $V(x)<x/v(x)$ which implies positive derivative of $V(x)/x$.

\begin{proof}
As in the proof of Theorem \ref{thm:gamma},
we consider a modified Markov chain $\{\widetilde X_n\}$ on the same probability 
space as $\{X_n\}$ with jumps $\widetilde\xi(x)=\xi(x)\I\{|\xi(x)|\le s(x)\}$,
and, as explained there, we can assume that $\{\widetilde X_n\}$ 
satisfies the unboundedness of trajectories condition \eqref{eq:irreducibility.lln}.

By the conditions \eqref{1.2.h.r} and \eqref{r.times.x.infty},
\begin{eqnarray*}
\frac{2m_1^{[s(x)]}(x)}{m_2^{[s(x)]}(x)} &\ge& \frac{2}{x}
\ \mbox{ for all sufficiently large }x,
\end{eqnarray*}
so the condition \eqref{r-cond.5.tr.inf} is satisfied
and Theorem \ref{thm:transience.inf} implies
a.s.\ convergence $\widetilde X_n\to\infty$ as $n\to\infty$.

The chain $\{\widetilde X_n\}$ satisfies all the conditions 
of Theorem \ref{thm:Hy.above}, hence
\begin{eqnarray*}
H_y^{\widetilde X}(x,x+1/r(x)] &\le& \frac{c}{v(x)r(x)} \quad\mbox{for some }c<\infty.
\end{eqnarray*}
Therefore, Lemma \ref{l:XY.equiv} is applicable to the chains $Y=X$
and $Z=\widetilde X$ with $l(x)=1/r(x)$ and then it suffices to prove that
\begin{eqnarray}\label{HY.to.infty.H.w}
\frac{V(Z_n)}{n} &\stackrel{p}\to& 1 \quad\mbox{ as }n\to\infty.
\end{eqnarray}

Let us evaluate the expectation of the increment of $V^{1+\alpha}(Z_n)$,
$\alpha\ge 0$: for all sufficiently large $x$,
\begin{eqnarray}\label{E12}
\lefteqn{\E \{V^{1+\alpha}(Z_{n+1})-V^{1+\alpha}(Z_n)\mid Z_n=x\}}\nonumber\\
&=& \E \{V^{1+\alpha}(x+\xi(x))-V^{1+\alpha}(x);\
|\xi(x)|\le s(x)\}\nonumber\\
&=& (V^{1+\alpha})'(x)m_1^{[s(x)]}(x)
+\E(V^{1+\alpha})''(x+\theta\xi(x))\xi^2(x)/2;\ |\xi(x)|\le s(x)\}\nonumber\\
&=& (1+\alpha)V^\alpha(x)\frac{1}{v(x)}m_1^{[s(x)]}(x)\nonumber\\
&& +(1+\alpha)\E \Bigl\{\Bigl(\alpha V^{\alpha-1}\frac{1}{v^2}
-V^\alpha \frac{v'}{v^2}\Bigr)(x+\theta\xi(x))\xi^2(x)/2;\
|\xi(x)|\le s(x)\Bigr\}.\nonumber\\
\end{eqnarray}
Owing to the condition \eqref{1.2.h}, the first term on the right hand side equals
\begin{eqnarray*}
(1+\alpha)V^\alpha(x)\frac{1}{v(x)} m_1^{[s(x)]}(x)
&=& (1+\alpha+o(1))V^\alpha(x)\quad\mbox{ as }x\to\infty.
\end{eqnarray*}
By \eqref{vx.prime.v2} and \eqref{Vx2.to.infty},
\begin{eqnarray}\label{V.alpha-1}
\lefteqn{\alpha \frac{V^{\alpha-1}(x+y)}{v^2(x+y)}-V^\alpha(x+y) 
\frac{v'(x+y)}{v^2(x+y)}}\nonumber\\
&&\hspace{18mm}=\ V^\alpha(x+y)\Bigl(\frac{\alpha}{V(x+y)v^2(x+y)}
+o(1)\frac{r(x+y)}{v(x+y)}\Bigr)\nonumber\\
&&\hspace{37mm}=\ o(V^\alpha(x+y)r(x+y)/v(x+y))\nonumber\\ 
&&\hspace{55mm}=\ o(V^\alpha(x)r(x)/v(x))
\end{eqnarray}
as $x\to\infty$ uniformly on the set $|y|\le s(x)=o(1/r(x))$,
due to the insensitivity conditions \eqref{vx.2.cv} and \eqref{Vxr.equiv.Vx}.
Since the drift condition \eqref{1.2.h.r} may be rearranged as
\begin{eqnarray}\label{m2.le.m1.rx}
m_2^{[s(x)]}(x) &\le& \frac{2m_1^{[s(x)]}(x)}{r(x)}
\ \sim\ \frac{2v(x)}{r(x)},
\end{eqnarray}
the relation \eqref{V.alpha-1} implies that the second term on the right 
hand side of \eqref{E12} is of order $o(V^\alpha(x))$ as $x\to\infty$.
Substituting altogether into \eqref{E12} we finally deduce that, as $x\to\infty$,
\begin{eqnarray}\label{E12.answer}
\E \{V^{1+\alpha}(Z_{n+1})-V^{1+\alpha}(Z_n)\mid Z_n=x\}
&=& (1+\alpha+o(1))V^\alpha(x).
\end{eqnarray}
Setting now $\alpha=0$ we get
\begin{eqnarray}\label{mean.H}
\E \{V(Z_{n+1})-V(Z_n)\mid Z_n=x\} &\to& 1
\quad\mbox{ as }x\to\infty.
\end{eqnarray}
Applying here the a.s.\ convergence $Z_n\to\infty$, we conclude that
\begin{eqnarray}\label{asy.2.h}
\E V(Z_n) &\sim& n\quad\mbox{ as }n\to\infty.
\end{eqnarray}
Next take $\alpha=1$ in \eqref{E12.answer}. Then
\begin{eqnarray*}
\E \{V^2(Z_{n+1})-V^2(Z_n)\}
&=& (2+o(1))\E V(Z_n)\ \sim\ 2n\quad\mbox{ as }n\to\infty.
\end{eqnarray*}
Therefore,
\begin{eqnarray*}
\E \Bigl(\frac{V(Z_n)}{n}\Bigr)^2 &\to& 1 \quad\mbox{ as }n\to\infty.
\end{eqnarray*}
Together with \eqref{asy.2.h} it yields convergence of variances
\begin{eqnarray*}
\V\frac{V(Z_n)}{n} &\to& 0\quad\mbox{as }n\to\infty
\end{eqnarray*}
which in its turn implies the desired convergence \eqref{HY.to.infty.H.w}.
\qed\end{proof}

\section{Strong Law of Large Numbers}
\label{sec:h.x.slln}

As usual, the strong law of large numbers requires stronger assumptions
than the law of large 
numbers.\index{Transience!strong law of large numbers}\index{Strong law of large numbers!for Markov chains}
Below we assume a stronger condition on $m_2^{[s(x)]}(x)$
than the drift condition \eqref{1.2.h.r} which can be seen as an upper bound
on $m_2^{[s(x)]}(x)$, see \eqref{m2.le.m1.rx}.

\begin{theorem}\label{th:slln}
Let the conditions of Theorem \ref{th:lln} hold. In addition, let
\begin{eqnarray}\label{1.2.h.slln}
m_2^{[s(x)]}(x) &\le& \frac{v(x)}{r(x)f(V(x))}\quad\mbox{for }x>\widehat x,
\end{eqnarray}
for some increasing function $f:\R ^+\to\R ^+$
such that both functions $f(x)$ and $x/f(x)$ are concave and
\begin{eqnarray}\label{sum.n.f}
\sum_{n=1}^\infty \frac{1}{nf(n)} &<& \infty.
\end{eqnarray}
Then
\begin{eqnarray*}
\frac{V(X_n)}{n} &\stackrel{a.s.}\to& 1 \quad\mbox{ as }n\to\infty.
\end{eqnarray*}
\end{theorem}

As for Theorem \ref{th:lln}, this convergence, 
due to the concavity of the inverse $V^{-1}$, implies
\begin{eqnarray*}
\frac{X_n}{V^{-1}(n)} &\stackrel{a.s.}\to& 1 \quad\mbox{ as }n\to\infty.
\end{eqnarray*}

\index{Transience!strong law of large numbers}
\index{Strong law of large numbers!for Markov chains}
\begin{corollary}\label{cor:lln.beta}
Let the condition \eqref{eq:irreducibility.lln} hold.
Let $\E \xi(x)\sim c/x^\beta$, where $c>0$ and $\beta\in[0,1)$, and
\begin{eqnarray*}
\sup_x\ \E |\xi(x)|^{1+\beta}\log^{1+\delta}(1+|\xi(x)|) &<& \infty
\quad\mbox{for some }\delta>0.
\end{eqnarray*}
Then 
\begin{eqnarray}\label{slln.X.}
\frac{X_n^{1+\beta}}{n} &\stackrel{a.s.}\to& c(1+\beta)
\quad\mbox{as }n\to\infty.
\end{eqnarray}
\end{corollary}

\begin{theopargself}
\begin{proof}[of Corollary \ref{cor:lln.beta}]
Here $v(x)=c/x^\beta$ and $V(x)=x^{1+\beta}/c(1+\beta)$. Observe that
\begin{eqnarray*}
m_2^{[s(x)]}(x) &=& \E\{\xi^2(x);\ |\xi(x)|\le s(x)\} \\
&\le& \frac{s^{1-\beta}(x)}{\log^{1+\delta}s(x)}
\E |\xi(x)|^{1+\beta}\log^{1+\delta}(1+|\xi(x)|).
\end{eqnarray*}
Consider $r(x)=\frac{\log^{(1-\beta)\delta/8}x}{x}$, 
a truncation level $s(x)=x/\log^{\delta/4}x$,
and a function $f(x)=\log^{1+\delta}(1+x)$,
then the conditions \eqref{1.2.h.slln} and \eqref{sum.n.f} are satisfied.
The conditions \eqref{1.2.h}, \eqref{1.2.h.r}, \eqref{lln.tail.-} 
and \eqref{lln.tail} are satisfied too because
\begin{eqnarray*}
\E\{|\xi(x)|;\ |\xi(x)|>s(x)\} &\le&
\frac{1}{s^\beta(x)\log^{1+\delta}s(x)}
\E |\xi(x)|^{1+\beta}\log^{1+\delta}(1+|\xi(x)|)\\
&\le& c\frac{\log^{\beta\delta/4}x}{x^\beta\log^{1+\delta}x}
\ =\ o(1/x^\beta)\ =\ o(v(x))
\end{eqnarray*}
and
\begin{eqnarray*}
\P\{|\xi(x)|>s(x)\} &\le&
\frac{1}{s^{1+\beta}(x)\log^{1+\delta}s(x)}
\E |\xi(x)|^{1+\beta}\log^{1+\delta}(1+|\xi(x)|)\\
&\le& c\frac{\log^{(1+\beta)\delta/4}x}{x^{1+\beta}\log^{1+\delta}x}
\ =\ o(v(x))\frac{1}{x\log^{1+\delta/2}x}
\end{eqnarray*}
where the quotient on the right hand side is integrable at infinity.
\qed\end{proof}
\end{theopargself}

\index{Transience!strong law of large numbers}
\index{Strong law of large numbers!for Markov chains}
\begin{corollary}\label{cor:lln.log}
Let $\E \xi(x)\sim c(\log x)^{1+\beta}/x$, $\beta>0$,
and
\begin{eqnarray*}
\sup_x\ \E \xi^2(x) &<& \infty.
\end{eqnarray*}
Then
\begin{eqnarray*}
\frac{X_n^2}{(\log X_n)^{1+\beta}n} &\stackrel{a.s.}\to& 2c
\quad\mbox{as }n\to\infty,
\end{eqnarray*}
which is equivalent to the following convergence
\begin{eqnarray*}
\frac{X_n^2}{n\log^{1+\beta}n} &\stackrel{a.s.}\to& 2^{2+\beta}c
\quad\mbox{as }n\to\infty.
\end{eqnarray*}
\end{corollary}

\begin{theopargself}
\begin{proof}[of Corollary \ref{cor:lln.log}]
Under this drift condition, 
$v(x)=c(\log x)^{1+\beta}/x$ for sufficiently large $x$ 
and then $V(x)\sim x^2/2c\log^{1+\beta}x$.
Observe that the value of $m_2^{[s(x)]}(x)$ is bounded here
regardless of the choice of the truncation level $s(x)$.
Consider $r(x)=\frac{\log^{\beta/4}x}{x}$, 
a truncation level $s(x)=x/\log^{\beta/4}x$
and a function $f(x)=\log^{1+\beta/2}(1+x)$.
Then the conditions \eqref{1.2.h.slln} and \eqref{sum.n.f} are satisfied.
The conditions \eqref{1.2.h}, \eqref{1.2.h.r}, 
\eqref{lln.tail.-} and \eqref{lln.tail}
are satisfied too because
\begin{eqnarray*}
\E\{|\xi(x)|;\ |\xi(x)|>s(x)\} &\le& \frac{1}{s(x)} \E \xi^2(x)\\
&\le& c\frac{\log^{\beta/4}x}{x}
\ =\ o((\log x)^{1+\beta}/x)\ =\ o(v(x))
\end{eqnarray*}
and
\begin{eqnarray*}
\P\{|\xi(x)|>s(x)\} &\le& \frac{1}{s^2(x)} \E \xi^2(x)\\
&\le& c\frac{\log^{\beta/2} x}{x^2}
\ =\ o(v(x))\frac{1}{x\log^{1+\beta/3}x}
\end{eqnarray*}
where the quotient on the right hand side is integrable at infinity.
\qed\end{proof}
\end{theopargself}

Notice that drift like $v(x)=(\log x)/x$ or more speedy decreasing
is excluded from consideration in Theorem \ref{th:slln}
because then $xv(x)=\log x$ and $f(x)$
cannot be chosen growing faster than $\log x$ to satisfy \eqref{1.2.h.slln},
thus the condition \eqref{sum.n.f} fails.
The LLN, Theorem \ref{th:lln}, is still applicable.

The proof of Theorem \ref{th:slln} is based on the following
generalisation of the SLLN to martingales, see e.g. \cite[Theorem 2.2]{HJP}.
\index{Strong law of large numbers!for martingales}

\begin{theorem}\label{thm:slln.mart}
Let $\{\mathcal F_n\}$ be a filtration and $\{X_n\}$ be a martingale 
with respect to $\{\mathcal F_n\}$ which is square integrable. If
$$
\sum_{n=1}^\infty \frac{\E(X_{n+1}-X_n)^2}{n^2}\ <\ \infty,
$$
then $X_n/n\to 0$ a.s. as $n\to\infty$.
\end{theorem}

\begin{theopargself}
\begin{proof}[of Theorem \ref{th:slln}]
As in Theorem \ref{th:lln}, it suffices to show that
\begin{eqnarray}\label{HY.to.infty.H}
\frac{V(Z_n)}{n} &\to& 1 \quad\mbox{a.s. as }n\to\infty.
\end{eqnarray}

Denote
\begin{eqnarray*}
m^V_n &:=& \E \{V(Z_{n+1})-V(Z_n)|{\mathcal F}_n\},
\end{eqnarray*}
where ${\mathcal F}_n=\sigma(Z_0,\ldots,Z_n)$.
By the Markov property, as it was calculated in the proof of Theorem
\ref{th:lln} with $\alpha=0$, on the event $Z_n\to\infty$,
\begin{eqnarray}\label{E12.answer.slln}
m^V_n &\to& 1\quad\mbox{as }n\to\infty;
\end{eqnarray}
Put
\begin{eqnarray*}
D_n &:=& V(Z_{n+1})-V(Z_n)-m^V_n,
\end{eqnarray*}
so that
\begin{eqnarray*}
V(Z_n) &=& \sum_{k=0}^{n-1} m^V_k+\sum_{k=0}^{n-1}D_k.
\end{eqnarray*}
By \eqref{E12.answer.slln} and the convergence $Z_n\to\infty$ we have
\begin{eqnarray*}
\frac{1}{n}\sum_{k=0}^{n-1} m^V_k &\to& 1
\quad\mbox{a.s. as }n\to\infty,
\end{eqnarray*}
and consequently the required convergence
\eqref{HY.to.infty.H} would follow once it is proven that
\begin{eqnarray}\label{22}
\frac{1}{n}\sum_{k=0}^{n-1}D_k &\to& 0
\quad\mbox{a.s. as }n\to\infty.
\end{eqnarray}
The process $\sum_{k=0}^{n-1}D_k$ constitutes
a martingale with respect to the filtration $\{{\mathcal F}_{n-1}\}$,
hence the a.s.\ convergence \eqref{22} would follow
by Theorem \ref{thm:slln.mart} if we have managed to prove that
the increments of this martingale satisfy the condition
\begin{eqnarray}\label{23}
\sum_{n=1}^\infty\frac{\E D_n^2}{n^2} &<& \infty.
\end{eqnarray}
By the construction of $D_k$ and due to the insensitivity condition \eqref{vx.2.cv},
for $x\ge\widehat x$,
\begin{eqnarray*}
\E \{D_k^2\mid Z_k=x\} &=& \V D_k\\
&\le& \E\{[V(x+\xi(x))-V(x)]^2;\ |\xi(x)|\le s(x)\}\\
&\le& c_1(V'(x))^2\E\{\xi^2(x);\ |\xi(x)|\le s(x)\}\\
&\le& c_2\frac{1}{v^2(x)}\frac{v(x)}{r(x)f(V(x))}
\ \le\ c_3\frac{V(x)}{f(V(x))},
\end{eqnarray*}
owing to \eqref{1.2.h.slln} and \eqref{Vx2.to.infty}. 
Since the function $y/f(y)$ is concave, by Jensen's inequality
\begin{eqnarray*}
\E D_k^2 &\le& c_3\frac{\E V(Z_k)}{f(\E V(Z_k))}
\ \le\ c_3\frac{2k}{f(k/2)},
\end{eqnarray*}
for sufficiently large $k$, as follows from \eqref{asy.2.h}.
For $x<\widehat x$,
\begin{eqnarray*}
\E \{D_k^2\mid Z_k=x\} &\le& V^2(\widehat x+s(\widehat x))\ =:\ c_4.
\end{eqnarray*}
Then it follows from concavity of $f(y)$ that
$\E D_k^2 \le 2c_3k/f(k)$ which yields
\begin{eqnarray*}
\sum_{k=1}^\infty\frac{\E D_k^2}{k^2}
&\le& \sum_{k=1}^\infty\Bigl(\frac{2c_3}{kf(k)}+\frac{c_4}{k^2}\Bigr) < \infty,
\end{eqnarray*}
by the condition \eqref{sum.n.f},
hence \eqref{23} holds and the proof is complete.
\qed\end{proof}
\end{theopargself}

\section{Integral renewal theorem for transient chain satisfying
law of large numbers}
\sectionmark{Integral renewal theorem for transient chain}
\label{sec:renewal.lln}

In this section we discuss asymptotics of the renewal measure for $\{X_n\}$ 
satisfying the conditions of the law of large numbers. 
Notice that, in particular, we do not assume convergence of 
the second moment at infinity.
\index{Renewal theorem!under law of large numbers}
\index{Markov chain!renewal theorem}

\begin{theorem}\label{thm:renewal.lln}
Under the conditions of the law of large numbers,
Theorem \ref{th:lln}, there exists an $\widehat x$
such that, given any distribution of $X_0$,
\begin{eqnarray*}
H(\widehat x,x] &\sim& V(x)\ =\ \int_0^x\frac{1}{v(y)}dy\quad\mbox{ as }x\to\infty.
\end{eqnarray*}
\end{theorem}

\begin{proof}
We split the proof of the asymptotics for $H$ into two parts, 
upper and lower bounds. First let us prove a proper upper bound.
The chain $\{X_n\}$ satisfies all the conditions of Theorem \ref{thm:Hy.above}.
For any $A>1$, by the Markov property and \eqref{y.back.x},
\begin{eqnarray}\label{estimate.for.Hy.2.1x.pre.new.lln}
\lefteqn{H(\widehat x,x]}\nonumber\\ 
&\le& \E\sum_{n=0}^{T\bigl(x+\frac{A}{r(x)}\bigr)-1}
\I\{\widehat x<X_n\le x\}
+\P\Bigl\{X_n\le x
\mbox{ for some }n\Big| X_0>x+\frac{A}{r(x)}\Bigr\}
\sup_z H_z(\widehat x,x]\nonumber\\
&\le& \E L(\widehat x,T(x+A/r(x)))
+\bigl(e^{\delta(R(x)-R(x+A/r(x)))}+o(1)\bigr)\sup_z H_z(\widehat x,x]
\end{eqnarray}
as $x\to\infty$ uniformly for all $A>1$, 
where a stopping time $T(t)$ is defined as
\begin{eqnarray*}
T(t) &:=& \min\{n\ge 1:X_n>t\},
\end{eqnarray*}
and $L$ is defined in \eqref{def.L}. We have
\begin{eqnarray*}
e^{\delta(R(x)-R(x+A/r(x)))}
&=& e^{-\delta\int_x^{x+A/r(x)} r(y)dy} 
\ \le\ e^{-\delta Ar(x+A/r(x))/r(x)}\le e^{-\delta A/2},
\end{eqnarray*}
for all sufficiently large $x$, due to the condition \eqref{rx.2.cr}.
Applying the upper bound of Theorem \ref{thm:Hy.above} to
the right hand side of \eqref{estimate.for.Hy.2.1x.pre.new.lln} 
we deduce that, for some $c<\infty$,
\begin{eqnarray*}
H(\widehat x,x] &\le&
\E L(\widehat x,T(x+A/r(x)))+ \bigl(e^{-\delta A/2}+o(1)\bigr)
c\int_{\widehat x}^{x+1/r(x)}\frac{1}{v(z)}dz
\end{eqnarray*}
as $x\to\infty$ uniformly for all $A>1$.
Applying now Theorem \ref{l:uniform} we deduce that
\begin{eqnarray*}
\E L(\widehat x,T(x+A/r(x))) &\le&
\int_{\widehat x}^{x+A/r(x)+s(x+A/r(x))}\frac{1}{v(z)}dz
\end{eqnarray*}
and therefore,
\begin{eqnarray}\label{estimate.for.Hy.2.1x.n.lln}
H(\widehat x,x] &\le&
\bigl(1+ce^{-\delta A/2}+o(1)\bigr)
\int_{\widehat x}^{x+(A+1)/r(x)}\frac{1}{v(z)}dz\nonumber\\
&\sim& \bigl(1+ce^{-\delta A/2}+o(1)\bigr) V(x)
\quad\mbox{as }x\to\infty,
\end{eqnarray}
for any fixed $A$, owing to \eqref{Vxr.equiv.Vx}. 
Letting now $A\to\infty$, we get the required upper bound for $H(0,x]$.

The lower bound is simpler. Indeed,
\begin{eqnarray*}
H(\widehat x,x] &=& \sum_{n=0}^\infty \P\{\widehat x<X_n\le x\}\\
&\ge&  \sum_{10V(\widehat x)\le n\le (1-\varepsilon)V(x)} \P\{V(\widehat x)<V(X_n)\le V(x)\}\\
&\ge&  \sum_{10V(\widehat x)\le n\le (1-\varepsilon)V(x)} 
\P\Bigl\{0.1<\frac{V(X_n)}{n}\le \frac{1}{1-\varepsilon}\Bigr\},
\end{eqnarray*}
for any fixed $\varepsilon>0$. 
Therefore, by the law of large numbers for $X_n$,
$V(X_n)/n\to 1$, hence
\begin{eqnarray*}
H(\widehat x,x] &\ge& (1-\varepsilon+o(1))V(x)\quad\mbox{as }x\to\infty.
\end{eqnarray*}
This concludes the proof due to the arbitrary choice of $\varepsilon>0$.
\qed\end{proof}

\section{Central limit theorem}
\label{sec:normal}

In this section we study the case where $xm_1(x)\to\infty$ as $x\to\infty$ 
and the strong law of large numbers holds
\begin{eqnarray}\label{as:slln}
\frac{X_n}{V^{-1}(n)} &\stackrel{a.s.}\to& 1 \quad\mbox{ as }n\to\infty,
\end{eqnarray}
given any distribution of $X_0$,
for sufficient conditions see Theorem \ref{th:slln}.
Then it is natural to expect a normal approximation
to the distribution of fluctuations around the mean value.
In the next result we specify additional conditions that guarantee
a normal approximation.

In addition to the condition \eqref{vx.to.infty},
let the function $v(x)$ be regularly varying at infinity 
with index $-\beta\in[-1,0]$. Then, by Karamata's theorem,
\begin{eqnarray}\label{Karamata}
V(x) &=& \int_0^x\frac{1}{v(y)}dy\ \sim\ \frac{1}{1+\beta}\frac{x}{v(x)}
\quad\mbox{ as }x\to\infty.
\end{eqnarray}
In this section we consider the case where the second truncated moment
of jumps has a positive limit at infinity, so the drift function
$m_1^{[s(x)]}(x)$ and the quotient $2m_1^{[s(x)]}(x)/m_2^{[s(x)]}(x)$
are asymptotically proportional to each other. 
For that reason any function $r(x)$ of order $o(v(x))$
delivers a lower bound for the quotient, that is,
satisfies the drift condition \eqref{1.2.h.r}.

Notice that the function $r(x)=\sqrt{v(x)/x}$ is asymptotically sandwiched 
between $1/x$ and $v(x)$, more precisely,
\begin{eqnarray}\label{x.vx.sandwich}
\frac{\sqrt{v(x)/x}}{1/x}\ \to\ \infty\quad&\mbox{and}&\quad
\frac{\sqrt{v(x)/x}}{v(x)}\ \to\ 0\quad\mbox{as }x\to\infty.
\end{eqnarray}
\index{Transience!central limit theorem}
\index{Central limit theorem!for Markov chains!with asymptotically zero drift}
Notice that the condition \eqref{vx.prime.v2} 
with $r(x)=\sqrt{v(x)/x}$ reduces to the following one
\begin{eqnarray}\label{vx.prime.v2.v}
v'(x) &=& o\Bigl(\sqrt{v^3(x)/x}\Bigr)\quad\mbox{as }x\to\infty.
\end{eqnarray}

\begin{theorem}\label{thm:clt.2}
Let the condition \eqref{vx.prime.v2.v} hold.
Let, for some increasing function $s(x)=o\bigl(\sqrt{x/v(x)}\bigr)$,
\begin{eqnarray}\label{1.2.clt.2}
m^{[s(x)]}_1(x) = v(x)+o(\sqrt{v(x)/x}) \ \mbox{ and }\ m^{[s(x)]}_2(x)\to b>0
\quad\mbox{as }x\to\infty
\end{eqnarray}
and the following conditions hold
\begin{eqnarray}\label{rec.3.1.E.doupl}
\E\{|\xi(x)|;\ \xi(x)\le-s(x)\} &=& o(v(x))\quad\mbox{as }x\to\infty,\\
\label{lln.tail.clt.2}
\P\{|\xi(x)| > s(x)\} &\le& p(x)v(x),
\end{eqnarray}
where $p(x)$ is a decreasing function integrable at infinity. Then
\begin{eqnarray*}
\frac{X_n-V^{-1}(n)}{\sqrt{b\frac{1+\beta}{1+3\beta}n}}
&\Rightarrow& N_{0,1}
\quad\mbox{ as }n\to\infty.
\end{eqnarray*}
\end{theorem}

%In the sequel we need the following generalisation of the central limit theorem to martingales,
%see e.g. \cite[Theorem 2.2]{Brown} and \cite[Theorem 2]{GSS}.\index{Central limit theorem!for martingales}

%\begin{theorem}\label{thm:clt.mart}
%Let $\mathcal F_n$ be a filtration and $X_n$ be a martingale 
%with respect to $\mathcal F_n$ which is square integrable. Assume that 
%$$
%\frac{\sum_{k=1}^n \E\{(X_{k+1}-X_k)^2\mid X_k\}}{\E X_n^2}\ \to\ 1\quad\mbox{in probability as }n\to\infty
%$$
%and that the Lindeberg condition holds, for all $\varepsilon>0$,
%$$
%\frac{1}{\E X_n^2}  \sum_{k=1}^n \E\{(X_{k+1}-X_k)^2;\ |X_{k+1}-X_k|>\varepsilon\sqrt{\E X_n^2}\}
%\ \to\ 0\quad\mbox{as }n\to\infty.
%$$
%Then 
%\begin{eqnarray*}
%\frac{X_n}{\sqrt{\E X_n^2}} &\Rightarrow& N(0,1)\quad\mbox{ as }n\to\infty.
%\end{eqnarray*}
%\end{theorem}

The proof is based on the following generalisation of the central limit theorem 
to martingales which goes back to \cite[Theorem 2]{Brown}.
Let $\{\mathcal F_n,n\ge 1\}$ be a filtration and $\{X_n,n\ge 1\}$ 
be a square integrable martingale with respect to $\{\mathcal F_n\}$. 
\index{Central limit theorem!for martingales}

\begin{theorem}\label{thm:clt.mart}
Let $\{X_n\}$ be a martingale such that
$$
\frac{\sum_{k=1}^n \E\{(X_{k+1}-X_k)^2\mid \mathcal F_k\}}{\E X_n^2}
\ \stackrel{p}\to\ 1\quad\mbox{as }n\to\infty
$$
and the conditioned Lindeberg condition holds: for all $\varepsilon>0$,
$$
\frac{1}{\E X_n^2}  \sum_{k=1}^n 
\E\{(X_{k+1}-X_k)^2\I\{|X_{k+1}-X_k|>\varepsilon\sqrt{\E X_n^2}\} \mid \mathcal F_k\}
\ \stackrel{p}\to\ 0\quad \mbox{as }n\to\infty.
$$
Then $X_n/\sqrt{\E X_n^2}$ converges weakly to
a standard normal distribution as $n\to\infty$.
\end{theorem}

\begin{theopargself}
\begin{proof}[of Theorem \ref{thm:clt.2}]
As in the proof of Theorem \ref{thm:gamma},
we consider a modified Markov chain $\{\widetilde X_n\}$ on the same probability 
space as $\{X_n\}$ with jumps $\widetilde\xi(x)=\xi(x)\I\{|\xi(x)|\le s(x)\}$,
and, as explained there, we can assume that $\{\widetilde X_n\}$ 
satisfies the unboundedness of trajectories condition \eqref{eq:irreducibility.lln}.

Notice that $\{\widetilde X_n\}$ satisfies the conditions 
\eqref{r-cond.2} and \eqref{rx.ge.1x.pr.eps} 
with $r(x)=\sqrt{v(x)/x}$, for a sufficiently large $\widehat x$.
The relation $s(x)=o(\sqrt{x/v(x)})$ is equivalent to $s(x)=o(1/r(x))$.
Since $v(x)$ is regularly varying at infinity, it satisfies
the condition \eqref{def.cv}.
Therefore, Theorem \ref{thm:Hy.above} applies to $\{\widetilde X_n\}$, hence
\begin{eqnarray*}
H_y^{\widetilde X}(x,x+1/r(x)] &\le& \frac{c_1}{v(x)r(x)},
\end{eqnarray*}
which in its turn allows us to apply Lemma \ref{l:XY.equiv} to a pair of
the chains $Y=X$ and $Z=\widetilde X$. Hence it suffices
to prove the statement of the theorem for the process $\{Z_n\}$,
that is, it is sufficient to prove that
\begin{eqnarray}\label{YAtoN}
\frac{Z_n-V^{-1}(n)}{\sqrt{b\frac{1+\beta}{1+3\beta}n}}
&\Rightarrow& N_{0,1}
\quad\mbox{ as }n\to\infty.
\end{eqnarray}
The analogue of \eqref{as:slln} for $Z_n$ reads as
\begin{eqnarray}\label{as:slln.Z}
\frac{Z_n}{V^{-1}(n)} &\stackrel{a.s.}\to& 1 \quad\mbox{ as }n\to\infty;
\end{eqnarray}

Denote
\begin{eqnarray*}
m^V_n &:=& \E \{V(Z_{n+1})-V(Z_n)\mid {\mathcal F}_n\},
\end{eqnarray*}
where ${\mathcal F}_n:=\sigma(Z_0,\ldots,Z_n)$. 
We have $m^V_n=m^V(Z_n)$ where
\begin{eqnarray*}
m^V(x) &:=& \E \{V(Z_{n+1})-V(Z_n)\mid Z_n=x\}\\
&=& \E \{V(x+\xi(x))-V(x);\ |\xi(x)|\le s(x)\}\\
&=& V'(x)m_1^{[s(x)]}(x)
+\frac{1}{2}\E\{V''(x+\theta\xi(x))\xi^2(x);\ |\xi(x)|\le s(x)\}\\
&=& \frac{1}{v(x)}m_1^{[s(x)]}(x)
-\frac{1}{2}\E \Bigl\{\frac{v'}{v^2}(x+\theta\xi(x))\xi^2(x);\
|\xi(x)|\le s(x)\Bigr\}.
\end{eqnarray*}
Then it follows from the conditions \eqref{1.2.clt.2} and
\eqref{vx.prime.v2.v} that
\begin{eqnarray}\label{mV1}
m^V(x) &=& 1+o(1/\sqrt{xv(x)})+O(v'(x)/v^2(x))\nonumber\\
&=& 1+o(1/\sqrt{xv(x)})\quad\mbox{as }x\to\infty.
\end{eqnarray}
Further, define
\begin{eqnarray*}
Q_n &:=& \E \{(V(Z_{n+1})-V(Z_n))^2|{\mathcal F}_n\}.
\end{eqnarray*}
We observe that $Q_n=Q(Z_n)$ where
\begin{eqnarray}\label{sim.bZ}
Q(x) &:=& \E \{(V(x+\xi(x))-V(x))^2;\ |\xi(x)|\le s(x)\}\nonumber\\
&=& \E\{(V'(x+\theta\xi(x))\xi(x))^2;\ |\xi(x)|\le s(x)\}\nonumber\\
&\sim& b/v^2(x)\quad\mbox{ as }x\to\infty,
\end{eqnarray}
because $V'(x+y)=1/v(x+y)\sim 1/v(x)$ as $x\to\infty$
uniformly for $|y|\le s(x)$.

Let us center $V(Z_n)$, that is, let us consider
\begin{eqnarray*}
\widetilde Z_n &:=& V(Z_n)-\sum_{j=0}^{n-1} m^V_j\\
&=& V(Z_n)-V(Z_{n-1})-m^V_{n-1} +\widetilde Z_{n-1},
\end{eqnarray*}
so $\{\widetilde Z_n\}$ constitutes a martingale with respect to 
the filtration $\{\mathcal F_n\}$.
It follows from the strong law of large numbers \eqref{as:slln.Z}
and from \eqref{sim.bZ} that
\begin{eqnarray*}
\lefteqn{v^2(V^{-1}(j))
\E\bigl\{\bigl(\widetilde Z_{j+1}-\widetilde Z_j\bigr)^2\mid \mathcal F_j\bigr\}}\\
&&\hspace{10mm}=\ v^2(V^{-1}(j))
\E\bigl\{\bigl(V(Z_{j+1})-V(Z_j)-M_j\bigr)^2\mid \mathcal F_j\bigr\}\\
&&\hspace{25mm}=\ v^2(V^{-1}(j))
\bigl[\E\bigl\{(V(Z_{j+1})-V(Z_j))^2\mid \mathcal F_j\bigr\}-M^2_j\bigr]\\
&&\hspace{35mm}=\ v^2(V^{-1}(j))
\bigl[\E\bigl\{\bigl(V(Z_{j+1})-V(Z_j)\bigr)^2\mid \mathcal F_j\bigr\}+O(1)\bigr]\\
&&\hspace{50mm}\stackrel{a.s.}\to\ b\quad\mbox{ as }j\to\infty,
\end{eqnarray*}
which implies the convergence
\begin{eqnarray}\label{sigma2.convergence}
\frac{1}{\sigma_n^2} \sum_{j=0}^{n-1}
\E\bigl\{\bigl(\widetilde Z_{j+1}-\widetilde Z_j\bigr)^2
\mid \mathcal F_j\bigr\} &\stackrel{a.s.}\to& 1
\quad\mbox{ as }n\to\infty,
\end{eqnarray}
where
\begin{eqnarray}\label{sigma.below}
\sigma_n^2 &:=& b\sum_{j=0}^{n-1}\frac{1}{v^2(V^{-1}(j))}\nonumber\\
&\ge& b\frac{n}{2}\frac{1}{v^2(V^{-1}((n-1)/2))}
\ \ge\ c_1\frac{n}{v^2(V^{-1}(n))}\quad\mbox{for some }c_1>0.
\end{eqnarray}

Since $|Z_1-Z_0|\le s(x)$ given $Z_0=x$,
\begin{eqnarray*}
\bigl|V(Z_1)-V(Z_0)\bigr| &=& V'(x+\theta\xi(x))|\xi(x)|\I\{|\xi(x)|\le s(x)\}\\
&\le& \frac{s(x)}{v(x+s(x))}.
\end{eqnarray*}
By the choice of $s(x)=o(\sqrt{x/v(x)})$, given $Z_0=x+y$,
\begin{eqnarray*}
\bigl|V(Z_1)-V(Z_0)\bigr|^2
&\le& \gamma(x)x/v^3(x)\quad\mbox{ for all }|y|\le x/2,
\end{eqnarray*}
where $\gamma(x)\to 0$ as $x\to\infty$.
Hence, on the event $|Z_n-V^{-1}(n)|\le V^{-1}(n/2)$,
\begin{eqnarray*}
\bigl|V(Z_{n+1})-V(Z_n)\bigr|^2
&\le& \gamma(V^{-1}(n))\frac{V^{-1}(n)}{v^3(V^{-1}(n))}\\
&\le& \gamma(V^{-1}(n))\frac{c_2n}{v^2(V^{-1}(n))},
\end{eqnarray*}
because $z/v(z)\le c_2V(z)$ for some $c_2<\infty$, by \eqref{Karamata}.
Then, on the same event, by \eqref{sigma.below},
\begin{eqnarray*}
\bigl|V(Z_{n+1})-V(Z_n)\bigr|^2
&\le& \gamma(V^{-1}(n))\frac{c_2}{c_1}\sigma^2_n.
\end{eqnarray*}
By the strong law of large numbers \eqref{as:slln.Z},
\begin{eqnarray*}
\P\{|Z_n-V^{-1}(n)|\le V^{-1}(n/2)\mbox{ for all sufficiently large }n\} &=& 1.
\end{eqnarray*}
This allows us to conclude that, for any fixed $\delta>0$,
\begin{eqnarray*}
\E\bigl\{\bigl(\widetilde Z_{j+1}-\widetilde Z_j\bigr)^2;\
|\widetilde Z_{j+1}-\widetilde Z_j|\ge \delta\sigma_n
\mid \mathcal F_j\bigr\} &\stackrel{a.s.}\to& 0
\quad\mbox{ as }j\to\infty,\ j\le n-1,
\end{eqnarray*}
hence
\begin{eqnarray*}
\frac{1}{\sigma_n^2} \sum_{j=0}^{n-1}
\E\bigl\{\bigl(\widetilde Z_{j+1}-\widetilde Z_j\bigr)^2;\
|\widetilde Z_{j+1}-\widetilde Z_j|\ge \delta\sigma_n
\mid \mathcal F_j\bigr\} &\stackrel{a.s.}\to& 0
\quad\mbox{ as }n\to\infty.
\end{eqnarray*}
So, the martingale $\{\widetilde Z_n\}$ satisfies
the conditions of the central limit theorem for mar\-tin\-gales---see 
Theorem \ref{thm:clt.mart}--- and
we conclude that
\begin{eqnarray*}
\frac{\widetilde Z_n}{\sigma_n}
= \frac{V(Z_n)-\sum_{j=0}^{n-1} m^V_j}{\sigma_n}
&\Rightarrow& N_{0,1}
\quad\mbox{ as }n\to\infty.
\end{eqnarray*}
Further, as follows from the decomposition \eqref{mV1}
for the mean drift of $V(Z_n)$,
\begin{eqnarray*}
\Biggl|\sum_{j=0}^{n-1} m^V_j-n\Biggr|
&\le& c_3\sum_{j=0}^{n-1} \I\{Z_j\le V^{-1}(j)/2\}+
\sum_{j=0}^{n-1}
\frac{\psi_j}{\sqrt{V^{-1}(j)v(V^{-1}(j)/2)}},
\end{eqnarray*}
where $\psi_j\to 0$ as $j\to\infty$.
The first sum on the right hand side is bounded by
\begin{eqnarray*}
\zeta &:=& c_3\sum_{j=0}^\infty \I\{Z_j\le V^{-1}(j)/2\},
\end{eqnarray*}
which is a proper random variable, due to the strong law of large numbers
\eqref{as:slln}, whereas the second one is of order
\begin{eqnarray*}
o(1)\sum_{j=0}^{n-1} \frac{1}{\sqrt{V^{-1}(j)v(V^{-1}(j))}}
&=& o\biggl(\frac{n}{\sqrt{V^{-1}(n)v(V^{-1}(n))}}\biggr)
\quad\mbox{as }n\to\infty.
\end{eqnarray*}
Since $V(z)\le z/v(z)$,
$$
\frac{V^{-1}(n)}{v(V^{-1}(n))}\ge V(V^{-1}(n))=n
$$
and hence
\begin{eqnarray*}
\frac{n}{\sqrt{V^{-1}(n)v(V^{-1}(n))}} &\le& \frac{\sqrt n}{v(V^{-1}(n))}.
\end{eqnarray*}
Combining altogether including the lower bound \eqref{sigma.below}
for $\sigma_n$, we get
\begin{eqnarray*}
\Biggl|\sum_{j=0}^{n-1} m^V_j-n\Biggr|
&\le& o(\sigma_n)+\zeta
\quad\mbox{ as }n\to\infty.
\end{eqnarray*}
Thus,
\begin{eqnarray*}
\frac{V(Z_n)-n}{\sigma_n} &\Rightarrow& N_{0,1}
\quad\mbox{ as }n\to\infty.
\end{eqnarray*}
To conclude convergence to a normal distribution for $Z_n$
itself, we make use of the mean-value theorem as follows
\begin{eqnarray*}
\frac{Z_n-V^{-1}(n)}{\sigma_n} &=&
\frac{V^{-1}(V(Z_n))-V^{-1}(n)}{\sigma_n}\\
&=& (V^{-1})'(\theta_n) \frac{V(Z_n)-n}{\sigma_n}
\end{eqnarray*}
where $\theta_n$ is sandwiched between $n$ and $V(Z_n)$.
Therefore, by the equality $V'=1/v$,
\begin{eqnarray*}
\frac{Z_n-V^{-1}(n)}{\sigma_n}
&=& v(V^{-1}(\theta_n)) \frac{V(Z_n)-n}{\sigma_n}.
\end{eqnarray*}
By the strong law of large numbers \eqref{as:slln.Z},
$\theta_n/n\to 1$ with probability $1$ as $n\to\infty$.
Therefore, $v(V^{-1}(\theta_n))/v(V^{-1}(n))\to 1$ and hence
\begin{eqnarray*}
\frac{Z_n-V^{-1}(n)}{\widetilde\sigma_n} &\Rightarrow& N_{0,1},
\end{eqnarray*}
where
\begin{eqnarray*}
\widetilde\sigma_n^2 &:=& \sigma_n^2 v^2(V^{-1}(n))\\
&=& b\sum_{j=0}^{n-1}\frac{v^2(V^{-1}(n))}{v^2(V^{-1}(j))}.
\end{eqnarray*}
The sequence $v^2(V^{-1}(j))$ is regularly varying with index
$-\displaystyle\frac{2\beta}{1+\beta}$, hence
\begin{eqnarray}\label{asym.for.sigma}
\sum_{j=0}^{n-1}\frac{v^2(V^{-1}(n))}{v^2(V^{-1}(j))}
&\sim& n^{-\frac{2\beta}{1+\beta}}
\sum_{j=0}^{n-1} j^{\frac{2\beta}{1+\beta}}\
\sim\ \frac{1+\beta}{1+3\beta}n\quad\mbox{ as }n\to\infty,
\end{eqnarray}
and the proof is complete.
\qed\end{proof}
\end{theopargself}

\index{Transience!central limit theorem}
\index{Central limit theorem!for Markov chains!with asymptotically zero drift}
\begin{theorem}\label{thm:clt.max}
Let the conditions of Theorem \ref{thm:clt.2} hold. Let
\begin{eqnarray}\label{con.on.beta}
\frac{\sqrt n v(V^{-1}(n))}{\log n} &\to& \infty\quad\mbox{as }n\to\infty.
\end{eqnarray}
Then
\begin{eqnarray*}
\frac{\max_{k\le n}X_k-V^{-1}(n)}{\sqrt{b\frac{1+\beta}{1+3\beta}n}}
&\Rightarrow& N_{0,1}
\quad\mbox{ as }n\to\infty.
\end{eqnarray*}
\end{theorem}

Since the function $v(V^{-1}(n))$ is regularly varying at infinity
with index $-\beta/(1+\beta)>-1/2$ provided $\beta\in[0,1)$,
the condition \eqref{con.on.beta} automatically holds for $\beta\in[0,1)$. 

\begin{proof}
It is again sufficient to prove the same result for the process
$\{Z_n\}$, that is, it is sufficient to show that, for $M_n:=\max_{k\le n}Z_k$,
\begin{eqnarray}\label{YAtoN.2}
\frac{M_n-V^{-1}(n)}{\sqrt{b\frac{1+\beta}{1+3\beta}n}}
&\Rightarrow& N_{0,1}
\quad\mbox{ as }n\to\infty.
\end{eqnarray}
Since $M_n\ge Z_n$, it suffices to show that, for all $\varepsilon>0$,
\begin{eqnarray*}
\P\{M_n\le Z_n+\varepsilon\sqrt n\} &\to& 1\quad\mbox{as }n\to\infty.
\end{eqnarray*}
Indeed,
\begin{eqnarray*}
\P\{M_n>Z_n+\varepsilon\sqrt n\} &\le& 
\P\{M_n\not\in[V^{-1}(n)/2,2V^{-1}(n)]\}\\
&&+\P\{Z_n<M_n-\varepsilon\sqrt n,\ M_n\in[V^{-1}(n)/2,2V^{-1}(n)]\}.
\end{eqnarray*}
Firstly, by the SLLN for $Z_n$, $M_n/V^{-1}(n)\to 1$ with probability 1, so
\begin{eqnarray*}
\P\{M_n\not\in[V^{-1}(n)/2,2V^{-1}(n)]\} &\to& 0\quad\mbox{as }n\to\infty.
\end{eqnarray*}
Secondly, 
\begin{eqnarray*}
\lefteqn{P\{Z_n<M_n-\varepsilon\sqrt n,\ M_n\in[V^{-1}(n)/2,2V^{-1}(n)]\}}\\
&&\hspace{10mm}\le\ \sum_{k=0}^{n-1} 
\P\{M_n=Z_k,\ Z_n<Z_k-\varepsilon\sqrt n,\ Z_k\in[V^{-1}(n)/2,2V^{-1}(n)]\}\\
&&\hspace{20mm}\le\ \sum_{k=0}^{n-1} 
\P\{Z_n<Z_k-\varepsilon\sqrt n,\ Z_k\in[V^{-1}(n)/2,2V^{-1}(n)]\}.
\end{eqnarray*}
Therefore,
\begin{eqnarray*}
\lefteqn{P\{Z_n<M_n-\varepsilon\sqrt n,\ M_n\in[V^{-1}(n)/2,2V^{-1}(n)]\}}\\
&&\hspace{8mm}\le\ \sum_{k=0}^{n-1} \int_{V^{-1}(n)/2}^{2V^{-1}(n)}
\P\{Z_k\in dy\}\P\{Z_n<y-\varepsilon\sqrt n\mid Z_k=y\}\\
&&\hspace{18mm}\le\ n\times \inf_{y\in[V^{-1}(n)/2,2V^{-1}(n)]}
\P\{Z_m<y-\varepsilon\sqrt n\mbox{ for some }m\ge 1\mid Z_0=y\}.
\end{eqnarray*}
The process $\{Z_n\}$ satisfies all the conditions of Theorem \ref{l:est.for.return}
with $r(x)=\nu(x)/b$, thus, for all $y\in[V^{-1}(n)/2,2V^{-1}(n)]$
\begin{eqnarray}\label{upper.clt.ret}
\P\{Z_m<y-\varepsilon\sqrt n\mid Z_0=y\}
&\le& e^{-\delta\int_{y-\varepsilon\sqrt n}^y\nu(z)dz}\nonumber\\
&\le& e^{-\delta\varepsilon\sqrt nv(y)}\nonumber\\
&\le& e^{-\delta\varepsilon\sqrt nv(2V^{-1}(n))},
\end{eqnarray}
because the function $\nu(z)$ is decreasing. 
Therefore, by the regular variation of $v$ and the condition \eqref{con.on.beta},
\begin{eqnarray*}
\inf_{y\in[V^{-1}(n)/2,2V^{-1}(n)]}\P\{Z_m<y-\varepsilon\sqrt n\mid Z_0=y\}
&=& o(1/n)\quad\mbox{as }n\to\infty,
\end{eqnarray*}
which yields
\begin{eqnarray*}
P\{Z_n<M_n-\varepsilon\sqrt n,\ M_n\in[V^{-1}(n)/2,2V^{-1}(n)]\}
&\to& 0\quad\mbox{as }n\to\infty.
\end{eqnarray*}
The proof is complete.
\qed\end{proof}

Recall that $T(x)=\min\{n:X_n>x\}$.
\index{Transience!central limit theorem}
\index{Central limit theorem!for Markov chains!with asymptotically zero drift}

\begin{corollary}\label{cor:clt.T}
Under the conditions of Theorem \ref{thm:clt.2} and \eqref{con.on.beta},
\begin{eqnarray*}
\frac{T(x)-V(x)}{\sqrt{b\frac{1+\beta}{1+3\beta}\frac{x}{v^3(x)}}}
&\Rightarrow& N_{0,1}
\quad\mbox{ as }x\to\infty.
\end{eqnarray*}
\end{corollary}

\begin{proof}
Since $\{T(x)\le n\} = \{\sup_{k\le n}X_k>x\}$,
\begin{eqnarray*}
\P\Biggl\{\frac{T(x)-V(x)}{\sqrt{b\frac{1+\beta}{1+3\beta}\frac{x}{v^3(x)}}}\le u\Biggr\}
&=& \P\Bigl\{\sup_{k\le n}X_k>x\Bigr\}
\end{eqnarray*}
where
\begin{eqnarray*}
n &:=& V(x)+u\sqrt{b\frac{1+\beta}{1+3\beta}\frac{x}{v^3(x)}}.
\end{eqnarray*}
Since $(V^{-1}(z))'=1/V'(V^{-1}(z))=v(V^{-1}(z))$ and $n\sim V(x)$,
$(V^{-1}(n))'\sim v(x)$. Therefore,
\begin{eqnarray*}
V^{-1}(n) &=& x+u\sqrt{b\frac{1+\beta}{1+3\beta}\frac{x}{v(x)}}+o(\sqrt{x/v(x)}).
\end{eqnarray*}
Hence,
\begin{eqnarray*}
\P\Bigl\{\sup_{k\le n}X_k>x\Bigr\} &=&
\P\Biggl\{\frac{\sup_{k\le n}X_k-V^{-1}(n)}{\sqrt{b\frac{1+\beta}{1+3\beta}n}}
>\frac{x-V^{-1}(n)}{\sqrt{b\frac{1+\beta}{1+3\beta}n}}\Biggr\}\\
&=& \P\Biggl\{\frac{\sup_{k\le n}X_k-V^{-1}(n)}{\sqrt{b\frac{1+\beta}{1+3\beta}n}}
>-u+o(1)\Biggr\},
\end{eqnarray*}
and reference to Theorem \ref{thm:clt.max} completes the proof.
\qed\end{proof}

\section{Functional central limit theorem}
\label{sec:normal.func}

In the last section we have proved the central limit theorem
for a transient Markov chain and the key idea of the proof
is extraction of a martingale for which the central
limit theorem is known from Brown \cite{Brown}, see Theorem \ref{thm:clt.mart}.
Since this reference also contains a functional version
of this result, it allows us to state and prove the following weak
convergence to a Gaussian process for $\{X_n\}$.
\index{Transience!convergence to!Brownian motion}

\begin{theorem}\label{thm:clt.2.func}
Under the conditions of Theorem \ref{thm:clt.2}, the process
\begin{eqnarray*}
\frac{X_{[nt]}-V^{-1}(nt)}{\sqrt{b\frac{1+\beta}{1+3\beta}n}},
\quad t\in[0,1],
\end{eqnarray*}
converges weakly in $D[0,1]$ as $n\to\infty$ to the process
\begin{eqnarray*}
t^{-\frac{\beta}{1+\beta}}B\Bigl(t^{\frac{1+3\beta}{1+\beta}}\Bigr),
\end{eqnarray*}
where $B(t)$ is a standard Brownian motion. The limiting process is
Gaussian with zero mean and covariance function
$t(t/s)^{\frac{\beta}{1+\beta}}$ for $s\ge t$.
\end{theorem}

\begin{proof}
Again as above, it suffices to prove the same convergence
for the process $Z_{[nt]}$.
The weak convergence in the space $D[0,1]$ to the limiting process
is equivalent to the following two statements: for any fixed $t_0\in(0,1)$,
\begin{eqnarray}\label{conv.nor.1}
\frac{Z_{[nt]}-V^{-1}(nt)}{\sqrt{b\frac{1+\beta}{1+3\beta}n}}
&\Rightarrow& t^{-\frac{\beta}{1+\beta}}B\Bigl(t^{\frac{1+3\beta}{1+\beta}}\Bigr)
\quad\mbox{as }n\to\infty\mbox{ in the space } D[t_0,1],
\end{eqnarray}
and
\begin{eqnarray}\label{conv.nor.2}
\sup_{t\le t_0}\ \biggl|
\frac{Z_{[nt]}-V^{-1}(nt)}{\sqrt{b\frac{1+\beta}{1+3\beta}n}}\biggr|
&\Rightarrow& 0\quad\mbox{as }n\to\infty,\ t_0\to 0.
\end{eqnarray}

The calculations of the last section leading to
the central limit theorem for $V(Z_n)$ allow us to apply
the functional limit theorem for martingales by Brown
\cite[Theorem 3]{Brown} to the process in $D[0,1]$ defined as
$(V(Z_k)-k)/\sigma_n$
on the interval $[\sigma_k^2/\sigma_n^2,\sigma_{k+1}^2/\sigma_n^2)$ where
\begin{eqnarray}\label{sigma2.equiv}
\sigma_n^2 &=& b\sum_{j=0}^{n-1}\frac{1}{v^2(V^{-1}(j))}
\ \sim\ b\frac{1+\beta}{1+3\beta}\frac{n}{v^2(V^{-1}(n))}
\quad\mbox{as }n\to\infty,
\end{eqnarray}
owing to \eqref{asym.for.sigma}. The process defined in this way
converges weakly in the space $D[0,1]$ to the Brownian motion, that is,
\begin{eqnarray}\label{Brown.cor}
\sum_{k=1}^n\frac{V(Z_k)-k}{\sigma_n}\I
\{\sigma_k^2/\sigma_n^2\le t<\sigma_{k+1}^2/\sigma_n^2\}
&\Rightarrow& B(t)\quad\mbox{in the space }D[0,1].\nonumber\\[-3mm]
\end{eqnarray}
The regular variation of $\sigma_n^2$ implies that
$$
\frac{\sigma_{[nt]}^2}{\sigma_n^2}\ \to\ t^{\frac{1+3\beta}{1+\beta}}
\quad\mbox{as }n\to\infty\mbox{ uniformly for all }t\in[t_0,1].
$$
Hence
\begin{eqnarray*}
\frac{V(Z_{[nt]})-nt}{\sigma_n}
&\Rightarrow& B\Bigl(t^{\frac{1+3\beta}{1+\beta}}\Bigr)
\ =\ t^{\frac{\beta}{1+\beta}}B(t)
\quad\mbox{in the space }D[t_0,1].
\end{eqnarray*}
Then we need to explain how to proceed from $V(Z_{[nt]})$
to $Z_{[nt]}$. By the mean-value theorem,
\begin{eqnarray*}
\frac{Z_{[nt]}-V^{-1}(nt)}{\sqrt{b\frac{1+\beta}{1+3\beta}n}} &=&
\frac{V^{-1}(V(Z_{[nt]}))-V^{-1}(nt)}{\sqrt{b\frac{1+\beta}{1+3\beta}n}}\\
&=& \frac{\sigma_n}{\sqrt{b\frac{1+\beta}{1+3\beta}n}}
(V^{-1})'(\theta) \frac{V(Z_{[nt]})-nt}{\sigma_n}
\end{eqnarray*}
where $\theta$ lies between $nt$ and $V(Z_{[nt]})$.
Therefore, by the equality $V'=1/v$,
\begin{eqnarray}\label{repr.for.Z}
\frac{Z_{[nt]}-V^{-1}(nt)}{\sqrt{b\frac{1+\beta}{1+3\beta}n}}
&=& \frac{\sigma_n}{\sqrt{b\frac{1+\beta}{1+3\beta}n}}
v(V^{-1}(\theta)) \frac{V(Z_{[nt]})-nt}{\sigma_n}.
\end{eqnarray}
It follows from the strong law of large numbers for $V(Z_n)$---see Theorem
\ref{th:slln}---that
$$
\frac{V(Z_{[nt]})}{n}\ \to\ t\quad\mbox{in the space }D[0,1],
$$
so $\theta/n\to t$ in $D[0,1]$ too.
Then, since $v$ is assumed regularly varying at infinity and
$v(V^{-1}(nt))/v(V^{-1}(n))\sim 1/t^{\beta/(1+\beta)}$ as $n\to\infty$,
$$
t^{\frac{\beta}{1+\beta}}\frac{v(V^{-1}(\theta))}{v(V^{-1}(n))}
\ \to\ 1\quad\mbox{in the space }D[t_0,1].
$$
Hence we may replace $v(V^{-1}(\theta))$ in \eqref{repr.for.Z}
by $t^{-\frac{\beta}{1+\beta}}v(V^{-1}(n))$ on the interval $t\in[t_0,1]$.
Taking into account \eqref{sigma2.equiv}, we deduce
the first required statement, \eqref{conv.nor.1}.

Further, the second statement, \eqref{conv.nor.2},
may be reformulated as, for all $\gamma>0$ and $\delta>0$
there exist $t_0>0$ and $n_0\in\N$ such that
\begin{eqnarray}\label{sup.V.gamma}
\P\biggl\{\sup_{t\le t_0}\biggl|
\frac{Z_{[nt]}-V^{-1}(nt)}{\sqrt{b\frac{1+\beta}{1+3\beta}n}}\biggr|
>\gamma\biggr\}
&\le& \delta\quad\mbox{for all }n>n_0.
\end{eqnarray}
Indeed, first choose $t_0$ such that
\begin{eqnarray*}
\P\Bigl\{\sup_{t\le t_0}|B(t)|>\gamma\Bigr\}
&\le& \delta/2.
\end{eqnarray*}
Then it follows from \eqref{Brown.cor} that there exists
an $n_0\in\N$ such that
\begin{eqnarray*}
\P\biggl\{\sup_{k:\sigma_k^2/\sigma^2_n\le t_0}
\biggl|\frac{V(Z_k)-k}{\sigma_n}\I
\{\sigma_k^2/\sigma_n^2\le t<\sigma_{k+1}^2/\sigma_n^2\}\biggr|>\gamma\biggr\}
&\le& \delta\quad\mbox{for all }n>n_0.
\end{eqnarray*}
Equivalently,
\begin{eqnarray*}
\P\biggl\{\sup_{k:\sigma_k^2/\sigma^2_n\le t_0}
\biggl|\frac{V(Z_k)-k}{\sigma_n}\biggr|>\gamma\biggr\}
&\le& \delta\quad\mbox{for all }n>n_0.
\end{eqnarray*}
If we take $k\le nt_0'$ where
$t_0'=t_0^{\frac{1+3\beta}{1+\beta}}/2$, then
$\sigma_k^2/\sigma^2_n\le t_0$ for all sufficiently large $n$. Therefore,
\begin{eqnarray*}
\P\biggl\{\sup_{t\le t_0'}
\biggl|\frac{V(Z_{[nt]})-nt}{\sigma_n}\biggr|>\gamma\biggr\}
&\le& \delta\quad\mbox{for all }n>n_0.
\end{eqnarray*}
Then we apply the same calculations as in \eqref{repr.for.Z}
and conclude \eqref{sup.V.gamma}.
\qed\end{proof}

\section{Normal approximation at high level}
\label{sec:normal.local}

In this section a version of the central limit theorem
is deduced for a Markov chain starting from a high level.
Such kind of normal approximation is more appropriate for
the purpose of proving asymptotics for renewal measure.

As in the last two sections we consider the case where the second 
truncated moment of jumps has a positive limit at infinity, 
so again the drift function $m_1^{[s(x)]}(x)$ and the quotient 
$2m_1^{[s(x)]}(x)/m_2^{[s(x)]}(x)$ are asymptotically proportional to each other. 
This allows us to choose a sufficiently small $\gamma>0$ such that
$$
r(x)\ :=\ \gamma v(x)
$$ 
makes the condition \eqref{r-cond.2} fulfilled for the chain $\{X_n\}$, 
for a sufficiently large $\widehat x$. 
Notice that $r(x)$ defined above satisfies the condition 
\eqref{rx.ge.1x.pr.eps} due to \eqref{vx.prime.v2} which now reads
\begin{eqnarray}\label{H.h.h}
v'(x) &=& o(v^2(x))\quad\mbox{as }x\to\infty,
\end{eqnarray}
which, in particular, specifies the insensitivity condition \eqref{vx.2.cv} as follows, 
for any fixed $c<\infty$,
\begin{eqnarray}\label{vx.2.cv.h}
v(x\pm c/v(x)) &\sim& v(x)\quad\mbox{as }x\to\infty.
\end{eqnarray}

In the previous sections we apply a convex function $V$
to $X_n$ in order to get a chain with an asymptotically (positive)
constant drift which helps us to prove the law of large numbers
and the central limit theorem. For the purposes of this 
section---normal approximation at high level $x$---it is more convenient 
to make calculations for $\{X_n\}$ itself because the drift of $\{X_n\}$ does not 
change much on time scale $O(1/v^2(x))$, due to \eqref{vx.2.cv.h},
provided the drift is proportional to $v(x)$.
\index{Transience!central limit theorem}
\index{Central limit theorem!for Markov chains!with asymptotically zero drift}

\begin{theorem}\label{l:clt}
Let, for some increasing function $s(x)=o(1/v(x))$ where 
a decreasing function $v(x)$ satisfies $xv(x)\to\infty$, 
\eqref{vx.2.cv.h} and \eqref{H.h.h},
\begin{eqnarray}\label{1.2}
m^{[s(x)]}_1(x) \sim v(x) &\mbox{and}& m^{[s(x)]}_2(x)\to b>0,\\
\label{rec.3.1.E}
\E\{|\xi(x)|;\ \xi(x)\le-s(x)\} &=& o(v(x))\quad\mbox{as }x\to\infty,\\
\label{rec.3.1}
\P\{|\xi(x)|>s(x)\} &\le& p(x)v(x),
\end{eqnarray}
where a decreasing function $p(x)>0$ is integrable at infinity.
%Let, for all $\varepsilon>0$,
%\begin{eqnarray}\label{1.2.lind}
%\E\{\xi^2(x);\ \varepsilon/v(x)\le|\xi(x)|\le s(x)\} &\to& 0\quad\mbox{as }x\to\infty.
%\end{eqnarray}
Then, for any fixed $t>1$ and $h\in\R$, 
\begin{eqnarray*}
\P_x\Bigl\{X_n-x\le\frac{h}{v(x)}\Bigr\}
-\Phi\biggl(\frac{h/v(x)-nv(x)}{\sqrt{nb}}\biggr) &\to& 0
\end{eqnarray*}
as $x$, $n\to\infty$ in such a way that $1/t\le nv^2(x)\le t$;
hereinafter $\Phi$ stands for the standard normal distribution function.
Moreover,
\begin{eqnarray*}
\sup_{x\le y\le x+o(1/v(x))}
\biggl|\P_y\Bigl\{X_n-x\le\frac{h}{v(x)}\Bigr\}
-\Phi\biggl(\frac{h/v(x)-nv(x)}{\sqrt{nb}}\biggr)\biggr| &\to& 0.
\end{eqnarray*}
\end{theorem}

%Notice that the conditions \eqref{rec.3.1.E}--\eqref{1.2.lind} hold
%for some increasing function $s(x)$ of order $o\bigl(\sqrt{x/v(x)}\bigr)$
%provided the family $|\xi(x)|$, $x\ge 0$, possesses a
%majorant $\Xi$, that is, $|\xi(x)|\le_{st}\Xi$ for all $x$,
%which is square integrable, $\E\Xi^2<\infty$;
%see Corollary \ref{cor:maj.p.e.g}.

We start with the following tightness result for $\{X_n\}$.

\begin{lemma}\label{l:lln}
Let, for some increasing function $s(x)=o(1/v(x))$ where 
a decreasing function $v(x)$ satisfies $xv(x)\to\infty$, 
\eqref{vx.2.cv.h} and \eqref{H.h.h},
\begin{eqnarray}\label{rec.1.1}
\delta v(x)\ \le\ m_1^{[s(x)]}(x) &\le& v(x)/\delta,
\end{eqnarray}
for some $\delta>0$ and all sufficiently large $x$. Assume also
\begin{eqnarray}\label{mom.cond.ui.clt}
\sup_x m_2^{[s(x)]}(x) &<& \infty,
\end{eqnarray}
and that the conditions \eqref{rec.3.1} and \eqref{rec.3.1.E} hold.
Then, for every fixed $t>0$ and $\varepsilon>0$,
there exists an $h<\infty$ such that
\begin{eqnarray*}
\P_x\Bigl\{-\frac{h}{v(x)}\le X_n-x\le \frac{h}{v(x)}
\mbox{ for all }n\le\frac{t}{v^2(x)}\Bigr\}
&\ge& 1-\varepsilon
\end{eqnarray*}
for all sufficiently large $x$.
\end{lemma}

\begin{proof}
As above, we consider a modified Markov chain $\{\widetilde X_n\}$ on the same probability 
space as $\{X_n\}$ with jumps $\widetilde\xi(x)=\xi(x)\I\{|\xi(x)|\le s(x)\}$,
and, as explained there, we can assume that $\{\widetilde X_n\}$ 
satisfies the unboundedness of trajectories condition \eqref{eq:irreducibility.lln}.

As discussed at the beginning of the section the chain $\{\widetilde X_n\}$ satisfies
the condition \eqref{r-cond.2} with $r(x):=\gamma v(x)$. 
Therefore, Theorem \ref{thm:Hy.above} is applicable to $\{\widetilde X_n\}$, hence
\begin{eqnarray*}
H_y^{\widetilde X}(x,x+1/v(x)] &\le& \frac{c_1}{v^2(x)},
\end{eqnarray*}
which in its turn allows us to apply Lemma \ref{l:XY.equiv} to a pair of
the chains $Y=X$ and $Z=\widetilde X$. 
Hence it suffices to prove the result of the lemma for $\{Z_n\}$.
That is, it is sufficient to show that, for a sufficiently large $h>0$,
\begin{eqnarray}\label{clt.2}
\P_x\Bigl\{-\frac{h}{v(x)}\le Z_n-x\le \frac{h}{v(x)}
\mbox{ for all }n\le\frac{t}{v^2(x)}\Bigr\}
&\ge& 1-\varepsilon
\end{eqnarray}
ultimately in $x$.

Similarly to \eqref{upper.clt.ret} we deduce that, for some $\gamma>0$,
\begin{eqnarray}\label{Y.to.infty.x2}
\P_x\Bigl\{\min_{n\ge 0}Z_n\le x-\frac{h}{v(x)}\Bigr\}
\ \le\ e^{-\gamma h} &\to& 0
\quad\mbox{as }x,\ h\to\infty.
\end{eqnarray}

Let us center $Z_n$, that is, let us consider the process
\begin{eqnarray}\label{X.purefied}
\widetilde Z_n &:=& Z_n-x-\sum_{j=0}^{n-1} m_1^{[s(Z_j)]}(Z_j),
\end{eqnarray}
which constitutes a martingale with respect
to $\mathcal F_n:=\sigma(Z_0,\ldots,Z_n)$. 
By the condition \eqref{rec.1.1}, we have, for all $N\ge 1$,
\begin{eqnarray*}
0\ <\ \sum_{n=0}^{N-1}m_1^{[s(Z_j)]}(Z_n)
&\le& N\frac{1}{\delta}\max_{z>x-h/v(x)} v(z)
\ \le\ N\frac{2}{\delta} v(x)
\end{eqnarray*}
on the event $\min_n Z_n>x-h/v(x)$, where the last inequality
follows for all sufficiently large $x$ from \eqref{vx.2.cv.h}. 
Hence, for any $y>0$,
\begin{eqnarray*}
\P_x\Bigl\{\min_n Z_n>x-\frac{h}{v(x)},\ \max_{n\le N}|Z_n-x|>y\Bigr\}
&\le& \P_x\Bigl\{\max_{n\le N}|\widetilde Z_n|>y-\frac{2}{\delta}Nv(x)\Bigr\}.
\end{eqnarray*}
By Doob's inequality for martingales,
\begin{eqnarray*}
\P_x\Bigl\{\max_{n\le N}|\widetilde Z_n|>y-\frac{2}{\delta} Nv(x)\Bigr\}
&\le& \frac{\E_x\widetilde Z_N^2}{(y-2Nv(x)/\delta)^2}.
\end{eqnarray*}
The second moments of jumps of the martingale $\{\widetilde Z_n\}$ are bounded
by some $c<\infty$---see the condition \eqref{mom.cond.ui.clt}; therefore,
\begin{eqnarray*}
\P_x\Bigl\{\max_{n\le N}|\widetilde Z_n|>y-\frac{2}{\delta} Nv(x)\Bigr\}
&\le& \frac{Nc}{(y-2Nv(x)/\delta)^2}.
\end{eqnarray*}
Taking now $N=t/v^2(x)$ and $y=h/v(x)$, we obtain that
\begin{eqnarray*}
\P_x\Bigl\{\max_{n\le N}|\widetilde Z_n|>y-\frac{2}{\delta} Nv(x)\Bigr\}
&\le& \frac{tc}{(h-2t/\delta)^2}\ \le\ \frac{\varepsilon}{2},
\end{eqnarray*}
for all sufficiently large $h$. Therefore,
\begin{eqnarray*}
\P_x\Bigl\{\min_n Z_n>x-\frac{h}{v(x)},\
\max_{n\le t/v^2(x)}|Z_n-x|>\frac{h}{v(x)}\Bigr\}
&\le& \frac{\varepsilon}{2},
\end{eqnarray*}
which together with \eqref{Y.to.infty.x2}
completes the proof of \eqref{clt.2}.
\qed\end{proof}

The proof of Theorem \ref{l:clt} is based on the following generalisation 
of the central limit theorem 
to a triangular array of martingales which goes back to \cite[Theorem 4]{GSS}.
\index{Central limit theorem!for triangular array of martingales}

\begin{theorem}\label{thm:clt.mart.array}
Let, for all $j\ge 1$, $\{\mathcal F_{n,j},n\ge 1\}$ 
be a filtration and $\{X_{n,j},n\ge 1\}$ be a square integrable martingale 
with respect to $\{\mathcal F_{n,j}\}$.
Let $n_j\to\infty$ as $j\to\infty$,
$$
\frac{\sum_{k=1}^{n_j} \E\{(X_{k+1,j}-X_{k,j})^2\mid \mathcal F_{k,j}\}}{\E X_{n_j,j}^2}
\ \stackrel{p}\to\ 1\quad\mbox{as }j\to\infty
$$
and conditioned Lindeberg condition hold: for all $\varepsilon>0$,
$$
\frac{1}{\E X_{n_j,j}^2}  \sum_{k=1}^{n_j} 
\E\{(X_{k+1,j}-X_{k,j})^2
\I\{|X_{k+1,j}-X_{k,j}|>\varepsilon\sqrt{\E X_{n_j,j}^2}\}
\mid \mathcal F_{k,j}\}
\ \stackrel{p}\to\ 0\quad \mbox{as }j\to\infty.
$$
Then $X_{n_j,j}/\sqrt{\E X_{n_j,j}^2}$ converges weakly to
a standard normal distribution as $j\to\infty$.
\end{theorem}

\begin{theopargself}
\begin{proof}[of Theorem \ref{l:clt}]
As shown in Lemma \ref{l:lln}, it suffices to prove the same result
for the chain $\{Z_n\}$, that is, it is sufficient to prove that
\begin{eqnarray}\label{YAtoGamma}
\frac{Z_n-x-nv(x)}{\sqrt{nb}} &\Rightarrow& N_{0,1}.
\end{eqnarray}
as $x$, $n\to\infty$ in such a way that $1/t\le nv^2(x)\le t$.

Since the chain $\{Z_n\}$ satisfies all the conditions of Lemma \ref{l:lln},
for any function $h(x)\to\infty$, given $Z_0=x$,
\begin{eqnarray}\label{YA.compact}
\P_x\Bigl\{-\frac{h(x)}{v(x)}\le Z_n-x\le \frac{h(x)}{v(x)}
\mbox{ for all }n\le\frac{t}{v^2(x)}\Bigr\}
&\to& 1\quad\mbox{as }x\to\infty.
\end{eqnarray}

The process $\{\widetilde Z_n\}$ defined in \eqref{X.purefied} constitutes
a martingale---parameterised by $x$---whose second moment of jumps
converges to $b$ as $x\to\infty$. Due to the construction of jumps
of $\{Z_n\}$ and $s(x)=o(1/v(x))$ we get, for any $\varepsilon>0$,
\begin{eqnarray*}
\E\{\xi^2(x);\ s(x)\ge|\xi(x)|\ge \varepsilon\sqrt n\} &\to& 0
\quad\mbox{as }x,\ n\to\infty
\end{eqnarray*}
in such a way that $1/t\le nv^2(x)\le t$. Together with
\eqref{Y.to.infty.x2} this implies that, for the same range of $x$ and $n$,
\begin{eqnarray*}
\E\{(Z_{k+1}-Z_k)^2;\ |Z_{k+1}-Z_k|\ge \varepsilon\sqrt n
\mid Z_0\ge x\} &\to& 0
\quad\mbox{uniformly for }k\le n.
\end{eqnarray*}
These observations guarantee that conditioned Lindeberg condition of
the central limit theorem for martingales in triangular array setting---see
Theorem \ref{thm:clt.mart.array}---is met for $\widetilde Z_n$ and $n$
satisfying $1/t\le nv^2(x)\le t$. Then we conclude that, given $Z_0=x$,
the random sequence
\begin{eqnarray*}
\frac{\widetilde Z_n}{\sqrt{nb}}
&=& \frac{Z_n-x-\sum_{j=0}^{n-1} m_1^{[s(Z_j)]}(Z_j)}{\sqrt{nb}}
\end{eqnarray*}
converges weakly as $x$, $n\to\infty$ to a standard normal distribution.

Let us choose $h(x)\to\infty$ sufficiently slow such that 
in the spatial range $x-h(x)/v(x)\le y\le x+h(x)/v(x)$ we have
$$
v(y)\sim m_1^{[s(y)]}(y)\sim m_1^{[s(x)]}(x)\sim v(x)\quad\mbox{as }x\to\infty,
$$
which is possible due to \eqref{vx.2.cv.h}.
Then within the temporal range $n\le t/v^2(x)$, 
we deduce from \eqref{YA.compact} that
\begin{eqnarray*}
\frac{\sum_{j=0}^{n-1} m_1^{[s(Z_j)]}(Z_j)}{\sqrt{nb}}-\sqrt{n/b} v(x)
&\stackrel{p}{\to}& 0 \quad\mbox{as }x\to\infty.
\end{eqnarray*}
Therefore,
\begin{eqnarray*}
\frac{Z_n-x}{\sqrt{nb}}-\sqrt{n/b}v(x) &=& \frac{Z_n-x-nv(x)}{\sqrt{nb}}
\end{eqnarray*}
converges weakly as $x\to\infty$ to a standard normal distribution
and the proof of the first result is complete. 

For the second statement, the same arguments with minor modification
apply to show that, for any sequences $x_k$ and $y_k$ 
such that $y_k\ge x_k$, $y_k-x_k=o(1/v(x_k))$ it holds true that
\begin{eqnarray*}
\P_{y_k}\Bigl\{X_n-x_k\le\frac{h}{v(x_k)}\Bigr\} -
\Phi\biggl(\frac{h/v(x_k)-nv(x_k)}{\sqrt{nb}}\biggr) &\to& 0
\end{eqnarray*}
as $k$, $n\to\infty$ in such a way that $1/t\le nv^2(x_k)\le t$.
Then the second statement is immediate, by contradiction.
\qed\end{proof}
\end{theopargself}

\section{Integro-local renewal theorem for transient chain with Normal limit}
\sectionmark{Integro-local renewal theorem with Normal limit}
\label{sec:renewal.normal}

In this section we discuss asymptotics of partial and full renewal
measure for $X$ with normal limit. We pay special attention to the fact
that both are investigated under truncation at level $s(x)=o(1/v(x))$.
\index{Renewal theorem!for normal convergence}

\begin{theorem}\label{thm:renewal}
Under the conditions of Theorem \ref{l:clt}, for every fixed $h>0$ and $B>0$,
\begin{eqnarray*}
\sum_{n=0}^{[B/v^2(x)]}\P_y\Bigl\{X_n\in\Bigl(x,x+\frac{h}{v(x)}\Bigr]\Bigr\}
&\sim& \frac{f(h,B)}{v^2(x)}
\end{eqnarray*}
as $x\to\infty$ uniformly for all $y\in[x,x+o(1/v(x))]$,
where $f(h,B)\uparrow h$ as $B\to\infty$.
\end{theorem}

\begin{proof}
Due to the normal approximation provided by Theorem \ref{l:clt}
we conclude that, for every fixed $B$,
\begin{eqnarray*}
\lefteqn{\sum_{n=0}^{[B/v^2(x)]}
\P_y\Bigl\{X_n\in\Bigl(x,x+\frac{h}{v(x)}\Bigr]\Bigr\}}\\
&&\hspace{30mm}=\ \sum_{n=0}^{[B/v^2(x)]}\Bigl(\Phi\Bigl(\frac{h-nv^2(x)}{\sqrt{nbv^2(x)}}\Bigr)
-\Phi\Bigl(-\frac{nv^2(x)}{\sqrt{nbv^2(x)}}\Bigr)+o(1)\Bigr)
\end{eqnarray*}
as $x\to\infty$ uniformly for all $y\in[x,x+o(1/v(x))]$.
Approximating the sum on the right by the integral we
obtain that its value is equal to
\begin{eqnarray}\label{limit.for.sum}
\frac{1}{v^2(x)}\int_0^B \Bigl(\Phi\Bigl(\frac{h-z}{\sqrt{bz}}\Bigr)
-\Phi\Bigl(-\frac{z}{\sqrt{bz}}\Bigr)\Bigr)dz
+o\Bigl(\frac{1}{v^2(x)}\Bigr)\quad\mbox{as }x\to\infty.
\end{eqnarray}
The last integral equals
\begin{eqnarray*}
f(h,B)\ =\ \int_0^B \Bigl(\Phi\Bigl(\frac{h-z}{\sqrt{bz}}\Bigr)
-\Phi\Bigl(-\frac{z}{\sqrt{bz}}\Bigr)\Bigr)dz &=&
\int_0^B \frac{dz}{\sqrt{bz}}\int_0^h\varphi\Bigl(\frac{u-z}{\sqrt{bz}}\Bigr)du.
\end{eqnarray*}
Changing the order of integration and making the substitution $z=v^2/b$, 
we obtain equalities
\begin{eqnarray*}
\frac{1}{\sqrt{2\pi}}\int_0^h du
\int_0^B \frac{1}{\sqrt{bz}}e^{-(u-z)^2/2bz}dz
&=& \frac{1}{\sqrt{2\pi}}\int_0^h e^{u/b} du
\int_0^B \frac{1}{\sqrt{bz}}e^{-u^2/2bz-z/2b}dz\\
&=& \frac{2}{b\sqrt{2\pi}}\int_0^h e^{u/b} du
\int_0^{\sqrt{bB}} e^{-u^2/2v^2-v^2/2b^2}dv.
\end{eqnarray*}
The limit of the internal integral as $B\to\infty$
is known---see, e.g \cite[p. 337, 3.325]{GR}---and is nothing else but
\begin{eqnarray*}
\int_0^\infty e^{-u^2/2v^2-v^2/2b^2}dv &=& \frac{b\sqrt{2\pi}}{2} e^{-u/b}.
\end{eqnarray*}
Combining altogether we deduce that
\begin{eqnarray*}
\int_0^\infty \Bigl(\Phi\Bigl(\frac{h-z}{\sqrt{bz}}\Bigr)
-\Phi\Bigl(-\frac{z}{\sqrt{bz}}\Bigr)\Bigr)dz &=& h.
\end{eqnarray*}
Together with \eqref{limit.for.sum} this implies the result.
\qed\end{proof}

%In order to derive 
Now let us turn to the asymptotic behaviour of the renewal measure.
%we strengthen conditions that guarantee results of Theorem \ref{thm:renewal}
%by narrowing possible range for the truncation level $s(x)$.
\index{Renewal theorem!for normal convergence}

\begin{theorem}\label{thm:renewal+}
Under the conditions of Theorem \ref{l:clt}, 
for every fixed $h>0$ and distribution of $X_0$,
\begin{eqnarray*}
H\Bigl(x,x+\frac{h}{v(x)}\Bigr] &\sim& \frac{h}{v^2(x)}
\quad\mbox{ as }x\to\infty.
\end{eqnarray*}
\end{theorem}

\begin{proof}
We consider the same function $r(x)$ as in the proof of Lemma \ref{l:lln},
so the conditions \eqref{r-cond.2} and \eqref{rx.ge.1x.pr.eps} 
are satisfied for all sufficiently large $x$.

We split the proof of the integro-local asymptotics for $H$
into two parts, upper and lower bounds.
First let us prove a proper upper bound.
By the Markov property it is sufficient to show that,
uniformly for all $y>x$,
\begin{eqnarray}\label{bound.Hy.upper}
\limsup_{x\to\infty} v^2(x)H_y\Bigl(x,x+\frac{h}{v(x)}\Bigr] &\le& h.
\end{eqnarray}
The chain $\{X_n\}$ satisfies all the conditions of Theorem \ref{thm:Hy.above}.
Then, for any $A>h$, by the Markov property and \eqref{y.back.x},
\begin{eqnarray}\label{estimate.for.Hy.2.1x.pre.new}
\lefteqn{H_y\Bigl(x,x+\frac{h}{v(x)}\Bigr]}\nonumber\\
&\le& \E_y\sum_{n=0}^{T\bigl(x+\frac{A}{v(x)}\bigr)-1}
\I\Bigl\{X_n\in\Bigl(x,x+\frac{h}{v(x)}\Bigr]\Bigr\}\nonumber\\
&&\hspace{10mm}+\P\Bigl\{X_n\le x+\frac{h}{v(x)}
\mbox{ for some }n\mid X_0>x+\frac{A}{v(x)}\Bigr\}
\sup_z H_z\Bigl(x,x+\frac{h}{v(x)}\Bigr]\nonumber\\
&\le& \E_y\sum_{n=0}^{T\bigl(x+\frac{A}{v(x)}\bigr)-1}
\I\Bigl\{X_n\in\Bigl(x,x+\frac{h}{v(x)}\Bigr]\Bigr\}\nonumber\\
&&\hspace{15mm}
+\Bigl(e^{\delta\bigl(R\bigl(x+\frac{h}{v(x)}\bigr)-R\bigl(x+\frac{A}{v(x)}\bigr)\bigr)}
+o(1)\Bigr)
\sup_z H_z\Bigl(x,x+\frac{h}{v(x)}\Bigr]
\end{eqnarray}
as $x\to\infty$ uniformly for all $A>h$ where a stopping time
$T$ is defined as
\begin{eqnarray*}
T(t) &:=& \min\{n\ge 1:X_n>t\}.
\end{eqnarray*}
We have
\begin{eqnarray*}
e^{\delta\bigl(R\bigl(x+\frac{h}{v(x)}\bigr)-R\bigl(x+\frac{A}{v(x)}\bigr)\bigr)}
&=& e^{-\delta\int_{x+h/v(x)}^{x+A/v(x)} r(y)dy} \\
&\le& e^{-\delta (A-h)r(x+A/v(x))/v(x)}\le e^{-\delta(A-h)/2},
\end{eqnarray*}
for all sufficiently large $x$.
Applying the upper bound of Theorem \ref{thm:Hy.above} to
the right hand side of \eqref{estimate.for.Hy.2.1x.pre.new} we deduce that,
for some $c<\infty$,
\begin{eqnarray}\label{estimate.for.Hy.2.1x.n}
H_y\Bigl(x,x+\frac{h}{v(x)}\Bigr] &\le&
\E_y\sum_{n=0}^{T\bigl(x+\frac{A}{v(x)}\bigr)-1}
\I\Bigl\{X_n\in\Bigl(x,x+\frac{h}{v(x)}\Bigr]\Bigr\}\nonumber\\
&&\hspace{30mm}+ \bigl(e^{-\delta(A-h)/2}+o(1)\bigr)\frac{c}{v^2(x)}
\end{eqnarray}
as $x\to\infty$, for all $A>h$.
The mean of the sum on the right hand side may be estimated as follows: for $C>A$,
\begin{eqnarray*}
\lefteqn{\E_y\sum_{n=0}^{T\bigl(x+\frac{A}{v(x)}\bigr)-1}
\I\Bigl\{X_n\in\Bigl(x,x+\frac{h}{v(x)}\Bigr]\Bigr\}}\\
&&\hspace{10mm}\le \E_y\sum_{n=0}^{[C/v^2(x)]}
\I\Bigl\{X_n\in\Bigl(x,x+\frac{h}{v(x)}\Bigr]\Bigr\}\\
&&\hspace{20mm}+\E_y\biggl\{\sum_{n=0}^{T\bigl(x+\frac{A}{v(x)}\bigr)-1}
\I\{X_n>x\};\
T\Bigl(x+\frac{A}{v(x)}\Bigr)>\frac{C}{v^2(x)}\biggr\}\\
&&\hspace{10mm}= \E_y\sum_{n=0}^{[C/v^2(x)]}
\I\Bigl\{X_n\in\Bigl(x,x+\frac{h}{v(x)}\Bigr]\Bigr\}\\
&&\hspace{20mm}+\E_y\biggl\{L\Bigl(x,T\Bigl(x+\frac{A}{v(x)}\Bigr)\Bigr);\
T\Bigl(x+\frac{A}{v(x)}\Bigr)>\frac{C}{v^2(x)}\biggr\}.
\end{eqnarray*}
%where $\widehat x$ is chosen so large that
%\begin{eqnarray*}
%\E\{\xi(x);\ \xi(x)\le s(x)\} &\ge& v(x)/2\quad\mbox{for all }x\ge\widehat x,
%\end{eqnarray*}
%which is possible due to the conditions \eqref{1.2} and \eqref{rec.3.1.E}.
For $y>x$, the second term on the right hand side is not greater than
\begin{eqnarray*}
\lefteqn{\E_y\biggl\{L\Bigl(x,T\Bigl(x+\frac{A}{v(x)}\Bigr)\Bigr);\
X_n<x-\frac{D}{v(x)}\mbox{ for some }n\ge 1\biggr\}}\\
&& +\E_y\biggl\{L\Bigl(x,T\Bigl(x+\frac{A}{v(x)}\Bigr)\Bigr);\\
&&\hspace{10mm}X_n\ge x-\frac{D}{v(x)}\mbox{ for all }n\le T\Bigl(\frac{A}{v(x)}\Bigr)-1,\
T\Bigl(x+\frac{A}{v(x)}\Bigr)>\frac{C}{v^2(x)}\biggr\}\\
&&\le\ \E_y\biggl\{L\Bigl(x,T\Bigl(x+\frac{A}{v(x)}\Bigr)\Bigr);\
X_n<x-\frac{D}{v(x)}\mbox{ for some }n\ge 1\biggr\}\\
&&\hspace{10mm} +\E_y\biggl\{L\Bigl(x-\frac{D}{v(x)},T\Bigl(x+\frac{A}{v(x)}\Bigr)\Bigr);\
L\Bigl(x-\frac{D}{v(x)},T\Bigl(x+\frac{A}{v(x)}\Bigr)\Bigr)>\frac{C}{v^2(x)}\biggr\}.
\end{eqnarray*}
Fix an $\varepsilon>0$.
By Theorem \ref{l:uniform}, for any fixed $A$ and $D$, the family of random variables
$$
v^2(x)
L\Bigl(x-\frac{D}{v(x)},T\Bigl(x+\frac{A}{v(x)}\Bigr)\Bigr),\quad x\le y,\ X_0=y,
$$
is uniformly integrable, hence, there is a $C=C(A,D)$ such that
\begin{eqnarray*}
\sup_{y:y\ge x}
v^2(x)\E_y\biggl\{L\Bigl(x-\frac{D}{v(x)},T\Bigl(x+\frac{A}{v(x)}\Bigr)\Bigr);\
L\Bigl(x-\frac{D}{v(x)},T\Bigl(x+\frac{A}{v(x)}\Bigr)\Bigr)>\frac{C}{v^2(x)}\biggr\}
&\le& \varepsilon,
\end{eqnarray*}
for all sufficiently large $x$. Since
\begin{eqnarray*}
\sup_{y>x}\P_y\biggl\{X_n<x-\frac{D}{v(x)}\mbox{ for some }n\ge 1\biggr\}
&\to& 0\quad\mbox{as }D\to\infty,
\end{eqnarray*}
by the uniform integrability that there exists a $D=D(A)$ such that
\begin{eqnarray*}
\sup_{y\ge x}
v^2(x)\E_y\biggl\{L\Bigl(x,T\Bigl(x+\frac{A}{v(x)}\Bigr)\Bigr);\
X_n<x-\frac{D}{v(x)}\mbox{ for some }n\ge 1\biggr\}
&\le& \varepsilon,
\end{eqnarray*}
for all sufficiently large $x$. 
Combining altogether we conclude that, uniformly for all $y\in(x,h/v(x)]$,
\begin{eqnarray*}
\lefteqn{\limsup_{x\to\infty} v^2(x)
\E_y\sum_{n=0}^{T\bigl(x+\frac{A}{v(x)}\bigr)-1}
\I\Bigl\{X_n\in\Bigl(x,x+\frac{h}{v(x)}\Bigr]\Bigr\}}\\
&&\hspace{20mm} \le\ \limsup_{x\to\infty} v^2(x)\E_y\sum_{n=0}^{[C/v^2(x)]}
\I\Bigl\{X_n\in\Bigl(x,x+\frac{h}{v(x)}\Bigr]\Bigr\}+2\varepsilon,
\end{eqnarray*}
which being substituted into \eqref{estimate.for.Hy.2.1x.n} gives
\begin{eqnarray*}
\limsup_{x\to\infty}v^2(x)H_y\Bigl(x,x+\frac{h}{v(x)}\Bigr]
&\le& \limsup_{x\to\infty}v^2(x)\E_y\sum_{n=0}^{[C/v^2(x)]}
\I\Bigl\{X_n\in\Bigl(x,x+\frac{h}{v(x)}\Bigr]\Bigr\}\\
&&\hspace{20mm} +ce^{-\delta(A-h)/2}+2\varepsilon.
\end{eqnarray*}
As already shown in Theorem \ref{thm:renewal},
\begin{eqnarray*}
v^2(x)\sum_{n=0}^{[C/v^2(x)]}
\P_y\Bigl\{X_n\in\Bigl(x,x+\frac{h}{v(x)}\Bigr]\Bigr\}
&\to& f(h,C)\quad\mbox{as }x\to\infty,
\end{eqnarray*}
which implies the following upper bound, for each fixed $A>1$,
\begin{eqnarray*}
\limsup_{x\to\infty}v^2(x)H_y\Bigl(x,x+\frac{h}{v(x)}\Bigr]
&\le& f(h,C)+ce^{-\delta(A-h)/2}+2\varepsilon,
\end{eqnarray*}
where $C=C(A,D(A))$.
Letting now $A\to\infty$, we get the required upper bound \eqref{bound.Hy.upper}.

Now let us proceed with the lower bound. First notice that,
by Theorem \ref{thm:renewal},
\begin{eqnarray}\label{lower.for.H}
\liminf_{x\to\infty}v^2(x) H_y\Bigl(x,x+\frac{h}{v(x)}\Bigr] &\ge& h
\end{eqnarray}
as $x\to\infty$ uniformly for all $y\in[x,x+o(1/v(x))]$.
It remains to prove that \eqref{lower.for.H} holds for any fixed $y$.
By the Markov property, it suffices to show that the overshoot over
the level $x$ is less than $s(x)$ with high probability, that is,
\begin{eqnarray}\label{overshoot}
\P_y\{X_{T(x)}-x>s(x)\} &\to& 0\quad\mbox{as }x\to\infty.
\end{eqnarray}
Indeed, for any fixed $x_0>0$,
\begin{eqnarray*}
\P_y\{X_{T(x)}-x>s(x)\} &\le&
\sum_{n=1}^\infty \int_0^x\P_y\{X_n\in dz\}\P\{z+\xi(z)>x+s(x)\}\\
&=& \biggl(\int_0^{x_0}+\int_{x_0}^x\biggr)
\P\{z+\xi(z)>x+s(x)\} H_y(dz).
\end{eqnarray*}
The first integral on the right hand side is bounded by
\begin{eqnarray*}
\int_0^{x_0} \P\{\xi(z)>s(x)\} H_y(dz) &\to& 0\quad\mbox{as }x\to\infty,
\end{eqnarray*}
due to the dominated convergence theorem.
Since $x-z+s(x)\ge s(z)$ for all $z\le x$, it follows from the condition \eqref{rec.3.1}
that the second integral is dominated by
\begin{eqnarray*}
\int_{x_0}^x \P\{\xi(z)>s(z)\} H_y(dz) &\le&
\int_{x_0}^\infty p(z)v(z) H_y(dz)\ \to\ 0
\quad\mbox{as }x_0\to\infty,
\end{eqnarray*}
see the calculations leading to \eqref{int.prH.fin}.
Altogether yields the convergence \eqref{overshoot} for the overshoot.
This concludes the proof.
\qed\end{proof}

Theorem \ref{thm:renewal} and the proof of Theorem \ref{thm:renewal+}
imply the following result.
\index{Renewal theorem!for normal convergence}

\begin{theorem}\label{thm:renewal++}
Under the conditions of Theorem \ref{l:clt}, for every fixed $h>0$,
\begin{eqnarray*}
\sum_{k=0}^n \P_y\Bigl\{X_k\in\Bigl(x,x+\frac{h}{v(x)}\Bigr]\Bigr\}
&=& \frac{1}{v^2(x)}f(h,nv^2(x))+o\Bigl(\frac{1}{v^2(x)}\Bigr)
\quad\mbox{as }x\to\infty
\end{eqnarray*}
uniformly for all $y\in[x,x+o(1/v(x))]$
and for all $n\ge 1$, where $f(h,z)\uparrow h$ as $z\to\infty$.
\end{theorem}

\begin{theorem}\label{thm:renewal.clt}
Under the conditions of Theorems \ref{thm:clt.2} and \ref{thm:renewal+},
given any distribution of $X_0$ and any fixed $h>0$,
\begin{eqnarray}\label{partial.sum.normal}
\sum_{k=0}^n\P\Bigl\{X_k\in\Bigl(x,x+\frac{h}{v(x)}\Bigr]\Bigr\}
&=& \frac{h}{v^2(x)}\Phi\Biggl(\frac{n-V(x)}{\sqrt{b\frac{1+\beta}{1+3\beta}\frac{x}{v^3(x)}}}
\Biggr)+o\Bigl(\frac{1}{v^2(x)}\Bigr)\nonumber\\[-3mm]
\end{eqnarray}
as $x\to\infty$ uniformly for all $n\ge 1$.
\end{theorem}

\begin{proof}
We have
\begin{eqnarray*}
\sum_{k=0}^n\P\Bigl\{X_k\in\Bigl(x,x+\frac{h}{v(x)}\Bigr]\Bigr\}
&=& \E\sum_{k=T(x)}^n\I\Bigl\{X_k\in\Bigl(x,x+\frac{h}{v(x)}\Bigr]\Bigr\}.
\end{eqnarray*}
As \eqref{overshoot} shows, $v(x)(X_{T(x)}-x)\to 0$ in probability.
This allows us to apply Theorem \ref{thm:renewal++}: as $x\to\infty$,
\begin{eqnarray*}
\E\sum_{k=T(x)}^n\I\Bigl\{X_k\in\Bigl(x,x+\frac{h}{v(x)}\Bigr]\Bigr\}
&=& \frac{1}{v^2(x)}\E f\bigl(h,v^2(x)(n-T(x))^+\bigr)+o\Bigl(\frac{1}{v^2(x)}\Bigr).
\end{eqnarray*}
Further, fix $u\in\R$ and take
$$
n=V(x)+u\sqrt{b\frac{1+\beta}{1+3\beta}\frac{x}{v^3(x)}}.
$$
Then
\begin{eqnarray*}
v^2(x)(n-T(x))^+ &=& \sqrt{b\frac{1+\beta}{1+3\beta}xv(x)}
\frac{(n-T(x))^+}{\sqrt{b\frac{1+\beta}{1+3\beta}\frac{x}{v^3(x)}}}\\
&=& \sqrt{b\frac{1+\beta}{1+3\beta}xv(x)}
\Biggl(u+\frac{V(x)-T(x)}{\sqrt{b\frac{1+\beta}{1+3\beta}\frac{x}{v^3(x)}}}\Biggr)^+.
\end{eqnarray*}
Since $xv(x)\to\infty$, the last quantity tends to infinity with probability
$$
\P\Biggl\{\frac{V(x)-T(x)}{\sqrt{b\frac{1+\beta}{1+3\beta}\frac{x}{v^3(x)}}}>-u\Biggr\}
\ \to\ \Phi(u)\quad\mbox{as }x\to\infty,
$$
and equals zero with probability going to $1-\Phi(u)$,
both by Corollary \ref{cor:clt.T}.
Taking into account that $f(h,z)\to h$ as $z\to\infty$, we conclude that
\begin{eqnarray*}
\E f\bigl(h,v^2(x)(n-T(x))^+\bigr) &\to& h\Phi(u)
\quad\mbox{as }x\to\infty,
\end{eqnarray*}
which completes the proof.
\qed\end{proof}

\section{Local renewal theorem for transient chain on $\Z$ with Normal limit}
\sectionmark{Local renewal theorem}
\label{sec:loc.renewal.n}

In this section we formulate and prove a local version of the renewal
theorem in the case of convergence to a
normal distribution. Following the technique developed so far,
we can only do this for a lattice Markov chain.
Without loss of generality, let $\Z$ be the minimal lattice where $X$
is living on. Similarly to the case of convergence
to a $\Gamma$-distribution, it is unlikely that the local renewal theorem
would be valid if we only assumed a regular asymptotic behaviour
of moments of jumps. We believe it can be only proven
if we assume weak convergence of jumps $\xi(x)$
to some random variable $\xi$ on $\Z$, that is,
\begin{equation}\label{weak.lim.n}
\xi(x)\Rightarrow\xi\quad\mbox{as } x\to\infty.
\end{equation}

\index{Local renewal theorem for!normal limit}
\begin{theorem}\label{thm:srt.n}
Let $v(x)$ be a decreasing differentiable function satisfying
$xv(x)\to\infty$ and $v'(x)=o(v^2(x))$ and let
\begin{eqnarray}\label{1.2+.n}
m_1(x) \sim v(x) \ &\mbox{ and }&\ m_2(x)\to b>0
\quad\mbox{as }x\to\infty,
\end{eqnarray}
and
$$
\limsup_{n\to\infty}X_n=\infty\quad\mbox{with probability }1.
$$
Furthermore we assume the convergence \eqref{weak.lim.n}.
Let $\Z$ be the minimal lattice for $\xi$, and let the limit $\xi$ satisfy
\begin{equation}\label{lim.mom.n}
\E\xi=0,\quad \E\xi^2=b.
\end{equation}
In addition, let the jumps $\xi(x)$ be bounded below and above
by $J$ uniformly for all $x\in\Z^+$, that is,
\begin{equation}\label{cond.J.n}
|\xi(x)|\ \le\ J\ \mbox{ for all }x\in\Z^+.
\end{equation}
Then
\begin{eqnarray}\label{ren.loc.h.n}
h(x):=H\{x\} &\sim& \frac{1}{v(x)}\quad\mbox{as }x\to\infty.
\end{eqnarray}
Moreover,
\begin{eqnarray}\label{ren.loc.h.1.n}
\P\Bigl\{\sum_{n=0}^\infty\I\{X_n=x\}>N\Bigr\}
&=& c_1(x)(1-c_2(x)v(x))^N,
\end{eqnarray}
where $c_1(x)$, $c_2(x)\to 1$ as $x\to\infty$,
hence the family of random variables
\begin{equation}\label{ren.loc.h.2.n}
v(x)\sum_{n=0}^\infty \I\{X_n=x\},\quad x\in\{1,2,3,\ldots\},
\end{equation}
is uniformly integrable.
\end{theorem}

More general results are derived in Chapter \ref{ch:asy.renewal},
via different technique based on the martingale approach.

\begin{proof}
Let $\delta>0$ and define two decreasing functions
$$
U_\pm(x)\ :=\ \int_x^\infty e^{-R_\pm(y)}dy,\quad x>0,
$$
where
$$
R_\pm(y)\ :=\ \frac{2\pm\delta}{b}\int_0^y v(z)dy.
$$
By the mean value theorem, for all $x$ and $j\in\Z$
there is a $\theta=\theta(j,x)\in(0,1)$ such that
\begin{eqnarray*}
U_\pm(x+j)-U_\pm(x) &=& -je^{-R_\pm(x+\theta j)}
\ \sim\ -je^{-R_\pm(x)}\quad\mbox{as }x\to\infty,
\end{eqnarray*}
because, for any fixed $u>0$,
$$
|R_\pm(x+u)-R_\pm(x)|\ \le\ \frac{2+\delta}{b}u v(x)\ \to\ 0
\quad\mbox{as }x\to\infty,
$$
due to $v(x)\to 0$. By L'H\^{o}pital's rule,
\begin{eqnarray*}
\lim_{x\to\infty}\frac{U_\pm(x)}{\frac{1}{v(x)}e^{-R_\pm(x)}}
&=& \lim_{x\to\infty}\frac{U_\pm'(x)}{\bigl(\frac{1}{v(x)}e^{-R_\pm(x)}\bigr)'}\\
&=& \lim_{x\to\infty}\frac{e^{-R_\pm(x)}}
{\bigl(\frac{v'(x)}{v^2(x)}+\frac{2\pm\delta}{b}\bigr)e^{-R_\pm(x)}}
\ =\ \frac{b}{2\pm\delta},
\end{eqnarray*}
owing to the condition $v'(x)=o(v^2(x))$. Therefore,
$$
U_\pm(x+j)-U_\pm(x)\ \sim\ -j\frac{2\pm\delta}{b}v(x)U_\pm(x).
$$
Then, since $\xi(x)$ are bounded below, we get for all fixed $k\ge 1$ that
\begin{eqnarray}\label{srt.1.n}
\nonumber
&&\E_{x+k}\left\{U_\pm(X_{\tau(x)})-U_\pm(x+k);\ \tau(x)<\infty\right\}\\
&&\hspace{5mm}\sim\ \frac{2\pm\delta}{b}v(x)U_\pm(x+k)
\E_{x+k}\{x+k-X_{\tau(x)};\ \tau(x)<\infty\},
\end{eqnarray}
where
$$
\tau(x)\ :=\ \min\{n\ge1:X_n\le x\}.
$$
Let us compute the drift of $U_\pm(X_n)$.
Since the jumps are bounded, by Taylor's expansion,
\begin{eqnarray*}
\lefteqn{\E(U_\pm(x+\xi(x))-U_\pm(x))}\\
&&\hspace{10mm}=\ U_\pm'(x)m_1(x)+\frac12 m_2(x)U_\pm''(x)m_2(x)
+O(U_\pm'''(x))\\
&&\hspace{10mm}=\ -e^{-R_\pm(x)}m_1(x)
+\frac{1\pm\delta/2}{b}v(x)e^{-R_\pm(x)}m_2(x)
+O\bigl(v^2(x)e^{-R_\pm(x)}\bigr)\\
&&\hspace{10mm}\sim\ \pm(\delta/2+o(1))v(x)e^{-R_\pm(x)}
\quad\mbox{as }x\to\infty.
\end{eqnarray*}

Therefore, the sequence $U_-(X_{n\wedge\tau(x)})$ is a supermartingale
for all sufficiently large $x$. Then, by the optional stopping theorem,
$$
\E_{x+k}\{U_-(X_{\tau(x)});\ \tau(x)<\infty\}\ \le\ U_-(x+k).
$$
This is equivalent to
$$
\E_{x+k}\left\{U_-(X_{\tau(x)})-U_-(x+k);\ \tau(x)<\infty\right\}
\ \le\ U_-(x+k)\P_{x+k}\{\tau(x)=\infty\}.
$$
Using now \eqref{srt.1.n}, we get
\begin{equation}\label{srt.2.n}
\P_{x+k}\{\tau(x)=\infty\}\ \ge\ \frac{2-2\delta}{b}v(x)
\E_{x+k}\{x+k-X_{\tau(x)};\ \tau(x)<\infty\}.
\end{equation}
Since $U_+(X_{n\wedge\tau(x)})$ is a submartingale
for all sufficiently large $x$,
$$
\E_{x+k}\{U_+(X_{\tau(x)});\ \tau(x)<\infty\}\ \ge\ U_+(x+k).
$$
This implies that
\begin{equation*}
\P_{x+k}\{\tau(x)=\infty\}\ \le\
\frac{2+2\delta}{b}v(x)\E_{x+k}\{x+k-X_{\tau(x)};\ \tau(x)<\infty\}.
\end{equation*}
Combining this lower bound with \eqref{srt.2.n} and
due to the arbitrary choice of $\delta>0$, we conclude that
\begin{equation}\label{srt.3.n}
\P_{x+k}\{\tau(x)=\infty\}\ =\
\frac{2+o(1)}{b}v(x)\E_{x+k}\{x+k-X_{\tau(x)};\ \tau(x)<\infty\}.
\end{equation}
The rest of the proof is literally almost the same
as that of Theorem \ref{thm:srt}.
\qed\end{proof}

\section{Comments to Chapter \ref{ch:transient.2}}

The weak law of large numbers in the form of \eqref{lln.X.} was originally proven
by Lamperti\index{Lamperti} in \cite[Theorem 7.1]{Lamp62}
under the condition that the fourth moment of jumps is bounded
and the drift is of order $\theta/x^\beta$, $\beta\in(0,1)$.
His proof is based on the method of moments as everything else in that paper.

The strong law of large numbers in the form of \eqref{slln.X.} 
for a nearest neighbour Markov chain was proven by Voit\index{Voit}
in \cite[Theorem 2.11]{Voit1992} via an orthogonal polynomials technique.

Various laws of large numbers---both weak and strong---and
central limit theorems were proven
by Keller,\index{Keller}
Kersting and Rosler\index{Rosler} \cite{KKR} under minimal moment condition
on positive part of jumps---the existence of square integrable majorant---and
under assumption that jumps are bounded below.
Strong law of large numbers under minimal moment condition
was proven by Kersting\index{Kersting} in \cite{Ker92}.

In \cite[Theorem 2.3]{MW2010}, Menshikov\index{Menshikov} and
Wade\index{Wade} have proved the strong law of large numbers 
in the form of \eqref{slln.X.} under the assumption
that moments of jumps of order $2+2\beta+\delta$, $\delta>0$, are bounded.
In the same paper, the authors have proved the central limit theorem like
Theorem \ref{thm:clt.2} for drift proportional to $1/x^\beta$
under the assumption that jumps have moments of order
$$
\max\Bigl(2+2\beta,1+\frac{2}{1+\beta}\Bigr)
$$
bounded.
\chapter{Asymptotics for renewal measure for transient Markov chain 
via martingale approach}
\chaptermark{Asymptotics for renewal measure}% for transient and null-recurrent chains}
\label{ch:asy.renewal}

For a transient Markov chain $\{X_n\}$ on $\R$ with asymptotically zero drift,
the average time spent by $\{X_n\}$ in the interval $(x,x+1]$
is roughly speaking the reciprocal of the drift
and tends to infinity as $x$ grows.

In this chapter we present a general approach relying on diffusion 
approximation to prove renewal theorems for Markov chains, for that reason 
we consider Markov chains which may be approximated by diffusion process. 
Then, if we have some result of renewal type for diffusion processes
as in Section \ref{subsec:diffusion.tr},
we should be able to obtain a similar result for a Markov chain
having similar asymptotic behaviour of the first two moments of jumps.  
In particular, we will see in the examples below that as soon as we have 
the Green function for the diffusion process we should, in principle, 
be able to construct an approximation for the Green function of the 
Markov chain  and thus to derive a renewal theorem.

We apply a martingale type technique and show that the asymptotic behaviour
of the renewal measure heavily depends on the rate at which the drift vanishes.
As in the last two chapters, 
the two main cases are distinguished, either the drift of the chain decreases as
$1/x$ or much slower than that, say as $1/x^\alpha$ for some $\alpha\in(0,1)$.
In contrast to the case of asymptotically positive drift 
considered in Chapter \ref{ch:asymp.hom},
the case of vanishing drift is quite tricky for the analysis due to 
the fact that the Markov chain tends to infinity rather slowly 
and hence one should take into account diffusion fluctuations.

\section{Asymptotics for renewal measure on growing intervals}
\sectionmark{Asymptotics for renewal measure on growing intervals}
\label{sec:renewal.growing}

Throughout this chapter we assume that the trajectories of $\{X_n\}$ are unbounded,
that is, 
\begin{equation}\label{eq:irreducibility}
\limsup_{n\to\infty} X_n = \infty\ \mbox{ a.s. }
\end{equation}
This condition holds true for any irreducible Markov chain on $\Zp$,
because such a chain stays at any finite collection of states only
finite time, with probability $1$.
\index{Renewal measure!asymptotics on growing intervals}

\begin{theorem}\label{thm:regular}
Let $\{X_n\}$ be such that~\eqref{eq:irreducibility} holds and 
\begin{eqnarray}\label{m1.m2.1x}
m_1^{[s(x)]}(x)\sim\frac{\mu}{x},&&\quad m_2^{[s(x)]}(x)\to b \in(0,\infty)
\quad\mbox{as }x\to\infty, 
\end{eqnarray}
for some $\mu>b/2$ and an increasing level $s(x)$
of order $o(x)$. Assume also that, 
\begin{eqnarray}\label{regular_left_tail}
\P\{|\xi(y)|\ge s(y)\} &\le& p(y)/y,  
\end{eqnarray}
for some decreasing integrable at infinity function $p(x)$, and
\begin{eqnarray}\label{majorant_third}
|\xi(y)| \I\{|\xi(y)|\le s(y)\} &\le_{st}& 
\widehat\xi\quad\mbox{for all }y\ge 0,
\end{eqnarray}
where 
\begin{equation}\label{majorant_third_moment_exists}
\E\widehat\xi^2<\infty. 
\end{equation}	
Then, for every function $h(x)\uparrow\infty$ of order $o(x),$ 
we have
$$
H(x,x+h(x)]\sim\frac{2}{2\mu-b}xh(x)\quad\mbox{as }x\to\infty.
$$
\end{theorem}

Notice that both conditions~\eqref{regular_left_tail} 
and~\eqref{majorant_third} are met for some $s(x)=o(x)$ if $|\xi(y)|\le_{st}\widehat\xi$
for all $y$ and for some $\widehat\xi$ satisfying~\eqref{majorant_third_moment_exists}.

In the course of the proof of this and subsequent theorems we construct 
a bounded non-negative supermartingale, which shows that $X_n\to\infty$ a.s. 
This convergence means transience of any set bounded on the right.

We now turn to the critical case $\mu=b/2$ where the properties 
of the chain---particularly recurrence and transience---depend 
on further terms in  asymptotic expansions for the moments of increments. 
As the next theorem shows this is also true for the renewal function.
\index{Renewal measure!asymptotics on growing intervals}

\begin{theorem}\label{thm:critical}
Let $\{X_n\}$ be such that~\eqref{eq:irreducibility} holds
and that there exist $m\ge1$, $\gamma>0$ and an increasing level
$s(x)$ of order $o(x)$ such that
\[
\frac{2m_1^{[s(x)]}(x)}{m_2^{[s(x)]}(x)}=\frac{1}{x}+\frac{1}{x\log x}
+\ldots+\frac{1}{x\log x\cdot\ldots\cdot\log_{(m-1)}x}
+\frac{\gamma+1+o(1)}{x\log x\cdot\ldots\cdot\log_{(m)}x}
\]
and $m_2^{[s(x)]}(x)\to b>0$ as $x\to\infty$. 
Assume that, for some $\varepsilon>0$,
\begin{eqnarray}\label{critical_left_tail}
\P\{|\xi(x)|>s(x)\} &=& o(1/x^2\log^{2+\varepsilon}x),\\
\label{critical_m3}
\E\{|\xi(x)|^3;\ |\xi(x)|\le s(x)\} &=& o(x/\log^{1+\varepsilon}x),\\ 
\label{majorant_third_critical}
|\xi(y)| \I\{|\xi(y)|<s(y)\} &\le_{st}& \widehat\xi,
\end{eqnarray}
where $\widehat\xi$ satisfies \eqref{majorant_third_moment_exists}.
Then, for every function $h(x)\uparrow\infty$ of order $o(x)$, we have
\begin{eqnarray*}
H(x,x+h(x)] &\sim& \frac{2h(x)}{b\gamma}x\log x\cdot\ldots\cdot\log_{(m)}x
\quad\mbox{as }x\to\infty.
\end{eqnarray*}
\end{theorem}

The proof of the integral renewal theorem in the case $\mu>b/2$ in 
Section \ref{sec:renewal} is based on the convergence of $X_n^2/n$ towards a
$\Gamma$-distribution. This approach is not applicable under the conditions
of Theorem~\ref{thm:critical}, although the convergence to 
a $\Gamma$-distribution is still valid. 
The reason is that some chains with $\mu=b/2$ are null-recurrent 
while other are transient, but this difference disappears in the weak limit. 
The only statement which can be obtained from weak convergence here
is the following lower bound:
\[
\lim_{x\to\infty}\frac{H(0,x]}{x^2}=\infty
\]

In the next theorem we consider the case where the drift 
decreases slower than $1/x$, that is, $m_1(x)x\to\infty$.
\index{Renewal measure!asymptotics on growing intervals}

\begin{theorem}\label{thm:weibull}
Let $\{X_n\}$ be such that~\eqref{eq:irreducibility} holds 
and that there exist a decreasing $v(x)$ satisfying $xv(x)\to\infty$ 
and $v'(x)=o(v^2(x))$ and an increasing level $s(x)=o(1/v(x))$ such that 
\[
m_1^{[s(x)]}(x)\sim v(x),\quad m_2^{[s(x)]}(x)\to b \in(0,\infty)
\quad\mbox{as }x\to\infty.
\]
Assume also that
\begin{align}\label{regular_left_tail_weibull}
&\P\{|\xi(y)|\ge s(y)\}\le p(y)v(y),\\
\label{majorant_third_weibull}
& \xi(y) \I\{|\xi(y)|<s(y) \} \le_{st} \widehat\xi
\quad \mbox{for all }y\ge 0,  
\end{align}
where $p(x)$ is a non-increasing, non-negative integrable at infinity function,	
and $\widehat\xi$ satisfies \eqref{majorant_third_moment_exists}.
Then, for every function $h(x)\uparrow\infty$ of order $o(1/v(x))$, we have
$$
H(x,x+h(x)]\sim\frac{h(x)}{v(x)}\quad\mbox{as }x\to\infty.
$$
\end{theorem}

In the two examples---nearest neighbour Markov chain and diffusion 
process---considered in Subsections \ref{subsec:nnmc.trans} and
\ref{subsec:diffusion.tr} it is possible to construct 
an appropriate martingale which allows us to find the renewal measure in a closed form. 
For general Markov chains
considered in the last three theorems, 
this martingale approach does not work because it is hopeless to construct 
such a martingale. 
However, it is possible to construct almost 
a martingale that allows us to derive the asymptotic behaviour
of the renewal measure; it is done in Section~\ref{sec:ilrtgi}.

\section{Proof of integro-local renewal theorem on growing intervals}
\label{sec:ilrtgi}

Let $r(x)$ be a decreasing differentiable function on $[0,\infty)$
satisfying the condition
\begin{eqnarray}\label{r.prime}
r'(x) &=& O(r^2(x))\quad\mbox{as }x\to\infty,
\end{eqnarray}
in the sequel $r(x)$ approximates the quotient 
$2m_1^{[s(x)]}(x)/m_2^{[s(x)]}(x)$. 
We shall impose assumptions on the truncated moments of Markov chains,
and doing that we always assume that the truncation function $s(x)$ 
increases and satisfies 
$$
s(x)=o\left(1/r(x)\right)\quad\mbox{as }x\to\infty.
$$
Define $R(x)=0$ for $x\le 0$,
\begin{align}\label{eq_u_x}
R(x)&:=\int_0^x r(y) dy,\quad x>0,\qquad
U(x)\ :=\ \int_x^\infty e^{-R(z)} dz,\quad x\in\R,
\end{align}
where $U(x)$ is assumed finite,
compare to $U$ defined in \eqref{def:U.dif}. Clearly, 
\begin{equation*}
\frac{U''(x)}{U'(x)} = -r(x).
\end{equation*}
Due to \eqref{r.prime},
\begin{equation}\label{eq.r.insensitivity}
r(x+y)\sim r(x),\quad R(x+y)-R(x)\to 0,
\quad\text{and}\quad e^{-R(x+y)}\sim e^{-R(x)}
\end{equation}
as $x\to\infty$ uniformly for $|y|\le s(x)$. Also,
\begin{equation}\label{U3}
U'''(x)=(r^2(x)-r'(x))e^{-R(x)}=O\bigl(r^2(x)e^{-R(x)}\bigr) 
\end{equation}
and, consequently, 
\begin{eqnarray}\label{U.3.uni}
U'''(x+y) &=& O\bigl(r^2(x)e^{-R(x)}\bigr)\quad\mbox{as }x\to\infty
\mbox{ uniformly for }|y|\le s(x).
\end{eqnarray}

Let 
$$
G(y)\ :=\ U(0)-U(y)\ =\ \int_0^y e^{-R(z)} dz.
$$
We start with a result showing that $G(X_n)$
is almost a martingale provided the quotient $2m_1^{[s(x)]}(x)/m_2^{[s(x)]}(x)$
is asymptotically proportional to $r(x)$. 
 
\begin{lemma}\label{lem:upper.U}
Let $\theta(y)$ be a non-negative bounded function. Let
\begin{eqnarray}\label{sigma_uniform2}
\E\{|\xi(y)|^3;\ |\xi(y)|\le s(y)\} &=&
o\bigl(m_2^{[s(y)]}(y)\theta(y)/r(y)\bigr)\quad\mbox{as }y\to\infty.
\end{eqnarray} 
{\rm (i)} If
\begin{eqnarray}\label{cond.tail.left}
\P\{\xi(y)<-s(y)\} &=& 0\quad\mbox{for all }y\ge 0,
\end{eqnarray}
and
\begin{equation}\label{m1_below}
\frac{2m_1^{[s(y)]}(x)}{m_2^{[s(y)]}(y)} \ge (1+\theta(y))r(y) 
\quad\mbox{for all sufficiently large }y,
\end{equation}  
then there exists a $y^*>0$ such that 
\begin{eqnarray*}	
\E\{G(y+\xi(y))-G(y);\ \xi(y)\le s(y)\} &\ge& 0\quad\mbox{for all }y>y^*.
\end{eqnarray*}
{\rm (ii)} If
\begin{eqnarray}\label{cond.tail.right}
\P\{\xi(y)>s(y)\} &=& 0\quad\mbox{for all }y\ge 0,
\end{eqnarray}
and
\begin{equation}\label{m1_above}
\frac{2m_1^{[s(y)]}(x)}{m_2^{[s(y)]}(y)} \le (1-\theta(y))r(y) 
\quad\mbox{for all sufficiently large }y,
\end{equation}  
then there exists a $y^*>0$ such that 
\begin{eqnarray*}	
\E\{G(y+\xi(y))-G(y);\ \xi(y)\ge -s(y)\} &\le& 0\quad\mbox{for all }y>y^*.
\end{eqnarray*}
\end{lemma}	

\begin{proof}
(i) Since the function $G(y)$ is increasing,
\begin{eqnarray*}
\E G(y+\xi(y))-G(y)
&\ge& \E\{G(y+\xi(y))-G(y);\ |\xi(y)|\le s(y)\},
\end{eqnarray*}
due to the condition \eqref{cond.tail.left}.
Since $G'(y)=e^{-R(y)}$, $G''(y)=-r(y)e^{-R(y)}$, and
$G'''(y+z)=O(r^2(y))e^{-R(y)}$ as $y\to\infty$ uniformly for all $|z|\le s(y)$
due to the upper bound \eqref{U.3.uni} on $U'''$ and \eqref{eq.r.insensitivity},
application of Taylor's expansion up to the third derivative yields that,
for some $\gamma=\gamma(y,\xi(y))\in[0,1]$,
\begin{eqnarray*}
\lefteqn{\E\{G(y+\xi(y))-G(y);\ |\xi(y)|\le s(y)\}}\\
&=& m_1^{[s(y)]}(y)G'(y)+\frac12 m_2^{[s(y)]}(y)G''(y)\\
&&\hspace{30mm}+ \frac16 \E\{\xi^3(y)G'''(y+\gamma\xi(y));\ |\xi(y)|\le s(y)\}\\
&=& m_1^{[s(y)]}(y)e^{-R(y)}
-\frac12 m_2^{[s(y)]}(y)r(y)e^{-R(y)}\\
&&\hspace{30mm}+ O\Bigl(r^2(y)e^{-R(y)} 
\E\{|\xi^3(y)|;\ |\xi(y)|\le s(y)\}\Bigr)\quad\mbox{as }y\to\infty.
\end{eqnarray*}
The sum of the first two terms on the right hand side equals
\begin{eqnarray*}
\frac12 e^{-R(y)}\bigl(2m_1^{[s(y)]}(y)-m_2^{[s(y)]}(y) r(y)\bigr)
&\ge& \frac12 e^{-R(y)} m_2^{[s(y)]}(y)\theta(y)r(y),
\end{eqnarray*}
due to the condition \eqref{m1_below}. 
The third term on the right hand side of the previous equation
is of order $o\bigl(m_2^{[s(y)]}(y)\theta(y)r(y)e^{-R(y)}\bigr)$
owing to the condition \eqref{sigma_uniform2}.
These observations conclude the proof of (i).

(ii) Since the function $G(y)$ is increasing,
\begin{eqnarray*}
\E G(y+\xi(y))-G(y)
&\le& \E\{G(y+\xi(y))-G(y);\ |\xi(y)|\le s(y)\},
\end{eqnarray*}
due to the condition \eqref{cond.tail.right}.
The rest of the proof is very similar to part (i).
\qed\end{proof}

\subsection{Upper bound}
\label{sec:upper}

Our derivation of an upper bound for the renewal measure of $\{X_n\}$ is based 
on the Lyapunov function $G^{**}_{h,x}(y)$ defined below in \eqref{eq_gh}.

For any $x$ and $h>0$, consider a piecewise differentiable function
\begin{equation}\label{eq_first_derivative_gh}
g^{**}_{h,x}(y) := 
\begin{cases}
0,& y\le x,\\
2(y-x),& y\in (x,x+h],\\
2h,& y\in (x+h,x+h+s(x+h)],\\
2he^{R(x+h+s(x+h))-R(y)},& y>x+h+s(x+h),
\end{cases}
\end{equation}
whose derivative satisfies
\begin{equation}\label{eq_first_derivative_gh.der}
g_{h,x}^{**\prime}(y)\ =\ 2\I\{y\in[x,x+h]\}\quad\mbox{for all }y<x+h+s(x+h),
\ y\not= x,x+h.
\end{equation}
Its integral---the function which originates from
the key function \eqref{G.dif} for diffusion processes,
\begin{equation}\label{eq_gh}
G^{**}_{h,x}(y) := \int_0^y g^{**}_{h,x}(z)dz,
\end{equation}
is an increasing bounded function, $G^{**}_{h,x}(\infty)<\infty$, because 
\begin{equation}\label{g**.le.G}
g^{**}_{h,x}(y)\ \le\ 2he^{R(x+h+s(x+h))-R(y)}\quad\mbox{for all }y,
\end{equation}
and hence,
\begin{eqnarray}\label{G.upper.at.infty}
G^{**}_{h,x}(\infty) &\le& 2h \int_x^\infty e^{R(x+h+s(x+h))-R(y)}dy\nonumber\\
&=& 2h e^{R(x+h+s(x+h))}U(x)\nonumber\\
&\le& 2h U(x) e^{R(x+h)+R'(x+h)s(x+h)}\nonumber\\
&\le& 2h U(x) e^{R(x+h)+r(x+h)s(x+h)},
\end{eqnarray}
because $R$ is concave. As $s(x)=o(1/r(x))$, 
\begin{eqnarray}\label{G.upper.at.infty.ap}
G^{**}_{h,x}(\infty) &\le& 2h U(x) e^{R(x+h)+o(1)}\nonumber\\
&\le& 2h U(x) e^{R(x)+o(1)}\quad\mbox{as }x\to\infty,
\end{eqnarray}
for $h\le s(x)$, due to \eqref{eq.r.insensitivity}.

The function $G^{**}_{h,x}(y)$ is convex for $y\le x+h$.
For $y>x+h$, the function $G^{**}_{h,x}(y)$ increases
in a concave way with slope $2h$ at point $x+h$. 
Notice that, for $y>x+h+s(x+h)$ and $z>0$,
\begin{eqnarray*}
G^{**}_{h,x}(y+z)-G^{**}_{h,x}(y) &=& 2he^{R(x+h+s(x+h))}(G(y+z)-G(y))
\end{eqnarray*}
and, due to \eqref{g**.le.G}, for $y>x+h+s(x+h)$ and $z\le 0$, 
\begin{eqnarray*}
G^{**}_{h,x}(y+z)-G^{**}_{h,x}(y) &\ge& 2he^{R(x+h+s(x+h))}(G(y+z)-G(y)).
\end{eqnarray*}
Therefore, for all $y>x+h+s(x+h)$ and $z\in\R$
\begin{eqnarray}\label{lower.G.1}
G^{**}_{h,x}(y+z)-G^{**}_{h,x}(y) &\ge& 2he^{R(x+h+s(x+h))}(G(y+z)-G(y)).
\end{eqnarray}
Further, for $y\in(x+h,x+h+s(x+h)]$,
\begin{equation*}
g^{**}_{h,x}(y+z)\ \ge\ 2he^{R(y)-R(y+z)}\quad\mbox{for }z>0,
\end{equation*}
and
\begin{equation*}
g^{**}_{h,x}(y+z)\ \le\ 2h\ \le\ 2he^{R(y)-R(y+z)}\quad\mbox{for }z\le 0.
\end{equation*}
Therefore, for $y\in(x+h,x+h+s(x+h)]$, 
\begin{eqnarray}\label{lower.G.2}
G^{**}_{h,x}(y+z)-G^{**}_{h,x}(y) &\ge& 2he^{R(y)}(G(y+z)-G(y)).
\end{eqnarray}

\begin{lemma}\label{lem:upper}
Assume that the conditions \eqref{sigma_uniform2}--\eqref{m1_below} hold.
Then there exists an $x^*>0$ such that, 
for all $x>x^*$, $y\ge 0$, $h\le s(x)$, and $t\in(0,h/2)$,
\begin{eqnarray}\label{eq_gh_super}	
\E G^{**}_{h,x}(y+\xi(y))-G^{**}_{h,x}(y)
&\ge& m_2^{[t]}(y)\I\{y\in [x+t,x+h-t]\}.
\end{eqnarray}
\end{lemma}	

\begin{proof}
Since the function $G^{**}_{h,x}(y)$ is zero for $y\le x$ and positive for $y>x$, 
the mean drift of $G^{**}_{h,x}$ is non-negative for all $y\in[0,x]$ and  
the inequality \eqref{eq_gh_super} follows for this range of $y$. 

Since $G^{**}_{h,x}(y)$ is increasing and due to \eqref{cond.tail.left},
\begin{eqnarray*}
\E G^{**}_{h,x}(y+\xi(y))-G^{**}_{h,x}(y)
&\ge& \E\{G^{**}_{h,x}(y+\xi(y))-G^{**}_{h,x}(y);\ |\xi(y)|\le s(y)\}
\ =:\ E.
\end{eqnarray*}
Positivity of $E$ for $y>x+h$ follows from \eqref{lower.G.1} and
\eqref{lower.G.2}, by Lemma \ref{lem:upper.U}.

Thus, it remains to estimate $E$ from below for $y\in[x,x+h]$.
By Taylor's expansion for $G^{**}_{h,x}$ with integral remainder term,
\begin{eqnarray}\label{E.Taylor}	
E &=& m_1^{[s(y)]}(y)g^{**}_{h,x}(y)
+\E\Bigl\{\int_y^{y+\xi(y)}g^{**\prime}_{h,x}(z)(y+\xi(y)-z)dz;\ |\xi(y)|\le s(y)\Bigr\}.
\nonumber\\[-2mm]
\end{eqnarray}
Since $g^{**}_{h,x}(z)\ge 0$ and $g_{h,x}^{**\prime}(z)\ge 0$ 
for all $z\in[0,x+h+s(x+h)]$,
we obtain for all sufficiently large $x$ and $y\in[x,x+h]$, $t\in(0,h/2)$,
\begin{eqnarray*}	
E &\ge& \E\Bigl\{\int_y^{y+\xi(y)}g^{**\prime}_{h,x}(z)(y+\xi(y)-z)dz;\ |\xi(y)|\le t\Bigr\}\\
&\ge& 2\I\{y\in [x+t,x+h-t]\}
\E\Bigl\{\int_y^{y+\xi(y)}(y+\xi(y)-z)dz;\ |\xi(y)|\le t\Bigr\}\\
&=& m_2^{[t]}(y)\I\{y\in [x+t,x+h-t]\},
\end{eqnarray*}
because $g_{h,x}^{**\prime}(z)=2$ for all $z\in(x,x+h]$ which concludes the proof.
\qed\end{proof}

\begin{proposition}\label{thm:renewal.ub}
Assume that conditions of Lemma~\ref{lem:upper} hold. 
Then there exists an $x^*>0$ such that, 
for all $x>x^*$, $h\le s(x)$, and $t\in(0,h/2)$,
\begin{eqnarray*}
H(x+t,x+h-t] &\le& \frac{G^{**}_{h,x}(\infty) - \E G^{**}_{h,x}(X_0)}
{\min_{y\in[x+t,x+h-t]} m_2^{[t]}(y)}.
\end{eqnarray*}
\end{proposition}

\begin{proof}
Consider the following decomposition 
$$
G^{**}_{h,x}(X_n)
=\sum_{k=0}^{n-1} (G^{**}_{h,x}(X_{k+1})-G^{**}_{h,x}(X_k))+ G^{**}_{h,x}(X_0).
$$
Since $G^{**}_{h,x}(y)$ is bounded by $G^{**}_{h,x}(\infty)$, we obtain 
\begin{eqnarray*}
G^{**}_{h,x}(\infty) &\ge& \E G^{**}_{h,x}(X_n)\\
&=& \E G^{**}_{h,x}(X_0)+\sum_{k=0}^{n-1} 
\E [G^{**}_{h,x}(X_{k+1})-G^{**}_{h,x}(X_k)]\\
&\ge& \E G^{**}_{h,x}(X_0)+\sum_{k=0}^{n-1} 
\E \{m_2^{[t]}(X_k);X_k\in(x+t,x+h-t]\},
\end{eqnarray*}
for $x>x_*$, by Lemma \ref{lem:upper}. Hence, for any $n$, 
\begin{eqnarray*}
\sum_{k=0}^{n-1} \P\{X_k \in (x+t,x+h-t]\}
&\le& \frac{G^{**}_{h,x}(\infty)-\E G^{**}_{h,x}(X_0)}
{\min_{y\in[x+t,x+h-t]}m_2^{[t]}(y)}.
\end{eqnarray*}
Letting $n$ to infinity we arrive at the conclusion. 
\qed\end{proof}

\subsection{Lower bound}
\label{sec:lower}

We now turn to an accompanying lower bound for the renewal measure.
To this end we consider a differentiable function
\begin{equation}\label{eq_first_derivative_gh*}
g^*_{h,x}(y) := 
\begin{cases}
0,& y\le x,\\
2(y-x),& y\in (x,x+h],\\
2he^{R(x+h)-R(y)},& y>x+h,
\end{cases}
\end{equation}
whose derivative satisfies
\begin{equation}\label{eq_first_derivative_gh.der*}
g_{h,x}^{*\prime}(y)\ \le\ 2\I\{y\in[x,x+h]\}\quad\mbox{for all }y\ge 0.
\end{equation}
Its integral---which similarly to \eqref{eq_gh} originates from
the key function \eqref{G.dif} for diffusion processes,
\begin{equation}\label{eq_gh*}
G^*_{h,x}(y) := \int_0^y g^*_{h,x}(z)dz,
\end{equation}
is an increasing bounded function, $G^*_{h,x}(\infty)<\infty$, and
\begin{eqnarray}\label{G.upper.at.infty*}
G^*_{h,x}(\infty) &=& h^2+2h e^{R(x+h)}U(x+h)\nonumber\\
&\ge& 2h e^{R(x)}U(x+h).
\end{eqnarray}
For $h\le s(x)=o(1/r(x))$,
\begin{eqnarray}\label{G.upper.at.infty.ap*}
G^*_{h,x}(\infty) &\ge& (2+o(1))h e^{R(x)}U(x)\quad\mbox{as }x\to\infty.
\end{eqnarray}

Also define a concave function
\begin{equation}\label{eq_gh**}
G^{*<}_{h,x}(y) := h^2+2he^{R(x+h)}\int_{x+h}^y e^{-R(z)}dz,
\end{equation}
whose derivative is $2he^{R(x+h)-R(y)}$ and $G^{*<}_{h,x}(x+h)=G^*_{h,x}(x+h)$.
Observe the inequality
\begin{equation}\label{G*ge**}
G^*_{h,x}(y) \ge G^{*<}_{h,x}(y)\quad\mbox{for all }y\le x+h,
\end{equation}
and the equality
\begin{equation}\label{G*=**}
G^*_{h,x}(y) = G^{*<}_{h,x}(y)\quad\mbox{for all }y\ge x+h.
\end{equation}
Hence, for $y>x+h$ and $z>0$,
\begin{eqnarray}\label{lower.G.1*}
G^*_{h,x}(y-z)-G^{*<}_{h,x}(y-z) &\le& G^*_{h,x}(y)-G^{*<}_{h,x}(y-z)\nonumber\\
&=& G^{*<}_{h,x}(y)-G^{*<}_{h,x}(y-z)\nonumber\\
&=& 2he^{R(x+h)}(G(y)-G(y-z)).
\end{eqnarray}

\begin{lemma}\label{thm:renewal.lb}
Assume that the conditions \eqref{sigma_uniform2},
\eqref{cond.tail.right} and \eqref{m1_above} hold. 
Then there exists an $x^*>0$ such that,
for all $x>x^*$, $y\ge 0$, $h\le s(x)$, and $t\in(0,h/2)$,
\begin{multline*}
\E G^*_{h,x}(y+\xi(y))-G^*_{h,x}(y)\\
\le \begin{cases}
0,& y\le x-s(x),\\
2h \E\{\xi(y);\xi(y)\in(x-y, s(y))\},& y\in(x-s(x),x-t],\\
(1+hr(y)) m_2^{[s(y)]}(y),& y\in(x-t,x+h+t],\\
3h\E\{|\xi(y)|; -s(y)<\xi(y)<x+h-y\},& y>x+h+t.%,\\
%0,& y-s(y)>x+h. 
\end{cases}	
\end{multline*}
\end{lemma}

\begin{proof}
Since $G^*_{h,x}(y)$ is increasing in $y$, we obtain 
\begin{eqnarray*}
\E G^*_{h,x}(y+\xi(y))-G^*_{h,x}(y) &\le&
\E\{G^*_{h,x}(y+\xi(y))-G^*_{h,x}(y);\ \xi(y)\ge -s(y)\}\nonumber\\
&=& \E\{G^*_{h,x}(y+\xi(y))-G^*_{h,x}(y);\ |\xi(y)|\le s(y)\}
\ =:\ E,
\end{eqnarray*}
due to \eqref{cond.tail.right}.

In the case $y\le x-s(x)$, we have $y+\xi(y)\le x-s(x)+s(y)\le x$, so 
$G^*_{h,x}(y+\xi(y))=G^*_{h,x}(y)=0$ 
and the conclusion of the lemma follows for $y\le x-s(x)$.

In the case $x-s(x)<y\le x-t$,
it follows from the definition of $G^*_{h,x}$ that
$G^*_{h,x}(x+z)\le 2hz$ for all $z>0$ which yields
$G^*_{h,x}(y+z)\le 2h(y-x+z)$ for all $y\le x$ and $z>0$. Therefore,
\begin{eqnarray}\label{case1}
E &\le& 2h \E\left\{\xi(y);\xi(y)\in(x-y, s(y)]\right\}, 
\end{eqnarray}
and the conclusion of the lemma follows for $x-s(x)<y\le x-t$. 

In the case $y\in (x-t,x+h+t]$, 
we proceed similarly to Lemma~\ref{lem:upper}. 
By Taylor's expansion \eqref{E.Taylor}, 
\begin{eqnarray*}
E &\le& m_1^{[s(y)]}(y)g^*_{h,x}(y)+m_2^{[s(y)]}(y)\\
&\le& \frac12m_2^{[s(y)]}(y)r(y)g^*_{h,x}(y)+m_2^{[s(y)]}(y)\\
&\le& m_2^{[s(y)]}(y)(hr(y)+1),
\end{eqnarray*}
due to \eqref{m1_above} where $\theta(y)\ge 0$, \eqref{eq_first_derivative_gh.der*}
and inequality $g^*_{h,x}(y)\le 2h$, for all sufficiently large $y$.
Thus the conclusion of the lemma follows for $y\in (x-t,x+h+t]$. 

In the case $y>x+h+t$, since the function $G(y)$ is concave,
\begin{eqnarray*}
G(y)-G(y-z) &\le& zG'(y-z)\ =\ ze^{-R(y-z)}\quad\mbox{for all }z>0. 
\end{eqnarray*}
Therefore, as $y\to\infty$,
\begin{eqnarray*}
G(y)-G(y-z) &\le& ze^{-R(y)}(1+o(1))\quad\mbox{uniformly for all }z\in[0,s(y)]. 
\end{eqnarray*}
Thus it follows from \eqref{lower.G.1*} that, as $y\to\infty$,
\begin{eqnarray}\label{G*-**}
G^*_{h,x}(y-z)-G^{*<}_{h,x}(y-z) &\le& 2hze^{R(x+h)-R(y)}(1+o(1))\nonumber\\
&\le& 2hz(1+o(1))\quad\mbox{uniformly for all }h,z\in[0,s(y)]. \nonumber\\[-1mm]
\end{eqnarray}
The inequality \eqref{G*ge**} and equality \eqref{G*=**} 
allow us to conclude that, for $y>x+h$,
\begin{eqnarray*}
E &=& \E\{G^{*<}_{h,x}(y+\xi(y))-G^{*<}_{h,x}(y);\ |\xi(y)|\le s(y)\}\\
&&+\E\{G^*_{h,x}(y+\xi(y))-G^{*<}_{h,x}(y+\xi(y));\ |\xi(y)|\le s(y)\}\\
&=& \E\{G^{*<}_{h,x}(y+\xi(y))-G^{*<}_{h,x}(y);\ |\xi(y)|\le s(y)\}\\
&&+\E\{G^*_{h,x}(y+\xi(y))-G^{*<}_{h,x}(y+\xi(y));\ \xi(y)\in[-s(y),x+h-y]\}\\
&\le& \E\{G^*_{h,x}(y+\xi(y))-G^{*<}_{h,x}(y+\xi(y));\ \xi(y)\in[-s(y),x+h-y]\},
\end{eqnarray*}
by the second statement of Lemma \ref{lem:upper.U}.
Applying here \eqref{G*-**} we deduce, for all sufficiently large $x$ and $y>x+h$,
\begin{eqnarray*}
E &\le& 3h \E\{|\xi(y)|;\ \xi(y)\in[-s(y),x+h-y]\}.
\end{eqnarray*}
Combining altogether we conclude the result of the lemma for $y>x+h+t$. 
\qed\end{proof}

\begin{proposition}\label{prop5}
Let the assumptions of Lemma~\ref{thm:renewal.lb} hold. 
Then there exists an $x^*>0$ such that, 
for all $x>x^*$, $y\ge 0$, $h\le s(x)$, and $t\in(0,h/2)$,
\begin{eqnarray*}
H(x-t,x+h+t] &\ge& 
\frac{G^*_{h,x}(\infty)-\E G^*_{h,x}(X_0)-\delta(x)}
{\max_{y\in[x-t,x+h+t]}(1+hr(y))m_2^{[s(y)]}(y)},
\end{eqnarray*}
where 
\begin{eqnarray*}
\delta(x) &=& 2h \int_{x-s(x)}^{x-t} H(dy) \E\{\xi(y);\,x-y<\xi(y)<s(y)\}\\
&&\hspace{5mm}+3h \int_{x+h+t}^\infty H(dy) \E\{|\xi(y)|;\,-s(y)<\xi(y)<x+h-y\}. 
\end{eqnarray*}
\end{proposition}	

\begin{proof}
Consider the decomposition 
$$
G^*_{h,x}(X_n)
=\sum_{k=0}^{n-1} (G^*_{h,x}(X_{k+1})-G^*_{h,x}(X_k))+ G^*_{h,x}(X_0).
$$
Therefore we deduce from Lemma \ref{thm:renewal.lb} that, 
for some $c<\infty$ and all $x>x_*$,
\begin{eqnarray*}
\lefteqn{\E G^*_{h,x}(X_n)}\\
&=& \E G^*_{h,x}(X_0)
+\sum_{k=0}^{n-1} \E(G^*_{h,x}(X_{k+1})-G^*_{h,x}(X_k))\\
&\le& \E G^*_{h,x}(X_0)+ \sum_{k=0}^{n-1} 
\E\left\{(1+hr(X_k))m_2^{[s(X_k)]}(X_k);X_k \in (x-t,x+h+t]\right\}\\ 
&&+2h \sum_{k=0}^{n-1} \int_{x-s(x)}^{x-t} \P\{X_k \in dy\} 
\E\{\xi(y);x-y<\xi(y)<s(y)\}\\  
&&+3h\sum_{k=0}^{n-1} \int_{x+h+t}^\infty
\P\{X_k \in dy\} \E\{|\xi(y)|;-s(y)<\xi(y)<x+h-y\}.
\end{eqnarray*}
Hence, for any $n$, 
\begin{eqnarray*}
\sum_{k=0}^{n-1} \P\{X_k \in (x-t,x+h+t]\} &\ge& 
\frac{\E G^*_{h,x}(X_n)-\E G^*_{h,x}(X_0)-\delta(x)}
{\max_{y\in[x-t,x+h+t]}(1+hr(y))m_2^{[s(y)]}(y)}.
\end{eqnarray*}
Letting $n$ to infinity we arrive at the conclusion
due to the convergence $G^*_{h,x}(X_n)\to G^*_{h,x}(\infty)$
which in its turn follows from Lemma~\ref{lem:upper} together with 
the martingale convergence theorem and the assumption~\eqref{eq:irreducibility}.  
\qed\end{proof}

In order to get a lower bound in a closed form, 
we need to derive conditions under which the term $\delta(x)$ 
in Proposition~\ref{prop5} is of order $o(G^*_{h,x}(\infty))$ as $x\to\infty$.
In the next result we demonstrate how to bound $\delta(x)$ 
provided an appropriate upper bound for the renewal measure is available.

\begin{lemma}\label{thm:renewal.ub.lower}
Let, for some $h=h(x)\le s(x)$ and $t=t(x)\le h/2$,
\begin{eqnarray}\label{general_upper}
\sup_{y:\ x/2\le y\le 2x}H(y,y+t] &\le& C_1t U(x)e^{R(x)}\quad\mbox{for some }C_1<\infty,
\end{eqnarray}
and, for some random variable $\xi$ with $\E\xi^2<\infty$,
\begin{eqnarray}\label{majorant_third.lower}
|\xi(y)| &\le_{st}& \xi\quad\mbox{for all }y\ge 0.
\end{eqnarray}
Then $\delta(x)\le c hU(x)e^{R(x)} \E\{\xi^2;\ |\xi|>t\}$ 
for some $c<\infty$.
\end{lemma}

\begin{proof}
Let us analyse the first term in $\delta(x)$.
The stochastic majorisation condition \eqref{majorant_third.lower} yields that 
\begin{eqnarray*}
\int_{x-s(x)}^{x-t} H(dy) \E\{\xi(y);\ x-y<\xi(y)<s(y)\} 
&\le& \int_{x-s(x)}^{x-t} H(dy) \E\{\xi;\ \xi>x-y\}.
\end{eqnarray*}
Further, using the upper bound \eqref{general_upper} we deduce 
\begin{eqnarray*}
\int_{x-s(x)}^{x-t} H(dy) \E\{\xi;\ \xi>x-y\} 
&\le& \sum_{n=1}^{s(x)/t}	H(x-(n+1)t,x-nt] \E\{\xi;\ \xi>nt\}\\ 
&\le& C_2tU(x)e^{R(x)} \sum_{n=1}^{s(x)/t} \E\{\xi;\ \xi>nt\} \\ 
&\le& C_2tU(x)e^{R(x)} \E\{\xi^2/t;\ \xi>t\}\\
&=& C_2U(x)e^{R(x)} \E\{\xi^2;\ \xi>t\}.
\end{eqnarray*}
Hence the first term in $\delta(x)$ is not greater than
$2C_2hU(x)e^{R(x)}\E\{\xi^2;\ \xi>t\}$ as required.

The second term in $\delta(x)$ can be bounded in the same way, namely
\begin{eqnarray*}
\lefteqn{\int_{x+h+t}^\infty H(dy) \E\{|\xi(y)|;\ -s(y)<\xi(y)<x+h-y\} }\\
&=& \int_{x+h+t}^{x+h+s(x)} H(dy) \E\{|\xi(y)|;\ -s(x)<\xi(y)<x+h-y\}\\
&&\hspace{20mm}\le\ \int_{x+h+t}^{x+h+s(x)} H(dy) \E\{|\xi|;\ \xi<x+h-y\}\\
&&\hspace{40mm}=\ \int_t^{s(x)} H(x+h+dy) \E\{|\xi|;\ \xi<-y\},
\end{eqnarray*}
and, as above, 
\begin{eqnarray*}
\lefteqn{\int_t^{s(x)} H(x+h+dy) \E\{|\xi|;\ \xi<-y\} }\\
&\le& \sum_{n=1}^{s(x)/t}	 H(x+h+nt,x+h+(n+1)t] \E\{|\xi|;\ \xi<-nt\}\\ 
&\le& C_3tU(x)e^{R(x)} \sum_{n=1}^{s(x)/t} \E\{|\xi|;\ \xi<-nt\} \\ 
&\le& C_3U(x)e^{R(x)} \E\{\xi^2;\ \xi<-t\},
\end{eqnarray*}
and we conclude the proof.
\qed\end{proof}

\subsection{On two Markov chains with asymptotically equal jumps}

As in Section \ref{sec:thresholds}, let $\{Y_n\}$ and $\{Z_n\}$ be two Markov chains 
with jumps $\eta(x)$ and $\zeta(x)$ respectively.
Denote by $H^Y$ and $H^Z$ their renewal measures.

\begin{lemma}\label{l:XY.renew.equiv}
Let the conditions of Lemma \ref{l:XY.equiv} hold. If there exists
a nonnegative function $g(x)$ such that
\begin{equation}\label{equiv.1}
H^Z(x,x+h(x)]\sim g(x)\quad\mbox{as }x\to\infty
\end{equation}
for any distribution of $Z_0$ and
\begin{equation}\label{equiv.2}
\sup_y H_y^Z(x,x+h(x)]=O(g(x))\quad\mbox{as }x\to\infty,
\end{equation}
then, for any distribution of $Y_0$,
$$
H^Y(x,x+h(x)]\sim g(x)\quad\mbox{as }x\to\infty.
$$

If, in addition, the family of random variables
$$
\frac{1}{g(x)}\sum_{n=0}^\infty \I\{Z_n\in(x,x+h(x)]\}
$$
is uniformly integrable, then
$$
\frac{1}{g(x)}\sum_{n=0}^\infty \I\{Y_n\in(x,x+h(x)]\}
$$
is so. 
\end{lemma}

\begin{proof}
Let us consider sequences of independent random fields 
$\{\eta_n(x),x\in\R\}_{n\ge 0}$ and
$\{\zeta_n(x),x\in\R\}_{n\ge 0}$ as in \eqref{rec.3.1.hy.n}
and then the Markov chains $\{Y_n\}$ and $\{Z_n\}$ as there.

Fix an $\varepsilon>0$ and let $x_\varepsilon$ be delivered by 
Lemma \ref{l:XY.equiv}. Let $\tau:=\min\{n\ge 0:Y_n>x_\varepsilon\}$ and 
consider $\{Z_k\}$ with initial value $Z_0=Y_\tau$. Define 
$$
\mu:=\min\{k\ge 1:Z_k\not= Y_{\tau+k}\}.
$$
By Lemma \ref{l:XY.equiv}, $\P\{\mu<\infty\}\le\varepsilon$. 
For $x>x_\varepsilon$,
\begin{eqnarray*}
\lefteqn{\sup_y H_y^Y(x,x+h(x)]}\\
&\le& \sup_y \E_y\sum_{n=\tau}^{\tau+\mu-1} \I\{Y_n\in(x,x+h(x)]\}
+\sup_y \E_y\sum_{n=\tau+\mu}^\infty \I\{Y_n\in(x,x+h(x)]\}.
\end{eqnarray*}
The first expectation on the right hand side is not greater than
$H_y^Z(x,x+h(x)]$ because $Y_n=Z_{n-\tau}$ between $\tau$ and $\tau+\mu-1$.
The second one possesses the following upper bound
\begin{eqnarray*}
\E_y\sum_{n=\tau+\mu}^\infty \I\{Y_n\in(x,x+h(x)]\} &=&
\E_y\Bigl\{\sum_{n=\tau+\mu}^\infty \I\{Y_n\in(x,x+h(x)]\}
\Big| \mu<\infty\Bigr\}
\P\{\mu<\infty\}\\ 
&\le& \sup_z H_z^Y(x,x+h(x)]\varepsilon.
\end{eqnarray*}
Therefore,
\begin{eqnarray}\label{equiv.3}
\sup_y H_y^Y(x,x+h(x)] &\le& \frac{1}{1-\varepsilon} \sup_yH_y^Z(x,x+h(x)].
\end{eqnarray}

For any distribution of $Y_0$ and $x>x_\varepsilon$ we have
\begin{eqnarray*}
\lefteqn{H^Y(x,x+h(x)]}\\
&=& \E\sum_{n=\tau}^{\tau+\mu-1} \I\{Y_n\in(x,x+h(x)]\}
+\E\sum_{n=\tau+\mu}^\infty \I\{Y_n\in(x,x+h(x)]\}\\
&=& \E\sum_{n=\tau}^{\tau+\mu-1} \I\{Z_n\in(x,x+h(x)]\}
+\E\sum_{n=\tau+\mu}^\infty \I\{Y_n\in(x,x+h(x)]\}\\
&=& \E H^Y_{Y_\tau}(x,x+h(x)]\\
& &\hspace{5mm}-\E\E_{Y_\tau}\sum_{n=\mu}^\infty \I\{Z_n\in(x,x+h(x)]\}
+\E\sum_{n=\tau+\mu}^\infty \I\{Y_n\in(x,x+h(x)]\}.
\end{eqnarray*}
As we have seen in the first part of the proof, for all $x$ large enough,
\begin{eqnarray*}
\E\sum_{n=\tau+\mu}^\infty \I\{Y_n\in(x,x+h(x)]\}
&\le& \varepsilon \sup_y H_y^Y(x,x+h(x)]\\
&\le& \frac{\varepsilon}{1-\varepsilon} \sup_y H_y^Z(x,x+h(x)],
\end{eqnarray*}
owing to \eqref{equiv.3}. Similarly,
\begin{eqnarray*}
\E_{Z_\tau}\sum_{n=\mu}^{\infty} \I\{Z_n\in(x,x+h(x)]\}
&\le& \E\P_{Y_\tau}(\mu<\infty)\sup_y H_y^Z(x,x+h(x)]\\
&\le& \varepsilon \sup_y H_y^Z(x,x+h(x)].
\end{eqnarray*}
Therefore,
\begin{eqnarray*}
|H^Y(x,x+h(x)]-\E H^Z_{Y_\tau}(x,x+h(x)]|
&\le& \frac{\varepsilon}{1-\varepsilon}\sup_y H_y^Z(x,x+h(x)].
\end{eqnarray*}
Letting $\varepsilon\to0$ and using \eqref{equiv.2} we conclude
\begin{eqnarray*}
|H^Y(x,x+h(x)]-\E H^Z_{Z_\tau}(x,x+h(x)]|
&=& o(g(x))\quad\mbox{as }x\to\infty.
\end{eqnarray*}
According to \eqref{equiv.1} and \eqref{equiv.2},
$\E H^Z_{Y_\tau}(x,x+h(x)]\sim g(x)$ which completes the proof.
\qed\end{proof}

\subsection{Proofs of Theorems \ref{thm:regular}, \ref{thm:critical},
and \ref{thm:weibull}}

\begin{theopargself}
\begin{proof}[of Theorem \ref{thm:regular}]
Consider a modified Markov chain $\{\widetilde X_n\}$ on the same probability 
space as $\{X_n\}$ with jumps $\widetilde\xi(x)$ defined as follows:
\begin{eqnarray*}%\label{def:Z.jumps}
\widetilde\xi(x) &=& \left\{
\begin{array}{ll}
\xi(x) &\mbox{if }|\xi(x)|\le s(x);\\
\mbox{any value} &\mbox{if }|\xi(x)|>s(x).
\end{array}
\right.
\end{eqnarray*}
If $\{\widetilde X_n\}$ does not satisfy the unboundedness of trajectories condition 
\eqref{eq:irreducibility}, then we can increase the value of $s(x)$
on some set bounded on the right in such a way that then $\{\widetilde X_n\}$ does satisfy 
\eqref{eq:irreducibility}. Indeed, it follows from the conditions
\eqref{m1.m2.1x}, \eqref{majorant_third} and \eqref{majorant_third_moment_exists}
that there exist a sufficiently high level $x_0$ and an $\varepsilon>0$
such that $\P\{\xi(x)\ge\varepsilon\}\ge\varepsilon$ for all $x\ge x_0$.
Then it suffices to increase $s(x)$ on the set $(-\infty,x_0]$
to ensure the condition \eqref{eq:irreducibility} for $\{\widetilde X_n\}$.

Without loss of generality we assume that $h(x)\le s(x)$.
Let us choose a function $t(x)\uparrow\infty$ of order $o(h(x))$ as $x\to\infty$.

Fix some $c>1$ and consider $r(x)=c/(1+x)$. Then, 
$$
R(x)=c\log(1+x)\quad \mbox{ and } \quad U(x)=(1+x)^{1-c}/(c-1).  
$$ 
Therefore, 
\begin{eqnarray}
\label{regular_G}
U(x)e^{R(x)}=\frac{x+1}{c-1}.  
\end{eqnarray}

The chain $\{\widetilde X_n\}$ satisfies the condition \eqref{cond.tail.left}.
Fix some $c^{**}\in(1,2\mu/b)$ and define $r^{**}(x)=c^{**}/(1+x)$,
which ensures the condition \eqref{m1_below} with 
$\theta(y)=\theta=(2\mu/bc^{**}-1)/2>0$.
The condition \eqref{sigma_uniform2} is immediate from the upper bound
\begin{equation}\label{eq_third_simplify}
\E\{|\xi(y)|^3;\ |\xi(y)|\le s(y)\}\le s(y)m_2^{[s(y)]}(y)
\end{equation}
and the relation $s(y) = o(y)$.  Also,
\begin{eqnarray*}
m_2^{[t(x)]}(x) &\to& b\quad\mbox{as }x\to\infty,
\end{eqnarray*}	
by the conditions \eqref{majorant_third} and \eqref{majorant_third_moment_exists}. 
As a result, by Proposition~\ref{thm:renewal.ub}, as $x\to\infty$,
\begin{eqnarray*}
\widetilde H(x+t(x),x+h(x)-t(x)] &\le& 
\frac{G^{**}_{h,x}(\infty)}{b+o(1)}\\
&\le& \frac{2+o(1)}{(c^{**}-1)b}xh(x),
\end{eqnarray*} 
owing to \eqref{G.upper.at.infty.ap} and
\eqref{regular_G}. Letting $c^{**}\to2\mu/b$, we get
$$
\widetilde H(x+t(x),x+h(x)-t(x)]\le\frac{2+o(1)}{2\mu-b}xh(x)
\quad\mbox{as }x\to\infty.
$$
Taking into account that $t(x)=o(h(x))$ we conclude the following upper bound  
\begin{eqnarray}\label{regular_upper}
\widetilde H(x,x+h(x)] &\le& \frac{2+o(1)}{2\mu-b}xh(x)\quad\mbox{as }x\to \infty.
\end{eqnarray}

The chain $\{\widetilde X_n\}$ satisfies the condition \eqref{cond.tail.right}.
Fix some $c^*>2\mu/b$ and define $r^*(x)=c^*/(1+x)$,
which ensures the condition \eqref{m1_above} with 
$\theta(y)=\theta=(1-2\mu/bc^*)/2>0$.
Then it follows from Proposition~\ref{prop5} that, as $x\to\infty$,
\begin{eqnarray*}
\widetilde H(x-t(x),x+h(x)+t(x)] &\ge& 
\frac{G^*_{h,x}(\infty)-\E G^*_{h,x}(X_0)-\delta(x)}{b+o(1)}\\
&\ge& (2+o(1))\frac{h(x)\frac{x}{c^*-1}-\delta(x)}{b+o(1)},
\end{eqnarray*}
due to \eqref{G.upper.at.infty.ap*} and \eqref{regular_G}. 
By the condition \eqref{majorant_third},
the chain $\{\widetilde X_n\}$ satisfies \eqref{majorant_third.lower}
which together with the upper bound \eqref{regular_upper} for the renewal
measure generated by $\{\widetilde X_n\}$ yields the upper bound 
for $\delta(x)$ delivered by Lemma \ref{thm:renewal.ub.lower}. Therefore,
\begin{eqnarray*}
\widetilde H(x-t(x),x+h(x)+t(x)] &\ge& \frac{2+o(1)}{(c^*-1)b}xh(x).
\end{eqnarray*}
owing to \eqref{regular_G}.
Letting here $c^*\to 2\mu/b$ and since $t(x)=o(h(x))$, we finally get
\begin{eqnarray*}
\widetilde H(x,x+h(x)] &\ge& \frac{2+o(1)}{2\mu-b}xh(x)\quad\mbox{as }x\to\infty.
\end{eqnarray*}
Combining this lower bound with the upper bound \eqref{regular_upper},
we conclude that
\begin{eqnarray*}
\widetilde H(x,x+h(x)] &\sim& \frac{2}{2\mu-b}xh(x)\quad\mbox{as }x\to\infty.
\end{eqnarray*}
Together with the condition \eqref{regular_left_tail} this allows us 
to apply Lemma \ref{l:XY.renew.equiv} to the two Markov chains, $Y=X$ 
and $Z=\widetilde X$, hence the same asymptotics
for the renewal measure generated by $\{X_n\}$.
\qed\end{proof}
\end{theopargself}

\begin{theopargself}
\begin{proof}[of Theorem \ref{thm:critical}]
As in the proof of Theorem \ref{thm:regular}, 
from the very beginning we may assume
that $|\xi(y)|\le s(y)$ for all $y$ which implies both 
\eqref{cond.tail.left} and \eqref{cond.tail.right}.
Without loss of generality we assume that $h(x)\le s(x)$.

Fix $c>1$ and consider
\begin{eqnarray*}
r(x) &=& \frac{1}{x+e_{(m)}}+\frac{1}{(x+e_{(m)})\log (x+e_{(m)})}\\
&&+\ldots+\frac{c}{(x+e_{(m)})\log (x+e_{(m)})\cdot\ldots\cdot\log_{(m)}(x+e_{(m)})},
\end{eqnarray*}
where $e_{(m)}>0$ is defined by $\log_{(m)} e_{(m)}=1$. Therefore,
\begin{eqnarray*}
R(x)&=&\log(x+e_{(m)})+\log\log(x+e_{(m)})\\
&&+\ldots+\log_{(m)}(x+e_{(m)})+c\log_{(m+1)}(x+e_{(m)})-C_m
\end{eqnarray*}
and
\begin{eqnarray*}
U(x)=\frac{e^{C_m}}{c-1}\left(\log_{(m)}(x+e_{(m)})\right)^{1-c},
\end{eqnarray*}
which implies from \eqref{G.upper.at.infty.ap} that, for $c^{**}<\gamma+1$,
\begin{eqnarray*}
G^{**}_{h(x),x}(\infty) &\le& 
\frac{2+o(1)}{c^{**}-1}h(x)x\log x\cdot\ldots\cdot\log_{(m)}x
\quad\mbox{as }x\to\infty,
\end{eqnarray*}
and from \eqref{G.upper.at.infty.ap*}, for $c^*>\gamma+1$,
\begin{eqnarray*}
G^*_{h(x),x}(\infty) &\ge& 
\frac{2+o(1)}{c^*-1}h(x)x\log x\cdot\ldots\cdot\log_{(m)}x
\quad\mbox{as }x\to\infty.
\end{eqnarray*}
Repeating the arguments used in the proof of Theorem~\ref{thm:regular}, 
we obtain the desired result.
\qed\end{proof}
\end{theopargself}

\begin{theopargself}
\begin{proof}[of Theorem \ref{thm:weibull}]
As in the proof of Theorem \ref{thm:regular}, 
from the very beginning we may assume
that $|\xi(y)|\le s(y)$ for all $y$ which implies both 
\eqref{cond.tail.left} and \eqref{cond.tail.right}.
Without loss of generality we assume that $h(x)\le s(x)$.
Let us choose a function $t(x)\uparrow\infty$ of order $o(h(x))$ as $x\to\infty$.

Fix some $c>0$ and consider $r(x)=cv(x)$. 
Then, by  l'H\^{o}spital's rule,
$$
\frac{U(x)}{U'(x)} \sim \frac{1}{r(x)}.  
$$ 
Therefore, as follows from \eqref{G.upper.at.infty.ap}
\begin{eqnarray}\label{regular_G_weibull}
G^{**}_{h(x),x}(\infty) &\le& (2+o(1))\frac{h(x)}{r(x)}\quad\mbox{as }x\to\infty,  
\end{eqnarray}
and from \eqref{G.upper.at.infty.ap*}
\begin{eqnarray}\label{regular_G_weibull*}
G^*_{h(x),x}(\infty) &\ge& (2+o(1))\frac{h(x)}{r(x)}\quad\mbox{as }x\to\infty.  
\end{eqnarray}

Considering $c^{**}<2/b$ and $c^*>2/b$ and repeating the arguments 
used in the proof of Theorem~\ref{thm:regular}, we conclude the proof.
\qed\end{proof}
\end{theopargself}

\section{Asymptotics for renewal measure on fixed intervals}
\label{sec:ren.local}

While the asymptotic behaviour of the renewal measure on growing intervals
is derived under assumptions on regular behaviour of the first two
moments only, it seems that the local renewal theorem can be only proved
for asymptotically homogeneous in space Markov chain.
The next result gives us a tool for deriving asymptotic behaviour
of the renewal measure on intervals from results for sufficiently slowly
growing intervals. It requires weak convergence of jumps at infinity,
that is, we consider an asymptotically homogeneous in space Markov chain
which is defined as a Markov chain such that, for some random variable $\xi$,
\begin{equation}\label{asymp.hom.i}
\xi(x) \Rightarrow \xi\quad\mbox{as }x\to\infty;
\end{equation} 
if there is no asymptotic homogeneity in space
then the asymptotic behaviour of $H(x,x+h]$ may be very different.
For Markov chains on $\Zp$ with bounded jumps,
it was studied in Sections \ref{sec:loc.renewal} and \ref{sec:loc.renewal.n} via
careful analysis of the returning probabilities at high level.
\index{Renewal measure!asymptotics on fixed intervals}

\begin{theorem}\label{thm:ah.renewal.i}
Let~\eqref{asymp.hom.i} hold and the family of random variables 
$\{|\xi(x)|,\ x\in\R\}$ admit an integrable majorant $\Xi$, 
that is, $\E\Xi<\infty$ and
\begin{eqnarray}\label{majoriz.i}
|\xi(x)| &\le_{\rm st}& \Xi
\quad\mbox{for all }x\in\R.
\end{eqnarray}
Assume that there exist a bounded function $v(x)>0$, 
a growing level $\widetilde t(x)\uparrow\infty$ 
and a constant $C_H<\infty$ such that, 
for any $t(x)\uparrow\infty$ satisfying $t(x)\le\widetilde t(x)$,
\begin{equation}\label{eq.growing.intervals}
\frac{v(x)H(x,x+t(x)]}{t(x)}\ \to\ C_H\quad\mbox{as }x\to\infty.
\end{equation}

If the limiting random variable $\xi$ is non-lattice, 
then $v(x)H(x,x+h]\to C_Hh$ as $x\to\infty$, for all fixed $h>0$.

If the chain $\{X_n\}$ is integer-valued and $\Z$ is the minimal lattice 
for the variable $\xi$, then $v(k)H\{k\}\to C_H$ as $k\to\infty$,
and, in addition, the family of random variables
\begin{equation}\label{ren.loc.h.2.n.uni}
v(k)\sum_{n=0}^\infty \I\{X_n=k\},\quad k>0,
\end{equation}
is uniformly integrable.
\end{theorem}

Let us apply the last result to chains considered in 
Theorems~\ref{thm:regular}--\ref{thm:weibull}.
In addition, under specific assumptions on the drift function
we are able to generalise the uniform integrability conclusion 
from lattice to general Markov chains.
\index{Renewal measure!asymptotics on fixed intervals}

\begin{corollary}\label{cor:regular}
Under the conditions of Theorem \ref{thm:regular},
\eqref{asymp.hom.i} and \eqref{majoriz.i}, we have, for every $h>0$, 
\begin{eqnarray*}
H(x,x+h] &\sim& \frac{2h}{2\mu-b}x\quad\mbox{as }x\to\infty,
\end{eqnarray*}
if the limiting random variable $\xi$ is non-lattice, and
\begin{eqnarray*}
H\{k\} &\sim& \frac{2}{2\mu-b}k\quad\mbox{as }k\to\infty,
\end{eqnarray*}
if the chain $\{X_n\}$ is integer-valued and $\Z$ is the minimal lattice 
for the variable $\xi$.

In addition, for some $\widehat x\in\R$, the family of random variables
\begin{equation*}
\frac{1}{x}\sum_{n=0}^\infty \I\{X_n\in(x,x+1]\},\quad x\ge \widehat x,
\end{equation*}
is uniformly integrable.
\end{corollary}

For lattice Markov chains, the last corollary is an improvement on
Theorem \ref{thm:srt} where the same asymptotics were only proven 
in the case of bounded jumps. 
%However, boundedness of jumps there allowed us
%to prove the uniform integrability of the properly normalised 
%local renewal process. Here it is not clear how to prove the uniform integrability
%because the operator approach seems to be not suitable for that.
%
A similar improvement on Theorem \ref{thm:srt.crit} holds true.

\begin{corollary}\label{cor:critical}
Under the conditions of Theorem \ref{thm:critical},
\eqref{asymp.hom.i} and \eqref{majoriz.i}, we have, for every $h>0$,
\begin{eqnarray*}
H(x,x+h] &\sim& \frac{2h}{b\gamma}x\log x\cdot\ldots\cdot\log_{(m)}x
\quad\mbox{as }x\to\infty,
\end{eqnarray*}
if the limiting random variable $\xi$ is non-lattice, and
\begin{eqnarray*}
H\{k\} &\sim& \frac{2}{b\gamma}k\log k\cdot\ldots\cdot\log_{(m)}k
\quad\mbox{as }k\to\infty,
\end{eqnarray*}
if the chain $\{X_n\}$ is integer-valued and $\Z$ is the minimal lattice 
for the variable $\xi$.

In addition, for some $\widehat x\in\R$, the family of random variables
\begin{equation*}
\frac{1}{x\log x\cdot\ldots\cdot\log_{(m)}x}
\sum_{n=0}^\infty \I\{X_n\in(x,x+1]\},\quad x\ge \widehat x,
\end{equation*}
is uniformly integrable.
\end{corollary}

\index{Renewal measure!asymptotics on fixed intervals}
\begin{corollary}\label{cor:weibull}
Under the conditions of Theorem~\ref{thm:weibull},
\eqref{asymp.hom.i} and \eqref{majoriz.i}, we have, for every $h>0$,
\begin{eqnarray*}
H(x,x+h] &\sim& \frac{h}{v(x)}\quad\mbox{as }x\to\infty,
\end{eqnarray*}
if the limiting random variable $\xi$ is non-lattice, and
\begin{eqnarray*}
H\{k\} &\sim& \frac{1}{v(k)}\quad\mbox{as }k\to\infty,
\end{eqnarray*}
if the chain $\{X_n\}$ is integer-valued and $\Z$ is the minimal lattice 
for the variable $\xi$.

In addition, for some $\widehat x\in\R$, the family of random variables
\begin{equation*}
v(x)\sum_{n=0}^\infty \I\{X_n\in(x,x+1]\},\quad x\ge \widehat x,
\end{equation*}
is uniformly integrable.
\end{corollary}

The last result is an improvement on Theorem \ref{thm:srt.n}
and it is particularly useful for the proof of the local asymptotics
for a random walk conditioned to stay positive -- which represents 
one of the classical examples of chains with asymptotically zero drift, 
see Proposition \ref{prop:rwcsn.lr}.

\section{Key renewal theorem}

We now turn to the renewal equation
\begin{equation*}
Z(B)\ =\ z(B)+\int_\R Z(dy)\, P(y,B),\ \ B\in\mathcal B(\R),
\end{equation*}
where $z$ is a finite nonnegative measure on $\R$. 
This is more than sufficient to ensure that
$$
Z(B)=\int_\R z(du)H_u(B),\ \ B\in\mathcal B(\R),
$$
is a unique locally finite solution to the renewal equation.  
The analysis of the preceding subsection of this paper
allows us to deduce the asymptotic behaviour of the measure $Z$ at infinity.
The proof is immediate from the dominated convergence theorem.
\index{Renewal measure!key theorem}

\begin{theorem}\label{reneqn}
Let $B\in\mathcal B(\R)$. 
Assume that, for some positive function $g(x)$ and for all $y\in\R$, 
$$
H_y(x+B)\ \sim\ g(x)\quad\mbox{as }x\to\infty,
$$
and, for some $c<\infty$,
$$
H_y(x+B)\ \le\ cg(x)\quad\mbox{for all }x,\ y\in\R.
$$
If $z$ is a finite measure, then
$$
Z(x+B)\ \sim\ z(\R)g(x)\quad\mbox{as }x\to\infty.
$$
\end{theorem}

\section{Proof of results of Section \ref{sec:ren.local}}

In this section, our first goal is to provide an approach that allows 
us to reduce the proof of the asymptotic behaviour of the renewal measure
on intervals to that on sufficiently slowly growing intervals,
that is, Theorem \ref{thm:ah.renewal.i}.

\begin{lemma}\label{key:lem1.i}
Assume that there exist functions $v(x)>0$ and $\widetilde t(x)\uparrow\infty$ 
such that, for any $t(x)\uparrow\infty$ satisfying $t(x)\le\widetilde t(x)$,
\begin{eqnarray*}
\sup_{x\ge 1}\frac{v(x)H(x,x+t(x)]}{t(x)} &<& \infty.
\end{eqnarray*}
Then,
\begin{eqnarray}\label{finite.bound.U.i}
\sup_{x\ge 1} v(x) H(x,x+1] &<& \infty.
\end{eqnarray}
\end{lemma}	

\begin{proof}
Suppose that \eqref{finite.bound.U.i} fails. 
Then there exists a sequence $x_n\uparrow\infty$ such that 
$$
\alpha_n:= v(x_n) H(x_n,x_n+1] \to \infty\quad\mbox{as }n\to\infty. 
$$
Since both $\alpha_n$ and $\widetilde t(x_n)$ tend to infinity, there
exists a sequence  $t_n\uparrow\infty$ such that  
$t_n\le\widetilde t(x_n)$ and $t_n=o(\alpha_n)$ as $n\to\infty$. 
Let $t$ be defined as follows 
$$
t(x) = t_n,\quad  x_n\le x< x_{n+1}. 
$$
Clearly, $t(x)\le\widetilde t(x)$ and $t(x)\uparrow\infty$. 
Then, eventually in $n$,
$$
\frac{v(x_n)H(x_n,x_n+t(x_n)]}{t(x_n)}
\ge 
\frac{v(x_n)H(x_n,x_n+1]}{t(x_n)} 
=\frac{\alpha_n}{t(x_n)}
\to\infty,
$$
which contradicts the hypothesis. 
\qed\end{proof}	

\begin{theopargself}
\begin{proof}[of Theorem \ref{thm:ah.renewal.i}]
By Lemma~\ref{key:lem1.i} it follows from the assumption 
\eqref{eq.growing.intervals}   
that  the supremum in  (\ref{finite.bound.U.i}) is finite. 
In turn, it allows us to apply Helly's Selection Theorem 
to the family of measures $\{v(x)H(x+\cdot),\ x\in\R\}$ 
(see, for example, Theorem 2 in \cite[Section VIII.6]{Feller}). 
Hence, there exists a sequence of points $x_n\to\infty$ such that 
the sequence of measures $v(x_n)H(x_n+\cdot)$ converges weakly to some measure 
$\lambda$ as $n\to\infty$. The following two results characterise $\lambda$.

\begin{lemma}\label{l.1.i}
Let $F$ denote the distribution of $\xi$.
A weak limit $\lambda$ of the sequence of measures $v(x_n)H(x_n+\cdot)$ 
satisfies the identity $\lambda=\lambda*F$.
\end{lemma}

\begin{proof}
The measure $\lambda$ is positive and $\sigma$-finite with necessity. 
Fix any smooth function $f(x)$ with a bounded support; 
let $A>0$ be such that $f(x)=0$ for $x\notin[-A,A]$. 
The weak convergence of measures means convergence of integrals
\begin{eqnarray}\label{conv.f.1.i}
\int_{-\infty}^\infty f(x)v(x_n)H(x_n+dx)
=\int_{-A}^A f(x)v(x_n)H(x_n+dx) \to \int_{-A}^A f(x)\lambda(dx)
\end{eqnarray}
as $n\to\infty$. On the other hand, due to the equality
$H(\cdot)=\P\{X_0\in\cdot\}+H * P(\cdot)$ we have the following
representation for the left side of \eqref{conv.f.1.i}:
\begin{equation}\label{conv.f.2.i}
\int_{-A}^A f(x)v(x_n)\P\{X_0\in x_n+dx\}
+\int_{-A}^A f(x)
\int_{-\infty}^\infty  P(x_n+y,x_n+dx)v(x_n)H(x_n+dy).
\end{equation}
Since $f$ and $v$ are  bounded,
\begin{equation}\label{conv.f.3.i}
\int_{-A}^A f(x)v(x_n)\P\{X_0\in x_n+dx\}
\le \|f\|_\infty \|v\|_\infty \P\{X_0\in[x_n-A,x_n+A]\} \to 0
\end{equation}
as $n\to\infty$. The second term in \eqref{conv.f.2.i} is equal to
\begin{eqnarray}\label{conv.f.4.i}
\int_{-\infty}^\infty 
v(x_n)H(x_n+dy)
\int_{-A}^A f(x)P(x_n+y,x_n+dx).
\end{eqnarray}
The weak convergence $P(t,t+\cdot)\Rightarrow F(\cdot)$ as $t\to\infty$ 
implies convergence of the inner integral in \eqref{conv.f.4.i}:
\begin{eqnarray*}
\int_{-A}^A f(x)P(x_n+y,x_n+dx)
&\to& \int_{-A}^A f(x)F(dx-y);
\end{eqnarray*}
here the rate of convergence can be estimated in the following way:
\begin{eqnarray*}
\Delta(n,y) &:=& \Biggl|\int_{-A}^A
f(x) (P(x_n+y,x_n+dx)-F(dx-y))\Biggr|\\
&=& \Biggl|\int_{-A}^A
f'(x)(\P\{\xi(x_n+y)\le x-y\}-F(x-y))dx\Biggr|\\
&\le& \|f'\|_\infty \int_{-A-y}^{A-y}
|\P\{\xi(x_n+y)\le x\}-F(x)|dx.
\end{eqnarray*}
Thus, the asymptotic homogeneity of the chain
yields for every fixed $C>0$ a uniform convergence
\begin{eqnarray}\label{Delta.1.i}
\sup_{y\in[-C,C]}\Delta(n,y) &\to& 0\quad\mbox{as }n\to\infty.
\end{eqnarray}
In addition, by the majorisation condition \eqref{majoriz.i}, for all $x\in\R$,
\begin{eqnarray*}
|\P\{\xi(x_n+y)\le x\}-F(x)| &\le& 2\P\{\Xi>|x|\}.
\end{eqnarray*}
Hence, for all $y$,
\begin{eqnarray}\label{Delta.2.i}
\Delta(n,y) &\le& 2\|f'\|_\infty \int_{-A-y}^{A-y}
\P\{\Xi>|x|\}dx\nonumber\\
&\le& 4A\|f'\|_\infty \P\{\Xi>|y|-A\}.
\end{eqnarray}
We have an upper bound
\begin{eqnarray*}
\Delta_n &:=&
\Biggl|\int_{-\infty}^\infty v(x_n) H(x_n+dy)
\Biggl(\int_{-A}^A f(x)P(x_n{+}y,x_n{+}dx)
-\int_{-A}^A f(x)F(dx{-}y)\Biggr)\Biggr|\\
&\le& \int_{-\infty}^\infty \Delta(n,y) v(x_n) H(x_n+dy).
\end{eqnarray*}
For any fixed $C>0$, \eqref{Delta.1.i} and \eqref{finite.bound.U.i} imply that
\begin{eqnarray*}
\int_{-C}^C \Delta(n,y) v(x_n)H(x_n+dy)
&\le& \sup_{y\in[-C,C]} \Delta(n,y)
\cdot\sup_n \bigl( v(x_n)H[x_n-C,x_n+C]\bigr)\\
&\to& 0 \quad\mbox{as }n\to\infty.
\end{eqnarray*}
The remaining part of the integral can be estimated by \eqref{Delta.2.i}:
\begin{eqnarray*}
\lefteqn{\limsup_{n\to\infty}\int_{|y|\ge C} \Delta(n,y) v(x_n)H(x_n+dy)}\\
&\le& 4A\|f'\|_\infty \limsup_{n\to\infty}
\int_{|y|\ge C} \P\{\Xi>|y|-A\} v(x_n)H(x_n+dy).
\end{eqnarray*}
Since $\Xi$ has finite mean, the property \eqref{finite.bound.U.i} 
of the renewal measure $H$ allows us to choose a sufficiently large $C$
in order to make the `$\limsup$' as small as we please.
Therefore, $\Delta_n \to 0$ as $n\to\infty$.
Hence, \eqref{conv.f.4.i} has the same limit as the sequence of integrals
\begin{eqnarray*}
\int_{-\infty}^\infty v(x_n)H(x_n+dy)
\int_{-A}^A f(x)F(dx-y).
\end{eqnarray*}
Now the weak convergence to $\lambda$
implies that \eqref{conv.f.4.i} has the limit
\begin{eqnarray}\label{conv.f.5.i}
\int_{-\infty}^\infty \lambda(dy)
\int_{-\infty}^\infty f(x)F(dx-y)
&=& \int_{-\infty}^\infty f(x)
\int_{-\infty}^\infty F(dx-y) \lambda(dy)\nonumber\\
&=& \int_{-\infty}^\infty f(x) (F*\lambda)(dx).
\end{eqnarray}
By \eqref{conv.f.1.i}--\eqref{conv.f.3.i}
and \eqref{conv.f.5.i}, we conclude the identity
\begin{eqnarray*}
\int_{-\infty}^\infty f(x)\lambda(dx)
&=& \int_{-\infty}^\infty f(x) (F*\lambda)(dx).
\end{eqnarray*}
Since the last identity holds for any smooth function $f$ 
with a bounded support, the measures $\lambda$ and $F*\lambda$ coincide
and the proof is complete.
\qed\end{proof} 

Further we use the following statement which is due to 
Choquet and Deny~\cite{CD}.

\begin{proposition}\label{l.2.i}
Let $F$ be a distribution not concentrated at $0$. Let $\lambda$ be a nonnegative measure
satisfying the equality $\lambda=\lambda*F$ and the property
$\sup\limits_{n\in\Z}\lambda[n,n+1]<\infty$.

If $F$ is non-lattice, then $\lambda$ is proportional to the Lebesgue measure.

If $F$ is lattice with minimal span $1$ and $\lambda(\R\setminus\Z)=0$, 
then $\lambda$ is proportional to the counting measure.
\end{proposition}

The concluding part of the proof of Theorem \ref{thm:ah.renewal.i} 
will be carried out for the non-lattice case. 
Choose any sequence of points $x_n\to\infty$ such that the measure 
$v(x_n)H(x_n+\cdot)$ converges weakly to some measure $\lambda$ as $n\to\infty$. 
It follows from Lemma \ref{l.1.i} and Proposition \ref{l.2.i} that then 
$\lambda(dx)=\alpha\cdot dx$ with some $\alpha$, i.e.,
\begin{eqnarray*}
v(x_n)H(x_n+dx) &\Rightarrow& \alpha\cdot dx\ \mbox{ as }n\to\infty.
\end{eqnarray*}
Then, for any $A>0$ and $k\in\{0,1,2,\ldots\} $, 
$$
v(x_n)H(x_n+kA, x_n+(k+1)A] \to \alpha A.
$$
Then, there exists a sufficiently slowly growing sequence 
$t_n\uparrow\infty$ such that 
$$
\frac{v(x_n)H(x_n, x_n+t_n]}{t_n} \to \alpha.
$$
It follows from the assumption \eqref{eq.growing.intervals} that $\alpha=C_H$.  

We complete the proof of the local limit of the renewal measure 
by contradiction argument. 
Suppose there exists a sequence $\{x_n\}$ such that 
\begin{equation}\label{eq_contr}
v(x_n) H(x_n,x_n+h]\ \not\to\ C_Hh\quad\mbox{as }n\to\infty. 
\end{equation}
However, by  Helly's Selection Theorem and  arguments above there exists a further subsequence ${x_{n_k}}$ for 
which 
$$
v(x_{n_k})H(x_{n_k}, x_{n_k}+h] \to C_Hh,
$$
which contradicts \eqref{eq_contr}.

Now let us show the uniform integrability in the lattice case,
to prove it, let us first notice that the Markov property implies
\begin{eqnarray}\label{E.N.bb}
\lefteqn{\P\Bigl\{\sum_{n=1}^\infty\I\{X_n=x\}>N\Bigr\}}\nonumber\\
&=& \P\{X_n=x\mbox{ for some }n\ge 0\}
\P^N\{X_n=x\mbox{ for some }n\ge 1\mid X_0=x\}\nonumber\\
&=& \P\{X_n=x\mbox{ for some }n\ge 0\}
\Bigl(1-\P\{X_n\not=x\mbox{ for all }n\ge 1\mid X_0=x\}\Bigr)^N.
\end{eqnarray}
We denote
$$
p_1(x)\ :=\ \P\{X_n=x\mbox{ for some }n\ge 0\};
$$
it tends to $1$ as $x\to\infty$ for the following reason.
For any fixed $\varepsilon>0$,
by the condition \eqref{majoriz.i} on jumps, there exists a 
sufficiently large $J$ such that, for all $X_0<x$,
\begin{eqnarray*}
1-\varepsilon &\le& \P\{X_n\in[x,x+J]\mbox{ for some }n\}\\
&=& \P\{X_n=x\mbox{ for some }n\}\\
&&\hspace{20mm} +\ \P\{X_n\in[x+1,x+J]\mbox{ for some }n,\
X_n\not= x\mbox{ for all }n\},
\end{eqnarray*}
and because the second probability on the right hand side
tends to zero as $x\to\infty$. Indeed, it is not greater than
\begin{eqnarray*}
\sum_{i=1}^J \P\{X_n=x+i\mbox{ for some }n,\ X_n\not= x\mbox{ for all }n\}
&\le& \sum_{i=1}^J \P\{X_n\not= x\mbox{ for all }n\mid X_0=x+i\}
\end{eqnarray*}
and the $i$th probability on the right hand side
converges as $x\to\infty$ to
\begin{eqnarray*}
\P\{S_n\not= 0\mbox{ for all }n\mid S_0=i\} &=& 0,
\end{eqnarray*}
due to $\E\xi=0$.  Hence, due to the arbitrary choice of $\varepsilon$,
$p_1(x)\to 1$ as $x\to\infty$. 

If follows from \eqref{E.N.bb} that
\begin{eqnarray*}
\lefteqn{\E\Bigl\{\sum_{n=1}^\infty\I\{X_n=x\};\ \sum_{n=1}^\infty\I\{X_n=x\}>N\Bigr\}}\\
&&\hspace{10mm}=\ \sum_{k=N}^\infty\P\Bigl\{\sum_{n=1}^\infty\I\{X_n=x\}>k\Bigr\}
+N\P\Bigl\{\sum_{n=1}^\infty\I\{X_n=x\}>N\Bigr\}\\
&&\hspace{20mm}=\ p_1(x)\bigl(1-p_2(x)\bigr)^N(N+1/p_2(x)).
\end{eqnarray*}
where
\begin{eqnarray*}
p_2(x) &=& \P\{X_n\not=x\mbox{ for all }n\ge 1\mid X_0=x\}.
\end{eqnarray*}
Taking into account that $p_1(x)\to 1$ and
\begin{eqnarray*}
\frac{p_1(x)}{p_2(x)} &=& \E\sum_{n=1}^\infty\I\{X_n=x\}
\ \sim\ \frac{C_H}{v(x)}\quad\mbox{as }x\to\infty,
\end{eqnarray*}
we derive asymptotics 
\begin{eqnarray*}
p_2(x) &\sim& p_1(x)\frac{v(x)}{C_H}\ \sim\ \frac{v(x)}{C_H}\quad\mbox{as }x\to\infty.
\end{eqnarray*}
Thus, for all sufficiently large $x$,
\begin{eqnarray*}
\frac{v(x)}{2C_H}\ \le\ p_2(x) &\le& \frac{2v(x)}{C_H},
\end{eqnarray*}
which yields, for all sufficiently large $x$,
\begin{eqnarray*}
\E\Bigl\{\sum_{n=1}^\infty\I\{X_n=x\};\ \sum_{n=1}^\infty\I\{X_n=x\}>N\Bigr\}
&\le& \bigl(1-v(x)/2C_H\bigr)^N(N+2C_H/\nu(x)),
\end{eqnarray*}
hence the required uniform integrability.
\qed\end{proof}
\end{theopargself}

\begin{theopargself}
\begin{proof}[{\rm of the uniform integrability of Corollary \ref{cor:regular}}]
It is enough to prove that, 
for some $h>0$ and $\widehat x\in\R$, the family of random variables
\begin{equation*}
\frac{1}{x}\sum_{n=0}^\infty \I\{X_n\in(x,x+h]\},\quad x\ge \widehat x,
\end{equation*}
is uniformly integrable. 
In its turn, by Lemma \ref{l:XY.renew.equiv},
it is sufficient to prove the last result for a Markov chain $\{Y_n\}$ with jumps
\begin{eqnarray*}
\eta(x) &:=& \max(\xi(x),\ -s(x)).
\end{eqnarray*}
This Markov chain satisfies all the conditions of Corollary \ref{c:est.for.return.power}
for all $\delta\in(0,2\mu/b-1)$. By the Markov property,
\begin{eqnarray*}
\P\Bigl\{\sum_{n=1}^\infty\I\{Y_n\in(x,x+h]\}>N\Bigr\}
&\le& \P\{Y_n\in(x,x+h]\mbox{ for some }n\ge 0\}\\
&&\hspace{-10mm}\times\sup_{y\in(x,x+h]}\P^N\{Y_n\in(x,x+h]\mbox{ for some }n\ge 1\mid Y_0=y\}\\
&&\hspace{-10mm}\le\sup_{y\in(x,x+h]}\P^N\{Y_n\in(x,x+h]\mbox{ for some }n\ge 1\mid Y_0=y\}.
\end{eqnarray*}
Therefore,
\begin{eqnarray*}
\lefteqn{\P\Bigl\{\sum_{n=1}^\infty\I\{Y_n\in(x,x+h]\}>N\Bigr\}}\\
&&\hspace{20mm}\le\Bigl(1-
\inf_{y\in(x,x+h]}\P\{Y_n\not\in(x,x+h]\mbox{ for all }n\ge 1\mid Y_0=y\}\Bigr)^N.
\end{eqnarray*}
Let us choose $h>0$ such that
\begin{eqnarray*}
p\ :=\ \P\{\xi>3h\} &>& 0,
\end{eqnarray*}
and then $\widehat x$ such that
\begin{eqnarray*}
\P\{\eta(x)>2h\} &\ge& \P\{\xi>3h\}=p>0\quad\mbox{for all }x>\widehat x,
\end{eqnarray*}
which is possible due to the asymptotic homogeneity \eqref{asymp.hom.i}.
Under such choice of $h$ and $\widehat x$, for all $x>\widehat x$ and $y\in(x,x+h]$,
\begin{eqnarray*}
\lefteqn{\P\{Y_n\not\in(x,x+h]\mbox{ for all }n\ge 1\mid Y_0=y\}}\\
&&\hspace{15mm}\ge\
\P\{\eta(y)>2h\}\P\{Y_n>x+h\mbox{ for all }n\ge 1\mid Y_0>x+2h\}\\
&&\hspace{30mm}\ge\ 
p\P\{Y_n>x+h\mbox{ for all }n\ge 1\mid Y_0>x+2h\}.
\end{eqnarray*}

As follows from Theorem \ref{l:est.for.return}, for all $z>x+2h$,
\begin{eqnarray*}
\P\{Y_n>x+h\mbox{ for all }n\ge 1\mid Y_0=z\} &=&
1-\P\{Y_n\le x+h\mbox{ for some }n\ge 1\mid Y_0=z\}\\
&\ge& 1-\biggl(\frac{x+h}{x+2h}\biggr)^\delta\\ 
&\sim& \delta h/x\quad\mbox{as }x\to\infty,
\end{eqnarray*}
which in its turn implies that, for all sufficiently large $x$,
\begin{eqnarray*}
\inf_{y\in(x,x+h]}\P\{Y_n\not\in(x,x+h]\mbox{ for all }n\ge 1\mid Y_0=y\}
&\ge& p\delta h/2x,
\end{eqnarray*}
and then
\begin{eqnarray*}
\P\Bigl\{\sum_{n=1}^\infty\I\{Y_n\in(x,x+h]\}>N\Bigr\}
&\le& (1-c/x)^N\quad\mbox{where }c=p\delta h/2>0.
\end{eqnarray*}
which implies the required uniform integrability.
\qed\end{proof}
\end{theopargself}

\begin{theopargself}
\begin{proof}[{\rm of the uniform integrability of Corollaries 
\ref{cor:critical} and \ref{cor:weibull} is the same}]
\qed\end{proof}
\end{theopargself}

\section{Comments to Chapter \ref{ch:asy.renewal}}

The renewal theory for a random walk with positive drift -- which is
the simplest example of a transient Markov chain (spatially and temporally 
homogeneous) -- has been intensively studied since 1940s. 
The integral (elementary) renewal theorem for a random walk with positive
jumps and finite mean goes back to Feller\index{Feller} \cite{Feller1941}
and states that $H(0,x]\sim x/\E\xi_1$ as $x\to\infty$. 
A more detailed information is available via the local renewal theorem, 
which was proved for lattice random variables in~\cite{EFP1949} and 
for non-lattice random variables in~\cite{B48}.
In the finite mean case the local renewal theorem 
gives the following sharp asymptotics  
$H(x,x+h]\to h/\E\xi_1$ as $x\to\infty$, for any fixed $h>0$. 
Later Blackwell\index{Blackwell} extended in~\cite{B53} the local renewal theorem 
to the case of i.i.d. random variables with positive mean 
that can take values of both signs using the important concept
of what was called by Feller ladder heights and ladder epochs. 
Original Blackwell's proof was considered to be quite complicated 
and a number of attempts were made to give an easier proof. 
A rather simple proof was given by Feller and Orey~\cite{FO61}, \index{Orey}
see also~\cite{Feller}. 
Further studies also considered behaviour of the remainder 
in the local renewal theorem, see~\cite{R1977}  and references therein. 
In the infinite mean case the asymptotics in Blackwell's theorem was not sharp. 
In 1960-70s a local renewal theorem was proved for regularly varying 
increments of index~$\alpha>1/2$, see~\cite{Garcia1962} and
\cite{Erickson70}. 
Subsequently there have been various improvements on these results, 
but the complete answer has been obtained very recently, 
see~\cite{CaravennaDoney2016}.

There exists a number of generalisations of the renewal theorem
for various stochastic processes. 
A natural extension is one for non-homogeneous (in time) random walk, 
that is a random walk with independent, 
but not necessarily identically distributed increments. 
Probably the first result in this direction was~\cite{CoxSmith1953}, 
where the local renewal theorem was derived from 
the local central limit theorem for  a non-homogeneous random walk. 
Further extensions may be found in~\cite{Smith1961, Williamson1965, Maejima1975}.
Renewal theorems for multidimensional random walks may be found 
in~\cite{Doney1966},\cite{Nagaev1980},~\cite{GH2013} and  recent paper~\cite{Berger2019}, see also references therein. 

The Markov setting has mostly been considered in the literature 
for the case of Markov modulated random walks, 
see, e.g.~\cite{Kesten1974,AMN1978,KlP2003} and~\cite{Shurenkov1985}.
In this setting one can usually use the Harris regeneration 
and split the process into independent cycles. 
Then, the traditional setting of Blackwell's theorem can be used. 

For the results cited above, it is essential that
the underlying process possesses some independence structure. 
In the present chapter we consider transient Markov chains 
where the cycle structure is not available 
which makes reduction to Blackwell's theorem impossible. 
Clearly, in order to observe some regular asymptotics 
for the renewal process, we need to assume some regular behaviour
of the Markov chain at infinity. In particular, if the drift of $X$,
$m_1(x)$, has a positive limit at infinity, say $a$, then
the local renewal result, $H(x,x+h]\to\ h/a$, 
is only known for an asymptotically homogeneous in space Markov chain,
see~\cite{Kor08}.
\chapter{Doob's $h$-transform:
transition from recurrent to transient chain and vice versa}
\chaptermark{Doob's $h$-transform}
\label{ch:change}

This short chapter is the most conceptual part of the book.
Our purpose here is to describe, without superfluous details, 
a change of measure strategy, which allows us to transform a recurrent chain 
into a transient one, and vice versa. It is motivated by the exponential 
change of measure technique which goes back to Cram\'er \cite{Cramer}
where, in the context of large deviations in the collective risk theory, 
it allows us to transform a negatively drifted random walk into one with positive drift.
Doob's $h$-transform is the most natural substitution for the exponential change
of measure in the context of Lamperti's problem,
that is, in the context of Markov chains with asymptotically zero drift.

Such transformations connect naturally previous chapters 
on asymptotic behaviour of transient chains with subsequent chapters, 
which are devoted to recurrent chains.

A very important, in comparison with the classical Doob's $h$-transform, 
novelty consists in the fact that we use weight functions 
which are not necessarily harmonic, they are only asymptotically harmonic at infinity.
The main challenge is to identify such functions under various drift scenarios.

\section{Doob's $h$-transform for transition kernels}
\label{sec:change}

\subsection{General change of measure methodology for transition kernels}

Let $S$ be a measurable space with a $\sigma$-algebra $\mathcal A(S)$.
%of measurable sets.
Let $P(x,A):S\times{\mathcal A}(S)\to\R^+$ be a non-negative transition
kernel on $S$, that is, it is measurable in $x$ for all fixed $A$ 
and it is a non-negative measure in $A$ for all fixed $x$.\index{Transition kernel}
It is not necessarily stochastic.

Let $U(x)>0$ be a positive measurable function such that
\begin{eqnarray}\label{cond.for.U}
\int_S U(y)P(x,dy) &<& \infty\quad\mbox{for all }x\in S;
\end{eqnarray}
such a function $U$ is called a {\it weight function}.
Then it allows us to define a new transition kernel
\begin{eqnarray*}
Q(x,A) &:=& \int_A \frac{U(y)}{U(x)}P(x,dy),
\end{eqnarray*}
which is just {\it Doob's $h$-transform}\index{Doob's $h$-transform}
for $P$ with weight function $U$.
If $U$ is a harmonic function for $P$, that is, if
\begin{eqnarray*}
U(x) &=& \int_S U(y)P(x,dy)\quad\mbox{for all }x\in S,
\end{eqnarray*}
then $Q$ is a transition probability kernel.

In order to ensure that the powers of $Q$ are well-defined,
we need to strengthen the condition \eqref{cond.for.U} as follows:
\begin{eqnarray}\label{cond.for.U.unif}
c_S\ :=\ \sup_{x\in S}\int_S \frac{U(y)}{U(x)}P(x,dy) &<& \infty.
\end{eqnarray}
Then it is legible to carry out the following standard calculations
\begin{eqnarray*}
Q^n(x,A) &:=&
\int_S Q(x,dy_1)\ldots\int_S Q(y_{n-2},dy_{n-1}) \int_A Q(y_{n-1},dy_n)\\
&=& \int_S \frac{U(y_1)}{U(x)} P(x,dy_1) \ldots
\int_S \frac{U(y_{n-1})}{U(y_{n-2})} P(y_{n-2},dy_{n-1})
\int_A \frac{U(y_n)}{U(y_{n-1})} P(y_{n-1},dy_n)\\
&=& \int_S \frac{U(y_n)}{U(x)} P(x,dy_1) \ldots
\int_S P(y_{n-2},dy_{n-1}) \int_A P(y_{n-1},dy_n)\\
&=& \int_A \frac{U(y_n)}{U(x)} P^n(x,dy_n)
\end{eqnarray*}
which shows that Doob's $h$-transform of the $n$th power of $P$, $P^n$,
is equal to the $n$th power of Doob's $h$-transform of $P$, $Q^n$.
Similarly, for any collection of sets $A_1$, \ldots, $A_n\in{\mathcal A}(S)$,
\begin{eqnarray*}
\lefteqn{\int_{A_1} Q(x,dy_1)\ldots\int_{A_{n-1}} Q(y_{n-2},dy_{n-1})
\int_{A_n} Q(y_{n-1},dy_n)}\\
&=& \int_{A_1} P(x,dy_1) \ldots
\int_{A_{n-1}} P(y_{n-2},dy_{n-1}) \int_{A_n}\frac{U(y_n)}{U(x)} P(y_{n-1},dy_n).
\end{eqnarray*}
Performing the inverse change of measure we get\index{Doob's $h$-transform!inverse}
\begin{eqnarray}\label{connection.new.0}
P^n(x,dy) &=& \frac{U(x)}{U(y)}Q^n(x,dy)
\end{eqnarray}
and
\begin{eqnarray}\label{h.transform.path}
\lefteqn{\int_{A_1} P(x,dy_1)\ldots\int_{A_{n-1}} P(y_{n-2},dy_{n-1})
\int_{A_n} P(y_{n-1},dy_n)}\nonumber\\
&=& \int_{A_1} Q(x,dy_1) \ldots
\int_{A_{n-1}} Q(y_{n-2},dy_{n-1}) \int_{A_n}\frac{U(x)}{U(y_n)} Q(y_{n-1},dy_n).
\end{eqnarray}

Denote
\begin{eqnarray*}
q(x) &:=& -\log Q(x,S).
\end{eqnarray*}
Let us consider the following normalised kernel
$$
\widehat P(x,dy)\ =\ \frac{Q(x,dy)}{Q(x,S)}\ =\ Q(x,dy)e^{q(x)}
$$
and let $\{\widehat X_n\}$ be a Markov chain
with these transition probabilities. Then
$$
Q(x,dy)\ =\ \widehat P(x,dy)e^{-q(x)}
$$
and hence, by \eqref{connection.new.0},
we arrive at the following basic equalities:
\begin{eqnarray}\label{connection.new.1}
P^n(x,dy) &=& \frac{U(x)}{U(y)}
\E_x\bigl\{e^{-\sum_{k=0}^{n-1}q(\widehat X_k)};\ \widehat X_n\in dy\bigr\}
\end{eqnarray}
and
\begin{eqnarray}\label{connection.new.1.path}
\lefteqn{\int_{A_1}P(x,y_1)\ldots \int_{A_{n-1}}P(y_{n-2},dy_{n-1})
P(y_{n-1},dy_n)}\nonumber\\
&&\hspace{1mm} =\ \frac{U(x)}{U(y_n)}
\E_x\bigl\{e^{-\sum_{k=0}^{n-1}q(\widehat X_k)};\
\widehat X_1\in A_1,\ldots,\widehat X_{n-1}\in A_{n-1},\widehat X_n\in dy_n\bigr\}.
\end{eqnarray}

\subsection{Application to killed Markov chain}

In this subsection we specify how the above transformation works
in the case that we are mostly interested in---the transition
kernel corresponding to a Markov chain killed at entering some fixed set.
Namely, let $\{X_n\}$ be a Markov chain with transition probabilities $P(\cdot,\cdot)$,
let $B\subset S$ be some fixed set, and let $\tau_B:=\min\{n\ge 1:X_n\in B\}$.
Consider a substochastic transition kernel
$$
P_B(x,A)\ :=\ P(x,A\setminus B)\ =\ \P_x\{X_1\in A,\ \tau_B>1\},
$$
which is the transition kernel corresponding to $\{X_n\}$
{\it killed}\index{Markov chain!killed} at entering $B$.

Given a weight function $U(x)>0$ for all $x\notin B$, 
the corresponding change of measure produces a transition kernel $Q$ 
which may be rewritten as follows
\begin{eqnarray}\label{def.Q.B}
Q(x,dy) &:=& \frac{U(y)}{U(x)}\P_x\{X_1\in dy,\tau_B>1\}\nonumber\\
&=& \frac{U(y)}{U(x)}\P_x\{X_1\in dy,X_1\notin B\}.
\end{eqnarray}

Consequently, performing the inverse change of measure we arrive
at the following basic equality:
\begin{eqnarray}\label{connection.new.B}
\P_x\{X_n\in dy,\tau_B>n\} &=& \frac{U(x)}{U(y)}Q^n(x,dy)\nonumber\\
&=& \frac{U(x)}{U(y)}
\E_x\bigl\{e^{-\sum_{k=0}^{n-1}q(\widehat X_k)};\ \widehat X_n\in dy\bigr\},
\end{eqnarray}
where
\begin{eqnarray}\label{def.q.B}
q(x) &:=& -\log \int_{S\setminus B} \frac{U(y)}{U(x)}P(x,dy)
\end{eqnarray}
and $\{\widehat X_n\}$ is a Markov chain with transition probabilities
\begin{eqnarray}\label{P.hat.B}
\widehat P(x,A) &=& \frac{Q(x,A)}{Q(x,S)}\ =\
\frac{\int_{A\setminus B} U(y) P(x,dy)} {\int_{S\setminus B} U(y) P(x,dy)}.
\end{eqnarray}
In other words, for any Borel function $f(y)$,
\begin{eqnarray}\label{connection.new.B.f}
\E_x\{f(X_n);\ \tau_B>n\} &=& U(x)\int_{S\setminus B}\frac{f(y)}{U(y)}
\E_x\bigl\{e^{-\sum_{k=0}^{n-1}q(\widehat X_k)};\ \widehat X_n\in dy\bigr\}
\nonumber\\
&=& U(x)\E_x\biggl\{e^{-\sum_{k=0}^{n-1}q(\widehat X_k)}
\frac{f(\widehat X_n)}{U(\widehat X_n)}\biggr\}.
\end{eqnarray}

\section{How to increase drift via change of measure
with weight function close to harmonic function}
\sectionmark{How to increase drift via change of measure}
\label{sec:change.spec}

\subsection{Stochastic kernel}

Let $\{X_n\}$ be a Markov chain on $\R$ with jumps $\xi(x)$.
Let, for some increasing function $s(x)$
and decreasing function $r(x)\to 0$ as $x\to\infty$,
\begin{eqnarray}\label{cond.m1m2r}
\frac{2m_1^{[s(x)]}(x)}{m_2^{[s(x)]}(x)} &\sim& -r(x),\\
\label{cond.m2b}
m_2^{[s(x)]}(x) &\to& b>0.
\end{eqnarray}

If we want to increase the drift---say if we need to pass from
a recurrent Markov chain to a transient one, then clearly
an increasing weight should be applied.
So, let $U(x)\ge 0$ be an increasing differentiable function such that,
for some $c_U>0$,
\begin{eqnarray}\label{cond.UprimeU.frac}
\frac{U'(x)}{U(x)} &\sim& c_Ur(x)\quad\mbox{as }x\to\infty
\end{eqnarray}
and
\begin{eqnarray}\label{cond.UprimeU}
U(x+y)\ \sim\ U(x) &\mbox{and}& U'(x+y)\ \sim\ U'(x)
\end{eqnarray}
as $x\to\infty$ uniformly for all $|y|\le s(x)$.

We assume that $U$ is close to a harmonic function in the following sense:
\begin{eqnarray}\label{cond.Uclose.harmonic}
\E_x U(X_1)=\E U(x+\xi(x)) &\sim& U(x)\quad\mbox{as }x\to\infty.
\end{eqnarray}
This condition provides the asymptotic stochasticity of $Q$,
that is, $Q(x,\R)\to 1$ as $x\to\infty$.

Let $Q$, $\widehat P(\cdot,\cdot)$, $\{\widehat X_n\}$, and $\widehat\xi(x)$
be defined for $P(\cdot,\cdot)$ with weight function $U$
as described in the last section.

\begin{lemma}\label{l:change.B}
Let conditions \eqref{cond.m1m2r}--\eqref{cond.Uclose.harmonic} hold. Then
\begin{eqnarray}\label{m1.sim.rx.B}
\E\{\widehat{\xi}(x);\ |\widehat{\xi}(x)|\le s(x)\} &\sim& (c_U-1/2)br(x),\\
\label{m2.to.b.B}
\E\{(\widehat{\xi}(x))^2;\ |\widehat{\xi}(x)|\le s(x)\} &\to& b
\end{eqnarray}
as $x\to\infty$, so hence
\begin{eqnarray*}
\frac{2\widehat m_1^{[s(x)]}(x)}{\widehat m_2^{[s(x)]}(x)} &\sim& (2c_U-1)r(x).
\end{eqnarray*}
In addition,
\begin{eqnarray}\label{cond.xi.le.B}
\P\{\widehat\xi(x)<-s(x)\} &\le& (1+o(1))\P\{\xi(x)<-s(x)\},\\
\label{cond.xi.le.B.e}
\E\{|\widehat\xi(x)|;\ \widehat\xi(x)<-s(x)\} &\le&
(1+o(1))\E\{|\xi(x)|;\ \xi(x)<-s(x)\},\\
\label{cond.xi.ge.B}
\P\{\widehat\xi(x)>s(x)\} &\le& (1+o(1))
\frac{\E\bigl\{U(x+\xi(x));\ \xi(x)>s(x)\bigr\}}{U(x)}.
\end{eqnarray}
\end{lemma}

\begin{proof}
By the construction of $\{\widehat X_n\}$ and the condition
\eqref{cond.Uclose.harmonic},
\begin{eqnarray*}
\E\{\widehat\xi(x);\ |\widehat\xi(x)|\le s(x)\} &=&
\frac{\E\{U(x+\xi(x))\xi(x);\ |\xi(x)|\le s(x)\}}{\E U(x+\xi(x))}\\
&\sim& \frac{\E\{U(x+\xi(x))\xi(x);\ |\xi(x)|\le s(x)\}}{U(x)}.
\end{eqnarray*}
By Taylor's theorem,
\begin{eqnarray*}
\E\{U(x+\xi(x))\xi(x);\ |\xi(x)|\le s(x)\} &=&
U(x)\E\{\xi(x);\ |\xi(x)|\le s(x)\}\\
&&\hspace{2mm}+\E\{U'(x+\theta\xi(x))\xi^2(x);\ |\xi(x)|\le s(x)\},
\end{eqnarray*}
where $\theta=\theta(x,\xi(x))\in(0,1)$.
The first term on the right hand side is equivalent to $-bU(x)r(x)/2$,
as follows from \eqref{cond.m1m2r} and \eqref{cond.m2b}.
By the condition \eqref{cond.UprimeU},
$U'(x+\theta\xi(x))\sim U'(x)$ as $x\to\infty$
uniformly for all $|\xi(x)|\le s(x)$ which implies, as $x\to\infty$,
\begin{eqnarray*}
\E\{U'(x+\theta\xi(x))\xi^2(x);\ |\xi(x)|\le s(x)\}
&\sim& U'(x)\E\{\xi^2(x);\ |\xi(x)|\le s(x)\}\\
&\sim& U'(x)b\ \sim\ c_Ub r(x) U(x),
\end{eqnarray*}
due to the conditions \eqref{cond.m2b} and \eqref{cond.UprimeU.frac}.
Altogether yields that
\begin{eqnarray*}
\E\{U(x+\xi(x))\xi(x);\ |\xi(x)|\le s(x)\} &\sim&
(c_U-1/2)br(x)U(x)\quad\mbox{as }x\to\infty,
\end{eqnarray*}
and \eqref{m1.sim.rx.B} follows. The second result, \eqref{m2.to.b.B},
follows if we apply \eqref{cond.m2b}, \eqref{cond.UprimeU}, 
and \eqref{cond.Uclose.harmonic} to the right hand side of
\begin{eqnarray*}
\E\{\widehat\xi^2(x);\ |\widehat\xi(x)|\le s(x)\} &=&
\frac{\E\{U(x+\xi(x))\xi^2(x);\ |\xi(x)|\le s(x)\}}{\E U(x+\xi(x))}.
\end{eqnarray*}

Using \eqref{cond.Uclose.harmonic} and recalling that $U$ is increasing,
we also get
\begin{eqnarray*}
\P\{\widehat{\xi}(x)<-s(x)\}
&=& \frac{\E\{U(x+\xi(x));\xi(x)<-s(x)\}}{\E U(x+\xi(x))}\\
&\sim& \frac{\E\{U(x+\xi(x));\xi(x)<-s(x)\}}{U(x)}\\
&\le& \P\{\xi(x)<-s(x)\},
\end{eqnarray*}
and similarly for \eqref{cond.xi.le.B.e}.
The last assertion, \eqref{cond.xi.ge.B},
follows again by \eqref{cond.Uclose.harmonic},
and the proof is complete.
\qed\end{proof}

\subsection{Killed Markov chain}

Let $\widehat x\in\R^+$ be some level.
For $\{X_n\}$ killed at entering $B:=(-\infty,\widehat x]$,
let us perform the change of measure with an increasing weight function $U$
and consider the corresponding kernel $Q$,
\begin{eqnarray}\label{Q.x.A.B}
Q(x,A) &=&
\frac{\E\{U(x+\xi(x));\ x+\xi(x)\in A\cap(\widehat x,\infty)\}}{U(x)},
\end{eqnarray}
and the embedded Markov chain $\{\widehat X_n\}$ with transition probabilities
\begin{eqnarray}\label{P.x.A.B}
\widehat P(x,A) &=&
\frac{\E\{U(x+\xi(x));\ x+\xi(x)\in A\cap(\widehat x,\infty)\}}
{\E\{U(x+\xi(x));\ x+\xi(x)>\widehat x\}},
\end{eqnarray}
if $\P\{x+\xi(x)>\widehat x\}>0$ and $\widehat P(x,A)=\I\{x\in A\}$ otherwise.
Let $\widehat\xi(x)$ be the jumps of $\{\widehat X_n\}$.

The following result is almost immediate from Lemma \ref{l:change.B}.

\begin{lemma}\label{l:change.B.0}
Let the conditions \eqref{cond.m1m2r}--\eqref{cond.Uclose.harmonic} hold
for some $s(x)\le x/2$ and let
\begin{eqnarray}\label{le.widehatx.1}
\P\{x+\xi(x)\le \widehat x\} &\to& 0
\quad\mbox{as }x\to\infty.
\end{eqnarray}
Then the conclusions \eqref{m1.sim.rx.B}--\eqref{cond.xi.ge.B} hold.

If, in addition, $c_U>1/2$ and \eqref{le.widehatx.1} holds for any $\widehat x$, 
then there exists a sufficiently large $\widehat x$ such that
\begin{eqnarray}\label{m1.ge.rx.B}
\E\{\widehat{\xi}(x);\ |\widehat{\xi}(x)|\le s(x)\} &\ge&
\frac{c_U-1/2}{2}br(x)
\quad\mbox{for all }x\ge\widehat x.
\end{eqnarray}
\end{lemma}

\begin{proof}
By the condition \eqref{le.widehatx.1},
\begin{eqnarray*}
\E\{U(x+\xi(x));\ x+\xi(x)\le\widehat x\} &\le&
U(\widehat x)\P\{x+\xi(x)\le\widehat x\}\ \to\ 0
\end{eqnarray*}
as $x\to\infty$, hence
\begin{eqnarray*}
\E\{U(x+\xi(x));\ x+\xi(x)>\widehat x\} &\sim& \E U(x+\xi(x))\ \sim\ U(x),
\end{eqnarray*}
owing to \eqref{cond.Uclose.harmonic}.
Thus \eqref{m1.sim.rx.B}--\eqref{cond.xi.ge.B} follow from Lemma \ref{l:change.B}
due to $s(x)\le x/2$.
\qed\end{proof}

\section{How to decrease drift via change of measure
with weight function close to harmonic function}
\sectionmark{How to decrease drift via change of measure}
\label{sec:change.spec.de}

\subsection{Stochastic kernel}

In this section let $\{X_n\}$ be a Markov chain on $\R$ such that,
for some increasing function $s(x)$
and decreasing function $r(x)\to 0$ as $x\to\infty$,
\begin{eqnarray}\label{cond.m1m2r.de}
\frac{2m_1^{[s(x)]}(x)}{m_2^{[s(x)]}(x)} &\sim& r(x),\\
\label{cond.m2b.de}
m_2^{[s(x)]}(x) &\to& b>0.
\end{eqnarray}

If we want to decrease the drift---say if we need to pass from
a transient Markov chain to a recurrent one, then clearly
a decreasing weight should be applied.
So, let $U(x)>0$ be a decreasing differentiable function such that
\eqref{cond.UprimeU.frac} for some $c_U<0$ and \eqref{cond.UprimeU} hold.
As in the previous section, we again assume that $U$ is close
to a harmonic function in the sense \eqref{cond.Uclose.harmonic}.

In the same way as Lemma \ref{l:change.B}, the following result follows.

\begin{lemma}\label{l:change.B.de}
Let conditions \eqref{cond.m1m2r.de}, \eqref{cond.m2b.de}
and \eqref{cond.UprimeU.frac}--\eqref{cond.Uclose.harmonic} hold. Then
\begin{eqnarray}\label{m1.sim.rx.B.de}
\E\{\widehat{\xi}(x);\ |\widehat{\xi}(x)|\le s(x)\} &\sim& (c_U+1/2)br(x),\\
\label{m2.to.b.B.de}
\E\{(\widehat{\xi}(x))^2;\ |\widehat{\xi}(x)|\le s(x)\} &\to& b
\end{eqnarray}
as $x\to\infty$, hence
\begin{eqnarray}\label{m1.m2.new}
\frac{2\widehat m_1^{[s(x)]}(x)}{\widehat m_2^{[s(x)]}(x)} &\sim& (2c_U+1)r(x).
\end{eqnarray}
In addition,
\begin{eqnarray}\label{cond.xi.le.B.de}
\P\{\widehat\xi(x)>s(x)\} &\le& (1+o(1))\P\{\xi(x)>s(x)\},\\
\label{cond.xi.ge.B.de}
\P\{\widehat\xi(x)<-s(x)\} &\le& (1+o(1))
\frac{\E\bigl\{U(x+\xi(x));\ \xi(x)<-s(x)\bigr\}}{U(x)}.
\end{eqnarray}
\end{lemma}

\subsection{Killed Markov chain}

Let $\widehat x\in\R^+$ be some level.
For $\{X_n\}$ killed at entering $B:=(-\infty,\widehat x]$,
let us perform the change of measure with a decreasing weight function $U$
and consider the corresponding kernel $Q$
and the embedded Markov chain $\{\widehat X_n\}$.

Then similarly to Lemma \ref{l:change.B.0} we get the following result.

\begin{lemma}\label{l:change.B.0.de}
Let the conditions \eqref{cond.m1m2r.de}, \eqref{cond.m2b.de},
and \eqref{cond.UprimeU.frac}--\eqref{cond.Uclose.harmonic} hold.
Then the conclusions \eqref{m1.sim.rx.B.de}--\eqref{cond.xi.ge.B.de} follow.
\end{lemma}

\section{Cycle structure of Markov chain and Doob's transform}
\sectionmark{Cycle structure and Doob's transform}
\label{sec:cycle}

Let a Markov chain $\{X_n\}$ on $\R$ be recurrent in the sense that,
for some $\widehat x\in\R$, the set $(-\infty,\widehat x]$ is recurrent,
that is,
\begin{eqnarray}\label{a.pi}
\P_x\{\tau_{(-\infty,\widehat x]}<\infty\} &=& 1
\quad\mbox{for all }x>\widehat x.
\end{eqnarray}
Let $\{X_n\}$ possess a sigma-finite non-negative invariant measure $\pi$,
that is, a measure $\pi$ that solves the equation
\begin{eqnarray*}
\pi(A) &=& \int_\R P(x,A)\pi(dx)\quad\mbox{for all }A\in{\mathcal B}(\R);
\end{eqnarray*}
we do not assume that this invariant measure is unique.
It follows from \eqref{a.pi} that
\begin{eqnarray}\label{a.pi.ge0}
\pi(-\infty,\widehat x] &>& 0.
\end{eqnarray}
The case of a finite $\pi$ corresponds to positive recurrence
while infinite $\pi$ corresponds to null recurrence.

In addition, assume that
\begin{eqnarray}\label{a.pi.le.inf}
\pi(-\infty,\widehat x] &<& \infty.
\end{eqnarray}
The conditions \eqref{a.pi.ge0} and \eqref{a.pi.le.inf}
allow us to construct an {\it aggregated} Markov\index{Markov chain!aggregated}
chain $\{X_n^*\}$ on $[\widehat x,\infty)$
with the following transition probabilities: for $x>\widehat x$,
\begin{eqnarray}\label{aggr.1}
P^*(x,A) &=&
\left\{
\begin{array}{ll}
P(x,A) &\mbox{for } A\subseteq(\widehat x,\infty),\\
P(x,(-\infty,\widehat x]) &\mbox{for } A=\{\widehat x\},
\end{array}
\right.
\end{eqnarray}
and
\begin{eqnarray}\label{aggr.2}
P^*(\widehat x,A) &=&
\left\{
\begin{array}{ll}
\displaystyle\int_{(-\infty,\widehat x]}\frac{P(y,A)}
{\pi(-\infty,\widehat x]}\pi(dy)
&\mbox{for }A\subseteq(\widehat x,\infty),\\
\displaystyle\int_{(-\infty,\widehat x]}
\frac{P(y,(-\infty,\widehat x])}{\pi(-\infty,\widehat x]}\pi(dy)
&\mbox{for }A=\{\widehat x\}.
\end{array}
\right.
\end{eqnarray}
Then the measure $\pi^*$ which aggregates states from
$(-\infty,\widehat x]$ to $\widehat x$, that is,
$\pi^*\{\widehat x\}=\pi(-\infty,\widehat x]$
and $\pi^*(A)=\pi(A)$ for all $A\subseteq(\widehat x,\infty)$,
is an invariant measure for $\{X_n^*\}$.
We assume that the atom $\widehat x$ is non-degenerate, that is,
\begin{eqnarray}\label{xhat.nong}
P^*(\widehat x,\{\widehat x\}) &<& 1.
\end{eqnarray}

\begin{lemma}\label{l:suff.Harris}
Let

(i) either $\pi$ be a probability measure and
\begin{eqnarray}\label{pi.unb}
\pi(\widehat x,\infty) &>& 0;
\end{eqnarray}

(ii) or $\pi$ be sigma-finite and, for any initial state $X_0$,
\begin{eqnarray}\label{pi.unb.2}
\P\bigl\{\limsup_{n\to\infty} X_n>\widehat x\bigr\} &=& 1.
\end{eqnarray}
Then \eqref{xhat.nong} follows.
\end{lemma}

\begin{proof}
(i) Consider a stationary Markov chain $\{X_n\}$ having distribution
$\pi$ for all $n$. If $P^*(\widehat x,\{\widehat x\})=1$ then
\begin{eqnarray*}
\int_{(-\infty,\widehat x]}
P(y,(\widehat x,\infty))\P\{X_0\in dy\} &=& 0
\end{eqnarray*}
and hence
\begin{eqnarray*}
\P\{X_1>\widehat x\} &=& \int_{(\widehat x,\infty)}
P(y,(\widehat x,\infty))\P\{X_0\in dy\}\\
&=& \P\{X_0>\widehat x,X_1>\widehat x\}.
\end{eqnarray*}
By induction,
\begin{eqnarray*}
\P\{X_n>\widehat x\}
&=& \P\{X_0>\widehat x,\ldots,X_n>\widehat x\},
\end{eqnarray*}
hence recurrence of the set $(-\infty,\widehat x]$ implies
convergence $\P\{X_n>\widehat x\}\to 0$ as $n\to\infty$
which contradicts the stationarity of $\{X_n\}$ and \eqref{pi.unb}.

(ii) The condition \eqref{a.pi.le.inf} allows us to consider a Markov chain 
$\{X_n\}$ with initial distribution concentrated on $(-\infty,\widehat x]$,
\begin{eqnarray*}
\P\{X_0\in dy\} &=& \frac{\pi(dy)}{\pi(-\infty,\widehat x]},
\quad y\le\widehat x.
\end{eqnarray*}
If $P^*(\widehat x,\{\widehat x\})=1$ then
\begin{eqnarray*}
\P\{X_1\le\widehat x\} &=& \int_{(-\infty,\widehat x]}
P(y,(-\infty,\widehat x])\P\{X_0\in dy\}\\
&=& \int_{(-\infty,\widehat x]}\frac{P(y,(-\infty,\widehat x])}{\pi(-\infty,\widehat x]}\pi(dy)\\
&=& P^*(\widehat x,\{\widehat x\})\ =\ 1.
\end{eqnarray*}
By induction, then $\P\{X_n\le\widehat x\}=1$ for all $n$
which contradicts \eqref{pi.unb.2}.
\qed\end{proof}

So, under the conditions \eqref{a.pi}, \eqref{a.pi.le.inf}
and \eqref{xhat.nong} the aggregated Markov chain $\{X_n^*\}$
on $[\widehat x,\infty)$ is Harris recurrent with a non-degenerate
atom at state $\widehat x$---for definition see \cite{MT}---regardless of
whether $\pi$ is finite or not.
Then the following representation for the invariant measure
$\pi^*$ via cycle structure \index{Markov chain!cycle structure}
(generated by the atom $\widehat x$) of the Markov chain $\{X_n^*\}$
is well known---see, e.g. \cite[Theorem 10.4.9]{MT},
\begin{eqnarray}\label{pi-def.tilde}
\pi^*(dy) &=& \pi^*(\widehat x)\sum_{n=1}^{\tau^*_{\widehat x}-1}
\P_{\widehat x}\{X_n^*\in dy\}\nonumber\\
&=& \pi^*(\widehat x)\sum_{n=1}^\infty
\P_{\widehat x}\{X_n^*\in dy;\ \tau^*_{\widehat x}>n\},
\quad y>\widehat x,
\end{eqnarray}
where $\tau^*_{\widehat x}=\min\{n\ge 1:X_n^*=\widehat x\}$.
This is equivalent to the following representation
for the invariant measure $\pi$ of $\{X_n\}$:
\begin{eqnarray}\label{pi-def}
\pi(dy) &=& \int_B\pi(dz)\sum_{n=1}^\infty \P_z\{X_n\in dy;\ \tau_B>n\},
\quad y>\widehat x,
\end{eqnarray}
where $B=(-\infty,\widehat x]$.
By the Markov property,
\begin{eqnarray*}
\P_z\{X_n\in dy,\ \tau_B>n\} &=&
\int_{\widehat x}^\infty\P_z\{X_1\in dx\}\P_x\{X_{n-1}\in dy,\ \tau_B>n-1\}.
\end{eqnarray*}
Therefore, for $y>\widehat x$,
\begin{eqnarray*}
\pi(dy)&=&\int_B\pi(dz)\int_{\widehat{x}}^\infty\P_z\{X_1\in dx\}
\sum_{n=0}^\infty\P_x\{X_{n}\in dy,\ \tau_B>n\}\\
&=&\int_{\widehat{x}}^\infty\mu(dx)
\sum_{n=0}^\infty\P_x\{X_{n}\in dy,\ \tau_B>n\},
\end{eqnarray*}
where
\begin{eqnarray}\label{6.mu.B}
\mu(dx) &:=& \int_B\pi(dz)\P_z\{X_1\in dx\}\nonumber\\
&=& \int_B\pi(dz)P(z,dx)
\end{eqnarray}
is a measure on $(\widehat x,\infty)$.
Substituting here \eqref{connection.new.B}, we get
\begin{eqnarray*}
\pi(dy)&=&\frac{1}{U(y)}\int_{\widehat x}^\infty\mu(dx)U(x)
\sum_{n=0}^\infty\E_x\{e^{-\sum_{k=0}^{n-1}q(\widehat X_k)};\
\widehat X_n\in dy\}.
\end{eqnarray*}
Consider the chain $\{\widehat X_n\}$ with initial distribution
\begin{eqnarray}\label{initial.hat}
\P\{\widehat X_0\in dz\}=\frac{\mu(dz)U(z)}{c^*},
\quad z\in(\widehat x,\infty),
\end{eqnarray}
where $c^*$ is a normalising constant,
\begin{eqnarray*}
c^* &:=& \int_{\widehat x}^\infty\mu(dx)U(x)\\
&=& \int_B\pi(dz)\int_{\widehat x}^\infty U(x)P(z,dx).
\end{eqnarray*}
Then
\begin{eqnarray*}
\pi(dy) &=& \frac{\widehat H^{(q)}(dy)}{U(y)}c^*,
\end{eqnarray*}
where the weighted renewal measure $\widehat H^{(q)}$
for $\{\widehat X_n\}$ is defined as
\begin{equation}\label{def.Hq.new.B}
\widehat H^{(q)}(dy)=\sum_{n=0}^\infty
\E\{e^{-\sum_{k=0}^{n-1} q(\widehat X_k)};\ \widehat X_n\in dy\}.
\end{equation}
The constant $c^*$ is finite if 
\begin{eqnarray*}
\sup_{z\in B}\int_{\widehat x}^\infty U(x)P(z,dx) &<& \infty.
\end{eqnarray*}
Provided the condition \eqref{cond.for.U.unif} holds,
the constant $c^*$ possesses the following upper bound:
\begin{eqnarray}\label{upper.for.c*}
c^* &\le& c_S\int_B U(z)\pi(dz),
\end{eqnarray}
which is not greater than $c_S U(\widehat x)\pi(-\infty,\widehat x]$
if the function $U(x)$ is increasing.

The above calculations imply, in particular, that
\begin{eqnarray}\label{pi.repr.B}
\pi(x_1,x_2] &=& c^*\int_{x_1}^{x_2}\frac{\widehat H^{(q)}(dy)}{U(y)}.
\end{eqnarray}

So, the main idea for investigation of the invariant measure is
to identify an increasing test function $U(x)$ which is sufficiently close
to a harmonic function in a sense that its drift is sufficiently small
for large $x$ which implies small values of $q(x)$.
We also need to choose $U(x)$ in such a way that the chain
$\{\widehat X_n\}$ is transient. Then the factorisation result for
the renewal function $\widehat H^{(q)}$, see Section \ref{sec:renewal.w},
and an integro-local renewal theorem for $\{\widehat X_n\}$ allow us
to derive asymptotics for the tail distribution
of the invariant measure $\pi$.

\section{Last visit decomposition and Doob's transform}
\sectionmark{Last visit decomposition and Doob's transform}
\label{sec:lvisit}

For pre-stationary distribution of $X_n$, we follow the last visit decomposition approach.
Let $\widehat x\in\R$, set $B:=(-\infty,\widehat x]$.
Regardless recurrence or transience of $\{X_n\}$,
splitting the trajectory of $\{X_n\}$ by the last visit to $B$, 
we get, for $y>\widehat x$,\index{Markov chain!last visit decomposition}
\begin{eqnarray*}
\P\{X_n\in dy\} &=&
\sum_{j=1}^n \P\{X_{n-j}\in B,X_{n-j+1},\ldots,X_{n-1}\not\in B,X_n\in dy\}\\
&=& \sum_{j=1}^n\int_B\P\{X_{n-j}\in dz\}\int_{\widehat x}^\infty
P(z,du)\P_u\{X_{j-1}\in dy,\tau_B>j-1\}.
\end{eqnarray*}
Substituting \eqref{connection.new.B}, we obtain the following equality
\begin{eqnarray}\label{repr.Xn.B.U}
\P\{X_n\in dy\} &=&
\sum_{j=1}^n\int_B\P\{X_{n-j}\in dz\}\int_{\widehat x}^\infty
P(z,du)\frac{U(u)}{U(y)}
\E_u\Bigl\{e^{-\sum_{k=0}^{j-2}q(\widehat{X}_k)};\ \widehat{X}_{j-1}\in dy\Bigr\},
\nonumber\\[-2mm]
\end{eqnarray}
where $q(x)$ and $\{\widehat X_n\}$ are defined in 
\eqref{def.q.B} and \eqref{P.hat.B} respectively.
Equivalently, for all $x>\widehat x$ and $h>0$,
\begin{eqnarray}\label{repr.Xn.B.U.x}
\lefteqn{\P\{X_n\in(x,x+h]\}}\nonumber\\
&=& \sum_{j=1}^n\int_B\P\{X_{n-j}\in dz\}\int_{\widehat x}^\infty
P(z,du)U(u) \E_u\biggl\{\frac{e^{-\sum_{k=0}^{j-2}q(\widehat X_k)}}
{U(\widehat X_{j-1})};\ \widehat X_{j-1}\in(x,x+h]\biggr\},
\nonumber\\[-1mm]
\end{eqnarray}

The last representation allows us to study the tail distribution
of a positive recurrent $\{X_n\}$ via considering a suitable increasing test 
function $U(x)$ which makes the chain $\{\widehat X_n\}$ transient. 
Then factorisation result for
the renewal function $H^{(q)}$ with weights, see Section \ref{sec:renewal.w},
and an integro-local renewal theorem for $\{\widehat X_n\}$
and convergence in total variation of $X_n$ to $\pi$ allow us
to derive asymptotics for the tail distribution of $X_n$.

%\section{Comments to Chapter \ref{ch:change}}

%An aggregated Markov chain $\{\widetilde X_n\}$ on $[\widehat x,\infty)$
%with transition probabilities \eqref{aggr.1} and \eqref{aggr.2}
%was introduced first time in \cite{}.
\chapter{Tail analysis for recurrent Markov chains
with drift proportional to $1/x$}
\chaptermark{Drift proportional to $1/x$}
\label{ch:power.asymptotics}

\section{Markov chains with asymptotically zero drift:\\
heavy-tailedness of invariant measure}
\sectionmark{Heavy-tailedness of invariant measure}
\label{sec:heavy-taildness}

In this chapter we consider a recurrent Markov chain $\{X_n\}$
possessing an invariant measure which is either
probabilistic in the case of positive recurrence or $\sigma$-finite
in the case of null recurrence. We denote this measure by $\pi$.

If we consider an irreducible aperiodic Markov chain on $\Z$,
then the existence of probabilistic invariant measure
is equivalent to finiteness of $\E_0\tau_0$ where
$\tau_0:=\min\{n\ge 1:X_n=0\}$. The case of null recurrence
corresponds to almost finite $\tau_0$ with infinite mean, $\E\tau_0=\infty$.
For the state space $\R$, a standard condition for recurrence
is Harris recurrence, see \cite{MT} for related definitions.
The Harris recurrence guarantees that an invariant measure is unique
up to a constant multiplier.

We consider the case where $\pi$ has right unbounded support,
that is, $\pi(x,\infty)>0$ for all $x$. Our main aim is to describe
the asymptotic behaviour of its tail, $\pi(x,\infty)$,
for a class of Markov chains with asymptotically zero drift.

We start with the following result which states
that a typical stationary Markov chain with
asymptotically zero drift generates a heavy-tailed
invariant distribution which is very different from the case of
Markov chains with asymptotically negative drift bounded 
away from zero.\index{Markov chain!invariant measure!heavy-tailedness}

\begin{theorem}\label{non.exp}
Let a Markov chain $\{X_n\}$ on $\R$ have asymptotically zero drift,
i.e. $m_1(x)\to 0$ as $x\to\infty$ and, in addition,
\begin{eqnarray}\label{non.gen}
\liminf_{x\to\infty}\ {\mathbb E}\{\xi^2(x);\ \xi(x)>0\} &>& 0.
\end{eqnarray}
Then any right unbounded invariant distribution $\pi$ of $\{X_n\}$ is 
heavy-tailed,\index{Invariant measure!heavy-tailed} that is,
\begin{eqnarray*}
\int e^{\lambda y}\pi(dy)\ =\ \infty\quad\mbox{for all }\lambda>0.
\end{eqnarray*}
\end{theorem}

\begin{proof}
Assume on the contrary that an invariant distribution $\pi$ is
right unbounded with finite exponential moment of some order
$\lambda>0$. Let $\{X_n\}$ be stationary with distribution $\pi$.
Then, for any $x_0$,
\begin{eqnarray}\label{equil.exp}
\E(V(X_1)-V(X_0)) &=& 0,
\end{eqnarray}
where $V(x):=\max(e^{\lambda x},e^{\lambda x_0})$.
Since
\begin{eqnarray*}
\E(V(X_1)-V(X_0)) &\ge& \E\{V(X_1)-V(X_0);\ X_0>x_0\}
\end{eqnarray*}
and since $X_0$ has right unbounded support,
it would be a contradiction with \eqref{equil.exp}
if we proved that, for some $x_0$,
\begin{eqnarray}\label{greater.0}
v(x):=\E\{V(X_1)-V(X_0)\mid X_0=x\} &>& 0
\quad\mbox{ for all }x>x_0.
\end{eqnarray}
For all $x>x_0$,
\begin{eqnarray*}
v(x)\ge \E e^{\lambda(x+\xi(x))}-e^{\lambda x}
&=& e^{\lambda x}(\E e^{\lambda \xi(x)}-1).
\end{eqnarray*}
Since $e^y\ge 1+y$ for all $y$ and
$e^y\ge 1+y+y^2/2$ for all $y>0$,
\begin{eqnarray*}
\E e^{\lambda \xi(x)}-1
&\ge& \lambda m_1(x)+\frac{\lambda^2}{2}\E\{\xi^2(x);\ \xi(x)>0\}.
\end{eqnarray*}
Due to $\lambda m_1(x)\to0$ as $x\to\infty$
and the condition \eqref{non.gen}, there exists a
sufficiently large $x_0$ such that the sum on the right hand
side of the last inequality is positive for all $x>x_0$
which proves \eqref{greater.0} and hence the theorem assertion.
\qed\end{proof}

Let us show by example that the condition \eqref{non.gen}
which is some kind of non-degeneracy of jumps
is essential for the theorem conclusion to hold.
Consider the skip-free Markov chain $\{X_n\}$ on $\Z^+$ described in
Section \ref{sec:intr.nnrm},
that is, $\xi(x)$ takes values $-1$, $1$ and $0$ only,
with probabilities $p_-(x)$, $p_+(x)$ and $p_0(x)$
respectively, $p_-(0)=0$. The invariant probabilities
$\pi(x)$, $x\in\Z^+$, are computed in \eqref{cont.c},
\begin{eqnarray*}
\pi(x) &=& \pi(0)\prod_{k=1}^x\frac{p_+(k-1)}{p_-(k)}.
\end{eqnarray*}
Consider the case where $p_+(x):=1/2(x+1)$ and
$p_-(x):=1/(x+1)$. In this case the drift is asymptotically
zero but the stationary probabilities are asymptotically equivalent
to $cx/2^x$ so the invariant distribution is light-tailed.
Clearly, here the condition \eqref{non.gen} fails.

\section{Stationary measure of recurrent chains: power-like asymptotics}
\sectionmark{Stationary measure: power-like asymptotics}
\label{sec:Lamperti.critical}

This section is devoted to the precise asymptotic behaviour
of the invariant measure in the case
where the drift asymptotically behaves like $c/x$.

As discussed in Sections \ref{sec:intr.nnrm} and \ref{sec:intr.rmdp},
there are two types of Markov chains for which the invariant
measure is explicitly calculable. Both are related to
skip-free processes, either on lattice $\Zp$ or on continious
state space $\Rp$.

The first case where the stationary distribution is explicitly
known is a Markov chain on $\Zp$ with $\xi(x)$ taking values
$-1$, $1$ and $0$ only, with probabilities $p_-(x)$, $p_+(x)$ and $p_0(x)$
respectively, $p_-(0)=0$, see Section \ref{sec:intr.nnrm}.
The second case is diffusion processes on $\Rp$ (slotted in time if we
wanted just a Markov chain), see Section \ref{sec:introduction}.
In both cases we observe power tail behaviour of invariant probabilities
in the case where the drift is asymptotically proportional to $-\mu/x$
as $x\to\infty$.

In this chapter we consider a recurrent Markov chain $\{X_n\}$
on $\R$ whose jumps are such that
\begin{eqnarray}\label{m1.and.m2.new}
m_2^{[s(x)]}(x)\to b>0\quad\mbox{ and }\quad m_1^{[s(x)]}(x)x\to -\mu\in\R
\quad\mbox{ as }x\to\infty,
\end{eqnarray}
where a function $s(x)=o(x)$ is increasing and $\mu>-b/2$;
\begin{itemize}
\item the case $\mu\in(-b/2,b/2)$ usually corresponds to null recurrence of $\{X_n\}$,
see Corollary \ref{cor:null},
\item the case $\mu>b/2$ corresponds to positive recurrence,
see Corollary \ref{cor:posrec};
\item in the case $\mu=b/2$ either null or positive recurrence can happen,
see Corollaries \ref{cor:posrec.log}, \ref{cor:null.log}.
\end{itemize}
In addition, we assume that
\begin{eqnarray}\label{r-cond.2.new}
\frac{2m_1^{[s(x)]}(x)}{m_2^{[s(x)]}(x)}
&=& -r(x)+o(p(x))\quad\mbox{as }x\to\infty
\end{eqnarray}
for some monotone function $r(x)\to0$ satisfying $r(x)x\to 2\mu/b>-1$
as $x\to\infty$ and some decreasing integrable at infinity function $p(x)\ge 0$.
Since $p(x)$ is decreasing and integrable, $p(x)x\to0$ as $x\to\infty$.
We also assume that
\begin{eqnarray}\label{r.p.prime}
r'(x) &=& O(1/x^2)\quad\mbox{ and }\quad p'(x)\ =\ O(1/x^2).
\end{eqnarray}
As follows from Lemma \ref{l:g.fin.p.2},
the second relation can always be satisfied by choosing a slower
decreasing integrable function $p(x)$.

Under \eqref{m1.and.m2.new}, 
an equivalent way to state the assumption \eqref{r-cond.2.new} is
\begin{eqnarray}\label{r-cond.2.equiv}
m_1^{[s(x)]}(x)+\frac{m_2^{[s(x)]}(x)}{2}r(x)
&=& o(p(x))\quad\mbox{as }x\to\infty.
\end{eqnarray}

Define a monotone function
\begin{eqnarray}\label{def.of.R.new}
R(x) &:=& \int_0^x r(y)dy,\quad x>0,
\end{eqnarray}
$R(x)=0$ for $x\le 0$. Since $xr(x)\to 2\mu/b>-1$,
\begin{eqnarray*}
\frac{R(x)}{\log x} &\to& \frac{2\mu}{b}\ >\ -1\quad\mbox{as }x\to\infty.
\end{eqnarray*}

Define the following increasing function which plays
the most important r\^ole in our analysis of recurrent Markov chains:
$U(x)=0$ for $x\le 0$ and, for $x>0$,
\begin{eqnarray}\label{def.u}
U(x) &:=& \int_0^x e^{R(y)}dy\ \to\ \infty\quad\mbox{as }x\to\infty,
\end{eqnarray}
again due to $2\mu/b>-1$; in what follows we show that the function $U(x)$
is very close to be a harmonic function for large values of $x$.
Note that the function $U(x)$ solves the equation $U''-rU'=0$ for $x>0$.

According to our assumptions,
$$
r(x)=\frac{2\mu}{b}\frac{1}{x}+\frac{\varepsilon(x)}{x},
$$
where $\varepsilon(x)\to0$ as $x\to\infty$.
In view of the representation theorem for slowly varying
functions, there exists a slowly varying at infinity function
$\ell(x)$ such that 
$$
e^{R(x)}=x^{\rho-1}\ell(x)\quad\mbox{and}\quad U(x)\sim x^\rho \ell(x)/\rho
\quad\mbox{where }\rho=2\mu/b+1>0.
$$

%We assume some kind of irreducibility of $X$ related to \eqref{rec.2.tr}.
%Precisely, we assume that, given any fixed $\widehat x$,
%the Markov chain $\widetilde X$ on $\R^+$
%defined by its transition probabilities
%\begin{eqnarray}\label{tilde.Xn}
%\widetilde P(x,A) &:=& \left\{
%\begin{array}{ll}
%\displaystyle
%\frac{\P_x\{X_1\in A\cap(\widehat x,\infty)\}}{\P_x\{X_1>\widehat x\}}
%&\mbox{if }\P_x\{X_1>\widehat x\}>0,\\
%\I\{x\in A\}&\mbox{if }\P_x\{X_1>\widehat x\}=0,
%\end{array}
%\right.
%\end{eqnarray}
%is such that
%\begin{eqnarray}\label{exersions}
%\sup_n \widetilde X_n &=& \infty\quad \mbox{ with probability 1},
%\end{eqnarray}
%whenever $\widetilde X_0>\widehat x$. The transition probabilities
%$\widetilde P$ is the normalised transition kernel of $X$ killed
%at entering $[0,\widehat x]$. If the original Markov chain $X$
%lives in $\Z^+$, then it is equivalent to irreducibility
%of $\widetilde X$ for all $\widehat x$.

The main result in this section is the following theorem
which provides exact asymptotics for stationary measure
of recurrent Markov chains with
asymptotically zero drift described above.
\index{Markov chain!invariant measure!power asymptotics}
\index{Invariant measure!power asymptotics}

\begin{theorem}\label{thm:pi.recurrent}
Let $\{X_n\}$ be a recurrent Markov chain %satisfying \eqref{pi.unb.2}
and let $\pi(\cdot)$ be its stationary measure.
Let $\pi(-\infty,x]<\infty$ for all $x$ and let, for any initial state,
\begin{eqnarray}\label{pi.limsup.inf}
\P\Bigl\{\limsup_{n\to\infty} X_n=\infty\Bigr\} &=& 1.
\end{eqnarray}
Let the first two truncated moments of jumps satisfy the conditions 
\eqref{m1.and.m2.new} and \eqref{r-cond.2.new}
where $r(x)$ and $p(x)$ satisfy the regularity condition \eqref{r.p.prime}.
Assume that the following integrability conditions hold
\begin{eqnarray}\label{cond.for.U.unif.52.local}
\sup_{x\in \R}\frac{\E U(\xi(x))}{1+U(x)} &<& \infty,
\end{eqnarray}
and, as $x\to\infty$,
\begin{eqnarray}\label{cond.xi.le}
\P\{|\xi(x)|>s(x)\} &=& o(p(x)/x),\\
\label{cond.3.moment}
\E\bigl\{|\xi(x)|^3;\ |\xi(x)|\le s(x)\bigr\} &=& o(x^2p(x)).
\end{eqnarray}
In addition, %if the function $r(x)$ is decreasing, 
let
\begin{eqnarray}\label{cond.xi.ge}
\E\bigl\{U(\xi(x));\ \xi(x)>s(x)\bigr\} &=& o(p(x)e^{R(x)}).
\end{eqnarray}
%while for an increasing (and hence negative) function $r(x)$, let
%\begin{eqnarray}\label{cond.xi.ge.E}
%\E\bigl\{\xi(x);\ \xi(x)>s(x)\bigr\} &=& o(p(x)).
%\end{eqnarray}
Then, for some $c>0$,
\begin{eqnarray*}
\pi(x_1,x_2] &\sim& c\int_{x_1}^{x_2}\frac{y}{U(y)}dy
\end{eqnarray*}
as $x_1$, $x_2\to\infty$ in such a way that $\liminf x_2/x_1>1$.
\end{theorem}

It follows from the condition \eqref{pi.limsup.inf} that
$\pi$ has right-unbounded support, that is, $\pi(x,\infty)>0$ for all $x$.
\index{Markov chain!invariant measure!power asymptotics}
\index{Invariant measure!power asymptotics}

\begin{corollary}\label{cor:pi.recurrent}
If $2\mu>b$, $\{X_n\}$ is positive recurrent, and the conditions of
Theorem \ref{thm:pi.recurrent} hold, then
$$
\pi(x,\infty)\sim \frac{c}{\rho-2}\frac{x^2}{U(x)}\quad\mbox{as }x\to\infty.
$$
If $2\mu\in(-b,b)$, $\{X_n\}$ is null recurrent, and the conditions of
Theorem \ref{thm:pi.recurrent} hold, then
$$
\pi(-\infty,x)\sim \frac{c}{2-\rho}\frac{x^2}{U(x)}\quad\mbox{as }x\to\infty.
$$
\end{corollary}

\index{Markov chain!invariant measure!power asymptotics}
\index{Invariant measure!power asymptotics}
\begin{corollary}\label{cor:pi.1x}
Let, in addition, $r(x)=2\mu/bx$.
If $2\mu>b$ and $\{X_n\}$ is positive recurrent, then
$$
\pi(x,\infty)\ \sim\ \frac{c\rho}{\rho-2}\frac{1}{x^{2\mu/b-1}}\quad\mbox{as }x\to\infty.
$$
If $2\mu\in(-b,b)$ and $\{X_n\}$ is null recurrent, then
$$
\pi(-\infty,x)\ \sim\ \frac{c\rho}{2-\rho}x^{1-2\mu/b}\quad\mbox{as }x\to\infty.
$$
\end{corollary}

In the case $2\mu=b$, we have the following result.
\index{Markov chain!invariant measure!power asymptotics}
\index{Invariant measure!power asymptotics}

\begin{corollary}\label{cor:pi.1x.2mu=b}
Let, in addition, for some $m\ge1$ and $\gamma\not=0$,
\[
r(x)\ =\ \frac{1}{x}+\frac{1}{x\log x}
+\ldots+\frac{1}{x\log x\cdot\ldots\cdot\log_{(m-1)}x}
+\frac{1+\gamma}{x\log x\cdot\ldots\cdot\log_{(m)}x}.
\] 
If $\gamma>0$ and $\{X_n\}$ is positive recurrent, then
$$
\pi(x,\infty)\ \sim\ \frac{2c}{\gamma}\frac{1}{\log_{(m)}^\gamma x}
\quad\mbox{as }x\to\infty.
$$
If $\gamma<0$ and $\{X_n\}$ is null recurrent, then
$$
\pi(-\infty,x)\ \sim\ \frac{2c}{-\gamma}\log_{(m)}^{-\gamma} x
\quad\mbox{as }x\to\infty.
$$
\end{corollary}

Before proving Theorem \ref{thm:pi.recurrent}
let us formulate and prove some auxiliary results.
First we construct a Lyapunov function needed.
Consider the function $r_p(x):=r(x)-p(x)$ and define
$R_p(x)=U_p(x)=0$ for $x\le 0$ and
\begin{eqnarray}\label{R.U.p.once.more}
R_p(x) &:=& \int_0^x r_p(y)dy,\quad
U_p(x)\ :=\ \int_0^x e^{R_p(y)}dy\quad\mbox{ for }x>0.
\end{eqnarray}
We have $r_p(x)\le r(x)$, $R_p(x)\le R(x)$, and
$U_p(x)\le U(x)$ for $x\ge 0$. Since
\begin{eqnarray*}
C_p &:=& \int_0^\infty p(y)dy\quad\mbox{is finite},
\end{eqnarray*}
we have
\begin{eqnarray}\label{equiv.for.R}
R_p(x) &=& R(x)-C_p+o(1)\quad\mbox{as }x\to\infty.
\end{eqnarray}
Therefore,
\begin{eqnarray}\label{equiv.for.U}
U_p(x) &\sim& e^{-C_p}U(x)\to\infty \quad\mbox{as }x\to\infty,
\end{eqnarray}
because $U(x)\to\infty$. Further, since $xr_p(x)=xr(x)-xp(x)\to 2\mu/b$,
\begin{eqnarray*}
\frac{U_p'(x)}{(x e^{R_p(x)})'} &=&
\frac{e^{R_p(x)}}{(1+xr_p(x)) e^{R_p(x)}}\ \to\ \frac{b}{2\mu+b}
\quad\mbox{as }x\to\infty.
\end{eqnarray*}
Then L'H\^opital's rule yields
\begin{eqnarray}\label{equiv.for.U.1}
U_p(x) &\sim& \frac{b}{2\mu+b}x e^{R_p(x)}\
 \sim\ \frac{be^{-C_p}}{2\mu+b}x e^{R(x)}
\quad\mbox{as }x\to\infty.
\end{eqnarray}

In the sequel we need to know the asymptotic behaviour of 
the drift of $U_p(X_n)$.

\begin{lemma}\label{L.Lyapunov}
Assume that \eqref{r-cond.2.new}, \eqref{r.p.prime}
and \eqref{cond.xi.le}--\eqref{cond.3.moment} hold. Then
\begin{eqnarray}\label{b.for.u}
\E U_p(x+\xi(x))-U_p(x) &\sim&-\frac{b}{2}p(x) e^{R_p(x)}\nonumber\\
&\sim&  -\frac{2\mu+b}{2}\frac{p(x)}{x} U_p(x)
\quad\mbox{as }x\to\infty,
\end{eqnarray}
where the last equivalence is due to \eqref{equiv.for.U.1}.
\end{lemma}

\begin{proof}
We start with the following decomposition:
\begin{eqnarray}\label{L.harm2.1}
\E U_p(x+\xi(x))-U_p(x)
&=& \E\{U_p(x+\xi(x))-U_p(x);\ \xi(x)<-s(x)\}\nonumber\\
&&\hspace{2mm} +\E\{U_p(x+\xi(x))-U_p(x);\ |\xi(x)|\le s(x)\}\nonumber\\
&&\hspace{6mm} +\E\{U_p(x+\xi(x))-U_p(x);\ \xi(x)>s(x)\}.\quad
\end{eqnarray}
Here the first term on the right hand side is negative
and may be bounded below as follows:
\begin{eqnarray}\label{L.harm2.2a}
\E\{U_p(x+\xi(x))-U_p(x);\ \xi(x)<-s(x)\}
&\ge& -U_p(x)\P\{\xi(x)<-s(x)\}\nonumber\\
&=& o(p(x)/x) U_p(x)\nonumber\\
&=& o(p(x)e^{R_p(x)}),
\end{eqnarray}
by the condition \eqref{cond.xi.le} and the equivalence \eqref{equiv.for.U.1}.
Furthermore, the third term on the right hand side of \eqref{L.harm2.1}
is positive and may be bounded in the following way:
\begin{eqnarray*}
\lefteqn{\E\{U_p(x+\xi(x))-U_p(x);\ \xi(x)>s(x)\}}\\
&&\hspace{10mm}\le\ \E\{U_p(x+\xi(x));\ \xi(x)>s(x)\}\\
&&\hspace{20mm}\le\ \E\{U_p(2x)+U_p(2\xi(x));\ \xi(x)>s(x)\}\\
&&\hspace{30mm}\le\ c \bigl(U_p(x)\P\{\xi(x)>s(x)\}+\E\{U_p(\xi(x));\ \xi(x)>s(x)\}\bigr),
\end{eqnarray*}
owing to the regular variation of $U_p$ at infinity. Hence,
\begin{eqnarray}\label{L.harm2.2b}
\E\{U_p(x+\xi(x))-U_p(x);\ \xi(x)>s(x)\}
&=& o(p(x)e^{R_p(x)}),
\end{eqnarray}
due to the conditions \eqref{cond.xi.le} and \eqref{cond.xi.ge}.
To estimate the second term on the right hand side of \eqref{L.harm2.1},
we make use of Taylor's expansion:
\begin{eqnarray}\label{L.harm2.2}
\lefteqn{\E\{U_p(x+\xi(x))-U_p(x);\ |\xi(x)|\le s(x)\}}\nonumber\\
&&\hspace{15mm} =\  U_p'(x)\E\{\xi(x);|\xi(x)|\le s(x)\}
+\frac{1}{2} U_p''(x)\E\{\xi^2(x);|\xi(x)|\le s(x)\}\nonumber\\
&&\hspace{35mm} +\frac{1}{6}\E\bigl\{U_p'''(x+\theta\xi(x))\xi^3(x);
|\xi(x)|\le s(x)\bigr\},
\end{eqnarray}
where $0\le\theta=\theta(x,\xi(x))\le 1$. By the construction of $U_p$,
\begin{eqnarray}\label{U.12.prime}
U_p'(x)=e^{R_p(x)}\quad\mbox{and}\quad
U_p''(x)=r_p(x)e^{R_p(x)}=(r(x)-p(x))e^{R_p(x)}.
\end{eqnarray}
Then it follows that
\begin{eqnarray}\label{L.harm2.2.1}
U_p'(x)m_1^{[s(x)]}(x)+\frac{1}{2}U_p''(x)m_2^{[s(x)]}(x)
&=& e^{R_p(x)} \Bigl(m_1^{[s(x)]}(x)+(r(x)-p(x))\frac{m_2^{[s(x)]}(x)}{2}\Bigr)\nonumber\\
&=& \frac{m_2^{[s(x)]}(x)}{2}e^{R_p(x)}
\biggl(\frac{2m_1^{[s(x)]}(x)}{m_2^{[s(x)]}(x)}+r(x)-p(x)\biggr)\nonumber\\
&=& -\frac{m_2^{[s(x)]}(x)}{2}e^{R_p(x)} p(x) (1+o(1))\nonumber\\
&\sim& -\frac{b}{2}e^{R_p(x)} p(x),
\end{eqnarray}
by the condition \eqref{r-cond.2.new}.

Finally, let us estimate the last term in \eqref{L.harm2.2}.
Notice that by the condition \eqref{r.p.prime} on the derivatives
of $r(x)$ and $p(x)$,
\begin{eqnarray*}
U_p'''(x) &=& \bigl(r'(x)-p'(x)+(r(x)-p(x))^2\bigr)e^{R_p(x)}
=O(1/x^2)e^{R_p(x)},
\end{eqnarray*}
so hence
\begin{eqnarray*}
\bigl|\E\bigl\{U_p'''(x+\theta\xi(x))\xi^3(x); |\xi(x)|\le s(x)\bigr\}\bigr|
&\le& \frac{c_1}{x^2}\E\bigl\{|\xi^3(x)|;\ |\xi(x)|\le s(x)\bigr\}e^{R_p(x)},
\end{eqnarray*}
because $s(x)=o(x)$ and the function $e^{R_p(x)}$ is regularly varying at infinity.
Then, in view of \eqref{cond.3.moment},
\begin{eqnarray}\label{full.vs.cond}
\bigl|\E\bigl\{U_p'''(x+\theta\xi(x))\xi^3(x);\ |\xi(x)|\le s(x)\bigr\}\bigr|
&=& o(p(x)e^{R_p(x)}).
\end{eqnarray}
Then it follows from \eqref{L.harm2.2}, \eqref{L.harm2.2.1} and
\eqref{full.vs.cond} that
\begin{eqnarray}\label{L.harm2.2.2}
\E\{U_p(x+\xi(x))-U_p(x);\ |\xi(x)|\le s(x)\}
&=& -\frac{b}{2}p(x)e^{R_p(x)}+o(p(x)e^{R_p(x)}).\nonumber\\[-1mm]
\end{eqnarray}
Substituting \eqref{L.harm2.2a}, \eqref{L.harm2.2b} and \eqref{L.harm2.2.2}
into \eqref{L.harm2.1}, we finally get the desired expression
for $\E U_p(x+\xi(x))-U_p(x)$. This completes the proof of the lemma.
\qed\end{proof}

Fix an $\widehat x>0$. Define a transition kernel $Q$ 
on $S=(\widehat x,\infty)$ via the following change of measure
\begin{eqnarray*}
Q(x,dy) &:=& \frac{U_p(y)}{U_p(x)} P(x,dy),\quad x,\ y>\widehat x.
\end{eqnarray*}
Since
\begin{eqnarray*}
\frac{\E U(x+\xi(x))}{1+U(x)} &\le& 
\frac{U(2x)}{1+U(x)}+\frac{\E U(2\xi(x))}{1+U(x)}
\end{eqnarray*}
and the function $U$  is regularly varying at infinity, 
the condition \eqref{cond.for.U.unif.52.local} implies that
\begin{eqnarray}\label{cond.for.U.unif.52}
\sup_{x\in \R}\frac{\E U(x+\xi(x))}{1+U(x)} &<& \infty.
\end{eqnarray}
Then it follows that  the kernel $Q$ satisfies the condition \eqref{cond.for.U.unif} 
which allows us to apply the machinery
developed in Chapter \ref{ch:change}. We have
\begin{eqnarray}\label{Q.x.dy.R}
Q(x,\R) &=& \frac{\E\{U_p(x+\xi(x));\ x+\xi(x)>\widehat x\}}{U_p(x)}.
\end{eqnarray}
Lemma \ref{L.Lyapunov} yields the following result.

\begin{corollary}\label{l:lyapunov}
Under the conditions of Lemma \ref{L.Lyapunov},
there exists an $\widehat x$ such that
$$
-(2\mu+b)\frac{p(x)}{x} U_p(x)
\ \le\ \E U_p(x+\xi(x))-U_p(x)\ \le\ 0
\quad\mbox{for all }x>\widehat x.
$$
\end{corollary}

Everywhere in what follows $\widehat x$ is any level guaranteed
by Corollary \ref{l:lyapunov},
$B=(-\infty,\widehat x]$ and $\tau_B:=\min\{n\ge 1:X_n\in B\}$.
Then the definition of the transition kernel $Q$ may be rewritten as follows
\begin{eqnarray}\label{def.Q}
Q(x,dy) &=& \frac{U_p(y)}{U_p(x)}\P_x\{X_1\in dy,\tau_B>1\}\nonumber\\
&=& \frac{U_p(y)}{U_p(x)}\P_x\{X_1\in dy,X_1>\widehat x\}.
\end{eqnarray}

It follows from the upper bound in Corollary \ref{l:lyapunov} that
$$
Q(x,\R)=\frac{\E\{U_p(x+\xi(x));\tau_B>1\}}{U_p(x)}
\le \frac{\E U_p(x+\xi(x))}{U_p(x)}\le 1\quad\mbox{for all }x>\widehat x.
$$
In other words, $Q$ restricted to $(\widehat x,\infty)$
is a substochastic kernel. It follows from \eqref{cond.xi.le} that
\begin{eqnarray}\label{Up.below.xhat}
\E\{U_p(x+\xi(x));\tau_B=1\} &=& \E_x\{U_p(X_1);\ X_1\le\widehat x\}\nonumber\\
&\le& U_p(\widehat x)\P\{x+\xi(x)\le\widehat x\}\ =\ o(p(x)/x).
\end{eqnarray}
Combining this with the lower bound in
Corollary \ref{l:lyapunov} we obtain that
\begin{eqnarray}\label{def.q}
q(x)\ :=\ -\log Q(x,\R) &=& O(p(x)/x).
\end{eqnarray}
Let us consider the following normalised kernel
$$
\widehat P(x,dy)\ :=\ \frac{Q(x,dy)}{Q(x,\R)}
$$
and let $\{\widehat X_n\}$ be a Markov chain
with this transition probabilities;
let $\widehat\xi(x)$ be its jump from the state $x$.
Consequently, by \eqref{connection.new.B},
\begin{eqnarray}\label{connection.new}
\P_x\{X_n\in dy,\tau_B>n\} &=& \frac{U_p(x)}{U_p(y)}
\E_x\bigl\{e^{-\sum_{k=0}^{n-1}q(\widehat X_k)};\ \widehat X_n\in dy\bigr\}.
\end{eqnarray}

\begin{lemma}\label{l:change}
Under the conditions of Lemma \ref{L.Lyapunov}, as $x\to\infty$,
\begin{eqnarray}\label{m1.sim.rx}
\E\{\widehat{\xi}(x);\ |\widehat{\xi}(x)|\le s(x)\} &\sim& \frac{\mu+b}{x},\\
\label{m2.to.b}
\E\{(\widehat{\xi}(x))^2;\ |\widehat{\xi}(x)|\le s(x)\} &\to& b,\\
\label{X*.le.gamma}
\P\{|\widehat{\xi}(x)|>s(x)\} &=& o(p(x)/x),\\
\label{X*.le.gamma.e}
\E\{|\widehat{\xi}(x)|;\ \widehat{\xi}(x)<-s(x)\} &=& o(p(x)),
\end{eqnarray}
for some decreasing integrable at infinity function $p(x)$.
Moreover, there exists a sufficiently large $\widehat x$ such that
\begin{eqnarray}\label{m1.ge.rx}
\E\{\widehat{\xi}(x);\ \widehat{\xi}(x)\le s(x)\} &\ge& \frac{\mu+b}{2x}
\quad\mbox{for all }x\ge\widehat x.
\end{eqnarray}
\end{lemma}

\begin{proof}
It follows from \eqref{equiv.for.U.1} that
\begin{eqnarray*}
\frac{U_p'(x)}{U_p(x)} &=& \frac{e^{R_p(x)}}{U_p(x)}
\ \sim\ \frac{2\mu+b}{bx}\quad\mbox{as }x\to\infty.
\end{eqnarray*}
So, the function $U_p$ satisfies the condition
\eqref{cond.UprimeU.frac} with $r(x)=1/x$ and $c_U=1+2\mu/b$.
Also $U_p$ satisfies \eqref{cond.UprimeU} for any $s(x)=o(x)$ because
\begin{eqnarray*}
\frac{U_p'(x+y)}{U_p'(x)} &=& \frac{e^{R_p(x+y)}}{e^{R_p(x)}}
\ \sim\ e^{R(x+y)-R(x)}
\ =\ e^{\int_x^{x+y}r(z)dz}\ =\ e^{O(s(x)/x)}\ =\ e^{o(1)}
\end{eqnarray*}
as $x\to\infty$ uniformly for all $|y|\le s(x)$, and, by \eqref{equiv.for.U.1},
\begin{eqnarray*}
\frac{U_p(x+y)}{U_p(x)} &\sim& \frac{x+y}{x}\frac{e^{R(x+y)}}{e^{R(x)}}
\ \sim\ e^{R(x+y)-R(x)} \ \to\ 1.
\end{eqnarray*}
The function $U_p$ satisfies \eqref{cond.Uclose.harmonic} by Lemma \ref{L.Lyapunov}.
Finally, the condition \eqref{le.widehatx.1} follows from \eqref{X*.le.gamma}.
So, all conditions of Lemma \ref{l:change.B.0} are met and
\eqref{m1.sim.rx}--\eqref{m1.ge.rx} follow.
\qed\end{proof}

Therefore, the chain $\{\widehat X_n\}$ satisfies the conditions
\eqref{1.2.G}--\eqref{rec.3.1.e} of Theorem \ref{thm:gamma}
with $\widehat\mu=\mu+b$ and $\widehat b=b$,
so that $\widehat\mu>\widehat b/2$. Further, the lower bound
\eqref{m1.ge.rx} for the drift of $\{\widehat X_n\}$ allows us to apply
Theorem \ref{l:uniform} to $\{\widehat X_n\}$ and to conclude that,
for $\widehat T(t)=\min\{n\ge 1:\widehat X_n>t\}$,
\begin{eqnarray*}
\E_y \widehat T(t)\ =\ \E_y \widehat L(\widehat x, \widehat T(t)) &<& \infty
\quad\mbox{for all }t>y,
\end{eqnarray*}
so hence, for any initial state $\widehat X_0=y$,
\begin{eqnarray*}
\P_y\Bigl\{\limsup_{n\to\infty}\widehat X_n=\infty\Bigr\} &=& 1.
\end{eqnarray*}
In its turn, then it follows from Theorem \ref{thm:transience.inf}
that $\widehat X_n\to\infty$ with probability 1.

So, Theorem \ref{thm:gamma} is applicable to $\{\widehat X_n\}$
which implies weak convergence of $(\widehat X_n)^2/n$
to a $\Gamma$-distribution with mean $2\mu+3b=(2+\rho)b$ and variance
$(2\mu+3b)2b=(2+\rho)2b^2$ where
$\rho=1+2\mu/b$,\index{Transience!convergence to!$\Gamma$-distribution}
that is, a $\Gamma$-distribution with probability density function
\begin{eqnarray}\label{gamma.density}
\gamma(u) &=& \frac{1}{(2b)^{1+\rho/2}\Gamma(1+\rho/2)}u^{\rho/2}e^{-u/2b}.
\end{eqnarray}
Furthermore, by Theorem \ref{thm:Hy.above},
there exists a $c<\infty$ such that
\begin{equation}\label{RF-bound}
\widehat H_y(x):=\sum_{n=0}^\infty \P_y\{\widehat X_n\le x\}\le c(1+x^2)
\quad\mbox{for all }x,y.
\end{equation}

Having this estimate proven we now deduce the following result.

\begin{lemma}\label{L.limit}
Under the conditions of Lemma \ref{L.Lyapunov},
\begin{eqnarray}\label{def:hz}
h(z) &:=& \lim_{n\to\infty}\E_z e^{-\sum_{k=0}^n q(\widehat X_k)}\ >\ 0
\quad\mbox{for all }z,
\end{eqnarray}
where $q$ is defined in \eqref{def.q}.
Moreover, $h(z)\to 1$ as $z\to\infty$.
\end{lemma}

\begin{proof}
The existence of $h(z)$ is immediate because
$e^{-\sum_{k=0}^n q(\widehat X_k)}$ is decreasing in $n$.
Since the function $e^{-x}$ is convex, by Jensen's inequality
\begin{eqnarray}\label{e.Jensen}
\E_z e^{-\sum_{k=0}^n q(\widehat X_k)} &\ge&
e^{-\E_z \sum_{k=0}^n q(\widehat X_k)}.
\end{eqnarray}
Thus, to show positivity it suffices to prove that
\begin{equation}\label{L.limit.1}
\E_z \sum_{k=1}^\infty q(\widehat X_k)<\infty,\quad z>\widehat x.
\end{equation}
Note that
\begin{eqnarray*}
\E_z \sum_{k=1}^\infty q(\widehat X_k) &\le&
\int_{\widehat x}^\infty q(y)\widehat H_z(dy)
\ \le\ c\int_{\widehat x}^\infty \frac{p(y)}{y}\widehat H_z(dy),
\end{eqnarray*}
because $q(y)=O(p(y)/y)$, see \eqref{def.q}.
But it has been already shown in the proof of
Lemma \ref{l:XY.equiv} that the last integral is finite.

To prove that $h(z)\to 1$, we note that
Theorem \ref{thm:transience.inf} implies, for every fixed $N>0$,
$$
\P_z\{\widehat X_n>N \mbox{ for all }n\ge 1\}\to 1\quad\mbox{as }z\to\infty,
$$
so that
$$
\widehat H_z(N) \to 0\quad\mbox{as }z\to\infty.
$$
Then, for every fixed $N$,
$$
\lim_{z\to\infty}\E_z \sum_{k=0}^\infty q(\widehat X_k)
\le \sup_{z>\widehat x} \int_N^\infty q(y)\widehat H_z(dy).
$$
According to \eqref{int.prH.fin},
$$
\lim_{N\to\infty} \sup_{z>\widehat x} \int_N^\infty q(y)\widehat H_z(dy) =0.
$$
Therefore, we infer that
$$
\lim_{z\to\infty}\E_z \sum_{k=0}^\infty q(\widehat X_k)=0,
$$
so we finally conclude $\lim_{z\to\infty}h(z)=1$ again from \eqref{e.Jensen}.
\qed\end{proof}

For tail asymptotics of recurrence times derived below in Section 
\ref{sec:recurrence.times}, we need the following two assertions.

\begin{corollary}\label{harm.F}
Assume that the conditions of Lemma \ref{L.Lyapunov} are valid.
Then $h(x)$ is a harmonic function for the kernel $Q$, that is,
$$
h(x)=\int_{\widehat x}^\infty h(y)Q(x,dy)\quad\mbox{for all }x>\widehat x.
$$
Furthermore,
\begin{eqnarray}\label{def:func.W}
W_p(x) &:=& h(x)U_p(x)
\end{eqnarray}
is a harmonic function for $\{X_n\}$
killed at the time of the first visit to $(-\infty,\widehat x]$:
$$
W_p(x)=\E_x\{W_p(X_1);\ X_1>\widehat x\}\quad\mbox{for all }x>\widehat x.
$$
\end{corollary}

\begin{proof}
By the Markov property,
\begin{eqnarray*}
\E_x e^{-\sum_{k=0}^n q(\widehat X_k)} 
&=& e^{-q(x)}\int_{\widehat x}^\infty \widehat P(x,dy) 
\E_y e^{-\sum_{k=0}^{n-1} q(\widehat X_k)}.
\end{eqnarray*}
Letting $n\to\infty$ and using the dominated convergence theorem, we get 
\begin{eqnarray*}
h(x) &=& e^{-q(x)}\int_{\widehat x}^\infty \widehat P(x,dy) h(y).
\end{eqnarray*}
Recalling now that $e^{-q(x)}\widehat P(x,dy)=Q(x,dy)$, 
we arrive at the first statement of the corollary.

Noting also that
$Q(x,dy)=\frac{U_p(y)}{U_p(x)}
\P_x\{\widehat X_1\in dy\}$ for all $x$, $y>\widehat x$, 
we conclude that $h(x)U_p(x)$ is harmonic for $\{\widehat X_n\}$ killed at leaving $B$,
and the proof is complete.
\qed\end{proof}

It turns out that being formally defined via the function $U_p(x)$,
the harmonic function $W_p(x)$ does not essentially depend on the choice of
an increasing integrable at infinity function $p(x)$
which only contribute to a constant multiplier.
This observation follows from the following result.

\begin{lemma}\label{lem:harm.uniq}
Let $V(x)$ be a positive harmonic function for
$\{X_n\}$ killed at the first visit to $B:=(-\infty,\widehat{x}]$, that is,
\begin{eqnarray}\label{harm.uniq.1}
V(x) &=& \E_x\{V(X_1);\ \tau_B>1\}\quad\mbox{for all }x>\widehat x.
\end{eqnarray}
If, for some $C_V>0$,
\begin{eqnarray}\label{harm.uniq.2}
V(x) &\sim& C_V U(x)\quad\text{as }x\to\infty,
\end{eqnarray}
then
\begin{eqnarray*}
V(x) &=& C_V\lim_{n\to\infty}\E_x\{U(X_n);\ \tau_B>n\}
\quad\mbox{for all }x>\widehat x.
\end{eqnarray*}
\end{lemma}

\begin{proof}
It follows from \eqref{harm.uniq.1} that, for all $n\ge1$,
\begin{equation}\label{harm.uniq.3}
V(x)=\E_x\{V(X_n);\ \tau_B>n\}\quad\mbox{for all }x>\widehat x.
\end{equation}
Fix an $\varepsilon>0$. Due to the assumption \eqref{harm.uniq.2},
there exists an $x_\varepsilon$ such that
$$
(1-\varepsilon)V(y)\ \le\ C_VU(y)\ \le\ (1+\varepsilon)V(y)
\quad\text{for all }y>x_\varepsilon.
$$
Therefore,
\begin{eqnarray}\label{harm.uniq.4-}
\lefteqn{(1-\varepsilon)
\E_x\{V(X_n);\ \tau_B>n,X_n>x_\varepsilon\}}\\
&&\hspace{10mm}\le\ C_V\E_x\{U(X_n);\ \tau_B>n,X_n>x_\varepsilon\}\nonumber\\
&&\hspace{20mm}\le\ (1+\varepsilon)
\E_x\{V(X_n);\ \tau_B>n,X_n>x_\varepsilon\}.\label{harm.uniq.4}
\end{eqnarray}
On the other hand, by the definition of $\{\widehat X_n\}$,
see \eqref{connection.new.B.f},
\begin{eqnarray*}
\E_x\{V(X_n);\tau_B>n,X_n\le x_\varepsilon\}
&=& U_p(x)\E_x\Bigl\{e^{-\sum_{k=0}^{n-1}q(\widehat X_k)}
\frac{V(\widehat X_n)}{U_p(\widehat X_n)};\ \widehat X_n\le x_\varepsilon\Bigr\}\\
&\le& U_p(x)\E_x\Bigl\{\frac{V(\widehat X_n)}{U_p(\widehat X_n)};\
\widehat X_n\le x_\varepsilon\Bigr\},
\end{eqnarray*}
since $q(y)$ is non-negative. Recalling that the chain $\{\widehat X_n\}$
is transient, we conclude convergence
\begin{eqnarray}\label{harm.uniq.5}
\E_x\{V(X_n);\ \tau_B>n,X_n\le x_\varepsilon\} &\to& 0
\quad\mbox{as }n\to\infty.
\end{eqnarray}
By the same argument,
\begin{eqnarray}\label{harm.uniq.6}
\E_x\{U(X_n);\ \tau_B>n,X_n\le x_\varepsilon\} &\to& 0
\quad\mbox{as }n\to\infty.
\end{eqnarray}
Combining \eqref{harm.uniq.4}, \eqref{harm.uniq.6}
and \eqref{harm.uniq.3}, we obtain
\begin{eqnarray*}
C_V\E_x\{U(X_n);\ \tau_B>n\} &\le&
(1+\varepsilon)\E_x\{V(X_n);\ \tau_B>n,X_n>x_\varepsilon\}+o(1)\\
&\le& (1+\varepsilon)V(x)+o(1)\quad\mbox{as }n\to\infty.
\end{eqnarray*}
Combining \eqref{harm.uniq.4-}, \eqref{harm.uniq.5}
and \eqref{harm.uniq.3}, we obtain
\begin{eqnarray*}
C_V\E_x\{U(X_n);\ \tau_B>n\} &\ge&
(1-\varepsilon)\E_x\{V(X_n);\ \tau_B>n,X_n>x_\varepsilon\}\\
&=& (1-\varepsilon)V(x)+o(1)\quad\mbox{as }n\to\infty.
\end{eqnarray*}
Therefore, for any fixed $\varepsilon>0$,
\begin{eqnarray*}
\frac{1-\varepsilon}{C_V}V(x) &\le&
\liminf_{n\to\infty}\E_x\{U(X_n);\ \tau_B>n\} \\
&\le& \limsup_{n\to\infty}\E_x\{U(X_n);\ \tau_B>n\}
\ \le\ \frac{1+\varepsilon}{C_V}V(x).
\end{eqnarray*}
Letting here $\varepsilon\to0$ we conclude the existence of a limit
of $\E_x\{U(X_n);\ \tau_B>n\}$ as $n\to\infty$ which equals $V(x)/C_V$.
\qed\end{proof}

Set
\begin{eqnarray}\label{W.as.lim}
W(x) &:=& \lim_{n\to\infty}\E_x\{U(X_n);\ \tau_B>n\}.
\end{eqnarray}
According to Corollary \ref{harm.F}, $W_p(x)=h(x)U_p(x)$ is harmonic 
and $W_p(x)\sim e^{-C_p}U(x)$. Then, by Lemma \ref{lem:harm.uniq},
\begin{eqnarray}\label{W.p.=.c.W}
W_p(x) &=& e^{-C_p}W(x).
\end{eqnarray}

Consider the following weighted renewal measure
\begin{equation}\label{def.Hq.new}
\widehat H^{(q)}_z(dx)\ =\ \sum_{j=0}^\infty
\E_z\{e^{-\sum_{k=0}^{j-1} q(\widehat X_k)};\ \widehat X_j\in dx\},
\end{equation}
and its finite time horizon version,
\begin{equation}\label{def.Hq.new.n}
\widehat H^{(q)}_{z,n}(dx)\ =\ \sum_{j=0}^n
\E_z\{e^{-\sum_{k=0}^{j-1} q(\widehat X_k)};\ \widehat X_j\in dx\}.
\end{equation}
Applying Lemma \ref{thm:renewal.2} and Theorem \ref{thm:renewal.gamma}
to $\widehat X$ and taking into account Lemma \ref{L.limit},
we get the following result.\index{Renewal theorem!for $\Gamma$ limit}

\begin{corollary}\label{cor:renewal.2.W.g}
Assume that the conditions of Lemma \ref{L.Lyapunov} are valid.
Then
$$
\widehat H_{z,n}^{(q)}(\widehat x,x] 
= h(z)\widehat H_{z,n}(\widehat x,x]+o(x^2)
= h(z)(\widehat I(n/x^2)+o(1))x^2
$$
as $x\to\infty$ uniformly for all $n$, where $\widehat I$ 
is a function defined in Theorem \ref{thm:renewal.gamma} 
with $\widehat\mu=\mu+b$ and $\widehat b=b$. In particular,
$$
\widehat H_z^{(q)}(\widehat x,x]\sim h(z)\widehat H_z(\widehat x,x]
\sim h(z)\frac{x^2}{2\mu+b}\quad\mbox{as }x\to\infty.
$$
\end{corollary}

Now we are ready to prove the main result of this section.

\begin{theopargself}
\begin{proof}[of Theorem \ref{thm:pi.recurrent}]
Since the function $y/U(y)$ is regularly varying at infinity
and $\liminf x_2/x_1>1$, it suffices to consider the case where $x_2=(1+h)x_1$, $h>0$.

Lemma \ref{l:suff.Harris} is applicable to the chain $\{X_n\}$,
so it is legible to use the cycle representation \eqref{pi-def}.
As follows from the representation \eqref{pi.repr.B} applied to $U_p$,
\begin{eqnarray}\label{pi.repr}
\pi(x,(1+h)x] &=& c^*\int_x^{(1+h)x}\frac{H^{(q)}(dy)}{U_p(y)}\nonumber\\
&\sim& c^*e^{C_p}\int_x^{(1+h)x}\frac{H^{(q)}(dy)}{U(y)}
\quad\mbox{as }x\to\infty,
\end{eqnarray}
due to $U_p(y)\sim e^{-C_p}U(y)$, see \eqref{equiv.for.U};
$H^{(q)}$ is defined in \eqref{def.Hq.new.B}.

Fix an $\varepsilon>0$ and $n\in\N$. Let $x_k=(1+hk/n)x$, $k=0$, \ldots, $n$.
Then, since the function $U$ is increasing,
\begin{eqnarray*}
\sum_{k=0}^{n-1} \frac{\widehat H^{(q)}(x_k,x_{k+1})}{U(x_{k+1})} &\le&
\int_x^{(1+h)x}\frac{\widehat H^{(q)}(dy)}{U(y)}
\ \le\ \sum_{k=0}^{n-1} \frac{\widehat H^{(q)}(x_k,x_{k+1})}{U(x_k)}.
\end{eqnarray*}
Now, according to Corollary \ref{cor:renewal.2.W.g},
\begin{eqnarray*}
\widehat H^{(q)}(x_k,x_{k+1}] &=& \int_\R\widehat H^{(q)}_z(x_k,x_{k+1}]
\P\{\widehat X_0\in dz\}\\
&=& c_q (x_{k+1}^2-x_k^2)+o(x^2)\quad\mbox{as }x\to\infty\mbox{ uniformly for all }k\le n-1,
\end{eqnarray*}
where $c_q:=\E h(\widehat X_0)/(2\mu+b)$. 
Consequently, for all sufficiently large $x$,
\begin{eqnarray*}
(c_q-\varepsilon) (x_{k+1}^2-x_k^2)\ \le\ 
\widehat H^{(q)}(x_k,x_{k+1}] &\le& (c_q+\varepsilon) (x_{k+1}^2-x_k^2)
\quad\mbox{for all }k\le n-1,
\end{eqnarray*}
which yields
\begin{eqnarray*}
(c_q-\varepsilon)\sum_{k=0}^{n-1} \frac{x_{k+1}^2-x_k^2}{U(x_{k+1})} &\le&
\int_x^{(1+h)x}\frac{\widehat H^{(q)}(dy)}{U(y)}
\ \le\ (c_q+\varepsilon)\sum_{k=0}^{n-1} \frac{x_{k+1}^2-x_k^2}{U(x_k)},
\end{eqnarray*}
hence
\begin{eqnarray*}
(c_q-\varepsilon)\frac{2h}{n}\sum_{k=0}^{n-1} \frac{x_{k+1}+x_k}{U(x_{k+1})} &\le&
\int_x^{(1+h)x}\frac{\widehat H^{(q)}(dy)}{U(y)}
\ \le\ (c_q+\varepsilon)\frac{2h}{n}\sum_{k=0}^{n-1} \frac{x_{k+1}+x_k}{U(x_k)}.
\end{eqnarray*}
Letting $n\to\infty$ and taking into account that the function $y/U(y)$ 
is regularly varying at infinity we derive that
\begin{eqnarray*}
\int_x^{(1+h)x}\frac{\widehat H^{(q)}(dy)}{U(y)}
&\sim& 2c_q\int_x^{(1+h)x}\frac{y}{U(y)}dy,
\end{eqnarray*}
which together with \eqref{pi.repr} concludes the proof.
\qed\end{proof}
\end{theopargself}

\begin{corollary}
\label{Cor.pos.recurrence}
Assume that the conditions of Theorem \ref{thm:pi.recurrent} are valid. Then
the integrability of the function $y/U(y)$ at infinity is necessary and sufficient for
the Markov chain $\{X_n\}$ on $\Rp$ to be positive recurrent.
\end{corollary}

\section{Local asymptotics of stationary probabilities}
\sectionmark{Local asymptotics of stationary probabilities}
\label{sec:Lamperti.critical.local}

In this section we derive sharp local asymptotics
for a stationary measure $\pi$ of recurrent irreducible Markov chain
with asymptotically zero drift of order $1/x$ at infinity.
Following Section \ref{sec:ren.local},
we assume that the jumps $\xi(x)$ converge weakly
to some random variable $\xi$ on $\R$, that is,
the asymptotic homogeneity condition \eqref{asymp.hom.i} holds.
\index{Markov chain!invariant measure!local power asymptotics}
\index{Invariant measure!local power asymptotics}

\begin{theorem}\label{thm:srt.pi.i}
Let a recurrent Markov chain $\{X_n\}$ with invariant measure $\pi(\cdot)$ 
satisfy the conditions of Theorem \ref{thm:pi.recurrent}. 
In addition, let $\xi(x)\Rightarrow\xi$ as $x\to\infty$
where $\E\xi=0$, $\E\xi^2=b$, and
\begin{eqnarray}\label{majorant_second}
|\xi(y)| \I\{|\xi(y)|\le s(y)\} &\le_{st}& 
\Xi\quad\mbox{for all }y\ge 0,
\end{eqnarray}
where $\E\Xi^2<\infty$. Then, in the lattice case,
\begin{eqnarray}\label{ren.loc.h.pi}
\pi(x) &\sim& c\frac{x}{U(x)}\quad\mbox{as }x\to\infty,
\end{eqnarray}
for some $c>0$. In the non-lattice case, for any $h>0$,
\begin{eqnarray}\label{ren.loc.h.pi.i}
\pi(x,x+h] &=& ch\frac{x}{U(x)}\quad\mbox{as }x\to\infty.
\end{eqnarray}
\end{theorem}

\index{Markov chain!invariant measure!local power asymptotics}
\index{Invariant measure!local power asymptotics}
\begin{corollary}\label{cor:pi.1x.pi}
Let, in addition, $r(x)=2\mu/bx$ and either $2\mu/b\in(-1,1)$ or $2\mu/b>1$,
so either null or positive recurrence holds respectively. 
Then, in the lattice case,
$$
\pi(x)\sim c\rho x^{-2\mu/b}\quad\mbox{as }x\to\infty,
$$
which agrees with the global asymptotics given in Corollary \ref{cor:pi.1x};
$\rho=2\mu/b+1>0$. In the non-lattice case, for any $h>0$,
$$
\pi(x,x+h] \sim ch\rho x^{-2\mu/b}\quad\mbox{as }x\to\infty.
$$
\end{corollary}

In the case $2\mu/b=1$, we have the following result.
\index{Markov chain!invariant measure!local power asymptotics}
\index{Invariant measure!local power asymptotics}

\begin{corollary}\label{cor:pi.1x.pi.2mu=b}
Let, in addition, for some $m\ge1$ and $\gamma\not=0$,
\[
r(x)\ =\ \frac{1}{x}+\frac{1}{x\log x}
+\ldots+\frac{1}{x\log x\cdot\ldots\cdot\log_{(m-1)}x}
+\frac{1+\gamma}{x\log x\cdot\ldots\cdot\log_{(m)}x}\quad\mbox{as }x\to\infty,
\] 
there $\gamma<0$ corresponds to null recurrence 
while $\gamma>0$ --- to positive recurrence. 
Then, in the lattice case,
$$
\pi(x)\ \sim\ \frac{2c}{x\log x\ldots\log_{(m-1)}x\log_{(m)}^{1+\gamma}x}
\quad\mbox{as }x\to\infty,
$$
which agrees with the global asymptotics given in Corollary \ref{cor:pi.1x.2mu=b}. 
In the non-lattice case, for any $h>0$,
$$
\pi(x,x+h]\ \sim\ \frac{2ch}{x\log x\ldots\log_{(m-1)}x\log_{(m)}^{1+\gamma}x}
\quad\mbox{as }x\to\infty.
$$
\end{corollary}

\begin{theopargself}
\begin{proof}[of Theorem \ref{thm:srt.pi.i}]
As in the proof of Theorem \ref{thm:pi.recurrent},
it follows from the representation \eqref{pi.repr.B} that
\begin{eqnarray*}
\pi(x,x+h] &=& c^*\int_x^{x+h}\frac{\widehat H^{(q)}(dy)}{U_p(y)},
\end{eqnarray*}
where
\begin{eqnarray*}
\widehat H^{(q)}(dy) &:=& \sum_{n=0}^\infty \P\{\widehat X_n\in dy\}.
\end{eqnarray*}
Since the function $U_p(y)$ is regularly varying at infinity
(and hence long-tailed at infinity),
\begin{eqnarray*}
\pi(x,x+h] &\sim& c^*\frac{\widehat H^{(q)}(x,x+h]}{U_p(x)}
\quad\mbox{as }x\to\infty.
\end{eqnarray*}
The Markov chain $\{\widehat X_n\}$ satisfies all the conditions
of Corollary \ref{cor:regular} with $\widehat\mu=\mu+b$
and $\widehat b=b$, so $2\widehat\mu-\widehat b=2\mu+b>0$. 
Indeed, the conditions \eqref{m1.m2.1x}--\eqref{regular_left_tail} are 
checked in Lemma \ref{l:change} and \eqref{eq:irreducibility} right after that.
The weak convergence \eqref{asymp.hom.i} for $\widehat\xi(x)$,
that is $\widehat\xi(x)\Rightarrow\xi$, follows from
that for the original jumps $\xi(x)$ because $U_p(x+y)/U_p(x)\to 1$
as $x\to\infty$, for any fixed $y\in\R$.
Finally, the majorisation condition \eqref{majoriz.i} holds
with a square integrable majorant, since it follows from \eqref{def.q} and
\eqref{def.Q} that, for all sufficiently large $x$,
\begin{eqnarray*}
\lefteqn{\P\{\widehat\xi(x)\I\{|\widehat\xi(x)|\le s(x)|\}>y\}}\\ 
&=& \frac{Q(x,(x+y,x+s(x)])}{Q(x,\R)}\\
&\le& 2 \frac{\E\{U_p(x+\xi(x));\ \xi(x)\I\{|\xi(x)|\le s(x)|\}>y\}}{U_p(x)}\\
&\le& 2\frac{U_p(2x)}{U_p(x)}\P\{\xi(x)\I\{|\xi(x)|\le s(x)|\}>y\} 
+2\frac{\E\{U_p(2\xi(x));\ \xi(x)\I\{|\xi(x)|\le s(x)|\}>y\}}{U_p(x)}\\
&\le& c_1\P\{\Xi>y\} +\frac{c_1}{U(x)} \E\{U(\Xi);\ y<\Xi\le s(x)\},
\end{eqnarray*}
owing to the regular variation of the function $U_p$,
and the conditions \eqref{equiv.for.U.1} and \eqref{majorant_second}. Since $s(x)\le x$,
\begin{eqnarray*}
\P\{\widehat\xi(x)\I\{|\widehat\xi(x)|\le s(x)|\}>y\} &\le& 2c_1\P\{\Xi>y\},
\end{eqnarray*}
which implies that 
\begin{eqnarray*}
\widehat\xi(x)\I\{|\widehat\xi(x)|\le s(x)|\} &\le_{st}& \widehat\Xi,
\end{eqnarray*}
where $\E\widehat\Xi^2<\infty$ due to the assumption $\E\Xi^2<\infty$.
In addition, 
\begin{eqnarray*}
\P\{\widehat\xi(x)\I\{|\widehat\xi(x)|\le s(x)|\}<-y\} &=& 
\frac{Q(x,[x-s(x),x-y))}{Q(x,\R)}\\
&\le& 2 \frac{\E\{U_p(x+\xi(x));\ \xi(x)\I\{|\xi(x)|\le s(x)|\}<-y\}}{U_p(x)}\\
&\le& 2\P\{\xi(x)\I\{|\xi(x)|\le s(x)|\}<-y\}\\
&\le& 2\P\{\Xi>y\},
\end{eqnarray*}
which implies that $\widehat\xi(x)\ge_{st} -\widehat\Xi$,
and the proof of existence of a square integrable majorant for the family of 
$\widehat\xi(x)\I\{|\widehat\xi(x)|\le s(x)|\}$ is complete. 

Hence, by Corollary \ref{cor:regular} and Lemma \ref{thm:renewal.2} 
applied to the Markov chain $\{\widehat X_n\}$, we deduce that
\begin{eqnarray*}
\widehat H^{(q)}(x,x+h] &\sim& 
c_q\frac{h+o(1)}{2\widehat\mu-\widehat b}x\quad\mbox{as }x\to\infty,
\end{eqnarray*}
which concludes the proof because $U_p(x)\sim c_3U(x)$ 
as $x\to\infty$, see \eqref{equiv.for.U}.
\qed\end{proof}
\end{theopargself}

\section{Pre-stationary distribution of positive recurrent chain
with power-like stationary measure}
\sectionmark{Pre-stationary distribution of positive recurrent chain}
\label{sec:pre-st.power}

In this section we assume that the distribution of $X_n$
converges in total variation distance to a unique invariant
distribution $\pi$ as $n\to\infty$, that is,
\begin{eqnarray}\label{total.var}
\sup_{A\in\mathcal B(\R)}|\P\{X_n\in A\}-\pi(A)| &\to& 0
\quad\mbox{as }n\to\infty;
\end{eqnarray}
for a countable Markov chain $\{X_n\}$ this condition holds  automatically
provided the chain is irreducible, aperiodic, and positive recurrent;
for a real-valued chain it is related to the Harris ergodicity, see e.g. \cite{MT}.
\index{Markov chain!pre-stationary distribution!asymptotics}
\index{Pre-stationary distribution!power asymptotics}

\begin{theorem}\label{thm:pre-st.pos.rec.}
Assume that all the conditions of Theorem \ref{thm:pi.recurrent} are valid
and that $\{X_n\}$ is positive recurrent satisfying \eqref{total.var}. Then
\begin{eqnarray*}
\P\{X_n>x\} &=& (F(n/x^2)+o(1))\pi(x,\infty)
\end{eqnarray*}
as $x\to\infty$ uniformly for all $n$, where
$$
F(u)\ :=\ \frac{1}{\widehat I(\infty)}
\int_0^u \widehat\Gamma(1/z)\Bigl[1-\frac{\rho}{2}
\Bigl(\frac{z}{u}\Bigr)^{\rho/2-1}\Bigr)\Bigr]dz
%\frac{\widehat I(u)+\rho\int_1^\infty z^{1-\rho}\widehat I(u/z^2)dz}
%{2\widehat I(\infty)\frac{\rho-1}{\rho-2}}
$$
is a continuous distribution function; $\widehat I$ and $\widehat\Gamma$
are functions defined in Theorem \ref{thm:renewal.gamma} 
with $\widehat\mu=\mu+b$ and $\widehat b=b$. In particular, if $n/x^2\to u\in(0,\infty)$ then
\begin{eqnarray*}
\P\{X_n>x\} &\sim& F(u)\pi(x,\infty),
\end{eqnarray*}
and if $n/x^2\to\infty$ then
\begin{eqnarray*}
\P\{X_n>x\} &\sim& \pi(x,\infty).
\end{eqnarray*}
\end{theorem}

\begin{proof}
Splitting all the paths according to the time of the last visit of $\{X_n\}$ to 
$B=(-\infty,\widehat x]$, see \eqref{repr.Xn.B.U.x}, 
we get, for $x>\widehat x$,
\begin{eqnarray}\label{pre-st.1}
\P\{X_n>x\} &=&
\sum_{j=1}^n\int_B\P\{X_{n-j}\in dz\}\int_{\widehat x}^\infty
P(z,du)U_p(u)\E_u\biggl\{\frac{e^{-\sum_{k=0}^{j-2}q(\widehat X_k)}}
{U_p(\widehat X_{j-1})};\ \widehat X_{j-1}>x\biggr\},
\nonumber\\[-1mm]
\end{eqnarray}
where $q(x)\ge 0$ and $\{\widehat X_n\}$ are defined in 
\eqref{def.q.B} and \eqref{P.hat.B} respectively.

Fix a sequence $N_x\to\infty$ such that $N_x=o(x^2)$.
Then, since $q\ge 0$ and $U_p$ is increasing,
\begin{eqnarray}\label{pre-st.2}
\lefteqn{\sum_{j=n-N_x+1}^n\int_B\P\{X_{n-j}\in dz\}
\int_{\widehat x}^\infty P(z,du)U_p(u)
\E_u\biggl\{\frac{e^{-\sum_{k=0}^{j-2}q(\widehat X_k)}}
{U_p(\widehat X_{j-1})};\ \widehat X_{j-1}>x\biggr\}}
\nonumber\\
&&\hspace{20mm}\le\ N_x\frac{1}{U_p(x)}\sup_{z\in B}
\int_{\widehat x}^\infty P(z,du) U_p(u)\phantom{mmmmmmmmmmmmmmmmm}\nonumber\\
&&\hspace{20mm}\le\ N_x\frac{c}{U_p(x)}\sup_{z\in B}(1+U_p(z))\nonumber\\
&&\hspace{20mm}=\ o(x^2/U_p(x)),
\end{eqnarray}
where the second bound follows from \eqref{cond.for.U.unif.52}.
Furthermore, the distribution of $X_{n-j}$ converges in total variation to $\pi$
uniformly for all $j\le n-N_x$, see \eqref{total.var}. Therefore, as $x\to\infty$,
\begin{eqnarray}\label{pre-st.3}
\lefteqn{\nonumber
\sum_{j=1}^{n-N_x}\int_B\P\{X_{n-j}\in dz\}
\int_{\widehat x}^\infty P(z,du)U_p(u)
\E_u\biggl\{\frac{e^{-\sum_{k=0}^{j-2}q(\widehat X_k)}}
{U_p(\widehat X_{j-1})};\ \widehat X_{j-1}>x\biggr\}}\\
&&\hspace{10mm}\sim
\sum_{j=1}^{n-N_x}\int_B\pi(dz)\int_{\widehat x}^\infty P(z,du)U_p(u)
\E_u\biggl\{\frac{e^{-\sum_{k=0}^{j-2}q(\widehat X_k)}}
{U_p(\widehat X_{j-1})};\ \widehat X_{j-1}>x\biggr\}.\phantom{mmmmm}
\end{eqnarray}
Similarly to \eqref{pre-st.2},
\begin{eqnarray}\label{pre-st.4}
\sum_{j=n-N_x+1}^n\int_B\pi(dz)\int_{\widehat x}^\infty P(z,du)U_p(u)
\E_u\biggl\{\frac{e^{-\sum_{k=0}^{j-2}q(\widehat X_k)}}
{U_p(\widehat X_{j-1})};\ \widehat X_{j-1}>x\biggr\}
&=& o(x^2/U_p(x)).\nonumber\\[-1mm]
\end{eqnarray}
Combining \eqref{pre-st.1}---\eqref{pre-st.4}, we obtain
\begin{eqnarray}\label{pre-st.5}
\nonumber
\lefteqn{\P\{X_n>x\}}\\
&=& (1+o(1))\sum_{j=1}^n \int_B\pi(dz)\int_{\widehat x}^\infty P(z,du)U_p(u)
\E_u\biggl\{\frac{e^{-\sum_{k=0}^{j-2}q(\widehat X_k)}}
{U_p(\widehat X_{j-1})};\ \widehat X_{j-1}>x\biggr\}\nonumber\\
&&\hspace{80mm} +o\biggl(\frac{x^2}{U_p(x)}\biggr)\nonumber\\
&=& (1+o(1))\int_B\pi(dz)\int_{\widehat x}^\infty P(z,du)U_p(u)
\sum_{j=1}^n \int_x^\infty
\E_u\biggl\{\frac{e^{-\sum_{k=0}^{j-2}q(\widehat X_k)}}{U_p(y)};\ 
\widehat X_{j-1}\in dy\biggr\}\nonumber\\
&&\hspace{80mm}+o\biggl(\frac{x^2}{U_p(x)}\biggr)\nonumber\\
&=& (1+o(1))\int_{\widehat x}^\infty \mu(du)U_p(u)
\int_x^\infty\frac{\widehat H^{(q)}_{u,n}(dy)}{U_p(y)}
+o\biggl(\frac{x^2}{U_p(x)}\biggr)\quad\mbox{as }x\to\infty,
\end{eqnarray}
where
\begin{eqnarray*}
\mu(du) &=& \int_B\pi(dz)P(z,du)
\end{eqnarray*}
is a measure on $(\widehat x,\infty)$, see \eqref{6.mu.B}, and 
\begin{eqnarray*}
\widehat H^{(q)}_{u,n}(A) &:=&
\sum_{j=1}^n 
\E_u\Bigl\{e^{-\sum_{k=0}^{j-2}q(\widehat X_k)};\ \widehat X_{j-1}\in A\Bigr\}
\end{eqnarray*}
is a measure on $(\widehat x,\infty)$ too.

For any fixed $u>\widehat x$, due to Corollary \ref{cor:renewal.2.W.g}, 
\begin{eqnarray}\label{hat.H.u.n}
\widehat H_{u,n}^{(q)}(\widehat x,y] 
&\sim& h(u) (\widehat I(n/y^2)+o(1))y^2
\quad\mbox{as } y\to\infty\mbox{ uniformly for all }n.
\end{eqnarray}
In addition, due to $q\ge 0$, 
\begin{eqnarray}\label{hat.H.u}
\sup_{u>\widehat x}\widehat H_{u,n}^{(q)}(\widehat x,y] 
&\le& \sup_{u>\widehat x}\sum_{j=1}^n \P_u\{\widehat X_{j-1}\in (\widehat x,y]\}
\ \le\ c_1y^2\quad\mbox{for all }y\mbox{ and }n,
\end{eqnarray}
for some $c_1<\infty$ as follows from 
the integral renewal theorem for $\{\widehat X_n\}$.
Integration by parts together with \eqref{hat.H.u.n}
implies that, for any fixed $u>\widehat x$,
\begin{eqnarray*}
\int_x^\infty \frac{\widehat H^{(q)}_{u,n}(dz)}{U_p(z)}
&=& -\frac{\widehat H^{(q)}_{u,n}(\widehat x,x]}{U_p(x)}
-\int_x^\infty \widehat H^{(q)}_{u,n}(\widehat x,z]d\frac{1}{U_p(z)}\\
&=& h(u)\Biggl[-\frac{\widehat I(n/x^2)x^2}{U_p(x)}
-\int_x^\infty \widehat I(n/z^2)z^2 d\frac{1}{U_p(z)}\Biggr]
+o\Bigl(\frac{x^2}{U_p(x)}\Bigr)
\end{eqnarray*}
as $x\to\infty$ uniformly for all $n$.
Taking into account that
$$
-\frac{d}{dz}\frac{1}{U_p(z)}\ =\ \frac{U'_p(z)}{U_p^2(z)}
\ =\ \frac{e^{R_p(z)}}{U_p^2(z)}\ \sim\ \frac{2\mu+b}{b}\frac{1}{zU_p(z)}
\quad\mbox{as }z\to\infty,
$$
owing to \eqref{equiv.for.U.1}, we deduce
\begin{eqnarray*}
\int_x^\infty \frac{\widehat H^{(q)}_{u,n}(dz)}{U_p(z)}
&=& h(u)\Biggl[-\frac{\widehat I(n/x^2)x^2}{U_p(x)}
+\frac{2\mu+b}{b}\int_x^\infty \frac{\widehat I(n/z^2)z}{U_p(z)}dz\Biggr]
+o\Bigl(\frac{x^2}{U_p(x)}\Bigr)\\
&=& h(u)\Biggl[-\frac{\widehat I(n/x^2)x^2}{U_p(x)}
+\frac{2\mu+b}{b}x^2\int_1^\infty \frac{\widehat I(n/x^2z^2)z}{U_p(xz)}dz\Biggr]
+o\Bigl(\frac{x^2}{U_p(x)}\Bigr).
\end{eqnarray*}
Since the function $U_p$ is regularly varying at infinity with index
$\rho=2\mu/b+1>2$, $U_p(xz)/U_p(x)\to z^\rho$ as $x\to\infty$. Therefore,
\begin{eqnarray}\label{pre-st.6}
\int_x^\infty \frac{\widehat H^{(q)}_{u,n}(dz)}{U_p(z)}
&=& h(u) \frac{x^2}{U_p(x)} \Biggl[-\widehat I(n/x^2)
+\rho\int_1^\infty \frac{\widehat I(n/x^2z^2)}{z^{\rho-1}}dz\Biggr]
+o\Bigl(\frac{x^2}{U_p(x)}\Bigr)\nonumber\\
&=& h(u) \frac{x^2}{U_p(x)} \widehat F_0(n/x^2)
+o\Bigl(\frac{x^2}{U_p(x)}\Bigr)
\end{eqnarray}
as $x\to\infty$ uniformly for all $n$, where
\begin{eqnarray*}
\widehat F_0(t) &:=& 
-\widehat I(t)+\rho\int_1^\infty \frac{\widehat I(t/z^2)}{z^{\rho-1}}dz\\
&=& -\widehat I(t)-\frac{\rho}{\rho-2} \int_1^\infty \widehat I(t/z^2) d\frac{1}{z^{\rho-2}}\\
&=& \widehat I(t)\frac{2}{\rho-2}
+\frac{\rho}{\rho-2} \int_1^\infty \frac{1}{z^{\rho-2}} d\widehat I(t/z^2).
\end{eqnarray*}
Therefore,
\begin{eqnarray*}
\widehat F_0(t) &=& 
\widehat I(t)\frac{2}{\rho-2}
-\frac{2\rho t}{\rho-2} \int_1^\infty \frac{1}{z^{\rho+1}}\widehat \Gamma(z^2/t)dz\\
&=& \widehat I(t)\frac{2}{\rho-2}
-\frac{\rho t^{1-\rho/2}}{\rho-2} \int_0^t u^{\rho/2-1}\widehat \Gamma(1/u)du\\
&=& \frac{2}{\rho-2} \int_0^t \widehat \Gamma(1/u)
\biggl(1-\frac{\rho}{2} \biggl(\frac{u}{t}\biggr)^{\rho/2-1}\biggr)du.
\end{eqnarray*}
Similarly, it follows from \eqref{hat.H.u} that
\begin{eqnarray*}
\sup_{u>\widehat x}\int_x^\infty \frac{\widehat H^{(q)}_{u,n}(dz)}{U_p(z)}
&\le& c_2\frac{x^2}{U_p(x)}.
\end{eqnarray*}
In addition,
\begin{eqnarray*}
\widehat c &=& \int_{\widehat x}^\infty h(u)U_p(u) \mu(du)\\
&=& \int_B \pi(dz)\int_{\widehat x}^\infty h(u)U_p(u) P(z,du)\ <\ \infty,
\end{eqnarray*}
as follows from \eqref{cond.for.U.unif.52}.
Hence the dominated convergence theorem is applicable to 
\eqref{pre-st.5}, so plugging \eqref{pre-st.6} into \eqref{pre-st.5}, we obtain
\begin{eqnarray}\label{pre-st.7}
\P\{X_n>x\} &=& \widehat c\frac{x^2}{U_p(x)}(\widehat F_0(n/x^2)+o(1))
\end{eqnarray}
as $x\to\infty$ uniformly for all $n$.
In particular, letting $n\to\infty$ we get that
\begin{eqnarray*}
\pi(x,\infty)\ =\ \lim_{n\to\infty}\P\{X_n>x\} &\sim&
\widehat c\frac{x^2}{U_p(x)}\widehat F_0(\infty)\ =\
\widehat c \frac{x^2}{U_p(x)}2\frac{\rho-1}{\rho-2}\widehat I(\infty),
\end{eqnarray*}
which concludes the proof.
\qed\end{proof}

\section{Tail asymptotics for recurrence times
of positive and null recurrent Markov chains}
\sectionmark{Tail asymptotics for recurrence times}
\label{sec:recurrence.times}

In this section we study the tail behaviour of the stopping time
$$
\tau_{\widehat x}:=\inf\{n\ge 1:\ X_n\le\widehat x\},
$$
in the case where $\tau_{\widehat x}$ is a proper random variable,
that is, $\{X_n\}$ is either positive or null recurrent
with respect to the set $(-\infty,\widehat x]$.
\index{Positive recurrence!asymptotics for return time}

\begin{theorem}\label{thm:rec.time}
Let the conditions of Theorem \ref{thm:pi.recurrent} hold.
Let $\widehat x$ be chosen as in Corollary \ref{l:lyapunov}
and Lemma \ref{l:change}.
Then there exists a constant $c<\infty$ such that
\begin{eqnarray}\label{rec.time:tau.upper}
\P_x\{\tau_{\widehat x}>n\} &\le&
c\frac{U(x)}{U(\sqrt{n})}\quad\mbox{for all }n\mbox{ and }x>\widehat x.
\end{eqnarray}
Further, for any fixed $x>\widehat x$,
\begin{eqnarray}\label{rec.time:tau.x.asy}
\P_x\{\tau_{\widehat x}>n\} &\sim&
\frac{1}{(2b)^{\rho/2}\Gamma(1+\rho/2)}
\frac{W(x)}{U(\sqrt{n})}\quad\mbox{as }n\to\infty,
\end{eqnarray}
where $W(x)$ is the harmonic function defined in \eqref{W.as.lim}.

In addition, if $X_0>\widehat x$ a.s. and $\E U(X_0)<\infty$ then
\begin{eqnarray}\label{rec.time:tau.asy}
\P\{\tau_{\widehat x}>n\} &\sim&
\frac{\E W(X_0)}{(2b)^{\rho/2}\Gamma(1+\rho/2)}
\frac{1}{U(\sqrt{n})}\quad\mbox{as }n\to\infty.
\end{eqnarray}
\end{theorem}

Notice that
\begin{eqnarray*}
\frac{W(x)}{U(\sqrt n)} &\sim&
\frac{U(x)}{U(\sqrt{n})}\quad\mbox{as }n,\ x\to\infty,
\end{eqnarray*}
due to Lemma \ref{L.limit} and the equivalence \eqref{equiv.for.U}.

In order to prove the upper bound \eqref{rec.time:tau.upper}
for the tail of $\tau_{\widehat x}$ we need a couple of preliminary results.
In Theorem \ref{l:est.for.return} we have already constructed
a function of a transient Markov chain which is a bounded supermartingale.
It turns out that for the Markov chain $\{\widehat X_n\}$ which is specially constructed
a similar result is valid under weaker conditions on the left tail distribution.
Recall the definition of the function $U_p$ in \eqref{R.U.p.once.more}.

\begin{lemma}\label{l:tau.upper.1}
For any $\varepsilon\in(0,1)$ and $a>1/\varepsilon$,
there exists an $x_*>\widehat x$ such that
$$
\min\biggl(\frac{U_{ap}^\varepsilon(\widehat X_n)}{U_p(\widehat X_n)},\
\frac{U_{ap}^\varepsilon(x_*)}{U_p(x_*)}\biggr)
$$
is a positive supermartingale.
\end{lemma}

\begin{proof}
By the definition of the chain $\{\widehat X_n\}$ and Jensen's inequality,
\begin{eqnarray}\label{E1Ua}
\E\frac{U_{ap}^\varepsilon(x+\widehat{\xi}(x))}{U_p(x+\widehat{\xi}(x))} &=&
\int_{\widehat x}^\infty \frac{U_{ap}^\varepsilon(y)}{U_p(y)}\frac{Q(x,dy)}{Q(x,\R)}\nonumber\\
&=& \frac{1}{\int_{\widehat x}^\infty U_p(y)P(x,dy)}
\int_{\widehat x}^\infty U_{ap}^\varepsilon(y) P(x,dy)\nonumber\\
&\le& \frac{1}{\int_{\widehat x}^\infty U_p(y)P(x,dy)}
\Bigl( \int_{\widehat x}^\infty U_{ap}(y) P(x,dy)\Bigr)^\varepsilon.
\end{eqnarray}
Due to Lemma \ref{L.Lyapunov} and \eqref{Up.below.xhat}, as $x\to\infty$,
\begin{eqnarray}\label{E1Ua.1}
\int_{\widehat x}^\infty U_p(y)P(x,dy)
&=& U_p(x)\Bigl(1-\frac{2\mu+b}{2}\frac{p(x)}{x}+o\Bigl(\frac{p(x)}{x}\Bigr)\Bigr)
\end{eqnarray}
and
\begin{eqnarray*}
\int_{\widehat x}^\infty U_{ap}(y)P(x,dy)
&=& U_{ap}(x)\Bigl(1-a\frac{2\mu+b}{2}\frac{p(x)}{x}+o\Bigl(\frac{p(x)}{x}\Bigr)\Bigr).
\end{eqnarray*}
Then
\begin{eqnarray*}
\Bigl(\int_{\widehat x}^\infty U_{ap}(y)P(x,dy)\Bigr)^\varepsilon
&=& U_{ap}^\varepsilon(x)\Bigl(1-a\varepsilon\frac{2\mu+b}{2}\frac{p(x)}{x}+o\Bigl(\frac{p(x)}{x}\Bigr)\Bigr)
\end{eqnarray*}
and it follows from $a\varepsilon>1$ that
\begin{eqnarray*}
\frac{1}{\int_{\widehat x}^\infty U_p(y)P(x,dy)}
\Bigl( \int_{\widehat x}^\infty U_{ap}(y) P(x,dy)\Bigr)^\varepsilon
&\le& \frac{U_{ap}^\varepsilon(x)}{U_p(x)}
\end{eqnarray*}
for all sufficiently large $x$, which completes the proof.
\qed\end{proof}

\begin{lemma}\label{l:T.exp}
For
\begin{eqnarray*}
\widehat T(z) &:=& \min\{n\ge 0:\widehat X_n>z\},\quad z>\widehat x,
\end{eqnarray*}
there exists a $\gamma>0$ such that, for all $n$ and $z$,
\begin{eqnarray*}
\sup_{x}\P_x\{T(z)>n\} &\le& c_4e^{-\gamma n/z^2}.
\end{eqnarray*}
\end{lemma}

\begin{proof}
It follows from the definition of the chain $\{\widehat X_n\}$ that
it can only visit $(-\infty,\widehat x]$ at time $0$. Therefore,
\begin{eqnarray*}
\widehat{T}(z) &\le& 1+\sum_{k=1}^{\widehat{T}(z)-1}\I\{\widehat X_k>\widehat x\}.
\end{eqnarray*}
Then, by Theorem \ref{l:uniform} with $v(x)\ge c/x$,
\begin{eqnarray}\label{T.upper.1}
\E_x\widehat{T}(z) &\le& c_2z^2+s(z)\ \le\ c_3z^2\quad\mbox{uniformly for all }x\mbox{ and }z.
\end{eqnarray}
Next, by the Markov property, for all $t$ and $s>0$,
\begin{eqnarray*}
\P_x\{\widehat T(z)>t+s\} &=&
\int_0^z \P_x\{\widehat T(z)>t, X_t\in du\}\P_u\{\widehat T(z)>s\}\\
&\le& \P_x\{\widehat T(z)>t\}\sup_{u\le z}\P_u\{\widehat T(z)>s\}.
\end{eqnarray*}
Therefore, a decreasing function
$g(t):=\sup_{u\le z}\P_u\{\widehat T(z)>tz^2\}$
satisfies the inequality $g(t+s)\ge g(t)g(s)$ and $g(0)=1$.
Then an increasing function $g_0(t):=\log(1/g(t))$ is convex
due to $g_0(t+s)\le g_0(t)+g_0(s)$ and $g_0(0)=0$.
By the bound \eqref{T.upper.1} and Markov's inequality,
there exists a $t_0$ such that $g(t_0)<1$ so that
$g(t_0)=e^{-\gamma}$ with $\gamma>0$, and $g_0(t_0)=\gamma>0$.
Then, by $g_0(0)=0$ and by the convexity of $g_0$, 
$g_0(t)\ge \gamma(t-t_0)$ for $t\ge t_0$,
which implies $g(t)\le e^{-\gamma(t-t_0)}$ equivalent to the lemma conclusion.
\qed\end{proof}

\begin{lemma}\label{l:n.2n}
For any fixed $\varepsilon\in(0,\rho)$,
there exists a constant $c_6=c_6(\varepsilon)$ such that,
for all $n$, $x$ and $y\in(x_*,\sqrt n]$,
\begin{eqnarray*}
\P_x\{\widehat X_k \le y\mbox{ for some }k\in[n+1,2n]\} &\le&
c_6\Bigl(\frac{y}{\sqrt n}\Bigr)^{\rho-\varepsilon}.
\end{eqnarray*}
\end{lemma}

\begin{proof}
For any $z>y$, the event whose probability we need to bound 
can only occur if either the chain $\{\widehat X_n\}$ does not exceed the
level $z$ by time $n$ or it does exceed this level and then falls
down below $y$. Therefore, by the Markov property,
the corresponding probability is not greater than the sum
\begin{eqnarray}\label{abvgd}
\P_x\{T(z)>n\}+\sup_{u\ge z}\P_u\{\widehat X_k \le y\mbox{ for some }k\ge 1\},
\end{eqnarray}
where the first term may be bounded above by Lemma \ref{l:T.exp}.
For the second term, by Lemma \ref{l:tau.upper.1},
we can apply the Doob inequality for supermartingales which
guarantees that there exists a constant $c_1(\varepsilon)$ such that,
for all $u\ge z$,
\begin{eqnarray*}
\P_u\{\widehat X_k\le y\text{ for some }k\ge 1\}
&\le& c_1(\varepsilon)\frac{U_p(y)}{U_p(u)}
\frac{U_{ap}^{\varepsilon/2\rho}(u)}{U_{ap}^{\varepsilon/2\rho}(y)}
\quad\mbox{for all }y\in(x_*,u].
\end{eqnarray*}
Hence the equivalence \eqref{equiv.for.U}
implies the existence of $c_2(\varepsilon)$ such that
\begin{eqnarray*}
\P_u\{\widehat X_k\le y\text{ for some }k\ge 1\}
&\le& c_2(\varepsilon)\biggl(\frac{U(y)}{U(u)}\biggr)^{1-\varepsilon/2\rho}
\quad\mbox{for all }y\in(x_*,u].
\end{eqnarray*}
Since $U$ is regularly varying at infinity with index $\rho$,
by Potter's bounds,
there exists a constant $c_3(\varepsilon)$ such that
\begin{eqnarray}\label{potter}
\frac{1}{c_3(\varepsilon)}\left(\frac{y}{u}\right)^{\rho+\varepsilon/2}
\ \le\ \frac{U(y)}{U(u)} &\le&
c_3(\varepsilon)\left(\frac{y}{u}\right)^{\rho-\varepsilon/2}
\quad\mbox{for all }y\in(x_*,u].
\end{eqnarray}
Consequently,
\begin{eqnarray}\label{newstep.3}
\sup_{u\ge z}\P_u\{\widehat X_k\le y\text{ for some }k\ge1\}
&\le& c_4(\varepsilon)\left(\frac{y}{z}\right)^{\rho-\varepsilon}.
\end{eqnarray}
Therefore, the estimates for each term in the upper bound \eqref{abvgd} give
\begin{eqnarray*}
\P_x\{\widehat X_k \le y\mbox{ for some }k\in[n+1,2n]\} &\le&
c_5\Bigl(e^{-\gamma n/z^2}+\Bigl(\frac{y}{z}\Bigr)^{\rho-\varepsilon}\Bigr).
\end{eqnarray*}
Optimisation of the right hand side with respect to $z$ is not
solvable in elementary functions, so we choose
\begin{eqnarray*}
z &:=& \sqrt{\frac{\gamma n}{\log((\sqrt n/y)^{\rho-\varepsilon})}},
\end{eqnarray*}
which is close to the optimal value. Then
\begin{eqnarray*}
\lefteqn{\P_x\{\widehat X_k \le y\mbox{ for some }k\in[n+1,2n]\}}\\
&&\hspace{30mm}\le\ c_5\Bigl(\Bigl(\frac{y}{\sqrt n}\Bigr)^{\rho-\varepsilon}
+\Bigl(\frac{y}{\sqrt{\gamma n}}\Bigr)^{\rho-\varepsilon}
\Bigl((\rho-\varepsilon)\log\frac{\sqrt n}{y}\Bigr)^{\frac{\rho-\varepsilon}{2}},
\end{eqnarray*}
which implies the lemma conclusion if we take $\varepsilon/2$ instead of
$\varepsilon$ on the right hand side.
\qed\end{proof}

\begin{theopargself}
\begin{proof}[of Theorem \ref{thm:rec.time}]
We start with the upper bound \eqref{rec.time:tau.upper}
which is the most difficult part of the theorem.
It follows from \eqref{connection.new} that
\begin{eqnarray}\label{tau.integral}
\P_x\{\tau_{\widehat x}>n\} &=&
U_p(x)\int_{\widehat x}^\infty \frac{1}{U_p(y)} Q^n(x,dy)\nonumber\\
&=& U_p(x) \int_{\widehat x}^\infty \frac{1}{U_p(y)}
\E_x\bigl\{e^{-\sum_{k=0}^{n-1}q(\widehat X_k)};\ \widehat X_n\in dy\bigr\}.
\end{eqnarray}
Since $q(x)\ge 0$,
\begin{eqnarray}\label{tau.integral.2}
\P_x\{\tau_{\widehat x}>n\} &\le& U_p(x) \E_x \frac{1}{U_p(\widehat X_n)}\nonumber\\
&\le& c_1U(x) \E_x \frac{1}{U(\widehat X_n)},
\end{eqnarray}
due to \eqref{equiv.for.U}. Summing up $n$ successive probabilities we get
\begin{eqnarray}\label{tau.2}
\sum_{k=n+1}^{2n}\P_x\{\tau_{\widehat x}>k\}
&\le& c_1U(x)\int_{\widehat x}^\infty \frac{1}{U(y)} \widehat H_{x,n}(dy)\nonumber\\
&=& c_1U(x)\biggl(\int_{\widehat x}^{\sqrt n}
+\int_{\sqrt n}^\infty\biggr) \frac{1}{U(y)} \widehat H_{x,n}(dy),
\end{eqnarray}
where
\begin{eqnarray*}
\widehat H_{x,n}(A) &:=& \sum_{k=n+1}^{2n} \P_x\{\widehat X_k\in A\}.
\end{eqnarray*}
The function $U$ increases, so
\begin{eqnarray}\label{tau.3}
\int_{\sqrt{n}}^\infty \frac{1}{U(y)} \widehat H_{x,n}(dy) &\le&
\frac{n}{U(\sqrt{n})}\quad\mbox{for all }x \mbox{ and }n.
\end{eqnarray}
Further, integrating by parts, we obtain
\begin{eqnarray*}
\int_{\widehat x}^{\sqrt{n}}\frac{1}{U(y)}\widehat H_{x,n}(dy)
&=& \frac{\widehat H_{x,n}(\widehat x,\sqrt{n}]}{U(\sqrt{n})}
+\int_{\widehat x}^{\sqrt{n}}\frac{U'(y)\widehat H_{x,n}(\widehat x,y]}{U^2(y)}dy\\
&\le& \frac{n}{U(\sqrt{n})}+\int_{\widehat x}^{\sqrt{n}}
\frac{e^{R(y)}\widehat H_{x,n}(\widehat x,y]}{U^2(y)}dy,
\end{eqnarray*}
owing to $U'=e^R$.
Combining this with \eqref{tau.2}, \eqref{tau.3} and
noting that $e^{R(y)}\sim\rho U(y)/y$, we conclude that
\begin{eqnarray}\label{tau.4}
\sum_{k=n+1}^{2n}\P_x\{\tau_{\widehat x}>k\}
&\le& 2c_1U(x)\frac{n}{U(\sqrt{n})}+c_2U(x)\int_{\widehat x}^{\sqrt{n}}
\frac{\widehat H_{x,n}(\widehat x,y]}{yU(y)}dy,
\end{eqnarray}
with some constant $c_2$ which does not depend on $x$.

Next we derive an upper bound for $\widehat H_{x,n}$. It is clear that
\begin{eqnarray*}
\widehat H_{x,n}(\widehat x,y]
&=& \E_x \sum_{k=n+1}^{2n}\I\{\widehat X_k\in(\widehat x,y]\}\\
&\le& \P_x\{\widehat X_k\in(\widehat x,y]\text{ for some }k\in[n+1,2n]\}
\sup_{s\le y}\sum_{k=0}^\infty \P_s\{\widehat X_k\in(\widehat x,y]\}\\
&\le& \sup_s \widehat H_s(\widehat x,y]\P_x\{\widehat X_k\in(\widehat x,y]
\text{ for some }k\in[n+1,2n]\}.
\end{eqnarray*}
Applying here Theorem \ref{thm:Hy.above} and Lemma \ref{l:n.2n}, we get
\begin{eqnarray*}
\widehat H_{x,n}(\widehat x,y] &\le&
c_3y^2 \Bigl(\frac{y}{\sqrt n}\Bigr)^{\rho-\varepsilon}.
\end{eqnarray*}
Therefore,
\begin{eqnarray*}
\int_{\widehat x}^{\sqrt{n}}\frac{\widehat H_{x,n}(y)}{yU(y)}dy
&\le& c_4\int_{\widehat x}^{\sqrt{n}}\frac{y}{U(y)}
\Bigl(\frac{y}{\sqrt n}\Bigr)^{\rho-\varepsilon}
dy.
\end{eqnarray*}
Substitution $y=u\sqrt n$ leads to the following expression
for the last integral:
\begin{eqnarray*}
\frac{n}{U(\sqrt n)}\int_{\widehat x/\sqrt n}^1
\frac{U(\sqrt n)}{U(u\sqrt n)} u^{1+\rho-\varepsilon}du.
\end{eqnarray*}
Applying the left hand side inequality from \eqref{potter} we get an upper bound
\begin{eqnarray*}
\int_{\widehat x}^{\sqrt{n}}\frac{\widehat H_{x,n}(y)}{yU(y)}dy
&\le& c_5\frac{n}{U(\sqrt n)}\int_0^1 u^{1-3\varepsilon/2} du
\ =\ c_6\frac{n}{U(\sqrt n)},
\end{eqnarray*}
provided $\varepsilon<1$. Substituting this upper bound
into \eqref{tau.4} we get that
\begin{eqnarray*}
\sum_{k=n+1}^{2n}\P_x\{\tau_{\widehat x}>k\} &\le& CU(x)\frac{n}{U(\sqrt{n})}.
\end{eqnarray*}
Therefore,
\begin{eqnarray*}
\P_x\{\tau_{\widehat x}>2n\} &\le& C\frac{U(x)}{U(\sqrt{n})}.
\end{eqnarray*}
Since $U$ is regularly varying at infinity, this completes the proof
of the upper bound \eqref{rec.time:tau.upper}.

Now let us prove tail asymptotics for $\tau_{\widehat x}$.
Fix an $\varepsilon>0$ and split the integral \eqref{tau.integral}
into two parts
\begin{eqnarray}\label{tau.1}
\P_x\{\tau_{\widehat x}>n\} &=&
U_p(x)\biggl(\int_{\widehat x}^{\varepsilon\sqrt n}
+\int_{\varepsilon\sqrt n}^\infty\biggr) \frac{1}{U_p(y)} Q^n(x,dy).
\end{eqnarray}

The asymptotic behaviour of the second integral here
relatively easy follows from the weak convergence to a $\Gamma$-distribution
and dominated convergence theorem. Indeed,
\begin{eqnarray}\label{tau.1.1}
\int_{\varepsilon\sqrt n}^\infty \frac{1}{U_p(y)} Q^n(x,dy) &=&
\frac{1}{U_p(\sqrt n)}
\int_{\varepsilon\sqrt n}^\infty \frac{U_p(\sqrt n)}{U_p(y)} Q^n(x,dy).
\end{eqnarray}
Monotonicity of $U_p$ implies the following upper bound for the integrand
on the right hand side:
\begin{eqnarray}\label{tau.1.2}
\sup_{n,\ y>\varepsilon\sqrt n}\frac{U_p(\sqrt n)}{U_p(y)} &\le&
\sup_n \frac{U_p(\sqrt n)}{U_p(\varepsilon\sqrt n)}\ <\ \infty,
\end{eqnarray}
because $U_p$ is regularly varying at infinity which also implies convergence
\begin{eqnarray}\label{tau.1.3}
\frac{U_p(\sqrt n)}{U_p(u\sqrt n)} &\to& \frac{1}{u^\rho}
\quad\mbox{as }n\to\infty.
\end{eqnarray}
It follows from Theorem \ref{thm:gamma} that $\widehat X^2_n/n$ converges
weakly to a $\Gamma$-distribution with probability density function
$\gamma(u)$, see \eqref{gamma.density}. Then, by Lemma \ref{thm:renewal.2.p},
the substochastic measure $Q^n(x,\sqrt n\cdot du)$
converges weakly as $n\to\infty$ to a measure with density function
$h(x)2u\gamma(u^2)$. The relations \eqref{tau.1.2} and \eqref{tau.1.3}
allow us to apply the dominated convergence theorem and to conclude that,
as $n\to\infty$,
\begin{eqnarray*}
\int_\varepsilon^\infty \frac{U_p(\sqrt n)}{U_p(u\sqrt n)} Q^n(x,\sqrt n\cdot du)
&\to& h(x)\int_\varepsilon^\infty \frac{2u}{u^\rho} \gamma(u^2)du\\
&=& h(x)\int_{\varepsilon^2}^\infty \frac{1}{u^{\rho/2}} \gamma(u)du\\
&=& h(x) \frac{e^{-\varepsilon^2/2b}}{(2b)^{\rho/2}\Gamma(1+\rho/2)}.
\end{eqnarray*}
Hence, \eqref{tau.1.1} and \eqref{def:func.W} finally imply
\begin{eqnarray}\label{tau.1.4}
U_p(x)\int_{\varepsilon\sqrt n}^\infty \frac{1}{U_p(y)} Q^n(x,dy) &\sim&
\frac{h(x)U_p(x)}{U_p(\sqrt n)}
\frac{e^{-\varepsilon^2/2b}}{(2b)^{\rho/2}\Gamma(1+\rho/2)}\nonumber\\
&=& \frac{W_p(x)}{U_p(\sqrt n)}
\frac{e^{-\varepsilon^2/2b}}{(2b)^{\rho/2}\Gamma(1+\rho/2)}\nonumber\\
&=& \frac{W(x)}{U(\sqrt n)}
\frac{e^{-\varepsilon^2/2b}}{(2b)^{\rho/2}\Gamma(1+\rho/2)},
\end{eqnarray}
due to \eqref{W.p.=.c.W} and \eqref{equiv.for.U.1}.
Letting $\varepsilon\downarrow 0$ we conclude the following lower bound
\begin{eqnarray}\label{tau.lower.1}
\liminf_{n\to\infty}U(\sqrt n)\P_x\{\tau_{\widehat x}>n\} &\ge&
\frac{W(x)}{(2b)^{\rho/2}\Gamma(1+\rho/2)},
\end{eqnarray}
which also follows by Fatou's lemma; however \eqref{tau.1.4} is still needed in the sequel.

Fix some $\delta>0$. By the Markov property,
\begin{eqnarray}\label{tau.deco}
\P_x\{\tau_{\widehat x}>n\} &=&
\int_{\widehat x}^\infty\P_x\{X_{(1-\delta)n}\in dy,\tau_{\widehat x}>(1-\delta)n\}
\P_y\{\tau_{\widehat x}>\delta n\}.
\end{eqnarray}
It follows from the upper bound \eqref{rec.time:tau.upper} that
\begin{eqnarray*}
\lefteqn{\int_{\widehat x}^{\varepsilon\sqrt{n}}\P_x\{X_{(1-\delta)n}\in dy,\tau_{\widehat x}>(1-\delta)n\} \P_y\{\tau_{\widehat x}>\delta n\}}\\
&&\hspace{10mm}\le \frac{C}{U_p(\sqrt{\delta n})}\int_{\widehat x}^{\varepsilon\sqrt{n}}
U_p(y)\P_x\{X_{(1-\delta)n}\in dy,\tau_{\widehat x}>(1-\delta)n\}\\
&&\hspace{20mm}= \frac{CU_p(x)}{U_p(\sqrt{\delta n})}
\int_{\widehat x}^{\varepsilon\sqrt{n}} Q^{(1-\delta)n}(x,dy)\\
&&\hspace{30mm}\le \frac{CU_p(x)}{U_p(\sqrt{\delta n})}\P_x\{\widehat X_{(1-\delta)n}\le\varepsilon\sqrt{n}\},
\end{eqnarray*}
since $Q$ is substochastic. The function $U_p$ is regularly
varying at infinity with index $\rho$, hence
$U_p(\sqrt{\delta n})/U_p(\sqrt n)\to\delta^{\rho/2}$ as $n\to\infty$.
Together with the weak convergence of $\widehat X^2_n/n$ to 
a $\Gamma$-distribution, it implies that, for all $\delta>0$,
\begin{eqnarray}\label{tau.9}
\lim_{\varepsilon\to0}\limsup_{n\to\infty}U_p(\sqrt{n})
\int_{\widehat x}^{\varepsilon\sqrt{n}}\P_x\{X_{(1-\delta)n}\in dy,\tau_{\widehat x}>(1-\delta)n\} \P_y\{\tau_{\widehat x}>\delta n\}
=0.\nonumber\\[-2mm]
\end{eqnarray}
Further,
\begin{eqnarray*}
\lefteqn{\int_{\varepsilon\sqrt{n}}^\infty\P_x\{X_{(1-\delta)n}\in dy,\tau_{\widehat x}>(1-\delta)n\} \P_y\{\tau_{\widehat x}>\delta n\}}\\
&&\hspace{10mm}\le \int_{\varepsilon\sqrt{n}}^\infty\P_x\{X_{(1-\delta)n}\in dy,\tau_{\widehat x}>(1-\delta)n\}\\
&&\hspace{20mm}=U_p(x)
\int_{\varepsilon\sqrt{n}}^\infty\frac{1}{U_p(y)}Q^{(1-\delta)n}(x,dy).
\end{eqnarray*}
As proven in \eqref{tau.1.4},
\begin{eqnarray}\label{tau.10}
U_p(x)\int_{\varepsilon\sqrt{n}}^\infty\frac{1}{U_p(y)}Q^{(1-\delta)n}(x,dy)
&\sim& \frac{W(x)}{U(\sqrt{(1-\delta)n})}
\frac{e^{-\varepsilon^2/2b(1-\delta)}}{(2b)^{\rho/2}\Gamma(1+\rho/2)}.
\hspace{10mm}
\end{eqnarray}
Substitution of \eqref{tau.9} and \eqref{tau.10} into \eqref{tau.deco}
leads to
\begin{eqnarray*}
\limsup_{n\to\infty}U(\sqrt n)\P_x\{\tau_{\widehat x}>n\} &\le&
\limsup_{n\to\infty}
\frac{W(x)U(\sqrt n)}{U(\sqrt{(1-\delta)n})}
\frac{1}{(2b)^{\rho/2}\Gamma(1+\rho/2)}.
\end{eqnarray*}
Since $U(\sqrt n)/U(\sqrt{(1-\delta)n})\to (1-\delta)^{-\rho/2}$
and $\delta>0$ may be chosen as small as we please,
we obtain an upper bound
\begin{eqnarray*}
\limsup_{n\to\infty}U(\sqrt n)\P_x\{\tau_{\widehat x}>n\} &\le&
\frac{W(x)}{(2b)^{\rho/2}\Gamma(1+\rho/2)},
\end{eqnarray*}
which together with the lower bound \eqref{tau.lower.1} completes the proof
of the asymptotics \eqref{rec.time:tau.x.asy}.

Conditioning on $X_0$, we conclude \eqref{rec.time:tau.asy}
by the dominated convergence theorem owing to
\eqref{rec.time:tau.x.asy} and \eqref{rec.time:tau.upper}.
\qed\end{proof}
\end{theopargself}

\begin{corollary}\label{cor:rec.tail.gen}
Under the conditions of Theorem \ref{thm:pi.recurrent},
for any initial distribution such that $\E U(X_0)<\infty$,
\begin{eqnarray*}
\P\{\tau_{\widehat x}>n\} &\sim&
\frac{\E\{W(X_1);\ X_1>\widehat x\}}
{(2b)^{\rho/2}\Gamma(1+\rho/2)} \frac{1}{U(\sqrt{n})}
\quad\mbox{ as }n\to\infty.
\end{eqnarray*}
\end{corollary}

\begin{proof}
We have
\begin{eqnarray*}
\lefteqn{\P\{\tau_{\widehat x}>n\}}\\
&=& \int_{\widehat x}^\infty \P_y\{\tau_{\widehat x}>n\}\P\{X_0\in dy\}
+\int_{-\infty}^{\widehat x}\P\{X_0\in dy\}\int_{\widehat x}^\infty P(y,dz)
\P_z\{\tau_{\widehat x}>n-1\}\\
&\sim& \frac{\E\{W(X_0);X_0>\widehat x\}
+\E\{W(X_1);X_0\le\widehat x,X_1>\widehat x\}}
{(2b)^{\rho/2}\Gamma(1+\rho/2)}
\frac{1}{U(\sqrt{n})},
\end{eqnarray*}
by Theorem \ref{thm:rec.time} and the result follows
from the harmonicity of $W$.
\qed\end{proof}

Next let us discuss an implication for a discrete state space
where it is possible to extend the results of the last theorem to
the hitting time for any finite subset $D$ of the state space,
$$
\tau_D\ :=\ \min\{n\ge 1: X_n\in D\}.
$$

\begin{theorem}\label{thm:rec.tail.D}
Assume that $\{X_n\}$ is a countable Markov chain on a state space
$\{z_0<z_1<z_2<\ldots\}$ satisfying the conditions of Theorem
\ref{thm:pi.recurrent}.
Then, for any finite subset $D$ of the state space,
there exists a $c=c(D)<\infty$ such that
\begin{eqnarray}\label{rec.tail.D.upper}
\P_x\{\tau_D>n\} &\le& c\frac{U(x)}{U(\sqrt n)}
\quad\mbox{for all }n\mbox{ and }x.
\end{eqnarray}
In addition, for any fixed initial state $x$,
\begin{eqnarray}\label{rec.tail.D.asy}
\P_x\{\tau_D>n\} &\sim& \frac{C(x,D)}{U(\sqrt n)}\quad\mbox{as }n\to\infty,
\end{eqnarray}
where
\begin{eqnarray*}
C(x,D) &:=& \frac{1}{(2b)^{\rho/2}\Gamma(1+\rho/2)}
\sum_{j=0}^\infty \E_x\{W(X_{j+1});\ X_{j+1}>\widehat x,\tau_D>j\}\ \in\ (0,\infty);
\end{eqnarray*}
here $\widehat x$ is any level guaranteed by Corollary \ref{l:lyapunov}
and such that $D\subseteq B:=[z_0,\widehat x]$.
\end{theorem}

\begin{proof}
Due to the upper bound \eqref{rec.time:tau.upper}
provided by Theorem \ref{thm:rec.time} and by the Markov property,
it is enough to prove \eqref{rec.tail.D.upper} for $x\le\widehat x$.
To start with, consider the case where $B\setminus D$
is a single point set, say $z_0$. Given $X_0=z_0$,
the distribution of the hitting time $\tau_B$ may be decomposed
as the following mixture of distributions, according to
the position of the chain at time $\tau_{\widehat x}$:
\begin{eqnarray*}
\P_{z_0}\{\tau_B>n\} &=& p\P_{z_0}\{\eta>n\}+(1-p)\P_{z_0}\{\theta>n\},
\end{eqnarray*}
where $p=\P_{z_0}\{X_{\tau_B}=z_0\}$, $1-p=\P_{z_0}\{X_{\tau_B}\in D\}$,
the distribution of the random variable $\eta$ is the conditional
distribution of $\tau_B$ given $X_{\tau_B}=z_0$ and
the distribution of $\theta$ is the conditional
distribution of $\tau_B$ given $X_{\tau_B}\in D$.
Since the chain may visit $z_0$ several times before hitting $D$, we get
\begin{eqnarray*}
\P_{z_0}\{\tau_D>n\} &=& (1-p)\sum_{k=0}^\infty p^k
\P\{\eta_1+\ldots+\eta_k+\theta>n\},
\end{eqnarray*}
where $\eta_k$ are independent copies of $\eta$.
By Theorem \ref{thm:rec.time}, the distribution of $\tau_B$ is regularly varying,
so the tail distributions of both of $\eta$ and $\theta$ possess
regularly varying upper bounds of order $c/U(\sqrt n)$ which is known
to be of subexponential type.
Thus, Kesten's bound---see, e.g. \cite[Sec. 3.10]{FKZ}---shows that
the random sum possesses the same regularly varying upper bound
and the proof of \eqref{rec.tail.D.upper} for the case $|B\setminus D|=1$
follows. Since we have only used the upper bound for the tail of $\tau_B$
in our proof of the upper bound for the tail of $\tau_D$,
we may apply the same arguments to the case of an arbitrary number
of states in $B\setminus D$, by induction on this number.

Now let us prove \eqref{rec.tail.D.asy}. For any $N<n/2$,
\begin{eqnarray}\label{sigma.1.2}
\P_x\{\tau_D>n\} &=& \P_x\{\tau_D>n,X_j\le\widehat x
\mbox{ for some }j\in[N,n]\}\nonumber\\
&&\hspace{10mm}+\P_x\{\tau_D>n,X_j>\widehat x\mbox{ for all }j\in[N,n]\}\nonumber\\
&=:& P_1+P_2.
\end{eqnarray}
Let us first show that the first probability becomes negligible
when $N$ increases. Indeed,
\begin{eqnarray*}
P_1 &\le& \P_x\{\tau_D>n,X_j\le\widehat x
\mbox{ for some }j\in[N,n/2]\}\\
&&+\ \P_x\{\tau_D>n,X_j\le\widehat x
\mbox{ for some }j\in(n/2,n-N]\}\\
&&+\ \P_x\{\tau_D>n,X_j>\widehat x\mbox{ for all }j\in[N,n-N],
X_j\le\widehat x\mbox{ for some }j\in(n-N,n]\}\\
&=:& P_{11}+P_{12}+P_{13}.
\end{eqnarray*}
As proven in Theorem \ref{thm:rec.time},
the tail of $\tau_B$ is regularly varying, hence
$$
\E\P_{X_N}\{\tau_B=n+k\}=o(\P\{\tau_B>n\})\quad\mbox{as }n\to\infty
$$
for any fixed $k\in\Z$, so that, for any fixed $N$,
\begin{eqnarray}\label{Cor.rec.tail.13}
P_{13} &\le& \E\P_{X_N}\{X_j>\widehat x\text{ for all }k\in[N,n-N],
X_j\le \widehat x\text{ for some }j\in(n-N,n]\}\nonumber\\
&=& \E\P_{X_N}\{\tau_B\in(n-2N,n-N]\}\nonumber\\
&=& o(\P\{\tau_B>n\})
\quad\mbox{as }n\to\infty.
\end{eqnarray}
By the Markov property,
\begin{eqnarray}\label{Cor.rec.tail.11}
P_{11} &\le& \P_x\{\tau_D>N\}\max_{y\le \widehat x}\P_y\{\tau_D>n/2\}
\ \le\ \frac{c_1}{U(\sqrt N)U(\sqrt{n/2})},
\end{eqnarray}
owing to \eqref{rec.tail.D.upper} because there is only finite
number of states in $[0,\widehat x]$ and
\begin{eqnarray}\label{Cor.rec.tail.12}
P_{12} &\le& \P_x\{\tau_D>n/2\}\max_{y\le \widehat x}\P_y\{\tau_D>N\}
\ \le\ \frac{c_2}{U(\sqrt{n/2})U(\sqrt N)}.
\end{eqnarray}
It follows from the inequalities \eqref{Cor.rec.tail.13}--\eqref{Cor.rec.tail.12}
and regular variation of $U$ that
\begin{eqnarray}\label{Cor.rec.tail.3}
\lim_{N\to\infty}\limsup_{n\to\infty}\ U(\sqrt n) P_1 &=& 0.
\end{eqnarray}
Further, decomposing all the trajectories according to the time of the last visit to 
$[z_0,\widehat x]$, we obtain by the Markov property,
\begin{eqnarray*}
P_2 &=& \sum_{j=0}^{N-1}\sum_{y\in B\setminus D} \P_x\{X_j=y,\tau_D>j\}
\P_y\{\tau_B>n-j\}\\
&\sim& \frac{1}{(2b)^{\rho/2}\Gamma(1+\rho/2)} \frac{1}{U(\sqrt{n})}
\sum_{j=1}^{N-1}\sum_{y\in B\setminus D} \P_x\{X_j=y,\tau_D>j\}
\E_y\{W(X_1);\ X_1>\widehat x\}
\end{eqnarray*}
as $n\to\infty$, by Corollary \ref{cor:rec.tail.gen}
because $U(\sqrt n)$ is regularly varying.
Summing up over $y$ we get, as $n\to\infty$,
\begin{eqnarray*}
P_2 &\sim& \frac{1}
{(2b)^{\rho/2}\Gamma(1+\rho/2)} \frac{1}{U(\sqrt{n})}
\sum_{j=0}^{N-1} \P_x\{W(X_{j+1}),X_{j+1}>\widehat x,\tau_D>j\},
\end{eqnarray*}
which being substituted into \eqref{sigma.1.2}
together with \eqref{Cor.rec.tail.3} gives the required answer
if we let $N\to\infty$.
\qed\end{proof}

\section{Limit theorems for positive and null recurrent chains
conditioned to stay above some level}
\sectionmark{Limit theorems for conditioned Markov chains}
\label{sec:cond.on.recurrent.time}

In this section we prove limit theorems for positive and null recurrent
Markov chains $\{X_n\}$ conditioned on the event 
$$
\{X_1>\widehat x,\ldots,X_n>\widehat x\}.
$$

\begin{theorem}\label{thm:cond.rec.time}
Let the conditions of Theorem \ref{thm:pi.recurrent} hold,
in particular, let $2\mu>-b$. Let $\E U(X_0)<\infty$. Then, for all $u>0$,
\begin{eqnarray*}
\P\Bigl\{\frac{X_n^2}{nb}>u\ \Big|\ \tau_{\widehat x}>n\Bigr\}
&\to& e^{-u/2}\quad\mbox{as }n\to\infty.
\end{eqnarray*}
\end{theorem}

\begin{proof}
For any fixed initial state $x>\widehat x$, by the change of measure,
\begin{eqnarray*}
\P_x\Bigl\{\frac{X_n^2}{nb}>u,\tau_{\widehat x}>n\Bigr\}
&=& U_p(x)\int_{\sqrt{unb}}^\infty\frac{1}{U_p(y)}Q_n(x,dy)\\
&\sim& \frac{W(x)}{U(\sqrt{n})}\frac{e^{-u/2}}{(2b)^{\rho/2}\Gamma(1+\rho/2)}.
\end{eqnarray*}
as shown in \eqref{tau.1.4}. Combining this with tail
asymptotics for $\tau_{\widehat x}$ given in Theorem \ref{thm:rec.time},
we arrive at the required result for $x>\widehat x$.
Then we follow the same arguments as in Corollary \ref{cor:rec.tail.gen}.
\qed\end{proof}

\begin{corollary}\label{Cor:cond.rec.tail}
Assume that $\{X_n\}$ is a countable Markov chain on a state
space $\{z_0<z_1<z_2<\ldots\}$.
Then, for any finite subset $D$ of the state space and for all $u>0$,
\begin{eqnarray*}
\P\Bigl\{\frac{X_n^2}{nb}>u\ \Big|\ \tau_D>n\Bigr\}
&\to& e^{-u/2}\quad\mbox{as }n\to\infty.
\end{eqnarray*}
\end{corollary}

\begin{proof}
Fix an $N\ge 1$.
It follows from \eqref{Cor.rec.tail.3} and asymptotic tail behaviour
of $\tau_D$---see Theorem \ref{thm:rec.tail.D}---that
\begin{eqnarray}\label{Cor.rec.tail.3.1}
\lim_{N\to\infty}\limsup_{n\to\infty}
\P\Bigl\{\frac{X_n^2}{nb}>u,X_j\le \widehat x
\text{ for some }j\in[N,n]\Big|\tau_D>n\Bigr\}
&=& 0.
\end{eqnarray}
Further, by the Markov property,
\begin{eqnarray*}
\lefteqn{\P\Bigl\{\frac{X_n^2}{nb}>u,\tau_D>n,X_j>\widehat x
\text{ for all }j\in[N,n]\Bigr\}}\\
&&\hspace{20mm}=\ \sum_{y>\widehat x} \P\{X_N=y,\tau_D>N\}
\P_y\Bigl\{\frac{X_{n-N}^2}{nb}>u,\tau_B>n-N\Bigr\}.
\end{eqnarray*}
Since $\E U(X_0)<\infty$ and $U$ is regularly varying,
$\E U(X_N)<\infty$ too.
Applying now Theorem \ref{thm:cond.rec.time}, we get
\begin{eqnarray*}
\lefteqn{\P\Bigl\{\frac{X_n^2}{nb}>u,\tau_D>n,X_j>\widehat x
\text{ for all }j\in[N,n]\Bigr\}}\\
&&\hspace{10mm}\sim\ e^{-u/2}
\sum_{y>\widehat x} \P\{X_N=y,\tau_D>N\} \P_y\left\{\tau_B>n-N\right\}\\
&&\hspace{10mm}=\ e^{-u/2}
\P\left\{\tau_D>n,X_j>\widehat x\text{ for all }j\in[N,n]\right\}\\
&&\hspace{10mm}=\ e^{-u/2}P_2,
\end{eqnarray*}
where $P_2$ is defined in \eqref{sigma.1.2},
which in combination with \eqref{Cor.rec.tail.3.1} yields
the required limit behaviour.
\qed\end{proof}

\begin{theorem}\label{thm:maximum}
Let the conditions of Theorem \ref{thm:pi.recurrent} hold. 
Then, for any $x>\widehat x$,
\begin{eqnarray*}
\P_x\Bigl\{\max_{n\le\tau_{\widehat x}}X_n>y\Bigr\}
&\sim& \frac{W(x)}{W(y)}\quad\mbox{as }y\to\infty,
\end{eqnarray*}
where $W$ is the harmonic function for $\{X_n\}$
killed at the time of the first visit to $(-\infty,\widehat x]$,
see Corollary \ref{harm.F}.
\end{theorem}

\begin{proof}
First notice that
$$
\P_x\Bigl\{\max_{n\le\tau_{\widehat x}}X_n>y\Bigr\}
= \P_x\{\tau_{\widehat x}>T(y)\}.
$$
The harmonicity of $W_p$ implies that the sequence $W_p(X_n)\I\{\tau_{\widehat x}>n\}$
is a martingale. Applying the optional stopping theorem to this martingale and
to the stopping time $\tau_{\widehat x}\wedge T(y)$, we obtain
$$
W_p(x)=\E_x\{W_p(X_{T(y)});\ \tau_{\widehat x}>T(y)\}.
$$
Since $W_p(z)\sim U_p(z)$ as $z\to\infty$, we have
\begin{eqnarray}\label{max.1}
\E_x\{U_p(X_{T(y)});\ \tau_{\widehat x}>T(y)\} &\to& W_p(x)
\quad\text{as }y\to\infty.
\end{eqnarray}
Let us split the expectation on the left hand side into two parts:
\begin{eqnarray}\label{deco.for.sup}
\E_x\{U_p(X_{T(y)});\ \tau_{\widehat x}>T(y)\} &=&
\E_x\{U_p(X_{T(y)});\ \tau_{\widehat x}>T(y),X_{T(y)}\le y+s(y)\}\nonumber\\
&&+\ \E_x\{U_p(X_{T(y)});\ \tau_{\widehat x}>T(y),X_{T(y)}>y+s(y)\}.\nonumber\\[-1mm]
\end{eqnarray}
Since $s(y)=o(y)$ and $U_p$ is a regularly varying function,
$U_p(y+s(y))\sim U_p(y)$ as $y\to\infty$, so
\begin{eqnarray}\label{max.3}
\lefteqn{\E_x\{U_p(X_{T(y)});\ \tau_{\widehat x}>T(y),X_{T(y)}\le y+s(y)\}}\nonumber\\
&&\hspace{30mm}\sim\ U_p(y)\P_x\{\tau_{\widehat x}>T(y),X_{T(y)}\le y+s(y)\}.
\end{eqnarray}
By the change of measure with function $U_p$ and the fact that the resulting
kernel $Q$ is substochastic,
\begin{eqnarray*}
\E_x\{U_p(X_{T(y)}),\tau_{\widehat x}>T(y),X_{T(y)}>y+s(y)\}
&\le& U_p(x)\P_x\{\widehat X_{\widehat{T}(y)}>y+s(y)\}.
\end{eqnarray*}
By the formula of total probability,
\begin{eqnarray*}
\P_x\{\widehat X_{\widehat{T}(y)}>y+s(y)\} &=&
\sum_{n=0}^\infty\int_{\widehat x}^y\P_x\{\widehat X_n\in dz,\widehat{T}(y)>n\}\P\{\widehat{\xi}(z)>y+s(y)-z\}\\
&\le& \int_{\widehat x}^y \P\{\widehat{\xi}(z)>s(y)\}\widehat H_x(dz)
\end{eqnarray*}
According to \eqref{X*.le.gamma}, $\P\{\widehat{\xi}(z)>s(z)\}=o(p(z)/z)$. 
Then, similar to the integral estimation in the proof of Lemma \ref{l:XY.equiv}, 
we conclude that
$$
\int_{\widehat x}^\infty \P\{\widehat{\xi}(z)>s(z)\}\widehat H_x(dz)<\infty.
$$
Consequently,
$$
\int_{\widehat x}^y \P\{\widehat{\xi}(z)>s(y)\}\widehat H_x(dz)\to0\quad\text{as }y\to\infty.
$$
As a result,
\begin{eqnarray}\label{deco.for.sup.1}
\E_x\{U_p(X_{T(y)});\ \tau_{\widehat x}>T(y),X_{T(y)}>y+s(y)\} &\to& 0
\quad\text{as }y\to\infty,
\end{eqnarray}
and hence
\begin{eqnarray}\label{deco.for.sup.2}
U_p(y)\P_x\{\tau_{\widehat x}>T(y),X_{T(y)}>y+s(y)\} &\to& 0
\quad\text{as }y\to\infty.
\end{eqnarray}
Applying \eqref{deco.for.sup.2} to \eqref{max.3} we get
\begin{eqnarray*}
\lefteqn{\E_x\{U_p(X_{T(y)});\ \tau_{\widehat x}>T(y),X_{T(y)}\le y+s(y)\}}\nonumber\\
&&\hspace{30mm}=\ (1+o(1))U_p(y)\P_x\{\tau_{\widehat x}>T(y)\}+o(1).
\end{eqnarray*}
Combining this with \eqref{deco.for.sup.1},
we obtain from \eqref{deco.for.sup} the following equality
\begin{eqnarray*}
\E_x\{U_p(X_{T(y)});\ \tau_{\widehat x}>T(y)\} &=&
(1+o(1))U_p(y)\P_x\{\tau_{\widehat x}>T(y)\}+o(1)\quad\mbox{as }y\to\infty.
\end{eqnarray*}
Plugging this into \eqref{max.1} gives
\begin{eqnarray*}
U_p(y)\P_x\{\tau_{\widehat x}>T(y)\} &\to& W_p(x)\quad\mbox{as }y\to\infty,
\end{eqnarray*}
which completes the proof due to $U_p(y)\sim W_p(y)=e^{-c_p}W(y)$.
\qed\end{proof}

Now we turn to functional limit theorems for a recurrent chain $\{X_n\}$
conditioned on $\{\tau_{\widehat x}>n\}$.
Durrett\index{Durrett} \cite{Durrett78} has suggested a method for deriving
functional limit theorems for conditional distributions of null recurrent
Markov chains from the corresponding limit theorems for unconditioned chains.
His approach is applicable in the case $\mu\in(-b/2,b/2)$.
It immediately follows from Theorems \ref{thm:gamma.func} and
\ref{thm:cond.rec.time} that the conditions of \cite[Theorem 3.9]{Durrett78}
are satisfied.
Therefore, the finite dimensional distributions of $\{X_{[nt]}/\sqrt{bn}\}$
conditioned on $\{\tau_{\widehat x}>n\}$ converge to that of
a (time-inhomogeneous) Markov process $X^+(t)$ which may be described
in terms of the limiting Bessel process $Bes(t)$ in Theorem
\ref{thm:gamma.func}---with drift $\mu/bx$ and diffusion
coefficient $1$---and in terms of its first hitting time for the origin,
$T_0=\min\{t:Bes(t)=0\}$. The process $X^+(t)$, $0\le t\le 1$,
is a Markov process on $\R^+$ starting at the origin, with entrance law
\begin{eqnarray*}
\P\{X^+(t)\in dy\}=\frac{y}{t^{\rho/2+1/2}}
e^{-y^2/2t}\mathbb{P}\{T_0>1-t\mid Bes(0)=y\}dy,
\quad t\in(0,1],
\end{eqnarray*}
where $\rho=1+2\mu/b$ and with transition kernel, for $t>s$,
\begin{eqnarray*}
\lefteqn{\P\{X^+(t)\in dy\mid X^+(s)=x\}}\\
&=& \frac{\mathbb{P}\{T_0>1-t\mid Bes(0)=y\}}
{\mathbb{P}\{T_0>1-s\mid Bes(0)=x\}}
\mathbb{P}\{Bes(t-s)\in dy,T_0>t-s\mid Bes(0)=x\}.
\end{eqnarray*}
It is easy to see that $X^+(t)$ converges
in probability to zero as $t\to0$. Then, using again Theorem 3.9
in \cite{Durrett78}, we conclude that the sequence of conditional
distributions is tight in $D[0,1]$.
Therefore, we get weak convergence in the space $D[0,1]$.

We follow a different strategy for positive recurrent Markov chains which
allows us to avoid proving a functional limit theorem for unconditioned
positive recurrent chains $\{X_{[nt]}\}$ with a starting point of order $\sqrt{n}$.
To the best of our knowledge, such a functional limit is only known
for chains on $\Zp$, see \cite[Theorem 5]{BK16}.

Below we suggest an alternative approach which is based on the change
of measure technique and uses functional limit theorems for transient chains.
As after Theorem \ref{thm:gamma.func} in Section \ref{sec:gamma.func},
we define $\{X^{(n)}(t)\}$ as a continuous piece-wise linear process
whose trajectories connect points $(k/n,X_k/\sqrt{bn})$ by segments.
The limiting process $X^+(t)$ may be equivalently defined via
values of $\E g(X^+)$ for all bounded continuous functionals $g$
on the space $C[0,1]$ as follows: as above, starting with the Bessel
process $Bes(t)$, now with drift $(\mu+b)/bx$ and diffusion coefficient $1$,
we define a Markov process $X^+$ starting at the origin and such that
$$
\E g(X^+)\ =\ 2^{\rho/2}\Gamma(1+\rho/2) \E \frac{g(Bes)}{Bes^{\rho}(1)}
\ =\ \frac{\E g(Bes)Bes^{-\rho}(1)}{\E Bes^{-\rho}(1)}.
$$
\index{Transience!convergence to!Bessel process}

\begin{theorem}\label{thm:func.cond}
Let the conditions of Theorem \ref{thm:pi.recurrent} hold,
in particular, let $2\mu>-b$. Then the process
$X^{(n)}$ conditioned on $\{\tau_{\widehat x}>n\}$ converges weakly
to $X^+(t)$ in the space $C[0,1]$ as $n\to\infty$.
\end{theorem}

\begin{proof}
It suffices to prove this weak convergence for the case
where $X_0>\widehat x$.
Let $g$ be a bounded continuous functional on the space $C[0,1]$.
We need to show that, for all $x>\widehat x$,
\begin{equation}\label{fc.1}
\E_x\{g(X^{(n)})\mid\tau_{\widehat x}>n\}\ \to\ \E g(X^+)
\quad\mbox{as }n\to\infty.
\end{equation}
Our strategy is to represent the expectation on the left hand side as
a functional of a transient Markov chain. So we consider the process
$\widehat X^{(n)}(t)$, $t\in[0,1]$, constructed as
a continuous piece-wise linear process
whose trajectories connect points $(k/n,\widehat X_k/\sqrt{bn})$ by segments
where $\{\widehat X_k\}$ is a transient Markov chain constructed in
Section \ref{sec:change} as Doob's $h$-transform with function $U_p$
of the original Markov chain $\{X_k\}$.
Then it follows from \eqref{connection.new.1.path} that,
for any bounded functional $g$ on the space $C[0,1]$,
\begin{eqnarray*}
\E_x\{g(X^{(n)})\mid\tau_{\widehat x}>n\}
&=& \frac{U_p(x)}{\P_x\{\tau_{\widehat x}>n\}}
\E_x\biggl\{\frac{e^{-\sum_{k=0}^{n-1}q(\widehat{X}_k)}}
{U_p(\widehat{X}_n)} g(\widehat{X}^{(n)})\biggr\}.
\end{eqnarray*}
By Theorem \ref{thm:rec.time} and definitions \eqref{def:hz}
and \eqref{def:func.W},
\begin{eqnarray*}
\P_x\{\tau_{\widehat x}>n\} &\sim&
\frac{1}{(2b)^{\rho/2}\Gamma(1+\rho/2)}\frac{W_p(x)}{U_p(\sqrt n)}\\
&\sim& \frac{1}{(2b)^{\rho/2}\Gamma(1+\rho/2)}\frac{U_p(x)}{U_p(\sqrt n)}
\E_x e^{-\sum_{k=0}^\infty q(\widehat{X}_k)}
\quad\mbox{as }n\to\infty.
\end{eqnarray*}
Therefore
\begin{eqnarray}\label{Xn.transf.g}
\E_x\{g(X^{(n)})\mid\tau_{\widehat x}>n\}
&\sim& \frac{(2b)^{\rho/2}\Gamma(1+\rho/2)}
{\E_x e^{-\sum_{k=0}^\infty q(\widehat{X}_k)}}
 \E_x\biggl\{\frac{e^{-\sum_{k=0}^{n-1}q(\widehat{X}_k)}}
{U_p(\widehat{X}_n)/U_p(\sqrt n)} g(\widehat{X}^{(n)})\biggr\}.
\nonumber\\
\end{eqnarray}

Fix a $\delta>0$. Since $g$ is bounded, it follows from
Theorem~\ref{thm:cond.rec.time} that, for all $\delta>0$ and $n$,
\begin{eqnarray}\label{fc.2}
\Bigl|\E\{g(X^{(n)})\I\{X_n\le\delta\sqrt{nb}\}
\mid\tau_{\widehat x}>n\}\Bigr|
&\le& \|g\|_\infty\P\{X_n\le\delta\sqrt{nb}
\mid\tau_{\widehat x}>n\}\nonumber\\
&\le& C\delta.
\end{eqnarray}
Applying \eqref{Xn.transf.g} to $g(X^{(n)})\I\{X^{(n)}(1)>\delta\}$ we get
\begin{eqnarray*}
\lefteqn{\E_x\{g(X^{(n)})\I\{X_n>\delta\sqrt{nb}\}
\mid\tau_{\widehat x}>n\}}\\
&\sim& \frac{(2b)^{\rho/2}\Gamma(1+\rho/2)}
{\E_x e^{-\sum_{k=0}^\infty q(\widehat{X}_k)}}
 \E_x\biggl\{\frac{e^{-\sum_{k=0}^{n-1}q(\widehat{X}_k)}}
{U_p(\widehat{X}_n)/U_p(\sqrt n)} g(\widehat{X}^{(n)})
\I\{\widehat{X}^{(n)}(1)>\delta\}\biggr\}.
\end{eqnarray*}
Due to the regular variation at infinity of the function $U_p$, we have a convergence
$$
\frac{U_p(\widehat{X}_n)}{U_p(\sqrt n)}
\ =\ \frac{U_p(\widehat{X}^{(n)}(1)\sqrt{nb})}{U_p(\sqrt n)}
\ \to\ (\widehat{X}^{(n)}(1))^\rho b^{\rho/2}\quad\mbox{as }n\to\infty
$$
uniformly on the event
$\{\widehat X^{(n)}(1)>\delta\}=\{\widehat X_n/\sqrt{nb}>\delta\}$.
Hence,
\begin{eqnarray}\label{fc.3}
\lefteqn{\E_x\{g(X^{(n)})\I\{X_n>\delta\sqrt{nb}\}
\mid\tau_{\widehat x}>n\}}\nonumber\\
&\sim& \frac{2^{\rho/2}\Gamma(1+\rho/2)}
{\E_x e^{-\sum_{k=0}^\infty q(\widehat{X}_k)}}
\E_x\biggl\{\frac{e^{-\sum_{k=0}^{n-1}q(\widehat{X}_k)}}{(\widehat{X}^{(n)}(1))^{\rho}}
g(\widehat{X}^{(n)})\I\{\widehat{X}^{(n)}(1)>\delta\}\biggr\}.\
\end{eqnarray}
The Bessel approximation proven in Theorem~\ref{thm:gamma.func}
still holds if a certain number of first values of the Markov chain
are fixed, hence we conclude that, for any bounded continuous
functional $g$ on the space $C[0,1]$,
\begin{eqnarray*}
\lefteqn{\E\biggl\{\frac{g(\widehat{X}^{(n)})}
{(\widehat{X}^{(n)}(1))^{\rho}}
\I\{\widehat{X}^{(n)}(1)>\delta\}\ \Big|\ X_0=z_0,\ldots,X_N=z_N\biggr\}}\\
&&\hspace{40mm}\to\ \E\biggl\{\frac{g(Bes)}{(Bes(1))^{\rho}}
\I\{Bes(1)>\delta\}\biggr\},
\end{eqnarray*}
for all $N$ and $z_0$, \ldots, $z_N$.
This makes it possible to apply Lemma \ref{thm:renewal.2.E}, hence
\begin{eqnarray*}
\lefteqn{\E_x\biggl\{\frac{e^{-\sum_{k=0}^{n-1}q(\widehat{X}_k)}}
{(\widehat{X}^{(n)}(1))^\rho}
g(\widehat{X}^{(n)})\I\{\widehat{X}^{(n)}(1)>\delta\}\biggr\}}\nonumber\\
&&\to\ \E e^{-\sum_{k=0}^\infty q(\widehat{X}_k)}
\E \biggl\{\frac{g(Bes)}{(Bes(1))^{\rho}}\I\{Bes(1)>\delta\}\biggr\}
\quad\mbox{as }n\to\infty.
\end{eqnarray*}
From this estimate and \eqref{fc.3} we obtain
\begin{eqnarray*}
\lefteqn{\E_x\{g(X^{(n)})\I\{X_n>\delta\sqrt{nb}\}
\mid\tau_{\widehat x}>n\}}\nonumber\\
&\sim& 2^{\rho/2}\Gamma(1+\rho/2)
\E \biggl\{\frac{g(Bes)}{(Bes(1))^{\rho}}\I\{Bes(1)>\delta\}\biggr\}
\quad\mbox{as }n\to\infty.
\end{eqnarray*}
Combining this with \eqref{fc.2}, an upper bound
$$
\E\left\{(Bes(1))^{-\rho};\ Bes(1)<\delta\right\}\ =\ O(\delta)
$$
and letting $\delta\to 0$, we get
\begin{eqnarray*}
\E\{g(X^{(n)})\mid\tau_{\widehat x}>n\} &\to&
2^{\rho/2}\Gamma(1+\rho/2) \E \frac{g(Bes)}{(Bes(1))^{\rho}}
\quad\mbox{as }n\to\infty,
\end{eqnarray*}
hence the desired convergence.
\qed\end{proof}

\begin{corollary}\label{Cor:func.cond}
Assume that $\{X_n\}$ is a countable Markov chain on a state
space $\{z_0<z_1<z_2<\ldots\}$.
Then, under the conditions of Theorem \ref{thm:pi.recurrent},
for any finite subset $D$ of the state space,
the process $X^{(n)}$ conditioned on $\{\tau_D>n\}$ converges weakly
to $X^+(t)$ in the space $C[0,1]$ as $n\to\infty$.
\end{corollary}

\begin{proof}
Fix $N\ge 1$.
It follows from \eqref{Cor.rec.tail.3} and the asymptotic tail behaviour
of $\tau_D$---see Theorem \ref{thm:rec.tail.D}---that
\begin{eqnarray}\label{Cor.rec.tail.3.9}
\lim_{N\to\infty}\limsup_{n\to\infty}
\E\bigl\{g(X^{(n)});X_j\le \widehat x
\text{ for some }j\in[N,n]\big| \tau_D>n\bigr\}
&=& 0.
\end{eqnarray}
By the Markov property,
\begin{eqnarray*}
\lefteqn{\E\bigl\{g(X^{(n)});\tau_D>n,X_j>\widehat x
\text{ for all }j\in[N,n]\bigr\}}\\
&=& \sum_{y_1,\ldots,y_N\not\in D}
\P\{X_1=y_1,\ldots,X_N=y_N\}
\E\{g(X^{(n)});\tau_B>n\mid X_1=y_1,\ldots,X_N=y_N\}.
\end{eqnarray*}
Applying now Theorem \ref{thm:func.cond} which is still
valid for the conditional expectations on the right hand side, we get
\begin{eqnarray*}
\lefteqn{\E\bigl\{g(X^{(n)});\tau_D>n,X_j>\widehat x
\text{ for all }j\in[N,n]\bigr\}}\\
&\sim& \E g(X^+)
\sum_{y_1,\ldots,y_N\not\in D} \P\{X_1=y_1,\ldots,X_N=y_N\}
\P\{\tau_B>n\mid X_1=y_1,\ldots,X_N=y_N\}\\
&=& \E g(X^+)
\P\left\{\tau_D>n,X_j>\widehat x\text{ for all }j\in[N,n]\right\}\\
&=& \E g(X^+)P_2,
\end{eqnarray*}
where $P_2$ is defined in \eqref{sigma.1.2},
which in combination with \eqref{Cor.rec.tail.3.9} yields
the required limit behaviour.
\qed\end{proof}

\section{Limit theorem in critical case $2\mu=b$}

In the critical case $\mu=b/2$ we have a different type of limit behaviour
which may be described in terms of the function
$$
G(x):=\int_{\widehat x}^x\frac{y}{U(y)}dy,
$$
which is slowly varying at infinity
because $U$ is regularly varying with index $\rho=2\mu/b+1=2$.
\index{Markov chain!limit theorem in critical case}

\begin{theorem}\label{thm:pre-st.log}
Let $\{X_n\}$ be a Markov chain on a countable set $\{z_0<z_1<z_2<\ldots\}$.
Let the conditions of Theorem \ref{thm:pi.recurrent} hold with $\mu=b/2$.
If $G(x)\to\infty$ as $x\to\infty$, then $G(X_n)/G(\sqrt n)$ 
converges weakly as $n\to\infty$ to an uniform distribution on the interval $[0,1]$.
\end{theorem}

\begin{corollary}\label{cor:limit.thm.2mu=b}
In particular case where, for some $m\ge1$ and $\gamma>0$,
\begin{eqnarray*}
r(x) &=& \frac{1}{x}+\frac{1}{x\log x}
+\ldots+\frac{1}{x\log x\cdot\ldots\cdot\log_{(m-1)}x}
+\frac{1-\gamma}{x\log x\cdot\ldots\cdot\log_{(m)}x},
\end{eqnarray*}
we have
\begin{eqnarray*}
R(x) &=& \log x+\log_{(2)}x+\ldots+\log_{(m)}x+(1-\gamma)\log_{(m+1)}x,\\
U(x) &\sim& \frac{x^2}{2}\log x\cdot\log_{(2)}x\cdot\ldots\cdot\log_{(m-1)}x\cdot\log_{(m)}^{1-\gamma}x,\\
\frac{x}{U(x)} &\sim& \frac{2}{\gamma x\log x\cdot\log_{(2)}x\cdot\ldots\log_{(m-1)}x\cdot\log_{(m)}^{1-\gamma}x},\\
G(x) &\sim& \frac{2}{\gamma}\log_{(m)}^\gamma x\quad\mbox{as }x\to\infty.
\end{eqnarray*}
Then the following weak converges holds true
\begin{eqnarray*}
\biggl(\frac{\log_{(m)}X_n}{\log_{(m)}\sqrt n}\biggr)^\gamma
&\Rightarrow& U[0,1]\quad\mbox{as }n\to\infty.
\end{eqnarray*}
\end{corollary}

\begin{theopargself}
\begin{proof}[of Theorem \ref{thm:pre-st.log}]
According to Corollary \ref{Cor.pos.recurrence},
the assumption $G(x)\to\infty$ implies null-recurrence of $\{X_n\}$.
Furthermore, by Theorem \ref{thm:rec.tail.D},
\begin{eqnarray*}
U(\sqrt{n})\P_x\{\tau_{z_0}>n\} &\to& C(x,z_0)\quad\mbox{as }n\to\infty.
\end{eqnarray*}

Let $T_k$ be the time intervals between consequent visits of $\{X_n\}$ to the
state $z_0$. All these random variables are independent. Moreover,
$T_2$, $T_3,\ldots$ are identically distributed and, for every $k\ge 2$,
\begin{eqnarray*}
\P\{T_k>n\} &\sim& \frac{C(z_0,z_0)}{U(\sqrt{n})}\quad\mbox{as }n\to\infty.
\end{eqnarray*}
Let $\theta_n$ denote the corresponding renewal process, that is,
$$
\theta_n:=\max\{k\ge1:T_1+T_2+\ldots+T_k\le n\}.
$$
Let us also introduce the sequence of undershoots:
$$
O_n:=n-(T_1+T_2+\ldots+T_{\theta_n}),\quad n\ge 1.
$$
It is clear from the definition of $\theta_n$ that
$$
\P\{O_n=j\}=\P\{X_{n-j}=z_0\}\P\{T_2>j\}\quad\mbox{for } 0\le j\le n-1
$$
and
$$
\P\{O_n=n\}=\P\{T_1>n\}.
$$
Then, for every $z>z_0$ we have
\begin{eqnarray*}
\P\{X_n>z\} &=& \sum_{j=1}^n\P\{X_{n-j}=z_0\}\P_{z_0}\{X_j>z,\tau_{z_0}>j\}\\
&=& \sum_{j=1}^n\P\{O_n=j\}\P\{X_j>z\mid\tau_{z_0}>j\}.
\end{eqnarray*}
According to Theorem \ref{thm:cond.rec.time},
$$
\P_{z_0}\{X_j>z\mid\tau_{z_0}>j\}=e^{-z^2/2bj}+o(1)\quad\mbox{as } j\to\infty
$$
uniformly for all $z$. In addition, for any fixed $j$,
$$
\P_{z_0}\{X_j>z\mid\tau_{z_0}>j\} \to 0\quad\mbox{as } z\to\infty.
$$
Therefore,
$$
\P_{z_0}\{X_j>z\mid\tau_{z_0}>j\}=e^{-z^2/2bj}+o(1)\quad\mbox{as } z\to\infty
$$
uniformly for all $j$. Hence,
\begin{equation*}
\P\{X_n>z\}=\E\exp\left\{-\frac{z^2}{2bO_n}\right\}+o(1)
\quad\mbox{as } z\to\infty,
\end{equation*}
which implies the following relation, as $n\to\infty$,
\begin{eqnarray}\label{G.1}
\P\Bigl\{\frac{G(X_n)}{G(\sqrt n)}>y\Bigr\} &=&
\P\{X_n>G^{-1}(yG(\sqrt{n}))\}\nonumber\\
&=& \E\exp\Bigl\{-\frac{1}{2b}\Bigl(\frac{G^{-1}(yG(\sqrt{n}))}{\sqrt{O_n}}\Bigr)^2\Bigr\}+o(1).
\end{eqnarray}

Since $\P\{T_2>n\}\sim C(z_0)/U(\sqrt n)$, we get, as $x\to\infty$,
\begin{eqnarray*}
\int_0^x\P\{T_2>y\}dy &\sim& C(z_0)\int_{\widehat x^2}^x\frac{1}{U(\sqrt y)}dy\\
&=& 2C(z_0)\int_{\widehat x}^{\sqrt x}\frac{u}{U(u)}du\\
&=& 2C(z_0)G(\sqrt x).
\end{eqnarray*}
Let us recall the following result by Erickson \cite[Theorem 6]{Erickson70}.

\begin{theorem}\label{thm:Erickson}
Let $S_n=\xi_1+\ldots+\xi_n$ be a random walk with positive jumps
such that the distribution $F$ of $\xi$ has infinite mean and 
$\overline F(t)=L(t)/t$, $t>0$, where $L$ is slowly varying at infinity. 
Let $N(t)=\max\{n:S_n\le t\}$ be the renewal process generated by the random walk, 
$Y(t)=t-S_{N(t)}$ be the undershoot and $Z(t)=S_{N(t)+1}-t$ be the overshoot.
Then, for $0<y\le 1$, $z>0$,
\begin{eqnarray*}
\P\biggl\{\frac{m(Y(t))}{m(t)}\le y,\ \frac{m(Z(t))}{m(t)}\le z\biggr\}
&\to& \min\{y,z\}\quad\mbox{as }t\to\infty,
\end{eqnarray*}
where 
\begin{eqnarray*}
m(t) &=& \int_0^t\overline F(u)du.
\end{eqnarray*}
\end{theorem}
Applying this result, we conclude, for all $y\in[0,1]$,
\begin{equation*}
\P\Bigl\{\frac{G(\sqrt{O_n})}{G(\sqrt n)}\le y\Bigr\}\to y
\quad\mbox{as }n\to\infty,
\end{equation*}
or in other words
\begin{equation}\label{G.2}
\P\{\sqrt{O_n} \le G^{-1}(yG(\sqrt n))\}\to y
\quad\mbox{as }n\to\infty.
\end{equation}
Since $G$ is a slowly varying function,
the inverse function satisfies the relation
\begin{eqnarray*}
G^{-1}(tu) &=& o(G^{-1}(u))\quad\mbox{as }u\to\infty,
\end{eqnarray*}
for any fixed $0<t<1$, so it follows from \eqref{G.2} that
\begin{equation*}
\frac{\sqrt{O_n}}{G^{-1}(yG(\sqrt n)))}\ \to\ 0
\quad\mbox{as }n\to\infty\mbox{ with probability }y
\end{equation*}
and
\begin{equation*}
\frac{\sqrt{O_n}}{G^{-1}(yG(\sqrt n)))}\ \to\ \infty
\quad\mbox{as }n\to\infty\mbox{ with probability }1-y.
\end{equation*}
Therefore,
$$
\E\exp\Bigl\{-\frac{1}{2b}\Bigl(\frac{G^{-1}(yG(\sqrt n))}{\sqrt{O_n}}\Bigr)^2\Bigr\}
\to 1-y\quad\mbox{as }n\to\infty,
$$
which completes the proof, due to \eqref{G.1}.
\qed\end{proof}
\end{theopargself}

\section{Comments to Chapter \ref{ch:power.asymptotics}}

In paper \cite{MP95}, Menshikov\index{Menshikov} and Popov\index{Popov}
investigated the behaviour of the invariant distribution
$\{\pi(x),x\in{\mathbb Z}^+\}$ for countable Markov chains
with asymptotically zero drift and with bounded jumps.
Some rough theorems for the local probabilities $\pi(x)$
were proven; if the condition \eqref{m1.and.m2.new} holds
then for every $\varepsilon>0$ there exist
constants $c_-=c_-(\varepsilon)>0$ and
$c_+=c_+(\varepsilon)<\infty$ such that
$$
c_-x^{-2\mu/b-\varepsilon}\ \le\ \pi(x)
\ \le\ c_+x^{-2\mu/b+\varepsilon}.
$$
The same bounds were obtained by Aspandiiarov\index{Aspandiiarov}
and Iasnogorodski\index{Iasnogorodski} in \cite{AI99};
their results also cover null-recurrent chains with $\mu>0$.

The paper \cite{Kor11} by Korshunov\index{Korshunov} is devoted
to the existence and non-existence of moments of invariant distribution.
In particular, it was proven there that if \eqref{m1.and.m2.new}
holds and the families of random variables
$\{(\xi^+(x))^{2+\gamma},x\ge0\}$ for some $\gamma>0$ and
$\{(\xi^-(x))^2,x\ge0\}$ are uniformly integrable
then the moment of order $\gamma$ of the
invariant distribution $\pi$ is finite if
$\gamma<2\mu/b-1$, and infinite if $\pi$ has unbounded
support and $\gamma>2\mu/b-1$. This result implies that
for every $\varepsilon>0$ there exists a $c(\varepsilon)$ such that
\begin{equation}\label{up.bound}
\pi(x,\infty)\ \le\ c(\varepsilon) x^{-2\mu/b+1+\varepsilon}.
\end{equation}

In \cite{DKW2013} we have found the asymptotic behaviour of $\pi(x,\infty)$
for positive recurrent chains under more restrictive moment conditions.
In particular, it has been assumed there 
that the third moments of jumps converge at infinity.

Concerning Theorem \ref{thm:rec.time},
Huillet\index{Huillet} \cite{Huillet2010} and Dette\index{Dette} \cite{Dette2001}
have obtained exact formulas for recurrence times for very special chains.
They use the orthogonal polynomials technique, which has been suggested
by Karlin\index{Karlin} and McGregor\index{McGregor} in \cite{KarlinMcgegor1959}.

Alexander\index{Alexander} \cite{Alex11} has considered recurrence
times for Markov chain with steps $\pm1$.
Using the standard embedding of such a random walk into
the corresponding Bessel process, he has found exact asymptotics for
$\P_x\{\tau_0=n\}$ for all $\rho>0$.
Unfortunately, this method applies only to a skip-free chain.

From the results in Hryniv\index{Hryniv} et al. 
\cite[Theorem 2.4]{HMW2013} one gets the bounds
$$
n^{-\rho/2}\log^{-\varepsilon}n\le\P_0\{\tau_0>n\}\le n^{-\rho/2}\log^{\rho+1+\varepsilon}
$$
for chains satisfying conditions similar to that
of Theorem~\ref{thm:rec.time} with $r(x)=2\mu/bx+o(1/x\log x)$.

Theorem \ref{thm:maximum} improves Theorem 2.3 by Hryniv\index{Hryniv} 
et al. \cite{HMW2013}
where lower and upper bounds were given with extra logarithmic term.
\chapter{Tail analysis for positive recurrent Markov chains
with drift going to zero slower than $1/x$}
\chaptermark{Drift slower than $1/x$}
\label{ch:Weibull.asymptotics}

In this chapter we consider a Markov chain $\{X_n\}$ which possesses
a stationary (invariant) probability distribution $\pi$ and such that 
the first two truncated moments of jumps satisfy the following condition
\begin{eqnarray}\label{4.m1.and.m2}
m_2^{[s(x)]}(x)\to b>0\quad\mbox{ and }\quad
m_1^{[s(x)]}(x)x\to -\infty\quad\mbox{ as }x\to\infty.
\end{eqnarray}
In this case the tail of $\pi$ typically decays faster than any power
function, it is usually of Weibullian type as seen below.

We have already observed this effect for chains with jumps $\pm1$ and $0$
in Example \ref{ex:nnmc.pi.asy.alpha}.
Let us consider such chains in more detail. Fix positive numbers
$a_+>a_-$, $\alpha\in(0,1)$ and consider a chain $\{X_n\}$ on $\Zp$ 
with transition probabilities up and down
$$
p_+(x)=\frac{1}{2}\left(1-\frac{a_+}{(x+1)^\alpha}\right),\quad
p_-(x)=\frac{1}{2}\left(1+\frac{a_-}{(x+1)^\alpha}\right), \quad x\ge1,
$$
$p_0(0)+p_+(0)=1$, $p_+(0)>0$.
Then, according to \eqref{cont.c},
$$
\pi(x)\ =\
\pi(0)\exp\biggl\{\sum_{k=1}^x\log\frac{p_+(k-1)}{p_-(k)}\biggr\}.
$$
From the definition of $p_\pm$ we get
$$
\log\frac{p_+(k-1)}{p_-(k)}=
\log\left(1-\frac{a_+}{k^\alpha}\right)-\log\left(1+\frac{a_-}{(k+1)^\alpha}\right).
$$
Set $d_\alpha:=\max\{j:j\alpha\le 1\}$.
Then, by Taylor's expansion of the logarithm function,
$$
\log\frac{p_+(k-1)}{p_-(k)}=
-\sum_{j=1}^{d_\alpha}\frac{a_+^j-(-a_-)^j}{j}k^{-j\alpha}+O(k^{-(d_\alpha+1)\alpha})
\quad\mbox{as }k\to\infty.
$$
Therefore,
\begin{eqnarray}\label{ex:10}
\pi(x)\ \sim\
C\exp\biggl\{-\sum_{j=1}^{d_\alpha-1}
\frac{a_+^j-(-a_-)^j}{j(1-j\alpha)}x^{1-j\alpha}
-\frac{a_+^{d_\alpha}-(-a_-)^{d_\alpha}}{d_\alpha}
\sum_{k=1}^x k^{-d_\alpha\alpha}\biggr\},
\end{eqnarray}
owing to Proposition \ref{gen.harm.s}.
If $d_\alpha<1/\alpha$ then we get
$$
\pi(x)\ \sim\
C\exp\biggl\{-\sum_{j=1}^{d_\alpha}
\frac{a_+^j-(-a_-)^j}{j(1-j\alpha)}x^{1-j\alpha}\biggr\},
$$
and if $d_\alpha=1/\alpha$ then
$$
\pi(x)\ \sim\
C x^q\exp\biggl\{-\sum_{j=1}^{1/\alpha-1}
\frac{a_+^j-(-a_-)^j}{j(1-j\alpha)}x^{1-j\alpha}\biggr\},
$$
where $q=-\alpha(a_+^{1/\alpha}-(-a_-)^{1/\alpha})$.
In this example we have
$$
m_1(x)=-\frac{a_++a_-}{2(x+1)^\alpha}\quad\mbox{and}\quad
m_2(x)=1-\frac{a_+-a_-}{2(x+1)^\alpha}.
$$

As follows from \eqref{ex:1}, a stationary density of a diffusion
with the same drift and diffusion coefficients
is asymptotically equivalent to, as $x\to\infty$,
\begin{eqnarray*}
\lefteqn{C\exp\biggl\{-(a_++a_-)
\int_0^x\frac{1}{(y+1)^\alpha-(a_+-a_-)/2}dy\biggr\}}\\
&&\hspace{10mm}\sim\ C\exp\biggl\{-(a_++a_-)\biggl(\sum_{j=1}^{d_\alpha-1}
\frac{(a_+-a_-)^{j-1}}{2^{j-1}(1-j\alpha)}x^{1-j\alpha}\\
&&\hspace{50mm}+\frac{(a_+-a_-)^{d_\alpha-1}}{2^{d_\alpha-1}}
\int_1^x y^{-d_\alpha\alpha}dy\biggr)\biggr\}.
\end{eqnarray*}
Comparing this expression to \eqref{ex:10}, we see that the main term is
the same but all correction terms have different coefficients. Since the
correction terms play a r\^ole in the case $\alpha\le1/2$ ($d_\alpha\ge2$),
we conclude that the densities are asymptotically equivalent for
$\alpha>1/2$ only. We also see that if $\alpha\le1/2$ then it is not
sufficient to know the asymptotic behaviour of the first and second moments only
to conclude the precise asymptotic behaviour of the tail of $\pi$;
we will see later on that higher moments also play a r\^ole if $\alpha\le 1/2$.

In general case where $r(x)\to 0$ while $r(x)x\to\infty$, 
the tail asymptotics of $\pi$ is something like $e^{-g(x)}$ where $g(x)/x\to 0$
(due to Theorem \ref{non.exp}) and $g(x)/\log x\to\infty$ 
(as may be guessed from Corollary \ref{cor:pi.1x}) as $x\to\infty$.

\section{Stationary measure of positive recurrent chains:
Weibullian-type asymptotics}
\label{sec:introduction.toW}

Our first result concerns the case where, roughly speaking,
$m_1(x)=o(1/\sqrt x)$ as $x\to\infty$. More precisely, we assume that
\begin{eqnarray}\label{4.r-cond.2}
\frac{2m^{[s(x)]}_1(x)}{m^{[s(x)]}_2(x)}
&=& -r(x)+o(p(x))\quad\mbox{as }x\to\infty,
\end{eqnarray}
where a decreasing differentiable function $r(x)>0$
satisfies $r(x)x\to\infty$ as $x\to\infty$ and
\begin{eqnarray}\label{4.r.2}
r^2(x) &=& o(p(x))\quad\mbox{as }x\to\infty,
\end{eqnarray}
where $p(x)\in[0,r(x)]$ is a decreasing differentiable function
which is assumed $r(x)$-insensitive, that is, 
$p(x\pm 1/r(x))\sim p(x)$, and integrable at infinity,
\begin{eqnarray}\label{4.p.int}
\int_0^\infty p(x)dx &<& \infty.
\end{eqnarray}
An increasing function $s(x)$ is assumed to be of order $o(1/r(x))$.
In view of \eqref{4.m1.and.m2}, the condition \eqref{4.r-cond.2}
is equivalent to
\begin{eqnarray}\label{4.r-cond.2.equiv}
m^{[s(x)]}_1(x)+\frac{m^{[s(x)]}_2(x)}{2}r(x)
&=& o(p(x))\quad\mbox{as }x\to\infty.
\end{eqnarray}
We also assume that
\begin{eqnarray}\label{4.r.prime}
|p'(x)| < |r'(x)|\ \mbox{ for all }x,\quad && |r'(x)|=o(r^2(x)) \ \mbox{ as }x\to\infty,
\end{eqnarray}
compare the second part of the this condition to
\eqref{cond.on.r} or \eqref{rx.ge.1x.pr}; it is valid for functions $r(x)$
like $x^{-\beta}$, $x^{-\beta}\log^\alpha x$ with
$\beta\in(0,1)$, $\log^\alpha x/x$ with $\alpha>0$, and excludes the function $r(x)=1/x$.
%For $\gamma>1/\beta$, they satisfy the condition \eqref{4.r.gamma}.
%For $\gamma<1/\beta+1$, they satisfy the condition \eqref{4.r.prime}.
%Notice that the stronger condition $r'(x)=O(r(x)/x)$
%implies $r(x+h(x))\sim r(x)$ for every function $h(x)=o(x)$,
%so that in this case the function $r(x)$ is intermediate regularly
%varying---see \cite[Theorem 2.47]{FKZ}. So, the relation \eqref{r.prime}
%allows to consider functions $r(x)$ which are more fluctuating
%than smooth regularly varying functions.

Define
\begin{eqnarray}\label{4.def.of.R}
R(x) &:=& \int_0^x r(y)dy,\quad x>0.
\end{eqnarray}
Since $xr(x)\to\infty$, $R(x)\to\infty$ as $x\to\infty$.
The function $R(x)$ is concave because $r(x)$ is decreasing.
As shown in Section \ref{sec:step.size}, $1/r(x)$ is a natural $x$-step 
responsible for the constant increase of the function $R(x)$.
Under the condition \eqref{4.r.prime} which in stronger than \eqref{cond.on.r}, 
we can derive an asymptotic version of the inequalities
\eqref{r.h.below} and \eqref{r.h.above} as follows: for all $h>0$,
\begin{eqnarray*}
\frac{1}{r(x)}-\frac{1}{r(x+h/r(x))} &=& \int_x^{x+h/r(x)}\frac{r'(y)}{r^2(y)}dy
\ =\ o(1/r(x))\quad\mbox{as }x\to\infty,
\end{eqnarray*}
which implies equivalence
\begin{eqnarray}\label{r.r.equi}
r(x+h/r(x)) &\sim& r(x)\quad\mbox{as }x\to\infty.
\end{eqnarray}
Therefore, for any fixed $h\in\R$,
\begin{eqnarray}\label{4.R.r.c}
R\Bigl(x+\frac{h}{r(x)}\Bigr) &=& R(x)+h+o(1)\quad\mbox{as }x\to\infty.
\end{eqnarray}

Consider the following function
\begin{eqnarray}\label{4.def.u}
U(x) &:=& \int_0^x e^{R(y)}dy,\quad x\ge 0.
\end{eqnarray}
Note that the function $U$ solves the equation $U''-rU'=0$.
The function $U(x)$ is convex. Since
\begin{eqnarray*}
\frac{U'(x)}{\bigl(\frac{1}{r(x)}e^{R(x)}\bigr)'} &=&
\frac{e^{R(x)}}{\bigl(1-\frac{r'(x)}{r^2(x)}\bigr)e^{R(x)}}
\end{eqnarray*}
and $|r'(x)|=o(r^2(x))$ by \eqref{4.r.prime},
L'H\^opital's rule yields that\index{Markov chain!invariant measure!Weibullian asymptotics} \index{Invariant measure!Weibullian asymptotics}
\begin{eqnarray}\label{4.U.r.R.asy}
U(x) &\sim& \frac{1}{r(x)}e^{R(x)} \quad\mbox{as }x\to\infty.
\end{eqnarray} 
The condition \eqref{4.r.2} is aimed at functions $r(x)$
of order $o(1/\sqrt x)$ where the tail asymptotics of the invariant measure
is determined by the functions $r$ and $U$ which are defined 
via the asymptotic behaviour of the first two truncated moments of jumps. 

\begin{theorem}\label{thm:tail.W}
Let $\{X_n\}$ be a positive recurrent Markov chain on $\R$ and
let $\pi(\cdot)$ be its invariant probability measure.
Let $\pi$ have right unbounded support, that is, $\pi(x,\infty)>0$ for all $x$. 

Let the first two moments of jumps truncated at
some increasing level $s(x)=o(1/r(x))$ satisfy the conditions
\eqref{4.m1.and.m2} and \eqref{4.r-cond.2} where the functions $r(x)$ and $p(x)$
satisfy \eqref{4.r.2} and \eqref{4.r.prime}.
Let the following integrability conditions hold
\begin{eqnarray}\label{4.cond.for.U.unif.52}
\sup_{x\in \R}\frac{\E U(x+\xi(x))}{1+U(x)} &<& \infty,
\end{eqnarray}
and, as $x\to\infty$,
\begin{eqnarray}\label{4.cond.xi.le}
\P\{|\xi(x)|>s(x)\} &=& o(r(x)p(x)),\\
\label{4.cond.xi.ge}
\E\bigl\{U(\xi(x));\ \xi(x)>s(x)\bigr\} &=& o(p(x)),\\
\label{4.cond.xi.gamma}
\sup_x \E\bigl\{|\xi(x)|^3;\ |\xi(x)|\le s(x)\bigr\} &<& \infty.
\end{eqnarray}
Then there exists a $c>0$ such that, for any fixed $h>0$,
$$
\pi\Bigl(x,x+\frac{h}{r(x)}\Bigr]\ \sim\ c\frac{1-e^{-h}}{r^2(x)U(x)}
\quad\mbox{ as }x\to\infty.
$$
In particular, 
$$
\pi(x,\infty)\ \sim\ \frac{c}{r^2(x)U(x)}\quad\mbox{ as }x\to\infty.
$$
\end{theorem}

Notice that the condition \eqref{4.r.2} excludes any
function $r(x)$ which decreases like $1/\sqrt x$ or slower.
In case where the absolute value of the first moment decreases 
slower than $1/\sqrt x$,
the conclusion of Theorem \ref{thm:tail.W} fails, in general.
In this case the answer heavily depends on asymptotic
properties of higher moments of the chain jumps.

In order to present the tail asymptotics for the invariant measure
in general case we need the following set of conditions.

Fix some $\gamma\in\{2,3,4,\ldots\}$ and a decreasing
integrable at infinity function $p(x)\in C^{\gamma-1}(\R^+)$.
Assume that there exists a decreasing function $r(x)$ satisfying 
\begin{eqnarray}\label{4.r.gamma}
r^\gamma(x) &=& o(p(x))\quad\mbox{as }x\to\infty,
\end{eqnarray}
and such that
\begin{eqnarray*}
\frac{2m^{[s(x)]}_1(x)}{m^{[s(x)]}_2(x)} &\sim& -r(x)
\quad\mbox{as }x\to\infty.
\end{eqnarray*}
We further assume that the following condition --- which involves 
all truncated moments of order up to $\gamma$ --- holds:
\begin{eqnarray}\label{4.r-cond.function.gen}
m^{[s(x)]}_1(x)+\sum_{j=2}^{\gamma} \frac{m^{[s(x)]}_j(x)}{j!}r^{j-1}(x)
&=& o(p(x))\quad\mbox{as }x\to\infty.
\end{eqnarray}
We also assume that the conditions \eqref{4.p.int} 
and \eqref{4.r.prime} hold, and that, as $x\to\infty$,
\begin{eqnarray}\label{4.r.prime.gen}
r^{(k)}(x) = o(p(x)),\quad &&
p^{(k)}(x) = o(p(x))\quad\mbox{for all }2\le k\le\gamma-1.
\end{eqnarray}
As follows from Lemma \ref{l:g.fin.p.2},
the second relation can be always satisfied by choosing a slower
decreasing integrable function $p(x)$.

Define $R(x)$ as in \eqref{4.def.of.R} and $U(x)$ as in \eqref{4.def.u}.
\index{Markov chain!invariant measure!Weibullian asymptotics}
\index{Invariant measure!Weibullian asymptotics}

\begin{theorem}\label{thm:tail.W.gen}
Let $\{X_n\}$ be a positive recurrent Markov chain on $\R$ and
$\pi(\cdot)$ be its invariant probability measure.
Let $\pi$ have right unbounded support, that is, $\pi(x,\infty)>0$ for all $x$.

Let $\gamma\in\{2,3,\ldots\}$.
Let the first $\gamma$ moments of jumps truncated at
some increasing level $s(x)=o(1/r(x))$ satisfy the conditions
\eqref{4.m1.and.m2} and \eqref{4.r-cond.function.gen} with functions 
$r(x)$ and $p(x)$ satisfying \eqref{4.r.gamma},
\eqref{4.r.prime} and \eqref{4.r.prime.gen}.
Let the following integrability conditions hold
\begin{eqnarray}\label{4.cond.for.U.unif.52.gen}
\sup_{x\in \R}\frac{\E U(x+\xi(x))}{1+U(x)} &<& \infty,
\end{eqnarray}
and, as $x\to\infty$,
\begin{eqnarray}\label{4.cond.xi.le.gen}
\P\{|\xi(x)|>s(x)\} &=& o(r(x)p(x)),\\
\label{4.cond.xi.ge.gen}
\E\bigl\{U(\xi(x));\ \xi(x)>s(x)\bigr\} &=& o(p(x)),\\
\label{4.cond.xi.gamma.gen}
\sup_x \E\bigl\{|\xi(x)|^{\gamma+1};\ |\xi(x)|\le s(x)\bigr\} &<& \infty.
\end{eqnarray}
Then there exists a $c>0$ such that, for any fixed $h>0$,
$$
\pi\Bigl(x,x+\frac{h}{r(x)}\Bigr]\ \sim\ c\frac{1-e^{-h}}{r^2(x)U(x)}
\quad\mbox{ as }x\to\infty.
$$
In particular, 
$$
\pi(x,\infty)\ \sim\ \frac{c}{r^2(x)U(x)}
\quad\mbox{ as }x\to\infty.
$$
\end{theorem}

%Notice that the condition \eqref{4.cond.xi.ge.gen} is fulfilled if, as $x\to\infty$,
%\begin{eqnarray}\label{4.cond.xi.ge.pre}
%\E\bigl\{U(\xi(x));\ \xi(x)>s(x)\bigr\} &=& o(p(x))\\
%\label{4.cond.xi.ge.pre.2}
%\P\bigl\{\xi(x)>s(x)\bigr\} &=& o(p(x)r(x)).
%\end{eqnarray}
%Indeed, since the function $R(x)$ is concave, for $y>0$,
%\begin{eqnarray*}
%U(x+y) &=& U(x)+\int_0^y e^{R(x+z)}dz\\
%&\le& U(x)+\int_0^y e^{R(x)+R(z)}dz
%= U(x)+e^{R(x)} U(y).
%\end{eqnarray*}
%Therefore,
%\begin{eqnarray*}
%\lefteqn{\E\{U(x+\xi(x));\ \xi(x)>s(x)\}}\\
%&&\hspace{20mm}\le\ U(x)\P\{\xi(x)>s(x)\}+e^{R(x)} \E\{U(\xi(x));\ \xi(x)>s(x)\},
%\end{eqnarray*}
%where the right hand side is of order $o(p(x)r(x))U(x)$
%by \eqref{4.cond.xi.ge.pre}, \eqref{4.cond.xi.ge.pre.2}, and \eqref{4.U.r.R.asy};
%and \eqref{4.cond.xi.ge.gen} follows.

Let us demonstrate how the function $r(x)$ may be constructed
under some regularity conditions.
Assume that $m^{[s(x)]}_1(x)$ possesses the following
decomposition with respect to some nonnegative decreasing function
$t(x)\in C^{\gamma}(\R^+)$:
\begin{eqnarray}\label{4.r-cond.2.gen}
m^{[s(x)]}_1 &=&
-t(x)+\sum_{j=2}^{\gamma-1} a_{1,j}t^j(x)+o(p(x)),
\end{eqnarray}
and that, for all $k=2,3,\ldots,\gamma$,
\begin{eqnarray}\label{4.r-cond.k.gen}
m^{[s(x)]}_k(x) &=&
\sum_{j=0}^{\gamma-k} a_{k,j}t^j(x)+o(t^{1-k}(x)p(x)),
\end{eqnarray}
where the function $t(x)$ satisfies the conditions
\eqref{4.r.prime} and \eqref{4.r.prime.gen} for $r(x)$.
Then there exists---see Lemma \ref{l:sufficient.W} below---a solution
to the equation \eqref{4.r-cond.function.gen} which may be represented as
\begin{eqnarray}\label{4.r-cond..gen}
r(x) &=& \sum_{j=1}^{\gamma-1} r_j t^j(x),
\end{eqnarray}
for some reals $r_1$, \ldots, $r_{\gamma-1}$. The function $r(x)$
satisfies the conditions \eqref{4.r.prime} and \eqref{4.r.prime.gen}.
In addition, since its derivative,
$$
r'(x) = t'(x)(1+O(t(x)))=t'(x)(1+o(1)),
$$
is non-positive ultimately in $x$, we may redefine the function $t(x)$
on a compact set so that the function $r(x)$ becomes decreasing.

Theorems~\ref{thm:tail.W} and \ref{thm:tail.W.gen} give, 
at first glance, the same answer:
$$
\pi\left(x,x+\frac{h}{r(x)}\right]\sim c\frac{1-e^{-h}}{r^2(x)U(x)}.
$$
The difference consists in the choice of the function $r(x)$.
In Theorem~\ref{thm:tail.W} this function should satisfy
\eqref{4.r-cond.2.equiv}, while in  Theorem~\ref{thm:tail.W.gen}
we use \eqref{4.r-cond.function.gen} instead of \eqref{4.r-cond.2.equiv}.
In order to explain the difference between \eqref{4.r-cond.function.gen}
and \eqref{4.r-cond.2.equiv} we consider the case where the first moment
behaves regularly at infinity. We first assume that \eqref{4.r-cond.2.equiv}
holds with $r(x)=x^{-\beta}\ell(x)$, $\beta\in(0,1)$. Due to the condition \eqref{4.r.2}
we may apply Theorem~\ref{thm:tail.W} for $\beta>1/2$ only. In this case
$$
R(x)=\int_0^x y^{-\beta}\ell(y)dy\ \sim\ \frac{1}{1-\beta}x^{1-\beta}\ell(x)
\quad\mbox{as }x\to\infty.
$$
Recalling that $U(x)\sim\frac{1}{r(x)}e^{R(x)}$, we then get
\begin{equation}\label{4.compar.1}
\pi(x,\infty)\sim c\frac{x^\beta}{\ell(x)}
\exp\left\{-\int_0^xy^{-\beta}\ell(y)dy\right\}\quad\mbox{as }x\to\infty
\end{equation}
and, in particular,
\begin{equation}\label{4.compar.2}
\log\pi(x,\infty)\sim -\frac{1}{1-\beta}x^{1-\beta}\ell(x)
\quad\mbox{as }x\to\infty.
\end{equation}
If $\beta\le 1/2$ then we have to use \eqref{4.r-cond.function.gen} with
$\gamma=\min\{k\in\Z:\ k\beta>1\}$. This choice of $\gamma$ follows from \eqref{4.r.gamma}.
In order to have a simpler representation for the answer we assume that \eqref{4.r-cond.2.gen}
and \eqref{4.r-cond.k.gen} are valid with $t(x)=x^{-\beta}\ell(x)$. 
As mentioned above, then
$$
r(x)=x^{-\beta}\ell(x)+\sum_{j=2}^\gamma r_jx^{-j\beta}\ell^j(x).
$$
Consequently,
$$
R(x)=\int_0^x y^{-\beta}\ell(y)dy+\sum_{j=2}^\gamma r_j\int_0^x y^{-j\beta}\ell^j(y)dy
$$
and
\begin{equation}\label{4.compar.3}
\pi(x,\infty)\sim c\frac{x^\beta}{\ell(x)}\exp\left\{-\int_0^xy^{-\beta}\ell(y)dy
+\sum_{j=2}^\gamma r_j\int_0^x y^{-j\beta}\ell^j(y)dy\right\}.
\end{equation}
Taking logarithm and comparing with \eqref{4.compar.2}, we see that the 
logarithmic asymptotics are the same for all $\beta\in(0,1)$,
however the exact asymptotics are different.
If, for example, $\beta\in(1/3,1/2]$ and $\ell(x)\equiv1$ 
then we get from \eqref{4.compar.3} that
$$
\pi(x,\infty)\sim cx^\beta\exp\left\{-\frac{1}{1-\beta}x^{1-\beta}
+\frac{r_2}{1-2\beta}x^{1-2\beta}\right\}.
$$
For $\beta>1/2$ we have only the first summand in the exponent. Finally, in the
borderline case $\beta=1/2$ we have
$$
\pi(x,\infty)\sim cx^{\beta+r_2}\exp\left\{-\frac{1}{1-\beta}x^{1-\beta}\right\},
$$
which again differs from the case $\beta>1/2$.

Lastly, let us discuss the case $\beta=1$, so where $r(x)=\ell(x)/x$ and $\ell(x)\to\infty$.
Let us consider a special case where $\ell(x)=c\log x$, $c>0$. Then
\begin{eqnarray*}
R(x) &=& \frac{c}{2}\log^2 x+c_1+o(1);\\
U(x) &\sim& c_2\frac{x}{\log x}e^{c(\log^2 x)/2}\quad\mbox{as }x\to\infty,
\end{eqnarray*}
which, due to Theorem \ref{thm:tail.W},
gives rise to the log-normal type of the tail behaviour of the invariant measure:
\begin{eqnarray*}
\pi(x,\infty) &\sim& c_3\frac{x}{\log x}e^{-c(\log^2 x)/2}\quad\mbox{as }x\to\infty.
\end{eqnarray*}

\section{Lyapunov function and corresponding change of measure}
\sectionmark{An appropriate Lyapunov function}
\label{sec:lyapunov}

%The Markov chain $X$ is assumed to be positive recurrent with invariant
%measure $\pi$. Let $B$ be a Borel set in $\R$ with $\pi(B)>0$;
%in our proofs we consider an interval $(-\infty,x_0]$. Denote, as above,
%$$
%\tau_B:=\min\{n\ge 1:X_n\in B\}.
%$$
In this section we construct a Lyapunov function which will be used 
to derive exact asymptotics in Theorems \ref{thm:tail.W} and \ref{thm:tail.W.gen}.

Consider a function $r_p(x):=r(x)-p(x)$.
We have $0\le r_p(x)\le r(x)$; this function is decreasing because
\begin{eqnarray*}
r_p'(x) &=& r'(x)-p'(x)<0,
\end{eqnarray*}
by the condition \eqref{4.r.prime}.
Define $R_p(x)=U_p(x)=0$ for $x\le 0$ and, for $x>0$,
\begin{eqnarray*}
R_p(x) &:=& \int_0^x r_p(y)dy,\qquad 0\le R_p(x)\le R(x),\\
U_p(x) &:=& \int_0^x e^{R_p(y)}dy,\qquad 0<U_p(x)\le U(x).
\end{eqnarray*}
Since the function $r_p(x)$ is decreasing,
the function $R_p(x)$ is concave. Since
\begin{eqnarray*}
\int_0^\infty r(y)dy = \infty
&\mbox{and}& C_p:=\int_0^\infty p(y)dy < \infty,
\end{eqnarray*}
we have that
\begin{eqnarray}\label{4.equiv.for.R}
R_p(x) &=& R(x)-C_p+o(1)\quad\mbox{as }x\to\infty.
\end{eqnarray}
Therefore,
\begin{eqnarray}\label{4.equiv.for.U}
U_p(x) &\sim& e^{-C_p}U(x) \quad\mbox{as }x\to\infty,
\end{eqnarray}
and, by \eqref{4.U.r.R.asy},
\begin{eqnarray}\label{4.equiv.for.U.1}
U_p(x) &\sim& \frac{1}{r(x)}e^{R(x)-C_p}
\sim \frac{1}{r_p(x)}e^{R_p(x)} \quad\mbox{as }x\to\infty.
\end{eqnarray}

Notice that the increments of the function $U_p$ obey the following 
useful upper bound, for all $x$, $y>0$,
\begin{eqnarray}\label{U.incr.upper.bound}
U_p(x+y)-U_p(x) &=& \int_0^y e^{R_p(x+z)}dz\nonumber\\ 
&\le& \int_0^y e^{R(x+z)}dz\nonumber\\ 
&\le& e^{R(x)}\int_0^y e^{R(z)}dz\ =\ e^{R(x)}U(y),
\end{eqnarray}
provided the function $r(x)$ is decreasing,
because then the function $R$ is concave as an integral of a decreasing function $r$.
%If the function $r(x)$ is increasing and hence negative, then
%\begin{eqnarray}\label{U.incr.upper.bound.E}
%U_p(x+y)-U_p(x) &\le& e^{R(x)}\int_0^y e^{R(z)}dz\ \le\ e^{R(x)}y,
%\end{eqnarray}

\begin{lemma}\label{l:U.gamma.drift.W}
Under the conditions of Theorem {\rm\ref{thm:tail.W.gen}}, as $x\to\infty$,
\begin{eqnarray}\label{4.b.for.u}
\E U_p(x+\xi(x))-U_p(x) &=&
-p(x)r(x)U_p(x) \Bigl(\frac{m^{[s(x)]}_2(x)}{2}+o(1)\Bigr).
\end{eqnarray}
\end{lemma}

\begin{proof}
We start with the following decomposition:
\begin{eqnarray}\label{4.L.harm2.1}
\E U_p(x+\xi(x))-U_p(x)
&=& \E\{U_p(x+\xi(x))-U_p(x);\ \xi(x)<-s(x)\}\nonumber\\
&& +\E\{U_p(x+\xi(x))-U_p(x);\ |\xi(x)|\le s(x)\}\nonumber\\
&&\hspace{2mm}+\E\{U_p(x+\xi(x))-U_p(x);\ \xi(x)>s(x)\}.\hspace{5mm}
\end{eqnarray}
Since the function $U_p(x)$ increases,
the first term on the right hand side may be bounded as follows:
\begin{eqnarray}\label{4.L.harm2.2a}
\left|\E\{U_p(x+\xi(x))-U_p(x);\ \xi(x)<-s(x)\}\right| &\le&
U_p(x)\P\{\xi(x)<-s(x)\}\nonumber\\
&=& o(p(x)r(x))U_p(x),
\end{eqnarray}
due to the condition \eqref{4.cond.xi.le.gen}.
To estimate the second term on the right hand side of \eqref{4.L.harm2.1},
we make use of Taylor's expansion:
\begin{eqnarray}\label{4.L.harm2.2}
\lefteqn{\E\{U_p(x+\xi(x))-U_p(x);\ |\xi(x)|\le s(x)\}}\nonumber\\
&=& \sum_{k=1}^\gamma\frac{U_p^{(k)}(x)}{k!} m^{[s(x)]}_k(x)
+\E\Bigl\{\frac{U_p^{(\gamma+1)}(x+\theta\xi(x))}{(\gamma+1)!}\xi^{\gamma+1}(x);\
|\xi(x)|\le s(x)\Bigr\},\nonumber\\[-2mm]
\end{eqnarray}
where $0\le\theta=\theta(x,\xi(x))\le 1$. By the construction of $U_p$,
\begin{eqnarray}\label{4.U.12.prime}
U_p'(x)=e^{R_p(x)},\qquad
U_p''(x)=r_p(x)e^{R_p(x)}=(r(x)-p(x))e^{R_p(x)},
\end{eqnarray}
and, for $k=3$, \ldots, $\gamma+1$,
\begin{eqnarray*}
U_p^{(k)}(x) = (e^{R_p(x)})^{(k-1)}
&=& \bigl(r_p^{k-1}(x)+o(p(x))\bigr)e^{R_p(x)}
\quad\mbox{as }x\to\infty,
\end{eqnarray*}
where the remainder terms in the parentheses on the right
are of order $o(p(x))$ by the conditions \eqref{4.r.prime}
and \eqref{4.r.prime.gen}. By the definition of $r_p(x)$, for $k\ge 3$,
\begin{eqnarray*}
r_p^{k-1}(x) &=& (r(x)-p(x))^{k-1}= r^{k-1}(x)+o(p(x)),
\end{eqnarray*}
which implies the relation
\begin{eqnarray}\label{4.U.k.prime}
U_p^{(k)}(x) &=& \bigl(r^{k-1}(x)+o(p(x))\bigr)e^{R_p(x)}
\quad\mbox{as }x\to\infty.
\end{eqnarray}
It follows from the equalities \eqref{4.U.12.prime} and \eqref{4.U.k.prime} that
\begin{eqnarray*}
\sum_{k=1}^\gamma \frac{U_p^{(k)}(x)}{k!}m^{[s(x)]}_k(x)
&=& e^{R_p(x)} \biggl(\sum_{k=1}^\gamma
\frac{r^{k-1}(x)}{k!}m^{[s(x)]}_k(x)
+o(p(x))-p(x)\frac{m^{[s(x)]}_2(x)}{2}\biggr)\\
&=& e^{R_p(x)} \biggl(o(p(x))-p(x)\frac{m^{[s(x)]}_2(x)}{2}\biggr),
\end{eqnarray*}
by the conditions \eqref{4.r-cond.function.gen}.
Hence, the equivalence \eqref{4.equiv.for.U.1} yields
\begin{eqnarray}\label{4.L.harm2.2.1}
\sum_{k=1}^\gamma \frac{U_p^{(k)}(x)}{k!}m^{[s(x)]}_k(x)
&=& -r(x)p(x)\frac{m^{[s(x)]}_2(x)}{2} U_p(x)+o(r(x)p(x))U_p(x).\qquad\quad
\end{eqnarray}
Owing to the condition \eqref{4.r.prime.gen} for $\gamma\ge3$
and \eqref{4.r.prime} for $\gamma=2$ on the derivatives of $r(x)$
and the condition \eqref{4.r.gamma},
\begin{eqnarray*}
U_p^{(\gamma+1)}(x) &=& (r^\gamma(x)+o(p(x))) e^{R_p(x)}\\
&=& o(p(x)) e^{R_p(x)}=o(p(x)r(x)) U_p(x).
\end{eqnarray*}
Then, due to \eqref{r.r.equi}, \eqref{4.R.r.c} and \eqref{4.U.r.R.asy},
the last term in \eqref{4.L.harm2.2} possesses the following bound:
\begin{eqnarray*}
\lefteqn{\Bigl|\E\Bigl\{\frac{U_p^{(\gamma+1)}(x+\theta\xi(x))}{(\gamma+1)!}\xi^{\gamma+1}(x);\
|\xi(x)|\le s(x)\Bigr\}\Bigr|}\\
&&\hspace{40mm} \le o(p(x)r(x))U_p(x))
\E\bigl\{|\xi(x)|^{\gamma+1};\ |\xi(x)|\le s(x)\bigr\}\\
&&\hspace{40mm} = o(p(x)r(x))U_p(x),
\end{eqnarray*}
by the condition \eqref{4.cond.xi.gamma.gen}. Therefore, it follows from
\eqref{4.L.harm2.2} and \eqref{4.L.harm2.2.1} that
\begin{eqnarray}\label{4.L.harm2.2b}
\lefteqn{\E\{U_p(x+\xi(x))-U_p(x);\ |\xi(x)|\le s(x)\}}\nonumber\\
&&\hspace{25mm} =-r(x)p(x)\frac{m^{[s(x)]}_2(x)}{2} U_p(x)+o(p(x)r(x))U_p(x).
\end{eqnarray}

Finally, the last term in \eqref{4.L.harm2.1} is of order
$o(p(x)r(x))U_p(x)$ due to the upper bound
\eqref{U.incr.upper.bound}, the equivalence \eqref{4.equiv.for.U.1}
and the condition \eqref{4.cond.xi.ge.gen}.
Substituting this together with \eqref{4.L.harm2.2a} and \eqref{4.L.harm2.2b}
into \eqref{4.L.harm2.1}, we arrive at the lemma conclusion.
\qed\end{proof}

\begin{corollary}\label{l:lyapunov.W}
Let the conditions of Theorem {\rm\ref{thm:tail.W.gen}} hold true.
Then there exists an $\widehat x$ such that the mean drift of
the function $U_p(x)$ is sandwiched as follows
$$
-bp(x)r(x) U_p(x) \ \le\
\E U_p(x+\xi(x))-U_p(x)\ \le\ 0 \quad\mbox{for all }x>\widehat x.
$$
\end{corollary}

\section{Proof of Theorem \ref{thm:tail.W.gen}}

Let us define a new transition kernel via the following change of measure
\begin{equation}\label{4.def.Q}
Q(x,dy):=\frac{U_p(y)}{U_p(x)}\P_x\{X_1\in dy,\tau_B>1\},
\end{equation}
where $B=(-\infty,\widehat x]$, $\widehat x$ is defined in
Corollary \ref{l:lyapunov.W}, and
$$
\tau_B:=\min\{n\ge 1:X_n\in B\}.
$$
It follows from the upper bound in Corollary \ref{l:lyapunov.W} that
$$
Q(x,\R)=\frac{\E\{U_p(x+\xi(x)),\tau_B>1\}}{U_p(x)}
\le \frac{\E U_p(x+\xi(x))}{U_p(x)}\le1
$$
for all $x>\widehat x$. In other words, $Q$ is a substochastic kernel
on $(\widehat x,\infty)$. Furthermore, combining the lower bound
in Corollary \ref{l:lyapunov.W} with the estimate --- 
due to \eqref{4.cond.xi.le.gen} ---
$$
\E\{U_p(x+\xi(x));\tau_B=1\}\le U_p(\widehat x)\P\{x+\xi(x)\le \widehat x\}
=o(p(x)r(x)),
$$
we obtain that
\begin{eqnarray}\label{4.def.q}
q(x):=-\log Q(x,\R) &=& O(p(x)r(x)).
\end{eqnarray}
Let us consider the following normalised kernel
$$
\widehat P(x,dy)=\frac{Q(x,dy)}{Q(x,\R)}
$$
and let $\{\widehat X_n\}$ denote the corresponding Markov chain;
let $\widehat\xi(x)$ be its jump from the state $x$.
Consequently, performing the inverse change of measure we arrive
at the following basic equality, see \eqref{connection.new.B}:
\begin{eqnarray}\label{4.connection}
\P_x\{X_n\in dy,\tau_B>n\} &=& \frac{U_p(x)}{U_p(y)}
\E_x\bigl\{e^{-\sum_{k=0}^{n-1}q(\widehat X_k)};\ \widehat X_n\in dy\bigr\}.
\end{eqnarray}

\begin{lemma}\label{l:change.W}
Under the conditions of Theorem \ref{thm:tail.W.gen}, as $x\to\infty$,
\begin{eqnarray}\label{4.m1.sim.rx}
\E\{\widehat\xi(x);\ |\widehat\xi(x)|\le s(x)\} &\sim& \frac{b}{2}r(x),\\
\label{4.m2.to.b}
\E\{(\widehat\xi(x))^2;\ |\widehat\xi(x)|\le s(x)\} &\to& b,\\
\label{4.X*.le.gamma}
\P\{|\widehat\xi(x)|>s(x)\} &=& o(p(x)r(x)).
\end{eqnarray}
Moreover, there exists a sufficiently large $\widehat x$ such that
\begin{eqnarray}\label{4.m1.ge.rx}
\E\{\widehat{\xi}(x);\ \widehat{\xi}(x)\le s(x)\} &\ge& \frac{b}{4}r(x)
\quad\mbox{for all }x\ge\widehat x.
\end{eqnarray}
\end{lemma}

\begin{proof}
We apply Lemma \ref{l:change.B.0}, so we need to check its conditions.
The conditions \eqref{cond.m1m2r} and \eqref{cond.m2b} 
are met due to the conditions \eqref{4.m1.and.m2} and \eqref{4.r-cond.2}. 
The condition \eqref{le.widehatx.1} is met because of \eqref{4.cond.xi.le.gen}.
Further, it follows from \eqref{4.equiv.for.R} and \eqref{4.equiv.for.U.1} that
\begin{eqnarray*}
\frac{U_p'(x)}{U_p(x)} &=& \frac{e^{R_p(x)}}{U_p(x)}
\ \sim\ \frac{e^{R(x)-C_p}}{\frac{1}{r(x)}e^{R(x)-C_p}}\ =\ r(x).
\end{eqnarray*}
So, the function $U_p$ satisfies the condition
\eqref{cond.UprimeU.frac} with $c_U=1$.
Also $U_p$ satisfies \eqref{cond.UprimeU} for any $s(x)=o(1/r(x))$ because
\begin{eqnarray*}
\frac{U_p'(x+y)}{U_p'(x)} &=& \frac{e^{R_p(x+y)}}{e^{R_p(x)}}
\ \sim\ e^{R(x+y)-R(x)}
\ =\ e^{\int_x^{x+y}r(z)dz}\ =\ e^{O(s(x)r(x))}\ =\ e^{o(1)}
\end{eqnarray*}
as $x\to\infty$ uniformly for all $|y|\le s(x)$,
and, by \eqref{4.equiv.for.U.1},
\begin{eqnarray*}
\frac{U_p(x+y)}{U_p(x)} &\sim& \frac{r(x)}{r(x+y)}\frac{e^{R(x+y)}}{e^{R(x)}}
\ \sim\ e^{R(x+y)-R(x)} \ \to\ 1.
\end{eqnarray*}
Finally, $U_p$ satisfies \eqref{cond.Uclose.harmonic} by 
Corollary \ref{l:lyapunov.W}.
So, all conditions of Lemma \ref{l:change.B.0} are met
and \eqref{4.m1.sim.rx}--\eqref{4.m1.ge.rx} follow.
\qed\end{proof}

Therefore, the chain $\{\widehat X_n\}$ satisfies the condition
\eqref{T.above.cond.1} in Theorem \ref{l:uniform}
with $\widehat v(x)=br(x)/4$, hence we conclude that,
for $\widehat T(t)=\min\{n\ge 1:\widehat X_n>t\}$,
\begin{eqnarray*}
\E_y \widehat T(t)\ =\ \E_y \widehat L(\widehat x, \widehat T(t)) &<& \infty
\quad\mbox{for all }t>y.
\end{eqnarray*}
Thus, for any initial state $\widehat X_0=y$,
\begin{eqnarray*}
\P\Bigl\{\limsup_{n\to\infty}\widehat X_n=\infty\Bigr\} &=& 1.
\end{eqnarray*}
In its turn, then it follows from Theorem \ref{thm:transience.inf}
that $\widehat X_n\to\infty$ with probability 1.

Further, $\widehat v(x)$ introduced above satisfies the condition \eqref{def.cv}
due to \eqref{r.r.equi}. Therefore, Theorem \ref{thm:Hy.above} is 
applicable to the chain $\{\widehat X_n\}$, 
and there exists a $c<\infty$ such that
\begin{eqnarray}\label{4.RF-bound}
\widehat H_y(x,x+1/r(x)] &:=& 
\sum_{n=0}^\infty \P_y\{\widehat X_n\in(x,x+1/r(x)]\}\nonumber\\
&\le& \frac{c}{r^2(x)}
\quad\mbox{for all }x,y>0.
\end{eqnarray}
Having this estimate we now prove the following result.

\begin{lemma}\label{L.limit.W}
Under the conditions of Theorem \ref{thm:tail.W.gen},
$$
h(z):=\lim_{n\to\infty}\E_z e^{-\sum_{k=0}^n q(\widehat X_k)}>0,
\quad z>\widehat x.
$$
Moreover, $h(z)\to 1$ as $z\to\infty$.
\end{lemma}

\begin{proof}
The existence of $h(z)$ is immediate from the monotonicity
of the sequence $e^{-\sum_{k=0}^n q(\widehat X_k)}$ in $n$.
By the convexity of the function $e^{-x}$,
to show positivity it suffices to prove that
\begin{equation}\label{4.L.limit.1}
\E_z \sum_{k=0}^\infty q(\widehat X_k)<\infty,\quad z>\widehat x.
\end{equation}
Note that
\begin{eqnarray*}
\E_z \sum_{k=0}^\infty q(\widehat X_k) &=&
\int_{\widehat x}^\infty q(y)\widehat H_z(dy)
\ \le\ c\int_{\widehat x}^\infty p(y)r(y)\widehat H_z(dy),
\end{eqnarray*}
because $q(y)=O(p(y)r(y))$.
But it has been already shown in the proof of
Lemma \ref{l:XY.equiv} that the last integral is finite under \eqref{4.RF-bound},
thus the first statement of the lemma is proven.

To prove the second claim we notice that it follows from Theorem
\ref{thm:transience.inf} that, for every fixed $N>0$,
$$
\P_z\{\widehat X_n>N \mbox{ for all }n\ge 1\}\to 1\quad\mbox{as }z\to\infty,
$$
hence
$$
\widehat H_z(N) \to 0\quad\mbox{as }z\to\infty.
$$
Then, for any fixed $N$,
$$
\lim_{z\to\infty}\E_z \sum_{k=0}^\infty q(\widehat X_k)
\le \sup_{z>\widehat x} \int_N^\infty q(y)\widehat H_z(dy).
$$
According to \eqref{int.prH.fin},
$$
\lim_{N\to\infty} \sup_{z>\widehat x} \int_N^\infty q(y)\widehat H_z(dy) =0.
$$
Therefore, we infer that
$$
\lim_{z\to\infty}\E_z \sum_{k=0}^\infty q(\widehat X_k)=0.
$$
From this relation and Jensen inequality we finally conclude
$\lim_{z\to\infty}h(z)=1$.
\qed\end{proof}

Consider the following weighted renewal measure on $(\widehat x,\infty)$
\begin{equation}\label{4.def.Hq}
\widehat H^{(q)}_z(dx)\ =\ \sum_{j=0}^\infty
\E_z\{e^{-\sum_{k=0}^{j-1} q(\widehat X_k)};\ \widehat X_j\in dx\},
\end{equation}
and its finite time horizon version,
\begin{equation}\label{4.def.Hq.n}
\widehat H^{(q)}_{z,n}(dx)\ =\ \sum_{j=0}^n
\E_z\{e^{-\sum_{k=0}^{j-1} q(\widehat X_k)};\ \widehat X_j\in dx\}.
\end{equation}

\index{Renewal theorem!for $\Gamma$ limit}

\begin{corollary}\label{cor:renewal.2.W}
Under the conditions of Theorem \ref{thm:tail.W.gen},
for every fixed $z\ge \widehat x$ and $h>0$,
\begin{eqnarray*}
\widehat H^{(q)}_z\Bigl(x,x+\frac{h}{r(x)}\Bigr] &\sim&
h(z)\widehat H_z\Bigl(x,x+\frac{h}{r(x)}\Bigr]
\ \sim\ h(z)\frac{h}{r^2(x)}\quad\mbox{ as }x\to\infty.
\end{eqnarray*}
\end{corollary}

\begin{proof}
It follows from Lemma \ref{thm:renewal.2} which applies
to $\{\widehat X_n\}$ due to Theorem \ref{thm:renewal+} 
and Lemmas \ref{l:change.W} and \ref{L.limit.W}.
\qed\end{proof}

We again use the representation \eqref{pi.repr.B}
applied to the test function $U_p$ which reads
\begin{eqnarray*}
\pi(x,x+h/r(x)] &=& c^*\int_x^{x+h/r(x)}\frac{\widehat H^{(q)}(dy)}{U_p(y)},
\end{eqnarray*}
where $\widehat H^{(q)}$ is defined in \eqref{def.Hq.new.B},
with initial distribution \eqref{initial.hat}.
%Notice that, in estimating the last integral, it is useless to follow
%the integration by parts because of Weibullian nature of the function
%$U_p(y)$; integration by parts only works well for regularly varying functions.
%For this reason 
We proceed with splitting the interval $(x,x+h/r(x))$
into small equal subintervals.
So, let us fix a large $m\in\mathbb Z^+$ and consider points
$$
x_k(m)=x+\frac{k-1}{m}\frac{h}{r(x)},\quad k\in\{1,2,\ldots,m+1\}.
$$
Then
\begin{eqnarray*}
\int_x^{x+h/r(x)}\frac{\widehat H^{(q)}(dy)}{U_p(y)} &=&
\sum_{k=1}^m \int_{x_k(m)}^{x_{k+1}(m)}
\frac{\widehat H^{(q)}(dy)}{U_p(y)}.
\end{eqnarray*}
Since the function $U_p(y)$ is increasing,
we have the following lower and upper bounds
\begin{eqnarray*}
\frac{\widehat H^{(q)}(x_k(m),x_{k+1}(m)]}{U_p(x_{k+1}(m))}
\le \int_{x_k(m)}^{x_{k+1}(m)}\frac{\widehat H^{(q)}(dy)}{U_p(y)}
\le \frac{\widehat H^{(q)}(x_k(m),x_{k+1}(m)]}{U_p(x_k(m))}.
\end{eqnarray*}
For every fixed $m$, it follows from Corollary \ref{cor:renewal.2.W}
that, as $x\to\infty$,
\begin{eqnarray*}
\widehat H^{(q)}\bigl(x_k(m),x_{k+1}(m)\bigr] &\sim&
\widehat H\bigl(x_k(m),x_{k+1}(m)\bigr]\int_B h(z)\P\{\widehat X_0\in dz\}\\
&=& \widehat H\bigl(x_k(m),x_{k+1}(m)\bigr]
\frac{\int_B h(z) U_p(z)\mu(dz)}{\int_B U_p(z)\mu(dz)},
\end{eqnarray*}
where the measure $\mu$ is defined in \eqref{6.mu.B}.
In its turn, Theorem \ref{thm:renewal+} yields the following asymptotics
\begin{eqnarray*}
\widehat H^{(q)}\bigl(x_k(m),x_{k+1}(m)\bigr] &\sim&
c \frac{h}{m r^2(x_k(m))}\quad\mbox{ as }x\to\infty,
\end{eqnarray*}
because
\begin{eqnarray*}
x_{k+1}(m)-x_k(m) &=& \frac{h}{mr(x)}\ \sim\ \frac{h}{mr(x_k(m))},
\end{eqnarray*}
where
\begin{eqnarray*}
c &:=& \frac{\int_B h(z) U_p(z)\mu(dz)}{\int_B U_p(z)\mu(dz)}.
\end{eqnarray*}
This implies the following asymptotic upper bound
\begin{eqnarray*}
\int_x^{x+h/r(x)} \frac{\widehat H^{(q)}(dy)}{U_p(y)} &\le& (c+o(1))
\frac{h}{m} \sum_{k=1}^\infty \frac{1}{r^2(x_k(m))U_p(x_k(m))}.
\end{eqnarray*}
Substituting the asymptotic relation \eqref{4.equiv.for.U.1} for $U_p$,
we arrive at the following upper bound:
\begin{eqnarray*}
\int_x^{x+h/r(x)} \frac{\widehat H^{(q)}(dy)}{U_p(y)} &\le& (c+o(1))
\frac{h}{m} \sum_{k=1}^m \frac{e^{-R(x_k(m))}}{r(x_k(m))}
\quad\mbox{as }x\to\infty.
\end{eqnarray*}
Letting $m\to\infty$ we approximate the sum on the right
multiplied by $h/m$ by the integral
\begin{eqnarray*}
r(x) \int_x^{x+h/r(x)} \frac{e^{-R(y)}}{r(y)} dy
&\sim& \int_0^{h/r(x)} e^{-R(x+y)} dy\\
&=& \frac{1}{r(x)} \int_0^h e^{-R(x+y/r(x))} dy\\
&\sim& \frac{1}{r(x)} e^{-R(x)}\int_0^h e^{-y} dy
= \frac{1-e^{-h}}{r(x)} e^{-R(x)}
\quad\mbox{as }x\to\infty,
\end{eqnarray*}
where we make use of \eqref{4.R.r.c}. In this way the upper bound
of Theorem \ref{thm:tail.W.gen} is done.

The corresponding lower bound may be derived in the same way
and the proof of Theorem \ref{thm:tail.W.gen} is complete.
\qed

\section{Sufficient condition for existence of $r(x)$
satisfying \eqref{4.r-cond.function.gen}}

\begin{lemma}\label{l:sufficient.W}
Let $\gamma\in\{2,3,\ldots\}$.
Assume that $m^{[s(x)]}_1(x)$ possesses the following
decomposition with respect to some nonnegative decreasing function
$t(x)\in C^{\gamma}(\R^+)$ satisfying the conditions
\eqref{4.r.prime} and \eqref{4.r.prime.gen} on $r(x)$:
\begin{eqnarray}\label{4.r-cond.2.gen.1}
m^{[s(x)]}_1(x) &=& -t(x)+\sum_{j=2}^{\gamma-1} a_{1,j}t^j(x)+o(p(x)),
\end{eqnarray}
and that, for every $k=2,3,\ldots,\gamma$,
\begin{eqnarray}\label{4.r-cond.k.gen.1}
m^{[s(x)]}_k(x) &=&
\sum_{j=0}^{\gamma-k} a_{k,j}t^j(x)+o(t^{1-k}(x)p(x)).
\end{eqnarray}
Then there exists a solution
to the equation \eqref{4.r-cond.function.gen} which possesses 
the following decomposition:
\begin{eqnarray}\label{4.r-cond..gen.1}
r(x) &=& \sum_{j=1}^{\gamma-1} r_j t^j(x),
\end{eqnarray}
for some reals $r_1$, \ldots, $r_{\gamma-1}$.
\end{lemma}

\begin{proof}
It is sufficient to find $r(x)$ satisfying the equality
\begin{eqnarray}\label{4.r-cond.function.gen.1}
m^{[s(x)]}_1(x)+\sum_{j=2}^{\gamma}
\frac{1}{j!}m^{[s(x)]}_{j}(x)r^{j-1}(x)&=& o(p(x)).
\end{eqnarray}
In order to find the coefficients $r_j$, let us substitute
\eqref{4.r-cond.2.gen.1}, \eqref{4.r-cond.k.gen.1} and \eqref{4.r-cond..gen.1}
into \eqref{4.r-cond.function.gen.1}. Then we arrive at the following equality:
\begin{eqnarray*}
\biggl(-t(x)+\sum_{j=2}^{\gamma-1} a_{1,j}t^j(x)\biggr)
+\sum_{j=2}^{\gamma} \frac{1}{j!}
\biggl(\sum_{k=0}^{\gamma-j} a_{j,k}t^k(x)\biggr)
\biggl(\sum_{k=1}^{\gamma-1} r_k t^k(x)\biggr)^{j-1} &=& o(p(x)).
\end{eqnarray*}
The coefficient of $t$ equals to $-1+r_1a_{2,0}/2$,
which implies 
$$
r_1=2/a_{2,0}.
$$
The coefficient of $t^2$ equals to
$a_{1,2}+\frac{1}{2}(a_{2,0}r_2+a_{2,1}r_1)+\frac{1}{6}a_{3,0}r_1^2$, which implies
$$
r_2=-\frac{2a_{1,2}+a_{2,1}r_1+a_{3,0}r_1^2/3}{a_{2,0}}.
$$
All further coefficients may be evaluated in recursive way.
\qed\end{proof}

\section{Local asymptotics of stationary probabilities}
\sectionmark{Local asymptotics of stationary probabilities}
\label{sec:Lamperti.critical.local.n}

Similarly to the case of $m_1(x)\sim -\mu/x$,
in this section we derive sharp local asymptotics
for stationary measure $\pi$ of a recurrent irreducible Markov chain
with asymptotically zero drift of order $r(x)$, $xr(x)\to\infty$.
Following Section \ref{sec:ren.local},
we assume that the jumps $\xi(x)$ converge weakly
to some random variable $\xi$ on $\R$, that is,
the condition \eqref{asymp.hom.i} holds.

\index{Markov chain!invariant measure!Weibullian asymptotics}
\index{Invariant measure!Weibullian asymptotics}

\begin{theorem}\label{thm:srt.pi.i.w}
Let $\{X_n\}$ be a positive recurrent Markov chain on $\R$ and
$\pi(\cdot)$ be its invariant probabilistic measure.
Let $\pi$ have right unbounded support, that is, $\pi(x,\infty)>0$ for all $x$.

Let $\gamma\in\{2,3,\ldots\}$.
Let the first $\gamma$ moments of jumps truncated at
some increasing level $s(x)=o(1/r(x))$ satisfy the conditions
\eqref{4.m1.and.m2} and \eqref{4.r-cond.function.gen} with functions 
$r(x)$ and $p(x)$ satisfying \eqref{4.r.gamma},
\eqref{4.r.prime} and \eqref{4.r.prime.gen}.
Let the following integrability conditions hold
\begin{eqnarray}\label{4.cond.for.U.unif.52.gen.w}
\sup_{x\in \R}\frac{\E U(x+\xi(x))}{1+U(x)} &<& \infty,
\end{eqnarray}
and, as $x\to\infty$,
\begin{eqnarray}\label{4.cond.xi.le.gen.w}
\P\{|\xi(x)|>s(x)\} &=& o(r(x)p(x)),\\
\label{4.cond.xi.ge.gen.w}
\E\bigl\{U(\xi(x));\ \xi(x)>s(x)\bigr\} &=& o(p(x)),\\
\label{4.cond.xi.gamma.gen.w}
\sup_x \E\bigl\{|\xi(x)|^{\gamma+1};\ |\xi(x)|\le s(x)\bigr\} &<& \infty.
\end{eqnarray}
Furthermore we assume convergence $\xi(x)\Rightarrow\xi$
and that $\E\xi=0$ and $\E\xi^2=b$. In addition, let
\begin{eqnarray}\label{cond.xi.major.i.w}
-\Xi_-\ \le_{st}\ \xi(x) &\le_{st}& \Xi_+\quad\mbox{for all }x,
\end{eqnarray}
where $\Xi_-^2<\infty$ and $\E U(\Xi_+)\Xi_+^2<\infty$. 
Then, in the lattice case,
\begin{eqnarray*}
\pi(x) &\sim& c e^{-\int_0^x r(y)dy}\quad\mbox{as }x\to\infty,
\end{eqnarray*}
for some $c>0$. In the non-lattice case, for any $h>0$,
\begin{eqnarray}\label{ren.loc.h.pi.i.w}
\pi(x,x+h] &\sim& che^{-\int_0^x r(y)dy}\quad\mbox{as }x\to\infty.
\end{eqnarray}
\end{theorem}

\begin{corollary}\label{cor:pi.1x.pi.n}
Let, in addition, $r(x)=\gamma/x^\beta$ where $\beta\in(1/2,1)$
and $\gamma>0$. Then, in the lattice case,
$$
\pi(x)\sim c e^{-\frac{\gamma}{1-\beta}x^{1-\beta}}\quad\mbox{as }x\to\infty,
$$
which agrees with the global asymptotics given in \eqref{4.compar.1}.
In the non-lattice case,
$$
\pi(x,x+h]= e^{-\frac{\gamma}{1-\beta}x^{1-\beta}}
\quad\mbox{as }x\to\infty.
$$
\end{corollary}

\begin{theopargself}
\begin{proof}[of Theorem \ref{thm:srt.pi.i.w}]
It is very similar to that of Theorem \ref{thm:srt.pi.i}.
Particularly, as it is shown there,
\begin{eqnarray*}
\pi(x,x+h] &\sim& c^*\frac{\widehat H^{(q)}(x,x+h]}{U_p(x)}
\quad\mbox{as }x\to\infty.
\end{eqnarray*}

The Markov chain $\{\widehat X_n\}$ satisfies all the conditions
of Corollary \ref{cor:weibull} with $\widehat\nu(x)=r(x)b/2$ and $\widehat b=b$. 
Indeed, the drift conditions and \eqref{regular_left_tail_weibull} are 
checked in Lemma \ref{l:change.W} and \eqref{eq:irreducibility} right after that.
The weak convergence \eqref{asymp.hom.i} for $\widehat\xi(x)$,
that is $\widehat\xi(x)\Rightarrow\xi$, follows from
that for the original jumps $\xi(x)$ because $U_p(x+y)/U_p(x)\to 1$
as $x\to\infty$, for any fixed $y\in\R$.
Finally, the majorisation condition \eqref{majoriz.i} holds
with a square integrable majorant, since it follows from \eqref{4.def.q}, 
\eqref{4.def.Q}, and \eqref{U.incr.upper.bound} that, 
for all sufficiently large $x$,
\begin{eqnarray*}
\P\{\widehat\xi(x)>y\} &=& \frac{Q(x,(x+y,\infty))}{Q(x,\R)}\\
&\le& 2 \frac{\E\{U_p(x+\xi(x));\ \xi(x)>y\}}{U_p(x)}\\
&\le& 2\P\{\xi(x)>y\} +2e^{R(x)}\frac{\E\{U(\xi(x));\ \xi(x)>y\}}{U_p(x)}\\
&\le& 2\P\{\xi(x)>y\} +c_1 \E\{U(\Xi_+);\ \Xi_+>y\},
\end{eqnarray*}
owing to \eqref{4.equiv.for.U.1} and \eqref{cond.xi.major.i.w}.
Therefore, due to the condition $\E U(\Xi_+)\Xi_+^2<\infty$,
there exists a random variable $\widehat\Xi_+$ such that
$\widehat\xi(x)\le_{st}\widehat\Xi_+$ and $\E\widehat\Xi_+^2<\infty$.
In addition, 
\begin{eqnarray*}
\P\{\widehat\xi(x)<-y\} &=& \frac{Q(x,(-\infty,x-y))}{Q(x,\R)}\\
&\le& 2 \frac{\E\{U_p(x+\xi(x));\ \xi(x)<-y\}}{U_p(x)}\\
&\le& 2\P\{\xi(x)<-y\}\ \le\ 2\P\{\Xi_->y\},
\end{eqnarray*}
which implies that $\widehat\xi(x)\ge_{st} -\widehat\Xi_-$
where $\E\widehat\Xi_-^2<\infty$ due to the condition $\E\Xi_-^2<\infty$,
and the proof of existence of a square integrable majorant 
for the family of $\widehat\xi(x)$ is complete. 

Hence, by Corollary \ref{cor:weibull} and Lemma \ref{thm:renewal.2}  
applied to the Markov chain $\{\widehat X_n\}$, we deduce that
\begin{eqnarray*}
\widehat H^{(q)}(x,x+h] &\sim& 
c_q\frac{h}{\widehat\nu(x)}\ \sim\ c_q\frac{2h}{br(x)}\quad\mbox{as }x\to\infty,
\end{eqnarray*}
which concludes the proof because $U_p(x)\sim c_3U(x)$ 
as $x\to\infty$, see \eqref{4.equiv.for.U}.
\qed\end{proof}
\end{theopargself}

\section{Pre-stationary distributions}
\label{sec:4.pre-stationary}

In this section we assume that the distribution of $X_n$
converges to $\pi$ in the total variation distance, see \eqref{total.var}.
\index{Markov chain!pre-stationary distribution!Weibullian asymptotics}
\index{Pre-stationary distribution!Weibullian asymptotics}

\begin{theorem}\label{thm:4.pre-st.pos.rec.}
Assume that all the conditions of Theorem \ref{thm:tail.W.gen} are valid.
If $r(x)$ is a regularly varying at infinity with index $-\beta\in[-1,0]$ and satisfying
$r'(x)=O(r(x)/x)$, then, for any fixed $h>0$,
\begin{eqnarray*}
\frac{\P\{X_n\in(x,x+h/r(x)]\}}{\pi(x,x+h/r(x)]} &=&
\Phi\biggl(\frac{n-V(x)}{\sqrt{b\frac{1+\beta}{1+3\beta}\frac{x}{r^3(x)}}}\biggr)+o(1)
\end{eqnarray*}
as $x\to\infty$ uniformly for all $n$, where the function $V(x)$ is given by
$$
V(x)=\int_0^x\left(\sum_{k=2}^\gamma\frac{m^{[s(y)]}_k(y)}{(k-2)!k}r^{k-1}(y)\right)^{-1}dy.
$$
\end{theorem}

\begin{proof}
Splitting all the paths according to the time of the last visit of $\{X_n\}$ 
to $B=(-\infty,\widehat x]$, see \eqref{repr.Xn.B.U}, 
we get, for $x>\widehat x$,
\begin{eqnarray}\label{4.pre-st.1}
\lefteqn{\P\{X_n\in(x,x+h/r(x)]\}}\nonumber\\ 
&=& \sum_{j=1}^n\int_B\P\{X_{n-j}\in dz\}\int_{\widehat x}^\infty
P(z,du)U_p(u)\E_u\biggl\{\frac{e^{-\sum_{k=0}^{j-2}q(\widehat X_k)}}
{U_p(\widehat X_{j-1})};\ \widehat X_{j-1}\in(x,x+h/r(x)]\biggr\},
\nonumber\\[-1mm]
\end{eqnarray}
where $q(x)\ge 0$ and $\{\widehat X_n\}$ are defined in 
\eqref{def.q.B} and \eqref{P.hat.B} respectively.

Fix a sequence $N_x\to\infty$ of order $o(1/r^2(x))$.
Then, since $q\ge 0$ and $U_p$ is increasing,
\begin{eqnarray}\label{4.pre-st.2}
\lefteqn{\sum_{j=n-N_x+1}^n\int_B\P\{X_{n-j}\in dz\}
\int_{\widehat x}^\infty P(z,du)U_p(u)
\E_u\biggl\{\frac{e^{-\sum_{k=0}^{j-2}q(\widehat X_k)}}
{U_p(\widehat X_{j-1})};\ \widehat X_{j-1}\in(x,x+h/r(x)]\biggr\}}
\nonumber\\
&&\hspace{20mm}\le\ N_x\frac{1}{U_p(x)}\sup_{z\in B}
\int_{\widehat x}^\infty P(z,du) U_p(u)\phantom{mmmmmmmmmmmmmmmmmmm}\nonumber\\
&&\hspace{20mm}\le\ N_x\frac{c}{U_p(x)}\sup_{z\in B}(1+U_p(z))\nonumber\\
&&\hspace{20mm}=\ o(1/r^2(x)U_p(x)),
\end{eqnarray}
where the second bound follows from the condition \eqref{4.cond.for.U.unif.52.gen}.
Furthermore, the distribution of $X_{n-j}$ converges in total variation to $\pi$
uniformly for all $j\le n-N_x$, see \eqref{total.var}. Therefore,
\begin{eqnarray}\label{4.pre-st.3}
\lefteqn{\nonumber
\sum_{j=1}^{n-N_x}\int_B\P\{X_{n-j}\in dz\}
\int_{\widehat x}^\infty P(z,du)U_p(u)
\E_u\biggl\{\frac{e^{-\sum_{k=0}^{j-2}q(\widehat X_k)}}
{U_p(\widehat X_{j-1})};\ \widehat X_{j-1}\in(x,x+h/r(x)]\biggr\}}\\
&&\hspace{5mm}\sim
\sum_{j=1}^{n-N_x}\int_B\pi(dz)\int_{\widehat x}^\infty P(z,du)U_p(u)
\E_u\biggl\{\frac{e^{-\sum_{k=0}^{j-2}q(\widehat X_k)}}
{U_p(\widehat X_{j-1})};\ \widehat X_{j-1}\in(x,x+h/r(x)]\biggr\}.
\nonumber\\[-1mm]
\end{eqnarray}
Similarly to \eqref{4.pre-st.2},
\begin{eqnarray}\label{4.pre-st.4}
\lefteqn{\sum_{j=n-N_x+1}^n\int_B\pi(dz)\int_{\widehat x}^\infty P(z,du)U_p(u)
\E_u\biggl\{\frac{e^{-\sum_{k=0}^{j-2}q(\widehat X_k)}}
{U_p(\widehat X_{j-1})};\ \widehat X_{j-1}\in(x,x+h/r(x)]\biggr\}}\nonumber\\
&&\hspace{70mm}=\ o(/r^2(x)U_p(x)).\phantom{mmmmmm}
\end{eqnarray}
Combining \eqref{4.pre-st.1}---\eqref{4.pre-st.4}, we obtain
\begin{eqnarray}\label{4.pre-st.5}
\nonumber
\lefteqn{\P\{X_n\in(x,x+h/r(x)]\}}\\
&=& \sum_{j=1}^n \int_B\pi(dz)\int_{\widehat x}^\infty P(z,du)U_p(u)
\E_u\biggl\{\frac{e^{-\sum_{k=0}^{j-2}q(\widehat X_k)}}
{U_p(\widehat X_{j-1})};\ \widehat X_{j-1}\in(x,x+h/r(x)]\biggr\}\nonumber\\
&&\hspace{80mm}+o(1/r^2(x)U_p(x))\nonumber\\
&=& \int_B\pi(dz)\int_{\widehat x}^\infty P(z,du)U_p(u)
\sum_{j=1}^n \int_x^{x+h/r(x)}
\E_u\biggl\{\frac{e^{-\sum_{k=0}^{j-2}q(\widehat X_k)}}{U_p(y)};\ 
\widehat X_{j-1}\in dy\biggr\}\nonumber\\
&&\hspace{80mm}+o(1/r^2(x)U_p(x))\nonumber\\
&=& \int_{\widehat x}^\infty \mu(du)U_p(u)
\int_x^{x+h/r(x)}\frac{\widehat H^{(q)}_{u,n}(dy)}{U_p(y)}
+o(1/r^2(x)U_p(x))\quad\mbox{as }x\to\infty,
\end{eqnarray}
where
\begin{eqnarray*}
\mu(du) &=& \int_B\pi(dz)P(z,du)
\end{eqnarray*}
is a measure on $(\widehat x,\infty)$, see \eqref{6.mu.B}, and 
\begin{eqnarray*}
\widehat H^{(q)}_{u,n}(A) &:=&
\sum_{j=1}^n 
\E_u\Bigl\{e^{-\sum_{k=0}^{j-2}q(\widehat X_k)};\ \widehat X_{j-1}\in A\Bigr\}
\end{eqnarray*}
is a measure on $(\widehat x,\infty)$ too.

\begin{lemma}\label{l:renewal.2.W}
Under the conditions of Theorem \ref{thm:4.pre-st.pos.rec.},
\begin{eqnarray*}
\widehat H_{z,n}^{(q)}(x,x+h/r(x)] 
&=& h(z)\widehat H_{z,n}(x,x+h/r(x)]+o(1/r^2(x))\\
&=& h(z)\frac{h}{r^2(x)}\Phi\biggl(\frac{n-V(x)}{\sqrt{b\frac{1+\beta}{1+3\beta}\frac{x}{r^3(x)}}}\biggr)
+o\Bigl(\frac{1}{r^2(x)}\Bigr)
\end{eqnarray*}
as $x\to\infty$ uniformly for all $n$, where $\Phi$ 
is the standard normal distribution function. 
\end{lemma}

\begin{proof}
We want to apply Lemma \ref{thm:renewal.2} and Theorem \ref{thm:renewal.clt}
to $\{\widehat X_n\}$ keeping in mind Lemma \ref{L.limit.W}.

In order to apply Theorem~\ref{thm:renewal.clt} we need to identify
a regularly varying decreasing function $v(x)$
such that $v'(x)=O(v(x)/x)$ and
\begin{equation}\label{4.pre-st.6}
\widehat m_1^{[s(x)]}(x)\ :=\
\E\{\widehat{\xi}(x);\ |\widehat{\xi}(x)|\le s(x)\}
\ =\ v(x)+o(\sqrt{v(x)/x}).
\end{equation}
By the definition of $\widehat\xi(x)$,
\begin{equation}\label{4.pre-st.7}
\widehat{m}_1^{[s(x)]}(x)\ =\
\frac{\E\left\{U_p(x+\xi(x))\xi(x);|\xi(x)|\le s(x)\right\}}{Q(x,\R^+)U_p(x)}.
\end{equation}
By Taylor's expansion,
\begin{eqnarray*}
\lefteqn{\E\{U_p(x+\xi(x))\xi(x);\ |\xi(x)|\le s(x)\}}\\
&&\hspace{1cm}=
\sum_{k=1}^\gamma\frac{U_p^{(k-1)}(x)}{(k-1)!}m_k^{[s(x)]}(x)
+\E\Bigl\{\frac{U_p^{(\gamma)}(x+\theta\xi(x))}{\gamma!}\xi^{\gamma+1}(x);\
|\xi(x)|\le s(x)\Bigr\}.
\end{eqnarray*}
It is clear that the assumption \eqref{4.cond.xi.gamma.gen} implies
boundedness of functions $m^{[s(x)]}_k(x)$ for all $k\le\gamma+1$.
From this fact and from \eqref{4.U.12.prime} and \eqref{4.U.k.prime}
we infer that
\begin{eqnarray*}
\lefteqn{\E\left\{U_p(x+\xi(x))\xi(x);|\xi(x)|\le s(x)\right\}}\\
&&\hspace{10mm}=\
U_p(x)m^{[s(x)]}_1(x)+U_p'(x)\sum_{k=2}^\gamma
\frac{m^{[s(x)]}_k(x)}{(k-1)!}r^{k-2}(x)\\
&&\hspace{50mm}+O(p(x)r(x))U_p(x)+O(r^\gamma(x))U_p(x).
\end{eqnarray*}
By \eqref{4.r.gamma},
\begin{eqnarray*}
\lefteqn{\E\left\{U_p(x+\xi(x))\xi(x);|\xi(x)|\le s(x)\right\}}\\
&&\hspace{15mm}=\
U_p(x)m^{[s(x)]}_1(x)+U_p'(x)\sum_{k=2}^\gamma\frac{m^{[s(x)]}_k(x)}{(k-1)!}r^{k-2}(x)+O(p(x))U_p(x).
\end{eqnarray*}
Substituting this relation into \eqref{4.pre-st.7} and using $Q(x,\Rp)=1+O(r(x)p(x))$, which is
immediate from \eqref{4.def.q}, we conclude that
\begin{equation}
\label{4.pre-st.8}
\widehat{m}_1^{[s(x)]}(x)=m^{[s(x)]}_1(x)+
\frac{U_p'(x)}{r_p(x)U_p(x)}\sum_{k=2}^\gamma\frac{m^{[s(x)]}_k(x)}{(k-1)!}r^{k-1}(x)
+O(p(x)).
\end{equation}
Recalling that $U'_p(x)=e^{R_p(x)}$ and using $r_p'(x)=O(r_p(x)/x)$ we get
$$
(U'_p(x)-r_p(x)U_p(x))'\ =\ -r'_p(x)U_p(x)
\ =\ O\Bigl(\frac{U_p(x)r_p(x)}{x}\Bigr).
$$
Since $r_p(x)U_p(x)\sim e^{R_p(x)}$, we have
$$
|U'_p(x)-r_p(x)U_p(x)|\ \le\ c_1\int_1^x \frac{e^{R_p(y)}}{y}dy
\quad\mbox{for some }c_1<\infty.
$$
The derivative of $U_p(x)/x$ is asymptotically equivalent to $e^{R_p(x)}/x$
because $r(x)x\to\infty$. Therefore, by L'Hopital's rule,
$$
|U'_p(x)-r_p(x)U_p(x)|\ =\ O(U_p(x)/x),
$$
or, in other words,
$$
\frac{U_p'(x)}{r_p(x)U_p(x)}=1+O(1/xr(x)).
$$
Plugging this into \eqref{4.pre-st.8}, we obtain
\begin{equation*}
\widehat m_1^{[s(x)]}(x)\ =\ m^{[s(x)]}_1(x)
+\sum_{k=2}^\gamma\frac{m^{[s(x)]}_k(x)}{(k-1)!}r^{k-1}(x)+O(1/x).
\end{equation*}
According to \eqref{4.r-cond.function.gen},
$$
m^{[s(x)]}_1(x)=-\sum_{k=2}^\gamma\frac{m^{[s(x)]}_k(x)}{k!}r^{k-1}(x)+o(p(x)).
$$
As a result we have the following asymptotic expansion
for the expectation of the truncation at levels $\pm s(x)$
of jumps for the chain $\{\widehat X_n\}$
\begin{equation*}
\widehat m_1^{[s(x)]}(x)\ =\
\sum_{k=2}^\gamma\frac{m^{[s(x)]}_k(x)}{(k-2)!k}r^{k-1}(x)+O(1/x).
\end{equation*}
Now it is clear that \eqref{4.pre-st.6} is valid with
$$
v(x)=\sum_{k=2}^\gamma\frac{m^{[s(x)]}_k(x)}{(k-2)!k}r^{k-1}(x),
$$
because, for some $c_2>0$,
$$
x\sqrt{v(x)/x}\ =\ \sqrt{v(x)x}\ \ge\ c_2\sqrt{r(x)x}\ \to\ \infty
\quad\mbox{as }x\to\infty,
$$
and so
$$
1/x\ =\ o(\sqrt{v(x)/x})\quad\mbox{as }x\to\infty.
$$
The function $v(x)$ is regularly varying at infinity since
$$
\frac{v(x)}{r(x)}\ \sim\ \frac{m^{[s(x)]}_2(x)}{2}\ \to\ \frac{b}{2},
$$
and the proof follows.
\qed\end{proof}

Since $U_p$ is increasing, we deduce the following lower and upper bounds
\begin{equation}\label{int.hatH.U}
\frac{\widehat H^{(q)}_{u,n}(x,x+h/r(x)]}{U_p(x+h/r(x))}\ \le\
\int_x^{x+h/r(x)}\frac{\widehat H^{(q)}_{u,n}(dz)}{U_p(z)}\ \le\
\frac{H^{(q)}_{u,n}(x,x+h/r(x)]}{U_p(x)}.
\end{equation}

For any fixed $u>\widehat x$, due to Lemma \ref{l:renewal.2.W}, 
\begin{eqnarray}\label{4.hat.H.u.n}
\widehat H^{(q)}_{u,n}\Bigl(x,x+\frac{h}{r(x)}\Bigr] &=&
h(u)\frac{h}{r^2(x)}\Phi\biggl(\frac{n-V(x)}
{\sqrt{b\frac{1+\beta}{1+3\beta}\frac{x}{r^3(x)}}}\biggr)
+o\Bigl(\frac{1}{r^2(x)}\Bigr)
\end{eqnarray}
as $y\to\infty$ uniformly for all $n$.
In addition, due to $q\ge 0$, 
\begin{eqnarray}\label{4.hat.H.u}
\sup_{u>\widehat x}\widehat H_{u,n}^{(q)}\Bigl(x,x+\frac{h}{r(x)}\Bigr] 
&\le& \sup_{u>\widehat x}\sum_{j=1}^n 
\P_u\Bigl\{\widehat X_{j-1}\in \Bigl(x,x+\frac{h}{r(x)}\Bigr]\Bigr\}
\ \le\ c_1\frac{1}{r^2(x)}
\end{eqnarray}
for all $x$ and $n$, for some $c_1<\infty$ as follows from 
\eqref{4.RF-bound}.

From the estimate \eqref{4.hat.H.u.n} and the equivalence 
$U_p(x+h/r(x))\sim e^{h}U_p(x)$---as follows from
\eqref{4.R.r.c}---we infer from \eqref{int.hatH.U} that,
for any fixed $u>\widehat x$,
$$
U_p(x)r^2(x)\int_x^{x+h/r(x)}\frac{\widehat H^{(q)}_{u,n}(dz)}{U_p(z)}\ \le\
h(u)h\Phi\biggl(\frac{n-V(x)}
{\sqrt{b\frac{1+\beta}{1+3\beta}\frac{x}{r^3(x)}}}\biggr)+o(1)
$$
and
$$
U_p(x)r^2(x)\int_x^{x+h/r(x)}\frac{\widehat H^{(q)}_{u,n}(dz)}{U_p(z)}\ \ge\
h(u)he^{-h}\Phi\biggl(\frac{n-V(x)}
{\sqrt{b\frac{1+\beta}{1+3\beta}\frac{x}{r^3(x)}}}\biggr)+o(1).
$$
Splitting the interval $(x,x+h/r(x)]$ into smaller intervals 
as it has been done in Theorem~\ref{thm:tail.W.gen}, 
we can justify the following asymptotics
\begin{eqnarray}\label{4.int.hat.U}
U_p(x)r^2(x)\int_x^{x+h/r(x)}\frac{\widehat H^{(q)}_{u,n}(dz)}{U_p(z)}\ \sim\
h(u)(1-e^{-h})\Phi\biggl(\frac{n-V(x)}
{\sqrt{b\frac{1+\beta}{1+3\beta}\frac{x}{r^3(x)}}}\biggr)+o(1)\nonumber\\[-2mm]
\end{eqnarray}
as $x\to\infty$ uniformly for all $n$.

Similarly, it follows from \eqref{4.hat.H.u} that
\begin{eqnarray*}
\sup_{u>\widehat x}\int_x^{x+h/r(x)} \frac{\widehat H^{(q)}_{u,n}(dz)}{U_p(z)}
&\le& \frac{c_2}{r^2(x)U_p(x)}.
\end{eqnarray*}
In addition,
\begin{eqnarray*}
\widehat c &=& \int_{\widehat x}^\infty h(u)U_p(u) \mu(du)\\
&=& \int_B \pi(dz)\int_{\widehat x}^\infty h(u)U_p(u) P(z,du)\ <\ \infty,
\end{eqnarray*}
as follows from the condition \eqref{4.cond.for.U.unif.52}.
Hence the dominated convergence theorem is applicable to 
\eqref{4.pre-st.5}, so plugging \eqref{4.int.hat.U} into \eqref{4.pre-st.5}, 
we obtain
\begin{eqnarray*}
\P\{X_n\in(x,x+h/r(x)]\} &=& \widehat c
\frac{1-e^{-h}}{r^2(x)U_p(x)}
\Phi\biggl(\frac{n-V(x)}{\sqrt{b\frac{1+\beta}{1+3\beta}\frac{x}{r^3(x)}}}\biggr)+
o\Bigl(\frac{1}{r^2(x)U_p(x)}\Bigr)
\end{eqnarray*}
as $x\to\infty$ uniformly for all $n$ and the proof is complete.
\qed\end{proof}

\section{Comments to Chapter \ref{ch:Weibull.asymptotics}}

Markov chains with drift satisfying $xm_1(x)\to\infty$ were considered
by Menshikov\index{Menshikov} and Popov\index{Popov} in \cite{MP95}
along with drift of order $1/x$.
They have derived rough asymptotics for $\{\pi(x),x\in{\mathbb Z}^+\}$
for countable Markov chains with asymptotically zero drift 
and with bounded jumps. Some rough theorems for the local 
probabilities $\pi(x)$ were proven; if
$$
2m_1(x)/m_2(x)\ \sim\ -\theta/x^\beta\quad\mbox{as }x\to\infty,
$$
then for every $\varepsilon>0$ there exist
constants $c_-=c_-(\varepsilon)>0$ and
$c_+=c_+(\varepsilon)<\infty$ such that
$$
c_-e^{-(\theta/(1-\beta)+\varepsilon)x^{1-\beta}}\ \le\ \pi(x)
\ \le\ c_+e^{-(\theta/(1-\beta)-\varepsilon)x^{1-\beta}}.
$$
The same bounds were obtained by Aspandiiarov\index{Aspandiiarov}
and Iasnogorodski\index{Iasnogorodski} in \cite{AI99}.

The paper \cite{Kor11} by Korshunov\index{Korshunov} is devoted to the
existence and non-existence of moments of invariant distribution.
In particular, it was proven there that if $m_1(x)\sim -\mu/x^\beta$
and $b(x)\to b$ hold and the families of random variables
$$
\{e^{\frac{\gamma}{2}\log^2(1+\xi^+(x))}(\xi^+(x))^2,\ x\ge0\}
\quad\mbox{for some }\gamma>0
$$
and $\{(\xi^-(x))^2,x\ge0\}$ are uniformly integrable,
then, for $X_0$ having invariant distribution $\pi$,
\begin{itemize}
\item $\E e^{\gamma X_0^{1-\beta}}<\infty$ for $\gamma<2\mu/(1-\beta)b$;
\item $\E e^{\gamma X_0^{1-\beta}}=\infty$ if $\pi$ has unbounded
support and $\gamma>2\mu/(1-\beta)b$.
\end{itemize}
This result implies that
for every $\varepsilon>0$ there exists a $c(\varepsilon)$ such that
\begin{equation}\label{up.bound.wei}
\pi(x,\infty)\ \le\
c(\varepsilon) e^{-(2\mu/(1-\beta)b-\varepsilon)x^{1-\beta}}.
\end{equation}
In that paper there is also some analysis for $\gamma=2\mu/(1-\beta)b$.
\chapter{Markov chains with asymptotically non-zero drift in Cram\'er's case}
\chaptermark{Chains with asymptotically non-zero drift}
\label{ch:asymp.hom}

In this chapter we consider Markov chains with asymptotically
constant (non-zero) drift. As we see in the previous chapter,
the slower $m_1(x)$ tends to zero the higher moments should 
behave regularly at infinity in order to make it possible to describe
the asymptotic tail betaviour of the invariant measure. 
Therefore, it is not surprising that in the case of
asymptoticaly negative drift bounded away from zero we will
assume that the distribution of jumps $\xi(x)$ converges weakly
as $x$ tends to infinity. This corresponds, roughly speaking,
to the assumption that {\it all} moments are regularly behaving at infinity.
In this chapter we slightly extend the notion of an asymptotically
homogeneous Markov chain by allowing extended limiting random variable.
\index{Markov chain!asymptotically homogeneous in space}

\begin{definition}
We say that $\{X_n\}$ is {\it asymptotically homogeneous in space}
\index{Markov chain!asymptotically homogeneous in space} if
\begin{equation}\label{asymp.hom}
\xi(x) \Rightarrow \xi\quad\mbox{as }x\to\infty,
\end{equation}
where $\xi$ is an extended random variable taking values in $\R\cup\{-\infty\}$.
\end{definition}

The class of asymptotically homogeneous chains is larger than the class
of additive Markov chains, which has been introduced by Aldous\ \cite{Aldous89},\index{Aldous}
where $\xi(x)$ is assumed convergent in the total variation norm.

The simplest and one of the most important examples of asymptotically
homogeneous Markov chains is
{\it a random walk with delay at zero}\index{Random walk!delayed at zero}
({\it Lindley recursion}):\index{Lindley recursion}
\begin{eqnarray}\label{def:W}
W_{n+1} &=& (W_n+\xi_{n+1})^+,\quad n\ge 0,
\end{eqnarray}
where $\{\xi_n\}$ are independent copies of $\xi$. In this example we
observe convergence in total variation. The process $\{W_n\}$ describes the
waiting time process in a single-server queue which is
a basic model in queueing theory.\index{Queueing system, $GI/GI/1$}

Another popular class of models closely related to asymptotically
homogeneous chains is originated from
stochastic recursions\index{Stochastic linear recursion}
$$
R_{n}=A_{n}R_{n-1}+B_{n},\ n\ge0,
$$
where $\{(A_n,B_n)\}$ are independent identically distributed random 
vectors in $\Rp\times\R$.
The sequence $R_n$ does not satisfy \eqref{asymp.hom},
but some function of it is an asymptotically homogeneous Markov chain,
for details see Goldie\index{Goldie} \cite[Section 2]{Goldie1991}
or Section \ref{sec:perpetuity} below.

\section{Local renewal theorem}
\label{sec: ah.renewal}

In this section we assume that \eqref{asymp.hom} holds and
that the mean of the limiting variable $\xi$ is positive.
Our aim is to study the asymptotic behaviour of the renewal measure.

In contrast to the case of asymptotically zero drift,
one can derive a renewal theorem for an asymptotically homogeneous
chain $\{X_n\}$ without use of limit theorems for $X_n$.
Instead, we apply some ideas of the operator approach proposed
by Feller\index{Feller} \cite{Feller}.
\index{Local renewal theorem for!Markov chain}
\index{Markov chain!asymptotically homogeneous in space!local renewal theorem}

\begin{theorem}\label{thm:ah.renewal}
Let $\xi(x)\Rightarrow\xi$ as $x\to\infty$ and $\E\xi>0$.
Let the family of random variables $\{|\xi(x)|,\ x\in\R\}$
admit an integrable majorant $\Xi$, that is, $\E\Xi<\infty$ and
\begin{eqnarray}\label{majoriz}
|\xi(x)| &\le_{\rm st}& \Xi
\quad\mbox{for all }x\in\R.
\end{eqnarray}
Assume that $X_n\to\infty$ with probability $1$ as $n\to\infty$ and, moreover,
its renewal measure satisfies
\begin{eqnarray}\label{finite.bound.U}
\sup_{x\in\R} H(x,x+1] &<& \infty.
\end{eqnarray}

If the limiting random variable $\xi$ is non-lattice,
then $H(x,x+h]\to h/\E\xi$ as $x\to\infty$, for all fixed $h>0$.

If the chain $\{X_n\}$ is integer-valued and $\Z$ is the minimal lattice
for the variable $\xi$, then $H\{n\}\to 1/\E\xi$ as $n\to\infty$.
\end{theorem}

The condition $\E\xi>0$ excludes possibility of an atom of $\xi$ at point $-\infty$.
The condition (\ref{majoriz}) and the dominated convergence theorem imply
$|\xi|\le_{\rm st}\Xi$, $\E|\xi|<\infty$ and $\E\xi(x)\to\E\xi$ as
$x\to\infty$; in particular, the chain $\{X_n\}$ has an asymptotically
space-homogeneous drift.

\begin{proof}%[Proof of Theorem \ref{thm:ah.renewal}]
First of all, the condition (\ref{finite.bound.U}) allows us to apply
Helly's Selection Theorem to the family of measures $\{H(x+\cdot),\ x\in\Rp\}$
(see, for example, Theorem 2 in \cite[Section VIII.6]{Feller}).
Hence, there exists a sequence of points $t_n\to\infty$ such that
the sequence of measures $H(t_n+\cdot)$ converges weakly to some measure
$\lambda$ as $n\to\infty$. The following result characterises $\lambda$,
it follows from Lemma \ref{l.1.i} with $\nu(x)\equiv 1$.

\begin{lemma}\label{l.1}
Let $F$ denote the distribution of $\xi$.
A weak limit $\lambda$ of the sequence of measures $H(t_n+\cdot)$
satisfies the identity $\lambda=\lambda*F$.
\end{lemma}

In the sequel the following auxiliary result is useful.

\begin{lemma}\label{l.2.2.2}
Let $\lambda_n$ be a sequence of measures on $\R$ weakly convergent 
to an absolutely continuous $\sigma$-finite measure $\lambda$.
Let $F_n:\R\to\R$ be a sequence of increasing functions weakly convergent
to an increasing function $F$. Then, for any $A>0$,
\begin{eqnarray*}
\int_0^A F_n(x)\lambda_n(dx) &\to& \int_0^A F(x)\lambda(dx)
\quad\mbox{as }n\to\infty.
\end{eqnarray*}
\end{lemma}

\begin{proof}
Firstly, by integration by parts,
\begin{eqnarray*}
\int_0^A F_n(x)(\lambda_n-\lambda)(dx) &=&
F_n(x)(\lambda_n-\lambda)[0,x]\Big|_0^A
-\int_0^A (\lambda_n-\lambda)[0,x]dF_n(x)\\
&\to& 0\quad\mbox{as }n\to\infty,
\end{eqnarray*}
because $\lambda_n[0,x]\to\lambda[0,x]$ as $n\to\infty$ uniformly for all 
$x\in[0,A]$, due to the weak convergence $\lambda_n\Rightarrow\lambda$
and the absolute continuity of the measure $\lambda$. Secondly,
\begin{eqnarray*}
\int_0^A F_n(x)\lambda(dx) &\to& \int_0^A F(x)\lambda(dx)
\quad\mbox{as }n\to\infty,
\end{eqnarray*}
by the dominated convergence theorem, because $F_n(x)\to F(x)$ almost everywhere 
due to the weak convergence of $F_n$ and monotonicity of $F_n(x)$ and $F(x)$. 
Altogether implies the desired convergence of integrals.
\qed
\end{proof}

The concluding part of the proof of Theorem \ref{thm:ah.renewal}
will be carried out for the non-lattice case.
Choose any sequence of points $t_n\to\infty$ such that the measure
$H(t_n+\cdot)$ converges weakly to some measure $\lambda$ as $n\to\infty$.
It follows from Lemma \ref{l.1} and Proposition \ref{l.2.i} that then
$\lambda(dx)=\alpha\cdot dx$ with some $\alpha$, i.e.,
\begin{eqnarray*}
H(t_n+dx) &\Rightarrow& \alpha\cdot dx\ \mbox{ as }n\to\infty.
\end{eqnarray*}

Now it suffices to prove that $\alpha=1/\E\xi$ for all sequences $t_n$
such that $H(t_n+\cdot)$ is weakly convergent.

Fix some $k\in{\N}$. Put
\begin{eqnarray*}
H^{(k)}(\cdot) &:=& H * P^k(\cdot)
\ =\ \sum_{j=k}^\infty\P\{X_j\in\cdot\}\\
&=& H(\cdot)-\sum_{j=0}^{k-1}\P\{X_j\in\cdot\}. 
\end{eqnarray*}
Then, due to the weak convergence $\P\{X_j\in t_n+\cdot\}\Rightarrow 0$ for all $j$,
\begin{eqnarray}\label{Uk.tn.to.lambda}
H^{(k)}(t_n+dx) &\Rightarrow& \alpha\cdot dx
\ \mbox{ as }n\to\infty.
\end{eqnarray}
Consider the measure $H^{(k)}-H^{(k+1)}=H^{(k)} * (I-P)$; by the definition of
the renewal measure it equals the distribution of $X_k$, that is, for any
{\it bounded} Borel set $B$, $H^{(k)}(B)-H^{(k+1)}(B)=\P\{X_k\in B\}$ (the equality
may fail for unbounded sets, say, for $(x,\infty]$). In particular,
\begin{eqnarray}\label{I.-.P.to.1}
(H^{(k)}-H^{(k+1)})(0,x] &=& \P\{X_k\in(0,x]\}
\to \P\{X_k>0\} \ \mbox{ as }x\to\infty.
\end{eqnarray}
On the other hand,
\begin{eqnarray}\label{I.-.P.+.-}
&&(H^{(k)}-H^{(k+1)})(0,x]\nonumber\\
&&\hspace{1cm}= \int_{-\infty}^\infty (I-P)(y,(0,x])H^{(k)}(dy)\nonumber\\
&&\hspace{1cm}=-\int_{-\infty}^0 P(y,(0,x])H^{(k)}(dy)
+\int_0^x P(y,(-\infty,0])H^{(k)}(dy)\nonumber\\
&&\hspace{15mm}+\int_0^x P(y,(x,\infty))H^{(k)}(dy)
-\int_x^\infty P(y,(0,x])H^{(k)}(dy).
\end{eqnarray}
By Lemma \ref{l.2.2.2}, the asymptotic homogeneity of the chain 
and weak convergence \eqref{Uk.tn.to.lambda} 
imply the following convergences of the integrals, for any fixed $A>0$:
\begin{eqnarray}\label{I.-.P.+.-.1}
\int_{t_n-A}^{t_n} P(y,(t_n,\infty))H^{(k)}(dy)
&\to& \alpha \int_0^A \P\{\xi>z\}dz
\end{eqnarray}
as $n\to\infty$, and
\begin{eqnarray}\label{I.-.P.+.-.2}
\int_{t_n}^{t_n+A} P(y,(0,t_n])H^{(k)}(dy)
&\to& \alpha \int_0^A \P\{\xi\le -z\}dz.
\end{eqnarray}
The majorisation condition (\ref{majoriz})
allows us to estimate the tails of the integrals:
\begin{eqnarray}\label{I.-.P.+.-.3}
\int_0^{t_n-A} P(y,(t_n,\infty))H^{(k)}(dy)
&\le& -\int_A^\infty \P\{\Xi>z\}H(t_n-dz)
\end{eqnarray}
and
\begin{eqnarray}\label{I.-.P.+.-.4}
\int_{t_n+A}^\infty P(y,(0,t_n])H^{(k)}(dy)
&\le& \int_A^\infty \P\{\Xi\ge z\}H(t_n+dz).
\end{eqnarray}
Since the majorant $\Xi$ is integrable,
the condition (\ref{finite.bound.U})
guarantees that the right hand sides of the inequalities
\eqref{I.-.P.+.-.3} and \eqref{I.-.P.+.-.4}
can be made as small as we please by the choice
of a sufficiently large $A$.
For these reasons we conclude from
\eqref{I.-.P.+.-}--\eqref{I.-.P.+.-.2} that
\begin{eqnarray*}
&&(H^{(k)}-H^{(k+1)})(0,t_n]\\
&&\hspace{1cm} \to\ -\int_{-\infty}^0 P(y,(0,\infty))H^{(k)}(dy)
+\int_0^\infty P(y,(-\infty,0])H^{(k)}(dy)\\
&&\hspace{30mm} +\alpha\int_0^\infty \P\{\xi>z\}dz
-\alpha\int_0^\infty \P\{\xi\le -z\}dz
\quad\mbox{as }n\to\infty.
\end{eqnarray*}
Together with \eqref{I.-.P.to.1} it implies the following equality,
for any fixed $k$:
\begin{eqnarray}\label{mu.k.int0.alpha}
\P\{X_k>0\}
&=& -\int_{-\infty}^0 P(y,(0,\infty))H^{(k)}(dy)
+\int_0^\infty P(y,(-\infty,0])H^{(k)}(dy)
+\alpha\E\xi.\nonumber\\[-2mm]
\end{eqnarray}
Now let $k\to\infty$, then both integrals go to zero.
For example, the first integral can be estimated
as follows, for all $A>0$:
\begin{eqnarray*}
\int_{-\infty}^0 P(y,(0,\infty))H^{(k)}(dy)
&\le& \int_{-\infty}^{-A} \P\{\Xi>-y\}H(dy)+H^{(k)}(-A,0].
\end{eqnarray*}
Here, for any fixed $A$, $H^{(k)}(-A,0]\to 0$ as $k\to\infty$,
due to \eqref{finite.bound.U}.
Therefore, \eqref{I.-.P.to.1} and
\eqref{mu.k.int0.alpha} imply that $1=\alpha\E\xi$
and the proof is complete.
\qed\end{proof}

In the next theorem we provide some simple
conditions sufficient for the condition (\ref{finite.bound.U}),
that is, for local compactness of the renewal measure.
Denote $a\wedge b=\min\{a,b\}$.

\begin{theorem}\label{thm:suff.for.tran}
Suppose that there exist $A>0$ and $\varepsilon>0$ such that
\begin{eqnarray}\label{sft.1}
\E(\xi(x)\wedge A) &\ge& \varepsilon\quad\mbox{for all }x\in\R.
\end{eqnarray}
In addition, let
\begin{eqnarray}\label{sft.2}
\P\{X_n>x\mbox{ for all } n\ge 1|X_0=x\} &\ge&
\delta> 0\quad\mbox{for all }x\in\R.
\end{eqnarray}
Then $H(x,x+h]\le (A+h)/\varepsilon\delta$
for all $x\in\R$ and $h>0$; in particular, \eqref{finite.bound.U} holds.
\end{theorem}

\begin{proof} By the Markov property, it suffices to show that
\begin{eqnarray}\label{renewal.boundedness}
H_y(x,x+h] &\le& (A+h)/\varepsilon\delta
\end{eqnarray}
for all $y\in(x,x+h]$.
Given $X_0\in(x,x+h]$, consider a stopping time
$$
T(x+h)=\min\{n\ge1:X_n>x+h\}.
$$
Since $X_{T(x+h)}\wedge(x+h+A)-X_0 \le A+h$ with
probability $1$,
\begin{eqnarray*}
A+h &\ge& \E(X_{T(x+h)}\wedge(x+h+A)-X_0)\\
&=& \sum_{n=1}^\infty \E
[X_n\wedge(x+h+A)-X_{n-1}\wedge(x+h+A)]{\I}\{T(x+h)\ge n\}.
\end{eqnarray*}
Hence, the definition of $T(x+h)$ implies
\begin{eqnarray*}
A+h &\ge& \sum_{n=1}^\infty
\E\{X_n\wedge(x+h+A)-X_{n-1}\wedge(x+h+A);T(x+h)\ge n\}\\
&=& \sum_{n=1}^\infty
\E\{X_n\wedge(x+h+A)-X_{n-1}|T(x+h)\ge n\}
\P\{T(x+h)\ge n\}.
\end{eqnarray*}
The Markov property and condition
(\ref{sft.1}) yield
\begin{eqnarray*}
\E\{X_n\wedge(x+h+A)-X_{n-1}|T(x+h)\ge n\}
&\ge& \E(\xi(X_{n-1})\wedge A)
\ge \varepsilon
\end{eqnarray*}
for all $n$. Therefore,
\begin{eqnarray*}
A+h &\ge& \varepsilon
\sum_{n=1}^\infty \P\{T(x+h)\ge n\}
= \varepsilon \E T(x+h).
\end{eqnarray*}
So, the expected number of visits to the interval
$(x,x+h]$ till the first exit from
$(-\infty,x+h]$ does not exceed
$(A+h)/\varepsilon$, independently of
the initial state $X_0\in(x,x+h]$.
By the condition \eqref{sft.2}, after exiting $(-\infty,x+h]$
the chain is above the level $X_T(x+h)$ forever
with probability at least $\delta$;
in particular, it does not visit the interval $(x,x+h]$ any more.
With probability at most $1-\delta$
the chain visits this interval again, and so on.
Concluding, we get that the expected number of
visits to the interval $(x,x+h]$ cannot exceed the value of
\begin{eqnarray*}
\frac{A+h}{\varepsilon} \sum_{n=0}^\infty (1-\delta)^n
&=& \frac{A+h}{\varepsilon\delta},
\end{eqnarray*}
and \eqref{renewal.boundedness} is proven.
The proof of Theorem \ref{thm:suff.for.tran} is complete.
\qed\end{proof}

\begin{corollary}\label{cor:min.tr}
Let the family of jumps $\{\xi(x),x\in\R\}$ possess an integrable
minorant with a positive mean, that is, there exists a random variable $\zeta$
such that $\E\zeta>0$ and $\xi(x)\ge_{\rm st}\zeta$ for all $x\in\R$. Then
\begin{eqnarray*}
H(x,x+h] &\le& (A+h)A/\varepsilon^2
\end{eqnarray*}
for all $A>0$ such that
$\varepsilon\equiv\E(\zeta\wedge A)>0$;
in particular, \eqref{finite.bound.U} holds.
\end{corollary}

\begin{proof} Consider the partial sums
$Z_n=\zeta_1+\ldots+\zeta_n$ of independent copies of $\zeta$.
Denote the first ascending ladder epoch by $\eta=\min\{n\ge1:Z_n>0\}$.
It is well known (see, for example, Theorem 2.3(c) in
\cite[Chapter VIII]{Aapq} that
\begin{eqnarray*}
\P\{Z_n>0\mbox{ for all }n\ge1\} &=& 1/\E\eta.
\end{eqnarray*}
Since
\begin{eqnarray*}
\P\{X_n>x\mbox{ for all } n\ge 1\mid X_0=x\}
&\ge& \P\{Z_n>0\mbox{ for all }n\ge1\}
\end{eqnarray*}
by the minorisation condition, the $\delta$ in Theorem
\ref{thm:suff.for.tran} is at least $1/\E\eta$.
Since $Z_{A,\eta_A}\le A$ where $Z_{A,n}:=\zeta_1\wedge A+\cdots+\zeta_n\wedge A$
and $\eta_A:=\min\{n\ge1:Z_{A,n}>0\}$, we get $\E\eta_A\le A/\varepsilon$
by Wald's equality $\E Z_{\eta_A}=\E\eta_A\E\zeta_1\wedge A$.
Then it follows from $\eta\le\eta_A$ that $\E\eta\le A/\varepsilon$,
which yields $\delta\ge \varepsilon/A$ and the corollary conclusion follows.
\qed\end{proof}

\section{Large deviation principle for stationary distribution}
\label{sec:ah.stationary.rough}

We now turn to the asymptotic behaviour of the stationary distribution of an
asymptotically homogeneous chain, that is, we assume that \eqref{asymp.hom}
holds with an extended limiting variable $\xi$.
We shall also assume that the limiting variable $\xi$ satisfies 
Cram\'er's condition:
\begin{eqnarray}\label{beta.cond}
\text{there exists a }\beta>0\text{ such that }\E e^{\beta\xi}=1.
\end{eqnarray}

As is well-known, the stationary measure of the random walk $\{W_n\}$
delayed at the origin---defined in \eqref{def:W}, say $\pi_{W}$, coincides with the distribution
of $\sup_{n\ge 0}\sum_{k=1}^n\xi_k$ where $\xi_k$'s are independent
copies of $\xi$.
Then, due to the classical Cram\'er---Lundberg approximation, for some $c>0$,
\index{Cram\'er--Lundberg!approximation}
\index{Random walk!delayed at zero!Cram\'er--Lundberg approximation}
\begin{eqnarray}\label{C.L.a}
\pi_W(x,\infty) &\sim& c e^{-\beta x}\quad\mbox{as }x\to\infty,
\end{eqnarray}
under the additional assumption $\E\xi e^{\beta\xi}<\infty$, in the non-lattice case;
in the lattice case $x$ is restricted to the lattice values.
Since the jumps of the chains $\{X_n\}$ and $\{W_n\}$ are asymptotically equivalent,
one could expect that the stationary tail distributions 
of $\{X_n\}$ and $\{W_n\}$ are asymptotically equivalent.
It turns out to be true on the logarithmic scale only.
\index{Markov chain!asymptotically homogeneous in space!large deviation principle}

\begin{theorem}\label{thm:log.tail}
Assume the asymptotic homogeneity \eqref{asymp.hom} and 
Cram\'er's condition \eqref{beta.cond}. If $\pi$ has right
unbounded support then the following lower bound holds:
\begin{eqnarray}\label{log.tail.2.1}
\liminf_{x\to\infty}\frac{\log\pi(x,\infty)}{x} &\ge& -\beta.
\end{eqnarray}
If, in addition,
\begin{eqnarray}\label{log.tail.1}
\sup_{x>0}\E e^{\lambda\xi(x)}\ <\ \infty,\quad
\sup_{x\le 0}\E e^{\lambda(x+\xi(x))}\ <\ \infty
\quad\text{for all }\lambda\in[0,\beta),
\end{eqnarray}
then
\begin{eqnarray}\label{log.tail.2}
\limsup_{x\to\infty}\frac{\log\pi(x,\infty)}{x} &\le& -\beta.
\end{eqnarray}
\end{theorem}

\begin{proof}
Fix some $\widehat x\in\R$ and consider an aggregated Markov chain
$\{X_n^*\}$ on $[\widehat x,\infty)$ with transition probabilities
defined in \eqref{aggr.1} and \eqref{aggr.2}.
As mentioned there, the measure $\pi^*$ that aggregates states from
$(-\infty,\widehat x]$ to $\widehat x$, that is,
$\pi^*\{\widehat x\}=\pi(-\infty,\widehat x]$
and $\pi^*(B)=\pi(B)$ for all $B\subseteq(\widehat x,\infty)$,
is an invariant measure for $\{X_n^*\}$.

First we derive the lower bound \eqref{log.tail.2.1} via comparison
of $\{X_n^*\}$ with a random walk delayed at zero;
we choose $\widehat x$ sufficiently large as follows.
For any $u$ consider a random variable $\eta(u)$ with tail distribution 
$$
\P\{\eta(u)>z\}\ =\ \inf_{v\ge u-1/u}\P\{\xi(v)>z+2/u\}.
$$
Then $\eta(u)$ stochastically increases as $u$ grows and
$\xi(v)\ge_{\rm st}\eta(u)$ for all $v\ge u-1/u$.
For any $A>0$, define $\eta_A(u):=\min\{\eta(u),A\}$.
Since the chain $\{X_n\}$ is asymptotically homogeneous,
we have $\eta_A(u)\le_{\rm st}\xi$ for all $u$ and
$\eta_A(u)\Rightarrow\xi$ as $A$, $u\to\infty$.
Hence, for all sufficiently large $A$ and $u$, there exists a unique solution 
$\beta_A(u)$ to the equation $\E e^{\beta_A(u)\eta_A(u)}=1$,
which is always not less than $\beta$.
In addition, $\beta_A(u)$ decreases as $A$ and $u$ grow, and
$$
\beta_A(u)\ \downarrow\ \beta\quad\mbox{as }A,\ u\to\infty.
$$
Fix an $\varepsilon\in(0,1)$ and choose sufficiently large $A$ and $\widehat x$
such that $\beta_A(\widehat x)\in[\beta,\beta+\varepsilon]$
and $\pi(\widehat x-1/\widehat x,\widehat x] >0$,
which is possible because $\pi$ has right-unbounded support.
Denote $\widehat\eta:=\eta_A(\widehat x)$. 
It follows from \eqref{aggr.2} that, for $u>\widehat x$,
\begin{eqnarray}\label{pi.eps.7}
P^*(\widehat x,(u,\infty)) &\ge&
\frac{1}{\pi(-\infty,\widehat x]}
\int_{\widehat x-1/\widehat x}^{\widehat x+0}P(z,(u,\infty))\pi(dz)\nonumber\\
&\ge& \frac{\pi(\widehat x-1/\widehat x,\widehat x]}{\pi(-\infty,\widehat x]}
\P\{\widehat\eta>u-\widehat x\}.
\end{eqnarray}

Consider the random walk $\{\widehat W_n\}$ delayed at $\widehat x$, that is,
$$
\widehat W_n\ =\ \max(\widehat x,\ \widehat W_{n-1}+\widehat\eta_n)
$$
where $\widehat\eta_n$ are independent copies of $\widehat\eta$.
By the construction of $\widehat\eta=\eta_A(\widehat x)$,
$\{\widehat W_n\}$ is dominated by $\{X_n\}$ above $\widehat x$,
more precisely, the following inequality is valid
for all $u>\widehat x$, $y>\widehat x$ and $m$:
\begin{eqnarray}\label{tt3}
\lefteqn{\P\{X_k>\widehat x\mbox{ for all }k<m,\
X_m>y\mid X_0=u\}}\nonumber\\
&&\hspace{20mm}\ge\ \P\{\widehat W_k>\widehat x\mbox{ for all }k<m,\
\widehat W_m>y\mid \widehat W_0=u\}.
\end{eqnarray}

Consider a stationary version of $\{X_n\}$, that is,
$X_n$ has distribution $\pi$ for all $n\ge 0$.
Then the distribution of $\max(\widehat x,X_n)$ on $(\widehat x,\infty)$
is the same as of $X_n^*$ given $X_0^*$
has distribution $\pi^*$.
Then, at any time $n$, the decomposition of all trajectories
with respect to the last visit of $\{X_k^*\}$ to
the state $\widehat x$ gives the following lower bound, for $y>\widehat x$,
\begin{eqnarray*}
\lefteqn{\pi(y,\infty)\ =\ \P\{X_n^*>y\}}\\
&\ge& \sum_{j=0}^{n-1}\P\{X_j^*=\widehat x\}
\int_{\widehat x+0}^\infty 	P^*(\widehat x,du)
\P\{X^*_k>\widehat x,k\in[j+2,n-1],\
X^*_n>y\mid X^*_{j+1}=u\}\\
&=& \pi(-\infty,\widehat x]\sum_{j=0}^{n-1}
\int_{\widehat x+0}^\infty 	P^*(\widehat x,du)
\P\{X_k>\widehat x,k\in[j+2,n-1],\ X_n>y\mid X_{j+1}=u\}\\
&\ge& \pi(-\infty,\widehat x]\sum_{j=0}^{n-1}
\int_{\widehat x+0}^\infty 	P^*(\widehat x,du)
\P\{\widehat W_k>\widehat x,k\in[j+2,n-1],\
\widehat W_n>y\mid \widehat W_{j+1}=u\},
\end{eqnarray*}
due to \eqref{tt3}. Since the probability
$\P\{\widehat W_k>\widehat x,k\in[j+2,n-1],\
\widehat W_n>y\mid \widehat W_{j+1}=u\}$
is increasing in $u$, integration by parts and \eqref{pi.eps.7}
yield that
\begin{eqnarray*}
\lefteqn{\int_{\widehat x+0}^\infty P^*(\widehat x,du)
\P\{\widehat W_k>\widehat x,k\in[j+2,n-1],\
\widehat W_n>y\mid \widehat W_{j+1}=u\}}\\
&&\ge\ \frac{\pi(\widehat x-1/\widehat x,\widehat x]}{\pi(-\infty,\widehat x]}
\int_{\widehat x+0}^\infty \P\{\widehat x+\widehat\eta\in du\}\\
&&\hspace{40mm}\times\P\{\widehat W_k>\widehat x,k\in[j+2,n-1],\
\widehat W_n>y\mid \widehat W_{j+1}=u\}.
\end{eqnarray*}
Therefore,
\begin{eqnarray}\label{tt4}
\pi(y,\infty) &\ge& \pi(\widehat x-1/\widehat x,\widehat x]
\sum_{j=0}^{n-1}\int_{\widehat x+0}^\infty
\P\{\widehat x+\widehat\eta\in du\}\nonumber\\
&&\hspace{5mm}\times \P\{\widehat W_k>\widehat x,k\in[j+2,n-1],\
\widehat W_n>y\mid \widehat W_{j+1}=u\}.\hspace{5mm}
\end{eqnarray}
On the other hand, applying the decomposition of all trajectories
of $\{\widehat W_n\}$ with respect to the last visit of $\{\widehat W_n\}$
to the state $\widehat x$ we deduce,
for $\widehat W_0=\widehat x$ and $y>\widehat x$,
\begin{eqnarray*}
\lefteqn{\P\{\widehat W_n>y\}}\\
&=& \sum_{j=0}^{n-1}\P\{\widehat W_j=\widehat x\}
\int_{\widehat x+0}^\infty
\P\{\widehat W_{j+1}\in du\mid\widehat W_j=\widehat x\}\\
&&\hspace{40mm}\times\P\{\widehat W_k>\widehat x,k\in[j+2,n-1],\
\widehat W_n>y\mid \widehat W_{j+1}=u\}\\
&\le& \sum_{j=0}^{n-1}\int_{\widehat x+0}^\infty
\P\{\widehat x+\widehat\eta\in du\}
\P\{\widehat W_k>\widehat x,k\in[j+2,n-1],\
\widehat W_n>y\mid \widehat W_{j+1}=u\}.\\[-2mm]
\end{eqnarray*}
Together with \eqref{tt4} it implies the following lower bound
\begin{eqnarray*}
\pi(y,\infty) &\ge& \pi(\widehat x-1/\widehat x,\widehat x]
\P\{\widehat W_n>y\}\quad\mbox{for }y>\widehat x.
\end{eqnarray*}
The Cram\'er--Lundberg approximation \eqref{C.L.a} yields that
\begin{eqnarray*}
\lim_{y\to\infty}\lim_{n\to\infty}\frac{\log\P\{\widehat W_n>y\}}{y}
&=& -\beta_A(\widehat x),
\end{eqnarray*}
so hence
\begin{eqnarray*}
\liminf_{y\to\infty}\frac{\log\pi(y,\infty)}{y}
&\ge& -\beta_A(\widehat x).
\end{eqnarray*}
Letting $\varepsilon\downarrow 0$ we conclude the assertion
\eqref{log.tail.2.1} because $\beta_A(\widehat x)\le\beta+\varepsilon$.

Let us now prove the upper bound \eqref{log.tail.2}.
Fix any $\lambda<\beta$.
Then the boundedness \eqref{log.tail.1} of exponential moments of jumps
of order $(\beta+\lambda)/2\in(\lambda,\beta)$ and weak convergence
$\xi(x)\Rightarrow\xi$ imply convergence of exponential moments
of order $\lambda$,
$$
\E e^{\lambda\xi(x)}\ \to\ \E e^{\lambda\xi}\ <\  \E e^{\beta\xi}\ =\ 1
\quad\mbox{as }x\to\infty,
$$
hence there exist $\widehat x\in\R$ and $\varepsilon>0$ such that
\begin{eqnarray}\label{E.lambda.0}
\E e^{\lambda\xi(x)} &\le& 1-\varepsilon\quad\mbox{for all }x\ge\widehat x.
\end{eqnarray}
Fix an $A>\widehat x$ and consider the function $g(x)=\min(e^{\lambda x},e^{\lambda A})$.
Let $\{X_n\}$ be in stationary regime, that is, let $X_n$ have distribution $\pi$
for all $n$. Since $g$ is bounded above---by $e^{\lambda A}$,
\begin{eqnarray}\label{0=3int}
0 &=& \E(g(X_1)-g(X_0))\nonumber\\
&=& \Bigl(\int_{-\infty}^{\widehat x}+
\int_{\widehat x}^A+\int_A^\infty\Bigr)(\E g(x+\xi(x))-g(x))\pi(dx).
\end{eqnarray}
The third integral on the right hand side is non-positive
because the increasing function $g(x)$ is constant for $x\ge A$.
The first integral is bounded above by
\begin{eqnarray*}
c_1 &:=& \sup_{x\le\widehat x}\E\{g(X_1)-g(x)\mid X_0=x\}\
\le\ \sup_{x\le\widehat x}\E e^{\lambda(x+\xi(x))},
\end{eqnarray*}
which is finite due to the condition \eqref{log.tail.1}.
The second integral is not greater than
\begin{eqnarray*}
\int_{\widehat x}^A (\E g(x+\xi(x))-g(x))\pi(dx) &\le&
\int_{\widehat x}^A (\E e^{\lambda(x+\xi(x))}-e^{\lambda x})\pi(dx)\\
&\le& -\varepsilon\int_{\widehat x}^A e^{\lambda x}\pi(dx),
\end{eqnarray*}
by \eqref{E.lambda.0}. Therefore, it follows from \eqref{0=3int} that
\begin{eqnarray*}
0 &\le& c_1-\varepsilon\int_{\widehat x}^A e^{\lambda x}\pi(dx).
\end{eqnarray*}
Due to the arbitrary choice of $A$, we get
\begin{eqnarray*}
\int_{\widehat x}^\infty e^{\lambda x}\pi(dx) &\le& c_1/\varepsilon,
\end{eqnarray*}
which implies $\pi(x,\infty)\le c_1e^{-\lambda x}/\varepsilon$
for all $x\ge\widehat x$.
Now the upper bound \eqref{log.tail.2} follows because we may chose
$\lambda<\beta$ as close to $\beta$ as we please.
\qed\end{proof}

\section{Sharp asymptotics for stationary distribution}
\label{sec:ah.stationary.sharp}

While logarithmic asymptotic law is universal for stationary distribution
of asymptotically homogeneous in space Markov chains,
it turns out that the exact asymptotic tail behaviour of $\pi$
depends not only on the distribution of $\xi$,
but also on the speed of convergence in \eqref{asymp.hom}.

The next result describes the case where this convergence is so fast
that the measure $\pi$ is asymptotically tail proportional
to the stationary measure of $\{W_n\}$.
\index{Markov chain!asymptotically homogeneous in space!exponential asymptotics}

\begin{theorem}\label{th:conv}
Assume the asymptotic homogeneity \eqref{asymp.hom} and 
Cram\'er's condition \eqref{beta.cond}. 
Let $\pi$ have right unbounded support. Suppose that
\begin{eqnarray}\label{Xi}
\xi(x) &\le_{st}& \Xi, \qquad x\in\R,
\end{eqnarray}
for some random variable $\Xi$ such that $\E\Xi e^{\beta\Xi}<\infty$ and
\begin{eqnarray}\label{conver}
|\E e^{\beta\xi(x)}-1| &\le& \gamma(x)
\end{eqnarray}
for some decreasing integrable at infinity function $\gamma(x)$.

If the distribution of $\xi$ is non-lattice then there exists a positive
constant $c$ such that
\begin{equation}\label{T.conv}
\pi(x,\infty)\ \sim\ c e^{-\beta x}\quad\mbox{as }x\to\infty.
\end{equation}
If $\{X_n\}$ takes values on $\Z$ and $\Z$ is the minimal lattice for $\xi$
then \eqref{T.conv} holds with $x$ restricted to integers.
\end{theorem}

The condition \eqref{conver} is quite close to be optimal.
If, for example, $\E e^{\beta\xi(x)}-1$ are of the same sign and not summable,
then $\pi(x)e^{\beta x}$ converges either to zero or to infinity, see
Corollary \ref{cor_power} below. Thus, if \eqref{conver} is violated, then
$\pi(x,\infty)$ may only have exponential asymptotics like \eqref{T.conv} 
in the case where $\E e^{\beta\xi(x)}-1$ is changing its sign infinitely often.

\begin{example}\label{ex:3}
Consider a Markov chain $\{X_n\}$ on $\Zp$ with jumps to the nearest neighbours
only:
$$
\P\{\xi(i)=1\}=1-\P\{\xi(i)=-1\}=p+\varphi(i).
$$
Assume that, as $i\to\infty$,
$$
\varphi(i)\sim\left\{
\begin{array}{ll}
i^{-\gamma}, &i=2k\\
-i^{-\gamma}, &i=2k+1
\end{array}
\right.
$$
for some $\gamma\in(1/2,1)$.
Clearly, then the asymptotic homogeneity \eqref{asymp.hom} and 
Cram\'er's condition \eqref{beta.cond} hold true while the condition \eqref{conver} fails.
\end{example}

Let us have a look at values of $\{X_n\}$ at even time epochs, i.e.,
let us consider the chain
$$
Y_k=X_{2k},\quad k\ge 0.
$$
Then we have
\begin{align*}
&\P_i\{Y_1-i=-2\}=(q-\varphi(i))(q-\varphi(i-1)),\\
&\P_i\{Y_1-i=0\}=(q-\varphi(i))(p+\varphi(i-1))+(p+\varphi(i))(q-\varphi(i+1)),\\
&\P_i\{Y_1-i=2\}=(p+\varphi(i))(p+\varphi(i+1)),
\end{align*}
where $q:=1-p$. From these equalities we obtain
\begin{align*}
&\E_i \left(\frac{q}{p}\right)^{Y_1-i}-1=
\left(\frac{p^2}{q^2}-1\right)\P_i\{Y_1-i=-2\}+\left(\frac{q^2}{p^2}-1\right)\P_i\{Y_1-i=2\}\\
&\hspace{5mm}=\left(\frac{p^2}{q^2}-1\right)(q-\varphi(i))(q-\varphi(i-1))
+\left(\frac{q^2}{p^2}-1\right)(p+\varphi(i))(p+\varphi(i+1))\\
&\hspace{5mm}=-q\left(\frac{p^2}{q^2}-1\right)(\varphi(i)+\varphi(i-1))
+p\left(\frac{q^2}{p^2}-1\right)(\varphi(i)+\varphi(i+1))
+O(i^{-2\gamma}).
\end{align*}
Noting that $\varphi(i)+\varphi(i+1)=O(i^{-\gamma-1})$,
we conclude that the sequence $|\E_i(q/p)^{Y_1-i}-1|$ is summable and,
consequently, we may apply Theorem \ref{th:conv}.
Since $\pi$ is stationary for $Y$ too,
we obtain $\pi(i)\sim c(p/q)^i$ as $i\to\infty$.
\qed

\begin{theopargself}
\begin{proof}[of Theorem~\ref{th:conv}]
We start, as usual, with the construction of an appropriate Lyapunov
function which is sufficiently close to a harmonic function.
Let $p$ be a bounded decreasing function $p(x):\R\to\Rp$
which is regularly varying at infinity with index $-1$
and integrable at infinity. Set
\begin{eqnarray}\label{def.g.p}
g(x) &:=& \min\Bigl(1,\ \int_x^\infty p(y)dy\Bigr)
\end{eqnarray}
and consider
\begin{eqnarray}\label{def.U.g.p}
U_p(x) &:=& e^{\beta x}(1+g(x)).
\end{eqnarray}
We want to show that there exists a $p(x)$ such that
\begin{eqnarray}\label{conv.1}
\E U_p(x+\xi(x))-U_p(x) &=& -e^{\beta x}p(x)(\E\xi e^{\beta\xi}+o(1))
\quad\mbox{as }x\to\infty.
\end{eqnarray}
By the definition of $U_p(x)$,
\begin{eqnarray}\label{conv.2}
\lefteqn{\E U_p(x+\xi(x))-U_p(x)}\nonumber\\
&&=\ e^{\beta x}
\bigl(\E e^{\beta\xi(x)}(1+g(x+\xi(x)))-1-g(x)\bigr)\nonumber\\
&&\hspace{3mm}=\ e^{\beta x}(1+g(x))(\E e^{\beta\xi(x)}-1)
+e^{\beta x}\E(g(x+\xi(x))-g(x))e^{\beta\xi(x)}.
\end{eqnarray}
Owing to Lemma \ref{l:denis}, 
the assumption \eqref{conver} yields the existence of $p(x)$
satisfying the conditions above and such that
\begin{eqnarray}\label{conv.3}
|\E e^{\beta\xi(x)}-1|=o(p(x)).
\end{eqnarray}
Fix some increasing function $s(x)=o(x)$ and split the second term
on the right hand side of \eqref{conv.2} into three parts:
\begin{eqnarray*}
&&\E(g(x+\xi(x))-g(x))e^{\beta\xi(x)}=
\E\{(g(x+\xi(x))-g(x))e^{\beta\xi(x)};\ \xi(x)<-s(x)\}\\
&&\hspace{45mm}+\E\{(g(x+\xi(x))-g(x))e^{\beta\xi(x)};\ |\xi(x)|\le s(x)\}\\
&&\hspace{45mm}+\E\{(g(x+\xi(x))-g(x))e^{\beta\xi(x)};\ \xi(x)>s(x)\}.
\end{eqnarray*}
Due to the decrease of $g$ and the boundedness of $g$ by $1$,
\begin{eqnarray}\label{conv.4}
0\ \le\ \E\{(g(x+\xi(x))-g(x))e^{\beta\xi(x)};\ \xi(x)<-s(x)\}
&\le& e^{-\beta s(x)}.
\end{eqnarray}
Since $p(x)$ is assumed regularly varying at infinity,
$g(x+\xi(x))-g(x)\sim -p(x)\xi(x)$ as $x\to\infty$
uniformly on the set $|\xi(x)|\le s(x)$. Therefore,
$$
\E\{(g(x+\xi(x))-g(x))e^{\beta\xi(x)};\ |\xi(x)|\le s(x)\}
\sim -p(x)\E\{\xi(x)e^{\beta\xi(x)};\ |\xi(x)|\le s(x)\}.
$$
Recalling that the family $\xi(x)$ possesses a majorant $\Xi$ with $\E\Xi e^{\beta\Xi}<\infty$,
we infer that
$$
\E\{\xi(x)e^{\beta\xi(x)};\ |\xi(x)|\le s(x)\}\ \to\
\E\xi e^{\beta\xi}\quad\mbox{as }x\to\infty.
$$
As a result,
\begin{eqnarray}\label{conv.5}
\E\{(g(x+\xi(x))-g(x))e^{\beta\xi(x)};\ |\xi(x)|\le s(x)\}
\ \sim\ -p(x)\E\xi e^{\beta\xi}.
\end{eqnarray}
The existence of $\Xi$ implies also that the function
$\E\{e^{\beta\xi(x)};\ \xi(x)>s(x)\}$ is
dominated by $\E\{e^{\beta\Xi};\ \Xi>s(x)\}$.
Since $\E\Xi e^{\beta\Xi}$ is finite, 
the last function is decreasing and summable provided that
$s(x)/x\to0$ sufficiently slow. Consequently,
there exists $p(x)$ such that
\begin{eqnarray}\label{conv.6}
\E\{(g(x+\xi(x))-g(x))e^{\beta\xi(x)};\ \xi(x)>s(x)\}\ =\ o(p(x)).
\end{eqnarray}
Combining \eqref{conv.4}--\eqref{conv.6}, we conclude that
$$
\E(g(x+\xi(x))-g(x))e^{\beta\xi(x)}
\ =\ -p(x)(\E\xi e^{\beta\xi}+o(1)).
$$
Plugging this relation and \eqref{conv.3} into \eqref{conv.2},
we obtain \eqref{conv.1}.

Consider, as usual, the transition kernel
$$
Q(x,dy)=\frac{U_p(y)}{U_p(x)}P(x,dy),\quad y\ge\widehat x.
$$
It follows from \eqref{conv.1} that,
for all $\widehat x$ sufficiently large,
\begin{eqnarray}\label{conv.8}
\nonumber
Q(x,\R)&=&\frac{1}{U_p(x)}\E\{U_p(x+\xi(x));\ x+\xi(x)\ge \widehat x\}\\
&\le&\frac{1}{U_p(x)}\E U_p(x+\xi(x))\ \le\ 1
\quad\mbox{for all }x\ge\widehat x.
\end{eqnarray}
In other words, $Q$ is a substochastic kernel.
Furthermore, it follows from the asymptotic homogeneity that
\begin{eqnarray}\label{conv.9}
Q(x,\R)\ \ge\ \P\{\xi(x)\ge0\}\ \ge\ \P\{\xi\ge0\}/2
\quad\mbox{for all } x\ge\widehat x,
\end{eqnarray}
if $\widehat x$ is chosen sufficiently large.
Using \eqref{conv.1} once again, we conclude that
\begin{eqnarray}\label{conv.10}
q(x)\ :=\ -\log Q(x,\R)\ =\ O(p(x))\quad\mbox{as }x\to\infty.
\end{eqnarray}
Let $\{\widehat X_n\}$ be a Markov chain on
$(\widehat x,\infty)$ with the transition kernel
$$
\widehat{P}(x,dy)\ =\ \frac{Q(x,dy)}{Q(x,\R)}
$$
and let $\widehat{\xi}(x)$ denotes its jump from state $x$.
It is immediate from the definition of $U_p$ that
$\widehat\xi(x)$ converges weakly to the distribution
$e^{\beta y}\P\{\xi\in dy\}$ as $x\to\infty$.
Furthermore, the assumption that $\E\Xi e^{\beta\Xi}<\infty$
and \eqref{conv.9} imply that the family of jumps $|\widehat\xi(x)|$
possesses an integrable majorant.
Therefore, there exists an $\widehat x$ such that
the family of jumps $\{\widehat\xi(x);\ x>\widehat x\}$ 
possesses a stochastic minorant with positive expectation.
Thus, Corollary~\ref{cor:min.tr} applies to the chain $\{\widehat X_n\}$
which in its turn allows us to apply Theorem \ref{thm:ah.renewal}:
If $\xi$ is non-lattice then, for all $h>0$,
$$
\widehat H(x,x+h]\ \to\ \frac{h}{\E\xi e^{\beta\xi}}
\quad\mbox{as }x\to\infty.
$$
If $\{X_n\}$ is an integer-valued Markov chain and $\Z$ is the minimal lattice
for $\xi$ then the previous relation
is valid for $h$ and $x$ restricted to integers.

Combining \eqref{conv.10} with the upper bound
$\sup_x\widehat H(x,x+h]<\infty$ we conclude
as in Lemma \ref{L.limit} that
$$
\sum_{k=0}^\infty \E q(\widehat X_k)\ <\ \infty.
$$
Thus, by Lemma~\ref{thm:renewal.2},
\begin{eqnarray}\label{conv.12}
\widehat{H}^{(q)}(x,x+h]\ \to\
\frac{h}{\E\xi e^{\beta\xi}}\E e^{-\sum_{k=0}^\infty q(\widehat X_k)}
\quad\mbox{as }x\to\infty.
\end{eqnarray}
Here, again, $h$ is an arbitrary positive number in the case when $\xi$
is non-lattice and $h$ is integer in the lattice case.

For the invariant distribution $\pi$ we have the following representation,
see \eqref{pi.repr.B},
$$
\pi(dy)\ =\
c_*\frac{\widehat H^{(q)}(dy)}{U_p(y)}.
$$

If $\xi$ is lattice then
$$
\pi\{n\}\ =\
c_*\frac{\widehat H^{(q)}\{n\}}{U_p(n)},
$$
and the result follows from \eqref{conv.12} and the fact
that $U_p(x)\sim e^{\beta x}$.

In the non-lattice case, for any fixed $h>0$,
\begin{eqnarray*}
c_*\frac{\widehat H^{(q)}(x,x+h]}{\max_{x\le y\le x+h}U_p(y)}
&\le& \pi(x,x+h]\ \le\ 
c_*\frac{\widehat H^{(q)}(x,x+h]}{\min_{x\le y\le x+h}U_p(y)}.
\end{eqnarray*}
Using again \eqref{conv.12}, we obtain lower and upper bounds
$$
che^{-\beta x-\beta h}(1+o(1))\ \le\ \pi(x,x+h]\ \le\ che^{-\beta x}(1+o(1)).
$$
Choosing $h$ small and summing bounds for $\pi(x+kh,x+(k+1)h]$
we obtain the required lower and upper bounds for $\pi(x,\infty)$
which completes the proof of the theorem.
\qed\end{proof}
\end{theopargself}

We now turn to the case where $\E e^{\beta\xi(x)}$ converges to $1$
in a non-summable way. Our next result describes the behaviour of
$\pi$ in terms of a non-uniform exponential change of measure.
\index{Markov chain!asymptotically homogeneous in space!exponential asymptotics}

\begin{theorem}\label{th:non.gen}
Suppose the asymptotic homogeneity condition \eqref{asymp.hom}
and that Cram\'er's condition \eqref{beta.cond} holds and, 
for some $\varepsilon>0$,
\begin{eqnarray}\label{Xi.eps}
\sup_{x\in\R}\E e^{(\beta+\varepsilon)\xi(x)}<\infty.
\end{eqnarray}
Assume also that there exists a differentiable function $\beta(x)>0$
such that
\begin{eqnarray}\label{E.beta.1}
|\E e^{\beta(x)\xi(x)}-1| &\le& \gamma(x),
\end{eqnarray}
and $|\beta'(x)|\le\gamma(x)$ where $\gamma(x)$ is a bounded decreasing
integrable at infinity function. Then, for some $c>0$,
\begin{eqnarray*}
\pi(x,\infty) &\sim& c e^{-\int_0^x \beta(y)dy}\quad\mbox{as }x\to\infty,
\end{eqnarray*}
where $x$ runs through integers in the lattice case.
\end{theorem}

\begin{proof}
The proof is quite similar to that of Theorem \ref{th:conv},
the only alteration is a slightly trickier choice of the Lyapunov
function $U_p$. Instead of $(1+g(x))e^{\beta x}$ we now define
$$
U_p(x)\ :=\ (1+g(x))e^{\int_0^x \beta(y)dy}.
$$
Let $\delta<\varepsilon$ and $c>1/(\varepsilon-\delta)$. 
Observe that, with necessity, $\beta(x)\to\beta$ so that,
by the condition \eqref{Xi.eps}, for all sufficiently large $x$,
\begin{eqnarray}\label{E.beta.e.tail}
\lefteqn{\E\bigl\{e^{\beta(x)\xi(x)};\ |\xi(x)|>c\log x\bigr\}}\nonumber\\
&\le& \E\bigl\{e^{(\beta+\delta)\xi(x)};\ \xi(x)>c\log x\bigr\}
+\E\bigl\{e^{(\beta-\delta)\xi(x)};\ \xi(x)<-c\log x\bigr\}\nonumber\\
&=& O(e^{-c(\varepsilon-\delta) \log x}
+e^{-c(\beta-\delta) \log x})
\ =\ O(1/x^{c(\varepsilon-\delta)})\quad\mbox{as }x\to\infty,
\end{eqnarray}
where without loss of generality we assume that $\varepsilon<\beta$.
Similarly,
\begin{eqnarray*}
\E\bigl\{e^{\int_x^{x+\xi(x)}\beta(y)dy};\ |\xi(x)|>c\log x\bigr\}
&=& O(1/x^{c(\varepsilon-\delta)})\quad\mbox{as }x\to\infty.
\end{eqnarray*}
Further, by the mean value theorem, for some $\theta=\theta(x,\xi)\in(0,1)$,
\begin{eqnarray}\label{EE.sqrtx}
\lefteqn{\left|\E\bigl\{e^{\int_x^{x+\xi(x)}\beta(y)dy};\ |\xi(x)|\le c\log x\bigr\}
-\E \bigl\{e^{\beta(x)\xi(x)};\ |\xi(x)|\le c\log x\bigr\}\right|}\nonumber\\
&=& \left|\E\bigl\{e^{\beta(x)\xi(x)}\bigl(
e^{\int_x^{x+\xi(x)}(\beta(y)-\beta(x))dy}-1\bigr);\
|\xi(x)|\le c\log x\bigr\}\right|\nonumber\\
&=& \left|\E\biggl\{\int_x^{x+\xi(x)}(\beta(y)-\beta(x))dy\times
e^{\beta(x)\xi(x)+\theta\int_x^{x+\xi(x)}(\beta(y)-\beta(x))dy};\
|\xi(x)|\le c\log x\biggr\}\right|\nonumber\\
&\le& \frac{\gamma(x-c\log x)}{2}
\E \bigl\{\xi^2(x)e^{\beta(x)\xi(x)+\int_x^{x+\xi(x)}|\beta(y)-\beta(x)|dy};\
|\xi(x)|\le c\log x\bigr\},
\end{eqnarray}
because, by the condition $|\beta'(x)|\le\gamma(x)$ on the derivative
of $\beta(y)$, for $|\xi(x)|\le c\log x$,
\begin{eqnarray*}
\biggl|\int_x^{x+\xi(x)}(\beta(y)-\beta(x))dy\biggr|
&\le& \int_x^{x+\xi(x)}|\beta(y)-\beta(x)|dy\\
&\le& \sup_{|z|\le c\log x}|\beta'(x+z)|\xi^2(x)/2\\
&\le& \gamma(x-c\log x)\xi^2(x)/2.
\end{eqnarray*}
Uniformly on the event $|\xi(x)|\le c\log x$, we have
$\gamma(x-c\log x)\xi^2(x)\le c^2\gamma(x-c\log x)\log^2 x\to 0$ as $x\to\infty$,
since the function $\gamma(x)$ is decreasing and integrable at infinity.
Therefore, for all sufficiently large $x$,
the right hand side of \eqref{EE.sqrtx} is not greater than
\begin{eqnarray*}
\gamma(x-c\log x) \E \bigl\{\xi^2(x)e^{\beta(x)\xi(x)};\ |\xi(x)|\le c\log x\bigr\}
&=& O(\gamma(x-c\log x))\quad\mbox{as }x\to\infty,
\end{eqnarray*}
owing to the condition \eqref{Xi.eps}. Hence, as $x\to\infty$,
\begin{eqnarray*}
\E e^{\int_x^{x+\xi(x)}\beta(y)dy}
&=& \E e^{\beta(x)\xi(x)}+O(\gamma(x-c\log x)+1/x^{c(\varepsilon-\delta)}).
\end{eqnarray*}
Taking into account \eqref{E.beta.1} and $c>1/(\varepsilon-\delta)$, we conclude that there exists
a decreasing integrable at infinity function $p_1(x)$ such that
\begin{eqnarray}\label{non.gen.2}
\E e^{\int_x^{x+\xi(x)}\beta(y)dy} &=& 1+O(p_1(x))
\quad\mbox{as }x\to\infty.
\end{eqnarray}

We have an equality
\begin{eqnarray*}
\E U_p(x+\xi(x))-U_p(x) &=&
U_p(x)\bigl(\E e^{\int_x^{x+\xi(x)}\beta(y)dy}-1\bigr)\\
&&\hspace{5mm}
+e^{\int_0^x\beta(y)dy}\E(g(x+\xi(x))-g(x))e^{\int_x^{x+\xi(x)}\beta(y)dy}.
\end{eqnarray*}
Using \eqref{non.gen.2} and recalling that $g(x)$ is bounded, we get
\begin{eqnarray*}
\E U_p(x+\xi(x))-U_p(x) &=& O(p_1(x)U_p(x))\\
&&\hspace{5mm}
+e^{\int_0^x\beta(y)dy}\E(g(x+\xi(x))-g(x))e^{\beta(x)\xi(x)}.
\end{eqnarray*}
Repeating the corresponding arguments from the proof of
Theorem~\ref{th:conv} and using \eqref{Xi.eps}, we obtain
$$
\E\bigl\{(g(x+\xi(x))-g(x))e^{\beta(x)\xi(x)};|\xi(x)|>c\log x\bigr\}
\ =\ o(1/x^{c(\varepsilon-\delta)})
$$
and
$$
\E\bigl\{(g(x+\xi(x))-g(x))e^{\beta(x)\xi(x)};|\xi(x)|\le c\log x\bigr\}
\ \sim\ - p(x)\E\xi(x)e^{\beta(x)\xi(x)}.
$$
Therefore, taking $p(x)\gg p_1(x)$, we get
$$
\E U_p(x+\xi(x))-U_p(x)\ \sim\ -p(x)U_p(x)\E\xi(x)e^{\beta(x)\xi(x)}.
$$
Using \eqref{Xi.eps} once again, we deduce convergence
$\E\xi(x)e^{\beta(x)\xi(x)}\to\E\xi e^{\beta\xi}$.
Consequently,
\begin{eqnarray}\label{non.gen.3}
\E U_p(x+\xi(x))-U_p(x) &=& -p(x)U_p(x)\E\xi e^{\beta\xi}(1+o(1))
\quad\mbox{as }x\to\infty.
\end{eqnarray}
This means that $U_p$ is an appropriate Lyapunov function,
and the remaining part of the proof literally repeats
that of Theorem~\ref{th:conv}.
\qed\end{proof}

Since $\beta(x)$ is not given in a closed form, Theorem \ref{th:non.gen}
cannot be seen as a final statement. For that reason we describe below
two cases where $\beta(x)$ can be computed provided
regular behaviour of the difference $\E e^{\beta \xi(x)}-1$.
\index{Markov chain!asymptotically homogeneous in space!exponential asymptotics}

\begin{corollary}\label{cor:non.sum}
Assume the condition \eqref{Xi.eps} and that there exists
a differentiable function $\alpha(x)$ such that
\begin{eqnarray}\label{alpha.12}
\alpha(x) &=& O(1/x^{1/2+\varepsilon}),\\
\label{alpha.prime.12}
\alpha'(x) &=& O(\gamma(x))\quad\mbox{as }x\to\infty,
\end{eqnarray}
and
\begin{eqnarray}\label{non.conver}
\E e^{\beta\xi(x)}-1 &=& \alpha(x)+O(\gamma(x))
\quad\mbox{as }x\to\infty,
\end{eqnarray}
where $\gamma(x)$ is a decreasing integrable at infinity function.
Suppose also that
\begin{eqnarray}\label{conver.of.E}
\E\xi(x)e^{\beta\xi(x)} &=& m+O(\gamma(x)/\alpha(x))
\quad\mbox{as }x\to\infty,
\end{eqnarray}
where $m:=\E\xi e^{\beta\xi}$. Then
\begin{eqnarray}\label{answer.1}
\pi(x,\infty) &\sim& c e^{-\beta x+A(x)/m}\quad\mbox{as }x\to\infty,
\end{eqnarray}
where $c>0$ and $A(x):=\int_0^x\alpha(y)dy$.
\end{corollary}

\begin{proof}
Take $\beta(x):=\beta-\alpha(x)/m$. By Taylor's theorem,
uniformly on the event $|\xi(x)|\le 1/\alpha(x)$,
\begin{eqnarray*}
e^{-\alpha(x)\xi(x)/m} &=& 1-\alpha(x)\xi(x)/m+O(\alpha^2(x)\xi^2(x))\quad\mbox{as }x\to\infty.
\end{eqnarray*}
By \eqref{E.beta.e.tail},
\begin{eqnarray*}
\E \{e^{\beta(x)\xi(x)};\ |\xi(x)|> 1/\alpha(x)\} &=& O(e^{-\varepsilon/2\alpha(x)})
\quad\mbox{as }x\to\infty.
\end{eqnarray*}
Altogether yields, by \eqref{conver.of.E},
\begin{eqnarray*}
\E e^{\beta(x)\xi(x)} &=& \E e^{\beta\xi(x)}
-\alpha(x)\E\xi(x)e^{\beta\xi(x)}/m+O(\alpha^2(x)+e^{-\varepsilon/2\alpha(x)})\\
&=& \E e^{\beta\xi(x)}-\alpha(x)
+O(\gamma(x)+\alpha^2(x)+e^{-\varepsilon/2\alpha(x)})\\
&=& 1+O(\gamma(x)+1/x^{1+2\varepsilon}+e^{-\varepsilon/2\alpha(x)})
\quad\mbox{as }x\to\infty.
\end{eqnarray*}
Thus, the function $\beta(x)$ satisfies all the conditions
of Theorem \ref{th:non.gen} and the proof is complete.
\qed\end{proof}

Notice that the key condition on the rate of convergence of
$\E e^{\beta\xi(x)}$ to $1$ that implies the asymptotics
\eqref{answer.1} in the last corollary is that the function
$\alpha^2(x)$ is integrable at infinity. If this condition fails, 
then the asymptotic behaviour of $\pi(x,\infty)$ is different 
from \eqref{answer.1} and requires higher moment
assumptions, which is specified in the following corollary.

\begin{corollary}\label{cor_power}
Assume the condition \eqref{Xi.eps} and that there exists
a differentiable function $\alpha(x)$ such that
\begin{eqnarray*}
|\alpha(x)| &=& O\Bigl(1/x^{\frac{1}{K+1}+\varepsilon}\Bigr)
\quad\mbox{as }x\to\infty,
\end{eqnarray*}
for some $K\in\N$ and $\varepsilon>0$,
\begin{equation}\label{cor:integr}
|\alpha'(x)|\le \gamma(x)
\end{equation}
and
$$
\E e^{\beta\xi(x)}-1=\alpha(x)+O(\gamma(x))
$$
for some decreasing integrable at infinity $\gamma(x)$.
Assume also that, for all $k=1$, $2$, \ldots, $K$,
\begin{equation}\label{cor.1}
M_k(x)=M_k+\sum_{j=1}^{M-k} D_{k,j}\alpha^j(x)+O(\gamma(x)/\alpha^k(x)),
\end{equation}
where $M_k(x):=\E\xi^k(x)e^{\beta\xi(x)}$ and $M_k:=\E\xi^ke^{\beta\xi}$.
Then there exist real numbers $c>0$ and $R_1$, $R_2$, \ldots, $R_K$ such that
\begin{eqnarray}\label{cor.2}
\pi(x,\infty) &\sim&
c\exp\biggl\{-\beta x+\sum_{k=1}^K R_k\int_0^x\alpha^k(y)dy\biggr\}
\quad\mbox{as }x\to\infty.
\end{eqnarray}
\end{corollary}

\begin{proof}
Define
$$
\Delta(x):=\sum_{k=1}^K R_k\alpha^k(x).
$$
In view of Theorem \ref{th:non.gen} it suffices to show that there exist
$R_1,R_2,\ldots,R_K$ such that
\begin{equation}
\left|\E e^{(\beta-\Delta(x))\xi(x)}-1\right|\le q(x)
\end{equation}
for some decreasing integrable function $q(x)$.
Indeed, $\Delta(x)$ is differentiable and $|\Delta'(x)|\le C|\alpha'(x)|$.
Therefore, we may apply Theorem \ref{th:non.gen}
with $\beta(x)=\beta-\Delta(x)$.

By Taylor's expansion, uniformly on the event $|\xi(x)|\le 1/\alpha(x)$,
\begin{eqnarray*}
e^{-\Delta(x)\xi(x)} &=& 
1+\sum_{k=1}^K \frac{(-\Delta(x))^k\xi^k(x)}{k!}+O(\Delta^{K+1}(x)\xi^{K+1}(x))\\
&=& 1+\sum_{k=1}^K \frac{(-\Delta(x))^k\xi^k(x)}{k!}+O(\alpha^{K+1}(x)\xi^{K+1}(x))
\quad\mbox{as }x\to\infty.
\end{eqnarray*}
By \eqref{E.beta.e.tail},
\begin{eqnarray*}
\E \{e^{\beta(x)\xi(x)};\ |\xi(x)|> 1/\alpha(x)\} &=& O(e^{-\varepsilon/2\alpha(x)})
\quad\mbox{as }x\to\infty.
\end{eqnarray*}
Therefore, as $x\to\infty$,
\begin{align*}
\E e^{(\beta-\Delta(x))\xi(x)}
&=\E e^{\beta\xi(x)}+\sum_{k=1}^K\frac{M_k(x)}{k!}(-\Delta(x))^k+
O(\alpha^{K+1}(x)+e^{-\varepsilon/2\alpha(x)})\\
&=1+\alpha(x)+\sum_{k=1}^K\frac{M_k(x)}{k!}(-\Delta(x))^k
+O(\gamma(x)+\alpha^{K+1}(x)+e^{-\varepsilon/2\alpha(x)}).
\end{align*}
So, we need to identify constants $R_1$, $R_2$, \ldots, $R_K$ such that 
\begin{align}\label{cor.3}
\alpha(x)+\sum_{k=1}^K\frac{M_k(x)}{k!}(-\Delta(x))^k
=O(\alpha^{K+1}(x)).
\end{align}
It follows from the assumption \eqref{cor.1} and the bound
$\Delta(x)=O(\alpha(x))$ that \eqref{cor.3} is equivalent to
$$
z+\sum_{k=1}^K\frac{1}{k!}\biggl(M_k+\sum_{j=1}^{K-k}D_{k,j}z^j\biggr)
\biggl(-\sum_{j=1}^K R_jz^j\biggr)^k=O(z^{K+1})
\quad\mbox{ as }z\to 0.
$$
Consequently, the coefficients of $z^k$ must be zero for all
$k\le  K$, and we can determine all $R_k$ recursively.
For example, the coefficient of $z$ equals $1-m_1R_1$.
Thus, $R_1=1/m_1$. Further, the coefficient
of $z^2$ is $-D_{1,1}R_1-m_1R_2+m_2R_1^2/2$ and, consequently,
$$
R_2=\frac{-D_{1,1}R_1+m_2R_1^2/2}{m_1}.
$$
All further coefficients can be found recursively.
\qed\end{proof}

If $\alpha(x)$ from Corollary \ref{cor_power}  decreases slower than
any power of $x$ but \eqref{cor:integr} and \eqref{cor.1} remain
valid, then one has, by the same arguments,
$$
\pi(x,\infty)=\exp\biggl\{-\beta x+\sum_{k=1}^K R_k\int_0^x\alpha^k(y)dy
+O\left(\int_0^x\alpha^{K+1}(y)dy\right)\biggr\}
$$
which can be seen as a corrected logarithmic asymptotic for $\pi$.
To obtain precise asymptotics one needs more information on the
moments $M_k(x)$.

\begin{corollary}\label{cor_full}
Assume the condition \eqref{Xi.eps} and that there exists
a differentiable function $\alpha(x)$ such that
\eqref{cor:integr} holds,
\begin{equation}\label{cor.0.full}
\E e^{\beta\xi(x)}-1=\alpha(x),\quad x\ge 0
\end{equation}
and
\begin{equation}\label{cor.1.full}
M_k(x)=M_k+\sum_{j=1}^\infty D_{k,j}\alpha^j(x)
\quad\mbox{for all }k\ge 1.
\end{equation}
Assume furthermore that
$$
\sup_{k\ge 1}\sum_{j=1}^\infty D_{k,j}r^j\ <\ \infty
\quad\mbox{for some }r>0.
$$
Then there exist real numbers $R_1$, $R_2$, \ldots, such that
$$
\pi(x,\infty)\ \sim\
c\,\exp\biggl\{-\beta x+\sum_{k=1}^\infty R_k\int_0^x\alpha^k(y)dy\biggr\}
\quad\mbox{as }x\to\infty.
$$
\end{corollary}

\begin{proof}
For all sufficiently large $x$ there is a positive solution
$\beta(x)$ to the equation
$$
\E e^{\beta(x)\xi(x)}\ =\ 1.
$$
Since $\E e^{\gamma\xi(x)}$ is finite for all $\gamma\le \beta+\varepsilon$,
we may rewrite the last equation as Taylor's series:
$$
\E e^{\beta\xi(x)}
+\sum_{k=1}^\infty\frac{(-\Delta(x))^k}{k!}\E\xi^k(x)e^{\beta\xi(x)}\ =\ 1,
$$
where $\Delta(x)=\beta-\beta(x)$. Taking into account \eqref{cor.0.full}
and \eqref{cor.1.full}, we then get
\begin{eqnarray}\label{cor.2.full}
\alpha(x)+\sum_{k=1}^\infty\frac{(-\Delta(x))^k}{k!}
\biggl(M_k+\sum_{j=1}^\infty D_{k,j}\alpha^j(x)\biggr) &=& 0.
\end{eqnarray}
Define
\begin{eqnarray*}
F(z,w) &:=& z+\sum_{k\ge 1}\frac{M_k}{k!}(-w)^k
+\sum_{k,j\ge 1}\frac{D_{k,j}}{k!}z^j(-w)^k.
\end{eqnarray*}
Therefore, \eqref{cor.2.full} can be written as $F(\alpha(x),\Delta(x))=0$.
In other words, we are looking for a function $w(z)$ satisfying
$F(z,w(z))=0$. Since $F(0,0)=0$ and $\frac{\partial}{\partial w}F(0,0)=-M_1<0$,
we may apply Theorem B.4 from Flajolet and Sedgewick \cite{FS}
which says that $w(z)$ is analytic in a vicinity of zero,
that is, there exists a $\rho>0$ such that
$$
w(z)\ =\ \sum_{n=1}^\infty R_nz^n,\quad |z|<\rho.
$$
Consequently,
$$
\Delta(x)\ =\ \sum_{n=1}^\infty R_n\alpha^n(x)
$$
for all $x$ such that $|\alpha(x)|<\rho$.

Applying Theorem \ref{th:non.gen} with $\beta(x)=\beta-\Delta(x)$, we get
$$
\pi(x,\infty)\ \sim\ ce^{-\beta x+\int_0^x\Delta(y)dy}
\quad\mbox{as }x\to\infty.
$$
Integrating $\Delta(y)$ term-wise, we complete the proof.
\qed\end{proof}

We finish with the following remark. In the proof of Corollary
\ref{cor_full} we have adapted the derivation of the Cram\'er series
in large deviations for sums of independent random variables,
see, e.g., Petrov \cite{Pet75}.
There is just one difference: we need analyticity of an implicit
function instead of analyticity of the inverse function.
%%%%%%%%%%%%%%%%%%%%%%%%%%%%%%%%%%%%%%%%%%%%%%%%%%%%%%%%%%%%%%%%%%%%%%%%%%%%%%%%%%%%

\section{Local central limit theorem}
\label{sec:a.h.clt}

We first state a version of the central limit theorem
for Markov chains on $\R$ with asymptotically constant drift.
\index{Markov chain!asymptotically homogeneous in space!central limit theorem}
\index{Central limit theorem!for Markov chains!asymptotically homogeneous in space}

\begin{theorem}\label{thm:ah.clt}
Let the family of jumps $|\xi(x)|$ possess a square integrable
majorant. Let $m_1(x)=\mu+o(1/\sqrt{x})$, $\mu>0$, let $m_2(x)\to b>0$
as $x\to\infty$, and let 
\begin{eqnarray}\label{eq:irreducibility.lln.loc}
\limsup_{n\to\infty}X_n &=& \infty \quad\mbox{with probability }1.
\end{eqnarray}
Then the strong law of large numbers holds
\begin{eqnarray}\label{slln.in.clt}
X_n/n &\stackrel{a.s.}{\to}& \mu\quad\mbox{as }n\to\infty.
\end{eqnarray}
Further,
\begin{eqnarray*}
\frac{X_n-\mu n}{\sqrt{bn}} &\Rightarrow& N_{0,1}\quad\mbox{as }n\to\infty
\end{eqnarray*}
and
\begin{eqnarray*}
\frac{\max_{k\le n}X_k-\mu n}{\sqrt{bn}} &\Rightarrow& N_{0,1}
\quad\mbox{as }n\to\infty.
\end{eqnarray*}
\end{theorem}

These statements are immediate from Corollary \ref{cor:lln.beta},
Theorems \ref{thm:clt.2} and \ref{thm:clt.max} respectively 
with $v(x)\equiv\mu$, so $\beta=0$.
In this special case there is a shorter proof based on the
characteristic functions method, see Korshunov\index{Korshunov}
\cite[Theorem 5]{Korshunov01}.
\index{Central limit theorem!for Markov chains!asymptotically homogeneous in space}

\begin{theorem}\label{thm:ah.clt.local}
Let the family of jumps $\{\xi(x),x\in\R\}$ possess
a stochastic square integrable minorant with positive mean
(so that the condition \eqref{eq:irreducibility.lln.loc} holds true)
and a square integrable stochastic majorant.
Assume weak convergence $\xi(x)\Rightarrow\xi$,
relation $m_1(x)=\mu+o(1/\sqrt{x})$
and upper bound $\P\{X_0<-x\}=o(1/\sqrt x)$ as $x\to\infty$.

If $\xi$ has a non-lattice distribution and, for all $A>0$,
\begin{eqnarray}\label{local.clt.1}
\sup_{|\lambda|\le A}\bigl|\E e^{i\lambda\xi(x)}-\E e^{i\lambda\xi}\bigr|
&=& o(1/x)\quad\mbox{as }x\to\infty,
\end{eqnarray}
then, for all $h>0$,
\begin{eqnarray*}
\sup_{x\in\R}\bigl|\sqrt{2\pi b n}\P\{X_n\in(x,x+h]\}-he^{-(x-n\mu)^2/2bn}\bigr|
&\to& 0\quad\mbox{as }n\to\infty.
\end{eqnarray*}

If $\xi$ is integer-valued and, $\Z$ is the minimal lattice for $\xi$ and
\begin{eqnarray}\label{local.clt.2}
\sup_{|\lambda|\le \pi}\bigl|\E e^{i\lambda\xi(x)}-\E e^{i\lambda\xi}\bigr|
&=& o(1/x)\quad\mbox{as }x\to\infty,
\end{eqnarray}
then
\begin{eqnarray*}
\sup_{x\in\Z}\bigl|\sqrt{2\pi b n}\P\{X_n=x\}-e^{-(x-n\mu)^2/2bn}\bigr|
&\to& 0\quad\mbox{as }n\to\infty.
\end{eqnarray*}
\end{theorem}

\begin{proof}
Let $\eta$ be a square integrable minorant with positive expectation
for the family $\{\xi(x),x\in\R\}$.
Let $\{\eta_k\}$ be independent copies of $\eta$
and set $S_n:=\eta_1+\ldots+\eta_n$. 
Then, by the minorisation assumption, for all $n$,
\begin{eqnarray*}
\lefteqn{\P\{X_k<k\E\eta/2\text{ for some }k\ge n\}}\\
&\le& \P\{X_0<-n\E\eta/4\}
+\P\{S_k-k\E\eta/4<k\E\eta/2\text{ for some }k\ge n\}\\
&\le& o(1/\sqrt n)
+\P\Bigl\{\sup_{k\ge n}\frac{S_k-k\E\eta}{k}<-\E\eta/4\Bigr\}.
\end{eqnarray*}
The sequence $(S_k-k\E\eta)/k$ constitutes a reverse martingale and
hence it follows from the Kolmogorov inequality that
\begin{eqnarray*}
\P\Bigl\{\sup_{k\ge n}\frac{S_k-k\E\eta}{k}<-\E\eta/4\Bigr\}
&\le& \frac{\E(S_n-n\E\eta)^2}{(\E\eta/4)^2n^2}
\ =\ O(1/n)\ =\ o(1/\sqrt n).
\end{eqnarray*}
Therefore,
\begin{eqnarray}\label{Kolm.1}
\P\{X_k<k\E\eta/2\text{ for some }k\ge n\}
&\le& o(1/\sqrt n)\quad\mbox{as }n\to\infty.
\end{eqnarray}

We proceed with the proof for the lattice case only,
the non-lattice case can be treated similarly.
By the inversion formula for lattice distributions,
\begin{eqnarray*}
\sqrt{n}\P\{X_n=x\}=\frac{1}{2\pi}\int_{-\pi\sqrt{n}}^{\pi\sqrt{n}}
e^{-i\lambda\frac{x-n\mu}{\sqrt{n}}} \E e^{i\lambda\frac{X_n-n\mu}{\sqrt{n}}}d\lambda.
\end{eqnarray*}
Therefore, using standard arguments,
\begin{eqnarray}\label{local.clt}
\lefteqn{\sup_x\left|\sqrt{n}\P\{X_n=x\}
-\frac{1}{\sqrt{2\pi b}}e^{-(x-n\mu)^2/2b}\right|}\nonumber\\
&&\hspace{8mm}\le\ \frac{1}{2\pi}\int_{-A}^A\bigl|\E e^{i\lambda\frac{X_n-n\mu}{\sqrt{n}}}-e^{-\lambda^2b/2}\bigr|d\lambda
\nonumber\\
&&\hspace{25mm}+\int_{|\lambda|\in(A,\pi\sqrt{n}]}\bigl|\E e^{i\lambda\frac{X_n-n\mu}{\sqrt{n}}}\bigr|d\lambda
+\int_{|\lambda|>A}e^{-\lambda^2b/2}d\lambda.\hspace{10mm}
\end{eqnarray}
It follows from the weak convergence to the normal law that
$\E e^{i\lambda\frac{X_n-n\mu}{\sqrt{n}}}\to e^{-\lambda^2b/2}$ uniformly
on compact $\lambda$-sets. Therefore, the first integral on the right hand side
of \eqref{local.clt} converges to zero as $n\to\infty$, for any fixed $A$.
Choosing $A$ sufficiently large we can make the integral
$\int_{|\lambda|>A}e^{-\lambda^2b/2}d\lambda$ as small as we please.
Thus, it remains to prove that the second integral
in \eqref{local.clt} is small too.

In order to prove this we need to show that the modulus of the
characteristic function in the second integral is sufficiently small.
Let us introduce an auxiliary time-inhomogeneous Markov chain
$\{\widetilde X_n\}$ with jumps at time $n$
$$
\widetilde\xi_n(x)\ =\ \left\{
\begin{array}{ll}
\xi(x) &\mbox{for } x> n\E\eta/4,\\
\xi(x_n) &\mbox{for }x\le n\E\eta/4,
\end{array}
\right.
$$
where $x_n>n\E\eta/4$.
Consider for simplicity even $n$ and define $\widetilde X_k$
for $k\ge n/2$ only; set $\widetilde X_{n/2}=X_{n/2}$.
Then it follows from the construction that, for all $u\in\R$,
\begin{eqnarray*}
\bigl|\E e^{iu X_n}\bigr| &\le& \bigl|\E e^{iu\widetilde X_n}\bigr|
+\P\{X_k\not= \widetilde X_k\mbox{ for some }k>n/2\}\\
&\le& \bigl|\E e^{iu\widetilde X_n}\bigr|
+\P\{X_k\le k\E\eta/4\mbox{ for some }k\ge n/2\}.
\end{eqnarray*}
From this estimate and \eqref{Kolm.1} we obtain
\begin{eqnarray}\label{local.clt.4}
\bigl|\E e^{iu X_n}\bigr| &\le&
\bigl|\E e^{iu\widetilde X_n}\bigr|+o(1/\sqrt n)
\quad\mbox{as }n\to\infty\mbox{ uniformly for all }u\in\R.
\end{eqnarray}
By the construction of $\{\widetilde X_k\}$, we have
\begin{eqnarray*}
\bigl|\E e^{iu\widetilde X_{k+1}}-\E e^{iu\xi}\E e^{iu\widetilde X_k}\bigr|
&=& \bigl|\E e^{iu\widetilde X_k}
\bigl(\E\{e^{iu\widetilde\xi_k(\widetilde X_k)}|\widetilde{X}_k\}
-\E e^{iu\xi}\bigr)\bigr|\\
&\le& \bigl|\E\{e^{iu\widetilde\xi_k(\widetilde X_k)}|\widetilde{X}_k\}
-\E e^{iu\xi}\bigr|\\
&\le& \sup_{x>k\E\eta/4} |\E e^{iu\xi(x)}-\E e^{iu\xi}|.
\end{eqnarray*}
Then, for all $k\ge n/2$,
\begin{eqnarray*}
\bigl|\E e^{iu\widetilde X_{k+1}}-\E e^{iu\xi}\E e^{iu\widetilde X_k}\bigr|
&\le& \sup_{x>n\E\eta/8} |\E e^{iu\xi(x)}-\E e^{iu\xi}|\\
&\le& \sup_{x>n\E\eta/8} \sup_{|u|\le\pi}|\E e^{iu\xi(x)}-\E e^{iu\xi}|\\
&=:& \delta_n\ =\ o(1/n),
\end{eqnarray*}
by the assumption of the theorem. Consequently, for $m=n/2$ we have
\begin{eqnarray*}
\bigl|\E e^{iu\widetilde{X}_n}-(\E e^{iu\xi})^{n-m}\E e^{iu\widetilde X_m}\bigr|
&=& \Bigl|\sum_{k=m}^{n-1}(\E e^{iu\xi})^{n-k-1}
\bigl(\E e^{iu\widetilde X_{k+1}}-\E e^{iu\xi}\E e^{iu\widetilde X_k}\bigr)\Bigr|\\
&\le& \delta_n\sum_{j=0}^{n-m-1}\bigl|\E e^{iu\xi}\bigr|^j,
\end{eqnarray*}
so hence
\begin{eqnarray*}
\bigl|\E e^{iu\widetilde{X}_n}\bigr|
&\le& \bigl|\E e^{iu\xi}\bigr|^{n/2}
+ \delta_n\sum_{j=0}^{n/2-1}\bigl|\E e^{iu\xi}\bigr|^j.
\end{eqnarray*}
Since $\Z$ is the minimal lattice for $\xi$, there exists an $\varepsilon>0$
such that $|\E e^{iu\xi}|\le e^{-\varepsilon u^2}$ for all $u\in[-\pi,\pi]$.
This implies that
\begin{eqnarray*}
\bigl|\E e^{iu\widetilde{X}_n}\bigr| &\le&
e^{-n\varepsilon u^2/2}+\delta_n\sum_{j=0}^{n/2-1}e^{-j\varepsilon u^2}.
\end{eqnarray*}
Substituting this into \eqref{local.clt.4} we obtain, 
uniformly for all $u\in\R$,
\begin{eqnarray*}
\bigl|\E e^{iu X_n}\bigr| &\le&
e^{-n\varepsilon u^2/2}+o(1/\sqrt n)
+o(1/n)\sum_{j=1}^{n/2}e^{-j\varepsilon u^2}
\quad\mbox{as }n\to\infty.
\end{eqnarray*}
Hence the second term in \eqref{local.clt}
possesses the following upper bound:
\begin{eqnarray*}
\lefteqn{\int_{|\lambda|\in(A,\pi\sqrt{n}]}
\bigl|\E e^{i\lambda\frac{X_n-n\mu}{\sqrt{n}}}\bigr|d\lambda}\\
&&\le\ \int_{|\lambda|\in(A,\pi\sqrt{n}]} e^{-\varepsilon\lambda^2/2}d\lambda
+o(1)+o(1/n)\sum_{j=1}^{n/2}\int_{|\lambda|\in(A,\pi\sqrt{n}]}
e^{-\varepsilon\lambda^2 j/n}d\lambda\\
&&\le\ 2\int_A^\infty e^{-\varepsilon\lambda^2/2}d\lambda
+o(1)+o(1/n)\sum_{j=1}^{n/2} \int_{-\infty}^\infty e^{-\varepsilon\lambda^2 j/n}d\lambda\\
&&=\ 2\int_A^\infty e^{-\varepsilon\lambda^2/2}d\lambda
+o(1)+o(1/n)\sum_{j=1}^{n/2} \sqrt{\frac{\pi n}{\varepsilon j}}
\quad\mbox{as }n\to\infty.
\end{eqnarray*}
Therefore,
$$
\limsup_{n\to\infty}\int_{|\lambda|\in(A,\pi\sqrt{n}]}
\bigl|\E e^{i\lambda\frac{X_n-n\mu}{\sqrt{n}}}\bigr|d\lambda
\le 2\int_A^\infty e^{-\varepsilon \lambda^2/2}d\lambda.
$$
Letting $A\to\infty$, we conclude the desired result.
\qed\end{proof}

\index{Markov chain!asymptotically homogeneous in space!renewal theorem}

\begin{theorem}\label{thm:ah.part.renewal}
Assume that all the conditions of Theorem~\ref{thm:ah.clt.local} hold.

If $\xi$ is a non-lattice random variable then, for all $h>0$,
$$
\sum_{k=0}^n\P\{X_k\in(x,x+h]\}\ =\
\frac{h}{\mu}\Phi\Bigl(\frac{n\mu-x}{\sqrt{xb/\mu}}\Bigr)+o(1)
$$
as $x\to\infty$ uniformly for all $n\ge 0$.

If $\xi$ is a lattice random variable
and $\Z$ is the minimal lattice for $\xi$ then
$$
\sum_{k=0}^n\P\{X_k=x\}=\frac{1}{\mu}
\Phi\Bigl(\frac{n\mu-x}{\sqrt{xb/\mu}}\Bigr)+o(1)\quad\mbox{as }x\to\infty,
$$
as $x\to\infty$ uniformly for all $n\ge 0$.
\end{theorem}

\begin{proof}
We again consider the lattice case only.
By the local limit theorem, for any fixed $A$, $B\in\R$, $A<B$,
\begin{eqnarray*}
\sum_{k=x/\mu+A\sqrt{x}}^{x/\mu+B\sqrt{x}}\P\{X_k=x\}
&=& \sum_{k=x/\mu+A\sqrt{x}}^{x/\mu+B\sqrt{x}}
\frac{1}{\sqrt{2\pi b k}}e^{-(x-\mu k)^2/2bk} +o(1)\\
&=& \sum_{k=x/\mu+A\sqrt{x}}^{x/\mu+B\sqrt{x}}
\frac{1}{\sqrt{2\pi b x/\mu}}e^{-(x-\mu k)^2\mu/2bx} +o(1).
\end{eqnarray*}
Thus, as $x\to\infty$,
\begin{eqnarray}\label{ah.part.1}
\nonumber
\sum_{k=x/\mu+A\sqrt{x}}^{x/\mu+B\sqrt{x}}\P\{X_k=x\}
&=& \sum_{k=A\sqrt x}^{B\sqrt x}
\frac{1}{\sqrt{2\pi xb/\mu}}e^{-(k/\sqrt x)^2\mu^3/2b} +o(1)\nonumber\\
&=& \int_A^B\frac{1}{\sqrt{2\pi b/\mu}}e^{-y^2\mu^3/2b}dy +o(1)\nonumber\\
&=& \frac{1}{\mu}\bigl(\Phi(\mu^{3/2}B/\sqrt b)-\Phi(\mu^{3/2}A/\sqrt b)\bigr)+o(1).
\hspace{10mm}
\end{eqnarray}
Together with Theorem~\ref{thm:ah.renewal} it implies that, 
for any $\varepsilon>0$, there exist $A$ and $B$ such that,
for all sufficiently large $x$,
\begin{eqnarray*}
\sum_{k=0}^{x/\mu+A\sqrt{x}}\P\{X_k=x\}+\sum_{k=x/\mu+B\sqrt{x}}^\infty\P\{X_k=x\}
&\le& \varepsilon.
\end{eqnarray*}
Therefore,
\begin{eqnarray*}
\sum_{k=0}^{x/\mu+A\sqrt x}\P\{X_k{=}x\} &\to& 0
\end{eqnarray*}
as $A\to-\infty$ uniformly for all $x$.
Combining this with \eqref{ah.part.1}, we get the desired relation.
\qed\end{proof}
%%%%%%%%%%%%%%%%%%%%%%%%%%%%%%%%%%%%%%%%%%%%%%

\section{Pre-stationary distributions}
\label{sec:ah.pre-limit}

\index{Markov chain!asymptotically homogeneous in space!pre-stationary distribution}

\begin{theorem}\label{thm:ah.pre-st}
Let the distribution of $X_n$ converge towards a stationary distribution $\pi$
in the total variation norm.
Assume that the conditions of Theorem~\ref{th:conv} are valid
and that the majorant $\Xi$ satisfies also the condition
\begin{eqnarray}\label{ah.pre-st.1}
\E\Xi^2e^{\beta\Xi} &<& \infty.
\end{eqnarray}
Assume also that
\begin{eqnarray}\label{ah.pre-st.2}
\E\xi(x)e^{\beta\xi(x)} &=& \E\xi e^{\beta\xi}+o(1/\sqrt{x})
\quad\mbox{as }x\to\infty.
\end{eqnarray}
If the limiting variable $\xi$ is non-lattice we assume that,
for any $A>0$,
\begin{eqnarray}\label{ah.pre-st.3}
\sup_{|\lambda|\le A}\bigl|\E e^{(\beta+i\lambda)\xi(x)}
-\E e^{(\beta+i\lambda)\xi}\bigr| &=& o(1/x)\quad\mbox{as }x\to\infty.
\end{eqnarray}
If $\xi$ is a lattice distribution and
$\Z$ is the minimal lattice for $\xi$ we assume that
\begin{eqnarray}\label{ah.pre-st.4}
\sup_{|\lambda|\le \pi}\bigl|\E e^{(\beta+i\lambda)\xi(x)}
-\E e^{(\beta+i\lambda)\xi}\bigr| &=& o(1/x)\quad\mbox{as }x\to\infty.
\end{eqnarray}
Then, uniformly for all $n\ge1$,
\begin{eqnarray}\label{Cramer.ans}
\frac{\P\{X_n>x\}}{\pi(x,\infty)} &=&
\Phi_{\sigma^2}\biggl(\frac{n\E\xi e^{\beta\xi}-x}
{\sqrt{x/\E\xi e^{\beta\xi}}}\biggr)+o(1)
\quad\text{as }x\to\infty,
\end{eqnarray}
where $\sigma^2=\E\xi^2e^{\beta\xi}-(\E\xi e^{\beta\xi})^2$.
\end{theorem}

\begin{proof}
Let $\{\widehat X_n\}$ be the Markov chain constructed in the proof
of Theorem~\ref{th:conv}.
We have shown there that the family $\widehat\xi(x)$ possesses
a stochastic minorant with positive mean and finite second moment
and a stochastic majorant with finite mean.
Assumption \eqref{ah.pre-st.1}
implies that there is a majorant with finite second moment.

We now turn to the asymptotic behaviour of $\E\widehat{\xi}(x)$.
As we have shown in the proof of Theorem~\ref{th:conv},
$\E\widehat\xi(x)\to\E\xi e^{\beta\xi}$. But, in order to apply
Theorem~\ref{thm:ah.part.renewal}, we have to show that
\begin{eqnarray}\label{ah.pre-st.5}
\E\widehat\xi(x) &=& \E\xi e^{\beta\xi}+o(1/\sqrt{x})
\quad\text{as }x\to\infty.
\end{eqnarray}
It follows from \eqref{conv.10} that
\begin{eqnarray}\label{ah.pre-st.6}
\E\widehat{\xi}(x) &=& \frac{\E\xi(x)U_p(x+\xi(x))}{U_p(x)}(1+o(1/x))
\quad\text{as }x\to\infty.
\end{eqnarray}
It is immediate from the definition \eqref{def.U.g.p} of $U_p$ that
\begin{eqnarray*}
\E\{\xi(x)U_p(x+\xi(x));\ \xi(x)>s(x)\}
&\le& U_p(x)\E\{\xi(x)e^{\beta\xi(x)};\ \xi(x)>s(x)\}.
\end{eqnarray*}
Thus, due to \eqref{ah.pre-st.1}, for any $s(x)=o(x)$,
\begin{eqnarray}\label{ah.pre-st.7}
\frac{\E\{\xi(x)U_p(x+\xi(x));\ \xi(x)>s(x)\}}{U_p(x)} &=& o(1/s(x))
\quad\text{as }x\to\infty.
\end{eqnarray}
Furthermore, we have an upper bound
\begin{eqnarray}\label{ah.pre-st.8}
\frac{\E\{\xi(x)U_p(x+\xi(x));\ \xi(x)<-s(x)\}}{U_p(x)}
&=& o(e^{-\beta s(x)/2})\quad\text{as }x\to\infty.\hspace{5mm}
\end{eqnarray}
Uniformly on the set $\{|\xi(x)|\le s(x)\}$ we have
$g(x+\xi(x))-g(x)\sim-p(x)\xi(x)$, see \eqref{def.g.p}. Therefore,
\begin{eqnarray*}
\lefteqn{\E\{\xi(x)U_p(x+\xi(x));\ |\xi(x)|<s(x)\}}\\
&&=\ e^{\beta x}\E\{\xi(x)(1+g(x+\xi(x)))e^{\beta\xi(x)};
\ |\xi(x)|\le s(x)\}\\
&&=\ U_p(x)\E\{\xi(x)e^{\beta\xi(x)};\ |\xi(x)|\le s(x)\}\\
&&\hspace{30mm}
-p(x)(1+o(1))e^{\beta x}\E\{\xi^2(x)e^{\beta\xi(x)};\ |\xi(x)|\le s(x)\}.
\end{eqnarray*}
Using again \eqref{ah.pre-st.1}, we obtain
\begin{eqnarray*}
\frac{\E\{\xi(x)U_p(x+\xi(x));\ |\xi(x)|<s(x)\}}{U_p(x)}
&=& \E\xi(x)e^{\beta\xi(x)}+O(p(x)+1/s(x)).
\end{eqnarray*}
Combining this estimate with \eqref{ah.pre-st.7} and \eqref{ah.pre-st.8},
and choosing $s(x)$ such that $s(x)/\sqrt x\to\infty$, we conclude that
$$
\frac{\E\{\xi(x)U_p(x+\xi(x))\}}{U_p(x)}
\ =\ \E\xi(x)e^{\beta\xi(x)}+o(1/\sqrt{x}).
$$
The relation \eqref{ah.pre-st.5} follows now from the assumption \eqref{ah.pre-st.2}.
The same arguments show that \eqref{local.clt.1} and \eqref{local.clt.2}
follow from \eqref{ah.pre-st.3} and \eqref{ah.pre-st.4} respectively.
Thus, $\{\widehat X_n\}$ satisfies all the conditions
of Theorem~\ref{thm:ah.part.renewal}.

It follows from the conditions on jumps that $\E e^{\beta X_n}<\infty$
for all $n$ which implies $\P\{X_n>x\}=o(e^{-\beta x})$ for any fixed
$n$ and hence \eqref{Cramer.ans}. 
So it remains to consider the case where $n\to\infty$.

Fix an $h>0$.
Applying \eqref{repr.Xn.B.U.x} with $U=U_p$ we deduce that, for $x>\widehat x$,
\begin{eqnarray*}
\lefteqn{\P\{X_n\in(x,x+h]\}}\\
&=& \sum_{j=1}^n\int_B\P\{X_{n-j}\in dz\}\int_{\widehat x}^\infty
P(z,du)U_p(u) \E_u\biggl\{\frac{e^{-\sum_{k=0}^{j-2}q(\widehat X_k)}}
{U_p(\widehat X_{j-1})};\ \widehat X_{j-1}\in(x,x+h]\biggr\}.
\end{eqnarray*}
By the conditions \eqref{Xi} and \eqref{ah.pre-st.1},
\begin{eqnarray}\label{P.upper.Xi}
P(z,(u,\infty))\ \le\ \P\{\Xi>u-\widehat x\}\ \le\ c_2e^{-\beta u}/u^2
\quad\mbox{for all }z\le\widehat x\mbox{ and }u>\widehat x.\hspace{10mm}
\end{eqnarray}
The function $U_p$ is increasing. Hence, for any $N_n=o(\sqrt n)$,  
\begin{eqnarray}\label{ah.pre-st.9}
\nonumber
\lefteqn{\sum_{j=n-N_n+1}^n\int_B\P\{X_{n-j}\in dz\}
\int_{\widehat x}^\infty P(z,du)U_p(u)
\E_u\biggl\{\frac{e^{-\sum_{k=0}^{j-2}q(\widehat X_k)}}{U_p(\widehat X_{j-1})};
\ \widehat X_{j-1}\in(x,x+h]\biggr\}}\\
\nonumber
&&\hspace{1cm}\le\ \frac{1}{U_p(x)}\sum_{j=n-N_n+1}^n
\int_B\P\{X_{n-j}\in dz\}\int_{\widehat x}^\infty P(z,du) U_p(u)
\P_u\{\widehat X_{j-1}\in(x,x+h]\}\\
\nonumber
&&\hspace{1cm}\le\ \frac{c_3}{U_p(x)\sqrt n}\sum_{j=n-N_n+1}^n
\int_B\P\{X_{n-j}\in dz\}\int_{\widehat x}^\infty P(z,du) U_p(u)\\
&&\hspace{1cm}\le\ c_4N_n/U_p(x)\sqrt n
\ =\ o(1/U_p(x)) \quad\mbox{as }n\to\infty,
\end{eqnarray}
where the second inequality follows by Theorem~\ref{thm:ah.clt.local} 
applied to $\{\widehat X_n\}$.
Since the distribution of $X_{n-j}$ converges in total variation to $\pi$,
for any $N_n\to\infty$,
\begin{eqnarray}\label{ah.pre-st.10}
\nonumber
\lefteqn{\sum_{j=1}^{n-N_n}\int_B\P\{X_{n-j}\in dz\}\int_{\widehat x}^\infty
P(z,du)U_p(u) \E_u\biggl\{\frac{e^{-\sum_{k=0}^{j-2}q(\widehat X_k)}}
{U_p(\widehat X_{j-1})};\ \widehat X_{j-1}\in(x,x+h]\biggr\}}\\
&=& (1+o(1))\sum_{j=1}^{n-N_n}\int_B\pi(dz)\int_{\widehat x}^\infty
P(z,du)U_p(u) \E_u\biggl\{\frac{e^{-\sum_{k=0}^{j-2}q(\widehat X_k)}}
{U_p(\widehat X_{j-1})};\ \widehat X_{j-1}\in(x,x+h]\biggr\}.
\nonumber\\[-1mm]
\end{eqnarray}
Similarly to \eqref{ah.pre-st.9},
\begin{eqnarray*}
\sum_{j=n-N_n+1}^n\int_B\pi(dz)\int_{\widehat x}^\infty
P(z,du)U_p(u) \E_u\biggl\{\frac{e^{-\sum_{k=0}^{j-2}q(\widehat X_k)}}
{U_p(\widehat X_{j-1})};\ \widehat X_{j-1}\in(x,x+h]\biggr\}
&=& o\Bigl(\frac{1}{U_p(x)}\Bigr).
\end{eqnarray*}
Combining this with \eqref{ah.pre-st.9} and \eqref{ah.pre-st.10}, we obtain
\begin{eqnarray*}
\lefteqn{\P\{X_n\in(x,x+h]\}}\\ 
&=& (1+o(1))\sum_{j=1}^n\int_B\pi(dz)\int_{\widehat x}^\infty
P(z,du)U_p(u) \E_u\biggl\{\frac{e^{-\sum_{k=0}^{j-2}q(\widehat X_k)}}
{U_p(\widehat X_{j-1})};\ \widehat X_{j-1}\in(x,x+h]\biggr\} 
+o\Bigl(\frac{1}{U_p(x)}\Bigr)\\ 
&=& (1+o(1))\int_{\widehat x}^\infty \mu(du)U_p(u) 
\sum_{j=1}^n \E_u\biggl\{\frac{e^{-\sum_{k=0}^{j-2}q(\widehat X_k)}}
{U_p(\widehat X_{j-1})};\ \widehat X_{j-1}\in(x,x+h]\biggr\} 
+o\Bigl(\frac{1}{U_p(x)}\Bigr)
\end{eqnarray*}
as $x\to\infty$ where
\begin{eqnarray*}
\mu(du) &=& \int_B\pi(dz)P(z,du)
\end{eqnarray*}
is a measure on $(\widehat x,\infty)$, see \eqref{6.mu.B}. 
Therefore, as $x\to\infty$,
\begin{eqnarray*}
\P\{X_n\in(x,x+h]\} &=& (1+o(1))\int_{\widehat x}^\infty \mu(du)U_p(u) 
\int_x^{x+h} e^{-\beta y} \widehat H_{u,n}^{(q)}(dy) 
+o(e^{-\beta x})
\end{eqnarray*}
where
$$
\widehat H_{u,n}^{(q)}(dy)\ =\
\sum_{j=1}^n\E_u\Bigl\{e^{-\sum_{k=0}^{j-2}q(\widehat X_k)};\ 
\widehat X_{j-1}\in dy\Bigr\}.
$$
In the non-lattice case, due to Lemma \ref{thm:renewal.2}, 
for any fixed $\Delta>0$,
$$
\widehat H_{u,n}^{(q)}(y,y+\Delta]\ \sim\
\E_u e^{-\sum_{k=0}^\infty q(\widehat X_k)}
\sum_{j=1}^n\P_u\{\widehat X_{j-1}\in(y,y+\Delta\}\quad\mbox{as }y\to\infty,
$$
hence
\begin{eqnarray*}
\lefteqn{\P\{X_n\in(x,x+h]\}}\\ 
&=& (1+o(1))\int_{\widehat x}^\infty
\mu(du)U_p(u) \E_u e^{-\sum_{k=0}^\infty q(\widehat X_k)}
\int_x^{x+h} e^{-\beta y} \widehat H_{u,n}(dy) 
+o(e^{-\beta x}),
\end{eqnarray*}
where the partial renewal measure of $\{\widehat X_n\}$,
$$
\widehat H_{u,n}(dy)\ =\ \sum_{j=1}^n\P_u\{\widehat X_{j-1}\in dy\},
$$
is asymptotically Lebesgue on the interval $[x,x+h]$, with coefficient 
$\frac{1}{\mu}\Phi_{\sigma^2}\Bigl(\frac{n\mu-x}{\sqrt{x/\mu}}\Bigr)$,
for any fixed $u>\widehat x$,
by Theorem \ref{thm:ah.part.renewal}; here $\mu:=\E\xi e^{\beta\xi}$.
Then, for any fixed $u>\widehat x$,
\begin{eqnarray*}
\int_x^{x+h} e^{-\beta y} \widehat H_{u,n}(dy)
&=& \frac{1}{\mu}\Phi_{\sigma^2}
\Bigl(\frac{n\mu-x}{\sqrt{x/\mu}}\Bigr)
\frac{1-e^{-\beta h}}{\beta} e^{-\beta x}+o(e^{-\beta x})
\quad\mbox{as }x\to\infty.
\end{eqnarray*}
Secondly,
\begin{eqnarray*}
\int_x^{x+h} e^{-\beta y} \widehat H_{u,n}(dy)
&\le& e^{-\beta x}\widehat H_{u,n}(x,x+h]
\ \le\ c_5e^{-\beta x},
\end{eqnarray*}
hence the dominated convergence theorem is applicable
owing to \eqref{P.upper.Xi}, so
\begin{eqnarray*}
\lefteqn{\P\{X_n>x\}}\\
&=& \frac{1}{\beta\mu}\Phi_{\sigma^2}
\Bigl(\frac{n\mu-x}{\sqrt{x/\mu}}\Bigr)
\int_{\widehat x}^\infty
\mu(du)U_p(u) \E_u e^{-\sum_{k=0}^\infty q(\widehat X_k)}
e^{-\beta x}+o(e^{-\beta x})\quad\mbox{as }x\to\infty.
\end{eqnarray*}
Together with Theorem~\ref{th:conv} 
that yields the required result \eqref{Cramer.ans}.

The lattice case can be concluded in a similar way.
\qed\end{proof}

We can determine the asymptotic behaviour of pre-stationary distributions
also in the case when \eqref{conver} fails.
\index{Markov chain!asymptotically homogeneous in space!pre-stationary distribution}

\begin{theorem}\label{thm:ah.pre-st.non}
Assume that the conditions of Theorem~\ref{th:non.gen} are valid.
Assume also that
\begin{eqnarray*}
%\label{ah.pre-st.2}
\E\xi(x)e^{\beta(x)\xi(x)}=\E\xi e^{\beta\xi}+o(1/\sqrt{x}).
\end{eqnarray*}
If the limiting variable $\xi$ is non-lattice we assume that,
for any $A>0$,
\begin{eqnarray*}
%\label{ah.pre-st.3}
\sup_{|\lambda|\le A}\bigl|\E e^{(\beta(x)+i\lambda)\xi(x)}
-\E e^{(\beta+i\lambda)\xi}\bigr| &=& o(1/x).
\end{eqnarray*}
If $\Z$ is the minimal lattice for $\xi$ we assume that
\begin{eqnarray*}
%\label{ah.pre-st.4}
\sup_{|\lambda|\le\pi}\bigl|\E e^{(\beta(x)+i\lambda)\xi(x)}
-\E e^{(\beta+i\lambda)\xi}\bigr| &=& o(1/x).
\end{eqnarray*}
Then, uniformly for all $n\ge1$,
$$
\frac{\P\{X_n>x\}}{\pi(x,\infty)}\ =\
\Phi_{\sigma^2}\biggl(\frac{n\E\xi e^{\beta\xi}-x}
{\sqrt{x/\E\xi e^{\beta\xi}}}\biggr)+o(1)
\quad\text{as }x\to\infty,
$$
where $\sigma^2=\E\xi^2e^{\beta\xi}-(\E\xi e^{\beta\xi})^2$.
\end{theorem}

The proof of this theorem is identical to that of
Theorem~\ref{thm:ah.pre-st} and for that reason we omit it.

\section{Comments to Chapter \ref{ch:asymp.hom}}

Theorem \ref{thm:ah.renewal} specifies Theorem 1 from
Korshunov\index{Korshunov} \cite{Kor08} for transient Markov chains on $\R$.

Borovkov\index{Borovkov} and Korshunov\index{Korshunov} 
\cite{BK1}, \cite[Sect. 27]{B1998}
proved exponential asymptotics for $\pi$ under the condition
\begin{eqnarray}\label{BK.cond}
\int_0^\infty dx\int_{-\infty}^\infty
e^{\beta y}\left|\P\{\xi(x)<y\}-\P\{\xi<y\}\right|dy &<& \infty,
\end{eqnarray}
without assuming a domination condition like \eqref{Xi}.

On the other hand, it is worth mentioning that \eqref{conver}
is weaker than conditions we found in the literature.
Firstly, \eqref{BK.cond} is definitely stronger than \eqref{conver}
and implies, in particular, that also the expectations
$\E\xi(x)e^{\beta\xi(x)}$ converge at summable rate.
Furthermore, to show that the constant $c$ in front of $e^{-\beta x}$
is positive the following condition is introduced in \cite{BK1}:
$$
\int_{0}^\infty\bigl(\E e^{\beta\xi(x)}-1\bigr)^-x\log xdx<\infty.
$$
Secondly, for chains on $\Zp$ Foley\index{Foley} and 
McDonald\index{McDonald} \cite{FD} used an assumption,
which can be rewritten in our notation as follows
$$
\sum_{i=0}^\infty\sum_{j\in\Z}e^{\beta j}|\P\{\xi(i)=j\}-\P\{\xi=j\}|<\infty.
$$

Theorems \ref{th:conv} and \ref{thm:ah.pre-st} were proven first time
in \cite{Kor04} via so-called evolution of masses, that is,
via analysis of non-stochastic kernels.

A lattice version of Theorem \ref{th:non.gen} was proven 
by Denisov\index{Denisov} et al. \cite{DKW2013+} 
following a different approach based on some useful method 
of construction of harmonic functions for Markov kernels on $\Zp$.
\chapter{Applications}
\label{ch:applications}

The main goal of this chapter is to demonstrate how the theory 
developed in the previous chapters can be useful 
for the study of various Markov models 
that give rise to Markov chains with asymptotically zero drift. 
Some of that models are quite popular in stochastic modelling: 
random walks conditioned to stay positive, state-dependent branching processes 
or branching processes with migration, stochastic difference equations. 
In contrast to the general approach discussed here, 
the methods available in the literature for investigation of these models 
are mostly model tailored.

We also introduce some new models, where our approach is applicable. 
For example, in Section \ref{sec:risk} we introduce a risk process with surplus-dependent 
premium rate, which converges to the critical threshold in the netto profit condition. 
Furthermore, we introduce a new class of branching processes 
with migration and with state-dependent offspring distributions.

\section{Random walk conditioned to stay positive}
\label{sec:h.x.cond.walks}

Let $\{S_n\}$ be a random walk with independent identically distributed
increments $\xi_k$, that is, $S_n=\xi_1+\xi_2+\ldots+\xi_n$, $n\ge 1$.
Let $\tau(x)$ be the first time epoch when $\{S_n\}$ starting at $x$
is non-positive:
$$
\tau(x):=\min\{n\ge1:x+S_n\le 0\}.
$$
We shall assume that the random walk $\{S_n\}$ is oscillating, that is,
$$
\liminf_{n\to\infty} X_n=-\infty,\quad
\limsup_{n\to\infty} X_n=\infty\quad\mbox{with probability 1.}
$$
In particular, $\P\{\tau(x)<\infty\}=1$ for all starting points $x$.
Let $\chi^-$ denote the first weak descending ladder height of $\{S_n\}$,
that is, $\chi^-=-S_{\tau(0)}$. Let $V(x)$ denote the renewal function 
generated by the weak descending ladder heights of the random 
walk:\index{Random walk!killed at leaving $(0,\infty)$!harmonic function}
\begin{eqnarray}\label{V.rf.theta}
V(x) &:=& 1+\sum_{k=1}^\infty\P\{\chi^-_1+\chi^-_2+\ldots+\chi^-_k<x\}\nonumber\\
&=& \E\theta(x),
\end{eqnarray}
where $\chi^-_k$ are independent copies of $\chi^-$ and
$\theta(x):=\min\{k:\chi^-_1+\chi^-_2+\ldots+\chi^-_k\ge x\}$.
In particular, $V(0)=1$.

It is well-known---see e.g. Kozlov\index{Kozlov} \cite{Kozlov}---that
$V(x)$ is a harmonic function for $\{S_n\}$ killed at leaving
$(0,\infty)$. More precisely,
$$
V(x)=\E\{V(x+S_1); \tau(x)>1\}\quad\mbox{for all }x\ge 0.
$$
This implies that Doob's $h$-transform
\begin{eqnarray}\label{def.cond}
P(x,dy) &:=& \frac{V(y)}{V(x)}\P\{x+S_1\in dy,\tau(x)>1\}
\end{eqnarray}
defines a stochastic transition kernel on $\R^+$.
Let $\{X_n\}$ be the corresponding Markov chain. It is usually called
{\it the random walk conditioned to stay positive}.
\index{Random walk!conditioned to stay positive}
This definition via Doob's $h$-transform is equivalent to
the construction of a random walk conditioned to stay positive
via the weak limit of conditional distributions, 
see Bertoin and Doney \cite{BD94}:
$$
P(x,B)\ =\ \lim_{n\to\infty}\P\{x+S_1\in B\mid\tau(x)>n\}.
$$

We now show that if $\E \xi_1=0$ and $\E \xi_1^2=:\sigma^2\in(0,\infty)$,
then $\{X_n\}$ has asymptotically zero drift.
We first observe that these moment conditions allow us to apply 
Lemma \ref{l:maj.p.e} with $\gamma=2$, $\alpha=\beta=1$ and to conclude that,
for some increasing $s(x)=o(x)$ and 
decreasing integrable at infinity $p(x)=o(1/x)$,
\begin{eqnarray}\label{m1.sx.cond}
\E\{|\xi_1|;\ |\xi_1|>s(x)\} &=& o(p(x))\quad\mbox{as }x\to\infty,
\end{eqnarray}
in particular, $\E\{\xi_1;\ \xi_1>-x\}=o(1/x)$, since $\E\xi_1=0$.
Then it follows from the definition \eqref{def.cond} of the kernel $P$ that
\begin{align*}
m_1(x)&:=\frac{1}{V(x)}\E\{V(x+\xi_1)\xi_1;\ \xi_1>-x\}\\
&=\frac{1}{V(x)}\E\{(V(x+\xi_1)-V(x))\xi_1;\ \xi_1>-x\}+\E\{\xi_1;\ \xi_1>-x\}\\
&=\frac{1}{V(x)}\E\{(V(x+\xi_1)-V(x))\xi_1;\ \xi_1>-x\}+o(1/x).
\end{align*}
The finiteness of the second moment also implies that the ladder heights 
have finite expectation, so by Blackwell's renewal theorem
(see, e.g. Durrett \cite[Theorem 2.6.4]{Durrett}), for any fixed $y>0$,
\begin{eqnarray}\label{srt}
V(x+y)-V(x) &\to& \frac{y}{\E\chi^-}\quad\mbox{as }x\to\infty,
\end{eqnarray}
in the non-lattice case;
in the lattice case both $x$ and $y$ are restricted to the lattice.
Hence $(V(x+\xi_1)-V(x))\xi_1$ converges to $\xi_1^2/\E\chi^-$ as $x\to\infty$.
By \eqref{srt}, 
$$
c_V\ :=\ \sup_x(V(x+1)-V(x))\ <\ \infty,
$$ 
which yields
\begin{eqnarray}\label{incr.V}
|V(x+y)-V(x)| &\le& c_V(|y|+1).
\end{eqnarray}
This allows us to apply the dominated
convergence theorem to infer that
\begin{eqnarray*}
\E\{(V(x+\xi_1)-V(x))\xi_1;\ \xi_1>-x\} &\to&
\frac{\E\xi_1^2}{\E\chi^-}\ =\ \frac{\sigma^2}{\E\chi^-}
\quad\mbox{as }x\to\infty.
\end{eqnarray*}
By the elementary renewal theorem (see, e.g. Durrett \cite[Theorem 2.6.3]{Durrett}), 
$V(x)\sim x/\E\chi^-$ and hence
\begin{eqnarray}\label{m1}
m_1(x) &\sim& \frac{\sigma^2}{x}\quad\mbox{as }x\to\infty.
\end{eqnarray}
For the second moment of jumps we have
\begin{align*}
m_2(x)&:=\frac{1}{V(x)}\E\{V(x+\xi_1)\xi_1^2;\ \xi_1>-x\}\\
&=\frac{1}{V(x)}\E\{(V(x+\xi_1)-V(x))\xi_1^2;\ \xi_1>-x\}+\E\{\xi_1^2;\ \xi_1>-x\}\\
&=\frac{1}{V(x)}\E\{(V(x+\xi_1)-V(x))\xi_1^2;\ \xi_1>-x\}+\sigma^2+o(1).
\end{align*}
It follows from \eqref{incr.V} that
\begin{eqnarray*}
|V(x+\xi_1)-V(x)|\xi_1^2 &\le& c_V(1+|\xi_1|)\xi_1^2\ \le\ c_V(1+x)\xi_1^2
\quad\mbox{for all }|\xi_1|\le x,
\end{eqnarray*}
hence
\begin{eqnarray*}
\frac{|V(x+\xi_1)-V(x)|}{V(x)}\xi_1^2 &\to& 0
\quad\mbox{as }x\to\infty,
\end{eqnarray*}
and, again by the dominated convergence theorem,
\begin{eqnarray*}
\frac{1}{V(x)}\E\{(V(x+\xi_1)-V(x))\xi_1^2;\ |\xi_1|\le x\}
&\to& 0\quad\text{as }x\to\infty.
\end{eqnarray*}
Therefore,
\begin{eqnarray*}
m_2(x) &=& \frac{1}{V(x)}\E\{(V(x+\xi_1)-V(x))\xi_1^2;\ \xi_1>x\}+\sigma^2+o(1).
\end{eqnarray*}
If $\E\{\xi_1^3;\ \xi_1>0\}$ is finite, then we may apply
the dominated convergence theorem to the expectation over the event
$\{\xi_1>x\}$ too and get that $m_2(x)\to\sigma^2$ as $x\to\infty$.
But if $\E\{\xi_1^3;\ \xi_1>0\}=\infty$ then
$\E\{(V(x+\xi_1)-V(x))\xi_1^2;\ \xi_1>x\}$
is infinite for all $x\ge 0$. Therefore, $m_2(x)\equiv\infty$
for any random walk with $\E\{\xi_1^3;\ \xi_1>0\}=\infty$.

Clearly, one can show directly that any random walk
conditioned to stay positive is transient while
the classical Lamperti criterion for transience---where
at least the second moment of jumps is assumed to be finite---is
only applicable to a random walk conditioned to stay positive
in the case of finite $\E\{\xi_1^3;\ \xi_1>0\}$.

Moreover, to the best of our knowledge, all known results on the
convergence towards $\Gamma$-distribution for Markov chains,
see Klebaner \cite{Kleb89},\index{Klebaner}
Kersting\index{Kersting} \cite{Kersting} or Denisov\index{Denisov} et al. 
\cite{DKW2013}, assume finiteness of $m_2(x)$.
However it is well-known that finiteness of $\sigma^2$
for a random walk is sufficient for the convergence of $X_n^2/n$
towards $\Gamma$-distribution for $X_n$ being a
random walk conditioned to stay positive.

Random walks conditioned to stay positive represent
an important class of Markov chains with asymptotically zero drift.
So we wanted that general limit theorems for Markov chains with
asymptotically zero drift covered the well known results for random walks
conditioned to stay positive.
This observation motivated us to state conditions for $\Gamma$-convergence 
in the previous chapters in terms 
of truncated moments and tail probabilities.

Repeating the arguments used above for the lower truncation at level $-x$,
we conclude that
\begin{eqnarray}\label{m1.m2.sx.cond}
m_1^{[s(x)]}(x)\ \sim\ \frac{\sigma^2}{x}\quad\mbox{and}\quad
m_2^{[s(x)]}(x)\ \to\ \sigma^2\quad\mbox{as }x\to\infty,
\end{eqnarray}
where $s(x)=o(x)$ is defined in \eqref{m1.sx.cond}.
Hence, for any $\varepsilon>0$,
\begin{align*}
\frac{2m_1^{[s(x)]}(x)}{m_2^{[s(x)]}(x)}\ \ge\ \frac{2-\varepsilon}{x}
\quad\mbox{for all sufficiently large }x.
\end{align*}
Thus, in order to apply the criterion for transience,
Theorem \ref{thm:transience.inf}, it remains to show that
\begin{equation}\label{c.trans}
\P\{\xi(x)<-s(x)\}\ \le\ \frac{p(x)}{x},
\end{equation}
for some decreasing integrable function $p$.
According to the construction of $\{X_n\}$,
this is equivalent to the following upper bound
$$
\frac{1}{V(x)}\E\{V(x+\xi_1);\ \xi_1<-s(x)\}\ \le\ \frac{p(x)}{x}.
$$
The function $V$ is increasing,
% and $V(x)\sim C_0x$,
hence it suffices to show that
\begin{eqnarray}\label{P.sx.cond}
\P\{|\xi_1|>s(x)\} &\le& \frac{p(x)}{x}
\end{eqnarray}
which in turn follows from Lemma \ref{l:maj.p.e} with $\gamma=2$,
$\beta=0$, and $\alpha=1$.
Thus $\{X_n\}$ is transient, by Theorem \ref{thm:transience.inf}.

To apply Theorem \ref{thm:gamma} on convergence to a $\Gamma$-distribution,
we additionally need to check
that
$$
\P\{\xi(x)>s(x)\}\ \le\ \frac{p(x)}{x},
$$
which is equivalent to
$$
\frac{1}{V(x)}\E\{V(x+\xi_1);\ \xi_1>s(x)\}\ \le\ \frac{p(x)}{x}.
$$
Since $V$ has asymptotically linear growth,
we may reduce the previous condition to
$$
\E\{x+\xi_1;\ \xi_1>s(x)\}\ \le\ p(x),
$$
which follows from \eqref{m1.sx.cond} and \eqref{P.sx.cond}. 
Therefore, by Theorem \ref{thm:gamma},
\index{Random walk!conditioned to stay positive!convergence to $\Gamma$-distribution}
\begin{eqnarray}\label{rwcsp.gamma}
\frac{X_n^2}{n} &\Rightarrow& \Gamma_{3/2,2\sigma^2}\quad\mbox{as }n\to\infty,
\end{eqnarray}
and, by Theorem \ref{thm:gamma.func}, the sequence of processes
$$
\frac{X_{[nt]}}{\sqrt{n\sigma^2}},\quad t\in[0,1],
$$
converges weakly in $D[0,1]$ to the Bessel process with drift
coefficient $1/x$, that is, the three-dimensional Bessel process.
In addition, the convergence to a $\Gamma$-distribution is also
accompanied by asymptotics for its integral renewal function;
by Theorem \ref{thm:renewal.gamma},\index{Random walk!conditioned to stay positive!renewal theorem}
$$
H(0,x]\ :=\ \sum_{n=1}^\infty\P\{X_n\le x\}
\ \sim\ \frac{x^2}{\sigma^2}\quad\mbox{as }x\to\infty.
$$

That random walk conditioned to stay positive converges weakly to a 
limit was shown by Bolthausen \cite{Bolthausen1976}, \index{Bolthausen}
following earlier work by Iglehart \cite{Iglehart1974}.\index{Iglehart}

Random walk conditioned to stay positive is a special example
of a Markov chain with asymptotically zero drift.
Its close connection to ordinary random walk allows us
to obtain a number of further results. More precisely,
by the definition of the transition kernel of $X$,
\begin{equation}\label{connection}
\P_z\{X_n\in dx\}=\frac{V(x)}{V(z)}\P\{z+S_n\in dx,\tau(z)>n\}.
\end{equation}
This allows us to use the fluctuation theory for random walks
in order to derive results for random walk conditioned to stay positive.
For example, Caravenna\index{Caravenna} and Chaumont\index{Chaumont} \cite{CC08}
have proved a functional limit theorem for $X$,
Bryn-Jones\index{Bryn-Jones} and Doney\index{Doney} \cite{BJD}
proved a local limit theorem for $X$.
Using results of Doney\index{Doney} \cite{Doney2012} one can also
derive asymptotics of local probabilities of small deviations of $\{X_n\}$.
Finally, results
by Jones and Doney\index{Doney} \cite{DoneyJones2012}
can be transferred into asymptotics of large deviation
probabilities for a random walk conditioned to stay positive.

We demonstrate the advantage of this connection to the fluctuation theory 
of ordinary random walks by the following version of Blackwell's theorem 
for random walks conditioned to stay positive.
\index{Random walk!conditioned to stay positive!Blackwell's theorem}

%For a random walk conditioned to stay positive one can prove
%the following version of Blackwell's theorem.
%\index{Random walk!conditioned to stay positive!Blackwell's theorem}

\begin{proposition}\label{prop:rwcsn.lr}
Assume that $\E\xi_1=0$, $\sigma^2:=\E\xi_1^2\in(0,\infty)$.
Then, for every fixed $\Delta>0$,
$$
h(x)\ :=\ H(x+\Delta)-H(x)\ \sim\ \frac{2\Delta}{\sigma^2}x
\quad\mbox{as }x\to\infty
$$
if the distribution of $\xi_1$ is non-lattice, and
$$
h(\Delta x)\sim\frac{2\Delta}{\sigma^2}x\quad\mbox{as }x\to\infty,\ x\in\Z,
$$
if $\Delta\Z$ is the minimal lattice for $\xi_1$.
\end{proposition}

\begin{proof}
Consider the non-lattice case. Define
$$
u(x)\ :=\ \E\sum_{n=1}^{\tau_0-1}\I\{S_n\in(x,x+\Delta]\}
\ =\ \sum_{n=1}^\infty\P\{S_n\in(x,x+\Delta],\tau(0)>n\}.
$$
Let $\chi^+_k$ be independent copies of the first strict ascending ladder
height $\chi^+:=S_{\eta_+}$, where $\eta_+=\min\{n\ge 1:S_n>0\}$. 
Then, by the classical duality lemma,
see e.g. Feller\index{Feller} \cite[Sect. XII.2]{Feller},
$$
\sum_{n=1}^\infty\P\{S_n\in(x,x+\Delta],\tau(0)>n\}\ =\
\sum_{k=1}^\infty\P\{\chi^+_1+\chi^+_2+\ldots+\chi^+_k\in(x,x+\Delta]\}.
$$
Applying Blackwell's theorem, we conclude in the non-lattice case that
\begin{eqnarray}\label{u-asymp}
u(x) &\to& \frac{\Delta}{\E\chi^+}\quad\mbox{as } x\to\infty.
\end{eqnarray}
This gives us the asymptotics for $h$ in the case of initial value $X_0=0$.
Indeed, by \eqref{connection} with $z=0$ where $V(0)=0$,
\begin{align*}
\sum_{n=1}^\infty\P_0\{X_n\in(x,x+\Delta]\}
&=\sum_{n=1}^\infty\int_x^{x+\Delta}\frac{V(y)}{V(0)}\P\{S_n\in dy,\tau(0)>n\}\\
&\sim V(x)\sum_{n=1}^\infty\P\{S_n\in(x,x+\Delta],\tau(0)>n\}=V(x)u(x)
\end{align*}
as $x\to\infty$, by \eqref{srt} which implies long-tailedness of the function $V$,
$V(x)\sim V(x+\Delta)$ as $x\to\infty$.
Recalling that $V(x)\sim x/\E\chi^-$ and using \eqref{u-asymp}, we obtain
$$
\sum_{n=1}^\infty\P_0\{X_n\in(x,x+\Delta]\}\
\sim\ \frac{\Delta}{\E\chi^-\E\chi^+}x\quad\mbox{as }x\to\infty.
$$
Then it only remains to apply the following identity which holds true 
for any zero drifted random walk with finite variance, 
see e.g. Feller \cite[Sect. XVIII.5, Theorem 1, or Sect. XII.10, Problem 10]{Feller},
\begin{eqnarray}\label{chi.+-}
\E\chi^-\E\chi^+ &=& \sigma^2/2.
\end{eqnarray}

Now let us consider an arbitrary initial value $X_0=z$.
In view of \eqref{connection} and $V(x)\sim V(x+\Delta)$ as $x\to\infty$,
\begin{eqnarray*}
\sum_{n=1}^\infty\P_z\{X_n\in(x,x+\Delta]\}
&\sim& \frac{V(x)}{V(z)}\sum_{n=1}^\infty\P\{z+S_n\in(x,x+\Delta],\tau(z)>n\}\\
&=& \frac{V(x)}{V(z)}\E\sum_{n=1}^{\tau(z)-1}\I\{S_n\in(x-z,x-z+\Delta]\}.
\end{eqnarray*}
Splitting the trajectory of $\{S_n\}$ by descending ladder epochs into
independent cycles and recalling the definition of $u(x)$, we obtain
\begin{equation}\label{repr}
\E\sum_{n=1}^{\tau(z)-1}\I\{S_n\in(x-z,x-z+\Delta]\}
=u(x-z)+\E\sum_{k=1}^{\theta(z)-1}u(x-z+\chi^-_1+\ldots+\chi^-_k),
\end{equation}
where $\theta(z)$ is defined in \eqref{V.rf.theta}.
By \eqref{u-asymp},
\begin{eqnarray*}
\E\sum_{n=1}^{\tau(z)-1}\I\{S_n\in(x,x+\Delta]\}
&\sim& \frac{\Delta}{\E\chi^+}\E\theta_z\quad\mbox{as }x\to\infty.
\end{eqnarray*}
Recalling that $\E\theta(z)=V(z)$, see \eqref{V.rf.theta}, 
and that $V(x)\sim x/\E\chi^-$ as $x\to\infty$, we finally get
\begin{eqnarray*}
\sum_{n=1}^\infty\P_z\{X_n\in(x,x+\Delta]\}
&\sim& \frac{\Delta}{\E\chi^+\E\chi^-}x\ =\ \frac{2\Delta}{\sigma^2}x
\end{eqnarray*}
for all fixed $z$, due to \eqref{chi.+-}.

In order to derive the same asymptotics for any initial distribution of the chain
it suffices to show that
\begin{eqnarray}\label{z-uniform}
\sup_{z\ge 0,\ x\ge 1}
\frac{1}{x}\sum_{n=1}^\infty\P_z\{X_n\in(x,x+\Delta]\} &<& \infty,
\end{eqnarray}
which allows us to apply the dominated convergence.
It follows from \eqref{connection} and \eqref{repr} that
\begin{eqnarray*}
\sum_{n=1}^\infty\P_z\{X_n\in(x,x+\Delta]\}
&\le& \frac{V(x+\Delta)}{V(z)}\Bigl(u(x-z)
+\E\sum_{k=1}^{\theta(z)-1}u(x-z+\chi^-_1+\ldots+\chi^-_k)\Bigr).
\end{eqnarray*}
Since $u_0:=\sup_x u(x)<\infty$,
\begin{eqnarray*}
\sum_{n=1}^\infty\P_z\{X_n\in(x,x+\Delta]\}
&\le& \frac{V(x+\Delta)}{V(z)}u_0\E\theta(z)\ =\ V(x+\Delta)u_0.
\end{eqnarray*}
Now \eqref{z-uniform} follows from the asymptotic linearity of $V$ 
and the proof in the non-lattice case is complete. The lattice case is similar.
\qed\end{proof}

Let us demonstrate an alternative proof based on Corollary \ref{cor:regular}.

\begin{proof}
Let us show that under the conditions stated the random walk conditioned 
to stay positive satisfies all the conditions of Corollary~\ref{cor:regular}.
Firstly, the condition \eqref{m1.m2.1x} holds with $\mu=\sigma^2$ and $b=\sigma^2$
as shown above in \eqref{m1.m2.sx.cond}. Secondly,
the condition \eqref{regular_left_tail} follows from \eqref{P.sx.cond}.

Thirdly, we also need to check the conditions 
\eqref{majorant_third}, \eqref{majorant_third_moment_exists}
and \eqref{majoriz.i}. To check the first one, we note that, 
\[
c_1\ :=\ \sup_x\frac{V(x+s(x))}{V(x)}\ <\ \infty,
\] 
hence, for $t\le s(x)=o(x)$,
\begin{eqnarray*}
\P\{|\xi(x)|>t, |\xi(x)|\le s(x)\}
&=& \biggl(\int_{-s(x)}^{-t}+\int_t^{s(x)}\biggr)\frac{V(x+u)}{V(x)}\P\{\xi_1\in du\}\\
&\le& c_1\P\{|\xi_1|>t\},
\end{eqnarray*}
and \eqref{majorant_third}--\eqref{majorant_third_moment_exists}
follows if we take $\widehat{\xi}$ defined by its tail as
\[
\P\{\widehat{\xi}>t\}=\min\{1,c_1\P\{|\xi_1|>t\}\},
\]
which is square integrable because $\xi_1$ is so.

Next, using once again \eqref{incr.V} we obtain
\begin{eqnarray*}
\P\{|\xi(x)|>t\}
&=& \biggl(\int_{-x}^{-t}+\int_t^\infty\biggr)
\frac{V(x+u)}{V(x)}\P\{\xi_1\in du\}\\
&\le& \P\{\xi_1<-t\}+\int_t^\infty
\Bigl(1+c_V\frac{u+1}{V(x)}\Bigr)\P\{\xi_1\in du\}\\
&\le& \P\{\xi_1<-t\}+\Bigl(1+\frac{c_V}{V(x)}\Bigr)\P\{\xi_1>t\}
+\frac{c_V}{V(x)}\E\{|\xi_1|;|\xi_1|>t\})\\
&\le& c_2(\P\{|\xi_1|>t\}+\E\{|\xi_1|;|\xi_1|>t\})\quad\mbox{for all }x,\ t>0.
\end{eqnarray*}
The right hand side is integrable due to $\E\xi_1^2<\infty$, 
so the condition \eqref{majoriz.i} is satisfied too.

Finally, the asymptotic homogeneity \eqref{asymp.hom.i} 
is immediate from \eqref{def.cond},  with $\xi=\xi_1$, 
because, for any fixed $u\in\R$, 
$V(x+u)/V(x)\to 1$ as $x\to\infty$, and the proof is complete.
\qed\end{proof}

\section{Reflected random walk with zero drift}
\label{sec:reflected}

Let $\eta_n$, $n\ge 1$, be a sequence of independent identically
distributed random variables with zero mean and finite variance.
The chain defined by
\begin{equation}\label{ref.1}
X_{n+1}=|X_n+\eta_{n+1}|,\quad n\ge0,
\end{equation}
is usually called a {\it reflected random walk.}\index{Random walk!reflected}
It follows from \eqref{ref.1} that
\begin{eqnarray*}
\xi(x)&=&(x+\eta)\I\{x+\eta\ge0\}-(x+\eta)\I\{x+\eta<0\}-x\\
&=&\eta-2(x+\eta)\I\{x+\eta<0\}=\eta+2(x+\eta)^-.
\end{eqnarray*}
This representation implies that, for any function $s(x)<x$,
\begin{eqnarray*}
m_1^{[s(x)]}(x)&=&\E\{\eta;|\eta|\le s(x)\}
+\E\{\eta+2(x+\eta)^-;|\eta+2(x+\eta)^-|\le s(x),\,\eta<-x\}\\
&=&\E\{\eta;|\eta|\le s(x)\}-\E\{2x+\eta;\ |2x+\eta|\le s(x)\}.
\end{eqnarray*}
From this equality and the assumption $\E\eta=0$ we infer that
\begin{eqnarray*}
|m_1^{[s(x)]}(x)| &\le& \E\{|\eta|;|\eta|>s(x)\}+s(x)\P\{\eta\le -2x+s(x)\}\\
&\le& 2\E\{|\eta|;|\eta|>s(x)\}.
\end{eqnarray*}
The assumption $\E\eta^2<\infty$ implies that there exists a function
$s(x)=o(x)$ such that $\E[|\eta|;|\eta|>s(x)]$ is integrable, 
see Lemma \ref{l:maj.p.e} with $\gamma=2$, $\alpha=\beta=1$.
Consequently, $|m_1^{[s(x)]}(x)|$ is also integrable.
Taking into account that
$$
m_2^{[s(x)]}(x)\to \E\eta^2\in(0,\infty),
$$
we finally obtain
$$
\frac{m_1^{[s(x)]}(x)}{m_2^{[s(x)]}(x)}=o(p(x))
$$
for some decreasing integrable function $p(x)$ satisfying $p'(x)=o(1/x^2).$
Therefore, the reflected random walk $X_n$ satisfies
\eqref{r-cond.2.equiv} with $r(x)\equiv0$. This implies that $U(x)=x$
in this case. Furthermore, the validity of \eqref{cond.xi.le},
\eqref{cond.xi.ge} and \eqref{cond.3.moment} easily follows from the assumption
$\E\eta^2<\infty$. Consequently, we may apply Theorems \ref{thm:pi.recurrent},
\ref{thm:rec.time} and \ref{thm:cond.rec.time} to the invariant measure
$\pi$ of the reflected random walk 
$\{X_n\}$:\index{Random walk!reflected!invariant measure}\index{Random walk!reflected!return probability}
\begin{eqnarray}\label{ref.2}
\pi(ax,x] &\sim& c(1-a)x\quad\text{as }x\to\infty,
\end{eqnarray}
to the down-crossing probabilities, 
for a sufficiently large $\widehat x$,\index{Random walk!reflected!down-crossing probability}
\begin{eqnarray}\label{ref.3}
\P_x\{\tau_{\widehat{x}}>n\}\
\sim\ \frac{V(x)}{\Gamma(3/2)\sqrt{2\E\eta^2}}n^{-1/2}
\quad\text{as }n\to\infty
\end{eqnarray}
and to the conditional distribution
\begin{eqnarray}\label{ref.4}
\P\{X_n>u\sqrt{n}\mid \tau_{\widehat{x}}>n\}\ \to\ e^{-u^2/2\E\eta^2}
\quad\text{as }n\to\infty.
\end{eqnarray}
In addition, we can apply Theorem \ref{thm:srt.pi.i} 
to conclude local asymptotics for the invariant measure
$\pi$ of the reflected random walk
\begin{eqnarray}\label{ref.2.loc}
\pi(x,x+h] &\to& ch\quad\text{as }x\to\infty,
\end{eqnarray}
for all $h>0$ in the non-lattice case;
in the lattice case both $x$ and $h$ should be restricted to the lattice.

Asymptotics in \eqref{ref.3} and \eqref{ref.4} coincide with
that for ordinary random walk, only the function $V(x)$ can be different.
This difference comes from the fact that reflection at zero can happen
in such a way that the position after
the reflection is again bigger than $\widehat{x}$.

One can also obtain asymptotics \eqref{ref.3} using the asymptotics 
for the first visit of a bounded set by a  one-dimensional random walk.  
Namely one can interpret $\tau_{\widehat x}$ as the first time the random walk visits 
a compact interval $[-\hat x, \hat x]$.  
Then, for arithmetic random walks the asymptotics \eqref{ref.3}
 follow from the results of Kesten and Spitzer \cite{KestenSpitzer} and 
 for general random walks from Vysotsky \cite{Vysotsky}, see also references therein.  

Relation \eqref{ref.2} implies that $\{X_n\}$ is null recurrent.
Recurrence of a reflected random walk with finite second moments
of increments has been shown by
Kemperman \cite{Kemperman74}.\index{Kemperman}
Non-positivity in the case of zero mean is immediate
from the fact that any ordinary driftless random walk is null-recurrent.

The local asymptotics \eqref{ref.2.loc} was proven by Brofferio\index{Brofferio} 
and Buraczewski\index{Buraczewski} in \cite[Theorem 1.3]{BroBuraczewski} under 
the assumption that $\E|\eta^-|^{3/2}<\infty$ and $\E(\eta^+)^2<\infty$.

\section{State-dependent branching processes with migration}
\label{sec:branching}

In this section we consider branching processes with reproduction law 
depending on the number of particles in the population: 
If there are $k$ particles in the population then the number of offspring
of every particle is an independent copy of a random variable $\zeta(k)\ge 0$.
Furthermore, we assume that there is a
migration of particles. This will be modelled by $\eta$'s:
given $k$ particles in the system, the number of migrants at time $n$ is
an independent copy of a random variable $\eta(k)$---which may take 
both positive and negative values. As a consequence we have the following 
Markov chain:\index{Branching process!with migration of particles}
\begin{equation}\label{BP1}
Z_{n+1}\ :=\ \biggl(\sum_{i=1}^{Z_n}\zeta_{n+1,i}(Z_n)+\eta_{n+1}(Z_n)\biggr)^+,
\quad n\ge0,
\end{equation}
where $\{\zeta_{n,i}(k),\,n\ge0,i\ge1\}$ are independent copies of
$\zeta(k)$ and $\{\eta_n(k),\, n\ge 1\}$ are independent copies of $\eta(k)$.
Then $\{Z_n\}$ is a Markov chain on $\Zp$.

There is also an alternative way to introduce migration of particles:
\begin{equation}\label{BP2}
Y_{n+1}\ :=\ \sum_{i=1}^{(Y_n+\eta_n(Y_n))^+}\zeta_{n+1,i}(Y_n),
\quad n\ge0.
\end{equation}
The only difference between these two models consists in the order of
branching and migration at every time step.
In \eqref{BP1} one performs first branching and then migration,
and in \eqref{BP2} these two mechanisms appear in the reversed order.

We shall assume that offspring random variables $\zeta(k)$ are such that
\begin{equation}\label{bp.expectation}
k(\E\zeta(k)-1)\ \to\ a_\zeta\in\R\quad\mbox{as }k\to\infty,
\end{equation}
and
\begin{equation}\label{bp.variance}
\sigma^2(k)\ :=\ \V\zeta(k)\ \to\ \sigma^2\in(0,\infty),
\end{equation}
and that the expectation of the migration quantity $\eta(k)$ converges:
\begin{equation}\label{bp.eta.expectation}
\E\eta(k)\ \to\ a_\eta\in\R\quad\mbox{as }k\to\infty.
\end{equation}
Under these assumptions the asymptotic behaviour
of the first two moments of jumps is as follows:
\begin{align*}
&\E\{Z_{n+1}-Z_n\mid Z_n=k\}\ \to\ a_\zeta+a_\eta,\\
&\E\{(Z_{n+1}-Z_n)^2\mid Z_n=k\}\ \sim\ \sigma^2k
\quad\mbox{as }k\to\infty;
\end{align*}
for the second relation we need to assume that $\E\eta^2(k)=o(k)$.

Linear growth of variances significantly complicates the analysis of
the Markov chain $\{Z_n\}$.
In order to get bounded variances we consider a chain
\begin{eqnarray}\label{def.X.sqrt.Z}
X_n &:=& \sqrt{Z_n},\quad n\ge 0,
\end{eqnarray}
whose jumps are
\begin{eqnarray*}
\xi(\sqrt k) &\stackrel{d}{=}&
\sqrt{\Biggl(\sum_{i=1}^k\zeta_{1,i}(k)+\eta_1(k)\Biggr)^+}-\sqrt k\\
&=& \sqrt{(S(k)+\eta_1(k))^+}-\sqrt k,
\end{eqnarray*}
where $S(k):=\zeta_{1,1}(k)+\ldots+\zeta_{1,k}(k)$.
It follows from the proof of the first result in the next subsection
that this Markov chain has asymptotically zero drift
and bounded second moment of jumps.

\subsection{Classification of near-critical branching processes}

We start with classification of branching processes satisfying
\eqref{bp.expectation}---\eqref{bp.eta.expectation}.
Under some mild conditions on $\zeta(k)$ and $\eta(k)$ we show that
\index{Branching process!near-critical!classification}
\begin{itemize}
\item[(i)] if $a_\zeta+a_\eta>\sigma^2/2$ then $\{Z_n\}$ is transient;
\item[(ii)] if $0<a_\zeta+a_\eta<\sigma^2/2$ then $\{Z_n\}$ is null recurrent;
\item[(iii)] if $a_\zeta+a_\eta<0$ then $\{Z_n\}$ is positive recurrent.
\end{itemize}

We start with evaluation of the first two truncated moments
of jumps $\xi(\sqrt k)$ of the chain $\{X_n\}$ defined in \eqref{def.X.sqrt.Z}
and of their left tails.

\begin{proposition}\label{prop:bp.class}
Let the moment conditions \eqref{bp.expectation}--\eqref{bp.eta.expectation} hold 
and let the family of random variables $\{|\eta(k)|,k\ge 0\}$
possess an integrable majorant $\eta$, that is,
\begin{eqnarray}\label{bp.eta.majorant}
|\eta(k)|\ \le_{\rm st}\ \eta\quad\mbox{for all }k\ge 0\quad
\mbox{and }\ \E\eta\ <\ \infty.
\end{eqnarray}
Then, there exists an increasing function $s(x)=o(x)$ such that
\begin{eqnarray}\label{rec.3.1.gamma.bp.class}
\P\{\xi(\sqrt k)<-s(\sqrt k)\} &\le& p(\sqrt k)/\sqrt k
\quad\mbox{for all }k\ge 0,
\end{eqnarray}
where a decreasing function $p(x)>0$ is integrable at infinity.
If, in addition, for some increasing function $t(x)=o(x)$,
\begin{eqnarray}\label{bp.corr1.2nd.class}
\E\{\zeta^2(k);\ \zeta(k)>t(k)\} &\to& 0\quad\mbox{as }k\to\infty,
\end{eqnarray}
then there exists an increasing function $\widetilde s(x)=o(x)$
such that, for all $s(x)\ge\widetilde s(x)$,
%, for any increasing function $s(x)$ such that $s(x)\ge 4t(x^2)/x$,
\begin{eqnarray}\label{1.2.G.bp.class}
m_1^{[s(\sqrt k)]}(\sqrt k)\sim \frac{a_\zeta+a_\eta-\sigma^2/4}{2\sqrt k}
&\mbox{and}& m_2^{[s(\sqrt k)]}(\sqrt k)\to \frac{\sigma^2}{4}
\quad\mbox{as }k\to\infty.\hspace{10mm}
\end{eqnarray}
\end{proposition}

\begin{theopargself}
\begin{proof}[of Proposition \ref{prop:bp.class}.]
Let us introduce events
\begin{eqnarray}\label{Ak.bp}
A_k &:=& \{|\xi(\sqrt k)|\le s(\sqrt k)\}\nonumber\\ 
&=& \bigr\{\bigl|\sqrt{(S(k)+\eta(k))^+}-\sqrt k\bigr|\le s(\sqrt k)\bigl\}.
\end{eqnarray}
Provided $s(x)=o(x)$, an equivalent way to define $A_k$  
for all sufficiently large $k$ is
\begin{eqnarray*}
(\sqrt k-s(\sqrt k))^2\ \le\ S(k)+\eta(k)\ \le\ (\sqrt k+s(\sqrt k))^2,
\end{eqnarray*}
that is,
\begin{eqnarray*}
-2\sqrt k s(\sqrt k)+s^2(\sqrt k)\ \le\ S(k)-k+\eta(k)\ \le\ 2\sqrt k s(\sqrt k)+s^2(\sqrt k).
\end{eqnarray*}
Therefore, again due to $s(x)=o(x)$, for all sufficiently large $k$ we have
\begin{eqnarray}\label{c.Ak.bp.1}
\{\xi^\pm(\sqrt k)>s(\sqrt k)\} &\subseteq& 
\{(S(k)-k)^\pm>\mbox{$\frac{3}{2}$}\sqrt k s(\sqrt k)\} \cup \{\eta^\pm(k)>s^2(\sqrt k)\}.\hspace{10mm}
\end{eqnarray}
The condition \eqref{bp.expectation} may be rewritten as
$\E S(k)-k\to a_\zeta$ as $k\to\infty$,
hence for all sufficiently large $k$,
\begin{eqnarray}\label{Ak-def.class}
\P\{\xi^\pm(\sqrt k)>s(\sqrt k)\}
&\le& \P\{(S(k)-\E S(k))^\pm>\sqrt k s(\sqrt k)\}+\P\{\eta>s^2(\sqrt k)\},
\nonumber\\[-1mm]
\end{eqnarray}
owing to the majorisation condition \eqref{bp.eta.majorant}.

By exponential Chebyshev's inequality,
\begin{eqnarray*}
\P\{S(k)-\E S(k)<-\sqrt k s(\sqrt k)\} &\le&
\Bigl(\E e^{\frac{\E\zeta(k)-\zeta(k)}{\sqrt k}}\Bigr)^ke^{-s(\sqrt k)}.
\end{eqnarray*}
By Taylor's expansion, for some $\theta\in[0,1]$,
\begin{eqnarray*}
\E e^{\frac{\E\zeta(k)-\zeta(k)}{\sqrt k}} &=&
1+\frac{1}{2k}\E(\E\zeta(k)-\zeta(k))^2
e^{\theta\frac{\E\zeta(k)-\zeta(k)}{\sqrt k}}\\
&\le& 1+\frac{1}{2k}\V\zeta(k)\ e^{\E\zeta(k)/\sqrt k},
\end{eqnarray*}
because the $\zeta(k)$ is non-negative.
Therefore, for some $c<\infty$,
\begin{eqnarray}\label{S.lt.bound}
\P\{S(k)-\E S(k)<-\sqrt k s(\sqrt k)\} &\le& ce^{-s(\sqrt k)}.
\end{eqnarray}
Further, since $\E(\sqrt\eta)^2=\E\eta<\infty$,
by Lemma \ref{l:maj.p.e} there exists an $s(x)=o(x)$ such that
\begin{eqnarray*}
\P\{\sqrt\eta>s(x)\} &\le& p_1(x)/x,
\end{eqnarray*}
for some decreasing integrable at infinity function $p_1(x)$. Therefore,
\begin{eqnarray}\label{ub.tau.bp.class}
\P\{\eta>s^2(\sqrt k)\} &\le& p_1(\sqrt k)/\sqrt k.
\end{eqnarray}
Substituting \eqref{S.lt.bound} and \eqref{ub.tau.bp.class}
into \eqref{Ak-def.class} we obtain
\begin{eqnarray*}
\P\{\xi(\sqrt k)<-s(\sqrt k)\} &\le& ce^{-s(\sqrt k)}+p_1(\sqrt k)/\sqrt k,
\end{eqnarray*}
so \eqref{rec.3.1.gamma.bp.class} follows provided $s(x)\ge 3\log x$.

Let us now show the relations \eqref{1.2.G.bp.class}
for the truncated moments. We start by showing that
\begin{eqnarray}\label{bp.lemma.2.class}
\E\{(S(k)-k)^2;\ |S(k)-k|>\sqrt ks(\sqrt k)\} &=&
o(k)\quad\mbox{as }k\to\infty.
\end{eqnarray}
Since the variance of $\zeta(k)$, $k\ge1$, is bounded,
taking $y=x/2$ in \eqref{bp.NF.2.l} we conclude
\begin{eqnarray}\label{bp.NF.2}
\lefteqn{\E\{(S(k)-\E S(k))^2;\ S(k)-\E S(k)>x\}}\nonumber\\
&&\le\ 2C(2)(k/x)^2+k\E\{(\zeta(k)-\E\zeta(k))^2;\ \zeta(k)-\E\zeta(k)>x/2\}\nonumber\\
&&\hspace{30mm}+\ k^2\P\{\zeta(k)-\E\zeta(k)>x/2\}.\hspace{5mm}
\end{eqnarray}
For $x=\sqrt ks(\sqrt k)$ which is greater
than $3t(k)$ provided $s(x)\ge 3t(x^2)/x=o(x)$, we obtain
\begin{eqnarray}\label{bp.class.sec.sec}
\lefteqn{\E\{(S(k)-\E S(k))^2;\ |S(k)-\E S(k)|>\sqrt ks(\sqrt k)\}}\nonumber\\
&&\hspace{25mm}\le\ Ck\Bigl(\frac{1}{s^2(\sqrt k)}
+\E\{\zeta^2(k);\ \zeta(k)>t(k)\}\Bigr),
\end{eqnarray}
and, by the condition \eqref{bp.corr1.2nd.class},
\begin{eqnarray*}
\E\{(S(k)-\E S(k))^2;\ |S(k)-\E S(k)|>\sqrt ks(\sqrt k)\} &=&
o(k)\quad\mbox{as }k\to\infty,
\end{eqnarray*}
which implies \eqref{bp.lemma.2.class} because $|\E S(k)-k|$ is bounded
due to the condition \eqref{bp.expectation}.

It follows from \eqref{c.Ak.bp.1} and then from \eqref{bp.lemma.2.class} that
\begin{eqnarray*}
\E\{(S(k)-k)^2;\ A_k^c\} &\le&
\E\{(S(k)-k)^2;\ |S(k)-k|\ge \sqrt k s(\sqrt k)\}\\
&&\hspace{30mm}+\E\{(S(k)-k)^2;\ \eta(k)\ge s^2(\sqrt k)\}\\
&=& o(k)+(\V S(k)+(\E S(k)-k)^2) \P\{\eta(k)\ge s^2(\sqrt k)\}\\
&=& o(k)\quad\mbox{as }k\to\infty,
\end{eqnarray*}
because the second moment of both $\zeta$'s and
the sequence $|\E S(k)-k|$ is bounded by the conditions
\eqref{bp.expectation} and \eqref{bp.variance}.
Hence the condition \eqref{bp.variance} allows us to conclude that
\begin{eqnarray}\label{bp.prop.4}
\frac{1}{k}\E\{(S(k)-k)^2;\ A_k\} &=& \sigma^2(k)+o(1)
\ \to\ \sigma^2\quad\mbox{as }k\to\infty.
\end{eqnarray}
Note also that
\begin{eqnarray}\label{bp.prop.5}
\frac{1}{k}\E\{|\eta(k)(S(k)-k)|;\ A_k\} &\le&
\frac{1}{k}\E|\eta(k)(S(k)-k)|\nonumber\\
&=& \frac{1}{k}\E|\eta(k)|\E|S(k)-k|\nonumber\\
&\le& \frac{1}{k}\E\eta\sqrt{\E(S(k)-k)^2}
\ =\ O(1/\sqrt k),
\end{eqnarray}
by the independence of $\eta(k)$ and $S(k)$, and the majorisation condition
\eqref{bp.eta.majorant}. Moreover, since
\begin{eqnarray*}
A_k &=&
\bigr\{k-2\sqrt ks(\sqrt k)+ s^2(\sqrt k)\le
S(k)+\eta(k)\le k+2\sqrt ks(\sqrt k)+ s^2(\sqrt k)\bigl\},
\end{eqnarray*}
for all sufficiently large $k$, we have
\begin{eqnarray}\label{Ak.eta.S}
A_k &\subseteq& \{|\eta(k)|\le 3\sqrt k s(\sqrt k)+|S(k)-k|\},
\end{eqnarray}
hence
\begin{eqnarray*}
\E\{\eta^2(k);\ A_k\} &\le&
\E\{\eta^2(k);\ |\eta(k)|\le 3\sqrt k s(\sqrt k)+|S(k)-k|\}.
\end{eqnarray*}
By the assumption $\E\eta<\infty$ on the majorant for $\eta(k)$'s,
it follows from Lemma \ref{l:p.V.ui.o} with $V(x)=x$ and $p=2$ that
\begin{eqnarray*}
\sup_k \E\{\eta^2(k);\ |\eta(k)|\le x\} &=& o(x)\quad\mbox{as }x\to\infty.
\end{eqnarray*}
Thus,
\begin{eqnarray}\label{bp.prop.6}
\E\{\eta^2(k);\ A_k\} &\le& o(\sqrt k s(\sqrt k)+\E|S(k)-k|)
\ =\ o(k)\quad\mbox{as }k\to\infty.
\end{eqnarray}
Combining \eqref{bp.prop.4}, \eqref{bp.prop.5} and
\eqref{bp.prop.6}, we conclude convergence
\begin{equation}\label{bp.prop.7}
\frac{1}{k}\E\{(S(k)-k+\eta(k))^2;\ A_k\}\ \to\ \sigma^2
\quad\mbox{as }k\to\infty.
\end{equation}

The upper bound \eqref{bp.lemma.2.class} implies that,
for any fixed $n\ge 1$,
\begin{eqnarray}\label{E.S-k.eta}
\E\{|S(k)-k|;\ |S(k)-k|>k/n\} &\le& \frac{n}{k}\E\{(S(k)-k)^2;\ |S(k)-k|>k/n\}\nonumber\\
&\to& 0\quad\mbox{as }k\to\infty.
\end{eqnarray}
Therefore, for some $s(x)=o(x)$,
\begin{eqnarray}\label{bp.lemma.1}
\E\{|S(k)-k|;\ |S(k)-k|>\sqrt ks(\sqrt k)\}
&\to& 0\quad\mbox{as }k\to\infty.
\end{eqnarray}
It follows from \eqref{c.Ak.bp.1} that
\begin{eqnarray}\label{first.ge.deco}
\lefteqn{\E\{|S(k)-k+\eta(k)|;\ |\xi(\sqrt k)|>s(\sqrt k)\}}\nonumber\\
&\le& \E\{|S(k)-k+\eta(k)|;\ |S(k)-k|>\sqrt k s(\sqrt k)\}\nonumber\\
&&\hspace{15mm}+\E\{|S(k)-k+\eta(k)|;\ |\eta(k)|>s^2(\sqrt k)\}\nonumber\\
&\le& \E\{|S(k)-k|;\ |S(k)-k|>\sqrt k s(\sqrt k)\}
+\E\eta\P\{|S(k)-k|>\sqrt k s(\sqrt k)\}\nonumber\\
&&\hspace{15mm}+\ \E|S(k)-k|\P\{\eta>s^2(\sqrt k)\}
+\E\{\eta;\ \eta>s^2(\sqrt k)\},\hspace{5mm}
\end{eqnarray}
by the independence of $S(k)$ and $\eta(k)$,
and by the majorisation condition \eqref{bp.eta.majorant}.
Applying now \eqref{bp.lemma.1} and
\eqref{ub.tau.bp.class} we get that
\begin{eqnarray}\label{bp.expec.tail}
\E\{|S(k)-k+\eta(k)|;\ A_k^c\} &\to& 0\quad\mbox{as }k\to\infty.
\end{eqnarray}
In its turn, this implies that
\begin{eqnarray}\label{bp.prop.3}
\E\{S(k)-k+\eta(k);\ A_k\} &=& \E(S(k)-k+\eta(k))+o(1)\nonumber\\
&\to& a_\zeta+a_\eta,
\end{eqnarray}
by the conditions \eqref{bp.expectation} and \eqref{bp.eta.expectation}.

Owing to Taylor's expansion we conclude that
\begin{eqnarray*}
\xi(\sqrt k) &=& \sqrt k\biggl(\sqrt{1+\frac{S(k)-k+\eta(k)}{k}}-1\biggr)\\
&=& \frac{S(k)-k+\eta(k)}{2\sqrt k}
-\frac18\frac{(S(k)-k+\eta(k))^2}{k\sqrt k}(1+o(1))
\end{eqnarray*}
as $k\to\infty$ uniformly on the event $A_k$ where uniformity
of $o(1)$ on this event follows from the relation $s(x)=o(x)$.
Then, using \eqref{bp.prop.3} and \eqref{bp.prop.7} we obtain
\begin{eqnarray*}
\sqrt km_1^{[s(\sqrt k)]}(\sqrt k)
&=& \sqrt k\E\{\xi(\sqrt k);\ A_k\}\\
&\to& \frac{a_\zeta+a_\eta-\sigma^2/4}{2}\quad\mbox{as }k\to\infty.
\end{eqnarray*}
To determine the asymptotic behaviour of the second
truncated moment we note that, uniformly on the event $A_k$,
\begin{eqnarray*}
\xi^2(\sqrt k) &=&
\biggl(\frac{S(k)-k+\eta}{\sqrt{S(k)+\eta(k)}+\sqrt k}\biggr)^2\\
&=& \frac{(S(k)-k+\eta(k))^2}{4k}(1+o(1))\quad\mbox{as }k\to\infty.
\end{eqnarray*}
Using \eqref{bp.prop.7} once again we conclude that
$$
m_2^{[s(\sqrt k)]}(\sqrt k)\ =\ \E\{\xi^2(\sqrt k);\
A_k\}\ \to\ \sigma^2/4\quad\mbox{as }k\to\infty.
$$
So, both convergences in \eqref{1.2.G.bp.class} hold true
and the proof is complete.
\qed\end{proof}
\end{theopargself}

\index{Branching process!near-critical!transience}

\begin{theorem}\label{thm:bp.class.1}
Assume that \eqref{bp.expectation}--\eqref{bp.eta.majorant}
and \eqref{bp.corr1.2nd.class} hold, and
\begin{eqnarray}\label{bp.to.inf}
\limsup_{n\to\infty} Z_n &=& \infty
\quad\mbox{with probability }1.
\end{eqnarray}
If $a_\zeta+a_\eta>\sigma^2/2$, then $\{Z_n\}$ is transient.
\end{theorem}

Since $\{Z_n\}$ lives on the non-negative integers,
the assumption \eqref{bp.to.inf} corresponds,
modulo some periodicity issues,
to irreducibility of the state space of the branching process.
This assumption obviously excludes existence of absorbing states.
So, standard non-degenerate critical Galton-Watson processes
do not satisfy this condition, but if one adds a non-trivial
immigration at zero, then \eqref{bp.to.inf} follows.
The same is true in the case of space-homogeneous immigration.

Probably the simplest sufficient condition for \eqref{bp.to.inf}
is the following one
$$
\P\{\eta(k)>0\}>0\quad\mbox{for all }k\ge0.
$$
In this case no any further restriction on the offspring
numbers $\zeta(k)$ is needed to guarantee \eqref{bp.to.inf}. 
For a near-critical process satisfying
\eqref{bp.expectation} and \eqref{bp.variance} one can relax
the restriction on the migration mentioned above.
Indeed,  \eqref{bp.expectation} and \eqref{bp.variance} imply that
$$
\E\zeta(k)(\zeta(k)-1)\to\sigma^2>0\quad\mbox{as }k\to\infty.
$$
Consequently, there exists a $k_0$ such that
$$
\inf_{k>k_0}\P\{\zeta(k)\ge2\}>0.
$$
Therefore, the desired irreducibility then follows from the conditions
$$
\P\{\eta(k)>0\}>0\quad\mbox{for all }k\le k_0
$$
and
$$
\P\{\eta(k)\ge -2k+1\}>0\quad\mbox{for all }k>k_0.
$$

\begin{theopargself}
\begin{proof}[of Theorem \ref{thm:bp.class.1}]
It is sufficient to check that the chain $\{X_n\}=\{\sqrt{Z_n}\}$ satisfies
all the conditions of Theorem \ref{thm:transience.inf}.
Fulfillment of the condition \eqref{r-cond.5.tr.inf}
of Theorem \ref{thm:transience.inf} follows from
\eqref{1.2.G.bp.class} and assumption $a_\zeta+a_\eta>\sigma^2/2$,
while \eqref{rec.1a.inf} follows from \eqref{rec.3.1.gamma.bp.class}
which completes the proof.
\qed\end{proof}
\end{theopargself}

\index{Branching process!near-critical!recurrence}
\index{Branching process!near-critical!positive recurrence}

\begin{theorem}\label{thm:bp.recurrence}
Assume that \eqref{bp.expectation}, \eqref{bp.eta.expectation} hold and,
for some $\varepsilon>0$,
\begin{eqnarray}\label{11.4.P}
\P\{\zeta(k)=0\} &\le& 1-2\varepsilon,\\
\label{11.4.eta}
\E\{\eta(k);\ \eta(k)<-k\varepsilon\} &\to& 0\quad\mbox{as }k\to\infty.
\end{eqnarray}
If $a_\zeta+a_\eta<0$ then the chain $\{Z_n\}$ is positive recurrent.

If, in addition, \eqref{bp.variance}, \eqref{bp.eta.majorant}
and \eqref{bp.corr1.2nd.class} hold and $a_\zeta+a_\eta<\sigma^2/2$,
then the chain $\{Z_n\}$ is recurrent.
\end{theorem}

\begin{proof}
For positive recurrence we show that the drift of $\{Z_n\}$,
\begin{eqnarray*}
\E\{Z_{n+1}-Z_n\mid Z_n=k\} &=& \E(S(k)+\eta(k))^+-k\\
&=& \E(S(k)+\eta(k))-k+\E(S(k)+\eta(k))^-,
\end{eqnarray*}
is negative and bounded away from zero for all sufficiently large
$k$ if $a_\zeta+a_\eta<0$, because
\begin{eqnarray*}
\E(S(k)+\eta(k))^-
&=& \E\{(S(k)+\eta(k))^-;\ \eta(k)<-k\varepsilon\}\\
&&\hspace{10mm} +\E\{(S(k)+\eta(k))^-;\ \eta(k)\ge -k\varepsilon\}\\
&\le& \E\{\eta^-(k);\ \eta(k)<-k\varepsilon\} +\E(S(k)-\varepsilon k)^-.
\end{eqnarray*}
The first expectation on the right hand side tends to zero as $k\to\infty$
due to the condition \eqref{11.4.eta}. The second expectation
tends to zero too, because, by the condition \eqref{11.4.P},
all $\zeta(k)$'s stochastically dominate a Bernoulli random variable $\zeta$ 
with success probability $2\varepsilon$, so
\begin{eqnarray*}
\E(S(k)-\varepsilon k)^- &\le&
\E((\zeta_1-\varepsilon)+\ldots+(\zeta_k-\varepsilon))^-\\
&\le& \varepsilon k\P((\zeta_1-\varepsilon)+\ldots+(\zeta_k-\varepsilon)<0)\\
&\le& \varepsilon k(1-\delta)^k\quad\mbox{for some }\delta>0,
\end{eqnarray*}
owing to $\E(\zeta-\varepsilon)=\varepsilon>0$; 
here $\zeta_i$'s are independent copies of $\zeta$.

Let us now check that the chain $\{\sqrt{Z_n}\}$ satisfies
all the conditions of Corollary~\ref{cor:rec} in the case
$a_\zeta+a_\eta<\sigma^2/2$.
In view of the condition \eqref{bp.corr1.2nd.class},
\begin{eqnarray}\label{bp.corr1.1.weak}
\P\{\zeta(k)>t(k)\}\ \le\ \frac{\E\{\zeta^2(k);\ \zeta(k)>t(k)\}}{t^2(k)}
&=& o(1/k^2)\quad\mbox{as }k\to\infty,
\end{eqnarray}
possibly with a faster growing level $t(x)$.
It follows from \eqref{bp.lemma.1} and Chebyshev's inequality that
\begin{eqnarray*}
\P\{S(k)-\E S(k)>\sqrt k s(\sqrt k)\} &=& o(1/k)\quad\mbox{as }k\to\infty,
\end{eqnarray*}
possibly with a faster increasing $s(x)$.
Together with \eqref{ub.tau.bp.class} and \eqref{c.Ak.bp.1}
this yields an upper bound
\begin{eqnarray}\label{rec.3.1.gamma.bp.weak}
\P\{\xi(\sqrt k)>s(\sqrt k)\} &=& o(1/k)\quad\mbox{as }k\to\infty.
\end{eqnarray}
The condition \eqref{eps.rec.1} follows from upper bounds
\begin{eqnarray*}
\lefteqn{\E\{\xi^3(\sqrt k);\ \xi(\sqrt k)\in[0,\sqrt k]\}}\\
&\le& \E\{\xi^3(\sqrt k);\ \xi(\sqrt k)\in[0,s(\sqrt k)]\}
+\E\{\xi^3(\sqrt k);\ \xi(\sqrt k)\in(s(\sqrt k),\sqrt k]\}\\
&\le& s(\sqrt k)\E\{\xi^2(\sqrt k);\ \xi(\sqrt k)\in[0,s(\sqrt k)]\}
+k^{3/2}\P\{\xi(\sqrt k)>s(\sqrt k)\}\\
&=& o(\sqrt k)\quad\mbox{as }k\to\infty,
\end{eqnarray*}
because the first term on the right hand side is of order
$O(s(\sqrt k))=o(\sqrt k)$ due to the second convergence in
\eqref{1.2.G.bp.class} while the second term is of the same order
by \eqref{rec.3.1.gamma.bp.weak}.

In order to show the validity of \eqref{eps.rec.2} we first note
that, by the concavity of $\sqrt y$,
\begin{eqnarray*}
\xi(\sqrt k) &\le& \sqrt{S(k)-k+\eta(k)}
\quad\text{on the event }\ \xi(\sqrt k)>0
\end{eqnarray*}
and
\begin{eqnarray*}
\{\xi(\sqrt k)\ge\sqrt k\} &=& \{S(k)-k+\eta\ge 3k\}.
\end{eqnarray*}
Then, by the Markov inequality,
\begin{eqnarray*}
\E\{\xi^{\varepsilon/2}(\sqrt k);\ \xi(\sqrt k)\ge\sqrt k\}
&\le& \E\{(S(k)-k+\eta(k))^{\varepsilon/4};\ S(k)-k+\eta(k)\ge 3k\}\\
&\le& \frac{1}{(\sqrt k)^{2-\varepsilon/2}}
\E\{S(k)-k+\eta(k);\ S(k)-k+\eta(k)\ge 3k\},
\end{eqnarray*}
hence \eqref{eps.rec.2} follows because the expectation on the right
hand side tends to zero as shown in \eqref{bp.expec.tail}.

For recurrence, it remains to prove that
$$
\frac{2m_1^{[\sqrt k]}(\sqrt k)}{m_2^{[\sqrt k]}(\sqrt k)}
\ \le\ \frac{1-\varepsilon}{\sqrt k}
$$
for all sufficiently large $k$. Since $s(x)<x$, we have
\begin{eqnarray*}
\frac{2m_1^{[\sqrt k]}(\sqrt k)}{m_2^{[\sqrt k]}(\sqrt k)}
&\le&\frac{2m_1^{[\sqrt k]}(\sqrt k)}{m_2^{[s(\sqrt k)]}(\sqrt k)}\\
&\le&\frac{2m_1^{[s(\sqrt k)]}(\sqrt k)}{m_2^{[s(\sqrt k)]}(\sqrt k)}
+\frac{2\E\{\xi(\sqrt k);\
\xi(\sqrt k)\in(s(\sqrt k),\sqrt k]\}}{m_2^{[s(\sqrt k)]}(\sqrt k)}.
\end{eqnarray*}
It follows from \eqref{1.2.G.bp.class} that
$$
\frac{2m_1^{[s(\sqrt k)]}(\sqrt k)}{m_2^{[s(\sqrt k)]}(\sqrt k)}
\ \sim\ \frac{a_\zeta+a_\eta-\sigma^2/4}{\sigma^2/4}\frac{1}{\sqrt k}.
$$
By \eqref{rec.3.1.gamma.bp.weak},
\begin{eqnarray*}
\E\{\xi(\sqrt k);\ \xi(\sqrt k)\in(s(\sqrt k),\sqrt k]\}
&\le& \sqrt k \P\{\xi(\sqrt k)>s(\sqrt k)\}
\ =\ o(1/\sqrt k),
\end{eqnarray*}
so hence the desired inequality follows because
$a_\zeta+a_\eta<\sigma^2/2$.
\qed\end{proof}

\index{Branching process!near-critical!null-recurrence}

\begin{theorem}\label{thm:bp.null.class}
Assume that \eqref{bp.expectation}--\eqref{bp.eta.majorant}
and \eqref{bp.to.inf} hold.
Let the family of random variables $\{\zeta^2(k),\ k\ge 1\}$
be uniformly integrable, that is,
\begin{eqnarray}\label{bp.uniform.integrability.class}
\sup_{k\ge 1}\ \E\{\zeta^2(k);\ \zeta(k)>t\} &\to& 0\quad\mbox{as }t\to\infty.
\end{eqnarray}
If $a_\zeta+a_\eta\in(0,\sigma^2/2)$, then the chain $\{Z_n\}$ is null-recurrent.
\end{theorem}

\begin{proof}
We apply Corollary \ref{cor:null}.
Note that the condition \eqref{bp.uniform.integrability.class}
implies fulfillment of \eqref{bp.corr1.2nd.class},
so the first two truncated moments of jumps $\xi(\sqrt k)$
satisfy the asymptotic relations \eqref{1.2.G.bp.class}.

Now let us show that the family of squares $\xi^2(\sqrt k)$
is uniformly integrable.
It follows from the definition of $\xi(\sqrt k)$ that, for all $y>0$,
\begin{eqnarray}\label{bp.null.3.class}
\P\{\xi(\sqrt k)>y\} &=& \P\{\sqrt{S(k)+\eta(k)}>\sqrt k+y\}\nonumber\\
&=&\P\{S(k)-k+\eta(k)>2\sqrt ky+y^2\}\nonumber\\
&\le& \P\{S(k)-k>\sqrt ky\}+\P\{\eta>y^2\}.
\end{eqnarray}
For the left tail, we have
\begin{eqnarray*}
\P\{\xi(\sqrt k)<-y\} &=& \P\{\sqrt{(S(k)+\eta(k))^+}<\sqrt k-y\}\\
&\le& \P\{S(k)-k+\eta(k)<-2\sqrt ky+y^2\}\\
&\le& \P\{S(k)-k<(-2\sqrt ky+y^2)/2\}+\P\{\eta>(2\sqrt ky-y^2)/2\}.
\end{eqnarray*}
Since $\xi(\sqrt k)\ge-\sqrt k$, we only have to consider
the values of $y\le\sqrt k$ in the last formula.
But for such values of $y$ we have $-2\sqrt ky+y^2\le -\sqrt ky$
and $-2\sqrt ky+y^2\le -y^2$, therefore
\begin{eqnarray*}
\P\{\xi(\sqrt k)\le -y\} &\le&
\P\{S(k)-k<-\sqrt ky/2\}+\P\{\eta>y^2/2\}.
\end{eqnarray*}
Combining this estimate with \eqref{bp.null.3.class}, we obtain
\begin{eqnarray}\label{bp.mod.xi.upper}
\P\{|\xi(\sqrt k)|\ge y\} &\le&
\P\Bigl\{\Bigl|\frac{S(k)-k}{\sqrt k}\Bigr|>y/2\Bigr\}+\P\{\sqrt{\eta}>y/2\}.
\end{eqnarray}
By the conditions \eqref{bp.expectation} and \eqref{bp.variance}
and by the uniform integrability
of $\zeta^2(k)$, the family of random variables $(S(k)-k)^2/k$
is uniformly integrable too.
The random variable $\sqrt\eta$ is square integrable.
Altogether implies uniform integrability of
the family of squares $\{\xi^2(\sqrt k),\ k\ge 1\}$.

The condition \eqref{m1.ge.cx} follows from
\eqref{rec.3.1.gamma.bp.class} and \eqref{1.2.G.bp.class},
due to $\xi(\sqrt k)\ge -\sqrt k$.
Then uniform integrability and asymptotics \eqref{1.2.G.bp.class}
allow us to apply Corollary~\ref{cor:null} in the case
$a_\zeta+a_\eta\in(0,\sigma^2/2)$ and to conclude the null recurrence of $\{X_n\}$,
and hence of $\{Z_n\}$.
\qed\end{proof}

\subsection{Convergence to $\Gamma$-distribution}

\index{Branching process!near-critical!convergence to $\Gamma$-distribution}

\begin{theorem}\label{thm:bp.1}
Assume that \eqref{bp.expectation}--\eqref{bp.eta.majorant},
\eqref{bp.corr1.2nd.class} and \eqref{bp.to.inf} hold, and that
\begin{eqnarray}\label{bp.corr1.1}
\P\{\zeta(k)>t(k)\} &\le& q(k)/k
\end{eqnarray}
for some increasing $t(x)=o(x)$ and a decreasing integrable function $q(x)$
such that the function $q(x)\sqrt x$ decreases.
If $a_\zeta+a_\eta>\sigma^2/2$ then
$$
\frac{Z_n}{n\sigma^2/4}\ =\ \frac{X_n^2}{n\sigma^2/4}
$$
converges weakly as $n\to\infty$ to a $\Gamma$-distribution with mean
$4(a_\zeta+a_\eta)/\sigma^2$ and variance $8(a_\zeta+a_\eta)/\sigma^2$.
In addition, the sequence of processes
$$
\sqrt{\frac{Z_{[nt]}}{n\sigma^2/4}},\quad t\in[0,1],
$$
converges weakly in $D[0,1]$ to a Bessel process with drift
coefficient $(2(a_\zeta+a_\eta)/\sigma^2-1/2)/x$
and unit diffusion coefficient.
\end{theorem}

A sufficient condition for \eqref{bp.corr1.1} is the existence of
a square integrable majorant $\Xi$ for the family of
random variables $\{\zeta(k),k\ge 1\}$, see Lemma \ref{l:maj.p.e}.

\begin{proof}
It is sufficient to check that the chain $\{X_n\}$ satisfies
all the conditions of Theorems \ref{thm:gamma} and \ref{thm:gamma.func}.
By Theorem \ref{thm:bp.class.1}, the chain $\{X_n\}$ is transient and
by Proposition \ref{prop:bp.class} the truncated
moments of its jumps $\xi(\sqrt k)$ satisfy \eqref{1.2.G.bp.class},
so the condition \eqref{1.2.G} follows with
$\mu=(a_\zeta+a_\eta-\sigma^2/4)/2$ and $b=\sigma^2/4$.
Then it remains to show that, for all $k$,
\begin{eqnarray}\label{rec.3.1.gamma.bp}
\P\{|\xi(\sqrt k)|>s(\sqrt k)\} &\le& p(\sqrt k)/\sqrt k,
\end{eqnarray}
which in particular implies, due to $\xi(\sqrt k)\ge -\sqrt k$, that
\begin{eqnarray*}
\E\{|\xi(\sqrt k)|;\ \xi(\sqrt k)<-s(\sqrt k)\} &\le& p(\sqrt k),
\end{eqnarray*}
where a decreasing function $p(x)>0$ is integrable at infinity.
It follows from the Fuk--Nagaev inequality \eqref{bp.NF.class.l} 
with $x=\sqrt ks(\sqrt k)$ and $y=x/2$ that
\begin{eqnarray*}
\P\{|S(k)-\E S(k)|>\sqrt k s(\sqrt k)\} &\le&
\frac{C}{s^4(\sqrt k)}+k\P\{\zeta(k)>\sqrt k s(\sqrt k)/2\}.
\end{eqnarray*}
Let us choose $s(x)=o(x)$ such that $s(x)\ge x^{3/4}$ and
$xs(x)\ge 2t(x^2)$ which is possible because $t(x)=o(x)$.
Then, by the condition \eqref{bp.corr1.1},
\begin{eqnarray}\label{ub.S.bp}
\P\{|S(k)-\E S(k)|>\sqrt k s(\sqrt k)\} &\le&
\frac{C}{(\sqrt k)^3}+k\P\{\zeta(k)> t(k)\}\nonumber\\
&\le& \frac{1}{\sqrt k}\biggl(\frac{C}{(\sqrt k)^2}+q(k)\sqrt k\biggr).
\end{eqnarray}
Together with \eqref{Ak-def.class} 
the upper bounds \eqref{ub.S.bp} and \eqref{ub.tau.bp.class} imply
\begin{eqnarray*}
\P\{|\xi(\sqrt k)|\ge s(\sqrt k)\}
&\le& \frac{1}{\sqrt k}\biggl(\frac{C}{(\sqrt k)^2}
+\widetilde q(k)\sqrt k\biggr),
\end{eqnarray*}
where $\widetilde q(x)$ is a decreasing integrable function. Since
\begin{eqnarray*}
\int_1^\infty\biggl(\frac{C}{x^2}+\widetilde q(x^2)x\biggr)dx
&=& C+\frac12\int_1^\infty \widetilde q(y)dy\ <\ \infty,
\end{eqnarray*}
the chain $\{X_n\}$ satisfies the condition \eqref{rec.3.1.gamma.bp}
and the proof is complete.
\qed\end{proof}

Assume that all the conditions of Theorem \ref{thm:bp.1}
apart from \eqref{bp.to.inf} are valid but $\P\{\eta(k)\le0\}=1$,
so the state $0$ is absorbing and the extinction probability is positive.
Denote
$$
q\ :=\ \P\{Z_n\to\infty\}\in(0,1).
$$
In parallel, let us introduce a branching process $\{\widetilde Z_n\}$
governed by the same stochastic mechanism as $\{Z_n\}$ with just one alteration:
we add a transition at zero,
if $\widetilde Z_n=0$ we put $\widetilde Z_{n+1}=1$;
this alternated chain is transient provided it is irreducible
and Theorem \ref{thm:bp.1} is applicable to it.
Since $\{\widetilde Z_n\}$ visits $0$ finitely many times,
we conclude that the distribution of $Z_n/n\sigma^2$ conditioned on
$\{Z_n\to\infty\}$ converges to the same $\Gamma$-distribution
as $\widetilde Z_n/n\sigma^2$ is converging to which implies
\begin{eqnarray}\label{bp.gamma}
\P\{Z_n/n\sigma^2 \le x\} &\to& (1-q)+q\Gamma(x)
\quad\mbox{as }n\to\infty.
\end{eqnarray}

The next result is aimed at covering the null-recurrent case.
\index{Branching process!near-critical!convergence to $\Gamma$-distribution}

\begin{theorem}\label{thm:bp.null}
Assume that \eqref{bp.expectation}--\eqref{bp.eta.majorant}
and \eqref{bp.to.inf} hold, $a_\zeta+a_\eta>0$
and there exists a decreasing function $\varepsilon(y)\to 0$ such that
\begin{eqnarray}\label{bp.majorant}
\P\{\zeta(k)>y\} &\le& \frac{\varepsilon(y)}{y^2}
\quad\mbox{for all }k\ge 1,\ y>0,
\end{eqnarray}
and
\begin{eqnarray}\label{bp.epsilon}
\int_1^\infty \frac{\varepsilon(y)}{y}dy &<& \infty,
\end{eqnarray}
then $4Z_n/n\sigma^2$ converges weakly as $n\to\infty$ 
to a $\Gamma$-distribution with mean $4(a_\zeta+a_\eta)/\sigma^2$
and variance $8(a_\zeta+a_\eta)/\sigma^2$.
In addition, the sequence of processes
$$
2\sqrt{\frac{Z_{[nt]}}{n\sigma^2}},\quad t\in[0,1],
$$
converges weakly in $D[0,1]$ to a Bessel process with drift
coefficient $(2(a_\zeta+a_\eta)/\sigma^2-1/2)/x$
and unit diffusion coefficient.
\end{theorem}

The conditions \eqref{bp.majorant}--\eqref{bp.epsilon} imply the
existence of a square integrable majorant $\Xi$ for the family of random
variables $\{\zeta(k),k\ge 1\}$, and not the other way around.
A sufficient condition for \eqref{bp.majorant}--\eqref{bp.epsilon}
is the existence of a majorant $\Xi$ such that
$\E\Xi^2\log^{1+\varepsilon}(1+\Xi)<\infty$ for some $\varepsilon>0$.
Note that we use the monotonicity
of the function $\varepsilon(y)$ when justify \eqref{bp.eps.use} below.

Also, instead of the conditions \eqref{bp.majorant}--\eqref{bp.epsilon}
we can assume existence of a majorant $\Xi$ for the family $\{\zeta(k)\}$
such that $\Xi^2\log(1+\Xi)$ is integrable, because then the function
$\E\{\Xi^2;\ \Xi>y\}/y^2$ is integrable at infinity.

\begin{proof}
Note that the condition \eqref{bp.majorant} implies that
the family $\{\zeta^2(k),k\ge 1\}$ is uniformly integrable,
hence \eqref{bp.corr1.2nd.class} holds, so the first two truncated
moments of jumps $\xi(\sqrt k)$ satisfy the asymptotic relations
\eqref{1.2.G.bp.class}. Also, by Theorem \ref{thm:bp.null.class},
the chain $\{X_n\}$ is either null recurrent or transient.

To prove convergence to a $\Gamma$-distribution,
let us check the conditions of Theorem~\ref{thm:pre-st.null}. 
Firstly, null recurrence or transience of $\{X_n\}$
implies convergence $X_n\to\infty$ in probability as $n\to\infty$.
Secondly, the sequence $|\xi(\sqrt k)|$ possesses a square-integrable
majorant $\Xi$.
Indeed, using \eqref{bp.NF.class.l} with $x=\sqrt k y/2$
and $y=x/2$ we get from \eqref{bp.mod.xi.upper} that
\begin{eqnarray*}
\P\{|\xi(\sqrt k)|\ge y\} &\le&
c/y^4+k\P\{|\zeta(k)-\E\zeta(k)|>\sqrt ky/4\}
+\P\{\sqrt{\eta}>y/2\}.
\end{eqnarray*}
Since $\zeta(k)\ge 0$ and the sequence $\E\zeta(k)$ is bounded,
there exists an $y_0$ such that
\begin{eqnarray*}
\P\{|\xi(\sqrt k)|\ge y\}\ \le\
c/y^4+k\P\{\zeta(k)>\sqrt ky/2\}+\P\{\sqrt{\eta}>y/2\}
\quad\mbox{for }y\ge y_0.
\end{eqnarray*}
Due to \eqref{bp.majorant} and monotonicity of the function $\varepsilon(y)$,
\begin{eqnarray}\label{bp.eps.use}
\P\{\zeta(k)>\sqrt ky/2\} &\le& 4\frac{\varepsilon(\sqrt ky/2)}{ky^2}
\ \le\ 4\frac{\varepsilon(y/2)}{ky^2},
\end{eqnarray}
hence
\begin{eqnarray*}
\P\{|\xi(\sqrt k)|\ge y\}\ \le\
c/y^4+4\varepsilon(y/2)/y^2+\P\{\sqrt{\eta}>y/2\}
\quad\mbox{for }y\ge y_0.
\end{eqnarray*}
Let $\Xi$ be a random variable taking values in $[y_0,\infty)$ such that
\begin{eqnarray*}
\P\{\Xi>y\} &=&
\min\{1,\ c/y^4+4\varepsilon(y/2)/y^2+\P\{\sqrt{\eta}>y/2\}\}
\quad\mbox{for }y\ge y_0.
\end{eqnarray*}
Clearly, $\Xi$ is a stochastic majorant for the sequence $\xi(\sqrt k)$.
The finiteness of $\E\Xi^2$ follows from the condition \eqref{bp.epsilon}
and the assumption $\E\eta<\infty$.

So it only remains to determine the asymptotic behaviour of the first two full
moments of jumps, that is, of $m_1(\sqrt k)$ and $m_2(\sqrt k)$.
We know from the proof of Theorem \ref{thm:bp.1} that
\begin{eqnarray*}
m_1^{[s(\sqrt k)]}(\sqrt k) &\sim&
\frac{a_\zeta+a_\eta-\sigma^2/4}{2\sqrt k}\quad\mbox{as }k\to\infty,
\end{eqnarray*}
for any $s(x)$ such that $s(x)/x\to0$ sufficiently slow.
From the existence of the majorant we infer that
\begin{eqnarray*}
\E\{|\xi(\sqrt k)|;\ |\xi(\sqrt k)|\ge s(\sqrt k)\}
&\le& \frac{1}{s(\sqrt k)}\E\{\Xi^2;\ \Xi\ge s(\sqrt k)\}.
\end{eqnarray*}
Consequently, we can choose $s(x)=o(x)$ such that
\begin{eqnarray*}
\E\{|\xi(\sqrt k)|;\ |\xi(\sqrt k)|>s(\sqrt k)\} &=& o(1/\sqrt k)
\quad\mbox{as }k\to\infty.
\end{eqnarray*}
This yields asymptotics
\begin{eqnarray*}
m_1(\sqrt k) &\sim& \frac{a_\zeta+a_\eta-\sigma^2/4}{2\sqrt k}
\quad\mbox{as }k\to\infty.
\end{eqnarray*}
Existence of a square-integrable majorant also gives convergence
$m_2(\sqrt k)\to \sigma^2/4$ as $k\to\infty$.
Thus, weak convergence of $Z_n/n\sigma^2$ to a
$\Gamma$-distribution now follows from Theorem~\ref{thm:pre-st.null}
and functional convergence follows from Theorem \ref{thm:gamma.func}.
\qed\end{proof}

\subsection{Tail asymptotics for non-extinction probability
of recurrent branching processes}

Basic topics in the theory of critical and near-critical
recurrent branching processes are the asymptotic behaviour of
the non-extinction probability and the limiting behaviour of the process
conditioned on the non-extinction. Let us demonstrate that
corresponding results for general Markov chains---Theorem
\ref{thm:rec.time} and Corollary~\ref{cor:rec.tail.gen}---may be
applied to near-critical branching processes. For that,
we have to find restrictions on $\zeta(k)$, $\eta(k)$, and $\eta$ which guarantee
fulfillment of \eqref{m1.and.m2.new}--\eqref{r.p.prime}
and \eqref{cond.xi.le}--\eqref{cond.3.moment}.

The hardest task, from the technical point of view,
consists in finding a regular function $r(x)$
such that \eqref{r-cond.2.new} takes place.
In what follows we concentrate on the case when one can take $r(x)=c/x$.

We first prove a refined version of Proposition \ref{prop:bp.class}
where we assume refined versions of the conditions
\eqref{bp.expectation}--\eqref{bp.eta.majorant}
on the moments of $\zeta(k)$'s and $\eta(k)$'s.
Hereinafter we consider $s(x)=x/\log^{1+\varepsilon}x$.

\begin{proposition}\label{prop:bp.class.new}
Let, for some $\varepsilon>0$,
\begin{eqnarray}\label{bp.2nd.moments}
\sup_{k\ge 1}\ \E\zeta^2(k)\log^{3+3\varepsilon}(1+\zeta(k))
&<& \infty,
\end{eqnarray}
let the majorisation condition \eqref{bp.eta.majorant} hold with $\eta$ satisfying 
\begin{eqnarray}\label{eta.mom}
\E\eta\log^{1+2\varepsilon}(1+\eta) &<& \infty,
\end{eqnarray}
and let there exist a decreasing integrable at infinity function $v(x)$
such that $xv(x^2)$ is decreasing too and, as $k\to\infty$,
\begin{eqnarray}\label{bp.prop.new.1}
\E\zeta(k) &=& 1+a_\zeta/k+o(v(k)),\\
\label{bp.prop.new.eta}
\E\eta(k) &=& a_\eta+o(kv(k)),\\
\label{bp.prop.new.2}
\V\zeta^2(k) &=& \sigma^2+o(kv(k)),\\
\label{bp.prop.new.3}
\E\{\zeta^3(k);\ \zeta(k)\le k\} &\le& k^2v(k).
\end{eqnarray}
Then, for $s(x)=x/\log^{1+\varepsilon}x$, there exists
a decreasing integrable function $p(x)$ such that
\begin{eqnarray}\label{bp.prop.new.4}
m_1^{[s(\sqrt k)]}(\sqrt k) &=&
\frac{a_\zeta+a_\eta-\sigma^2/4}{2\sqrt k}+o(p(\sqrt k)),\\
\label{bp.prop.new.5}
m_2^{[s(\sqrt k)]}(\sqrt k) &=& \sigma^2/4+o(\sqrt kp(\sqrt k))
\quad\mbox{as }k\to\infty.
\end{eqnarray}
\end{proposition}

\begin{proof}
Due to the condition \eqref{bp.2nd.moments},
the condition \eqref{bp.corr1.2nd.class} is valid with any $t(k)\to\infty$.
Take $t(k)=\sqrt k$. 
Then it follows from the upper bound \eqref{bp.NF.2}
with $s(x)=x/\log^{1+\varepsilon} x$
and the condition \eqref{bp.2nd.moments} that
\begin{eqnarray}\label{bp.lemma.new.1}
\lefteqn{\E\{(S(k)-k)^2;\ |S(k)-k|>\sqrt ks(\sqrt k)\}}\nonumber\\
&&\hspace{10mm}\le\ C\Bigl(\log^{2+2\varepsilon}k+
k\E\{\zeta^2(k);\ \zeta(k)>\sqrt ks(\sqrt k)\}\Bigr)\nonumber\\
&&\hspace{20mm}\le\ C\Bigl(\log^{2+2\varepsilon}k+
k\frac{\E\zeta^2(k)\log^{3+3\varepsilon}(1+\zeta(k))}
{\log^{3+3\varepsilon}(1+\sqrt k)}\Bigr)\nonumber\\
&&\hspace{30mm}=\ O(k/\log^{3+3\varepsilon}k)\quad\mbox{as }k\to\infty.
\end{eqnarray}
Therefore,
\begin{eqnarray}\label{bp.lemma.new.2}
\E\{|S(k)-k|;\ |S(k)-k|>\sqrt ks(\sqrt k)\} &\le&
\frac{\E\{(S(k)-k)^2;\ |S(k)-k|>\sqrt ks(\sqrt k)\}}{\sqrt k s(\sqrt k)}\nonumber\\
&=& O(1/\log^{2+2\varepsilon}k)
\ \mbox{as }k\to\infty.
\end{eqnarray}

By Taylor's expansion,
\begin{eqnarray}\label{bp.prop.new.6}
\xi(\sqrt k) &=& \frac{S(k)-k+\eta(k)}{2\sqrt k}
-\frac{1}{8}\frac{(S(k)-k+\eta(k))^2}{k\sqrt k}+
\theta\frac{(S(k)-k+\eta(k))^3}{k^2\sqrt k},\nonumber\\[-2mm]
\end{eqnarray}
where $\theta=\theta((S(k)-k+\eta(k))/\sqrt k)$ is bounded on the event $A_k$
defined in \eqref{Ak.bp}.
Let us estimate the expectation of every term in \eqref{bp.prop.new.6}.

Recalling from \eqref{c.Ak.bp.1} that
$A_k^c\subseteq\{|S(k)-k+\eta(k)|>\sqrt ks(\sqrt k)\}$
for all $k$ sufficiently large, we obtain
\begin{eqnarray*}
\biggl|\E\Bigl\{\frac{S(k)-k+\eta(k)}{\sqrt k};\ A_k\Bigr\}
-\frac{a_\zeta+a_\eta}{\sqrt k}\biggr|
&\le& \biggl|\frac{\E S(k)-k-a_\zeta}{\sqrt k}\biggr|
+\biggl|\frac{\E\eta(k)-a_\eta}{\sqrt k}\biggr|\\
&&\hspace{-30mm}+\frac{1}{\sqrt k}\E\{|S(k)-k+\eta(k)|;\ |S(k)-k+\eta(k)|>\sqrt ks(\sqrt k)\}.
\end{eqnarray*}
The first two terms on the right hand side are of order $o(\sqrt kv(k))$
by the conditions \eqref{bp.prop.new.1} and \eqref{bp.prop.new.eta}.
Taking also into account the upper bounds
\eqref{first.ge.deco} and \eqref{bp.lemma.new.2} we derive
\begin{eqnarray*}
\biggl|\E\Bigl\{\frac{S(k)-k+\eta(k)}{\sqrt k};\ A_k\Bigr\}
-\frac{a_\zeta+a_\eta}{\sqrt k}\biggr| &\le&
o(\sqrt kv(k))+O(1/\sqrt k\log^{2+2\varepsilon}k)\\
&&\hspace{-25mm}+\frac{1}{\sqrt k}\Bigl(
\E|S(k)-k|\P\{\eta>s^2(\sqrt k)\}
+\E\{\eta;\ \eta>s^2(\sqrt k)\}\Bigr).
\end{eqnarray*}
The assumption \eqref{eta.mom} implies that
\begin{eqnarray*}
\frac{1}{\sqrt k}\E\{\eta;\ \eta>s^2(\sqrt k)\}
&=& \frac{1}{\sqrt k}\E\{\eta;\ \eta>k/\log^{2+2\varepsilon}\sqrt k\}\\
&=& o(1/\sqrt k\log^{1+\varepsilon}k)\quad\mbox{as }k\to\infty,
\end{eqnarray*}
hence, by the Markov inequality,
\begin{eqnarray*}
\P\{\eta>s^2(\sqrt k)\}
&\le& \E\{\eta;\ \eta>s^2(\sqrt k)\}/s^2(\sqrt k)\\
&=& o((\log^{1+\varepsilon}k)/k)\quad\mbox{as }k\to\infty,
\end{eqnarray*}
Combining this with the upper bound
$\E|S(k)-k|=O(\sqrt k)$, we conclude that
\begin{eqnarray}\label{bp.prop.new.8}
\biggl|\E\Bigl\{\frac{S(k)-k+\eta(k)}{\sqrt k};\ A_k\Bigr\}
-\frac{a_\zeta+a_\eta}{\sqrt k}\biggr|
&\le& o(\sqrt kv(k))+O(1/\sqrt k\log^{1+\varepsilon}k)\nonumber\\
&=& o(p(\sqrt k)),
\end{eqnarray}
where the function
\begin{eqnarray}\label{px.def.bp}
p(x) &:=& xv(x^2)+1/x\log^{1+\varepsilon/2}x
\end{eqnarray}
is decreasing and integrable at infinity because $\varepsilon>0$ and
$$
\int_1^\infty xv(x^2)dx\ =\ \frac12 \int_1^\infty v(y)dy\ <\ \infty.
$$

For the second term on the right hand side of \eqref{bp.prop.new.6},
we have
\begin{eqnarray}\label{bp.prop.new.10}
\nonumber
\Bigl|\E\Bigl\{\frac{(S(k)-k+\eta(k))^2}{k};\ A_k\Bigr\}-\sigma^2\Bigr|
&\le& \Bigl|\frac{\E(S(k)-k)^2}{k}-\sigma^2\Bigr|\\
&&\hspace{-50mm} +\ \frac{\E\{(S(k)-k)^2;\ A_k^c\}}{k}
+2\Bigl|\frac{\E\{(S(k)-k)\eta(k);\ A_k\}}{k}\Bigr|
+\frac{\E\{\eta^2(k);\ A_k\}}{k}\nonumber\\
&=:& E_1+E_2+E_3+E_4.
\end{eqnarray}
The first term on the right hand side may be bounded as follows:
\begin{eqnarray}\label{bp.prop.new.11}
E_1 &=& \Bigl|\V\zeta(k)-\sigma^2+\frac{(\E S(k)-k)^2}{k}\Bigr|
\ \le\ o(kv(k))+O(1/k),
\end{eqnarray}
by the conditions \eqref{bp.prop.new.1} and \eqref{bp.prop.new.2}.
Using \eqref{c.Ak.bp.1}, we obtain
\begin{eqnarray*}
E_2 &\le& \frac{\E\{(S(k)-k)^2;\ |S(k)-k|>\sqrt ks(\sqrt k)\}}{k}
+\frac{\E\{(S(k)-k)^2;\ |\eta(k)|>s^2(\sqrt k)\}}{k}\\
&\le& O(1/\log^{3+3\varepsilon}k)+c_2\P\{\eta>s^2(\sqrt k)\}
\quad\mbox{as }k\to\infty,
\end{eqnarray*}
by the upper bound \eqref{bp.lemma.new.1}
and independence of $S(k)$ and $\eta(k)$. 
Therefore, by the condition \eqref{eta.mom},
\begin{eqnarray}\label{bp.prop.new.12}
E_2 &=& O(1/\log^{3+3\varepsilon}k).
\end{eqnarray}
By \eqref{bp.prop.5},
\begin{eqnarray}\label{bp.prop.new.13}
E_3 &=& O(1/\sqrt k)\quad\mbox{as }k\to\infty.
\end{eqnarray}
Finally, due to the condition \eqref{eta.mom}
we deduce similarly to \eqref{bp.prop.6} that
\begin{eqnarray}\label{bp.prop.new.14}
E_4 &=& o(1/\log^{1+\varepsilon}k)\quad\mbox{as }k\to\infty.
\end{eqnarray}
Combining \eqref{bp.prop.new.10}--\eqref{bp.prop.new.14}, we obtain
\begin{eqnarray}\label{bp.prop.new.15}
\frac{\E\{(S(k)-k+\eta(k))^2;\ A_k\}}{k\sqrt k} &=&
\frac{\sigma^2}{\sqrt k}+o(p(\sqrt k)),
\end{eqnarray}
where $p(x)$ is defined in \eqref{px.def.bp}.

As follows from the definition of $A_k$, see \eqref{Ak.bp},
$$
|S(k)-k+\eta(k)|\ \le\ 3\sqrt ks(\sqrt k)\quad\mbox{on the event }A_k,
$$
hence the remainder term in \eqref{bp.prop.new.6}
possesses the following upper bound:
\begin{eqnarray}\label{bp.prop.new.17}
\E\{|S(k)-k+\eta(k)|^3;\ A_k\} &\le&
3\sqrt ks(\sqrt k)\E\{(S(k)-k+\eta(k))^2;\ A_k\}\nonumber\\
&=& O(k\sqrt ks(\sqrt k))\quad\mbox{as }k\to\infty,
\end{eqnarray}
as follows from \eqref{bp.prop.new.15}.

Combining \eqref{bp.prop.new.6}, \eqref{bp.prop.new.8},
\eqref{bp.prop.new.15} and \eqref{bp.prop.new.17}, we colclude that
\begin{eqnarray}
\E\{\xi(\sqrt k);\ A_k\} &=&
\frac{a_\zeta+a_\eta-\sigma^2/4}{2\sqrt k}+o(p(\sqrt k)),
\end{eqnarray}
where $p(x)$ is defined \eqref{px.def.bp},
so \eqref{bp.prop.new.4} is proven.

In order to prove \eqref{bp.prop.new.5} we first use
Taylor's expansion for the function 
$$
(\sqrt{1+u}-1)^2\ =\ \frac{u^2}{4}-\frac{u^3}{6}\frac{3}{4(1+\theta_1u)^{5/2}},
\quad\theta_1\in(0,1), 
$$
to conclude
$$
\xi^2(\sqrt k)\ =\
\frac{(S(k)-k+\eta(k))^2}{4k}+\widetilde\theta\frac{(S(k)-k+\eta(k))^3}{k^2},
$$
where $\widetilde\theta=\widetilde\theta(S(k),\eta(k))$
is bounded on the event $A_k$.
Then we apply \eqref{bp.prop.new.15} and \eqref{bp.prop.new.17}
to conclude \eqref{bp.prop.new.5}.
\qed\end{proof}

Under the conditions of Proposition \ref{prop:bp.class.new}
we have that, with $s(x)=x/\log^{1+\varepsilon}x$,
$$
\frac{2m_1^{[s(x)]}(x)}{m_2^{[s(x)]}(x)}\ =\
-\frac{\sigma^2/4-a_\zeta-a_\eta}{\sigma^2/4}\ \frac{1}{x}+o(p(x))
\quad\mbox{as }x\to\infty.
$$
This means that \eqref{r-cond.2.new} holds with
$$
r(x)\ =\ \frac{\rho-1}{1+x},
$$
where
$$
\rho\ =\ \frac{\sigma^2/2-a_\zeta-a_\eta}{\sigma^2/4}.
$$
\index{Branching process!near-critical!convergence to $\Gamma$-distribution}

\begin{theorem}\label{thm:bp.conditioned}
Assume that all the conditions of Proposition \ref{prop:bp.class.new}
are valid and that $a_\zeta+a_\eta<\sigma^2/2$.

Assume that $\E\eta\log^{3+3\varepsilon}(1+\eta)<\infty$
and $\E\eta^{\rho/2}<\infty$. Assume that
\begin{eqnarray}\label{bp.cond.1}
\sup_{k\ge 1}\E \zeta^{\rho/2}(k) &<& \infty.
\end{eqnarray}

Then, for each starting state $z$,
\begin{eqnarray}\label{bp.cond.2}
\P_z\{Z_k>z_*\mbox{ for all }k\le n\} &\sim& \frac{c(z)}{n^{\rho/2}}
\quad\mbox{as }n\to\infty
\end{eqnarray}
and, for all $u>0$,
\begin{eqnarray}\label{bp.cond.3}
\P_z\Bigl\{\frac{2Z_n}{n\sigma^2}>u\ \Big|\ Z_k>z_* \mbox{ for all }k\le n\Bigr\}
&\to& e^{-u}\quad\mbox{as }n\to\infty,
\end{eqnarray}
where $z_*$ is the minimal accessible state of $\{Z_n\}$.
\end{theorem}

It is easy to see that if $\P\{\zeta(k)=0\}>0$ and
$\P\{\eta(k)\le 0\}>0$ for all $k$ then $z_*=0$.
Furthermore, if $\P\{\eta(k)\le0\}=1$ then $0$ is an absorbing state
and we have typical for branching processes statements:
\begin{eqnarray*}
\P_z\{Z_n>0\} &\sim& c(z)/n^{\rho/2}\quad\mbox{as }n\to\infty
\end{eqnarray*}
and, for all $u>0$,
\begin{eqnarray*}
\P_z\left\{\frac{2Z_n}{n\sigma^2}>u\ \Big|\ Z_n>0 \right\} &\to& e^{-u}
\quad\mbox{as }n\to\infty.
\end{eqnarray*}

\begin{proof}%[Proof of Theorem \ref{thm:bp.conditioned}]
We again put $s(x)=x/\log^{1+\varepsilon}x$ and check
sufficient conditions for results
from Section \ref{sec:cond.on.recurrent.time}.
We start with the following auxiliary upper bound, for all $\rho>0$,
\begin{eqnarray}\label{rho.half.S}
\E\{(S(k)-\E S(k))^{\rho/2};\ S(k)-\E S(k)>\sqrt ks(\sqrt k)\}
&=& (\sqrt k)^{\rho-1} o(q(\sqrt k))\nonumber\\
\end{eqnarray}
as $k\to\infty$, for some decreasing integrable at infinity function $q(x)$.
By Lemma \ref{M-Z}, for all $\rho>0$, $\E|S(k)-\E S(k)|^{\rho/2}=O(k^{\rho/4})$
and hence $\E|S(k)-k|^{\rho/2}=O(k^{\rho/4})$.
Then, by the condition \eqref{bp.cond.1},
\begin{eqnarray*}
\E\{(S(k)-\E S(k))^{\rho/2};\ S(k)-\E S(k)>\sqrt ks(\sqrt k)\}
&\le& \E|S(k)-\E S(k)|^{\rho/2}\\
&=& (\sqrt k)^{\rho-1}
O(1/(\sqrt k)^{\rho/2-1})\\
&=& (\sqrt k)^{\rho-1} o(q_1(\sqrt k))
\end{eqnarray*}
as $k\to\infty$, for some decreasing integrable at infinity function $q_1(x)$,
provided $\rho>4$.
If $\rho\in(0,4]$ then, by the Chebyshev-type inequality 
and by the upper bound \eqref{bp.NF.2},
\begin{eqnarray*}
\lefteqn{\E\{(S(k)-\E S(k))^{\rho/2};\ S(k)-\E S(k)>\sqrt ks(\sqrt k)\}}\\
&&\hspace{5mm}\le\ \frac{\E\{(S(k)-\E S(k))^2;\ S(k)-\E S(k)>\sqrt ks(\sqrt k)\}}
{(\sqrt ks(\sqrt k))^{2-\rho/2}}\\
&&\hspace{10mm}\le\ \frac{C}{(\sqrt ks(\sqrt k))^{2-\rho/2}}\biggl[
\frac{1}{s^2(\sqrt k)}+k\E\{\zeta^2(k);\ \zeta(k)>\sqrt ks(\sqrt k)/2\}\biggr].
\end{eqnarray*}
Applying the condition \eqref{bp.2nd.moments} we conclude that
\begin{eqnarray}\label{aux.S.rho.2}
\lefteqn{\E\{(S(k)-\E S(k))^{\rho/2};\ S(k)-\E S(k)>\sqrt ks(\sqrt k)\}}\nonumber\\
&&\hspace{10mm}\le\ \frac{C_1}{(\sqrt ks(\sqrt k))^{2-\rho/2}}
\frac{k}{\log^{3+3\varepsilon}\sqrt k}\nonumber\\
&&\hspace{20mm}=\ 
\frac{C_1}{k^{1-\rho/2}\log^{(1+\rho/2)(1+\varepsilon)}\sqrt k}\\
&&\hspace{30mm}=\ (\sqrt k)^{\rho-1} o(q_2(\sqrt k))\quad\mbox{as }k\to\infty,\nonumber
\end{eqnarray}
which completes the proof of \eqref{rho.half.S} for all $\rho>0$.

In Proposition \ref{prop:bp.class.new} we have checked
the condition \eqref{r-cond.2.new} for the chain $\{\sqrt{Z_n}\}$.
The fulfilment of the condition \eqref{cond.xi.le} for the left tail
was proven in Proposition \ref{prop:bp.class}. 
For the right tail, it is enough to notice that, by the Chebyshev inequality
and by the upper bound \eqref{aux.S.rho.2} with $\rho=2$,
\begin{eqnarray*}
\P\{S(k)-\E S(k)>\sqrt ks(\sqrt k)\}
&\le& \frac{\E\{(S(k)-\E S(k))^2;\ S(k)-\E S(k)>\sqrt ks(\sqrt k)\}}
{(\sqrt ks(\sqrt k))^2}\\
&=& O(1/k\log^{1+\varepsilon}k)\quad\mbox{as }k\to\infty.
\end{eqnarray*}
So it only remains to validate the conditions 
\eqref{cond.for.U.unif.52.local}, \eqref{cond.3.moment} and
\eqref{cond.xi.ge} under
the assumptions of Theorem~\ref{thm:bp.conditioned}.

Since $r(x)=\frac{\rho-1}{1+x}$, the function $U(x)$ is asymptotically
equivalent to $cx^\rho$ with some positive constant $c$.
Thus, we can replace $U(x)$ by $x^\rho$ in \eqref{cond.for.U.unif.52.local}
and \eqref{cond.xi.ge}.
In particular, then \eqref{cond.for.U.unif.52.local}
follows from \eqref{cond.xi.ge}.

We start with \eqref{cond.3.moment}. It follows from the upper bound
\begin{eqnarray*}
|\xi(\sqrt k)| &\le& \frac{|S(k)-k+\eta(k)|}{\sqrt k}
\end{eqnarray*}
and \eqref{bp.prop.new.17} that
\begin{eqnarray*}
\E\{|\xi(\sqrt k)|^3;\ |\xi(\sqrt k)|\le s(\sqrt k)\}
	&=& O(s(\sqrt k))\ =\ O((\sqrt k)^2/\sqrt k\log^{1+\varepsilon}k).
\end{eqnarray*}
This implies \eqref{cond.3.moment} with,
say $p(x)=1/x\log^{1+\varepsilon/2}x$.

Let us now check fulfillment of \eqref{cond.xi.ge}. 
First we note that, due to the concavity of the root function,
\begin{eqnarray}\label{bp.cond.sum.1}
\lefteqn{\E\{\xi^\rho(\sqrt k);\ \xi(\sqrt k)>s(\sqrt k)\}}\nonumber\\
&\le& \E\{(S(k)-k+\eta(k))^{\rho/2};\ \xi(\sqrt k)>s(\sqrt k)\}\nonumber\\
&\le& \E\{(S(k)-k+\eta(k))^{\rho/2};\ S(k)+\eta(k)>(\sqrt k+s(\sqrt k))^2\}\nonumber\\
&\le& \E\{(S(k)-k+\eta(k))^{\rho/2};\ S(k)-k>\sqrt ks(\sqrt k)\}\nonumber\\
&&\hspace{20mm}+\ \E\{(S(k)-k+\eta(k))^{\rho/2};\ \eta(k)>s^2(\sqrt k)\}.
\end{eqnarray}
Owing to the independence of $S(n)$ and $\eta(k)$,
the first expectation on the right hand side is not greater,
up to a constant factor, than the sum
\begin{eqnarray*}
\lefteqn{\E\eta^{\rho/2}\P\{S(k)-\E S(k)>\sqrt ks(\sqrt k)\}}\\
&&\hspace{20mm}+\E\{(S(k)-\E S(k))^{\rho/2};\ S(k)-\E S(k)>\sqrt ks(\sqrt k)\}\\
&\le& c \E\{(S(k)-\E S(k))^{\rho/2};\ S(k)-\E S(k)>\sqrt ks(\sqrt k)\},\quad c<\infty,
\end{eqnarray*}
due to the condition $\E\eta^{\rho/2}<\infty$.
Then it follows from \eqref{rho.half.S} that
\begin{eqnarray}\label{bp.cond.sum.2}
\E\{(S(k)-k+\eta(k))^{\rho/2};\ S(k)-k>\sqrt ks(\sqrt k)\} &=&
(\sqrt k)^{\rho-1} o(q(\sqrt k))\quad\mbox{as }k\to\infty.\nonumber\\
\end{eqnarray}

The second expectation on right hand side of \eqref{bp.cond.sum.1}
is not greater, up to a constant factor, than the sum
\begin{eqnarray*}
\E (S(k)-k)^{\rho/2}\P\{\eta>s^2(\sqrt k)\}
+\E\{\eta^{\rho/2};\ \eta>s^2(\sqrt k)\},
\end{eqnarray*}
owing to the independence of $S(n)$ and $\eta$.
Again by Lemma \ref{M-Z}, 
\begin{eqnarray*}
\E|S(k)-k|^{\rho/2} &=& O((\sqrt k)^{\rho/2})
\ =\ O((\sqrt k)^{\rho-1}/(\sqrt k)^{\rho/2-1})\quad\mbox{for all }\rho>0.
\end{eqnarray*}
For all $\rho>0$,
\begin{eqnarray*}
\P\{\eta>s^2(\sqrt k)\} &\le&
\frac{(\E\eta)\log^{2+2\varepsilon}\sqrt k}{k},
\end{eqnarray*}
so
\begin{eqnarray*}
\E (S(k)-k)^{\rho/2}\P\{\eta>s^2(\sqrt k)\}
&=& (\sqrt k)^{\rho-1}O\Bigl(\frac{\log^{2+2\varepsilon}\sqrt k}{(\sqrt k)^{\rho/2+1}}\Bigr)\\
&=& (\sqrt k)^{\rho-1}o(q_3(\sqrt k))\quad\mbox{ as }k\to\infty.
\end{eqnarray*}
If $\rho>1$ then, due to the condition $\E\eta^{\rho/2}<\infty$,
\begin{eqnarray*}
\E\{\eta^{\rho/2};\ \eta>s^2(\sqrt k)\} &=& o(1)
\ =\ (\sqrt k)^{\rho-1} o(1/(\sqrt k)^\rho)\\
&=& (\sqrt k)^{\rho-1} o(q_4(\sqrt k))\quad\mbox{ as }k\to\infty.
\end{eqnarray*}
If $\rho\in(0,1]$ then, due to the condition $\E\eta\log^{3+3\varepsilon}(1+\eta)<\infty$,
\begin{eqnarray*}
\E\{\eta^{\rho/2};\ \eta>s^2(\sqrt k)\} &\le&
\frac{\E\{\eta\log^{3+3\varepsilon}\eta;\ \eta>s^2(\sqrt k)\}}
{(s^2(\sqrt k))^{1-\rho/2}\log^{3+3\varepsilon}s^2(\sqrt k)}\\
&=& (\sqrt k)^{\rho-1} o\Bigl(\frac{\log^{(2-\rho)(1+\varepsilon)}\sqrt k}
{\sqrt k\log^{3+3\varepsilon}(\sqrt k)}\Bigr)\\
&=& (\sqrt k)^{\rho-1} o(q_5(\sqrt k))\quad\mbox{as }k\to\infty.
\end{eqnarray*}
Altogether implies that
\begin{eqnarray}\label{bp.cond.sum.3}
\E\{(S(k)+\eta(k))^{\rho/2};\ \eta(k)>s^2(\sqrt k)\}
&=& (\sqrt k)^{\rho-1} o(q(\sqrt k)).
\end{eqnarray}
Substituting \eqref{bp.cond.sum.2} and \eqref{bp.cond.sum.3}
into \eqref{bp.cond.sum.1} we get \eqref{cond.xi.ge}.

Relations \eqref{bp.cond.2} and \eqref{bp.cond.3} now follow from
Corollaries \ref{cor:rec.tail.gen} and \ref{Cor:cond.rec.tail}.
\qed\end{proof}

The processes $\{Z_n\}$ and $\{Y_n\}$---defined in 
\eqref{BP1} and \eqref{BP2} respectively---are formally different.
But it is intuitively clear that the difference in their
definitions should have no influence on their asymptotic behaviour.
Let us show how, in the case of identically distributed $\zeta(k)$ 
and non-positive $\eta$, 
one can transfer asymptotics for one process into corresponding
asymtotics for another one. Indeed, if we define
$$
W_{2k+1}:=(W_{2k}+\eta_{k+1})^+,\quad W_{2k+2}
=\sum_{i=1}^{W_{2k+1}}\zeta_{k+1,i},\quad k\ge0,
$$
then $Y_0=W_0=m$ implies that $Y_n=W_{2n}$ and $Z_n=W_{2n+1}$
with $Z_0=(m+\eta_1)^+$.
In the case of emigration process---where $\P\{\eta\le0\}=1$---we have
that the sequence of events $\{W_k=0\}$ is increasing.
If \eqref{bp.cond.2} is valid for every fixed starting point
$Z_0$ then it is also valid for $Z_0=(m+\eta)^+$. As a result, we have
$$
\P\{Y_n>0\mid Y_0=m\}\sim\sum_{j=1}^m\P\{m+\eta=j\}
\P\{Z_n>0\mid Z_0=j\}\sim c(m)n^{-\rho/2}.
$$
Furthermore, let $\{Z_n=W_{2n+1}\}$ satisfy the conditions 
of Theorem \ref{thm:bp.conditioned}.
Then it follows that
\begin{eqnarray}\label{Z.k.cond.}
\P\{Z_n\le k\mid Z_n>0\} &\to& 0\quad\mbox{as }n\to\infty,\mbox{ for all }k>0.
\end{eqnarray}
Recalling that $Y_n=\sum_{i=1}^{Z_{n-1}}\zeta_{n,i}$, it implies
the following version of the weak law of large numbers:
\begin{eqnarray}\label{Z.k.cond}
\P\Bigl\{\Bigl|\frac{Y_n}{Z_{n-1}}-1\Bigr|>\varepsilon\ \Big|\ Z_{n-1}>0\Bigr\} &\to& 0
\quad\mbox{as }n\to\infty,\mbox{ for all }\varepsilon>0.
\end{eqnarray}
This yields, due to \eqref{bp.cond.3}, 
\begin{eqnarray*}
\P\Bigl\{\frac{2Y_n}{n\sigma^2}>u\ \Big|\ Z_{n-1}>0\Bigr\}
&\sim& \P_z\Bigl\{\frac{2Z_n}{n\sigma^2}>u\ \Big|\ Z_{n-1}>0\Bigr\},\quad u>0.
\end{eqnarray*}
We also have inequalities
\begin{eqnarray*}
\P\{Z_{n-1}>0\}\ \ge\ \P\{Y_n>0\} &\ge& \E[1-\P^{Z_{n-1}}\{\zeta=0\}].
\end{eqnarray*}
Combining this with \eqref{Z.k.cond} we conclude that
\begin{eqnarray*}
\P\{Z_{n-1}>0\} &\sim&\P\{Y_n>0\}\quad\mbox{as }n\to\infty.
\end{eqnarray*}
Therefore,
\begin{eqnarray*}
\P\Bigl\{\frac{2Y_n}{n\sigma^2}>u\ \Big|\ Y_n>0\Bigr\}
&\to& e^{-u},\quad u>0.
\end{eqnarray*}

If $\inf_k\P\{\eta(k)>0\}>0$ then $0$ is not absorbing and,
consequently, $\{Z_n\}$ is irreducible.
Then we can apply Theorem~\ref{thm:pi.recurrent} to $\sqrt{Z_n}$
and derive the tail behaviour
of the stationary measure of $\{Z_n\}$: for any constants $a<b$ we have
$$
\pi_Z(ak,bk)\sim C\int_{\sqrt{ak}}^{\sqrt{bk}}y^{1-\rho}dy\
\text{ as }k\to\infty.
$$

It follows from Theorem~\ref{thm:bp.conditioned} that $\{Z_n\}$
is positive recurrent when $\rho>2$.
In this case we may apply also Theorem~\ref{thm:pre-st.pos.rec.}
and obtain tail asymptotics for $Z_n$.

If $\rho\in(0,2)$ then the pre-limiting behaviour of $Z_n$
is described in Theorem~\ref{thm:pre-st.null}. If $\rho=2$ then,
due to \eqref{bp.cond.2}, $\{Z_n\}$ is also
null-recurrent but its behaviour is not covered by
Theorem~\ref{thm:pre-st.null}.
Here we can apply Theorem~\ref{thm:pre-st.log}.
Since $G(x)\sim\log x$ under the  assumptions of
Theorem~\ref{thm:bp.conditioned}, we conclude that
\begin{eqnarray}\label{bp.log-scaled}
\lim_{n\to\infty}\P\left\{\frac{\log Z_n}{\log n}\le x\right\}=x,\quad x\in[0,1].
\end{eqnarray}

\section{Cram\'er--Lundberg risk processes with level-dependent premium rate}
\sectionmark{Risk processes}
\label{sec:risk}

In context of the collective theory of risk,
the classical {\it Cram\'er--Lundberg model}
\index{Cram\'er--Lundberg!classical model}
%(or the {\it compound Poisson model})
is defined as follows. An insurance company receives
the constant inflow of premium at rate $v$, that is,
the premium income is assumed to be linear in time with rate $v$.
It is also assumed that the claims
incurred by the insurance company arrive according to
%a homogeneous Poisson process $N_t$ with intensity $\lambda$
a homogeneous renewal process $N(t)$ with intensity $\lambda$
and the sizes (amounts) $\xi_n\ge 0$ of the claims
are independent copies of a random variable $\xi$ with finite mean $b$.
The $\xi$'s are assumed independent of the process $N(t)$.
The company has an initial risk reserve $x=R(0)\ge0$.

Then the risk reserve $R(t)$ at time $t$ is equal to
\begin{eqnarray*}
R(t) &=& x+vt-\sum_{i=1}^{N(t)}\xi_i.
\end{eqnarray*}
The probability
\begin{eqnarray*}
\P\{R(t)\ge 0\mbox{ for all }t\ge0\}
&=& \P\Bigl\{\min_{t\ge 0}R(t)\ge 0\Bigr\}
\end{eqnarray*}
is the probability of ultimate survival and
\begin{eqnarray*}
\psi(x) &:=& \P\{R(t)<0\mbox{ for some }t\ge0\}\\
&=& \P\Bigl\{\min_{t\ge 0}R(t)<0\Bigr\}
\end{eqnarray*}
is the probability of ruin.\index{Ruin probability}
\index{Cram\'er--Lundberg!classical model!ruin probability} 
We have
\begin{eqnarray*}
\psi(x) &=& \P\Bigl\{\sum_{i=1}^{N(t)}\xi_i-vt>x
\mbox{ for some }t\ge0\Bigr\}.
\end{eqnarray*}
Since $v>0$, the ruin can only occur at a claim epoch. Therefore,
\begin{eqnarray*}
\psi(x) &=& \P\Bigl\{\sum_{i=1}^n\xi_i-vT_n>x
\mbox{ for some }n\ge1\Bigr\},
\end{eqnarray*}
where $T_n$ is the $n$th claim epoch, so that
$T_n=\tau_1+\ldots+\tau_n$ where the $\tau_k$'s are independent copies
of a random variable $\tau$ with finite mean $1/\lambda$, so that
$N(t):=\max\{n\ge 1:T_n\le t\}$.
Denote $X_i:=\xi_i-v\tau_i$ and $S_n:=X_1+\ldots+X_n$, then
\begin{eqnarray*}
\psi(x) &=& \P\Bigl\{\sup_{n\ge 1}S_n>x\Bigr\}.
\end{eqnarray*}
This relation represents the ruin probability problem
as the tail probability problem for
the maximum of the associated random walk $\{S_n\}$.
Let the {\it net-profit condition}
$$
v\ >\ v_c\ :=\ \E\xi/\E\tau\ =\ \lambda\E\xi
$$
hold, thus $\{S_n\}$ has a negative drift, hence by the strong law of large numbers
$S_n\to-\infty$ a.s., so $\psi(x)\to 0$ as $x\to\infty$.

In this section we consider a risk process where the premium
rate $v(y)$ depends on the current level of risk reserve $R(t)=y$,
so $R(t)$ satisfies the equality
\begin{eqnarray}\label{risk.1}
R(t) &=& x+\int_0^t v(R(s))ds-\sum_{j=1}^{N(t)}\xi_j;
\end{eqnarray}
$v(y)$ is assumed to be a measurable non-negative function.
The probability of ruin given initial risk reserve $x$
is again denoted by $\psi(x)$.
Since the ruin can only occur at a claim epoch,
the ruin probability may be reduced to that for
the embedded Markov chain $R_n:=R(T_n)$, $n\ge 1$, $R_0:=x$, that is,
\begin{eqnarray*}
\psi(x) &=& \P\{R_n<0\mbox{ for some }n\ge0\}.
\end{eqnarray*}

In this section we consider the case where $v(y)$ approaches 
the critical value $v_c$ at infinity, that is,
\begin{eqnarray}\label{risk.appr}
v(y) &\to& v_c\quad\mbox{as }y\to\infty.
\end{eqnarray}
Then the Markov chain $\{R_n\}$ has asymptotically zero drift and, 
as follows from Theorem \ref{non.exp.return}, the ruin probability
decays slower than any exponential function, that is, for any $\lambda>0$,
\begin{eqnarray*}
e^{\lambda x}\psi(x) &\to& \infty\quad\mbox{as }x\to\infty.
\end{eqnarray*}
The main goal in this section is to investigate how the rate of convergence 
in \eqref{risk.appr} is reflected in how quickly the ruin probability $\psi(x)$ 
is vanishing for large $x$.
Let us get some intuition on what kind of phenomena we could expect here
by considering a model where $\psi(x)$ is known in closed form.

To the best of our knowledge, the only case where $\psi(x)$
is explicitly calculable is the case of exponentially distributed
$\tau$ and $\xi$, say with parameters $\lambda$ and $\mu$ respectively,
so hence $v_c=\lambda/\mu$.
In this case, for some $c_0\in(0,1)$,
\begin{eqnarray}\label{psi.x.formula}
\psi(x) &=& c_0\int_x^\infty \frac{1}{v(y)}
\exp\Bigl\{-\mu y+\lambda\int_0^y\frac{dz}{v(z)}\Bigr\}dy\nonumber\\
&=& c_0\int_x^\infty \frac{1}{v(y)}
\exp\Bigl\{\lambda\int_0^y\Bigl(\frac{1}{v(z)}-\frac{1}{v_c}\Bigr)dz\Bigr\}dy,
\end{eqnarray}
provided the outer integral is convergent from $0$ to infinity,
see, e.g. Corollary 1.9 in Albrecher\index{Albrecher} and 
Asmussen\index{Asmussen} \cite[Ch. VIII]{AA}. Then, by \eqref{risk.appr},
\begin{eqnarray*}
\psi(x) &\sim& \frac{c_0}{v_c}\int_x^\infty
\exp\Bigl\{\lambda\int_0^y\Bigl(\frac{1}{v(z)}-\frac{1}{v_c}\Bigr)dz\Bigr\}dy
\quad\mbox{as }x\to\infty.
\end{eqnarray*}

If the premium rate $v(z)\ge v_c$ approaches $v_c$ at the rate of $\theta/z$,
$\theta>0$, more precisely, if
\begin{eqnarray}\label{risk.2}
\Bigl|v(z)-v_c-\frac{\theta}{z}\Bigr| &\le& p(z)\quad\mbox{for all }z>1,
\end{eqnarray}
where $p(z)>0$ is an integrable at infinity decreasing
function, then we get
\begin{eqnarray*}
\frac{1}{v(z)} &=&
\frac{1}{v_c}-\frac{\theta}{v_c^2z}+O(p(z)+z^{-2})
\end{eqnarray*}
and consequently
\begin{eqnarray*}
\lambda\int_0^y\Bigl(\frac{1}{v(z)}-\frac{1}{v_c}\Bigr)dz
&=& -\frac{\theta\mu^2}{\lambda}\log y +c_1+o(1)\quad\mbox{as }y\to\infty,
\end{eqnarray*}
where $c_1$ is a finite number. Let $\theta>\lambda/\mu^2$.
Then, for $C:=c_0e^{c_1}/(\theta\mu-\lambda/\mu)>0$,
\begin{eqnarray}\label{risk.exp}
\psi(x) &\sim& \frac{C}{x^{\theta\mu^2/\lambda-1}}\quad\mbox{as }x\to\infty.
\end{eqnarray}
A similar asymptotic expression can be obtained also in the case where the
Laplace transforms of variables $\xi_1$ and $\tau_1$ are rational functions,
see Albrecher\index{Albrecher} et al. \cite{A_all}.

If the premium rate $v(z)$ approaches $v_c$ at the rate of $\theta/z^\alpha$,
$\theta>0$ and $\alpha\in(0,1)$, more precisely, if
\begin{eqnarray}\label{risk.2.hx}
\Bigl|v(z)-v_c-\frac{\theta}{z^\alpha}\Bigr| &\le& p(z)\quad\mbox{for all }z>1,
\end{eqnarray}
where $p(z)>0$ is an integrable at infinity decreasing function, then we get
\begin{eqnarray*}
\frac{1}{v(z)} &=&
\frac{1}{v_c}\sum_{j=0}^\infty \Bigl(-\frac{\theta}{v_c}\Bigr)^j
\frac{1}{z^{\alpha j}}+O(p(z)).
\end{eqnarray*}
Let $\gamma:=\min\{k\in\N:k\alpha>1\}$. Then
\begin{eqnarray*}
\frac{1}{v(z)} &=&
\frac{1}{v_c}\sum_{j=0}^{\gamma-1} \Bigl(-\frac{\theta}{v_c}\Bigr)^j
\frac{1}{z^{\alpha j}}+O(p_1(z)),
\end{eqnarray*}
where $p_1(z)=p(z)+z^{-\gamma\alpha}$ is integrable at infinity.
Consequently, if $1/\alpha$ is not integer, then
\begin{eqnarray*}
\lambda\int_0^y\Bigl(\frac{1}{v(z)}-\frac{1}{v_c}\Bigr)dz
&=& \frac{\lambda}{v_c} \int_1^y
\sum_{j=1}^{\gamma-1} \Bigl(-\frac{\theta}{v_c}\Bigr)^j
\frac{1}{z^{\alpha j}}dz+c_2+o(1)\\
&=& \frac{\lambda}{v_c}
\sum_{j=1}^{\gamma-1} \Bigl(-\frac{\theta}{v_c}\Bigr)^j
\frac{y^{1-\alpha j}}{1-\alpha j}+c_3+o(1)\quad\mbox{as }y\to\infty,
\end{eqnarray*}
where $c_3$ is a finite number because $p_1(x)$ is integrable.
In the case of integer $1/\alpha$,
\begin{eqnarray*}
\lambda\int_0^y\Bigl(\frac{1}{v(z)}-\frac{1}{v_c}\Bigr)dz
&=& \frac{\lambda}{v_c}
\sum_{j=1}^{\gamma-2} \Bigl(-\frac{\theta}{v_c}\Bigr)^j
\frac{y^{1-\alpha j}}{1-\alpha j}\\
&&\hspace{20mm}+
\frac{\lambda}{v_c} \Bigl(-\frac{\theta}{v_c}\Bigr)^{\gamma-1}\log y
+c_4+o(1)\quad\mbox{as }y\to\infty.
\end{eqnarray*}

Let, for example, $\alpha\in(1/2,1)$. Then
\begin{eqnarray*}
\lambda\int_0^y\Bigl(\frac{1}{v(z)}-\frac{1}{v_c}\Bigr)dz
&=& -\frac{\theta\mu^2}{\lambda(1-\alpha)} y^{1-\alpha} +c_3+o(1)\quad\mbox{as }y\to\infty.
\end{eqnarray*}
Therefore, for $C_1:=c_0e^{c_3}/\theta\mu>0$ and $C_2:=\theta\mu^2/\lambda(1-\alpha)>0$,
\begin{eqnarray}\label{risk.exp.new}
\psi(x) &\sim& C_1x^\alpha e^{-C_2x^{1-\alpha}}\quad\mbox{as }x\to\infty.
\end{eqnarray}

Let us extend these results to not necessarily exponential distributions
where there are no formulas like \eqref{psi.x.formula} for $\psi(x)$ available.
In that case we can only derive lower and upper bounds for $\psi(x)$.

\subsection{Approaching critical premium rate at rate of $\theta/x$}
\label{subsec:rate.1x}

Denote the jumps of the chain $\{R_n=R(T_n)\}$ by $\xi(x)$
and by $m_k^{[s(x)]}(x)$ its $k$th truncated moment.

\begin{proposition}\label{prop.risk}
Assume the rate of convergence \eqref{risk.2} and 
that both $\E \tau_1^2$ and $\E\xi_1^2$ are finite.
Then, for any $\varepsilon>0$,
 there exists an increasing function $s(x)=o(x)$ such that
\begin{eqnarray*}
\frac{2m_1^{[s(x)]}(x)}{m_2^{[s(x)]}(x)} &\ge& \frac{\rho+1}{x}+o(p_1(x))
\quad\mbox{as }x\to\infty,
\end{eqnarray*}
for some decreasing integrable function $p_1(x)$, where
\begin{eqnarray*}
\rho &:=& \frac{2\theta\E \tau}{\V\xi+v_c^2\V\tau}-1.
\end{eqnarray*}
If, in addition, both $\E \tau^2\log(1+\tau)$ and
$\E\xi^2\log(1+\xi)$ are finite,
then there exists an increasing function $s(x)=o(x)$ such that
\begin{eqnarray*}
\frac{2m_1^{[s(x)]}(x)}{m_2^{[s(x)]}(x)} &=& \frac{\rho+1}{x}+o(p_2(x))
\quad\mbox{as }x\to\infty,
\end{eqnarray*}
for some decreasing integrable function $p_2(x)$.
\end{proposition}

\begin{proof}
The dynamics of the risk reserve between two consequent claims
is governed by the differential equation $R'(t)=v(R(t))$.
Let $V_x(t)$ denote its solution with the initial value $x$, so then
\begin{eqnarray*}
V_x(t) &=& x+\int_0^t v(V_x(s))ds.
\end{eqnarray*}
By \eqref{risk.2},
\begin{eqnarray*}
v(y) &\le& v_c+\theta/y+p(y)\\
&\le& v_c+\theta/x+p(x)\quad\mbox{for all }y\ge x,
\end{eqnarray*}
therefore
\begin{eqnarray}\label{risk.Vx.above}
V_x(t)-x &\le& v_ct+\theta t/x+p(x)t,\quad t>0.
\end{eqnarray}
On the other hand, again by \eqref{risk.2},
\begin{eqnarray*}
v(y) &\ge& v_c+\theta/y-p(y)\\
&\ge& v_c+\theta/y-p(x)\quad\mbox{for all }y\ge x,
\end{eqnarray*}
Hence,
\begin{eqnarray*}
V_x(t)-x &\ge& v_ct+\theta\int_0^t \frac{ds}{V_x(s)}-p(x)t\\
&\ge& v_ct+\theta\int_0^t \frac{ds}{x+(v_c+\theta/x+p(x))s}-p(x)t\\
&=& v_ct+\frac{\theta}{v_c+\theta/x+p(x)}
\log\bigl(1+(v_c+\theta/x+p(x))t/x\bigr)-p(x)t,
\end{eqnarray*}
where the second inequality follows from the upper bound \eqref{risk.Vx.above}. Therefore,
\begin{eqnarray}\label{risk.Vx.below}
V_x(t)-x &\ge& v_ct+\frac{\theta}{v_c+\theta/x+p(x)}
\log\bigl(1+v_ct/x\bigr)-p(x)t,
\end{eqnarray}

Since $\xi(x)=V_x(\tau)-x-\xi$,
it follows from \eqref{risk.Vx.above} and \eqref{risk.Vx.below} that
\begin{eqnarray}\label{risk.6}
\lefteqn{v_c\tau-\xi+\frac{\theta}{v_c+\theta/x+p(x)}
\log\Bigl(1+\frac{v_c\tau}{x}\Bigr)-p(x)\tau}\nonumber\\
&&\hspace{40mm}\le\ \xi(x)\ \le\
v_c\tau-\xi +\frac{\theta\tau}{x}+p(x)\tau.\hspace{10mm}
\end{eqnarray}
Recalling that $v_c=\E\xi/\E\tau$, we get
\begin{eqnarray*}
\frac{\theta}{v_c+\theta/x+p(x)}\E\log\Bigl(1+\frac{v_c\tau}{x}\Bigr)-p(x)\E\tau
&\le& m_1(x)\ \le\ \frac{\theta}{x}\E\tau+p(x)\E\tau.
\end{eqnarray*}
By the inequality $\log(1+z)\ge z-z^2/2$ for $z\ge 0$,
\begin{eqnarray*}
\E\log\Bigl(1+\frac{v_c\tau}{x}\Bigr) &\ge&
\frac{v_c \E\tau}{x}-\frac{v^2_c \E\tau^2}{2x^2}.
\end{eqnarray*}
Therefore,
\begin{eqnarray}\label{risk.7}
m_1(x) &=& \frac{\theta \E\tau}{x}+O(p(x)+1/x^2)\quad\mbox{as }x\to\infty.
\end{eqnarray}
From this expression we have
\begin{eqnarray*}
m_2(x) &=& \V\xi(x)+m_1^2(x)\\
&=& \V(V_x(\tau)-x-\xi)+O(p^2(x)+1/x^2)\\
&=& \V(V_x(\tau)-x)+\V\xi+O(p^2(x)+1/x^2)
\quad\mbox{as }x\to\infty.
\end{eqnarray*}
Recalling that
\begin{eqnarray*}
v_ct-p(x)t &\le& V_x(t)-x\ \le\ v_ct+\frac{\theta}{x}t+p(x)t,
\end{eqnarray*}
we get
\begin{eqnarray*}
(v_c-p(x))\E\tau &\le& \E (V_x(\tau)-x)\ \le\ (v_c+\theta/x+p(x))\E\tau
\end{eqnarray*}
and
\begin{eqnarray*}
(v_c-p(x))^2\E \tau^2 &\le& \E (V_x(\tau)-x)^2
\ \le\ (v_c+\theta/x+p(x))^2\E\tau^2.
\end{eqnarray*}
Hence,
\begin{eqnarray*}
\V(V_x(\tau)-x) &=& v_c^2\V\tau+O(1/x)
\quad\mbox{as }x\to\infty,
\end{eqnarray*}
which in its turn implies
\begin{eqnarray}\label{risk.8}
m_2(x) &=& \V\xi+v_c^2\V\tau+O(1/x)
\quad\mbox{as }x\to\infty.
\end{eqnarray}
Together with \eqref{risk.7} it yields that
\begin{eqnarray*}
\frac{2m_1(x)}{m_2(x)} &=&
\frac{2\theta\E \tau}{\V\xi+v_c^2\V\tau} \cdot
\frac{1}{x}+O(p(x)+1/x^2)\quad\mbox{as }x\to\infty.
\end{eqnarray*}

Recall that we need such kind of expansion for the truncated moments.
For any truncation level $s(x)$ we have
\begin{eqnarray}\label{risk.8a}
\nonumber
\lefteqn{|V_x(\tau)-x-\xi|\I\{|V_x(\tau)-x-\xi|>s(x)\}}\\
\nonumber
&\le& (V_x(\tau)-x+\xi)\I\{V_x(\tau)-x>s(x)\text{ or }\xi>s(x)\}\\
&\le& (V_x(\tau)-x)\I\{V_x(\tau)-x>s(x)\}+\xi\I\{\xi>s(x)\}\nonumber\\
&&\hspace{5mm}+ \xi\I\{V_x(\tau)-x>s(x)\}+(V_x(\tau)-x)\I\{\xi>s(x)\}.
\end{eqnarray}
Since $V_x(t)-x\le c_1t$ for some $c_1<\infty$, we get
\begin{eqnarray*}
|m_1(x)-m_1^{[s(x)]}| &\le&
\E\{|V_x(\tau)-x-\xi|;\ |V_x(\tau)-x-\xi|>s(x)\}\\
&\le& c_1\E\{\tau;\ \tau>s(x)/c_1\}
+\E\{\xi;\xi>s(x)\}\\
&&\hspace{0.5cm}+\E\xi\P\{\tau>s(x)/c_1\}+c_1\E\tau\P\{\xi>s(x)\}.
\end{eqnarray*}
It follows from the finiteness of $\E\tau^2$ and $\E\xi^2$ that
there exists an increasing function $s_1(x)=o(x)$ such that both
$\E\{\tau;\ \tau>s_1(x)/c_1\}$ and $\E\{\xi;\ \xi>s_1(x)\}$
are integrable, see Lemma \ref{l:maj.p.e}.
Consequently, $|m_1(x)-m_1^{[s_1(x)]}(x)|$
is bounded by a decreasing integrable function.
Combining this with \eqref{risk.7}, we conclude that
\begin{eqnarray}\label{risk.9}
m_1^{[s_1(x)]}(x) &=& \frac{\theta\E\tau}{x}+o(p_2(x))\quad\mbox{as }x\to\infty,
\end{eqnarray}
where $p_2$ is a decreasing integrable function.

It follows from \eqref{risk.8} and \eqref{risk.9} that
\begin{eqnarray*}
\frac{2m_1^{[s(x)]}(x)}{m_2^{[s(x)]}(x)} &\ge&
\frac{2m_1^{[s(x)]}(x)}{m_2(x)}
\ \ge\ \frac{1+\rho}{x}+o(p_3(x))
\quad\mbox{as }x\to\infty,
\end{eqnarray*}
and the first result follows.

Similar to \eqref{risk.8a},
\begin{eqnarray*}
\lefteqn{(V_x(\tau)-x-\xi)^2\I\{|V_x(\tau)-x-\xi|>s(x)\}}\\
&\le& 2[(V_x(\tau)-x)^2+\xi^2]\I\{V_x(\tau)-x>s(x)\text{ or }\xi>s(x)\}\\
&\le& 2(V_x(\tau)-x)^2\I\{V_x(\tau)-x>s(x)\}+2\xi^2\I\{\xi>s(x)\}\\
&&\hspace{5mm}+ 2\xi^2\I\{V_x(\tau)-x>s(x)\}+2(V_x(\tau)-x)^2\I\{\xi>s(x)\}.
\end{eqnarray*}
Then, due to the upper bound $V_x(t)-x\le c_1t$, for some $c_2<\infty$,
\begin{eqnarray*}
0\ \le\ m_2(x)-m_2^{[s(x)]}(x) &=& \E\{(V_x(\tau)-x-\xi)^2;\ |V_x(\tau)-x-\xi|>s(x)\}\\
&\le& c_2\Bigl(\E\{\tau^2;\ \tau>s(x)/c_1\}
+\E\{\xi^2;\ \xi>s(x)\}\\
&&\hspace{10mm} +\E\xi^2\P\{\tau>s(x)/c_1\}
+\E\tau^2\P\{\xi>s(x)\}\Bigr).
\end{eqnarray*}
It follows from the finiteness of $\E\xi^2\log(1+\xi)$ and
$\E\tau^2\log(1+\tau)$ that there exists
an increasing function $s_2(x)=o(x)$ such that both
$x^{-1}\E\{\tau^2;\ \tau>s_2(x)/c_1\}$ and
$x^{-1}\E\{\xi^2;\ \xi>s_2(x)\}$ are integrable at infinity,
see Lemma \ref{l:maj.p.e.log}.
Then $(m_2(x)-m_2^{[s_2(x)]}(x))/x$ is integrable too.
From this fact and \eqref{risk.8} we get
\begin{eqnarray}\label{risk.10}
m_2^{[s_2(x)]}(x)=\V\xi+v_c^2\V\tau+o(xp_4(x))
\quad\mbox{as }x\to\infty,
\end{eqnarray}
for some decreasing integrable function $p_4(x)$. Taking now
$s(x)=\max(s_1(x),s_2(x))=o(x)$ we conclude the desired result 
from \eqref{risk.9} and \eqref{risk.10}.
\qed\end{proof}

\index{Cram\'er--Lundberg!classical model!ruin probability bounds}
\begin{theorem}\label{thm:risk}
Assume that both $\E\xi^2$ and $\E\tau^2$ are finite. If
\begin{eqnarray*}
\theta &>& \frac{\V\xi+v_c^2\V\tau}{2\E\tau},
\end{eqnarray*}
then $R_n$ is transient or, equivalently, $\psi(x)<1$ for all $x>0$. Set
\begin{eqnarray*}
\rho &=& \theta\frac{2\E\tau}{\V\xi+v_c^2\V\tau}-1
\ >\ 0.
\end{eqnarray*}
If both $\E \tau^2\log(1+\tau)$ and $\E\xi^{\rho+2}$ are finite,
then there exist positive constants $c_1$ and $c_2$ such that
\begin{eqnarray*}
\frac{c_1}{(1+x)^\rho} &\le& \psi(x)\ \le\ \frac{c_2}{(1+x)^\rho}
\quad\mbox{for all } x>0.
\end{eqnarray*}
\end{theorem}

\begin{proof}
By Proposition \ref{prop.risk},
\begin{eqnarray*}
\frac{2m_1^{[s(x)]}(x)}{m_2^{[s(x)]}(x)} &\ge& \frac{1+\varepsilon}{x}
\end{eqnarray*}
for some small $\varepsilon$ and for all $x\ge x_0(\varepsilon)$.
Furthermore, from the elementary bound $\P\{\xi(x)<-s(x)\}\le \P\{\xi>s(x)\}$
and the finiteness of $\E\xi^2$ we infer that,
for some increasing function $s(x)=o(x)$,
$$
\P\{\xi(x)<-s(x)\}\ \le p(x)/x,
$$
where $p(x)$ is a decreasing integrable at infinity function,
see Lemma \ref{l:maj.p.e}. In addition, the Markov chain $\{R_n\}$ dominates
a similar Markov chain generated by a risk process with constant premium
rate $v_c$ which represents a zero-drift random walk which is null-recurrent and
hence satisfying the condition \eqref{rec.2.tr}.
Thus, all the conditions of Theorem~\ref{thm:transience.inf}
are valid and, consequently, the chain $\{R_n\}$ is transient.

To prove the second part of the theorem, let us show that
all conditions of Theorem \ref{thm:transient.return} hold true.
The conditions \eqref{m1.and.m2.return}--\eqref{r.p.prime.return}
are valid for $\{R_n\}$ with $r(x)=(\rho+1)/(x+1)$ as follows from
Proposition \ref{prop.risk}.
For this $r(x)$ we have $U(x)=1/\rho(x+1)^\rho$ for $x>0$
and $U(x)=1/\rho$ for $x\le 0$.
The condition \eqref{cond.xi.le.return} on the right tail of $\xi(x)$
holds because
\begin{eqnarray*}
\P\{\xi(x)>s(x)\} &\le& \P\{V_x(\tau)-x>s(x)\}\\
&\le& \P\{\tau>s(x)/c_1\}
\ =\ o(p(x)/x)\quad\mbox{as }x\to\infty,
\end{eqnarray*}
due to the assumption $\E\tau^2<\infty$, see Lemma \ref{l:maj.p.e},
and due to the relation $U(x)\sim xe^{-R(x)}/\rho$.
By the same argument,
the condition \eqref{cond.xi.ge.return} holds because
\begin{eqnarray*}
\E\{U(x+\xi(x));\ \xi(x)<-s(x)\} &\le& c\P\{\xi(x)<-s(x)\}\\
&\le& c\P\{\xi>s(x)\}\\
&=& o(p(x)/x^{\rho+1})\quad\mbox{as }x\to\infty,
\end{eqnarray*}
due to the assumption $\E\xi^{\rho+2}<\infty$,
again by Lemma \ref{l:maj.p.e}. Obviously,
\begin{eqnarray*}
|\xi(x)| &\le& V_x(\tau)-x+\xi\ \le\ c_1\tau+\xi\ =:\ \Xi,
\end{eqnarray*}
where $\Xi$ is square integrable, so by Lemma \ref{l:p.V.maj.o}
with $\alpha=1$ and $\gamma=2$, the condition \eqref{cond.3.moment.return}
on the third truncated moment is also met for $\{R_n\}$.
\qed\end{proof}

\subsection{Approaching critical premium rate at the rate of $\theta/x^\alpha$}
\label{subsec:rate.hx}

In this subsection we consider the case \eqref{risk.2.hx} with some $\alpha\in(0,1)$.
In order to understand the asymptotic behaviour of the ruin probability
under this rate of approaching the critical value $v_c$, we first derive asymptotic estimates
for the moments of $V_x(\tau)-x$. Define
\begin{eqnarray*}
\gamma: &=& \min\{k\ge1:\alpha k>1\}.
\end{eqnarray*}

\begin{lemma}\label{lem:risk.2.V}
Let $\E\tau^\gamma<\infty$ and
\begin{eqnarray}\label{risk.2.v12}
v_-(x) &\le& v(x)\ \le\ v_+(x)\quad\mbox{for all }x,
\end{eqnarray}
where both $v_-(x)$ and $v_+(x)$ are decreasing functions.
Then, for all $k\le\gamma$,
\begin{eqnarray}\label{risk.2.bounds}
\E\tau^k v_-(x+\tau v_+(x))\ \le\ \E(V_x(\tau)-x)^k &\le& v_+^k(x)\E\tau^k.
\end{eqnarray}
If, in addition, $\E\tau^{\gamma+1-\alpha}<\infty$ and 
\eqref{risk.2.hx} holds true, then there exists an integrable
decreasing function $p_1(x)$ such that, for all $k\le\gamma$,
\begin{eqnarray}\label{risk.2.1}
\E(V_x(\tau)-x)^k &=&
(v_c+\theta/x^\alpha)^k\E\tau^k+O(p_1(x))
\quad\mbox{as }x\to\infty.
\end{eqnarray}
\end{lemma}

\begin{proof}
Due to \eqref{risk.2.v12}, $v(z)\le v_+(x)$ for all $z\ge x$. Hence
\begin{eqnarray}\label{risk.2.v2}
\nonumber
V_x(t) &=& x+\int_0^t v(V_x(s))ds\\
&\le& x+\int_0^t v_+(x)ds
\ =\ x+tv_+(x),
\end{eqnarray}
and the inequality on the right hand side of \eqref{risk.2.bounds} follows.
It follows from the left hand side inequality in \eqref{risk.2.v12}
and from the last upper bound for $V_x(t)$ that
\begin{eqnarray}\label{risk.2.v1}
V_x(t)-x &\ge& \int_0^t v_-(V_x(t))ds
\ \ge\ tv_-(x+tv_+(x)),
\end{eqnarray}
and the left hand side bound in \eqref{risk.2.bounds} is proven.

Owing to \eqref{risk.2.hx}, $v(z)$ is sandwiched between the two
eventually decreasing functions $v_{\pm}(z):=v_c+\theta/z^\alpha\pm p(z)$.
Therefore, applying the right hand side bound in \eqref{risk.2.bounds} we get
\begin{eqnarray}\label{risk.2.above}
\E(V_x(\tau)-x)^k &\le& (v_c+\theta/x^\alpha+p(x))^k\E\tau^k\nonumber\\
&=& (v_c+\theta/x^\alpha)^k\E\tau^k+O(p(x))\quad\mbox{as }x\to\infty.
\end{eqnarray}
From the lower bound in \eqref{risk.2.bounds} we deduce, for all $k\le\gamma$,
\begin{eqnarray*}
\E(V_x(\tau)-x)^k &\ge&
\E\tau^k\Bigl(v_c+\frac{\theta}{(x+\tau(v_c+\theta/x^\alpha+p(x)))^\alpha}-p(x)\Bigr)^k\\
&\ge& \E\tau^k\Bigl(v_c+\frac{\theta}{(x+c_1\tau)^\alpha}\Bigr)^k+O(p(x))
\quad\mbox{for some }c_1<\infty.
\end{eqnarray*}
Hence,
\begin{eqnarray*}
\E(V_x(\tau)-x)^k &\ge&
\E\Bigl\{\tau^k\Bigl(v_c+\frac{\theta}{(x+c_1\tau)^\alpha}\Bigr)^k;
\ \tau\le x\Bigr\}+O(p(x)).
\end{eqnarray*}
By the inequality $1/(1+y)^\alpha\ge 1-\alpha y\wedge 1$,
we infer that, for $c_2=\alpha c_1$,
\begin{eqnarray*}
\frac{1}{(x+c_1t)^\alpha}
&\ge& \frac{1}{x^\alpha}\Bigl(1-\frac{c_2t}{x}\wedge 1\Bigr).
%\ =\ \frac{1}{x^\alpha}- \frac{c_2t}{x^{\alpha+1}}\wedge \frac{1}{x^\alpha}
%\quad\text{for all }x\ge 1.
\end{eqnarray*}
Therefore, for all $k\le\gamma$,
\begin{eqnarray}\label{risk.2.3}
\E(V_x(\tau)-x)^k &\ge&
\E\tau^k\Bigl(v_c+\frac{\theta}{x^\alpha}
-\frac{c_2\theta\tau}{x^{\alpha+1}}\I\{\tau\le x/c_2\}
-\frac{1}{x^\alpha}\I\{\tau>x/c_2\}\Bigr)^k
+O(p(x))\nonumber\\
&\ge& \Bigl(v_c+\frac{\theta}{x^\alpha}\Bigr)^k\E\tau^k
-\frac{c_3}{x^\alpha}\E\{\tau^k;\ \tau>x/c_2\}\nonumber\\
&&\hspace{10mm} -c_3\sum_{j=1}^k\frac{1}{x^{j(\alpha+1)}}\E\{\tau^{k+j};\ \tau\le x/c_2\}
-c_3p(x),
\hspace{5mm}
\end{eqnarray}
for some $c_3<\infty$. Then, due to the integrability of $p(x)$,
in order to prove that
\begin{eqnarray}\label{risk.2.below}
\E(V_x(\tau)-x)^k &\ge& (v_c+\theta/x^\alpha)^k\E\tau^k-p_1(x)
\end{eqnarray}
for some decreasing integrable $p_1(x)$, it suffices to show that
$$
x^{-\alpha}\E\{\tau^\gamma;\ \tau>x\}
$$
and
$$
x^{-j(\alpha+1)}\E\{\tau^{\gamma+j};\ \tau\le x\}
$$
are bounded by decreasing integrable at infinity functions.
Indeed, the integral of the first function---which decreases itself---is finite 
due to the finiteness of the $(\gamma+1-\alpha)$ moment of $\tau$. 
Concerning the second function, first notice that
\begin{eqnarray*}
x^{-j(\alpha+1)}\E\{\tau^{\gamma+j};\ \tau\le x\} &\le&
\frac{\E\{\tau^{\gamma+1};\ \tau\le x\}}{x^{1+\alpha}},\quad j\ge 1.
\end{eqnarray*}
The right hand side is bounded by a decreasing integrable at infinity function
due to the moment condition on $\tau$ and  Lemma \ref{l:p.V.maj.o}.
So, \eqref{risk.2.below} is proven
which together with \eqref{risk.2.above} completes the proof.
\qed\end{proof}

\begin{proposition}
\label{prop:risk.2}
Assume the rate of convergence \eqref{risk.2.hx}.
If both $\E\tau^{1+\gamma}$ and $\E\xi^{1+\gamma}$ are finite,
then there exists $s(x)=o(x^\alpha)$ such that, for all $k\le\gamma$,
\begin{eqnarray*}
m_k^{[s(x)]}(x) &=&
\sum_{j=0}^k\frac{a_{k,j}}{x^{\alpha j}} +O(x^{\alpha(k-1)}p_2(x))
\quad\mbox{as }x\to\infty,
\end{eqnarray*}
where $p_2(x)$ is a decreasing integrable at infinity function and
\begin{eqnarray*}
a_{k,j} &:=& {k\choose j}\theta^j \E\tau^j(v_c\tau-\xi)^{k-j},
\quad j\le k\le\gamma.
\end{eqnarray*}
\end{proposition}

\begin{proof}
It follows from the definition of $\xi(x)$ that
\begin{eqnarray*}
\E\xi^k(x) &=&
\E(V_x(\tau)-x-\xi)^k
\ =\ \sum_{i=0}^k{k\choose i}\E(V_x(\tau)-x)^i\E(-\xi)^{k-i}.
\end{eqnarray*}
Applying Lemma \ref{lem:risk.2.V}, we then obtain
\begin{eqnarray*}
m_k(x)\ :=\ \E\xi^k(x) &=& 
\sum_{i=0}^k{k\choose i}\Bigl(v_c+\frac{\theta}{x^\alpha}\Bigr)^i
\E\tau^i\E(-\xi)^{k-i}+O(p_1(x))\\
&=& \sum_{i=0}^k{k\choose i}\E\tau^i\E(-\xi)^{k-i}
\sum_{j=0}^i{i\choose j}v_c^{i-j}\Bigl(\frac{\theta}{x^\alpha}\Bigr)^j+O(p_1(x))\\
&=:& \sum_{j=0}^k\frac{a_{k,j}}{x^{\alpha j}}+O(p_1(x))\quad\mbox{as }x\to\infty,
\end{eqnarray*}
where
\begin{eqnarray*}
a_{k,j} &:=& {k\choose j}\theta^j 
\sum_{i=j}^k{k-j\choose i-j}\E\tau^i\E(-\xi)^{k-i}v_c^{i-j}\\
&=& {k\choose j}\theta^j \E\sum_{i=0}^{k-j}{k-j\choose i}\tau^{i+j}(-\xi)^{k-j-i}v_c^i\\
&=& {k\choose j}\theta^j \E\tau^j(v_c\tau-\xi)^{k-j}.
\end{eqnarray*}
Now, in view of \eqref{risk.8a} we have
\begin{eqnarray*}
\lefteqn{|m_k(x)-m_k^{[s(x)]}(x)|}\\
&&\hspace{10mm}=\ O\Bigl(\E\{(V_x(\tau)-x)^k;\
V_x(\tau)-x>s(x)\}+\E\{\xi^k;\xi>s(x)\}\\
&&\hspace{20mm}+\ \E\xi^k\P\{V_x(\tau)-x>s(x)\}
+\E(V_x(\tau)-x)^k\P\{\xi>s(x)\}\Bigr)\\
&&\hspace{10mm}=\ O\Bigl(\E\{\tau^k;\
\tau>s(x)/c_1\}+\E\{\xi^k;\xi>s(x)\}\Bigr)\quad\mbox{as }x\to\infty.
\end{eqnarray*}
Since $\E\tau^{\gamma+1}<\infty$, for all $k\le\gamma$,
\begin{eqnarray*}
x^{-\alpha(k-1)}\E\{\tau^k;\ \tau>s(x)/c_1\} &=& 
o(1/x^{\alpha(k-1)} s^{\gamma+1-k}(x))\\
&=& o(1/s^\gamma(x))\quad\mbox{as }x\to\infty,
\end{eqnarray*}
for any $s(x)=o(x^\alpha)$. By the definition of the $\gamma$, 
$\alpha\gamma>1$. Therefore, for an increasing function 
$s(x)=x^\alpha/\log x=o(x^\alpha)$, 
the function $1/s^\gamma(x)$ is integrable at infinity.
The same arguments work for $\xi$, so the value of
$x^{-\alpha(k-1)}|m_k(x)-m_k^{[s(x)]}(x)|$ 
is bounded by a decreasing integrable at infinity function,
 and the proof is complete.
\qed\end{proof}

Now we state the main result in this subsection.
\index{Cram\'er--Lundberg!classical model!ruin probability bounds}

\begin{theorem}\label{thm:risk.2}
Assume the rate of convergence \eqref{risk.2.hx}.
If $\E\tau^{\gamma+1}<\infty$ and $\E e^{r\xi^{1-\alpha}}<\infty$ for some
$$
r\ >\ r_1\ :=\
\frac{2\theta\E\tau}{\V\xi+v_c^2\V\tau},
$$
then there exist constants $r_2$, $r_3$, \ldots, $r_{\gamma-1}\in\R$,
and $0<C_1<C_2<\infty$ such that
\begin{itemize}
\item[(i)] if $\alpha=1/(\gamma-1)$ for an integer $\gamma\ge 2$,
then, for $x>1$,
\begin{eqnarray}\label{risk.2.6}
\frac{C_1x^\alpha}{x^{r_{\gamma-1}}}
\exp\Biggl\{-\sum_{j=1}^{\gamma-2}\frac{r_j}{1-\alpha j} x^{1-\alpha j}\Biggr\}
\ \le\ \psi(x)\ \le\
\frac{C_2x^\alpha}{x^{r_{\gamma-1}}}
\exp\Biggl\{-\sum_{j=1}^{\gamma-2}\frac{r_j}{1-\alpha j} x^{1-\alpha j}\Biggr\},
\nonumber\\
\end{eqnarray}
\item[(ii)] if $\alpha<1/(\gamma-1)$ then
\begin{eqnarray}\label{risk.2.6*}
C_1x^\alpha\exp\Biggl\{-\sum_{j=1}^{\gamma-1}\frac{r_j}{1-\alpha j} x^{1-\alpha j}\Biggr\}
\ \le\ \psi(x)\ \le\
C_2x^\alpha\exp\Biggl\{-\sum_{j=1}^{\gamma-1}\frac{r_j}{1-\alpha j} x^{1-\alpha j}\Biggr\}.
\nonumber\\
\end{eqnarray}
\end{itemize}
\end{theorem}

\begin{proof}
We first show that there exist constants
$r_1,r_2,\ldots,r_{\gamma-1}$ such that
$$
r(x):=\sum_{j=1}^{\gamma-1}\frac{r_j}{(1+x)^{\alpha j}}
$$
satisfies \eqref{4.r-cond.function.gen.hy}.
We can determine all these numbers recursively.
Indeed, as proven in Proposition~\ref{prop:risk.2},
\begin{eqnarray*}
m_1^{[s(x)]}(x) &=& \frac{\theta\E\tau}{x^\alpha}+o(p_2(x))\quad\mbox{as }x\to\infty
\end{eqnarray*}
and
\begin{eqnarray*}
m_2^{[s(x)]}(x) &=& \V\xi+v_c^2\V\tau+O(x^{-\alpha})
\quad\mbox{as }x\to\infty.
\end{eqnarray*}
If we now take
\begin{eqnarray*}
r_1 &=& \frac{2\theta\E\tau}{\V\xi+v_c^2\V\tau},
\end{eqnarray*}
then
\begin{eqnarray*}
-m^{[s(x)]}_1(x)+\sum_{j=2}^{\gamma} (-1)^j\frac{m^{[s(x)]}_j(x)}{j!}r^{j-1}(x)
&=& O(x^{-2\alpha})\quad\mbox{as }x\to\infty,
\end{eqnarray*}
for any choice of $r_2$, $r_3$, \ldots, $r_{\gamma-1}$.
Then we can choose $r_2$ such that
the coefficient of $x^{-2\alpha}$ is also zero, and so on.
The conditions \eqref{4.r.gamma.hy} and \eqref{4.r.prime.gen.hy}
are satisfied for $r(x)$.

We have
\begin{eqnarray*}
U(x) &=& \int_x^\infty \exp\Bigl\{-
\int_0^y \sum_{j=1}^{\gamma-1}\frac{r_j}{(1+z)^{\alpha j}} dz \Bigr\}dy.
\end{eqnarray*}
The conditions \eqref{cond.3.moment.return.hy}, \eqref{cond.xi.le.return.hy}
and \eqref{cond.xi.ge.return.hy} are immediate from the moment assumptions
on $\tau$ and $\xi$. Thus, the announced bounds for the ruin probability 
follow from Theorem~\ref{thm:transient.return.hy}.
\qed\end{proof}

\section{Stochastic difference equations:
approach via asymptotically homogeneous chains}
\sectionmark{Stochastic difference equations}
\label{sec:perpetuity}

Let $(A_n,B_n)$ be a sequence of independent identically distributed
random vectors in $(\Rp)^2$. Consider a stochastic linear 
recursion\index{Stochastic linear recursion}
\begin{align}\label{stoch.rec}
R_n=A_nR_{n-1}+B_n,\quad n\ge 1
\end{align}
with some independent starting point $R_0$.
The sequence $\{R_n\}$ is a Markov chain.
The assumption $B_n\ge0$ is by far not standard but we choose it,
because non-negative stochastic difference equations allow us
a more straightforward analysis via Markov chains on $\Rp$.

We also assume that $\P\{A_1>1\}>0$ which guarantees that
\begin{eqnarray*}
\P\Bigl\{\limsup_{n\to\infty}R_n=\infty\Bigr\} &=& 1.
\end{eqnarray*}

It is immediate from \eqref{stoch.rec} that
$$
R_n=R_0\prod_{j=1}^n A_j+\sum_{k=1}^n B_k\prod_{j=k+1}^nA_j,\quad n\ge1.
$$
Then, for every $n\ge1$, the distribution of the variable $R_n$
coincides with that of
\begin{equation}\label{perpet}
D_n:=R_0\prod_{j=1}^n A_j+\sum_{k=1}^n B_k\prod_{j=1}^{k-1}A_j,
\end{equation}
which is called a {\it perpetuity}.\index{Perpetuity}
The coincidence of marginal distributions is not the only
connection between sequences $\{R_n\}$ and $\{D_n\}$.
Vervaat\index{Vervaat} \cite{Vervaat} has shown that the Markov chain
$\{R_n\}$ is positive recurrent if and only if
$$
D_\infty:=\sum_{k=1}^\infty B_k\prod_{j=1}^{k-1}A_j<\infty\quad\text{a.s.}
$$
In this case, the sequence $\{R_n\}$ converges weakly to the distribution
of $D_\infty$ and, furthermore, this distribution is a unique
solution to a fixed point equation
$$
D_\infty\stackrel{d}=A_1D'_\infty+B_1,
$$
where $D'_\infty$ is independent of $(A_1,B_1)$ and $D'_\infty$ and $D_\infty$
are identically distributed.

We are going to show how one can determine the asymptotic behaviour
of the invariant distribution of $\{R_n\}$ by using results from
Chapter~\ref{ch:asymp.hom}. First we notice that the chain $\{R_n\}$
is not asymptotically homogeneous in space. In order to transform it to
an asymptotically homogeneous chain we define a function
\begin{eqnarray}\label{eq:f}
f(x) &:=& \left\{
\begin{array}{ll}
\log x&\mbox{ for }x\ge e,\\
x/e&\mbox{ for }x\in[0,e],
\end{array}
\right.
\end{eqnarray}
so $f(x):\Rp\to\Rp$ is a continuous strictly increasing function
such that $f(x)\ge \log x$ for all $x\ge0$.
Since $f(x)$ is strictly increasing, the sequence
\begin{equation}
\label{jump.ge.0}
X_n=f(R_n)
\end{equation}
is a Markov chain on the state space $\Rp$. Let $\xi(x)$ denote the jumps
of this chain. It is immediate from the definition of $f(x)$ that, for
all $x\ge e$,
\begin{eqnarray}
\label{jump-expr}
\xi(x) &=& \left\{
\begin{array}{ll}
\log \left(A_1+e^{-x}B_1\right),&\mbox{if }A_1e^{x}+B_1\ge e,\\
\frac{A_1e^{x}+B_1}{e}- x,&\mbox{if }A_1e^{x}+B_1\in[0,e].
\end{array}
\right.
\end{eqnarray}
Therefore,
$$
\xi(x)\Rightarrow\log A_1\in[-\infty,\infty),
$$
that is, $\{X_n\}$ is asymptotically homogeneous. Furthermore,
\begin{equation}
\label{R_to_X}
\P\{R_n>x\}=\P\{D_n>x\}=\P\{X_n>\log x\},\quad x\ge e.
\end{equation}

\subsection{Positive recurrent case}
If $\E\log A_1\in[-\infty,0)$ then,
according to Lemma 1.7 in \cite{Vervaat}, $D_\infty<\infty$ provided 
$\E\log(1+B_1)<\infty$. In the following theorem we describe the asymptotic
behaviour of the distribution of $D_\infty$,
which is also a stationary distribution for the chain $\{R_n\}$.
\index{Stochastic linear recursion!stationary distribution!tail asymptotics}

\begin{theorem}\label{thm:perpet.cramer}
Suppose that $\E A_1^{\beta}=1$ for some $\beta>0$ and
$\E (A_1+B_1)^\beta<\infty$. Then
\begin{eqnarray}\label{D.infty.beta.rough}
\log\P\{D_\infty>x\} &\sim& -\beta\log x\quad\mbox{ as }x\to\infty.
\end{eqnarray}
If, in addition, 
\begin{equation}\label{extra_cond}
\E (\log^+(A_1+B_1))(A_1+B_1)^\beta<\infty.
\end{equation}
and the distribution of $\log A_1$ is non-lattice then, for some $c>0$,
\begin{eqnarray}\label{D.infty.beta}
\P\{D_\infty>x\} &\sim& \frac{c}{x^\beta}\quad\mbox{ as }x\to\infty.
\end{eqnarray}
\end{theorem}

\begin{proof}
The logarithmic asymptotics follow from the asymptotic
homogenuity of the chain $\{X_n=f(D_n)\}$ and Theorem \ref{thm:log.tail}.

It follows from \eqref{jump-expr} that
$$
\xi(x)\le \log^+(A_1+B_1),\quad x\ge e.
$$
For $x\le e$ we have
$$
\xi(x)\le f(A_1e^x+B_1)\le 1+\log^+(A_1+B_1).
$$
As a result,
\begin{equation}\label{xi-majorant}
\xi(x)\ \le\ \Xi:=1+\log^+(A_1+B_1)\quad\mbox{for all }x\ge0,
\end{equation}
and $\E\Xi e^{\beta\Xi}<\infty$.
Thus, to apply Theorem \ref{th:conv}, it is sufficient to check that
$|\E e^{\beta\xi(x)}-1|$ is dominated by a decreasing integrable function.

Using \eqref{jump-expr}, we get the following lower bound, for all $x>e$,
\begin{eqnarray*}
\E e^{\beta\xi(x)}&\ge&\E\{(A_1+e^{-x}B_1)^\beta;\ A_1+e^{-x}B_1>e^{1-x}\}\\
&\ge& \E A_1^\beta-\E\{(A_1+e^{-x}B_1)^\beta;\ A_1+e^{-x}B_1\le e^{1-x}\}\\
&\ge& 1-e^{\beta-\beta x}.
\end{eqnarray*}
To obtain an upper bound, we first notice that
\begin{eqnarray*}
\E e^{\beta\xi(x)}&=&\E \{e^{\beta\xi(x)};\xi(x)\le 1-x\}
+\E \{e^{\beta\xi(x)};\xi(x)>1-x\}\\
&\le& e^{\beta-\beta x}+\E\{(A_1+e^{-x}B_1)^\beta;\ A_1+e^{-x}B_1>e^{1-x}\}\\
&\le& e^{\beta-\beta x}+\E(A_1+e^{-x}B_1)^\beta.
\end{eqnarray*}
If $\beta\le1$ then $(u+v)^\beta\le u^\beta+v^\beta$ for all $u$, $v\ge0$. 
Set $u=A_1$ and $v=e^{-x}B_1$, then
$$
\E e^{\beta\xi(x)}\le \E A_1^\beta+(e^\beta+\E B_1^\beta)e^{-\beta x}.
$$
If $\beta>1$ then
\begin{eqnarray*}
(u+v)^\beta &\le& u^\beta+\beta v(u+v)^{\beta-1}\\
&\le& u^\beta+c_\beta vu^{\beta-1}+c_\beta v^\beta,
\end{eqnarray*}
where $c_\beta=\beta 2^{\beta-1}$. Therefore,
\begin{eqnarray*}
\E e^{\beta\xi(x)} &\le& \E A_1^\beta+(e^\beta+c_\beta\E B_1^\beta)e^{-\beta x}
+c_\beta e^{-x}\E A_1^{\beta-1}B_1\\
&=& 1+(e^\beta+c_\beta\E B_1^\beta)e^{-\beta x}
+c_\beta e^{-x}\E A_1^{\beta-1}B_1,
\end{eqnarray*}
where $\E A_1^{\beta-1}B_1<\infty$, because
\begin{eqnarray*}
A_1^{\beta-1}B_1 &\le& (A_1+B_1)^{\beta-1}(A_1+B_1)
\ =\ (A_1+B_1)^\beta.
\end{eqnarray*}
As a result,
$$
|\E e^{\beta\xi(x)}-1|=O\bigl(e^{-(1\wedge\beta)x}\bigr),
$$
which completes the proof.
\qed\end{proof}

\subsection{Null-recurrent case}

As mentioned above, the distribution of $R_n$
is the same as that of $D_n$ defined in \eqref{perpet}.
The sequence $D_n$ dominates an increasing sequence
\begin{eqnarray*}
T_n &:=& \sum_{k=1}^n B_k\prod_{j=1}^{k-1}A_j.
\end{eqnarray*}
If $\E\log A_1=0$ then $S_n:=\log A_1+\ldots+\log A_n$ is an oscillating
random walk, so $S_n>0$ infinitely often with probability $1$.
Equivalently,
\begin{eqnarray*}
\prod_{j=1}^{k-1}A_j &>& 1\quad\mbox{infinitely often with probability 1,}
\end{eqnarray*}
which implies convergence $T_n\to\infty$ as $n\to\infty$
with probability $1$. Hence, in the case $\E\log A_1=0$,
\begin{eqnarray}\label{lln.Rn}
R_n &\to& \infty\quad\mbox{in probability as }n\to\infty.
\end{eqnarray}

\index{Stochastic linear recursion!null recurrent!limit theorem}
\begin{theorem}\label{thm:Rn-limit}
Assume that  $\E\log A_1=0$, $\sigma^2:=\E\log^2A_1\in(0,\infty)$
and $\E\log^2B_1<\infty$. Then
$$
\frac{\log R_n}{\sqrt{\sigma^2n}}\ \Rightarrow\ |\eta|
\quad\mbox{as }n\to\infty,
$$
where $\eta$ has a standard normal distribution.
In addition, the process
$$
\frac{\log R_{[tn]}}{\sqrt{\sigma^2n}},\ t\in[0,1],
$$
converges weakly in $D[0,1]$ to a Bessel process with drift $0$
and diffusion coefficient $1$ as $n\to\infty$, that is,
to a reflected Brownian motion $|B(t)|$.
\end{theorem}

\begin{proof}
Note that the weak convergence of $\frac{\log R_n}{\sqrt{\sigma^2n}}$ to
$|\eta|$ is equivalent to the weak convergence of 
$\frac{X_n^2}{\sigma^2 n}=\frac{f^2(R_n)}{\sigma^2 n}$
towards $\eta^2$. Since $\eta^2$ is $\Gamma$-distributed with parameters
$1/2$ and $1/2$, the desired convergence would be proven if it was shown
that the conditions of Theorem~\ref{thm:pre-st.null} hold with $\mu=0$.
Then automatically the functional convergence follows too,
	see Theorem \ref{thm:gamma.func}.

We start by construction of a square integrable majorant for the jumps
$\xi(x)$. It follows from the definition of $f(x)$ that
$$
\xi(x)=f(A_1e^{x}+B_1)-x\ge \log(A_1e^{x}+B_1)-x\ge\log A_1,
$$
because $B_1\ge 0$.
Furthermore, according to \eqref{xi-majorant},
$$
\xi(x)\le 1+\log^+(A_1+B_1).
$$
From these two inequalities we infer that
$$
|\xi(x)|^2\le 2\left(1+\log^2A_1+\log^2(A_1+B_1)\right).
$$
Since the random variable on the right hand side is integrable, we have
constructed a suitable majorant.

Recalling that $\xi(x)\Rightarrow\log A_1$ and
using the Lebesgue theorem, we infer that
$$
m_2(x)=\E\xi^2(x)\ \to\ \E\log^2A_1=\sigma^2\quad\mbox{as }x\to\infty.
$$
Therefore, it remains to determine the asymptotic behaviour of
$m_1(x)=\E\xi(x)$. We start with the following upper bound
\begin{eqnarray*}
\E\xi(x)&=& \E f(A_1e^x+B_1)-x\\
&=& \E\{\log(A_1+e^{-x}B_1);\ A_1e^x+B_1>e\}\\
&&\hspace{25mm} +\E\Bigl\{\frac{A_1e^x+B_1}{e}-x-\log A_1;\ A_1e^x+B_1\le e\Bigr\}\\
&\le& \E\log(A_1+e^{-x}B_1)
+\E\Bigl\{\frac{A_1e^x+B_1}{e}-x;\ A_1e^x+B_1\le e\Bigr\}.
\end{eqnarray*}
Since $\frac{A_1e^x+B_1}{e}-x\in [-x,1-x]$ and $\log A_1\le 1-x$
if $A_1e^x+B_1\le e$,
\begin{eqnarray}\label{Rn_m1_1}
\nonumber
\Bigl|\E\Bigl\{\frac{A_1e^x+B_1}{e}-x;\ A_1e^x+B_1\le e\Bigr\}\Bigr|
&\le& x\P\{-\log A_1\ge x-1\}\\
&=& o(p_1(x))\quad\mbox{as }x\to\infty,
\end{eqnarray}
for some decreasing integrable at infinity function $p_1(x)$,
due to the assumption $\E\log^2 A_1<\infty$, see Lemma \ref{l:maj.p.e}.

By the assumption $\E\log A_1=0$,
\begin{eqnarray*}
\E\log(A_1+e^{-x}B_1) &=& \E\log A_1+\E\log(1+e^{-x}B_1/A_1)\\
&=& \E\{\log(1+e^{-x}B_1/A_1);\ B_1/A_1\le e^{x/2}\}\\
&&+\E\{\log(1+e^{-x}B_1/A_1);\ B_1/A_1\in(e^{x/2},e^x]\}\\
&&+\E\{\log(1+e^{-x}B_1/A_1);\ B_1/A_1>e^x\}\\
&=:& E_1+E_2+E_3.
\end{eqnarray*}
Using the inequality $\log(1+u)\le u$ we derive $E_1\le\log(1+e^{-x/2})\le e^{-x/2}$. Next,
\begin{eqnarray*}
E_2 &\le& (\log 2)\P\{B_1/A_1>e^{x/2}\}\\
&\le& \P\{\log B_1-\log A_1>x/2\}\\
&\le& \P\{\log B_1>x/4\}+\P\{-\log A_1>x/4\}\\
&=& o(p_2(x))\quad\mbox{as }x\to\infty,
\end{eqnarray*}
for some decreasing integrable at infinity function $p_2(x)$,
due to the assumptions $\E\log^2 A_1<\infty$ and $\E\log^2 B_1<\infty$, 
see Lemma \ref{l:maj.p.e}. Finally, by the same moment conditions,
\begin{eqnarray}\label{Rn_m1_3}
\nonumber
E_3 &\le& \E\{\log(2B_1/A_1);\ \log(B_1/A_1)>x\}\ =\ o(p_3(x))
\quad\mbox{as }x\to\infty,
\end{eqnarray}
for some decreasing integrable at infinity function $p_3(x)$,
see Lemma \ref{l:maj.p.e}.
Combining altogether, we obtain
\begin{eqnarray}\label{asy.m1.R}
m_1(x)=o(p_4(x)).
\end{eqnarray}
for some decreasing integrable at infinity function $p_4(x)$.
Thus, all moment conditions of Theorem~\ref{thm:pre-st.null} are met
with $\mu=0$. Together with the convergence to infinity \eqref{lln.Rn}
this completes the proof.
\qed\end{proof}

\index{Stochastic linear recursion!null recurrent!stationary distribution}
\begin{theorem}\label{thm:Rn_stationary}
Under the conditions of Theorem~\ref{thm:Rn-limit},  
the chain $\{R_n\}$ is null recurrent. 
In addition, if $\pi_R$ is an invariant measure of $\{R_n\}$
satisfying $\pi_R[0,x]<\infty$ for all $x$, then
$$
\pi_R(x_1,x_2]\ \sim\ c\log(x_2/x_1)
$$
as $x_1$, $x_2\to\infty$ in such a way that
$$
1<\liminf\frac{\log x_2}{\log x_1}\le\limsup\frac{\log x_2}{\log x_1}<\infty.
$$
\end{theorem}

\begin{proof}
We start with checking the moment condition of Corollary \ref{cor:null}
for $\{X_n\}$. It follows from the existence of a square integrable
majorant for the family of jumps that, for any $s(x)\to\infty$,
$$
m_2^{[s(x)]}(x)\ \to\ \sigma^2>0\quad\mbox{as }x\to\infty,
$$
and that there exists an $s(x)=o(x)$ such that
\begin{equation}\label{Rn_m1_4}
\E\{|\xi(x)|;|\xi(x)|\ge s(x)\}=o(p_5(x))\quad\mbox{as }x\to\infty,
\end{equation}
for some decreasing, integrable at infinity function $p_5(x)$,
see Lemma \ref{l:maj.p.e}. Together with \eqref{asy.m1.R} it implies that
\begin{equation}\label{Rn_m1_5}
m_1^{[s(x)]}(x)=o(p_4(x)+p_5(x)),
\end{equation}
and, hence,
$$
\frac{2m_1^{[s(x)]}(x)}{m_2^{[s(x)]}(x)}
\ =\ o(p_4(x)+p_5(x))\ =\ o(1/x)
\quad\mbox{as }x\to\infty.
$$
Thus, applying Corollary~\ref{cor:null},
we conclude that the chain $X_n=f(R_n)$ is null recurrent.
Consequently, $\{R_n\}$ is null recurrent as well.

Furthermore, \eqref{Rn_m1_5} and $m_2^{[s(x)]}(x)\to\sigma^2$ imply
that the function $U(x)$ defined in \eqref{def.u} has asymptotically
linear growth, $U(x)\sim Cx$ as $x\to\infty$.
Notice that the chain $\{X_n\}$ satisfies the moment conditions
\eqref{cond.xi.le}, \eqref{cond.3.moment}, and \eqref{cond.xi.ge}
from Theorem~\ref{thm:pi.recurrent}. Indeed,
the condition \eqref{cond.xi.le} is immediate from
the existence of a square integrable majorant,
For the same reason, the condition \eqref{cond.3.moment} 
follows from Lemma \ref{l:p.V.maj.o} with $\alpha=1$ and $\gamma=2$.
The condition \eqref{cond.xi.ge} follows from \eqref{Rn_m1_4} and from the fact
that $U(x)\sim Cx$. 

Then it follows from Theorem~\ref{thm:pi.recurrent} that
the stationary measure of $\{X_n\}$ has a linear growth:
$$
\pi_X(y_1,y_2)\sim c(y_2-y_1)
$$
provided $y_1$, $y_2\to\infty$ in such a way that
$$
1<\liminf\frac{y_2}{y_1}\le\limsup\frac{y_2}{y_1}<\infty.
$$
But it is clear that $\pi_R(x_1,x_2]=\pi_X(\log x_1,\log x_2]$
for all $x_1<x_2$ sufficiently large and the proof is complete.
\qed\end{proof}

\section{Application to the ALOHA network}
\sectionmark{ALOHA}
\label{sec:aloha}

We also illustrate the results with the Markov chain arising
from the model of the original ALOHA packet switching 
network,\index{Stochastic linear recursion!null recurrent!limit theorem}
originally proposed by Abramson\index{Abramson} \cite{Abramson},
and which was indeed a motivation for Borovkov\index{Borovkov},
Fayolle\index{Fayolle} and Korshunov\index{Korshunov} \cite{BFK}.
Let us first briefly recall the salient features of the system.

(a) A single error-free channel is shared among
an infinite population of users (or stations),
which retransmit messages of constant length (packets).
Time is slotted and may be considered discrete.
Users are syncronised with respect to the slots,
so that packets are transmitted at the beginning of slots only.
Each slot is equal to the time required to transmit a packet.

(b) Each transmission is within reception range of every user.
When more than one user transmits simultaneously,
packets collide (interfere) and none is received correctly.
These collisions are treated as transmission errors and each user
must strive to retransmit its colliding packet until it is
correctly received. The users all employ the same algorithm
for this purpose and have to resolve the contention without
the benefit of any other source of information on other user's
activity save the common channel.

(c) Each user with a colliding packet will repeatedly transmit
each time with a certain probability, until it hits a free slot
and thus succeeds.

The main drawback of the ALOHA protocol described above is that,
left to their own devices, the nodes congest the channel which,
in the absence of additional control, is non-ergodic.
The approach suggested in \cite{LK} was to let retransmission
probabilities be a function of the number of blocked stations at time $t$.
Such a retransmission control policy can stabilise the channel.

Let $A_n$ be the number of new packets generated by the stations
which are not blocked during the $n$th slot. 
We shall assume the $A_n$, $n\ge 1$,
form a sequence of independent identically distributed random variables,
with $\P\{A_1=k\}=p(k)$, $k\ge 0$, and finite expectation.
Let $X_n$, $n\ge 0$, be the number of blocked stations at time $n$
(i.e. observed at the beginning of the $n$th slot) and
$f(X_n)$ the probability that a blocked station retransmits
during this $n$th slot; so we consider centralised ALOHA algorithm
where information about the number of blocked stations is available
to the stations. Given $X_n=k$, the random number of messages
in the $n$th slot has a binomial distribution. Hence, $\{X_n\}$ forms
a Markov chain.

Define the quantity
\begin{eqnarray}\label{qn.fn}
q(k) &=& p(0)kf(k)(1-f(k))^{k-1}+p(1)(1-f(k))^k,
\end{eqnarray}
which represents the probability of successful transmission
%, in the wide sense, the mean downward drift of $X_n$
in the $n$th slot, given the event $X_n=k$. Clearly, if
$\E A_1<\liminf_{k\to\infty} q(k)$, then $\{X_n\}$ is positive
recurrent and possesses a probabilistic invariant measure.
If $\E A_1>\limsup_{k\to\infty} q(k)$, then $\{X_n\}$ is transient.

Our goal is to describe the asymptotic behaviour of $\{X_n\}$
in the asymptotically zero drift case.
Assume that, for all sufficiently large $k$,
\begin{eqnarray}\label{f.nfn}
f(k) &=& f/k\quad\mbox{for some }f>0.
\end{eqnarray}
Then \eqref{qn.fn} gives the following
limiting probability of successful transmission:
\begin{eqnarray}\label{q.qn}
q &=& \lim_{k\to\infty} q(k)\ =\ e^{-f}(fp(0)+p(1)).
\end{eqnarray}
Its maximal value $p(0)e^{p(1)/p(0)-1}$ is attained at $f=1-p(1)/p(0)$.

%For the problem to be meaningful, we have to choose $p(0)>p(1)$.
%Otherwise $q$, given by \eqref{q.qn}, would be a decreasing function of $f$;
%then $\E A_1>p(1)>q$ and the system could never be ergodic.
By direct computation, the first and second moments of the jumps
of $\{X_n\}$ are equal to
\begin{eqnarray}\label{mb.fn}
m_1(k) &=& \E A_1-q+\mu/k+O(1/k^2),\\
m_2(k) &=& b+O(1/k)\quad\mbox{as }k\to\infty,
\end{eqnarray}
where
$$
\mu\ =\ \frac{f^2e^{-f}}{2}[p(1)+p_0(f-2)],\qquad
b\ =\ \E A_1^2+(p(0)f-p(1))e^{-f}.
$$
It can never happen that $\mu\le -b/2$ because
\begin{eqnarray*}
e^f(2\mu+b) &=& f^2[p(1)+p_0(f-2)]+e^f\E A_1^2+(p(0)f-p(1))\\
&\ge& f^2[p(1)+p_0(f-2)]+(1+f)p(1)+(p(0)f-p(1))\\
&=& (f^2+f)p(1)+p_0f(f-1)^2\ >\ 0.
\end{eqnarray*}
\index{ALOHA packet switching network!transience}
\index{ALOHA packet switching network!null recurrence}
\index{ALOHA packet switching network!convergence to $\Gamma$-distribution}

\begin{theorem}
Let $\E A_1=q$. Then the Markov chain $\{X_n\}$ of the ALOHA protocol
is non-ergodic and the following main situations can take place:

(i) If $\mu>b/2$ then $\{X_n\}$ is transient;

(ii) If $-b/2<\mu<b/2$ then $\{X_n\}$ is null recurrent.

In addition, $X_n^2/n$ converges weakly as $n\to\infty$ to
a $\Gamma_{1/2+\mu/b,2b}$-distribution and, moreover, the process
$$
\frac{\log X_{[tn]}}{\sqrt{bn}},\ t\in[0,1],
$$
converges weakly in $D[0,1]$ to a Bessel process with drift $\mu/bx$
and diffusion coefficient $1$ as $n\to\infty$.
\end{theorem}

\begin{proof}
It is immediate from Corollaries \ref{cor:tr.log}, \ref{cor:null}
and Theorems \ref{thm:pre-st.null} and \ref{thm:gamma.func}.
\qed\end{proof}

\section{Comments to Chapter \ref{ch:applications}}

\subsection{Near-critical branching processes}

%Crane\index{Crane} et al. \cite{CGVWW}

Apparently Lamperti \cite{Lamperti1970}\index{Lamperti} was the first
who applied Markov chains to the study of branching processes and,
in particular,
Markov chains with asymptotically zero drift, see \cite{Lamperti1972}.
The use of square root transform for critical Galton--Watson
branching processes has been suggested
by Nagaev\index{Nagaev} and Wachtel\index{Wachtel} in \cite{NV06}.

Kersting\index{Kersting} \cite{Ker86} has studied transience and recurrence criteria 
for sequences of the form $X_{n+1}=X_n+g(X_n)+\xi_{n+1}$ 
where $\{\xi_k\}$ are square integrable martingale differences.
It is worth mentioning that state-dependent branching processes 
with migration---which were considered in Section \ref{sec:branching}---can 
be represented in this form.

For state-dependent processes without migration the weak convergence to a
$\Gamma$-distribution has been obtained in several papers.
Klebaner\index{Klebaner} \cite{Klebaner84}
has shown this convergence for processes satisfying $\max_{k\ge1}\E\zeta^m(k)<\infty$
for all $m\ge1$. H\"opfner \cite{Hopfner85}\index{H\"opfner}
has proved the same result under weaker moment assumptions.
He has shown that \eqref{bp.gamma} holds for processes satisfying
$\E\zeta(k)=1+a/k$, $|\sigma^2(k)-\sigma^2|=O(1)$ and
$\max_{k\ge1}\E\zeta^2(k)\log(1+\zeta(k))<\infty$.
Restrictions in Theorem~\ref{thm:bp.1} are significantly
weaker than those in the papers cited above.

Convergence of critical branching processes with immigration to
a $\Gamma$-distribution has been first proven by
Seneta~\cite{Seneta70}.\index{Seneta}
More precisely, he has shown that if $\zeta(k)$ are identically
distributed with expectation $1$ and variance $\sigma^2$ and if $\eta$
is non-negative with finite expectation then $Z_n/n$ converges weakly
to a $\Gamma$-distribution. If $\E\eta>\sigma^2/2$
then this is a particular case of our Theorem \ref{thm:bp.1}.
If $\E\eta\le\sigma^2/2$ then, in order to apply Theorem~\ref{thm:bp.null},
we have to check the validity of \eqref{bp.majorant} and \eqref{bp.epsilon}.
For identically distributed variables this condition is particularly
satisfied if $\E\zeta^2\log^{1+\varepsilon}(1+\zeta)<\infty$
for some $\varepsilon>0$.

For size-dependent processes without migration the asymptotic behaviour
of the non-extinction probability and the corresponding conditional
distribution has been studied earlier by 
H\"opfner\index{H\"opfner} \cite{Hopfner86}.
Assumptions in that paper are quite restrictive:
$\E\zeta(k)=1+a/k$ with some $a\in(0,\sigma^2/2]$,
$|\sigma^2(k)-\sigma^2|=O(1/k)$ and
$\max_{k\ge1}\E\zeta^{2+\delta}(k)<\infty$ for some $\delta>0$.
If $a<\sigma^2/2$ then the results in \cite{Hopfner86} coincide
with that in Theorem~\ref{thm:bp.conditioned},
but if $a=\sigma^2/2$ (this corresponds to $\rho=0$)
then \eqref{bp.cond.3} is still valid and $\P\{Z_n>0\}\sim c/\log n$.
This particular case is not covered by Theorem~\ref{thm:bp.conditioned}.

Zubkov\index{Zubkov} \cite{Zubkov72} has investigated the recurrence times 
to zero for branching processes with immigration.
He has shown that if $\E\eta<\sigma^2/2$ then there exists
a slowly varying function $L$ such that
$$
\P\{\min_{k\le n}Z_k>0\}\sim L(n)n^{2\E\eta/\sigma^2-1}.
$$
It is also shown there that one can take $L(n)\equiv C>0$
if and only if $\E\eta\log(1+\eta)<\infty$.
Vatutin\index{Vatutin} \cite{Vatutin77b} has shown that \eqref{bp.cond.3} 
holds under the same conditions.
Zubkov's\index{Zubkov} result shows that the restrictions
$\E|\eta|\log(1+|\eta|)<\infty$ and \eqref{bp.2nd.moments}
in Theorem~\ref{thm:bp.conditioned} are optimal for purely power tail
of the recurrence times.

Vatutin\index{Vatutin} \cite{Vatutin77a} has initiated the study 
of branching processes with emigration. More precisely, he has considered 
sequence $\{Y_n\}$ given by \eqref{BP2} with identically distributed
$\zeta(k)$ with mean one and $\eta\equiv-1$.  
For $\sigma^2=\E(\zeta-1)^2>2$ he has proven that
$\P\{Y_n>0\mid Y_0=m\}\sim L_m(n)n^{-1-2/\sigma^2}$ and that 
$L_m(n)\equiv c_m>0$ if anf only if $\E\zeta^2\log(1+\zeta)<\infty$. 
Moreover, for $\sigma^2<2$ he has shown that
$\P\{Y_n>0|Y_0=m\}\sim c_m n^{-1-2/\sigma^2}$ if and only if 
$\E\zeta^{1+2/\sigma^2}<\infty$.
Finally, assuming that all moments of $\zeta$ are finite, 
he has proved that $2Y_n/n\sigma^2$ conditioned on non-extinction 
converges weakly to a standard exponential distribution.
Kaverin \cite{Kaverin90}\index{Kaverin}
has generalized this results to all processes $Y_n$ satisfying
$\E(-\eta)^{[2+2/\sigma^2]}<\infty$, $\E\zeta^{1+2/\sigma^2}<\infty$ 
in the case $\sigma^2<2$ and $\E\zeta^2\log(1+\zeta)<\infty$ 
in the case $\sigma^2=2$.
Specialising Theorem~\ref{thm:bp.conditioned} to identically distributed 
$\zeta(k)$ and non-positive $\eta$, we conclude that \eqref{bp.cond.2} 
and \eqref{bp.cond.3} hold for all processes $Z_n$ satisfying 
$\E(-\eta)<\infty$, $\E\zeta^2\log(1+\zeta)<\infty$ and
$\E\zeta^{1+2/\sigma^2}<\infty$ in the case $\sigma^2<2$.
We see that our restrictions on the
emigration component $\eta$ are much weaker than that in \cite{Kaverin90}.

Kosygina\index{Kosygina} and Mountford\index{Mountford} \cite{KM11}
have proved \eqref{bp.cond.2} for a special model of branching
processes with migration.
This model appears in the description of excited random walks on integers.

First result of this type has been obtained by Foster\index{Foster} 
\cite{Foster71} for a critical Galton--Watson process with immigration at zero.
Formally, we cannot say that Foster's result follows from
\eqref{bp.log-scaled}. But since all calculations we have made in
the proof of Theorem~\ref{thm:bp.conditioned} are valid for processes 
without migration, it is easy to see that adding immigration at zero 
does not change the asymptotic behaviour of truncated moments. 
Therefore, Theorem~\ref{thm:pre-st.log} is applicable to the process 
from \cite{Foster71} if the number of immigrating individuals has finite mean.

Nagaev\index{Nagaev} and Khan\index{Khan} \cite{NK80} have proved
\eqref{bp.log-scaled} for a critical process with migration.
More precisely, they have considered the sequence $Y_n$
defined in \eqref{BP2} with identically distributed $\zeta(k)$
with mean one and finite variance.
Let us compare our moment assumptions with that in \cite{NK80}.
First we note that if $\zeta(k)$ are identically distributed
and have finite variance then \eqref{bp.prop.new.1}-\eqref{bp.prop.new.3}
hold automatically. The assumption \eqref{bp.2nd.moments}
which states $\E\zeta^2\log^{1+\varepsilon}(1+|\zeta|)<\infty$
is a bit more restrictive than the second moment assumption
in \cite{NK80}.
Further, we have assumed that $\E|\eta|\log(1+|\eta|)$ is finite,
which is weaker than the corresponding condition in \cite{NK80}.
It is assumed there that $\E\eta^2<\infty$ and $\P\{\eta>-m\}=1$
for some $m\ge1$.

Comparing our theorems with the known in the literature results
for branching processes with migration, we conclude that the only
weakness of the transformation $\sqrt{Z_n}$ is
the fact that it is not clear how to deal with the case when one
has tail asymptotics with non-trivial slowly varying functions.
Recall that the only obstacle is to show \eqref{r-cond.2.new}
in the case when $2m_1^{[s(x))]}(x)/m_2^{[s(x))]}(x)-c/x$ is
not integrable for any constant $c$.

\subsection{Stochastic difference equations}

Theorem \ref{thm:perpet.cramer} is due to Kesten\index{Kesten} \cite[Theorem 5]{Kesten1973};
for a complete proof and further related results
see Goldie~\cite[Theorem 4.1]{Goldie1991}.\index{Goldie}
In these papers a weaker moment condition $\E (\log A_1)A_1^\beta<\infty$
has been used. We have imposed \eqref{extra_cond} since we have to construct
a majorant $\Xi$ for the jumps $\xi(x)$ such that $\Xi e^{\beta\Xi}<\infty$.
One can prove the Kesten--Goldie result by using results
for asymptotically homogeneous chains under optimal moment assumptions.
Such a proof can be found in \cite{Korshunov2016}.

Theorem \ref{thm:Rn-limit} has been proven by Hitczenko\index{Hitczenko} 
and Wesolowski\index{Wesolowski} in \cite{HW2011}.

The asymptotic behaviour of $\pi_R$ in the null recurrent 
case discussed in Theorem \ref{thm:Rn_stationary} has been
studied in the literature. The most general results have been proven by
Babillot,\index{Babillot} Bougerol\index{Bougerol} and 
Elie\index{Elie} \cite{BBE1997} and by Brofferio\index{Brofferio}
and Buraczewski\index{Buraczewski} \cite{BroBuraczewski}:
if $\E\log A_1=0$ and $\E|\log A_1|^{2+\delta}+\E|\log B_1|^{2+\delta}<\infty$
then it was proven in \cite{BBE1997} that there exists 
a slowly varying function $L(x)$ such that
$$
\pi_R(ax,bx]\sim \log(b/a)L(x),\quad x\to\infty;
$$
it was shown in \cite[Theorem 1.1]{BroBuraczewski} that
$L(x)$ is a constant.

Theorem \ref{thm:Rn_stationary} says nothing about $\pi_R(ax,bx]$,
since $\log(bx)/\log(ax)\to1$ as $x\to\infty$.
But our result implies that a slowly varying function from the previous
relation cannot converge to either zero or infinity.
%Buraczewski\index{Buraczewski} \cite{Buraczewski2007} has shown that 
%if $\E A_1^\delta$, $\E A_1^{-\delta}$ and $\E B_1^\delta$ are finite 
%for some $\delta>0$ then $L(x)$ can be replaced by a constant.
Based on our Theorem~\ref{thm:Rn_stationary} it is plausible to expect
that $L(x)$ is a constant under the assumption 
that the second moment of both $A_1$ and $B_1$ is finite.

For thorough discussion on the topic see 
Buraczewski et al\index{Buraczewski} \cite{BDM}.

\backmatter%%%%%%%%%%%%%%%%%%%%%%%%%%%%%%%%%%%%%%%%%%%%%%%%%%%%%%%

%\bibliographystyle{abbrv}
%\bibliography{Lyapunov_book}
%%%%%%%%%%%%%%%%%%%%%%%% referenc.tex %%%%%%%%%%%%%%%%%%%%%%%%%%%%%%
% sample references
% %
% Use this file as a template for your own input.
%
%%%%%%%%%%%%%%%%%%%%%%%% Springer-Verlag %%%%%%%%%%%%%%%%%%%%%%%%%%
%
% BibTeX users please use
% \bibliographystyle{}
% \bibliography{}
%

\printindex

\end{document}